\documentclass[12pt]{amsbook}

\usepackage{graphics,amsmath,amssymb,amsthm,mathrsfs}

\newcommand{\br}{\mathbb{R}}

\newcommand{\varep}{\varepsilon}
\newcommand{\e}{\varepsilon}
\newcommand{\brd}{\mathbb{R}^d}

\newcommand{\wN}{\widetilde{N}}
\newcommand{\per}{\text{\rm per}}
\newcommand{\loc}{\text{\rm loc}}

\newtheorem{thm}{\bf Theorem}[section]
\newtheorem{lemma}[thm]{\bf Lemma}
\newtheorem{cor}[thm]{\bf Corollary}
\newtheorem{prop}[thm]{\bf Proposition}

\newtheorem{definition}[thm]{\bf Definition}
\newtheorem{remark}[thm]{\bf Remark}
\newcommand{\average}{-\!\!\!\!\!\!\int}
\newcommand{\rd}{\mathbb{R}^d}
\newcommand{\ta}{\text{tan}}

\numberwithin{equation}{section}

\numberwithin{section}{chapter}

\begin{document}

\frontmatter

\bibliographystyle{amsplain}

\title{\bf Lectures on Periodic  Homogenization \\
of Elliptic Systems}

\author{Zhongwei Shen}

\address{Department of Mathematics, University of Kentucky, Lexington, Kentucky 40506, USA}

\curraddr{}
\email{zshen2@uky.edu}

\thanks{Supported in part by NSF Grant 1600520}



\date{}



\date{ }

\maketitle

\setcounter{page}{4}

\tableofcontents
\mainmatter

\chapter*{Preface}

In recent years considerable advances have been made in quantitative homogenization of partial differential equations
in the periodic and non-periodic settings.
This monograph surveys the theory of quantitative homogenization for
 second-order linear elliptic systems in divergence form
 with rapidly oscillating periodic coefficients,
$$
\mathcal{L}_\e =-\text{\rm div} \big (A(x/\e)\nabla \big),
$$
in a bounded domain $\Omega$ in $\br^d$.
 It begins with a review of the classical qualitative homogenization theory, and
  addresses   the problem of convergence rates of solutions.
  The main body of the monograph  investigates various interior and boundary regularity estimates
   (H\"older, Lipschitz, $W^{1, p}$, nontangnetial-maximal-function)
    that are uniform in the small parameter $\e>0$. 
 Additional topics include convergence rates for Dirichlet eigenvalues and
 asymptotic expansions of fundamental solutions, Green functions, and Neumann functions.

In Chapter \ref{chapter-1} we present the quantitative homogenization theory for $\mathcal{L}_\e$,
which has been well understood since 1970's. We start out with a review of  basic facts
on weak solutions for elliptic systems with bounded measurable coefficients, and 
use the method of (formal) asymptotic expansions to derive the formula
for the homogenized (or effective) operator $\mathcal{L}_0$.
We then prove the classical results on the homogenization of Dirichlet and Neumann 
boundary value problems for $\mathcal{L}_\e$.

In Chapter \ref{chapter-C} we address the issue of convergence rates for solutions and two-scale expansions.
Various estimates  in $L^p$ and in $H^1$ are obtained without smoothness assumptions on the 
coefficient matrix $A$.

Chapters \ref{chapter-2}, \ref{chapter-3} and \ref{chapter-4} are devoted to the study of
sharp regularity estimates, which are uniform in $\e>0$, for solutions of $\mathcal{L}_\e (u_\e)=F$.
The case of interior estimates is treated in  Chapter \ref{chapter-2}.
We  use a compactness method to establish the Lipschitz estimate down to the microscopic scale $\e$
under the ellipticity and periodicity assumptions.
With additional smoothness assumptions on $A$,
this, together with a simple blow-up argument, leads the full-scale H\"older and Lipschitz estimates.
The compactness method, which originated from the study of the regularity 
theory in the calculus of variations and minimal surfaces, was introduced to the study of homogenization problems by
M. Avellaneda and F. Lin \cite{AL-1987}. 
In this chapter we also introduce a real-variable method for establishing $L^p$ and $W^{1, p}$ estimates.
The method, originated in a paper by L. Caffarelli and I. Peral \cite{CP-1998},
may be regarded a refined and dual version of the celebrated Calder\'on-Zygmund Lemma.
As corollaries of interior estimates, we obtain asymptotic expansions for the fundamental solution
$\Gamma_\e (x, y)$ and its derivatives $\nabla_x \Gamma_\e (x, y)$, $\nabla_y\Gamma_\e (x, y)$ and
$\nabla_ x \nabla_y \Gamma_\e (x, y)$, as $\e\to 0$.

In Chapter \ref{chapter-3} we study the boundary regularity estimates for solutions of
$\mathcal{L}_\e (u_\e)=F$ in $\Omega$ with the Dirichlet condition $u_\e=f$ on $\partial\Omega$.
The boundary Lipschitz estimate is proved by the compactness method mentioned above.
A key step is to prove the Lipschitz estimate for the so-called Dirichlet correctors.
The real-variable method introduced in Chapter \ref{chapter-2} is used to establish the boundary $W^{1, p}$ 
estimates. It effectively reduces to the problem to certain (weak) reverse H\"older inequalities. 

 In Chapter \ref{chapter-4} we prove the boundary H\"older, Lipschitz, and $W^{1, p}$ estimates
 for solutions of $\mathcal{L}_\e(u_\e)=F$ in $\Omega$ with
the Neumann condition $\frac{\partial u_\e}{\partial \nu_\e}=g$ in $\partial\Omega$.
Here we introduce a general scheme, recently developed by S. N. Armstrong and C. Smart \cite{Armstrong-2016}
in the study of stochastic homogenization, for establishing regularity estimates at large scale.
Roughly speaking, the scheme states that if a function $u$ is well approximated by 
$C^{1, \alpha}$ functions at every scale greater than $\e$, then $u$ is Lipschitz continuous at every scale greater than $\e$.

In Chapter \ref{chapter-6} we revisit the problem of convergence rates.
We establish an $O(\e)$ error estimate in $H^1$ for a two-scale expansion involving the Dirichlet correctors
and use it to prove a convergence result for the Dirichlet eigenvalue $\lambda_{\e, k}$.
We   derive the asymptotic expansions for the Green function $G_\e(x, y)$ and its derivatives, as
$\e \to 0$.
Analogous results are also obtained for the Neumann function $N_\e (x, y)$.

Chapter \ref{chapter-7} is devoted to the study of $L^2$ boundary value problems for
$\mathcal{L}_\e (u_\e)=0$ in a Lipschitz domain $\Omega$.
We establish optimal estimates in terms of nontangential maximal functions for Dirichlet problems with boundary
data in $L^2(\partial\Omega)$ and in $H^1(\partial\Omega)$
as well as the Neumann problem with boundary data in $L^2(\partial\Omega)$.
This is achieved by the method of layer potentials - the classical method of integral equations.
The asymptotic results for the fundamental solution $\Gamma_\e (x, y)$ in Chapter \ref{chapter-2}
are used to obtain the $L^p$ boundedness of singular integrals on $\partial\Omega$, associated with the single and double 
potentials. The proof of Rellich estimates,
\begin{equation}\label{001}
\Big\|\frac{\partial u_\e}{\partial\nu_\e} \Big\|_{L^2(\partial\Omega)}
\approx \|\nabla_{\tan } u_\e \|_{L^2(\partial\Omega)},
\end{equation}
which are crucial in the use of the method of layer potentials in Lipschitz domains,
is divided into two cases. 
In the small-scale case, where diam$(\Omega)\le \e$, the estimates are obtained by using 
Rellich identities and a three-step approximation argument.
The proof of (\ref{001}) for the large-scale case, where diam$(\Omega)>\e$, uses an error estimate
in $H^1(\Omega)$ for a two-scale expansion obtained in Chapter \ref{chapter-C}.
For reader's convenience we also include a section in which we prove (\ref{001}) and solve the
$L^2$ Dirichlet, Neumann and regularity problems in  Lipschitz domains for the case $\mathcal{L}=-\Delta$.

Part  of this monograph is based on lecture notes for courses I taught at 
several summer schools and at the University of Kentucky.
Much of material in  Chapters 6 and 7 is taken from my joint papers \cite{KLS-2013, KLS-2014} with Carlos Kenig and Fang-Hua Lin,
and from \cite{KS-2011-L} with Carlos Kenig.
I would like to express my deep gratitude to Carlos Kenig and Fang-Hua Lin
for introducing me to the research area of homogenization and for their important contribution to our joint work.

\aufm{Zhongwei Shen\\ Lexington, Kentucky\\ Fall 2017}

%
%
%

\chapter[Elliptic Systems with Periodic Coefficients]{Elliptic Systems of Second Order
with Periodic Coefficients}\label{chapter-1}

In this monograph we shall be concerned with a family of second-order linear elliptic 
systems in divergence form with rapidly oscillating periodic coefficients,
\begin{equation}\label{operator}
\mathcal{L}_\varep =-\text{div} \big( A\left({x}/{\varep}\right)\nabla\big)
=-\frac{\partial }{\partial x_i}
\left[ a_{ij}^{\alpha\beta}\left(\frac{x}{\varep}\right)
\frac{\partial}{\partial x_j} \right], \quad \varep>0,
\end{equation}
in $\mathbb{R}^d$
(the summation convention that the repeated indices are summed is used throughout).
We will always assume that the coefficient matrix (tensor)
$$
A(y)=\big(a_{ij}^{\alpha\beta} (y)\big),
\text{ with }1\le i,j\le d \text{ and } 1\le \alpha, \beta\le m,
$$
 is real, bounded measurable, and satisfies certain ellipticity condition,
 to be specified later.
We also assume that $A$ is 1-periodic; i.e.,
\begin{equation}\label{periodicity}
A(y+z)=A(y) \quad \text{ for a.e. } y\in \br^d \text{ and } z\in \mathbb{Z}^d.
\end{equation}
Observe that by a linear transformation one may replace $\mathbb{Z}^d$ in (\ref{periodicity}) by
any lattice in $\mathbb{R}^d$.

In this chapter we present the qualitative homogenization theory for $\mathcal{L}_\varep$.
We start out in Section \ref{section-1.1} with basic facts on weak solutions of second-order elliptic systems
in divergence form.
In Section \ref{section-1.2} we use the method of (formal) asymptotic expansions to derive the formula
for the homogenized (or effective) operator $\mathcal{L}_0$ with
{\it constant} coefficients.
In Section \ref{section-1.3} we prove some classical theorems on the homogenization of boundary value problems for
second-order elliptic systems.
In particular, we will show that if  $u_\varep \in H_0^1(\Omega;\br^m)$
and $\mathcal{L}_\varep (u_\varep)= F$ in a bounded Lipschitz domain $\Omega$, where $F\in H^{-1}(\Omega;\br^m)$,
then
$u_\varep \to u_0$ strongly in $L^2(\Omega;\br^m)$ and weakly in $H_0^1(\Omega;\br^m)$,
as $\varep\to 0$.
Moreover, the function $u_0\in H_0^1(\Omega;\br^m)$ is a solution of
$\mathcal{L}_0(u_0)=F$ in $\Omega$.
Section \ref{section-1.4} is devoted to the qualitative homogenization of 
elliptic systems of linear elasticity.

Throughout the monograph we will use $C$ and $c$ to denote positive constants
 that are independent of the parameter $\varep>0$.
They may change from line to line and depend on $A$ and/or $\Omega$.
We will use $\average_E u$ to denote the $L^1$ average of a function $u$ over a set $E$; i.e.
$$
\average_E u =\frac{1}{|E|} \int_E u.
$$



\section{Weak solutions}\label{section-1.1}

In this section we review  basic facts on weak solutions of second-order elliptic systems
 with bounded measurable coefficients. 
 For convenience of reference  it will be done
in the context of operator $\mathcal{L}_\varep$.
However,  the periodicity condition (\ref{periodicity}) is not used in the section.

For a domain $\Omega$ in $\br^d$ and $1\le p\le  \infty$, let
$$
W^{1,p} (\Omega;\br^m)=\Big\{ u\in L^p(\Omega;\br^m): \ \nabla u\in L^p(\Omega; \mathbb{R}^{m \times d} )\Big\}.
$$
Equipped with the  norm 
$$
\| u\|_{W^{1,p}(\Omega)} :=\left\{ \| \nabla u\|^p_{L^p(\Omega)} + \|  u\|^p_{L^p(\Omega)}\right\}^{1/p}
$$
for $1\le p<\infty$, and $\|u\|_{W^{1, \infty}(\Omega)} :=\|u\|_{L^\infty(\Omega)} +\| \nabla u\|_{L^\infty(\Omega)}$, 
$W^{1,p}(\Omega;\br^m)$ is a Banach space.
For $1<p<\infty$, let $W^{1,p}_0(\Omega;\br^m)$ denote the closure of $C_0^\infty(\Omega;\br^m)$ 
in $W^{1,p}(\Omega;\br^m)$
and $W^{-1, p} (\Omega;\br^m)$ the dual of $W^{1,p^\prime}_0(\Omega;\br^m)$, where
$p^\prime=\frac{p}{p-1}$.
If $p=2$, we often use the usual notation: $H^1(\Omega;\br^m)=W^{1,2}(\Omega;\br^m)$,
$H_0^1(\Omega;\br^m)=W_0^{1,2}(\Omega;\br^m)$, and $H^{-1}(\Omega;\br^m)=W^{-1,2}(\Omega;\br^m)$.

\begin{definition}
{\rm
Let $\mathcal{L}_\varep =-\text{div} (A(x/\varep)\nabla)$ with $A(y)=\big(a_{ij}^{\alpha\beta} (y)\big)$.
For $F\in H^{-1}(\Omega;\br^m)$,
we call $u_\varep\in H^1(\Omega;\br^{m})$ 
 a weak solution of
$\mathcal{L}_\varep (u_\varep) = F$ in $\Omega$, if
\begin{equation}
\label{weak-solution}
\int_\Omega A(x/\varep)\nabla u_\varep \cdot \nabla \varphi\, dx
=\langle F, \varphi\rangle_{H^{-1}(\Omega)\times H^1_0(\Omega)}
\end{equation}
for any test function $\varphi\in C_0^\infty(\Omega;\br^m)$.
}
\end{definition}

To establish the existence of weak solutions for the Dirichlet problem, we introduce the following ellipticity condition:
there exists a constant $\mu>0$ such that 
\begin{align}
\|A\|_\infty   & \le \mu^{-1},\label{weak-e-1}\\
\mu \int_{\brd}
|\nabla u|^2\, dx
& \le \int_{\brd} A\nabla u\cdot \nabla u\, dx \quad
\text{ for any $u\in C_0^\infty(\br^d; \mathbb{R}^m)$.}
\label{weak-e-2}
\end{align}
Observe that the condition (\ref{weak-e-1})-(\ref{weak-e-2}), which is referred as
the $V$-ellipticity,  is invariant under translation and dilation.
In particular, if $A=A(x)$ satisfies (\ref{weak-e-1})-(\ref{weak-e-2}),
so does $A^\e=A(x/\e)$ with the same constant $\mu$.

\begin{lemma}\label{e-lemma}
The integral condition (\ref{weak-e-2}) implies 
 the following algebraic condition, 
\begin{equation}\label{L-ellipticity}
\mu |\xi|^2 |\eta|^2\le {a}_{ij}^{\alpha\beta} (y) \xi_i\xi_j\eta^\alpha \eta^\beta
\end{equation}
for a.e.~$y\in \brd$,
where $\xi=(\xi_1, \dots, \xi_d)\in \brd$ and $\eta=(\eta^1, \dots, \eta^m)\in \mathbb{R}^m$.
\end{lemma}

\begin{proof}
Since $A$ is real, it follows from (\ref{weak-e-2}) that
\begin{equation}\label{el-1}
\mu \int_{\brd}
|\nabla u|^2\, dx
 \le  Re\int_{\brd} A\nabla u\cdot \nabla \overline{u}\, dx \quad
\text{ for any $u\in C_0^\infty(\br^d; \mathbb{C}^m)$.}
\end{equation}
Fix  $y\in \brd$, $\xi\in \br^d$ and $\eta\in \br^m$.
Let 
$$
 u(x)=\varphi_\e (x) t^{-1} e^{i t \xi \cdot x} \eta
 $$
in (\ref{el-1}), where $t>0$, $\varphi_\e (x)=\e^{-d/2} \varphi (\e^{-1} (x-y))$ and $\varphi$ is a function in $ C_0^\infty(\brd)$
with $\int_{\brd} \varphi^2\, dx=1$.
Using
$$
\frac{\partial u^\beta}{\partial x_j}
=i \varphi_\e  e^{it\xi \cdot x} \xi_j \eta^\beta +\frac{\partial \varphi_\e}{\partial x_j}
t^{-1} e^{it\xi \cdot x} \eta^\beta,
$$
we see that as $ t\to \infty$,
$$
\aligned
\int_{\brd} |\nabla u|^2\, dx
 & =|\xi|^2 |\eta|^2 \int_{\brd} |\varphi_\e|^2\, dx + O(t^{-1}),\\
 Re \int_{\brd} A\nabla u \cdot \nabla \overline{u}\, dx
 &=\int_{\brd} a_{ij}^{\alpha\beta} (x)\xi_i\xi_j \eta^\alpha \eta^\beta |\varphi_\e |^2\, dx +O(t^{-1}).
 \endaligned
 $$
 In view of (\ref{el-1}) this implies that
 \begin{equation}\label{el-2}
 \mu |\xi|^2 |\eta|^2 \int_{\brd} |\varphi_\e|^2\, dx
 \le \int_{\brd} a_{ij}^{\alpha\beta} (x)\xi_i\xi_j \eta^\alpha \eta^\beta |\varphi_\e |^2\, dx.
 \end{equation}
 Since $\varphi_\e^2(x)=\e^{-d} \varphi^2 (\varep^{-1}(x-y))$ is a mollifier,
 the inequality (\ref{L-ellipticity}) follows   by letting $\e \to 0$ in (\ref{el-2}).
\end{proof}

The ellipticity condition (\ref{L-ellipticity})
is called the Legendre-Hadamard condition.
It follows from  Lemma \ref{e-lemma} that  in the scale case ($m=1$), the conditions (\ref{weak-e-2}) and (\ref{L-ellipticity})
are equivalent. 
By using the Plancherel Theorem one may also show the equivalency when $m\ge 2$ and  $A$ is constant.

\begin{thm}\label{Ca-theorem-0}
Suppose that $A$ satisfies (\ref{weak-e-1})-(\ref{weak-e-2}).
Let $u_\varep\in H^1(\Omega; \br^m)$ be a weak solution of
$\mathcal{L}_\varep (u_\varep)=F +\text{\rm div} (G)$ in $\Omega$, where
$F\in L^2(\Omega; \mathbb{R}^m)$ and $G=(G_i^\alpha)\in L^2(\Omega; \mathbb{R}^{m\times d})$.
Then, for any $\psi\in C_0^1(\Omega)$,
\begin{equation}\label{Ca-0}
\int_\Omega |\nabla u_\varep|^2 |\psi|^2\, dx
\le C\, \left\{
 \int_\Omega |u_\varep|^2 |\nabla \psi|^2\, dx
+\int_\Omega |G|^2 |\psi|^2\, dx
+\int_\Omega |F||u_\varep ||\psi|^2\, dx \right\},
\end{equation}
where $C$ depends only on $\mu$.
\footnote{The constants $C$ and $c$ in this monograph may also depend on
$d$ and $m$. However, this fact is irrelevant to our investigation and will be ignored.}
\end{thm}

\begin{proof}
Note that by (\ref{weak-solution}),
\begin{equation}\label{weak-sol-0}
\int_\Omega A(x/\varep)\nabla u_\varep \cdot \nabla \varphi\, dx
=\int_\Omega F^\alpha \varphi^\alpha\, dx
-\int_\Omega G_i^\alpha \frac{\partial\varphi^\alpha}{\partial x_i}\, dx
\end{equation}
for any $\varphi=(\varphi^\alpha)\in C_0^\infty(\Omega;\br^m)$.
Since $u_\varep\in H^1(\Omega; \br^m)$, 
by a density argument,
(\ref{weak-sol-0}) continues to hold for any $\varphi\in H^1_0(\Omega; \br^m)$.
Observe that
\begin{equation}\label{Ca-Id}
\aligned
A(x/\e)\nabla (\psi u_\e) \cdot \nabla (\psi u_\e)
= &A(x/\e)\nabla u_\e \cdot \nabla (\psi^2 u_\e)
 +A(x/\e) (\nabla \psi) u_\e \cdot \nabla (\psi u_\e)\\
&-A(x/\e)\nabla (\psi u_\e) \cdot (\nabla \psi) u_\e
+A(x/\e)(\nabla \psi) u_\e \cdot (\nabla \psi) u_\e,
\endaligned
\end{equation}
where $\psi\in C_0^1(\Omega)$.
It follows that
$$
\aligned
 \int_\Omega &
A(x/\varep)\nabla (u_\varep \psi)\cdot \nabla (u_\varep \psi)\, dx\\
& =\int_\Omega
A(x/\varep)\nabla u_\varep \cdot \nabla (u_\varep \psi^2)\, dx
+\int_\Omega A(x/\e) (\nabla \psi) u_\e \cdot \nabla (\psi u_\e)\, dx\\
&\qquad -\int_\Omega A(x/\e)\nabla (\psi u_\e) \cdot (\nabla \psi) u_\e\, dx
+\int_\Omega A(x/\varep) u_\varep \nabla \psi \cdot u_\varep \nabla \psi\, dx\\
&=\int_\Omega F^\alpha u_\varep^\alpha \psi^2\, dx
-\int_\Omega G_i^\alpha \frac{\partial}{\partial x_i} \big( u_\varep^\alpha \psi^2\big)\, dx
 +\int_\Omega A(x/\e) (\nabla \psi) u_\e \cdot \nabla (\psi u_\e)\, dx\\
 &\qquad  -\int_\Omega A(x/\e)\nabla (\psi u_\e) \cdot (\nabla \psi) u_\e\, dx
+\int_\Omega A(x/\varep) u_\varep \nabla \psi \cdot u_\varep \nabla \psi\, dx,
\endaligned
$$
where we have used (\ref{weak-sol-0}) with $\varphi=u_\varep\psi^2$ for the last step.
Hence, by (\ref{weak-e-1})-(\ref{weak-e-2}),
$$
\aligned
\mu \int_\Omega |\nabla (u_\varep\psi) |^2 \, dx
& \le \int_\Omega A(x/\varep)\nabla (u_\varep\psi)\cdot \nabla (u_\varep\psi)\, dx\\
 & \le  \int_\Omega |F ||u_\varep| |\psi |^2\, dx
+C\int_\Omega |G\psi | |\nabla (u_\varep\psi) |\, dx
+C\int_\Omega |G\psi ||u_\e \nabla \psi | \, dx\\
& \qquad
+C \int_\Omega |\nabla (\psi u_\e)| |u_\e | |\nabla \psi|\, dx
+C \int_\Omega |u_\varep |^2 |\nabla \psi|^2\, dx.
\endaligned
$$
This yields (\ref{Ca-0}) by applying the Cauchy inequality
\begin{equation}\label{Cauchy}
ab\le \delta a^2 +\frac{b^2}{4\delta},
\end{equation}
where $a,b \ge 0$ and $\delta>0$.
\end{proof}

For a ball 
$$
B=B(x_0,r)=\{ x\in \br^d: \, |x-x_0|<r\}
$$
in $\br^d$, we will use $tB$ to denote $B(x_0,tr)$, the ball with the same center and $t$ times the radius
as $B$.
Let $u_\varep\in H^1(2B;\br^m)$ be a weak solution of
$\mathcal{L}_\varep (u_\varep)=F +\text{\rm div} (G)$ in $2B$, where 
$F\in L^2(2B; \mathbb{R}^m)$ and $G=(G_i^\alpha)\in L^2(2B; \mathbb{R}^{m\times d})$.
Then 
\begin{equation}\label{Cacciopoli-1.1}
\int_{sB} |\nabla u_\varep|^2 \, dx \le 
C \left\{ \frac{1}{(t-s)^2r^2}\int_{tB} |u_\varep-E |^2\, dx
+r^2 \int_{tB} |F|^2\, dx
+\int_{tB} |G|^2\, dx \right\}
\end{equation}
for any $1<s<t<2$ and $E\in \mathbb{R}^m$, where $C$ depends only on $\mu$.
The inequality (\ref{Cacciopoli-1.1}) is called  (interior) Caccioppoli's inequality.
To see (\ref{Cacciopoli-1.1}),
one applies  Theorem \ref{Ca-theorem-0} to $u_\varep-E$ and choose $\psi\in C_0^1(tB)$
so that $0\le \psi\le 1$, $\psi=1$ on $sB$, and
$|\nabla \psi|\le C (t-s)^{-1}r^{-1}$.

\begin{thm}[Reverse H\"older inequality]\label{reverse-holder-theorem}
Suppose that $A$ satisfies conditions (\ref{weak-e-1})-(\ref{weak-e-2}).
Let $u_\varep\in H^1(2B;\br^m)$ be a weak solution of
$\mathcal{L}_\varep (u_\varep)=0$ in $2B$, where $B=B(x_0, r)$
for some $x_0\in \br^d$ and $r>0$. Then there exists some $p>2$, depending only on $\mu$ (and $d, m$),
such that
\begin{equation}\label{reverse-Holder-1.1}
\left(\average_B |\nabla u_\varep|^p\, dx \right)^{1/p}
\le C
\left(\average_{2B} |\nabla u_\varep|^2\, dx \right)^{1/2},
\end{equation}
where $C$ depends only on $\mu$.
\end{thm}

\begin{proof}
Suppose $\mathcal{L}_\e (u_\e)=0$ in $2B$.
It follows from (\ref{Cacciopoli-1.1}) by Sobolev-Poncar\'e inequality that
\begin{equation}\label{r-h-s1}
\left(\average_{sB} |\nabla u_\varep|^2\, dx \right)^{1/2}
\le \frac{C}{t-s}
\left(\average_{tB} |\nabla u_\varep|^q\, dx \right)^{1/q},
\end{equation}
where $1<s<t<2$ and $\frac{1}{q}=\frac12 +\frac{1}{d}$.
This gives a reverse H\"older inequality, which has the so-called self-improving property.
We refer the reader to \cite[Chapter V]{Gia-book} for a proof of the property.
\end{proof}

We are interested in the Dirichlet boundary value problem, 
\begin{equation}\label{Dirichlet-problem-1.1}
\left\{
\begin{aligned}
\mathcal{L}_\varep (u_\varep)  & =F +\text{\rm div} (G)& \quad & \text{ in } \Omega,\\
 u_\varep & =f  &\quad  & \text{ on } \partial \Omega,
 \end{aligned}
 \right.
 \end{equation}
 and the Neumann boundary value problem,
 \begin{equation}\label{Neumann-problem-1.1}
 \left\{
\begin{aligned}
\mathcal{L}_\varep (u_\varep)  & =F +\text{\rm div} (G) &\quad & \text{ in } \Omega,\\
\frac{\partial u_\varep}{\partial\nu_\varep}  & =g -n\cdot G &\quad& \ \text{ on } \partial\Omega,
\end{aligned}
\right.
\end{equation}
 with non-homogeneous boundary conditions,
 where the conormal derivative $\frac{\partial u_\varep}{\partial \nu_\varep}$
 on $\partial\Omega$ is defined by
 \begin{equation}\label{definition-of-conormal}
 \left(\frac{\partial u_\varep}{\partial \nu_\varep}\right)^\alpha
 =n_i (x)a_{ij}^{\alpha\beta} (x/\varep) \frac{\partial u_\varep^\beta}{\partial x_j},
 \end{equation}
 and $n=(n_1, \dots, n_d)$ denotes the outward unit normal to $\partial\Omega$.
 
 Let $\Omega$ be a bounded Lipschitz domain in $\rd$.
 The space $H^{1/2}(\partial\Omega)$ may be defined as the subspace of
 $L^2(\partial\Omega)$ of functions $f$ for which
 $$
\| f\|_{ H^{1/2}(\partial\Omega)}:
 =\left\{ \int_{\partial\Omega} |f|^2\, d\sigma  +
 \int_{\partial\Omega}\int_{\partial\Omega}
 \frac{|f(x)-f(y)|^2}{|x-y|^{d}}\, d\sigma (x)d\sigma (y)\right\}^{1/2}<\infty.
 $$
 
\begin{thm}\label{theorem-1.1-2}
Assume that $A$ satisfies (\ref{weak-e-1})-(\ref{weak-e-2}).
Let $\Omega$ be a bounded Lipschitz domain in $\br^d$.
 Then, for any $f\in H^{1/2}(\partial\Omega;\br^m)$,
$F\in L^2(\Omega;\br^m)$ and $G\in L^2(\Omega; \br^{m\times d})$,
 there exists
a unique $u_\varep\in H^1(\Omega;\br^m)$ such that
$\mathcal{L}_\varep (u_\varep)=F +\text{\rm div} (G)$ in $\Omega$ and $u_\varep=f$ on $\partial\Omega$
in the sense of trace. Moreover,
the solution satisfies the energy estimate
 \begin{equation}\label{Dirichet-estimate-1.1}
 \| u_\varep \|_{H^1(\Omega)}
 \le C\,  \Big\{ \|F\|_{L^2(\Omega)} +\| G\|_{L^2(\Omega)}
 +\| f\|_{H^{1/2}(\partial\Omega)}\Big\},
 \end{equation}
 where $C$ depends only on $\mu$ and $\Omega$.
 \end{thm}
 
 \begin{proof}
 In the case where $f=0$ on $\partial\Omega$,
this follows by applying the Lax-Milgram Theorem to the bilinear
form 
\begin{equation}\label{bilinear-form}
B[u,v]=\int_\Omega A(x/\varep)\nabla u \cdot \nabla v\, dx
\end{equation}
on the Hilbert space $H^1_0(\Omega;\mathbb{R}^m)$.
In general, if $f\in H^{1/2}(\partial\Omega;\br^m)$, then $f$ is the trace of a function  $w$
 in $H^1(\Omega;\br^m)$ with 
 $$
 \| w\|_{H^1(\Omega)} \le C\, \| f\|_{H^{1/2}(\partial\Omega)}.
 $$
 By considering $u_\varep -w$, one may reduce the general case to
the case where $f=0$.
 \end{proof}
 
We now consider the Neumann boundary value problem.
Let $H^{-1/2}(\partial\Omega;\br^m)$ denote the
dual of $H^{1/2}(\partial\Omega;\br^m)$.

\begin{definition}
{\rm
We call $u_\varep  \in H^1(\Omega;\br^m)$  a weak solution of
the Neumann problem (\ref{Neumann-problem-1.1}) with data $F\in L^2(\Omega;\br^m)$,
$G\in L^2(\Omega; \br^{m\times d})$  and
$g\in H^{-1/2}(\partial\Omega;\br^m)$, if
\begin{equation}\label{weak-solution-Neumann-1.1}
\int_\Omega A(x/\varep)\nabla u_\varep\cdot \nabla \varphi\, dx
=\int_\Omega F\cdot \varphi\, dx
-\int_\Omega G \cdot \nabla \varphi\, dx
+\langle g, \varphi\rangle_{H^{-1/2}(\partial \Omega)\times H^{1/2}(\partial\Omega)}
\end{equation}
for any $\varphi\in C^\infty (\br^d;\br^m)$.
}
\end{definition}

If $m\ge 2$, the ellipticity condition in (\ref{weak-e-1})-(\ref{weak-e-2})
 is not sufficient for solving the Neumann  problem. 
 As such, we introduce the very strong ellipticity condition,
 also called the Legendre condition:
 there exists a constant $\mu>0$ such that
 \begin{equation}\label{s-ellipticity}
 \mu |\xi|^2 \le a_{ij}^{\alpha\beta}(y) \xi_i^\alpha \xi_j^\beta \le \frac{1}{\mu} |\xi|^2
 \end{equation}
 for a.e. $y\in \brd$, where $\xi =(\xi_i^\alpha)\in \mathbb{R}^{m\times d}$.
 It is easy to see that 
 $$
 \text{(\ref{s-ellipticity})}
  \Longrightarrow
  \text{(\ref{weak-e-1})-(\ref{weak-e-2})}.
  $$
  
 \begin{thm}\label{s-NP-theorem}
Let $\Omega$ be a bounded Lipschitz domain in $\brd$ and
$A$ satisfy (\ref{s-ellipticity}).
Assume that $F\in L^2(\Omega; \mathbb{R}^m)$, $G\in L^2(\Omega; \mathbb{R}^{m\times d})$ and
$g\in H^{-1/2}(\partial\Omega; \mathbb{R}^m)$ satisfy the compatibility condition
\begin{equation}\label{s-comp}
\int_\Omega F\cdot b \, dx  +\langle g, b\rangle_{H^{-1/2}(\partial\Omega)\times H^{1/2}(\partial\Omega)} =0
\end{equation}
for any $b\in \mathbb{R}^m$. Then the Neumann problem (\ref{Neumann-problem-1.1}) has a weak solution
$u_\e$, unique up to a constant in $\mathbb{R}^m$, in $H^1(\Omega; \mathbb{R}^m)$.
Moreover, the solution satisfies the energy estimate
\begin{equation}\label{estimate-s-NP}
\|\nabla u_\varep\|_{L^2(\Omega)} 
\le C \Big\{ \| F\|_{L^2(\Omega)} +\| G \|_{L^2(\Omega)} + \| g\|_{H^{-1/2}(\partial\Omega)} \Big\},
\end{equation}
where $C$ depends only on $\mu$ and $\Omega$.
 \end{thm}
 
 \begin{proof}
 Using (\ref{s-ellipticity}), one obtains 
 $$
 \mu \int_\Omega |\nabla u|^2\, dx \le \int_\Omega  A(x/\e)\nabla u\cdot \nabla u\, dx
 $$
 for any $u\in H^1(\Omega; \mathbb{R}^m)$.
 The results follow from the Lax-Milgram Theorem by considering the bilinear form
(\ref{bilinear-form}) on the Hilbert space $H^1(\Omega; \mathbb{R}^m)/\mathbb{R}^m$.
  \end{proof}
  


\section[Asymptotic expansions]{Two-scale asymptotic expansions and the homogenized operator}\label{section-1.2}

Let
$\mathcal{L}_\varep =-\text{\rm div}(A(x/\varep)\nabla)$ with
matrix $A=A(y)$ satisfying (\ref{weak-e-1})-(\ref{weak-e-2}).
Also assume that $A$ is 1-periodic.
In this section we use the method of formal two-scale asymptotic expansions to derive the formula for the homogenized (effective)
operator for $\mathcal{L}_\varep$.

 Suppose that $\mathcal{L}(u_\varep)=F$ in $\Omega$.
Let 
\begin{equation}\label{definition-of-Y}
Y=[0,1)^d\cong \br^d/\mathbb{Z}^d
\end{equation}
be the elementary cell for the lattice $\mathbb{Z}^d$. 
In view of the coefficients of $\mathcal{L}_\varep$,
one seeks a solution $u_\varep$ in the form 
\begin{equation}\label{expansion-1}
u_\varep (x)=u_0(x, x/\varep) + \varep u_1 (x, x/\varep) +\varep^2 u_2 (x, x/\varep)+\cdots,
\end{equation}
where the functions $u_j(x,y)$ are defined on $\Omega \times \brd$ and  1-periodic in $y$,
for any $x\in \Omega$.

Note that if $\phi_\varep (x) =\phi (x,y)$ with $y=x/\varep$, then
$$
\frac{\partial \phi_\varep}{\partial x_j}
=\frac{1}{\varep} \frac{\partial \phi}{\partial y_j} +\frac{\partial \phi}{\partial x_j}.
$$
It follows that
\begin{equation}\label{expansion-3}
\aligned
\mathcal{L}_\varep \big( u_j(x, x/\varep)\big)
= \varep^{-2} {L}^0 \big(u_j(x,y)\big) (x, x/\varep)
&+\varep^{-1} {L}^1 \big(u_j (x,y)\big) (x, x/\varep)\\
& +L^2 \big(u_j(x,y)\big) (x, x/\varep),
\endaligned
\end{equation}
where the operators $L^0, L^1, L^2$ are defined by
\begin{equation}\label{expansion-4}
\aligned
L^0( \phi (x,y)) & =-\frac{\partial}{\partial y_i}
\left\{ a_{ij}^{\alpha\beta} (y) \frac{\partial \phi^\beta}{\partial y_j}\right\},\\
L^1 (\phi(x,y))
& =-\frac{\partial}{\partial x_i} \left\{ a_{ij}^{\alpha\beta} (y) \frac{\partial \phi^\beta}{\partial y_j} \right\}
-\frac{\partial}{\partial y_i}
\left\{ a_{ij}^{\alpha\beta} (y) \frac{\partial \phi^\beta}{\partial x_j} \right\},\\
L^2 (\phi(x,y))
& =-\frac{\partial}{\partial x_i}\left\{ a_{ij}^{\alpha\beta} (y) \frac{\partial \phi^\beta}{\partial x_j} \right\}.
\endaligned
\end{equation}
In view of (\ref{expansion-1}) and (\ref{expansion-3}) we obtain, at least formally,
\begin{equation}\label{expansion-5}
\aligned
\mathcal{L}_\varep (u_\varep)= \varep^{-2} L^0(u_0)
& +\varep^{-1} \big\{ L^1(u_0) +L^0(u_1)\big\}\\
&+\big\{ L^2(u_0) +L^1(u_1) +L^0(u_2) \big\} +\cdots.
\endaligned
\end{equation}
Since $\mathcal{L}_\varep (u_\varep)=F$, by identifying the powers of $\varep$, it follows from (\ref{expansion-5})
that
\begin{equation}\label{expansion-6.1}
L^0(u_0)=0,
\end{equation}
\begin{equation}\label{expansion-6.2}
L^1(u_0) +L^0(u_1)=0,
\end{equation}
\begin{equation}\label{expansion-6.3}
L^2(u_0) +L^1(u_1) +L^0(u_2)=F.
\end{equation}
Using the fact that $u_0(x,y)$ is 1-periodic in $y$, we may deduce from (\ref{expansion-6.1}) that
$$
\int_Y A(y)\nabla_y u_0\cdot \nabla_y u_0\, dy=0.
$$
Under the ellipticity  condition (\ref{weak-e-1})-(\ref{weak-e-2}) and periodicity condition (\ref{periodicity}), we 
will show that 
\begin{equation}\label{periodic-Korn}
\mu \int_Y |\nabla_y \phi|^2\, dy
\le \int_Y  A(y)\nabla_y \phi \cdot \nabla_y \phi\, dy \qquad \text{ for any } \phi\in H^1_{\per} (Y; \br^m).
\end{equation}
See Lemma \ref{lemma-elliptic-w}.
It follows that $\nabla_y u_0=0$. Thus 
$u_0(x,y)$ is independent of $y$; i.e.,
\begin{equation}\label{expansion-7}
u_0(x,y)=u_0(x).
\end{equation}
Here and henceforth,  $H^k_{\per}(Y;\br^m)$ 
denotes the closure in $H^k(Y;\br^m)$ of  $C^\infty_{\per}(Y;\br^m)$, 
the set of $C^\infty$ and 1-periodic  functions in $\br^d$.

To derive the equation for $u_0$, we first use
 (\ref{expansion-7}) and (\ref{expansion-6.2}) to obtain 
\begin{equation}\label{expansion-8}
\left( L^0(u_1)\right)^\alpha
=-\left( L^1(u_0)\right)^\alpha
=\frac{\partial }{\partial y_i} \left\{ a_{ij}^{\alpha\beta} (y)\right\} \frac{\partial u_0^\beta}{\partial x_j}.
\end{equation}
By the Lax-Milgram Theorem and (\ref{periodic-Korn})
one can show that
if $h\in L^2_{\loc}  (\br^d;\br^{m \times d})$ is 1-periodic,  the cell problem
\begin{equation}\label{expansion-9}
\left\{\aligned
& L^0(\phi)   =\text{ div} (h) \quad \text{ in } Y,\\
& \phi \in H^1_{\per} (Y;\br^m),
\endaligned
\right.
\end{equation}
has a unique (up to constants) solution.
In view of (\ref{expansion-8}) we may write
\begin{equation}\label{expansion-10}
u_1^\alpha(x,y) = \chi_j^{\alpha\beta} (y) \frac{\partial u_0^\beta}{\partial x_j} (x)
+\widetilde{u}_1^\alpha (x),
\end{equation}
where, for each $1\le j\le d$ and $1\le \beta\le m$, the function
$\chi_j^\beta =(\chi_j^{1\beta }, \dots, \chi_j^{m\beta})\in H_\per^1(Y;\br^m)$ is the unique solution
of the following cell problem:
\begin{equation}\label{cell-problem}
\left\{
\aligned
& L^0 (\chi_j^\beta)=-L^0 (P^\beta_j) \quad \text{ in } Y,\\
&\chi_j^\beta (y) \text{ is 1-periodic},\\
& \int_{Y}
\chi_j^\beta \, dy =0.
\endaligned
\right.
\end{equation}
In (\ref{cell-problem}) and henceforth, $P_j^\beta =P_j^\beta (y)=y_je^\beta$, where 
$e^\beta =(0, \dots, 1, \dots, 0)$ with $1$ in the $\beta^{th}$ position.
 Note that the $\alpha^{th}$ component of $-L^0(P_j^\beta)$
 is $\frac{\partial}{\partial y_i} \big( a_{ij}^{\alpha\beta} (y)\big)$.
  
We now use the equations (\ref{expansion-6.3}) and (\ref{expansion-10}) to obtain
$$
\aligned
\big(L^0(u_2)\big)^\alpha & =F^\alpha-\big(L^2 (u_0)\big)^\alpha-\big(L^1(u_1)\big)^\alpha\\
&=F^\alpha (x)+a_{ij}^{\alpha\beta} (y) \frac{\partial^2 u_0^\beta}{\partial x_i\partial x_j}
+a_{ij}^{\alpha\beta} (y) \frac{\partial^2 u_1^\beta}{\partial x_i\partial y_j}
+\frac{\partial}{\partial y_i}
\left\{ a_{ij}^{\alpha\beta} (y) \frac{\partial u_1^\beta}{\partial x_j} \right\}\\
&
= F^\alpha(x)+a_{ij}^{\alpha\beta} (y) \frac{\partial^2 u_0^\beta}{\partial x_i\partial x_j}
+a_{ij}^{\alpha\beta}(y) \frac{\partial \chi_k^{\beta\gamma}}{\partial y_j} \cdot 
\frac{\partial^2 u_0^\gamma}{\partial x_i \partial x_k}
+\frac{\partial}{\partial y_i}
\left\{ a_{ij}^{\alpha\beta} (y) \frac{\partial u_1^\beta}{\partial x_j} \right\}.
\endaligned
$$
It follows  by an integration in $y$ over $Y$ that
\begin{equation}\label{expansion-12}
-\average_Y \left[ a_{ij}^{\alpha\beta} (y) +a_{ik}^{\alpha\gamma} (y) \frac{\partial\chi_j^{\gamma\beta}}{\partial y_k} \right]\, dy
\cdot \frac{\partial^2 u_0^\beta}{\partial x_i\partial x_j} (x) =F^\alpha (x)
\end{equation}
in $\Omega$.

\begin{definition}
{\rm
Let $\widehat{A} =(\widehat{a}_{ij}^{\alpha\beta})$, where $1\le i, j\le d$,  $1\le \alpha, \beta\le m$, and
\begin{equation}
\label{homogenized-coefficient}
\widehat{a}_{ij}^{\alpha\beta}
=\average_Y
\left[ a_{ij}^{\alpha\beta}
+a_{ik}^{\alpha\gamma}
\frac{\partial}{\partial y_k}\left( \chi_j^{\gamma\beta}\right)\right]
\, dy,
\end{equation}
and define
\begin{equation}\label{operator-L-0}
 \mathcal{L}_0=-\text{div}(\widehat{A}\nabla).
 \end{equation}
 }
 \end{definition}
 
 In summary we have formally deduced that the leading term $u_0$ in the expansion (\ref{expansion-1}) 
 depends only on $x$ and that $u_0$ is a solution of
 $\mathcal{L}_0( u_0)=F$ in $\Omega$.
 As we shall prove in the next section, the constant coefficient operator $\mathcal{L}_0$
  is indeed the homogenized operator for
$ \mathcal{L}_\varep$. 

\medskip

\noindent{\bf  Correctors and effective coefficients.}

\medskip

\noindent The constant matrix $\widehat{A}$ is called the matrix of effective or homogenized coefficients.
Because of (\ref{expansion-10}) we  call the 1-periodic matrix 
$$
\chi(y) =\big(\chi_j^\beta (y)\big)=\big(\chi_j^{\alpha\beta}(y)\big),
$$
 with $1\le j\le d$ and $1\le \alpha, \beta\le m$, 
 the matrix of (first-order) correctors for $ \mathcal{L}_\varep$.
 Define
 \begin{equation}\label{a-per}
 a_{\per} \big(\phi, \psi\big)=\average_Y a_{ij}^{\alpha\beta} (y) \frac{\partial \phi^\beta}{\partial y_j} \cdot
 \frac{\partial\psi^\alpha}{\partial y_i}\, dy
 \end{equation}
 for $\phi=(\phi^\alpha)$ and $\psi=(\psi^\alpha)$. In view of (\ref{cell-problem})
 the corrector $\chi_j^\beta$ is the unique function in $H^1_{\per}(Y; \mathbb{R}^m)$
 such that $\int_Y \chi_j^\beta=0$ and
 \begin{equation}\label{per-bilinear}
 a_{\per} \big(\chi_j^\beta, \psi\big) =-a_\per \big(P_j^\beta, \psi\big) \qquad \text{ for any } \psi\in H^1_\per (Y; \mathbb{R}^m).
 \end{equation}
 It follows from (\ref{periodic-Korn}) and (\ref{per-bilinear}) with $\psi=\chi_j^\beta$ that
 \begin{equation}\label{corrector-L-2}
 \|\chi_j^\beta\|_{H^1(Y)} \le C,
 \end{equation}
 where $C$ depends only on $\mu$.

 With the summation convention the first equation in (\ref{cell-problem}) may be written as
\begin{equation}\label{corrector-equation}
\frac{\partial}{\partial y_i}
\left[ a_{ij}^{\alpha\beta}
+a_{ik}^{\alpha\gamma}
\frac{\partial}{\partial y_k}\left( \chi_j^{\gamma\beta}\right)\right]=0
\quad
\text{ in }\brd;
\end{equation}
 i.e., $\mathcal{L}_1 (\chi_j^\beta +P_j^\beta)=0$
 in $\br^d$. It follows from the reverse H\"older estimate (\ref{reverse-Holder-1.1}) and (\ref{corrector-L-2}) that
\begin{equation}\label{corrector-L-p}
\|\nabla \chi_j^\beta\|_{L^p(Y)} \le C_0 \qquad \text{ for some } p>2,
\end{equation}
where $p$ and $C_0$ depend only on $\mu$.
This implies that  $\chi_j^\beta$ are H\"older continuous if 
$d=2$. By the classical De Giorgi - Nash theorem, $\chi_j^\beta$ is
also H\"older continuous if $m=1$ and $d\ge 3$.
If $m\ge 2$ and $d\ge 3$, we may use Sobolev imbedding and (\ref{corrector-L-p}) to obtain 
\begin{equation}\label{corrector-L-q}
\|\chi_j^\beta\|_{L^q(Y)} \le C\qquad \text{ for some } q>\frac{2d}{d-2}.
\end{equation}
We further note that by rescaling,
 \begin{equation}\label{corrector-solution}
 \mathcal{L}_\varep \left\{ P_j^\beta (x) +\varep \chi_j^\beta(x/\varep)\right\} =0
 \quad \text{ in } \br^d
 \end{equation}
for any $\varep>0$.

We now proceed to prove the inequality (\ref{periodic-Korn}), on which the existence of
correctors $(\chi_j^\beta)$ depends.
The proof uses the property of weak convergence for periodic functions.

\begin{prop}\label{periodic-prop}
Let $\{ h_\ell\}$ be a sequence of 1-periodic functions. 
Assume that $\|h_\ell\|_{L^2(Y)} \le C$ and 
$$
\average_Y h_\ell (y) dy \to c_0 \qquad\text{ as } \ell\to \infty.
$$
Let $\varep_\ell \to 0$. Then $h_\ell (x/\varep_\ell) \rightharpoonup c_0$ weakly in $L^2(\Omega)$ as
$\ell\to \infty$, where $\Omega$ is a bounded domain in $\rd$. 
In particular, if $h$ is 1-periodic and $h\in  L^2(Y)$,
 then 
 $$
 h(x/\varep)\rightharpoonup \average_Y h \quad \text{ weakly in } L^2(\Omega), \text{ as }\varep\to 0.
 $$
\end{prop}

\begin{proof}
By considering the periodic function $h_\ell -\average_Y h_\ell$, we may assume that $\int_Y h_\ell=0$ and hence $c_0=0$.
Let $u_\ell\in H_{\per}^2(Y)$ be a 1-periodic function such that $\Delta u_\ell=h_\ell$ in $Y$. Let $g_\ell=\nabla u_\ell$.
Then $h_\ell =\text{\rm div}(g_\ell)$ and $\|g_\ell\|_{L^2(Y)} \le C \| h_\ell\|_{L^2(Y)} \le C$.
Note that
$$
h_\ell ({x}/{\varep_\ell}) =\varep_\ell\,  \text{div} \big\{ g_\ell({x}/{\varep_\ell})\big\}.
$$
It follows that, if $\varphi\in C_0^1(\Omega)$,
\begin{equation}\label{1.2-3}
\int_\Omega h_\ell (x/\varep_\ell) \varphi (x) \, dx
=-\varep_\ell \int_\Omega g_\ell (x/\varep_\ell) \cdot \nabla \varphi (x) \, dx
\to 0,
\end{equation}
as $\varep_\ell \to 0$.
This is because, if $\Omega\subset B(0, R)$,
$$
\aligned
\int_\Omega | g_\ell (x/\varep_\ell)|^2 \, dx
&\le \varep_\ell^d \int_{B(0, R/\varep_\ell)} |g_\ell(y)|^2\, dy\\
& \le C\, \| g_\ell\|_{L^2(Y)}^2\\
& \le C,
\endaligned
$$
where  we have used the periodicity of $g_\ell$ for the second inequality and
$C$ depends on $R$.
Similarly,  
\begin{equation}\label{1.2-5}
\|h_\ell (x/\varep_\ell)\|_{L^2(\Omega)} \le C\, \|h_\ell\|_{L^2(Y)} \le C.
\end{equation}
In view of (\ref{1.2-3}) and (\ref{1.2-5})
 we may conclude that $h_\ell(x/\varep_\ell)\rightharpoonup 0$ weakly in $L^2(\Omega)$.
\end{proof}

\begin{lemma}\label{lemma-elliptic-w}
Suppose that $A=A(y)$ is 1-periodic and satisfies 
the ellipticity condition (\ref{weak-e-1})-(\ref{weak-e-2}).
Then the inequality (\ref{periodic-Korn}) holds 
for any $\phi\in H^1_{\per}(Y; \mathbb{R}^m)$.
\end{lemma}

\begin{proof}
Let $u_\e(x)=\e \eta(x) \phi(x/\e)$,
where $\phi$ is a 1-periodic function in $C^\infty(\brd; \mathbb{R}^m)$ and
$\eta\in C_0^\infty(\brd)$ with $\int_{\brd} \eta^2 \, dx=1$.
Since $A(x/\e)$ satisfies the condition (\ref{weak-e-2}), It follows that
\begin{equation}\label{weak-ee}
\mu \int_{\brd}
|\nabla u_\e |^2\, dx
\le \int_{\brd} A (x/\e) \nabla u_\e \cdot \nabla u_\e \, dx.
\end{equation}
We now take  limits by letting $\e\to 0$ on both sides of (\ref{weak-ee}). Using 
$$
\nabla u_\e  (x)=\eta (x) \nabla \phi (x/\e) +\e \nabla \eta(x) \cdot \phi(x/\e)
$$
and Proposition \ref{periodic-prop}, we see that as $\e\to 0$,
$$
\aligned
\int_{\brd} A(x/\e)\nabla u_\e\cdot \nabla u_\e\, dx
 & \to \average_Y A\nabla \phi \cdot \nabla \phi\, dy \int_{\brd} \eta^2\, dx, \\
 \int_{\brd} |\nabla u_\e |^2\, dx
 &\to \average_Y |\nabla \phi|^2\, dy \int_{\brd} \eta^2\, dx.
 \endaligned
 $$
 This, together with (\ref{weak-ee}), yields  (\ref{periodic-Korn}).
\end{proof}
  The following lemma gives the ellipticity for $\mathcal{L}_0$.
  
  \begin{lemma}\label{weak-L-0}
  Suppose that $A=A(y)$ is 1-periodic and satisfies  (\ref{weak-e-1})-(\ref{weak-e-2}).
Then
\begin{equation}\label{weak-eee}
\mu |\xi|^2 |\eta|^2\le \widehat{a}_{ij}^{\alpha\beta} \xi_i\xi_j\eta^\alpha \eta^\beta
\le \mu_1 |\xi|^2 |\eta|^2
\end{equation}
for any $\xi=(\xi_1, \dots, \xi_d)\in \brd$ and $\eta=(\eta^1, \dots, \eta^m)\in \mathbb{R}^m$,
where $\mu_1$ depends only on $\mu$ (and $d$, $m$).  
  \end{lemma}
  
  \begin{proof}
  The second inequality in (\ref{weak-eee}) follows readily from the energy estimate
  $\|\chi_j^\beta\|_{H^1(Y)}\le C$, where $C$ depends only on $\mu$.
  To prove the first inequality, we will show that
  \begin{equation}\label{weak-eee-1}
  \mu \int_{\brd} |\nabla \phi |^2\, dx
  \le \int_{\brd} \widehat{A}
  \nabla \phi\cdot \nabla \phi\, dx \quad
\text{ for any $\phi \in C_0^\infty(\br^d; \mathbb{R}^m)$.}
\end{equation}
As we pointed out earlier, since
$\widehat{A}$ is constant,
this is equivalent to the first inequality in (\ref{weak-eee}).

To establish (\ref{weak-eee-1}),
we fix $\phi=(\phi^\alpha)\in C_0^\infty(\brd; \mathbb{R}^m)$ and let
$$
u_\e = \phi  +\e \chi_j^\beta (x/\e)\frac{\partial \phi^\beta}{\partial x_j}
$$
in (\ref{weak-ee}) and then take the limits as $\e\to 0$.
Using 
$$
\aligned
\nabla u_\e  &=\nabla\phi +\nabla \chi_j^\beta (x/\e) \frac{\partial \phi^\beta}{\partial x_j}
+\e \chi_j^\beta (x/\e) \frac{\partial}{\partial x_j} \nabla \phi^\beta\\
&= \nabla \big(P_j^\beta +\chi_j^\beta\big) (x/\e)  \frac{\partial \phi^\beta}{\partial x_j}
+\e \chi_j^\beta (x/\e) \frac{\partial}{\partial x_j} \nabla \phi^\beta
\endaligned
$$
and Proposition \ref{periodic-prop},
we see that as $\e\to 0$,
$$
\aligned
\int_{\brd} A(x/\e)\nabla u_\e \cdot \nabla u_\e\, dx
  &\to  \average_Y A\nabla \big(P_j^\beta +\chi_j^\beta \big)
\cdot \nabla \big(P_i^\alpha +\chi_i^\alpha) \, dy
\int_{\brd} \frac{\partial \phi^\beta}{\partial x_j}\cdot 
\frac{\partial \phi^\alpha}{\partial x_i}\, dx\\
&=\int_{\brd} \widehat{A}\nabla \phi \cdot \nabla \phi\, dx.
\endaligned
$$
Observe that $\nabla u_\e \rightharpoonup \nabla \phi$ weakly in $L^2(\brd; \mathbb{R}^{m\times d})$.
It follows that
$$
\aligned
\mu \int_{\brd} |\nabla \phi|^2\, dx
 &\le \liminf_{\e\to 0}\, 
\mu \int_{\brd} |\nabla u_\e|^2\,  dx\\
&\le \lim_{\e\to 0}
\int_{\brd} A(x/\e)\nabla u_\e \cdot \nabla u_\e\, dx\\
&=
\int_{\brd} \widehat{A}\nabla \phi \cdot \nabla \phi\, dx.
\endaligned
$$
This completes the proof.
  \end{proof}
 
We end this section with a useful observation on the homogenized matrix for the adjoint operator 
$$\mathcal{L}_\e^*=-\text{div} \big(A^*(x/\e)\nabla \big).
$$

 \begin{lemma}\label{adjoint-lemma}
 Let $A^*=\big( a_{ij}^{*\alpha\beta}\big)$ denote the adjoint of $A$, 
 where $a^{*\alpha\beta}_{ij}=a_{ji}^{\beta\alpha}$.
 Then 
 $
 \widehat{A^*}=\big( \widehat{A}\big)^*.
 $
 In particular, if $A(y)$ is symmetric, i.e. $a_{ij}^{\alpha\beta} (y)=a_{ji}^{\beta\alpha}(y)$
for $1\le i,j\le d$ and $1\le \alpha, \beta\le m$,  so is $\widehat{A}$.
 \end{lemma}
 
 \begin{proof}
 Let $\chi^* (y) =\big(\chi_j^{*\beta} (y) \big)=\big(\chi_j^{*\alpha\beta} (y)\big)$ 
 denote the matrix of correctors for $\mathcal{L}_\e^*$;
i.e. $\chi_j^{*\beta} $ is the unique function in $ H_\per^1(Y;\mathbb{R}^m)$ such that $\int_Y \chi_j^{*\beta}=0$ and
\begin{equation}\label{corrector*}
 a^*_{\per} (\chi_j^{*\beta}, \psi) =-a^*_\per (P_j^\beta, \psi) \qquad \text{ for any } \psi\in H^1_\per (Y; \mathbb{R}^m),
 \end{equation}
 where
 $a_\per^* (\phi, \psi)=a_\per (\psi, \phi)$.
 Observe that by (\ref{per-bilinear}) and (\ref{corrector*}),
 \begin{equation}\label{homogenized-coefficient-1}
\aligned
\widehat{a}^{\alpha\beta}_{ij}
&=a_\per\big(P_j^\beta +\chi_j^\beta, P_i^\alpha\big) 
=a_\per \big( P_j^\beta +\chi_j^\beta, P_i^\alpha +\chi_i^{*\alpha} \big)\\
&=a_\per^* \big(P_i^\alpha +\chi_i^{*\alpha}, P_j^\beta +\chi_j^\beta\big)
=a_\per^* \big( P_i^\alpha +\chi_i^{*\alpha}, P_j^\beta\big)\\
&=a_{\per}^* \big(P_i^\alpha+\chi_i^{*\alpha}, P_j^\beta +\chi_j^{*\beta}\big)\\
&=\widehat{a}^{*\beta\alpha}_{ji},
\endaligned
\end{equation}
for $1\le \alpha, \beta\le m$ and $1\le i, j \le d$.
This shows that $\big(\widehat{A}\big)^*=\widehat{A^*}$.
 \end{proof}



\section{Homogenization of elliptic systems}\label{section-1.3}

We start with a Div-Curl Lemma.

\begin{thm}\label{Div-Curl-Lemma}
Let $\{ u_\ell\}$ and $\{v_\ell\} $ be two bounded sequences in $ L^2(\Omega; \br^d)$.
Suppose that 

\begin{enumerate}

\item

 $u_\ell \rightharpoonup u$ and $v_\ell \rightharpoonup v$ weakly in $L^2(\Omega;\br^d)$;

\item

{\rm curl}$(u_\ell)=0$ in $\Omega$ and {\rm div}$(v_\ell)\to f$ strongly in $H^{-1}(\Omega)$.
 
 \end{enumerate}
Then
$$
\int_\Omega
(u_\ell\cdot v_\ell) \, \varphi\, dx
\to 
\int_\Omega
(u\cdot v) \,\varphi\, dx
$$
as $\ell \to \infty$, for any scalar function $\varphi\in C_0^1(\Omega)$.
\end{thm}

\begin{proof}
By considering
$$
u_\ell\cdot v_\ell
=(u_\ell-u)\cdot (v_\ell-v)-u\cdot v
+u_\ell \cdot v +u\cdot v_\ell,
$$
we may assume that $u_\ell \rightharpoonup 0$, $v_\ell \rightharpoonup 0$ weakly in $L^2(\Omega; \br^d)$
and that div$(v_\ell)\to 0$ strongly in $H^{-1}(\Omega)$.
By a partition of unity we may also assume that $\varphi\in C_0^1(B)$ for some ball $B\subset \Omega$.

Since curl$(u_\ell)=0$ in $\Omega$, there exists $U_\ell  \in H^1(B)$ such that
$u_\ell=\nabla U_\ell$ in $B$ and $\int_B U_\ell \, dx=0$.
It follows that
$$
\aligned
\int_B (u_\ell \cdot v_\ell)\varphi\, dx
& =\int_B (\nabla U_\ell\cdot v_\ell) \varphi\, dx \\
&=-\langle\text{\rm div} (v_\ell), U_\ell \varphi\rangle_{H^{-1}(B)\times H_0^1(B)}
-\int_B U_\ell (v_\ell \cdot \nabla\varphi)\, dx.
\endaligned
$$
Hence,
\begin{equation}\label{1.2-0}
\Big|\int_B (u_\ell \cdot v_\ell)\varphi \, dx \Big|\, 
\le \| \text{\rm div}(v_\ell)\|_{H^{-1}(B)}
\| U_\ell \varphi\|_{H^1_0(B)}
+\| U_\ell \|_{L^2(B)} 
\| v_\ell \cdot \nabla \varphi\|_{L^2(B)}.
\end{equation}
We will show that both terms in the RHS of (\ref{1.2-0})
converge to zero.

By Poincar\'e inequality, 
$$
\| U_\ell\|_{L^2(B)} \le C\,  \| u_\ell\|_{L^2(B)} \le C.
$$ 
Thus,
$$
 \| \text{\rm div}(v_\ell)\|_{H^{-1}(B)}
\| U_\ell \varphi\|_{H^1_0(B)}
 \to 0.
$$
Using $\|U_\ell \|_{L^2(B)} \le C$, $\nabla U_\ell =u_\ell \rightharpoonup 0$ weakly in $L^2(B; \br^d)$,
 and $\int_B U_\ell =0$,
we may deduce that if $\{ U_{\ell_k}\}$ is a subsequence of $\{ U_\ell\}$ and
converges weakly in $L^2(B)$, then it must converge weakly to zero.
This implies that the full sequence $U_\ell \rightharpoonup 0$ weakly in $L^2(B)$.
It follows that $U_\ell \rightharpoonup 0$ weakly in $H^1(B)$ and therefore $U_\ell \to 0$ strongly in $L^2(B)$.
Consequently, 
$$
\| U_\ell\|_{L^2(B)} 
\| v_\ell \cdot \nabla \varphi\|_{L^2(B)}
\le C\,  \| U_\ell \|_{L^2(B)} \to 0
$$
 as $\ell\to \infty$. This completes the proof.
\end{proof}

The next theorem shows that the sequence of operators  $\{ \mathcal{L}^\ell_{\e_\ell}\}$ is G-compact in the sense of
G-convergence.

\begin{thm}\label{theorem-1.3.4}
Let $\{ A_\ell (y)\} $ be a sequence of 1-periodic matrices satisfying (\ref{weak-e-1})-(\ref{weak-e-2})
with the same constant $\mu$.
Let $F_\ell\in H^{-1}(\Omega; \mathbb{R}^m)$.
Suppose that 
\begin{equation}\label{0-1.3.4}
\mathcal{L}_{\varep_\ell}^\ell (u_\ell)=F_\ell \quad \text{ in }\Omega,
\end{equation}
 where
$\varep_\ell\to 0$, $u_\ell\in H^1(\Omega; \mathbb{R}^m)$, and
$$
\mathcal{L}^\ell_{\varep_\ell} =-\text{\rm div}  \big( A_\ell (x/\varep_\ell)\nabla\big).
$$
We further assume that 
\begin{equation}\label{a-1.3.4}
\left\{
\aligned
F_\ell &\to F \text{ in } H^{-1}(\Omega; \br^m),\\
u_\ell  & \rightharpoonup u \quad \text{ weakly in }H^1(\Omega;\br^m),\\
\widehat{A_\ell}  &\to A^0,
\endaligned
\right.
\end{equation}
 where $\widehat{A_\ell}$ denotes the matrix of effective coefficients for $A_\ell$. Then
\begin{equation}\label{b-1.3.4}
A_\ell (x/\varep_\ell)\nabla u_\ell \rightharpoonup A^0\nabla u \quad \text{ weakly in } L^2(\Omega;\br^{ m \times d}),
\end{equation}
$A^0$ is a constant matrix satisfying the ellipticity condition (\ref{weak-eee}), and $u$ is a weak solution of
\begin{equation}\label{c-1.3.4}
-\text{\rm div} (A^0\nabla u)=F \quad \text{ in }\  \Omega.
\end{equation}
\end{thm}

\begin{proof}
We first note that since $\widehat{A_\ell}\to A^0$ and
$\widehat{A_\ell}$ satisfies (\ref{weak-eee}),
so does $A^0$.
Also, (\ref{c-1.3.4}) follows directly from (\ref{0-1.3.4}) and (\ref{b-1.3.4}).
To see (\ref{b-1.3.4}),
we let $\{ u_{\ell^\prime}\}$ be a subsequence such that
$$
A_{\ell^\prime} (x/\e_{\ell^\prime})\nabla u_{\ell^\prime} \rightharpoonup H \quad
\text{ weakly in } L^2(\Omega; \mathbb{R}^{m\times d})
$$
for some $H\in L^2 (\Omega; \mathbb{R}^{m\times d})$
and show that $H= A^0 \nabla u$.
This would imply that the whole sequence $A_\ell (x/\e_\ell)\nabla u_\ell$
converges weakly to $A^0\nabla u$ in $L^2(\Omega; \mathbb{R}^{m\times d})$.

With loss of generality we assume that
\begin{equation}\label{a-1.3.4-1}
A_{\ell} (x/\e_{\ell})\nabla u_{\ell} \rightharpoonup H \quad
\text{ weakly in } L^2(\Omega; \mathbb{R}^{m\times d})
\end{equation}
for some $H=(H_i^\alpha)\in L^2 (\Omega; \mathbb{R}^{m\times d})$.
Let $\chi^*_\ell (y)=\big( \chi_{k, \ell}^{*\beta}(y)\big)$
denote the correctors associated with the matrix $A^*_\ell $, the adjoint of $A_\ell$.
Fix $1\le k\le d$, $1\le \gamma\le m$ and consider the identity
\begin{equation}\label{1.3.1}
\aligned
\int_\Omega A_\ell (x/\e_\ell)\nabla u_{\ell}
\cdot
& \nabla \Big( P_k^\gamma +\e_\ell  \chi_{k, \ell}^{*\gamma} (x/\e_\ell) \Big)\cdot \psi \, dx \\
& = \int_\Omega 
\nabla u_{\ell}
\cdot A^*_\ell (x/\e_\ell)\nabla \Big( P_k^\gamma +\e_\ell  \chi_{k, \ell}^{*\gamma} (x/\e_\ell) \Big)\cdot \psi \, dx ,\\
\endaligned
\end{equation}
where $\psi \in C_0^1(\Omega)$.
By Proposition \ref{periodic-prop},
\begin{equation}\label{1.3.2}
\aligned
\nabla \Big( P_k^\gamma +\e_\ell  \chi_{k, \ell}^{*\gamma} (x/\e_\ell) \Big)
  =& \nabla P_k^\gamma +\nabla \chi_{k, \ell}^{*\gamma} (x/\e_\ell)\\
 \rightharpoonup& 
\nabla P_k^\gamma
\endaligned
\end{equation}
weakly in $L^2(\Omega)$, where we have used the fact $\int_Y \nabla \chi^{*\gamma}_{k,\ell}\, dy=0$.
Since $\mathcal{L}^\ell_{\varep_\ell} (u_{\varep_\ell})=F_\ell$ in $\Omega$,
 in view of (\ref{a-1.3.4-1}) and (\ref{1.3.2}),
it follows by Theorem \ref{Div-Curl-Lemma} that the LHS of (\ref{1.3.1})
converges to
$$
\int_\Omega \left( H \cdot \nabla P_k^\gamma\right)\psi\, dx
=\int_\Omega H_k^\gamma \psi\, dx.
$$
Similarly, note that $\nabla u_{\ell} \rightharpoonup \nabla u$ and
$$
\aligned
A^*_\ell (x/\e_\ell)\nabla \Big( P_k^\gamma +\e_\ell  \chi_{k, \ell}^{*\gamma} (x/\e_\ell) \Big)
\rightharpoonup
& \lim_{\ell \to \infty}
\average_Y
A^*_\ell  \Big( \nabla P_k^\gamma +  \nabla\chi_{k, \ell}^{*\gamma}  \Big)\, dy\\
=& \lim_{\ell \to \infty} \widehat{A^*_\ell} \, \nabla P_k^\gamma\\
= & (A^0)^* \, \nabla P_k^\gamma
\endaligned
$$
weakly in $L^2(\Omega)$,
where we have used Proposition \ref{periodic-prop} as well as Lemma \ref{adjoint-lemma}.
Since 
$$
\mathcal{L}^{\ell *}_{\e_\ell} \big\{ P_k^\gamma +\varep_\ell \chi_k^{*\gamma} (x/\e_\ell)\big\}=0
\quad \text{ in } \br^d,
$$
we may use Theorem \ref{Div-Curl-Lemma} again to claim that the RHS of (\ref{1.3.1})
converges to
$$
\int_\Omega \left(\nabla u \cdot (A^0)^*\nabla P_k^\gamma \right) \psi\, dx.
$$
As a result, since $\psi\in C_0^1(\Omega)$ is arbitrary, it follows that
\begin{equation}\label{1.3.4}
H_k^\gamma =\nabla u \cdot (A^0)^* \nabla P_k^\gamma
=A^0\nabla u \cdot \nabla P_k^\gamma \quad \text{ in } \Omega.
\end{equation}
This shows that $H=A^0\nabla u$ and completes the proof.
\end{proof}

We now use Theorem \ref{theorem-1.3.4} to establish the qualitative homogenization 
of the Dirichlet and Neumann problems for $\mathcal{L}_\e$.
The proof only uses a special case of Theorem \ref{theorem-1.3.4}, where $A_\ell =A$ is fixed.
The general case is essential  in a compactness argument we will use in Chapters 
\ref{chapter-2} and \ref{chapter-3} for regularity estimates that are uniform in $\e>0$.

\medskip

\noindent{\bf Homogenization of Dirichlet Problem (\ref{Dirichlet-problem-1.1}).}

\medskip

\noindent Assume that $A$ satisfies the elliptic condition (\ref{weak-e-1})-(\ref{weak-e-2})
and is 1-periodic.
Let $F\in L^2(\Omega;\br^m)$, $G\in L^2(\Omega; \br^{m\times d})$
 and $f\in H^{1/2}(\partial\Omega;\br^m)$.
By Theorem \ref{theorem-1.1-2} there exists a unique $u_\varep \in H^1(\Omega;\br^m)$ such
that 
$$
\mathcal{L}_\varep (u_\varep)=F +\text{div}(G) \quad \text{ in } \Omega
\quad \text{ and  } \quad u_\varep =f \quad \text{ on } \partial\Omega
$$
(the boundary data is taken in the sense of trace). Furthermore, the solution $u_\varep$
satisfies
$$
\| u_\varep\|_{H^1(\Omega)} \le C\Big \{ \| F\|_{L^2(\Omega)} 
+\|G\|_{L^2(\Omega)}+\| f\|_{H^{1/2}(\partial\Omega)}\Big\},
$$
where $C$ depends only on $\mu$ and $\Omega$.

Let $\{ u_{\varep^\prime}\}$ be a subsequence of $\{ u_{\varep}\}$ such that as $\varep^\prime\to 0$,
$
 u_{\varep^\prime}  \rightharpoonup u 
\text{ weakly in } H^1(\Omega;\br^m)
$
for some $u\in H^1(\Omega;\br^m)$.
It follows readily from Theorem \ref{theorem-1.3.4} that
$A(x/\e^\prime)\nabla u_{\e^\prime} \rightharpoonup \widehat{A}\nabla u$ and
$
\mathcal{L}_0 (u)=F +\text{div} (G)  \text{ in } \Omega.
$
Since $f\in H^{1/2}(\partial\Omega; \br^m)$, there exists $\Phi\in H^1(\Omega; \br^m)$ such that
$\Phi=f$ on $\partial\Omega$. Using the facts that $u_{\e^\prime} -\Phi \rightharpoonup u-\Phi$ weakly in 
$H^1(\Omega; \br^m)$ and $u_{\e^\prime}-\Phi \in H^1_0(\Omega; \br^m)$,
we see that $u-\Phi \in H^1_0(\Omega; \br^m)$.
Hence, $u=f$ on $\partial\Omega$.
Consequently, $u$ is the unique weak solution to the Dirichlet problem,
$$
\mathcal{L}_0 (u_0)=F +\text{div}(G) \quad \text{ in } \Omega
\quad \text{ and  } \quad u_0 =f \quad \text{ on } \partial\Omega.
$$
Since $\{ u_\varep\}$ is  bounded in $H^1(\Omega; \br^m)$ and thus
any sequence $\{ u_{\varep_\ell}\}$ with $\varep_\ell \to 0$
contains a subsequence that converges weakly in $H^1(\Omega; \br^m)$,
one may conclude that as $\e\to 0$,
\begin{equation}\label{H-conv}
\left\{
\aligned
A(x/\varep)\nabla u_\varep  &\rightharpoonup \widehat{A}\nabla u_0 &\quad &
\text{ weakly in } L^2(\Omega; \mathbb{R}^{m\times d}),\\
u_\varep  &\rightharpoonup u_0 &\quad &\text{ weakly in  }H^1(\Omega;\br^m).
\endaligned
\right.
\end{equation}
By the compactness of the embedding $H^1(\Omega; \br^m)\subset L^2(\Omega; \br^m)$,
we also obtain 
\begin{equation}
u_\e \to u_0 \quad \text{ strongly in } L^2(\Omega; \br^m).
\end{equation}

\medskip

\noindent{\bf Homogenization of  Neumann Problem (\ref{Neumann-problem-1.1}).}

\medskip
\noindent Assume that $A$ satisfies the Legendre  ellipticity condition (\ref{s-ellipticity})
and is 1-periodic. To establish the homogenization theorem for the Neumann problem,
we first show that the homogenized matrix $\widehat{A}$ also satisfies the Legendre condition.
This ensures that the corresponding Neumann problem for $\mathcal{L}_0$ is well posed.

\begin{lemma}\label{s-homo}
Suppose that $A=A(y)$ is 1-periodic and satisfies the Legendre  condition (\ref{s-ellipticity}).
Then $\widehat{A}$ also satisfies the Legendre condition. In fact,
\begin{equation}\label{s-ee}
\mu |\xi|^2 \le \widehat{a}_{ij}^{\alpha\beta} \xi_i^\alpha \xi_j^\beta 
\le \mu_1 |\xi|^2
\end{equation}
for any $\xi=(\xi_i^\alpha)\in \mathbb{R}^{m\times d}$, where
$\mu_1>0$ depends only on $\mu$ (and $d, m$).
\end{lemma}

\begin{proof}
The proof for the second inequality in (\ref{s-ee}) is the same as in the proof
of Lemma \ref{weak-L-0}.
To see the first, we fix $\xi =(\xi_i^\alpha)\in \mathbb{R}^{m\times d}$ and
let $\phi=\xi_i^\alpha P_i^\alpha$, $\psi=\xi_i^\alpha \chi_i^\alpha$.
Observe that by (\ref{s-ellipticity}),
$$
\aligned
\widehat{a}_{ij}^{\alpha\beta} \xi_i^\alpha \xi_j^\beta
&= a_{\per} (\phi +\psi, \phi+\psi)\\
&\ge \mu \average_Y  |\nabla \phi +\nabla \psi|^2\, dy\\
&=\mu \average_Y |\nabla \phi|^2\, dy +\mu \average_Y |\nabla \psi|^2\, dy,
\endaligned
$$
where we have also used the fact $\int_Y \nabla \chi_i^\alpha\, dy=0$.
It follows that
$$
\aligned
\widehat{a}_{ij}^{\alpha\beta} \xi_i^\alpha \xi_j^\beta
&\ge \mu \average_Y |\nabla \phi|^2\, dy\\
&=\mu |\xi|^2. 
\endaligned
$$
This finishes the proof.
\end{proof}

Let $F\in L^2(\Omega;\br^m)$, $G\in L^2(\Omega; \br^{m\times d})$,
and $g\in H^{-1/2}(\partial\Omega;\br^m)$, 
the dual of $H^{1/2}(\partial\Omega;\br^m)$.
Assume that $F$, $G$ and $g$ satisfy the compatibility condition (\ref{s-comp}).
By Theorem \ref{s-NP-theorem} the Neumann problem (\ref{Neumann-problem-1.1})
has a unique (up to a constant in $\mathbb{R}^m$) solution.
Furthermore, if $\int_\Omega u_\varep\, dx =0$, by (\ref{estimate-s-NP}) and Poincar\'e 
inequality,
$$
\| u_\varep\|_{H^1(\Omega)}
\le C\, \Big\{ \| F\|_{L^2(\Omega)} +\| G\|_{L^2(\Omega)}
+\| g\|_{H^{-1/2}(\partial\Omega)}\Big\},
$$
where $C$ depends only on  $\mu$ and $\Omega$.
Let $\{ u_{\varep^\prime}\}$ be a subsequence of $\{u_\varep\}$ such that
$u_{\varep^\prime} \rightharpoonup u_0$ weakly in $H^1(\Omega;\br^m)$
for some $u_0\in H^1(\Omega; \br^m)$.
It follows from Theorem \ref{theorem-1.3.4} that
$$
A(x/\varep^\prime)\nabla u_{\varep^\prime}
\rightharpoonup \widehat{A}\nabla u_0
\quad \text{ weakly in } L^2(\Omega;\br^{ m\times d }).
$$
By taking limits in  (\ref{weak-solution-Neumann-1.1}) we see that $u_0$ is a  weak solution 
to the Neumann problem:
\begin{equation}\label{h-NP-1.3}
\mathcal{L}_0 (u_0) =F +\text{\rm div}(G) \quad \text{in } \Omega
\quad \text{ and } \quad \frac{\partial u_0}{\partial\nu_0} =g -n\cdot G\quad \text{ on } \partial\Omega,
\end{equation}
and that $\int_\Omega u_0\, dx=0$,
where 
\begin{equation}\label{h-conormal}
\left(\frac{\partial u_0}{\partial \nu_0}\right)^\alpha
=n_i \widehat{a}_{ij}^{\alpha\beta} \frac{\partial u_0^\beta}{\partial x_j}
\end{equation}
is the conormal derivative associated with the operator $\mathcal{L}_0$.
Since such $u_0$ is unique,
we may conclude that as 
$\e\to 0$, $u_\varep \rightharpoonup u_0$ weakly in $H^1(\Omega;\br^m)$ and thus strongly
 in $L^2(\Omega;\br^m)$. We also obtain 
 $A(x/\varep)\nabla u_{\varep}
\rightharpoonup \widehat{A}\nabla u_0
\text{ weakly in } L^2(\Omega;\br^{ m\times d })$.



\section{Elliptic systems of linear elasticity}\label{section-1.4}

In this section we consider the elliptic system of linear elasticity
$\mathcal{L}_\e =-\text{div} \big(A(x/\e)\nabla \big)$. We assume that
the coefficient matrix
$A(y)=\big(a_{ij}^{\alpha\beta} (y)\big)$, with $1\le i, j, \alpha, \beta \le d$, is 1-periodic and
satisfies  the elasticity condition, denoted by $A\in E(\kappa_1, \kappa_2)$,
\begin{equation}\label{ellipticity}
\aligned
& a_{ij}^{\alpha\beta} (y) =a_{ji}^{\beta\alpha} (y)=a_{\alpha j}^{i\beta} (y),\\
& \kappa_1 |\xi|^2 \le a_{ij}^{\alpha\beta} (y) \xi_i^\alpha\xi_j^\beta
\le \kappa_2 |\xi|^2 
\endaligned
\end{equation}
for a.e. ~$y\in \br^d$  and for any {\it symmetric} matrix
$ \xi=(\xi_i^\alpha)\in \br^{ d \times d }$,
where  $\kappa_1, \kappa_2 $ are positive constants.

\begin{lemma}\label{Korn-1-thm}
Let $\Omega$ be a bounded  domain in $\br^d$.
Then
\begin{equation}\label{Korn-1}
\sqrt{2}\,  \|\nabla u\|_{L^2(\Omega)}
\le \|\nabla u + (\nabla u)^T\|_{L^2(\Omega)}
\end{equation}
for any $u\in H^1_0(\Omega; \br^d)$,
where $(\nabla u)^T$ denotes the transpose of $\nabla u$. 
\end{lemma}

\begin{proof}
By a density argument,
to prove (\ref{Korn-1}), which is called the first Korn inequality, 
we may assume that $u\in C_0^\infty(\Omega; \br^d)$.
This allows us to use integration by parts to obtain 
$$
\aligned
\int_\Omega |\nabla u +(\nabla u)^T|^2\, dx
& =\int_\Omega \left(\frac{\partial u^\alpha}{\partial x_i}
+\frac{\partial u^i}{\partial x_\alpha}\right)\left(\frac{\partial u^\alpha}{\partial x_i}
+\frac{\partial u^i}{\partial x_\alpha}\right)\, dx\\
&= 2\int_\Omega |\nabla u|^2\, dx
+2 \int_\Omega \frac{\partial u^\alpha}{\partial x_i} \frac{\partial u^i}{\partial x_\alpha}\, dx\\
&=2\int_\Omega |\nabla u|^2\, dx
-2\int_\Omega u^\alpha \frac{\partial}{\partial x_\alpha} \big(\text{\rm div} (u)\big)\, dx\\
&=2\int_\Omega |\nabla u|^2\, dx
+2 \int_\Omega |\text{\rm div} (u)|^2\, dx\\
& \ge 2\int_\Omega |\nabla u|^2\, dx,
\endaligned
$$
from which the inequality (\ref{Korn-1}) follows.
\end{proof}

\begin{lemma}\label{elasticity-prop}
Suppose $A=\big(a_{ij}^{\alpha\beta}\big)\in E(\kappa_1, \kappa_2)$. Then
\begin{equation}\label{elasticity}
\frac{\kappa_1}{4} |\xi + \xi^T |^2 \le
a_{ij}^{\alpha\beta} \xi_i^\alpha\xi_j^\beta \le
\frac{\kappa_2}{4} |\xi + \xi^T |^2
\quad
\text{ for any } \xi =(\xi_i^\alpha)\in \br^{d\times d}.
\end{equation}
\end{lemma}

\begin{proof}
Note that by the symmetry conditions in (\ref{ellipticity}),
\begin{equation}\label{symmetry}
a_{ij}^{\alpha\beta}
=a_{ji}^{\beta\alpha}
=a_{\alpha j}^{i\beta}
=a_{j\alpha}^{\beta i}
=a_{\beta\alpha}^{ji}
=a_{\alpha\beta}^{ij}
=a_{i\beta}^{\alpha j}.
\end{equation}
It follows that for any $\xi=(\xi_i^\alpha) \in \mathbb{R}^{d\times d}$,
$$
a_{ij}^{\alpha\beta}
\xi_i^\alpha\xi_j^\beta
=\frac{1}{4} a_{ij}^{\alpha\beta}
(\xi_i^\alpha +\xi_\alpha^i) (\xi_j^\beta +\xi_\beta^j),
$$
from which (\ref{elasticity}) follows readily from (\ref{ellipticity}).
\end{proof}

It follows from Lemmas \ref{Korn-1-thm} and \ref{elasticity-prop} that
$$
\aligned \int_{\brd} A\nabla u\cdot \nabla u\, dx
& \ge \frac{\kappa_1}{4}\int_{\brd} |\nabla u +(\nabla u)^T|^2\, dx\\
&\ge \frac{\kappa_1}{2} \int_{\brd} |\nabla u|^2\, dx,
\endaligned
$$
for any $u\in C_0^1(\brd; \brd)$.
This shows that the elasticity condition (\ref{ellipticity}) implies
the ellipticity condition (\ref{weak-e-1})-(\ref{weak-e-2}) for some
$\mu>0$ depending only on $\kappa_1$ and $\kappa_2$.
Consequently, all results proved in previous sections under the condition (\ref{weak-e-1})-(\ref{weak-e-2})
hold for the elasticity system. 
In particular, the matrix of homogenized coefficients may be defined 
and satisfies the ellipticity condition (\ref{weak-eee}).
However, a stronger result can be proved.

 \begin{thm}\label{ellipticity-theorem}
Suppose that $A=\big(a_{ij}^{\alpha\beta}\big)\in E(\kappa_1, \kappa_2)$ and is 1-periodic.
Let $\widehat{A}=\big(\widehat{a}_{ij}^{\alpha\beta}\big)$ be its matrix of effective coefficients. Then 
$\widehat{A}\in E(\kappa_1, \kappa_2)$.
\end{thm}

\begin{proof}
 Let the bilinear form $a_{\per} (\cdot, \cdot)$ be defined by (\ref{a-per}).
Observe that
\begin{equation}\label{corrector-form}
\aligned
\widehat{a}_{ij}^{\alpha\beta} & =a_\per \big(P_j^\beta +\chi_j^\beta, P_i^\alpha \big)\\
&=a_\per \big(P_j^\beta +\chi_j^\beta, P_i^\alpha \pm \chi_i^\alpha\big),
\endaligned
\end{equation}
where  we have used (\ref{per-bilinear})
and $\chi_i^\alpha\in H^1_{\per}(Y; \brd)$
for the second equality. 
Since $a_{ij}^{\alpha\beta}=a_{ji}^{\beta\alpha}$,
we have $a_{\per} (\phi, \psi)=a_{\per} (\psi, \phi)$.
It follows that $\widehat{a}_{ij}^{\alpha\beta}=\widehat{a}_{ji}^{\beta\alpha}$.
Also, using $a_{ij}^{\alpha\beta}=a_{\alpha j}^{i\beta}$ and (\ref{homogenized-coefficient}),
we obtain $\widehat{a}_{ij}^{\alpha\beta}=\widehat{a}_{\alpha j}^{i\beta}$.

Let $\xi=(\xi_i^\alpha)\in \br^{ d \times d }$ be a symmetric matrix.
Let $\phi= \xi_j^\beta P_j^\beta$ and $\psi= \xi_j^\beta \chi_j^\beta$.
It follows from (\ref{corrector-form}) and (\ref{elasticity}) that
$$
\aligned
\widehat{a}_{ij}^{\alpha\beta} \xi_i^\alpha\xi_j^\beta
&
=a_\per \big( \phi+\psi,  \phi+\psi \big)\\
 &
 \ge\frac{\kappa_1}{4}
 \average_Y  |\nabla \phi +\nabla \psi + (\nabla \phi)^T + (\nabla \psi)^T|^2\, dy\\
 &=  \frac{\kappa_1}{4}
  \average_Y 
  |\nabla \phi +(\nabla \phi)^T|^2\, dy
  +\frac{\kappa_1}{4}
  \average_Y 
  |\nabla \psi +(\nabla \psi)^T|^2\, dy,
  \endaligned
 $$
 where we have used the observation $\int_Y \nabla \chi_j^\beta \, dy=0$ for the last step.
 Since $\nabla \phi=\xi=\xi^T$,
 this implies that
 \begin{equation}\label{1.2.1-1}
 \aligned
 \widehat{a}_{ij}^{\alpha\beta} \xi_i^\alpha\xi_j^\beta
 &\ge \frac{\kappa_1}{4}  \average_Y 
  |\nabla \phi +(\nabla \phi)^T|^2\, dy \\
&=\frac{\kappa_1}{4} |\xi +\xi^T|^2\\
&=\kappa_1 |\xi|^2.
\endaligned
\end{equation}
Also, note that by (\ref{corrector-form}),
$$
\aligned
\widehat{a}_{ij}^{\alpha\beta} \xi_i^\alpha\xi_j^\beta
&
=a_\per \big( \phi+\psi,  \phi-\psi \big)\\
&=a_\per (\phi, \phi)- a (\psi, \psi)\\
&\le a_\per (\phi, \phi)\\
& \le \kappa_2 |\xi|^2,
 \endaligned
$$
where we have used the fact $a_\per (\phi, \psi)=a_\per (\psi, \phi)$ and
$a_\per (\psi, \psi)\ge 0$.
\end{proof}

As we pointed out earlier,
the results for the Dirichlet problem in Section \ref{section-1.3} hold for the elasticity operator.
Additional work is needed for the Neumann problem (\ref{Neumann-problem-1.1}),
 as the elasticity condition (\ref{ellipticity})
does not imply the Legendre ellipticity condition.

    Let
\begin{equation}\label{rigid}
\mathcal{R}=\Big\{\phi=Bx +b: \ B\in \br^{d\times d} 
\text{ is skew-symmetric and } b \in \br^d \Big\}
\end{equation}
denote the space of rigid displacements, with
$$
\text{\rm dim} (\mathcal{R})=\frac{d (d+1)}{2}.
$$
Using the symmetric condition $a_{ij}^{\alpha\beta}=a_{\alpha j}^{i\beta}$,
we see that $A(x/\varep)\nabla u\cdot \nabla \phi=0$ for any $\phi\in \mathcal{R}$.
Consequently,  the existence of solutions of (\ref{Neumann-problem-1.1}) implies that
\begin{equation}\label{compatibility}
\int_\Omega F\cdot \phi\, dx
-\int_\Omega G \cdot \nabla \phi\, dx +
\langle g, \phi\rangle_{H^{-1/2}(\partial \Omega)\times H^{1/2}(\partial\Omega)}=0
\end{equation}
for any  $\phi\in \mathcal{R}$.

\begin{thm}\label{theorem-1.1-3}
Let $\Omega$ be a bounded Lipschitz domain in $\brd$ and $A\in E(\kappa_1, \kappa_2)$.
Assume that $F\in L^2(\Omega;\br^d)$, $G\in L^2(\Omega; \br^{d\times d})$ 
 and $g\in H^{-1/2}(\partial\Omega;\br^d)$ satisfy the
compatibility condition (\ref{compatibility}). Then
the Neumann problem (\ref{Neumann-problem-1.1})
has a weak solution $u_\varep$,
unique up to an element of $\mathcal{R}$,
 in $ H^1(\Omega;\br^d)$. Moreover, the 
solution satisfies the energy estimate,
\begin{equation}\label{Neumann-estimate-1.1}
\|\nabla u_\varep\|_{L^2(\Omega)}
\le C \left\{
\| F\|_{L^2(\Omega)}
+\| G\|_{L^2(\Omega)}
+\| g\|_{H^{-1/2}(\partial\Omega)} \right\},
\end{equation}
where $C$ depends only on $\kappa_1, \kappa_2$ and $\Omega$.
\end{thm}

\begin{proof}
This again follows from the Lax-Milgram Theorem by considering the bilinear form
(\ref{bilinear-form}) on the Hilbert space $H^1(\Omega; \mathbb{R}^d)/\mathcal{R}$.
To prove $B[u, v]$ is coercive, one applies  the second Korn inequality
\begin{equation}\label{Korn-2}
\int_\Omega |\nabla u|^2\, dx
\le C \int_\Omega |\nabla u + (\nabla u)^T|^2\, dx
\end{equation}
for any $u\in H^1(\Omega; \br^d)$ with the property that
$u\perp \mathcal{R}$ in $H^1(\Omega; \br^d)$ or in $L^2(\Omega; \br^d)$.
We refer the reader to \cite{OSY-1992} for a proof of (\ref{Korn-2}).
\end{proof}
 
 \begin{thm}\label{s-NP-h-thm}
Assume that $A$ is 1-periodic and satisfies the elasticity condition (\ref{ellipticity}).
Let $u_\varep \in H^1(\Omega; \mathbb{R}^d)$ be the weak solution to the Neumann problem (\ref{Neumann-problem-1.1})
with $\int_\Omega u_\e\cdot \phi=0$ for any $\phi \in \mathcal{R}$, given by Theorem \ref{theorem-1.1-3}.
Then 
\begin{equation}\label{weak-e-c-n}
\left\{
\aligned
u_\e   & \rightharpoonup u_0 & \ & \text{ weakly in } H^1(\Omega; \mathbb{R}^d),\\ 
A(x/\e)\nabla u_\e  & \rightharpoonup \widehat{A}\nabla u_0
& \ & \text{ weakly in } L^2(\Omega; \mathbb{R}^{d\times d}),
\endaligned
\right.
\end{equation}
where $u_0$ is the unique weak solution to the Neumann  problem,
\begin{equation}\label{s-e-NP-h}
\mathcal{L}_0 (u_0)= F +\text{\rm div} (G) \quad \text{ in } \Omega
\quad \text{ and } \quad 
\frac{\partial u_0}{\partial \nu_0} =g-n\cdot G \quad \text{ on } \partial\Omega,
\end{equation}
with $\int_\Omega u_0 \cdot \phi=0$ for any $\phi\in \mathcal{R}$.
\end{thm}

\begin{proof}
The proof is similar to that for the  Neumann problem under the Legendre ellipticity condition.
We point out that  Theorem \ref{ellipticity-theorem} is needed for the existence and uniqueness 
of the Neumann problem (\ref{s-e-NP-h}) for $\mathcal{L}_0$.
\end{proof}

We end this section with some observations on systems of linear elasticity.

Let $A(y)=\big(a_{ij}^{\alpha\beta} (y)\big)\in E(\kappa_1, \kappa_2)$.
Define
\begin{equation}\label{mod-a-0}
\widetilde{a}_{ij}^{\alpha\beta} (y)
=a_{ij}^{\alpha\beta} (y) +\mu \delta_{i\alpha}\delta_{j\beta} -\mu \delta_{i\beta}\delta_{j\alpha},
\end{equation}
where $0<\mu\le \kappa_1/2$.
The following proposition shows that $\widetilde{A}=(\widetilde{a}_{ij}^{\alpha\beta})$
is symmetric and satisfies  the Legendre ellipticity condition.

\begin{prop}\label{mod-a-1}
Let $A\in E(\kappa_1, \kappa_2)$ and
$\widetilde{a}_{ij}^{\alpha\beta}$ be defined by (\ref{mod-a-0}).
Then $\widetilde{a}_{ij}^{\alpha\beta} =\widetilde{a}_{ji}^{\beta\alpha}$, and
\begin{equation}\label{mod-a-2}
\widetilde{a}_{ij}^{\alpha\beta} \xi_i^\alpha \xi_j^{\beta} \ge \mu |\xi|^2
\end{equation}
for any $\xi=(\xi_i^\alpha) \in \mathbb{R}^{d\times d}$.
\end{prop}

\begin{proof}
The symmetry property is obvious . To see (\ref{mod-a-2}),
we let $\xi =(\xi_i^\alpha)\in \mathbb{R}^{d\times d}$ and
recall that by (\ref{elasticity}),
$$
a_{ij}^{\alpha\beta} \xi_i^\alpha\xi_j^\beta 
\ge \frac{\kappa_1}{4} |\xi +\xi^T|^2.
$$
It follows that
$$
\aligned
\widetilde{a}_{ij}^{\alpha\beta} \xi_i^\alpha\xi_j^\beta
&\ge \frac{\kappa_1}{4} |\xi +\xi^T|^2 +\mu |\xi|^2 -\mu \xi_i^j \xi_j^i\\
&=\frac{\kappa_1}{2} \big( |\xi|^2 +\xi_i^j \xi_j^i \big) 
+\mu |\xi|^2 -\mu \xi_i^j \xi_j^i\\
&=\mu |\xi|^2 +\frac12 \left(\frac{\kappa_1}{2} -\mu \right) |\xi +\xi^T|^2\\
&\ge\mu |\xi|^2,
\endaligned
$$
where we have used the assumption  $\mu\le \kappa_1/2$ for the last step.
\end{proof}

\begin{prop}\label{mod-a-3}
Let $\widetilde{A}(y)=(\widetilde{a}_{ij}^{\alpha\beta})$ be defined by (\ref{mod-a-0}).
Then \begin{equation}\label{mod-a-4}
\int_\Omega A(x/\e)\nabla u \cdot \nabla \varphi \, dx
=\int_\Omega \widetilde{A}(x/\e) \nabla u \cdot \nabla \varphi\, dx
\end{equation}
for any $u \in H^1_{\loc} (\Omega; \brd)$ and $\varphi\in C_0^\infty(\Omega; \brd)$.
\end{prop}

\begin{proof}
Let $u\in C^\infty(\Omega; \brd)$ and $\varphi\in C_0^\infty(\Omega; \brd)$.
Note that
\begin{equation}\label{mod-a-5}
\aligned
\int_\Omega \left( \widetilde{A}(x/\e) -A(x/\e)\right)\nabla u \cdot \nabla \varphi\, dx
&=\mu \int_\Omega \left( \delta_{i\alpha} \delta_{j\beta} -\delta_{i\beta} \delta_{j\alpha} \right)
\frac{\partial u^\beta}{\partial x_j} \cdot \frac{\partial \varphi^\alpha}{\partial x_i} \, dx\\
&=\mu \int_\Omega \text{\rm div} (u) \cdot \text{div} (\varphi) \, dx
-\mu\int_\Omega \frac{\partial u^\beta}{\partial x_\alpha} \cdot \frac{\partial \varphi^\alpha}{\partial x_\beta}\, dx\\
&=0,
\endaligned
\end{equation}
where we have used integration by parts for the last step.
By a density argument one may deduce that (\ref{mod-a-5}) continues to hold
for any $u\in H^1_{\loc}(\Omega; \brd)$ and $\varphi\in C_0^\infty(\Omega; \brd)$.
\end{proof}

Let $\widetilde{\mathcal{L}}_\e =-\text{\rm div} (\widetilde{A}(x/\e)\nabla)$.
It follows from Proposition \ref{mod-a-3} that if $u_\e \in H^1_{\loc}(\Omega; \brd)$, then 
\begin{equation}\label{equivalency}
\mathcal{L}_\e (u_\e)=F \quad \text{ in } \Omega \quad
\text{ if and only if } \quad \widetilde{\mathcal{L}}_\e (u_\e)=F \quad \text{ in } \Omega,
\end{equation}
where $F\in \big(C_0^\infty(\Omega; \brd)\big)^\prime$ is a distribution.
In view of Proposition \ref{mod-a-1} this allows us to treat the system of linear elasticity as a special case
of elliptic systems satisfying the Legendre condition and the symmetry condition.
Indeed, the approach works well for interior regularity estimates as well as for boundary estimates with
the Dirichlet condition.
However, we should point out that since the Neumann boundary condition depends on the 
coefficient matrix, the re-writing of the system of elasticity changes the Neumann problem.
More precisely, let $\frac{\partial u_\e}{\partial\widetilde{\nu_\e}}$ denote the conormal derivative 
associated with $\widetilde{\mathcal{L}}_\e$, then
$$
\left(\frac{\partial u_\e}{\partial\widetilde{\nu_\e}}\right)^\alpha
=\left(\frac{\partial u_\e}{\partial{\nu_\e}}\right)^\alpha
+\mu n_\alpha \text{\rm div} (u_\e) -\mu n_\beta \frac{\partial u_\e^\beta}{\partial x_\alpha}.
$$

\section{Notes}

Material in Section \ref{section-1.1} is standard for second-order linear elliptic systems in divergence form
with bounded measurable coefficients. 

The formal asymptotic expansions as well as other results in Section \ref{section-1.2}
may be found in the classical book \cite{BLP-1978}.

Much of the material in Sections \ref{section-1.3} and \ref{section-1.4}  is more or less
well known and may be found in books  \cite{JKO-1993, OSY-1992}.



\chapter{Convergence Rates, Part I}\label{chapter-C}
\setcounter{equation}{0}

Let $\mathcal{L}_\varep =-\text{\rm div} (A(x/\varep)\nabla)$ for $\varep>0$,
where $A(y)=\big(a_{ij}^{\alpha\beta} (y)\big)$ is 1-periodic and satisfies certain ellipticity condition.
Let $\mathcal{L}_0=-\text{\rm div} (\widehat{A}\nabla)$,
where $\widehat{A}=\big(\widehat{a}_{ij}^{\alpha\beta}\big)$ denotes the matrix of effective coefficients, given by  
(\ref{homogenized-coefficient}).
For $F\in L^2(\Omega; \br^m)$ and $\varep\ge 0$, consider the Dirichlet problem 
\begin{equation}\label{D-C}
\left\{
\aligned
\mathcal{L}_\varep (u_\varep) & =F & \quad  &\text{ in } \Omega,\\
u_\varep & = f & \quad & \text{ on } \partial\Omega,
\endaligned
\right.
\end{equation}
and the Neumann problem 
\begin{equation}\label{N-C}
\left\{
\aligned
\mathcal{L}_\varep (u_\varep) & =F &  &  \text{ in } \Omega,\\
\frac{\partial u_\varep}{\partial \nu_\varep} & = g &  & \text{ on } \partial\Omega,\\
 u_\varep   \perp \br^m   & \text{ in }  L^2(\Omega; \br^m),
\endaligned
\right.
\end{equation}
where $f \in H^1(\partial\Omega; \br^m)$, and
$g\in L^2(\partial\Omega; \br^m)$.
It is shown in Section \ref{section-1.3} that as $\varep\to 0$,
$u_\varep$ converges to $u_0$ weakly in $H^1(\Omega; \br^m)$ and
strongly in $L^2(\Omega; \br^m)$.
In this chapter we investigate the problem of convergence rates in $H^1$ and $L^2$.

In Section \ref{section-c-1} we introduce the flux correctors and
study the properties of an $\e$-smoothing operator $S_\e$.
Sections \ref{section-c-2} and \ref{section-c-3} are devoted to  error estimates
of two-scale expansions in $H^1$ for Dirichlet and Neumann problems in a Lipschitz domain $\Omega$, respectively.
We will show that 
\begin{equation}\label{r-1/2}
\| u_\varep -u_0 -\varep \chi(x/\varep)\eta_\varep S_\varep^2(\nabla u_0)\|_{H^1(\Omega)}
\le C\sqrt{\e} \| u_0\|_{H^2(\Omega)},
\end{equation}
where $\eta_\varep$ is a cut-off function satisfying (\ref{eta-e}).
Moreover, if $A$ satisfies the symmetry condition $A^*=A$, we obtain 
\begin{equation}\label{rate-1/2-0}
 \| u_\varep -u_0 -\varep \chi(x/\varep)\eta_\varep S_\varep^2(\nabla u_0)\|_{H^1(\Omega)}
\le\left\{\aligned
&  C \sqrt{\varep}\, \Big\{  \| F\|_{L^q(\Omega)}  +\| f\|_{H^1(\partial\Omega)} \Big\}
 & & \text{ for (\ref{D-C})},\\
& C \sqrt{\varep}\, \Big\{  \| F\|_{L^q(\Omega)} + 
\| g\|_{L^2(\partial\Omega)} \Big\}
& & \text{ for (\ref{N-C})},
\endaligned
\right.
\end{equation}
for $0<\e<1$,
where $q=\frac{2d}{d+1}$.
The constant $C$ in (\ref{r-1/2})-(\ref{rate-1/2-0}) depends only on the ellipticity constant $\mu$ and $\Omega$.

The $O(\sqrt{\varep})$ rate in $H^1(\Omega)$ given by (\ref{r-1/2}) and (\ref{rate-1/2-0})
is more or less sharp.
 Note that the error estimates  (\ref{r-1/2})-(\ref{rate-1/2-0})
also imply the $O(\sqrt{\varep})$ rate for $u_\e -u_0$ in $L^2(\Omega)$.
However, this is not sharp.
In fact, it will be proved in Sections \ref{section-c-4} and \ref{section-c-5} that
if $\Omega$ is a bounded Lipschitz domain and $A^*=A$, the scaling-invariant estimate
\begin{equation}\label{c-L-2}
\| u_\varep -u_0\|_{L^p(\Omega)} \le C\varep\| u_0\|_{W^{2,q}(\Omega)},
\end{equation}
holds for $p=\frac{2d}{d-1}$ and $q=p^\prime=\frac{2d}{d+1}$,
where $C$ depends only on $\mu$ and $\Omega$.
Without the symmetry condition, it is shown that
\begin{equation}\label{c-100}
\| u_\e -u_0\|_{L^2(\Omega)} \le C \e \| u_0\|_{H^2(\Omega)},
\end{equation}
under the assumption that $\Omega$ is a bounded $C^{1,1,}$ domain.
In Section \ref{elasticity-2} we address the problem of
convergence rates for elliptic systems of linear elasticity.

No smoothness condition on $A$ will be imposed on the coefficient matrix $A$ in this chapter.
Further results on convergence rates may be found in
Chapter \ref{chapter-6} under additional smoothness conditions on $A$.
 


\section
{Flux correctors and $\varep$-smoothing}\label{section-c-1}

Throughout this section we assume that $A=A(y)$ is 1-periodic and
satisfies the $V$-ellipticity condition (\ref{weak-e-1})-(\ref{weak-e-2}).
For $1\le i,j\le d$ and $1 \le \alpha, \beta\le m$, let
\begin{equation}\label{definition-of-b}
b_{ij}^{\alpha\beta} (y)
= a_{ij}^{\alpha\beta} (y)
+a_{ik}^{\alpha\gamma} (y) \frac{\partial}{\partial y_k}\left( \chi_j^{\gamma\beta} (y)\right)
-\widehat{a}_{ij}^{\alpha\beta},
\end{equation}
where the repeated index $k$ is summed from $1$ to $d$ and $\gamma$ from $1$ to $m$.
Observe that the matrix $B(y)=\big( b_{ij}^{\alpha\beta} (y)\big)$
 is 1-periodic and that $\|B\|_{L^p(Y)} \le C_0$ for some $p>2$ and
$C_0>0$ depending on $\mu$.
Moreover, it follows from the definitions of $\chi_j^\beta$ and $\widehat{a}_{ij}^{\alpha\beta}$
 in Section \ref{section-1.2} that
\begin{equation}\label{1.4.0}
\frac{\partial}{\partial y_i} \left( b_{ij}^{\alpha\beta}\right) =0 \quad \text{ and } \quad
\int_Y b_{ij}^{\alpha\beta}(y)\, dy =0.
\end{equation}

\begin{prop}\label{lemma-1.4.1}
There exist $\phi_{kij}^{\alpha\beta} \in H_\per^1(Y)$, where
$1\le i,j,k\le d$ and $1\le \alpha, \beta \le m$, such that
\begin{equation}\label{definition-of-F}
b_{ij}^{\alpha\beta} =\frac{\partial}{\partial y_k} \left( \phi_{kij}^{\alpha\beta}\right)
\quad \text{ and }\quad
\phi_{kij}^{\alpha\beta}=-\phi_{ikj}^{\alpha\beta}.
\end{equation}
Moreover, if $\chi=(\chi_j^\beta)$ is H\"older continuous, then $\phi_{kij}^{\alpha\beta}\in L^\infty(Y)$.
\end{prop}

\begin{proof}
Since $\int_Y b_{ij}^{\alpha\beta}\, dy=0$,
there exists $f_{ij}^{\alpha\beta}\in H_\per^2(Y)$ such that $\int_Y f_{ij}^{\alpha\beta}\, dy =0$
and 
\begin{equation}\label{1.4.1-1}
\Delta f_{ij}^{\alpha\beta} =b_{ij}^{\alpha\beta}\quad \text{ in }Y.
\end{equation}
Moreover, 
$$
\aligned
\|f_{ij}^{\alpha\beta}\|_{H^2(Y)}  & \le C\,\| b_{ij}^{\alpha\beta} \|_{L^2(Y)} \\
&\le C.
\endaligned
$$
Define
$$
\phi_{kij}^{\alpha\beta} (y)
=\frac{\partial}{\partial y_k} \left( f_{ij}^{\alpha\beta}\right)
-\frac{\partial}{\partial y_i}\left(f_{kj}^{\alpha\beta}\right).
$$
Clearly, $\phi_{kij}^{\alpha\beta}\in H_\per^1(Y)$ and $\phi_{kij}^{\alpha\beta}=-\phi_{ikj}^{\alpha\beta}$.
Using $\frac{\partial}{\partial y_i} \big\{ b_{ij}^{\alpha\beta}\big\} =0$,
we may deduce from (\ref{1.4.1-1}) that $\frac{\partial}{\partial y_i} \big\{ f_{ij}^{\alpha\beta}\big\}$
is a 1-periodic harmonic function and thus is constant.
Hence,
$$
\aligned
\frac{\partial}{\partial y_k}
\left\{ \phi_{kij}^{\alpha\beta}\right\}
& =\Delta \left\{ f_{ij}^{\alpha\beta}\right\}
-\frac{\partial^2}{\partial y_k \partial y_i}
\left\{ f_{kj}^{\alpha\beta} \right\}\\
&=\Delta \left\{ f_{ij}^{\alpha\beta}\right\}\\
&=b_{ij}^{\alpha\beta}.
\endaligned
$$

Suppose that the corrector $\chi$ is H\"older continuous.
Recall that $\mathcal{L}_1 \big(\chi_j^\beta +P_j^\beta\big)=0$ in $\brd$.
By Caccioppoli's inequality, 
$$
\int_{B(y, r)} |\nabla \chi|^2\, dx
\le \frac{C}{r^2} \int_{B(y, 2r)} |\chi (x)-\chi(y)|^2\, dx
+C\,  r^d.
$$
This
implies that  $|\nabla \chi|$ is in the Morrey space $L^{2, \rho} (Y)$ for some $\rho>d-2$; i.e.,
$$
\int_{B(y,r)} |\nabla\chi|^2 \, dx \le C\,  r^\rho \quad \text{ for } y\in Y \text{ and } 0<r<1.
$$
Consequently,  $b_{ij}^{\alpha\beta}\in L^{2,\rho}(Y)$ for some $\rho>d-2$ and 
$$
\sup_{x\in Y}
\int_Y \frac{|b_{ij}^{\alpha\beta} (y)|}{|x-y|^{d-1}}\, dy\le C.
$$
In view of (\ref{1.4.1-1}), using a potential representation for Laplace's equation, one may deduce that 
$$
\aligned
\|\nabla f_{ij}^{\alpha\beta}\|_{L^\infty (Y)}
& \le
C \|  f_{ij}^{\alpha\beta}\|_{L^2(Y)}
+ C \sup_{x\in Y}
\int_Y \frac{|b_{ij}^{\alpha\beta} (y)|}{|x-y|^{d-1}}\, dy\\
& \le C.
\endaligned
$$
It follows that
$\phi_{kij}^{\alpha\beta}\in L^\infty(Y)$.
\end{proof}

\begin{remark}\label{remark-1.4.1}
{\rm
Recall that if $d=2$ or $m=1$, the function $\chi$ is H\"older continuous.
As a result, we obtain  $\| \phi^{\alpha\beta}_{kij}\|_\infty\le C$.
In the case where $d\ge 3$ and $m\ge 2$, we have
 $b_{ij}^{\alpha\beta}\in L^p(Y)$ for some $p>2$.
 It follows that
$\nabla^2 f_{ij}^{\alpha\beta}\in L^p(Y)$ for some $p>2$.
By Sobolev imbedding this implies that
$\phi_{kij}^{\alpha\beta} \in L^q(Y)$ for some $q>\frac{2d}{d-2}$.
We mention that if the coefficient matrix $A$ is in $\text{VMO}(\rd)$
(see (\ref{VMO}) for the definition),
then $\chi (y)$ is H\"older continuous and consequently, 
$\phi=\big(\phi_{kij}^{\alpha\beta}\big)$ is bounded.
}
\end{remark}

\begin{remark}\label{remark-c-1}
{\rm
A key property of $\phi =\big(\phi_{kij}^{\alpha\beta} \big)$,
which follows form (\ref{definition-of-F}), is the  identity:
\begin{equation}\label{phi-identity-0}
 b_{ij}^{\alpha\beta} (x/\varep)\,
\frac{\partial \psi^\alpha}{\partial x_i} 
=\varep\, \frac{\partial}{\partial x_k} \left\{ \phi_{kij}^{\alpha\beta} 
\left({x}/{\varep}\right)  \frac{\partial  \psi^\alpha}{\partial x_i}\right\}
\end{equation}
for any $\psi =(\psi^\alpha)\in H^2(\Omega; \mathbb{R}^m)$.
}
\end{remark}

Let $u_\varep \in H^1(\Omega; \mathbb{R}^m)$, $u_0\in H^2(\Omega; \mathbb{R}^m)$,  and
$$
w_\varep =u_\varep -u_0 - \varep\, \chi(x/\varep) \nabla u_0.
$$
A direct computation shows that
\begin{equation}\label{duality-energy}
A(x/\varep)\nabla u_\varep -\widehat{A} \nabla u_0 - B(x/\varep)\nabla u_0
=A(x/\varep)\nabla w_\varep +\varep\, A(x/\varep)\chi (x/\varep)\nabla^2 u_0,
\end{equation}
where $B(y) =\big(b_{ij}^{\alpha\beta} (y)\big)$ is defined by (\ref{definition-of-b}).
It follows that 
\begin{equation}\label{flux-rate}
\aligned
&\| A(x/\varep)\nabla u_\varep -\widehat{A}\nabla u_0 -B(x/\varep)\nabla u_0\|_{L^2(\Omega)}\\
 &\qquad\le C\, \| \nabla w_\varep\|_{L^2(\Omega)}
 + C\, \varep\, \| \chi(x/\varep) \nabla^2 u_0\|_{L^2(\Omega)}\\
 &\qquad
 \le C \, \| \nabla u_\varep -\nabla u_0 - \nabla \chi(x/\varep) \nabla u_0\|_{L^2(\Omega)}
 + C\,  \varep\, \| \chi(x/\varep) \nabla^2 u_0\|_{L^2(\Omega)},
 \endaligned
\end{equation}
where $C$ depends only on $\mu$.
This indicates  that
 the 1-periodic matrix-valued function $B(y)$ plays the same role  for the flux $A(x/\varep)\nabla u_\varep$
as $\nabla \chi (y)$ does for $\nabla u_\varep$.
Since $b_{ij}^{\alpha\beta} =\frac{\partial}{\partial y_k} \big(\phi_{kij}^{\alpha\beta}\big)$,
the 1-periodic function $\phi=\big(\phi_{kij}^{\alpha\beta}\big)$
is called the flux corrector.

To deal with the fact that the correctors $\chi$ and $\phi$ may be
unbounded (if $d\ge 3$ and $m\ge 2$), we introduce an $\varep$-smoothing operator $S_\varep$.

\begin{definition}
{\rm
Fix $\rho\in C_0^\infty(B(0, 1/2))$ such that $\rho \ge 0$ and $\int_{\mathbb{R}^d} \rho\, dx =1$.
For $\varep>0$, define
\begin{equation} \label{Steklov}
S_\varep (f) (x) =\rho_\varep*f (x) =\int_{\mathbb{R}^d} f(x-y)\rho_\varep (y)\, dy,
\end{equation}
where $\rho_\varep (y) =\varep^{-d} \rho(y/\varep)$.
}
\end{definition}

The following two propositions  contain the most useful properties of $S_\varep$ for us.

\begin{prop}\label{lemma-1.5.3}
Let $f\in L^p_{\loc}(\rd)$ for some $1\le p<\infty$. Then for any $g\in L^p_{\loc}(\rd)$,
\begin{equation}\label{1.5.3-0}
\| g(x/\varep)\, S_\varep (f)\|_{L^p(\mathcal{O})}
\le  C\sup_{x\in \rd} \left(\average_{B(x,1/2)} |g|^p\right)^{1/p} \| f\|_{L^p(\mathcal{O}_{\varep/2})},
\end{equation}
where $\mathcal{O}\subset \mathbb{R}^d$ is open,
$$
\mathcal{O}_t =\big\{ x\in \mathcal{O}: \, \text{\rm dist}(x, \mathcal{O})<t\big\},
$$
and $C$ depends only on $p$.
\end{prop}
 
\begin{proof}
By H\"older's inequality, 
$$
|S_\varep (f)(x)|^p\le \int_{\mathbb{R}^d} | f(y)|^p \rho_\varep (x-y)\, dy.
$$
This, together with Fubini's Theorem, gives (\ref{1.5.3-0}) for the case $\mathcal{O}=\mathbb{R}^d$.
The general case follows from the observation that $S_\varep (f) (x) =S_\varep (f \chi_{\mathcal{O}_{\varep/2}}) (x)$
for any $x\in \mathcal{O}$.
\end{proof}

It follows from (\ref{1.5.3-0}) that if $g$ is 1-periodic and belongs to $L^p(Y)$, then
\begin{equation}\label{1.5.3-00}
\| g(x/\varep) S_\varep (f) \|_{L^p(\mathcal{O})} \le C \| g\|_{L^p(Y)} \| f\|_{L^p(\mathcal{O}_{\varep/2})},
\end{equation}
where $C$ depends only on $p$. A similar argument gives
\begin{equation}\label{1.5.3-000}
\| g(x/\varep) \nabla S_\varep (f) \|_{L^p(\mathcal{O})} \le C \varep^{-1}
\| g\|_{L^p(Y)} \| f\|_{L^p(\mathcal{O}_{\varep/2})}.
\end{equation}

\begin{prop}\label{lemma-1.5.4}
Let $f\in W^{1, p}(\rd)$ for some $1\le p\le \infty$. Then
\begin{equation}\label{1.5.4-0}
\| S_\varep (f) -f \|_{L^p(\rd)} \le  \varep\, \| \nabla f\|_{L^p(\rd)}.
\end{equation}
Moreover, if $q=\frac{2d}{d+1}$,
\begin{equation}\label{1.5.4-00}
\aligned
\| S_\varep (f)\|_{L^2(\br^d)}
& \le C\varep^{-1/2} \|  f\|_{L^q(\br^d)},\\
\| S_\varep (f) -f \|_{L^2(\br^d)} &\le C\varep^{1/2}\| \nabla f\|_{L^q(\br^d)},
\endaligned
\end{equation}
where $C$ depends only on $d$.
\end{prop}

\begin{proof}
Using 
$$
f(x+y)-f(x)=\int_0^1 \nabla f(x+ty)\cdot y\, dt
$$
and Minkowski's inequality, we see that
$$
\| f(\cdot +y ) -f(\cdot)\|_{L^p(\br^d)}
\le | y| \|\nabla f\|_{L^p(\br^d)}
$$
for any $y\in \br^d$.
By Minkowski's inequality again,
$$
\aligned
\| S_\varep (f)- f\|_{L^p(\br^d)}
&\le \int_{\mathbb{R}^d} \| f(\cdot -\varep y) -f(\cdot)\|_{L^p(\br^d)} \, \rho (y)\, dy\\
& \le \int_{\mathbb{R}^d}  \varep |y|  \rho (y)\, dy \, \|\nabla f\|_{L^p(\br^d)}\\
&\le \varep\, \| \nabla f\|_{L^p(\br^d)},
\endaligned
$$
which gives (\ref{1.5.4-0}).

Next,
we note that the Fourier transform of $S_\varep (f)$ is given by
$\widehat{\rho} (\varep  \xi ) \widehat{f} (\xi)$.
By Plancheral's theorem and H\"older's inequality,
$$
\aligned
\int_{\br^d} |S_\varep (f)|^2\, dx
&=\int_{\br^d} |\widehat{\rho}(\varep \xi)|^2 |\widehat{f}(\xi)|^2\, d\xi\\
 &\le \left(\int_{\br^d} |\widehat{\rho} (\varep \xi)|^{2d}\, d\xi\right)^{1/d} \| \widehat{f}\|_{L^{q^\prime}(\br^d)}^2\\
&\le C \varep^{-1} \| f\|_{L^q(\br^d)}^2,
\endaligned
$$
where we have used the Hausdorff-Young inequality 
$\| \widehat{f}\|_{L^{q^\prime}(\br^d)}\le \| f\|_{L^q(\br^d)}$ 
in the last step. This gives the first inequality in (\ref{1.5.4-00}).
Similarly,
using $\widehat{\rho}(0)=\int_{\br^d} \rho=1$, we obtain
$$
\aligned
\| S_\varep (f)- f\|_{L^2(\br^d)}
& =\| (\widehat{\rho}(\varep\xi)-1) \widehat{f}\|_{L^2(\br^d)}\\
&\le C \left(\int_{\br^d} |\widehat{\rho} (\varep \xi) -\widehat{\rho} (0)|^{2d} |\xi|^{-2d}\, d\xi\right)^{1/(2d)}
\| \widehat{\nabla f} \|_{L^{q^\prime}(\br^d)}\\
&\le C\varep^{1/2} \| \nabla f\|_{L^q(\br^d)},
\endaligned
$$
where we have also used the fact
$|\widehat{\rho}(\xi) -\widehat{\rho}(0)|\le C |\xi|$.
\end{proof}

We finish this section with some estimates for integrals on boundary layers. 
Let
\begin{equation}\label{b-layer-1}
{\Omega}_t
=\big\{ x\in \Omega: \ \text{\rm dist}(x, \partial\Omega)<t  \big\},
\end{equation}
where $t>0$.

\begin{prop}\label{layer-prop-1}
Let $\Omega$ be a bounded Lipschitz domain in $\br^d$ and $q=\frac{2d}{d+1}$.
Then for any $f\in W^{1, q}(\Omega)$,
\begin{equation}\label{bl-estimate-1}
\| f\|_{L^2(\Omega_t)}
\le C t^{1/2}  \| f\|_{W^{1, q}(\Omega)} 
\quad \text{ and } \quad
\| f\|_{L^2(\partial\Omega)} \le C  \| f\|_{W^{1, q}(\Omega)},
\end{equation}
where $C$ depends only on $\Omega$.
\end{prop}

\begin{proof}
Let $x=(x^\prime, x_d)\in \br^d$.
Using the fundamental theorem of calculus, it is not hard to show that
$$
\int_{|x^\prime|<r} |f(x^\prime, s)|^2\, dx^\prime
\le \frac{C}{r} \int_{|x^\prime|<r} \int_0^r | f(x^\prime, x_d)|^2 \, dx^\prime dx_d
+C \int_{|x^\prime|<r} \int_0^r | f|\,  |\nabla f | \, dx^\prime dx_d
$$
for any $s\in (0, r)$.
It follows that
$$
\int_{0}^t\int_{|x^\prime|<r} |f (x^\prime, s)|^2\, dx^\prime ds
\le \frac{C t}{r} \int_{|x^\prime|<r} \int_{0}^r | f (x^\prime, x_d)|^2 \, dx^\prime dx_d
+C  t\int_{|x^\prime|<r} \int_{0}^r | f |\,  |\nabla f | \, dx^\prime dx_d
$$
for any $t\in (0, r)$. By covering $\partial\Omega$ with coordinate patches, we obtain 
$$
\aligned
\int_{{\Omega}_t} |f |^2\, dx
 & \le  Ct\int_{\Omega}  |f |^2\, dx
+ C t \int_{\Omega} |f |\, |\nabla f |\, dx\\
&\le C t \| f\|^2_{L^2(\Omega)} + C t \| f\|_{L^{q^\prime}(\Omega)}
\|  \nabla f\|_{L^q(\Omega)},
\endaligned
$$
which, together with  the Sobolev inequality $\| f\|_{L^{q^\prime}(\Omega)} \le C \| f\|_{W^{1, q}(\Omega)}$,
 gives the first inequality in (\ref{bl-estimate-1}).
 To see the second, we note that by a similar argument,
 $$
\int_{{\partial \Omega}} |f |^2\, d\sigma
  \le  C\int_{\Omega}  |f |^2\, dx
+ C  \int_{\Omega} |f |\, |\nabla f |\, dx,
$$
which is bounded by $C\| f\|^2_{W^{1, q}(\Omega)}$.
\end{proof}

\begin{prop}\label{layer-prop-2}
Let $\Omega$ be a bounded Lipschitz domain in $\br^d$ and $q=\frac{2d}{d+1}$.
Let $g\in L^2_{\loc} (\br^d)$ be a 1-periodic function.
Then, for any $f\in W^{1, q}(\Omega)$,
\begin{equation}\label{bl-estimate-2}
\int_{\Omega_{2 t}\setminus \Omega_t}
| g(x/\varep)|^2 |S_\varep (f)|^2\, dx
\le C t \, \| g\|^2_{L^2(Y)} \| f\|_{W^{1, q}(\Omega)}^2,
\end{equation}
where $t\ge \varep$ and $C$ depends only on  $\Omega$.
\end{prop}

\begin{proof}
We may assume that $t$ is small.
Let $\mathcal{O}=\Omega_{2t}\setminus \Omega_t$. It follows from (\ref{1.5.3-00}) that
$$
\aligned
\int_{\Omega_{2 t}\setminus \Omega_t}
| g(x/\varep)|^2 |S_\varep (f)|^2\, dx
&\le C \| g\|^2_{L^2(Y)}
\| f\|^2_{L^2(\mathcal{O}_{\varep/2})}\\
&\le C t  \| g\|^2_{L^2(Y)} \|  f\|^2_{W^{1, q}(\Omega)},
\endaligned
$$
where we have used (\ref{bl-estimate-1}) for the last inequality.
\end{proof}



\section{Convergence rates in $H^1$ for Dirichlet problem}\label{section-c-2}

Throughout this section we assume that $A$ is 1-periodic and
satisfies the $V$-ellipticity condition (\ref{weak-e-1})-(\ref{weak-e-2}).
The symmetry condition $A^*=A$ will be imposed for some sharp results in Lipschitz domains.

Fix a cut-off function $\eta_\varep \in C_0^\infty(\Omega)$ such that
\begin{equation}\label{eta-e}
\left\{
\aligned
& 0\le \eta_\varep \le 1, \ \ |\nabla \eta_\varep|\le C/\varep,\\
&\eta_\varep (x)=1 \quad \text{ if dist}(x, \partial\Omega) \ge 4\varep,\\
& \eta_\varep (x)=0 \quad \text{ if dist}(x, \partial\Omega)\le 3\varep.
\endaligned
\right.
\end{equation}
Let $S_\varep^2=S_\varep \circ S_\varep$ and define
\begin{equation}\label{w-c-2}
w_\varep =u_\varep  -u_0   -\varep \chi(x/\varep) \eta_\varep S^2_\varep( \nabla {u}_0),
\end{equation}
where $u_\varep \in H^1(\Omega; \br^m)$ is the weak solution of Dirichlet problem (\ref{D-C}) and
$u_0$ the homogenized solution.

\begin{lemma}\label{lemma-c-2-1}
Let $\Omega$ be a bounded Lipschitz domain in $\br^d$ and $\Omega_t$ be defined by (\ref{b-layer-1}).
Then, for any $\psi\in H^1_0(\Omega; \br^m)$,
\begin{equation}\label{c-2-1-00}
\aligned
 \Big|\int_\Omega & A(x/\varep) \nabla w_\varep \cdot \nabla \psi\, dx \Big|\\
& \le C \|\nabla \psi\|_{L^2(\Omega)}
\Big\{ \varep \| S_\varep(\nabla^2 u_0)\|_{L^2(\Omega\setminus \Omega_{3\varep})}
+\|\nabla u_0 -S_\varep (\nabla u_0)\|_{L^2(\Omega\setminus \Omega_{2\varep})}\Big\}\\
&\qquad\qquad
+ C \|\nabla \psi \|_{L^2(\Omega_{4\varep})}
\| \nabla u_0\|_{L^2(\Omega_{5\varep})},
\endaligned
\end{equation}
where $w_\varep$ is given by (\ref{w-c-2})
and $C$ depends only on $\mu$ and $\Omega$.
\end{lemma}

\begin{proof}
Since $w_\varep\in H^1(\Omega; \br^m)$, it suffices to prove (\ref{c-2-1-00})
for any $\psi \in C_0^\infty(\Omega; \br^d)$.
Note that
$$
\aligned
& A(x/\varep)\nabla w_\varep\\
&=A(x/\varep)\nabla u_\varep -A(x/\varep)\nabla u_0 
-A(x/\varep)\nabla \chi(x/\varep) \eta_\varep S^2_\varep (\nabla{u}_0)
-\varep A(x/\varep)\chi(x/\varep) \nabla \big(\eta_\varep S^2_\varep(\nabla {u}_0)\big)\\
&=A(x/\varep)\nabla u_\varep
-\widehat{A} \nabla u_0
+ \big(\widehat{A}-A(x/\varep) \big) \big[ \nabla u_0 -\eta_\varep S^2_\varep(\nabla{u}_0)\big]\\
&\qquad\qquad
 -B(x/\varep) \eta_\varep S^2_\varep(\nabla{u}_0)
-\varep A(x/\varep)\chi(x/\varep) \nabla \big(\eta_\varep S^2_\varep( \nabla{u}_0)\big),
\endaligned
$$
where we have used the fact
\begin{equation}\label{B}
B(y)= A(y)+ A(y) \nabla \chi (y) -\widehat{A}.
\end{equation}
Using
\begin{equation}\label{weak-equation}
\int_\Omega A(x/\varep)\nabla u_\varep \cdot \nabla \psi\, dx
=\int_\Omega \widehat{A} \nabla u_0 \cdot \nabla \psi\, dx
\end{equation}
for any $\psi\in C^\infty_0(\Omega; \br^m)$,
we obtain 
\begin{equation}\label{c-2-1-2}
\aligned
\Big|\int_\Omega A(x/\varep) \nabla w_\varep \cdot \nabla \psi\, dx \Big|
&\le C \int_\Omega (1-\eta_\varep) |\nabla u_0|  |\nabla \psi|\, dx\\
&\qquad +C \int_{\Omega}
 \eta_\varep |\nabla u_0 -S^2_\varep (\nabla{u}_0)| \, |\nabla\psi|\, dx\\
& \qquad 
+\Big|\int_\Omega \eta_\varep B(x/\varep) 
S^2_\varep ( \nabla{u}_0 ) \cdot \nabla \psi\, dx \Big|\\
&\qquad
+C  \varep 
\int_\Omega |\chi(x/\varep)\nabla \big(\eta_\varep
S^2_\varep (\nabla {u}_0)\big)| |\nabla \psi|\, dx.
\endaligned
\end{equation}\
Since $\eta_\varep=1$ in $\Omega\setminus \Omega_{4\varep}$,
by Cauchy inequality, the first  term in the RHS of (\ref{c-2-1-2}) is bounded by
$$
C \| \nabla u_0\|_{L^2(\Omega_{4\varep})} \| \nabla \psi\|_{L^2(\Omega_{4\varep})}.
$$
Using $\eta_\varep =0$ in $\Omega_{3\varep}$
and
$$
\aligned
\| \nabla u_0 -S_\varep^2(\nabla u_0)\|_{L^2(\Omega\setminus \Omega_{3\varep})}
&\le \| \nabla u_0 -S_\varep (\nabla u_0)\|_{L^2(\Omega\setminus \Omega_{3\varep})}
+\| S_\varep (\nabla u_0) -S_\varep^2(\nabla u_0)\|_{L^2(\Omega\setminus \Omega_{3\varep})}\\
&\le C \| \nabla u_0 -S_\varep (\nabla u_0)\|_{L^2(\Omega\setminus \Omega_{2\varep})},
\endaligned
$$
we may bound the second term by
$$
 C \|\nabla u_0 -S_\varep(\nabla u_0)\|_{L^2(\Omega\setminus \Omega_{2\varep})}
\|\nabla \psi\|_{L^2(\Omega)}.
$$
 Also, by Cauchy inequality and (\ref{1.5.3-00}),
 the fourth term in the RHS of (\ref{c-2-1-2}) is dominated by
 $$
 C \|\nabla {u}_0\|_{L^2(\Omega_{5\varep})} \|\nabla \psi\|_{L^2(\Omega_{4\varep})}
+ C \varep 
\| S_\varep(\nabla^2 u_0) \|_{L^2(\Omega\setminus \Omega_{2\varep})}
\| \nabla \psi\|_{L^2(\Omega)}.
$$

Finally, To handle the third term in the RHS of (\ref{c-2-1-2}), we use the identity (\ref{phi-identity-0}) to obtain
\begin{equation}\label{c-2-1-3}
\aligned
\eta_\varep B(x/\varep)  S^2_\varep(\nabla {u}_0)\cdot \nabla \psi
&=b_{ij}^{\alpha\beta}(x/\varep) S^2_\varep \left(\frac{\partial {u}_0^\beta}{\partial x_j}\right)
\frac{\partial \psi^\alpha}{\partial x_i} \eta_\varep \\
& =\varep \frac{\partial}{\partial x_k}
\left(\phi_{kij}^{\alpha\beta} (x/\varep) \frac{\partial \psi^\alpha}{\partial x_i} \right)
S^2_\varep \left(\frac{\partial {u}_0^\beta}{\partial x_j}\right)
\eta_\varep.
\endaligned
\end{equation}
It follows from (\ref{c-2-1-3}) and integration by parts that
$$
\aligned
&\Big|\int_\Omega \eta_\varep B(x/\varep) S^2_\varep (\nabla {u}_0 ) \cdot \nabla \psi\, dx \Big|\\
& \le C \varep \int_\Omega
\eta_\varep |\phi(x/\varep)| |\nabla \psi| |S^2_\varep(\nabla^2 {u}_0)|\, dx
+ C \varep \int_{\Omega} 
|\nabla \eta_\varep| |\phi(x/\varep)| |\nabla \psi| |S^2_\varep (\nabla{u}_0)|\, dx\\
& \le C  \varep
\|\nabla \psi\|_{L^2(\Omega)}
\| S_\varep (\nabla^2 {u}_0)\|_{L^2(\Omega\setminus \Omega_{2\varep})} 
+ C\| \nabla \psi\|_{L^2(\Omega_{4\varep})} \| \nabla {u}_0\|_{L^2(\Omega_{5\varep})},
\endaligned
$$
where we have used Cauchy inequality and
(\ref{1.5.3-00}) for the second inequality.
This completes the proof.
\end{proof}

The next theorem gives the $O(\sqrt{\e})$ convergence rate in $H^1$.

\begin{thm}\label{m-thm-2.2-1}
Assume that $A$ is 1-periodic and satisfies (\ref{weak-e-1})-(\ref{weak-e-2}).
Let $\Omega$ be a bounded Lipschitz domain in $\br^d$.
Let $w_\varep$ be given by (\ref{w-c-2}).
Then for $0<\varep<1$ ,
\begin{equation}\label{thm-2.2-1a}
\|\nabla w_\e\|_{L^2(\Omega)}
\le C
\Big\{ \e \|\nabla^2 u_0\|_{L^2(\Omega\setminus \Omega_{\e})}
+\| \nabla u_0 \|_{L^2(\Omega_{5\e})}\Big\}.
\end{equation}
Consequently,
\begin{equation}\label{thm-2.2-1-0}
\| w_\e \|_{H_0^1(\Omega)}
\le C\sqrt{\e}\,  \| u_0\|_{H^2(\Omega)}.
\end{equation}
The constant $C$ depends only on $\mu$ and $\Omega$.
\end{thm}

\begin{proof}
Since $w_\e\in H^1_0(\Omega; \br^m)$, we may take $\psi=w_\e$ in (\ref{c-2-1-00}).
This, together with the ellipticity condition (\ref{weak-e-1})-(\ref{weak-e-2}), gives
\begin{equation}\label{thm-2.2-1-00}
\|\nabla w_\e\|_{L^2(\Omega)}
\le C
\Big\{ \e \|\nabla^2 u_0\|_{L^2(\Omega\setminus \Omega_{2\e})}
+\| \nabla u_0 -S_\e (\nabla u_0)\|_{L^2(\Omega\setminus \Omega_{2\e})}
+\| \nabla u_0 \|_{L^2(\Omega_{5\e})}\Big\}.
\end{equation}
Choose $\widetilde{\eta}_\e\in C_0^\infty(\Omega)$ such that
$0\le \widetilde{\eta}_\e \le 1$, $\widetilde{\eta}_\e=0$ in $\Omega_\e$,
$\widetilde{\eta}_\e =1$ in $\Omega\setminus\Omega_{3\e/2}$, and $|\nabla \widetilde{\eta}|
\le C \e^{-1}$.
It follows that
\begin{equation}\label{thm-2.2-1-1}
\aligned
\| \nabla u_0 -S_\e (\nabla u_0)\|_{L^2(\Omega\setminus \Omega_{2\e})}
&\le \| \widetilde{\eta}_\e (\nabla {u}_0)  -S_\e (\widetilde{\eta}_\e \nabla {u}_0)\|_{L^2(\br^d)}\\
&\le C \e\, \|\nabla  (\widetilde{\eta}_\e \nabla u_0)\|_{L^2(\br^d)}\\
&\le C\Big\{  \e  \| \nabla^2 u_0\|_{L^2(\Omega\setminus \Omega_\e)}
+\| \nabla u_0\|_{L^2(\Omega_{2\e})} \Big\}
\endaligned
\end{equation}
where we have used (\ref{1.5.4-0}) for the second inequality.
The estimate (\ref{thm-2.2-1a}) now follows from (\ref{thm-2.2-1-00}) and (\ref{thm-2.2-1-1})
Note that, by (\ref{bl-estimate-1}),
\begin{equation}\label{bl-e-10}
\|\nabla u_0\|_{L^2(\Omega_{5\e})}
\le C\sqrt{\e} \|u_0\|_{H^2(\Omega)}.
\end{equation}
The inequality  (\ref{thm-2.2-1-0}) follows from (\ref{thm-2.2-1a}) and (\ref{bl-e-10}).
\end{proof}

 Under the additional symmetry condition
$A^*=A$, a better estimate can be obtained, using sharp regularity estimates for $\mathcal{L}_0$.

\begin{thm}\label{main-thm-2.2}
Suppose that $A$ is 1-periodic and satisfies (\ref{weak-e-1})-(\ref{weak-e-2}).
Also assume that $A^*=A$.
Let $\Omega$ be a bounded Lipschitz domain in $\mathbb{R}^d$.
Let $w_\varep$ be given by (\ref{w-c-2}).
Then, for $0<\varep<1$,
\begin{equation}\label{c-2-2-00}
\| w_\varep\|_{H_0^1(\Omega)}
\le C \sqrt{\varep}\, \Big\{ \| F\|_{L^q(\Omega)} +\| f\|_{H^1(\partial\Omega)} \Big\},
\end{equation}
where $q=\frac{2d}{d+1}$ and
$C$ depends only on $\mu$ and $\Omega$. 
\end{thm}

\begin{definition}
{\rm For a continuous function $u$ in a bounded Lipschitz domain $\Omega$, 
the nontangential maximal function of $u$ is defined by
\begin{equation}\label{non-max}
 (u)^* (x)
=\sup \Big\{ |u(y)|: \  y\in \Omega \text{ and } |y-x|< C_0 \, \text{dist} (y, \partial\Omega) \Big\}
\end{equation}
for $x\in \partial\Omega$, where $C_0>1$ is a sufficiently large constant depending on $\Omega$.
See Section \ref{section-5.1} for more details on $(u)^*$.
}
\end{definition}

The proof of Theorem \ref{main-thm-2.2} relies on the following regularity result 
for solutions of $\mathcal{L}_0 (u)=0$ in $\Omega$.

\begin{lemma}\label{lemma-Lip-d}
Assume that $A$ satisfies the same conditions as in Theorem \ref{main-thm-2.2}.
Let $\Omega$ be a bounded Lipschitz domain in $\br^d$.
Let $u\in H^1(\Omega; \br^m)$ be a weak solution to the Dirichlet problem:
$\mathcal{L}_0 (u)=0$ in $\Omega$ and $u=f$ on $\partial\Omega$,
where $f\in H^1(\partial\Omega; \br^m)$.
Then 
\begin{equation}\label{Lip-d-m}
\|(\nabla u)^*\|_{L^2(\partial\Omega)}
\le C \| f\|_{H^1(\partial\Omega)},
\end{equation}
where $C$ depends only on $\mu$ and $\Omega$.
\end{lemma}

\begin{proof}
By Lemmas \ref{weak-L-0} and \ref{adjoint-lemma},
$\widehat{A}$ satisfies the Legendre-Hadamard ellipticity condition 
(\ref{weak-eee}) and is symmetric.
As a result,  the estimate (\ref{Lip-d-m})  follows from \cite{Fabes-1988, Gao-1991}.
We refer the reader to Chapter \ref{chapter-7} for nontangential-maximal-function estimates in Lipschitz domains.
In particular, estimate (\ref{Lip-d-m}) is proved for solutions of $\mathcal{L}_\e (u_\e)=0$ in $\Omega$ under the conditions
(\ref{s-ellipticity}) and (\ref{periodicity}) as well as the H\"older continuity condition on $A$.
\end{proof}

\begin{proof}[\bf Proof of Theorem \ref{main-thm-2.2}]

We start by taking $\psi=w_\varep\in H_0^1(\Omega; \mathbb{R}^m)$ in (\ref{c-2-1-00}).
By the ellipticity  condition (\ref{weak-e-1})-(\ref{weak-e-2}),
this gives
\begin{equation}\label{m-thm-2.2-0}
\| \nabla w_\varep\|_{L^2(\Omega)}
\le C \Big\{ \varep \|S_\varep( \nabla^2 u_0)\|_{L^2(\Omega\setminus \Omega_{3\varep})}
+\|\nabla u_0\|_{L^2(\Omega_{5\varep})}
+\|\nabla u_0 -S_\varep(\nabla u_0)\|_{L^2(\Omega\setminus \Omega_{2\varep})} \Big\}.
\end{equation}
To bound  the RHS of (\ref{m-thm-2.2-0}),
we write
$ u_0= v_0 +\phi$, where
\begin{equation}\label{v-0}
v_0(x) =\int_\Omega \Gamma_0 (x-y) F(y)\, dy
\end{equation}
and $\Gamma_0 (x)$ denotes the matrix of fundamental solutions for the operator
$\mathcal{L}_0$ in $\br^d$, with pole at the origin.
It follows from the singular integral and fractional integral estimates \cite{Stein-1970} that
\begin{equation}\label{s-f}
\aligned
\|\nabla^2 v_0\|_{L^q(\br^d)} & \le C\| F\|_{L^q(\Omega)},\\
\|\nabla v_0\|_{L^p(\br^d)}
&\le C \| F\|_{L^q(\Omega)},
\endaligned
\end{equation}
where $p=\frac{2d}{d-1}$ and $q=p^\prime=\frac{2d}{d+1}$.
This, together with (\ref{1.5.4-00}) and (\ref{bl-estimate-1}), yields that
$$
\aligned
\varep \| S_\varep(\nabla^2 v_0)\|_{L^2(\Omega\setminus\Omega_{3\varep})}
& \le C \varep^{1/2} \| \nabla^2 v_0\|_{L^q(\br^d)}\\
&\le C \varep^{1/2} \| F\|_{L^q(\Omega)},
\endaligned
$$
and that
$$
\aligned
\|\nabla v_0\|_{L^2(\Omega_{5\varep})}
 &\le C \varep^{1/2} \| \nabla v_0\|_{W^{1, q}(\Omega)}\\
&\le C \varep^{1/2} \| F\|_{L^q(\Omega)}.
\endaligned
$$
Also, note that by (\ref{1.5.4-00}),
$$
\aligned
\|\nabla v_0 -S_\varep (\nabla v_0)\|_{L^2(\Omega\setminus \Omega_{2\varep})}
& \le C \varep^{1/2} \|\nabla^2 v_0\|_{L^q(\br^d)}\\
&\le C \varep^{1/2} \| F\|_{L^q(\Omega)}.
\endaligned
$$
In summary we have proved that
\begin{equation}\label{v-0-estimate}
\varep \|S_\varep( \nabla^2 v_0)\|_{L^2(\Omega\setminus \Omega_{3\varep})}
+\|\nabla v_0\|_{L^2(\Omega_{5\varep})}
+\|\nabla v_0 -S_\varep(\nabla v_0)\|_{L^2(\Omega\setminus \Omega_{2\varep})} 
\le C \varep^{1/2} \| F\|_{L^q(\Omega)}.
\end{equation}

It remains to bound the LHS of (\ref{v-0-estimate}), with $v_0$ replaced by $\phi$.
To this end we first note that $\mathcal{L}_0 (\phi)=0$ in $\Omega$
and $\phi = f -v_0$ on $\partial\Omega$.
This allows us to apply Lemma \ref{lemma-Lip-d}.
Since
$$
\| v_0\|_{H^1(\partial\Omega)}
\le C \| v_0\|_{W^{2, q}(\Omega)} \le C \| F\|_{L^q(\Omega)},
$$
where we have used  (\ref{bl-estimate-1}) for the first inequality, we obtain
\begin{equation}\label{d-m-1}
\aligned
\| (\nabla \phi)^*\|_{L^2(\partial\Omega)}
  &\le C  \Big\{ \| f\|_{H^1(\partial\Omega)} + \|  v_0\|_{H^1(\partial\Omega)} \Big\}\\
  &\le C\Big\{ \| f\|_{H^1(\partial\Omega)}
  +\| F\|_{L^q(\Omega)} \Big\}.
  \endaligned
\end{equation}
It follows that
$$
\aligned
\|\nabla \phi\|_{L^2(\Omega_{5\varep})}
 & \le C \varep^{1/2}  \| (\nabla \phi)^*\|_{L^2(\partial\Omega)}\\
 &\le C \varep^{1/2}
 \Big\{ \| f\|_{H^1(\partial\Omega)}
  +\| F\|_{L^q(\Omega)} \Big\}.
  \endaligned
  $$
  
 Next, we use the interior estimate for $\mathcal{L}_0$,
 $$
 |\nabla^2 \phi (x)|^2 \le \frac{C}{r^{d+2}}
 \int_{B(x,r)} |\nabla \phi (y)|^2\, dy,
 $$
 where $r=\text{\rm dist}(x, \partial\Omega)/8$, and Fubini's Theorem to obtain 
 $$
 \aligned
 \int_{\Omega\setminus \Omega_{\varep}} |\nabla^2 \phi (x) |^2\, dx
 &\le C \int_{\Omega\setminus \Omega_{\varep/2}} |\nabla \phi (x)|^2
 \big[ \text{\rm dist} (x, \partial\Omega)\big]^{-2}\, dx \\
 &\le C \varep^{-1} \int_{\partial\Omega}
 | (\nabla \phi)^*|^2\, d\sigma.
 \endaligned
 $$
 Hence,
 $$
 \aligned
 \varep \|S_\varep (\nabla^2 \phi)\|_{L^2(\Omega\setminus \Omega_{3\varep})}
 &\le C \varep \| \nabla^2 \phi\|_{L^2(\Omega\setminus\Omega_{2\varep})}\\
 &\le C \varep^{1/2}
 \|  (\nabla \phi)^*\|_{L^2(\partial\Omega)}\\
 &\le C \varep^{1/2}
 \Big\{ \| f\|_{H^1(\partial\Omega)}
  +\| F\|_{L^q(\Omega)} \Big\}.
  \endaligned
  $$

Finally,   we observe that as in (\ref{thm-2.2-1-1}),
\begin{equation}\label{diff-2.2}
\aligned
\| \nabla \phi  -S_\varep(\nabla \phi)\|_{L^2(\Omega\setminus \Omega_{2\varep})}
&\le C \Big\{ \varep \|\nabla^2 \phi \|_{L^2(\Omega\setminus \Omega_{\varep})}
+\| \nabla \phi \|_{L^2(\Omega_{2\varep})} \Big\}\\
& \le C \varep^{1/2}
 \Big\{ \| f\|_{H^1(\partial\Omega)}
  +\| F\|_{L^q(\Omega)} \Big\},
\endaligned
\end{equation}
where we have used (\ref{1.5.4-0}) for the second inequality.
As a result, we have proved that
$$
\aligned
& \varep \|S_\varep( \nabla^2 \phi)\|_{L^2(\Omega\setminus \Omega_{3\varep})}
+\|\nabla \phi \|_{L^2(\Omega_{5\varep})}
+\|\nabla \phi -S_\varep(\nabla \phi)\|_{L^2(\Omega\setminus \Omega_{2\varep})} \\
&\qquad \le C \varep^{1/2} \Big\{ \| f\|_{H^1(\partial\Omega)}
+\| F\|_{L^q(\Omega)}\Big\}.
\endaligned
$$
This, together with (\ref{m-thm-2.2-0}) and (\ref{v-0-estimate}),
gives (\ref{c-2-2-00}).
\end{proof}

\begin{remark}\label{re-H-1}
{\rm
Since $\| F\|_{L^q(\Omega)} =\|\mathcal{L}_0 (u_0)\|_{L^q(\Omega)} \le C \| \nabla^2 u_0\|_{L^q(\Omega)}$ and
$\| f\|_{H^1(\partial\Omega)} \le C\| u_0\|_{W^{2, q}(\Omega)}$,
it follows from (\ref{c-2-2-00}) that
\begin{equation}\label{c-2-2-0}
\| w_\varep\|_{H^1_0(\Omega)}
\le C \sqrt{\varep}\, \| u_0\|_{W^{2, q}(\Omega)},
\end{equation}
where $C$ depends only on $\mu$ and $\Omega$.
}
\end{remark}

We now consider the two-scale expansions without the $\varep$-smoothing.
 
\begin{thm}\label{theorem-d-c-2-1}
Assume that $A$ is 1-periodic and satisfies (\ref{weak-e-1})-(\ref{weak-e-2}).
Let $\Omega$ be a bounded Lipschitz domain in $\br^d$.
Let $u_\varep\in H^1(\Omega; \br^m)$ be the weak solution of
Dirichlet problem (\ref{D-C}) and $u_0$ the homogenized solution.
Then, if $u_0\in W^{2, d}(\Omega; \br^m)$,
\begin{equation}\label{c-2-1-0}
\| u_\varep -u_0 -\varep \chi(x/\varep) \nabla u_0\|_{H^1(\Omega)}
\le C \sqrt{\varep}\,  \| u_0\|_{W^{2, d}(\Omega)},
\end{equation}
where $0<\varep<1$ and
$C$ depends only on $\mu$ and $\Omega$.
Furthermore, if the corrector $\chi$ is bounded and $u_0\in H^2(\Omega; \br^m)$, then
\begin{equation}\label{c-2-1-1}
\| u_\varep -u_0 -\varep \chi(x/\varep) \nabla u_0\|_{H^1(\Omega)}
\le C \sqrt{\varep}\,  \| u_0\|_{H^2(\Omega)},
\end{equation}
where $0<\varep<1$ and $C$ depends only on $\mu$, $\|\chi\|_\infty$ and $\Omega$.
\end{thm}

To prove Theorem \ref{theorem-d-c-2-1}, 
we need to control the $L^2$ norm of $\chi(x/\varep) \psi$.

\begin{lemma}\label{lemma-c-2-3}
Let $\chi=\big(\chi_j^{\alpha\beta}\big)$ be the matrix of
correctors defined by (\ref{cell-problem}).
Then
\begin{equation}\label{c-2-3-0}
\| \nabla \chi(x/\varep) \psi \|_{L^2(\Omega)}
\le  C \Big\{ \varep \, \|\nabla \psi\|_{L^d(\Omega)}
+ \| \psi\|_{L^d(\Omega)} \Big\}
\end{equation}
for any $\psi\in W^{1, d}(\Omega)$,
where $C$ depends only on $\mu$ and $\Omega$.
Moreover, if  $\chi$ is bounded, then
\begin{equation}\label{c-2-3-1}
\| \nabla \chi(x/\varep) \psi\|_{L^2(\Omega)}
\le  C (1+\|\chi\|_\infty)\Big\{ \varep \, \|\nabla \psi\|_{L^2(\Omega)}
+ \| \psi\|_{L^2(\Omega)} \Big\}
\end{equation}
for any $\psi\in H^1(\Omega)$, where $C$ depends only on $\mu$ and $\Omega$.
\end{lemma}

\begin{proof}
Let
$$
u_\varep=\varep \chi_j^\beta (x/\varep) +P_j^\beta(x-x_0).
$$
Since $\mathcal{L}_\varep (u_\varep)=0$ in $\br^d$,
it follows by Theorem \ref{Ca-theorem-0} that
$$
\int_{\br^d} |\nabla u_\varep|^2 |\varphi|^2\, dx \le C \int_{\br^d} |u_\varep|^2 |\nabla \varphi|^2\, dx
$$
for any $\varphi\in C_0^1(\br^d)$. Thus, if
$\varphi\in C_0^1(B(x_0, 2\varep))$,
$$
\aligned
\int_{B(x_0, 2\varep)} |\nabla \chi(x/\varep)|^2 |\varphi|^2\, dx
& \le C \int_{B(x_0, 2\varep)} |\varphi|^2  + C \varep^2 \int_{B(x_0, 2\varep)} |\nabla \varphi|^2\, dx\\
& \qquad\qquad+C \varep^2 \int_{B(x_0, 2\varep)} |\chi(x/\varep)|^2 |\nabla \varphi|^2\, dx.
\endaligned
$$
Let $\varphi=\psi \widetilde{\eta}_\varep$,
where $\psi\in C_0^1(\br^d)$ and $\widetilde{\eta}_\varep$ is a cut-off function in $C_0^1(B(x_0, 2\varep))$ with the properties
that $0\le \widetilde{\eta}_\varep\le 1$, $\widetilde{\eta}_\varep=1$ on $B(x_0, \varep)$ and
$|\nabla \widetilde{\eta}_\varep|\le C/\varep$.
We obtain 
\begin{equation}\label{c-2-3-2}
\aligned
\int_{B(x_0, \varep)} |\nabla \chi(x/\varep)|^2 |\psi|^2\, dx
& \le C \int_{B(x_0, 2\varep)} (1+|\chi(x/\varep)|^2)  |\psi|^2 \, dx \\
&\qquad
+ C \varep^2 \int_{B(x_0, 2\varep)}
(1+|\chi(x/\varep)|^2)  |\nabla \psi|^2\, dx.
\endaligned
\end{equation}
By integrating the inequality above in $x_0$ over $\br^d$
we see that
\begin{equation}\label{c-2-3-3}
\int_{\br^d} |\nabla \chi(x/\varep)|^2 |\psi|^2\, dx
 \le C \int_{\br^d} (1+|\chi(x/\varep)|^2)  |\psi|^2  \, dx
+ C \varep^2 \int_{\br^d}
(1+|\chi(x/\varep)|^2)  |\nabla \psi|^2\, dx
\end{equation}
for any  $\psi\in C_0^1(\br^d)$.

If $\chi$ is bounded, it follows from (\ref{c-2-3-3}) that
$$
\|\nabla\chi(x/\varep) \psi\|_{L^2(\Omega)}
\le C (1+\|\chi\|_\infty) \Big\{ \| \psi\|_{L^2(\br^d)} +\varep\, \|\nabla \psi\|_{L^2(\br^d)} \Big\}
$$
for any $\psi\in C_0^1(\br^d)$.
By a limiting argument, the inequality holds for any $\psi\in H^1(\mathbb{R}^d)$.
This gives (\ref{c-2-3-1}), using the fact that for any $\psi\in H^1(\Omega)$, one may extend it to
a function $\widetilde{\psi}$ in $H^1(\br^d)$ so that
$\|\widetilde{\psi}\|_{L^2(\br^d)} \le C \, \| \psi\|_{L^2(\Omega)}$ and
$\|\widetilde{\psi}\|_{H^1(\br^d)} \le C \, \| \psi\|_{H^1(\Omega)}$ \cite{Stein-1970}.

Finally, recall that if $d\ge 3$ and $m\ge 2$,
$|\chi| \in L^q(Y)$ for $q=\frac{2d}{d-2}$.
For any $\psi\in W^{1, d}(\Omega)$,
we may extend it to a function $\widetilde{\psi}$ with compact support in $W^{1, d}(\br^d)$ so that
$$
\|\widetilde{\psi}\|_{L^d (\br^d)} \le C \, \| \psi\|_{L^d(\Omega)},\ \ 
\|\widetilde{\psi}\|_{W^{1, d}(\br^d)} \le C \, \| \psi\|_{W^{1,d}(\Omega)},
$$ 
and $\widetilde{\psi} (x)=0$ if dist$(x, \Omega)\ge 1$.
In view of (\ref{c-2-3-3}) we obtain 
$$
\aligned
\|\nabla \chi(x/\varep) \psi\|_{L^2(\Omega)}
& \le C \Big\{ \|\widetilde{\psi}\|_{L^d(\br^d)}
+\varep\, \|\nabla \widetilde{\psi}\|_{L^d(\br^d)}\Big\}\\
&\le C  \Big\{ \|{\psi}\|_{L^d(\Omega)}
+\varep\, \|\nabla {\psi}\|_{L^d(\Omega)}\Big\},\\
\endaligned
$$
using H\"older's  inequality and the fact that
$$
\|\chi(x/\varep)\|_{L^q(B(x_0, R))} \le C R^{d/q} \|\chi\|_{L^q(Y)}.
$$
This completes the proof.
\end{proof}

\begin{proof}[\bf Proof of Theorem \ref{theorem-d-c-2-1}]
Suppose $\chi$ is bounded.
To prove (\ref{c-2-1-1}),
in view of (\ref{c-2-2-0}),
 it suffices to show that
\begin{equation}\label{diff-100}
 \|\varep \chi(x/\varep)\nabla u_0 -\varep \chi(x/\varep) \eta_\varep
S^2_\varep (\nabla {u}_0)\|_{H^1(\Omega)}
\le C \sqrt{\varep} \, \| u_0\|_{H^2(\Omega)}.
\end{equation}
To this end, we note that the LHS of (\ref{diff-100}) is bounded by
$$
\aligned
& C \varep  \| \chi(x/\varep) \big(\nabla u_0 -\eta_\varep
S^2_\varep(\nabla {u}_0)\big) \|_{L^2(\Omega)}
+ C  \|\nabla \chi(x/\varep) \big(\nabla u_0 -\eta_\varep
S^2_\varep(\nabla {u}_0)\big) \|_{L^2(\Omega)}\\
&\qquad
+C  \varep 
\|\chi(x/\varep) \nabla \big(\nabla u_0 -\eta_\varep
S^2_\varep(\nabla {u}_0)\big) \|_{L^2(\Omega)}\\
&\le C \varep \| \nabla (\nabla u_0 -\eta_\varep S^2_\varep (\nabla u_0)) \|_{L^2(\Omega)}
+ C \|\nabla u_0 -\eta_\varep S^2_\varep (\nabla u_0)\|_{L^2(\Omega)},
\endaligned
$$
where we have used (\ref{c-2-3-1}).
Note that 
$$
\aligned
 \varep \| \nabla (\nabla u_0 -\eta_\varep S^2_\varep (\nabla u_0)) \|_{L^2(\Omega)}
 &\le \varep \|\nabla^2 u_0\|_{L^2(\Omega)}
 +\varep \|\nabla \big(\eta_\varep S^2_\varep (\nabla u_0)\big) \|_{L^2(\Omega)}\\
& \le C \varep \| u_0\|_{H^2(\Omega)}
 + C \| \nabla u_0\|_{L^2(\Omega_{5\varep})}\\
 &\le C \sqrt{\varep}\, \| u_0\|_{H^2(\Omega)},
 \endaligned
 $$
and
 $$
 \aligned
 \|\nabla u_0 -\eta_\varep S^2_\varep (\nabla u_0)\|_{L^2(\Omega)}
 &\le C \| \nabla u_0\|_{L^2(\Omega_{5\varep})}
+ \|\nabla u_0 -S^2_\varep (\nabla u_0)\|_{L^2(\Omega\setminus \Omega_{4\varep})}\\
&\le C \sqrt{\varep}\, \| u_0\|_{H^2(\Omega)},
\endaligned
$$
where we also used (\ref{diff-2.2}) for the last inequality.

To prove (\ref{c-2-1-0}) under the assumption $u_0\in W^{2, d}(\Omega; \br^m)$, 
we extend $u_0$ to a function $\widetilde{u}_0$
in $W^{2, d}(\br^d; \br^m)$ so that $\|\widetilde{u}_0\|_{W^{2, d}(\br^d)}\le C \| {u}_0\|_{W^{2, d}(\Omega)}$
Observe that
$$
\aligned
& \|\varep \chi(x/\varep)\nabla u_0 -\varep \chi(x/\varep) S^2_\varep (\nabla \widetilde{u}_0)\|_{H^1(\Omega)}\\
&\le C\varep  \| \chi(x/\varep) \big(\nabla u_0 -S^2_\varep(\nabla \widetilde{u}_0)\big) \|_{L^2(\Omega)}
+ C  \|\nabla \chi(x/\varep) \big(\nabla u_0 -S^2_\varep(\nabla \widetilde{u}_0)\big) \|_{L^2(\Omega)}\\
&\qquad
+C  \varep
\|\chi(x/\varep) \big(\nabla^2 u_0 -S^2_\varep(\nabla^2 \widetilde{u}_0)\big) \|_{L^2(\Omega)}\\
&\le C  \varep \| \widetilde{u}_0\|_{W^{2, d}(\mathbb{R}^d)}
+C\,  \|\nabla \widetilde{u}_0 -S^2_\varep(\nabla \widetilde{u}_0)\|_{L^d(\rd)}\\
&\le C  \varep \| u_0\|_{W^{2, d}(\Omega)},
\endaligned
$$
where we have used H\"older's inequality and (\ref{c-2-3-0}) for the second inequality and
(\ref{1.5.4-0}) for the last.
Also, it follows from (\ref{bl-estimate-2}) that
$$
\aligned
 \| \varep \chi(x/\varep) S^2_\varep (\nabla \widetilde{u}_0)
-\varep \chi(x/\varep) \eta_\varep S^2_\varep (\nabla u_0)\|_{H^1(\Omega)}
& =\varep \| \chi(x/\varep) S^2_\varep(\nabla \widetilde{u}_0) (1-\eta_\varep)\|_{H^1(\Omega)}\\
& \le C \sqrt{\varep}\, \| \widetilde{u}_0\|_{H^2(\mathcal{O})}\\
& \le C \sqrt{\varep} \, \| \widetilde{u}_0\|_{W^{2, d} (\mathbb{R}^d)}\\
&
\le C \sqrt{\varep}\, \| u_0\|_{W^{2, d}(\Omega)},
\endaligned
$$
where $\mathcal{O}=\{ x\in \Omega: \, \text{dist}(x, \Omega)<1 \}$.
As a result, we have proved that
$$
\| \varep \chi(x/\varep) \nabla u_0
-\varep \chi(x/\varep) \eta_\varep S^2_\varep (\nabla u_0)\|_{H^1(\Omega)}
\le C\sqrt{\varep}\, \| u_0\|_{W^{2, d}(\Omega)}.
$$
This, together with (\ref{c-2-2-0}),
yields the estimate (\ref{c-2-1-0}).
\end{proof}

\begin{remark}\label{r-c-2-10}
{\rm 
For $\varep\ge 0$, let $u_\varep\in H^1(\Omega; \br^m)$ be the weak solution to the
Dirichlet problem
\begin{equation}\label{D-C-2}
\mathcal{L}_\varep (u_\varep)=F_\varep \quad \text{ in } \Omega
\quad \text{ and } \quad u_\varep =f_\varep \quad \text{ on } \partial\Omega,
\end{equation}
where $F_\varep \in H^{-1}(\Omega; \br^m)$ and $f_\varep \in H^{1/2}(\Omega; \br^m)$.
Then
\begin{equation}\label{r-c-2-11}
\aligned 
 & \| u_\varep -u_0-
\varep \chi(x/\varep) \nabla u_0 \|_{H^1(\Omega)}\\
&
 \le \left\{
 \aligned
 & C \, \Big\{
 \sqrt{\varep} \, \| u_0\|_{W^{2, d}(\Omega)} +\| F_\varep -F_0\|_{H^{-1}(\Omega)}
 +\| f_\varep -f_0\|_{H^{1/2}(\partial\Omega)} \Big\},\\
 & C \,\Big\{ \sqrt{\varep}\,  \| u_0\|_{H^2(\Omega)}
+\| F_\varep -F_0\|_{H^{-1}(\Omega)}
 +\| f_\varep -f_0\|_{H^{1/2}(\partial\Omega)} \Big\},
 \text{ if  $\chi$ is bounded.}
 \endaligned
 \right.
 \endaligned
 \end{equation}
To see (\ref{r-c-2-11}) one applies Theorem \ref{theorem-d-c-2-1} to the weak solution of
$\mathcal{L}_\varep (v_\varep)=F_0$ in $\Omega$ with $v_\varep=f_0$ on $\partial\Omega$
and uses Theorem \ref{theorem-1.1-2} to estimate
$\| u_\varep -v_\varep\|_{H^1(\Omega)}$.
}
\end{remark}



\section{Convergence rates in $H^1$ for  Neumann problem}\label{section-c-3}

In this section we extend the results in Section \ref{section-c-2} to
 solutions of the Neumann problem (\ref{N-C}).
 Throughout this section we assume that $A$ satisfies the Legendre ellipticity 
 condition (\ref{s-ellipticity}).

\begin{lemma}\label{lemma-c-3-1s}
Let $u_\varep$ be the solution of (\ref{N-C}) with
$ \int_\Omega u_\varep \, dx =0$,
and $u_0$ the homogenized solution. 
Let $w_\varep$ be defined as in (\ref{w-c-2}).
Then the inequality (\ref{c-2-1-00}) holds  for any $\psi\in H^1(\Omega; \br^m)$.
\end{lemma}

\begin{proof}
Since $w_\varep\in H^1(\Omega; \br^m)$, it
suffices to prove (\ref{c-2-1-00}) for $\psi\in C^\infty(\br^d; \br^m)$.
Using
$$
\int_\Omega A(x/\varep)\nabla u_\varep \cdot \nabla \psi\, dx
=\int_\Omega \widehat{A} \nabla u_0\cdot \nabla \psi\, dx
$$
for any $\psi\in C^\infty(\br^d; \br^m)$, we see that
the inequality (\ref{c-2-1-2}) continues to hold
for any $\psi\in C^\infty(\br^d; \br^m)$.
The rest of the proof is exactly the same as that of Lemma \ref{lemma-c-2-1}.
\end{proof}

The following theorem gives the $O(\sqrt{\e})$ convergence rate for the Neumann problem.

\begin{thm}\label{main-thm-2.3a}
Assume that $A$ is 1-periodic and satisfies (\ref{s-ellipticity}).
Let $\Omega$ be a bounded Lipschitz domain in $\mathbb{R}^d$.
Let $w_\varep$ be the same as in Lemma \ref{lemma-c-3-1s}.
Then
\begin{equation}\label{thm-2.3-1a}
\| w_\e\|_{H^1(\Omega)}
\le C
\Big\{ \e \|\nabla^2 u_0\|_{L^2(\Omega\setminus \Omega_{\e})}
+\e \| \nabla u_0\|_{L^2(\Omega)}
+\| \nabla u_0 \|_{L^2(\Omega_{5\e})}\Big\}
\end{equation}
for $0<\e<1$. Consequently, if $u_0\in H^2(\Omega; \br^m)$,
\begin{equation}\label{c-2-3-000}
\| w_\e\|_{H^1(\Omega)}
\le C \sqrt{\e} \| u_0\|_{H^2(\Omega)}.
\end{equation}
The constant  $C$ depends only on $\mu$ and $\Omega$.
\end{thm}

\begin{proof}
Since $w_\e\in H^1(\Omega; \br^m)$,  by Lemma \ref{lemma-c-3-1s},
we may take $\psi=w_\e$ in (\ref{c-2-1-00}).
This, together with the ellipticity condition (\ref{s-ellipticity}), gives
\begin{equation}\label{thm-2.2-1-00a}
\aligned
\|\nabla w_\e\|_{L^2(\Omega)}
 &\le C
\Big\{ \e \|\nabla^2 u_0\|_{L^2(\Omega\setminus \Omega_{2\e})}
+\| \nabla u_0 -S_\e (\nabla u_0)\|_{L^2(\Omega\setminus \Omega_{2\e})}
+\| \nabla u_0 \|_{L^2(\Omega_{5\e})}\Big\}\\
&\le C
\Big\{ \e \|\nabla^2 u_0\|_{L^2(\Omega\setminus \Omega_{\e})}
+\| \nabla u_0 \|_{L^2(\Omega_{5\e})}\Big\},
\endaligned
\end{equation}
where the second inequality follows from the same argument as  in the proof of Theoem \ref{m-thm-2.2-1}.
 By Poincar\'e inequality  we obtain 
$$
\aligned
\| w_\e\|_{H^1(\Omega)}
&\le C \|\nabla w_\e\|_{L^2(\Omega)} 
+ C \Big| \int_\Omega \e \chi (x/\e) \eta_\e S_\e^2 (\nabla u_0) \, dx \Big|\\
&\le C \|\nabla w_\e\|_{L^2(\Omega)}  + C \e \| \nabla u_0\|_{L^2(\Omega)},
\endaligned
$$
where we have used the fact that $\int_\Omega u_\e\, dx =\int_\Omega u_0\, dx =0$.
Estimates (\ref{thm-2.3-1a}) and (\ref{c-2-3-000}) now follow from (\ref{thm-2.2-1-00a}).
\end{proof}

The estimates in Theorem \ref{main-thm-2.3a} can be improved under the additional 
symmetry condition.

\begin{lemma}\label{lemma-Lip-ns}
Assume that  $\big(\widehat{A})^*=\widehat{A}$.
Let $u\in H^1(\Omega; \br^m)$ be a weak solution to the Neumann problem:
$\mathcal{L}_0 (u)=0$ in $\Omega$ and $\frac{\partial u}{\partial \nu_0}
=g$ on $\partial\Omega$,
where $g\in L^2(\partial\Omega; \br^m)$ and $\int_{\partial \Omega} g\, d\sigma =0$.
Then 
\begin{equation}\label{Lip-d-n-s}
\|(\nabla u)^*\|_{L^2(\partial\Omega)}
\le C \| g\|_{L^2(\partial\Omega)},
\end{equation}
where $C$ depends only on $\mu$ and $\Omega$.
\end{lemma}

\begin{proof}
This was proved in \cite{Fabes-1988, DKV-1988}.
See Chapter \ref{chapter-7}.
\end{proof}

\begin{thm}\label{main-thm-2.3}
Suppose that $A$ is 1-periodic and satisfies (\ref{s-ellipticity}).
Also assume that $A^*=A$.
Let $\Omega$ be a bounded Lipschitz domain in $\mathbb{R}^d$.
Let $w_\varep$ be the same as in Lemma \ref{lemma-c-3-1s}.
Then
\begin{equation}\label{c-2-3-00s}
\| w_\varep\|_{H^1(\Omega)}
\le C \sqrt{\varep}\,
\Big\{ \| F\|_{L^q(\Omega)} +\| g\|_{L^2(\partial\Omega)} \Big\},
\end{equation}
where $q=\frac{2d}{d+1}$  and
$C$ depends only on $\mu$ and $\Omega$. 
\end{thm}

\begin{proof} 
It follows from Lemma \ref{lemma-c-3-1s} by letting $\psi=w_\e\in H^1(\Omega; \br^m)$ and using (\ref{s-ellipticity}) that
\begin{equation}\label{c-2-3-01s}
\|\nabla w_\varep  \|_{L^2(\Omega)}
 \le C 
\Big\{ \varep \|S_\varep (\nabla^2 u_0)\|_{L^2(\Omega\setminus \Omega_{2\varep})}
+\| \nabla u_0 -S_\varep (\nabla u_0)\|_{L^2(\Omega\setminus \Omega_{2\varep})}
+\|\nabla u_0\|_{L^2(\Omega_{5\varep})} \Big\}.
\end{equation}
To prove (\ref{c-2-3-00s}) under the assumption $A^*=A$,
it suffices to bound the RHS of (\ref{c-2-3-01s}) by the RHS of (\ref{c-2-3-00s}).
We proceed as in the proof of Theorem \ref{main-thm-2.2}
by writing $u_0 =v_0 +\phi$, where $v_0$ is given by (\ref{v-0}).
The terms involving $v_0$ are handled exactly in the same manner as before.
Similarly, to control the terms involving $\phi$,
it suffices to bound the $L^2(\partial\Omega)$ norm of
$\mathcal{M}(\nabla \phi)$. To do this, we note that
 $$
 \mathcal{L}_0(\phi)=0 \quad \text{ in }
 \Omega \quad \text{ and } \quad
\frac{\partial \phi}{\partial \nu_0}=g -\frac{\partial v_0}{\partial\nu_0} \quad
\text{ on } \partial\Omega.
$$
It follows by Lemma \ref{lemma-Lip-ns}  that
$$
\|(\nabla \phi)^*\|_{L^2(\partial\Omega)}
\le C\| g\|_{L^2(\partial\Omega)} + C\|\nabla v_0\|_{L^2(\partial\Omega)}.
$$
 Note that
 $
 \|\nabla v_0\|_{L^2(\partial\Omega)}\le C \|v\|_{W^{2, q}(\Omega)} \le C\, \| F\|_{L^q(\Omega)}.
 $
  Thus,
 $$
 \|(\nabla \phi)^*\|_{L^2(\partial\Omega)}
\le C\Big\{ \| g\|_{L^2(\partial\Omega)} + \| F\|_{L^q(\Omega)} \Big\}.
$$
This completes the proof.
\end{proof}

\begin{thm}\label{theorem-c-3s}
Assume that $A$ is 1-periodic and satisfies (\ref{s-ellipticity}).
Let $\Omega$ be a bounded Lipschitz domain in $\br^d$.
Let $u_\varep\in H^1(\Omega; \br^m)$ be the weak solution of
the Neumann problem (\ref{N-C}) with $\int_\Omega u_\e\, dx =0$.
Let $u_0$ be the homogenized solution.
Then, if $u_0\in W^{2, d}(\Omega; \br^m)$ and $0<\e<1$,
\begin{equation}\label{c-3-3-0s}
\| u_\varep -u_0 -\varep \chi(x/\varep) \nabla u_0\|_{H^1(\Omega)}
\le C \sqrt{\varep}\,  \| u_0\|_{W^{2, d}(\Omega)},
\end{equation}
where $C$ depends only on $\mu$ and $\Omega$.
Furthermore, if the corrector $\chi$ is bounded and $u_0\in H^2(\Omega; \br^m)$, then
\begin{equation}\label{c-3-3-1s}
\| u_\varep -u_0 -\varep \chi(x/\varep) \nabla u_0\|_{H^1(\Omega)}
\le C \sqrt{\varep}\,  \| u_0\|_{H^2(\Omega)}
\end{equation}
for any $0<\varep<1$,
where $C$ depends only on $\mu$, $\|\chi\|_\infty$ and $\Omega$.
\end{thm}

\begin{proof}
The proof is exactly the same as that of Theorem \ref{theorem-d-c-2-1}, where
the Dirichlet boundary condition was never used.
\end{proof}

\begin{remark}\label{remark-c-3s}
{\rm
For $\varep\ge 0$, let $u_\varep\in H^1(\Omega; \br^m)$ be the weak solution to the Neumann problem
\begin{equation}\label{N-C-3s}
\mathcal{L}_\varep (u_\varep)=F_\varep \quad \text{ in } \Omega
\quad \text{ and } \quad \frac{\partial u_\varep}{\partial\nu_\varep}=g_\varep,
\quad \text{ on } \partial\Omega,
\end{equation}
with $\int_\Omega u_\e\, dx =0$,
where $F_\varep\in L^2(\Omega; \br^m)$, $g_\varep\in L^2(\partial\Omega; \br^m)$ and
$\int_{\partial\Omega} g_\e\, d\sigma=0$.
Then 
\begin{equation}\label{c-3-4-0s}
\aligned 
 & \| u_\varep -u_0-
\varep \chi(x/\varep) \nabla u_0 \|_{H^1(\Omega)}\\
&
 \le \left\{
 \aligned
 & C \, \Big\{
 \sqrt{\varep} \, \| u_0\|_{W^{2, d}(\Omega)} +\| F_\varep -F_0\|_{H^{-1}_0(\Omega)}
 +\| g_\varep -g_0\|_{H^{-1/2}(\partial\Omega)} \Big\},\\
 & C \,\Big\{ \sqrt{\varep}\,  \| u_0\|_{H^2(\Omega)}
+\| F_\varep -F_0\|_{H^{-1}_0(\Omega)}
 +\| g_\varep -g_0\|_{H^{-1/2}(\partial\Omega)} \Big\},
 \text{ if  $\chi$ is bounded,}
 \endaligned
 \right.
 \endaligned
 \end{equation}
 where $H_0^{-1}(\Omega; \br^m)$ denotes the dual of $H^1(\Omega; \br^m)$.
 
 To see (\ref{c-3-4-0s}), we use
 $$
 \| u_\varep -u_0-
\varep \chi(x/\varep) \nabla u_0 \|_{H^1(\Omega)}
\le \| v_\varep -u_0-
\varep \chi(x/\varep) \nabla u_0 \|_{H^1(\Omega)} +\| u_\varep -v_\varep\|_{H^1(\Omega)},
$$
where $v_\varep$ is the solution of
 $$
 \mathcal{L}_\varep (v_\varep)=F_0\quad \text{ in } \Omega \quad \text{ and }\quad
 \frac{\partial v_\varep}{\partial\nu_\varep}=g_0,
 $$
such that $\int_\Omega v_\e\, dx=0$.
By Theorem \ref{theorem-c-3s},
 $\| v_\varep -u_0-\varep \chi(x/\varep)\nabla u_0\|_{H^1(\Omega)}$
is bounded by $C \sqrt{\varep} \,\| u_0\|_{W^{2, d}(\Omega)}$,
and by $C\sqrt{\varep}\, \| u_0\|_{H^2(\Omega)}$ if $\chi$ is bounded.
Since $\mathcal{L}_\varep (u_\varep -v_\varep)=F_\varep -F_0$ in $\Omega$
and $\frac{\partial }{\partial \nu_\varep}  (u_\varep -v_\varep)=g_\varep-g_0$
on $\partial\Omega$, by Theorem \ref{s-NP-theorem},
we see that
$$
\| u_\varep -v_\varep\|_{H^1(\Omega)}
\le C\, \Big\{ \| F_\varep -F_0\|_{H^{-1}_0(\Omega)}
+\| g_\varep -g_0\|_{H^{-1/2}(\partial\Omega)}\Big\}.
$$
}
\end{remark}



\section{Convergence rates in $L^p$
 for Dirichlet problem}\label{section-c-4}

In this section we establish  a sharp $O(\varep)$
 convergence rate in $L^p$ with
 $p=\frac{2d}{d-1}$ for Dirichlet problem (\ref{D-C}),
under the assumption that  $A$ satisfies (\ref{weak-e-1})-(\ref{weak-e-2}) and
$A^*=A$.
Without the symmetry condition we obtain an $O(\e)$ convergence rate in $L^2$.

\begin{lemma}\label{lemma-c-4-0}
Let $\Omega$ be a bounded Lipschitz domain in $\mathbb{R}^d$.
Let $w_\varep$ be given by (\ref{w-c-2}), where $u_\varep$ is the weak solution of (\ref{D-C})
and $u_0$ the homogenized solution.
Assume that $u_0\in W^{2, q}(\Omega; \mathbb{R}^m)$ for $q=\frac{2d}{d+1}$.
Then, for any $\psi\in C_0^\infty (\Omega; \mathbb{R}^m)$,
\begin{equation}\label{c-4-0-0}
\Big|\int_{\mathbb{R}^d} A(x/\varep)\nabla w_\varep \cdot \nabla \psi\, dx \Big|
\le C \| u_0\|_{W^{2, q}(\Omega)}
\Big\{ \varep \| \nabla \psi\|_{L^p(\Omega)}
+\sqrt{\varep}\, \|\nabla \psi\|_{L^2(\Omega_{4\varep})} \Big\},
\end{equation}
where $p=q^\prime=\frac{2d}{d-1}$ and
$C$ depends only on $\mu$  and $\Omega$.
\end{lemma}

\begin{proof}
An inspection of the proof of Lemma \ref{lemma-c-2-1} shows that
\begin{equation}\label{c-2-1-000}
\aligned
 \Big|\int_\Omega & A(x/\varep) \nabla w_\varep \cdot \nabla \psi\, dx \Big|\\
& \le C \|\nabla \psi\|_{L^p(\Omega)}
\Big\{ \varep \| S_\varep(\nabla^2 u_0)\|_{L^q(\Omega\setminus \Omega_{3\varep})}
+\|\nabla u_0 -S_\varep (\nabla u_0)\|_{L^q(\Omega\setminus \Omega_{2\varep})}\Big\}\\
&\qquad\qquad
+ C \|\nabla \psi \|_{L^2(\Omega_{4\varep})}
\| \nabla u_0\|_{L^2(\Omega_{5\varep})}.
\endaligned
\end{equation}
Note that $\|\nabla u_0\|_{L^2(\Omega_{5\varep})}\le C \sqrt{\varep}\, \| u_0\|_{W^{2, q}(\Omega)}$ and
$$
\| S_\varep (\nabla^2  u_0)\|_{L^q(\Omega\setminus \Omega_{3\varep})}
\le C\| \nabla^2 u_0\|_{L^q(\Omega)}.
$$
Since $u_0\in W^{2, q}(\Omega; \mathbb{R}^m)$,
there exists $\widetilde{u}_0\in W^{2, q}(\mathbb{R}^d; \mathbb{R}^m)$
such that $ \widetilde{u}_0=u_0$ in $\Omega$ and
$$
\| \widetilde{u}_0\|_{W^{2, q}(\mathbb{R}^d)}
\le C \| u_0\|_{W^{2, q}(\Omega)}.
$$
It follows that
\begin{equation}\label{c-4-0-1}
\aligned
\| \nabla u_0 -S_\varep(\nabla u_0)\|_{L^q(\Omega\setminus\Omega_{3\varep})}
& \le \| \nabla \widetilde{u}_0 -S_\varep (\nabla \widetilde{u}_0)\|_{L^q(\mathbb{R}^d)}\\
&\le C \varep \|\nabla^2 \widetilde{u}_0\|_{L^q(\mathbb{R}^d)}\\
&\le C \varep \| u_0\|_{W^{2, q}(\Omega)},
\endaligned
\end{equation}
where we have used (\ref{1.5.4-0}) for the second inequality.
In view of (\ref{c-2-1-000}) we have proved  (\ref{c-4-0-0}).
\end{proof}

\begin{lemma}\label{lemma-Lip-d-2}
Assume that $\big(\widehat{A}\big)^*=\widehat{A}$.
Let $\Omega$ be a bounded Lipschitz domain in $\br^d$.
Let $u \in H^1_0(\Omega; \br^m)$ be a weak solution to the Dirichlet problem:
$\mathcal{L}_0 (u)=G$ in $\Omega$ and $u=0$ on $\partial\Omega$,
where $G\in C_0^\infty(\Omega; \br^m)$.
Then
\begin{equation}\label{Lip-d-20}
\|\nabla u \|_{L^2(\Omega_t)}
\le C t^{1/2}\| G\|_{L^q(\Omega)},
\end{equation}
for $0<t<\text{\rm diam}(\Omega)$, and 
\begin{equation}\label{Lip-d-21}
\|\nabla u \|_{L^p(\Omega)}\le C\| G\|_{L^q(\Omega)},
\end{equation}
where $p=\frac{2d}{d-1}$,
$q=p^\prime=\frac{2d}{d+1}$,
and $C$ depends only on $\mu$ and $\Omega$.
\end{lemma}

\begin{proof}
Write $u=\phi +v_0$, where $v_0$ is defined by  (\ref{v-0}).
The proof of (\ref{Lip-d-20}) is essentially contained in that of Theorem \ref{m-thm-2.2-0}.
To see (\ref{Lip-d-21}), we note that by the fractional integral estimate,
$$
\| \nabla v_0\|_{L^p(\Omega)} \le C \| G\|_{L^q(\Omega)}.
$$
To estimate $\nabla \phi$, we use the following inequality 
\begin{equation}\label{max-sobolev}
\left(\int_\Omega |\psi|^p\, dx\right)^{1/p}
\le C \left(\int_{\partial\Omega} |(\psi)^*|^2\, d\sigma\right)^{1/2},
\end{equation}
where $p=\frac{2d}{d-1}$, $\psi$ is a continuous function in $\Omega$ and
$(\psi)^*$ denotes the nontangential maximal function of $\psi$.
This gives
$$
\aligned
\|\nabla \phi\|_{L^p(\Omega)}
& \le C \|(\nabla \phi)^*\|_{L^2(\partial\Omega)}
\le C \|\nabla v_0\|_{L^2(\partial\Omega)}\\
& \le C \| v_0\|_{W^{2, q}(\Omega)}
\le C \| G\|_{L^q(\Omega)},
\endaligned
$$
where we have used Lemma \ref{lemma-Lip-d} for the second inequality 
and (\ref{bl-estimate-1}) for the third.

Finally, to prove (\ref{max-sobolev}), we observe that  for any $x\in \Omega$ and
$\hat{x}\in \partial\Omega$ with $|x-\hat{x}|=\text{dist}(x, \partial\Omega)=r$,
$$
|\psi (x)|\le (\psi)^* (y),
$$
if  $y\in \partial\Omega$ and $|y-\hat{x}|\le c r$.
It follows that
\begin{equation}\label{max-s-1}
\aligned
|\psi (x)| & \le\frac{C}{r^{d-1}} \int_{B(\hat{x}, cr)\cap \partial\Omega} |(\psi)^*|\, d\sigma\\
&\le C \int_{\partial\Omega} \frac{(\psi)^*(y)}{|x-y|^{d-1}} \, d\sigma (y).
\endaligned
\end{equation}
Let $f\in C_0^1(\Omega)$ and 
$$
g(y)=\int_\Omega \frac{|f(x)|}{|x-y|^{d-1}} dx.
$$
Note that by (\ref{max-s-1}),
$$
\aligned
\Big|\int_{\Omega}  \psi (x) f(x)\, dx \Big|
&\le C \int_\Omega \int_{\partial\Omega} 
\frac{(\psi)^* (y) |f (x)|}{|x-y|^{d-1}}\, d\sigma (y) dx\\
&=C \int_{\partial \Omega} \mathcal{M}(\psi) g\, d\sigma\\
&\le C \| (\psi)^*\|_{L^2(\partial\Omega)}
\| g\|_{L^2(\partial\Omega)}\\
&\le C \|(\psi)^*\|_{L^2(\partial\Omega)}
\| g\|_{W^{1, q}(\Omega)}\\
&\le C \|(\psi)^*\|_{L^2(\partial\Omega)} \| f\|_{L^q(\Omega)},
\endaligned
$$
where $q=\frac{2d}{d+1}$ and we have used (\ref{bl-estimate-1}) for the third inequality and
singular integral estimates for the fourth.
The inequality (\ref{max-sobolev}) follows by duality.
\end{proof}

The next theorem gives a sharp $O(\e)$ convergence rate in $L^p$ with $p=\frac{2d}{d-1}$.

\begin{thm}\label{theorem-1.5.7}
Assume that $A$ is 1-periodic and satisfies (\ref{weak-e-1})-(\ref{weak-e-2}).
Also assume that $A^*=A$.
Let $\Omega$ be a bounded Lipschitz domain in $\rd$.
Let $u_\varep$ $(\varep\ge 0)$ be the weak solution of (\ref{D-C}).
Assume that $u_0\in W^{2, q}(\Omega; \br^d)$ for $q=\frac{2d}{d+1}$.
Then
\begin{equation}\label{1.5.7-0}
\| u_\varep -u_0\|_{L^p(\Omega)} \le C\varep \| u_0\|_{W^{2, q}(\Omega)},
\end{equation}
where $p=q^\prime=\frac{2d}{d-1}$ and $C$ depends only on $\mu$ and $\Omega$.
\end{thm}

\begin{proof}
Let $w_\varep$ be given by (\ref{w-c-2}).
For any $G\in C_0^\infty(\Omega; \br^m)$, let $v_\varep\in H_0^1(\Omega; \br^m)$ ($\varep\ge 0$)
be the weak solution to the Dirichlet problem,
\begin{equation}\label{c-4-12}
\mathcal{L}^*_\varep (v_\varep)=G \quad \text{ in } \Omega \quad
\text{ and } \quad  v_\varep=0 \quad \text{ on } \partial\Omega.
\end{equation}
Define
$$
r_\varep = v_\varep - v_0 -\varep \chi(x/\varep) \eta_\varep S^2_\varep (\nabla v_0),
$$
where 
 $\eta_\varep$ is a cut-off function satisfying (\ref{eta-e}).
Observe that 
\begin{equation}\label{c-4-13}
\aligned
 \Big|\int_\Omega  w_\varep \cdot G\, dx \Big|
& =\Big| \int_\Omega
A(x/\varep) \nabla w_\varep \cdot \nabla v_\varep\, dx \Big|\\
&\le \Big|\int_\Omega
A(x/\varep) \nabla w_\varep\cdot \nabla r_\varep\, dx \Big|
+\Big|\int_\Omega A(x/\varep)\nabla w_\varep \cdot \nabla v_0\, dx \Big|\\
& \qquad\qquad
+\Big|\int_\Omega
A(x/\varep)\nabla w_\varep \cdot
 \nabla \big( \varep \chi (x/\varep) \eta_\varep S^2_\varep(\nabla v_0)\big) \, dx \Big|\\
& =I_1 + I_2 + I_3.
\endaligned
\end{equation}

To estimate $I_1$, we note that by (\ref{c-2-2-0}) and (\ref{c-2-2-00}),
$$
\|\nabla w_\varep\|_{L^2(\Omega)}
\le C \sqrt{\varep}\, \| u_0\|_{W^{2, q}(\Omega)}
\quad \text{ and } \quad
\|\nabla r_\varep\|_{L^2(\Omega)}
\le C \sqrt{\varep}\, \| G\|_{L^q(\Omega)},
$$
where we have used the assumption $\big(\widehat{A})^*=\widehat{A}$.
By Cauchy inequality this gives
\begin{equation}\label{c-4-15}
I_1\le C \varep \| u_0\|_{W^{2, q}(\Omega)} \| G\|_{L^q(\Omega)}.
\end{equation}

Next, to bound $I_2$, we use Lemma \ref{lemma-c-4-0} to obtain 
\begin{equation}\label{c-4-16}
\aligned
I_2 & \le C  \| u_0\|_{W^{2, q}(\Omega)}
\Big\{ \varep \| \nabla v_0 \|_{L^p(\Omega)}
+\sqrt{\varep} \, 
\| \nabla v_0 \|_{L^2(\Omega_{4\varep})}\Big\}\\
&\le C \varep  \| u_0\|_{W^{2,q}(\Omega)} \| G\|_{L^q(\Omega)},
\endaligned
\end{equation}
where we use Lemma \ref{lemma-Lip-d-2} for the last step.

To estimate $I_3$, we let 
$$
\varphi_\varep=\varep \chi (x/\varep) \eta_\varep S^2_\varep (\nabla v_0).
$$
Using (\ref{1.5.3-0}) as well as the observation 
$\|\nabla S^2_\varep (f)\|_{L^p(\mathbb{R}^d)}
\le C\varep^{-1} \| f\|_{L^p(\mathbb{R}^d)}$,
it is not hard to show that
$$
\varep \|\nabla \varphi_\varep\|_{L^p(\Omega)}
+\sqrt{\varep}\, \|\nabla \varphi_\varep\|_{L^2(\Omega_{4\varep})}
\le C \varep \|\nabla v_0\|_{L^p(\Omega)}
+C \sqrt{\varep}\, \|\nabla v_0\|_{L^2(\Omega_{5\varep})}.
$$
As in the case of $I_2$, by Lemma \ref{lemma-c-4-0},
this implies that
\begin{equation}\label{c-4-17}
\aligned
I_3 &
\le C \| u_0\|_{W^{2, q}(\Omega)}
\Big\{ \varep \|\nabla \varphi_\varep\|_{L^p(\Omega)}
+\sqrt{\varep}\, \|\nabla \varphi_\varep\|_{L^2(\Omega_{4\varep})}\Big\}\\
&\le C\, \varep \, \| u_0\|_{W^{2, q}(\Omega)} \| G\|_{L^q(\Omega)}.
\endaligned
\end{equation}

In view of (\ref{c-4-13})-(\ref{c-4-17}) we have proved that
\begin{equation}\label{c-4-18}
\Big|\int_\Omega w_\varep \cdot G\, dx \Big|
\le C \varep \| u_0\|_{W^{2, q}(\Omega)} \| G\|_{L^q(\Omega)},
\end{equation}
where $C$ depends only on $\mu$  and $\Omega$.
By duality this implies that
$$
\| w_\varep\|_{L^p(\Omega)} \le C \varep \| u_0\|_{W^{2, q}(\Omega)}.
$$
It follows that
$$
\aligned
\| u_\varep -u_0\|_{L^p(\Omega)}
 &\le \| w_\varep\|_{L^p(\Omega)}
+\|\varep \chi(x/\varep) \eta_\varep S^2_\varep(\nabla {u}_0)\|_{L^p(\Omega)}\\
&\le C \varep \| u_0\|_{W^{2, q}(\Omega)}.
\endaligned
$$
\end{proof}

\begin{remark}\label{remark-2-scaling-rate}
{\rm
Since $p=\frac{2d}{d-1}$ and $q=\frac{2d}{d+1}$, we have
$$
\frac{1}{q}-\frac{1}{p}=\frac{1}{d}.
$$
It follows that the estimate (\ref{1.5.7-0}) is scaling-invariant.
Consequently,  the constant $C$ in the estimate can be made to be
independent of diam$(\Omega)$, if diam$(\Omega)$ is large.
Indeed, suppose that $\mathcal{L}_\e (u_\e) =\mathcal{L}_0 (u_0)$ in $\Omega_R$
and $u_\e=u_0$ on $\partial\Omega_R$, where $R>1$ and
$\Omega_R= \big\{ x\in \brd:  x/R \in \Omega \big\}$.
Let $v_\e (x)=u_\varep (Rx)$ and $v_0(x)=u_0(Rx)$ for $x\in \Omega$.
Then
$$
\mathcal{L}_{\frac{\varep}{R}} (v_\e)=\mathcal{L}_0 (v_0)
\quad \text{ in } \Omega \quad \text{ and } \quad
v_\e=v_0 \quad \text{ on } \partial\Omega.
$$
It follows that
$$
\| v_\e -v_0\|_{L^p(\Omega)} \le C \left(\frac{\e}{R}\right) \| v_0\|_{W^{2, q}(\Omega)}.
$$
Using $\frac{1}{q}-\frac{1}{p}=\frac{1}{d}$, 
by a change of variables, we see that 
$$
\| u_\varep -u_0\|_{L^p(\Omega_R)}
\le C \varep \| u_0\|_{W^{2,q}(\Omega_R)},
$$ 
where $C$ depends only on $\mu$ and $\Omega$.
}
\end{remark}

The next theorem gives the $O(\e)$ convergence rate in $L^2$
without the symmetry condition, assuming that $\Omega$ is
a bounded $C^{1,1}$ domain. The smoothness assumption on $\Omega$ ensures the $H^2$
estimates for $\mathcal{L}_0$.

\begin{thm}\label{C-2-thm-D}
Assume that $A$ is 1-periodic and satisfies (\ref{weak-e-1})-(\ref{weak-e-2}).
Let $\Omega$ be a bounded $C^{1,1}$ domain in $\br^d$.
Let $u_\e\in H^1(\Omega; \br^m)$ $(\e\ge 0)$ be the weak solution of 
(\ref{D-C}). Assume that $u_0\in H^2(\Omega; \br^m)$. Then
\begin{equation}\label{C-2-N-0}
\| u_\e -u_0\|_{L^2(\Omega)}
\le C \e\, \| u_0\|_{H^2(\Omega)},
\end{equation}
where $C$ depends only on $\mu$ and $\Omega$.
\end{thm}

\begin{proof}
The proof is similar to that of Theorem \ref{theorem-1.5.7}.
By Theorem \ref{m-thm-2.2-1},  the estimates
$$
\| w_\e \|_{H^1_0(\Omega)} \le C \sqrt{\e} \| u_0\|_{H^2(\Omega)}
\quad 
\text{ and } 
\quad
\| r_\e\|_{H^1_0(\Omega)} \le C \sqrt{\e}\| v_0\|_{H^2(\Omega)}
$$
hold without the symmetry condition.
It follows from the proof of Theorem \ref{theorem-1.5.7} that
\begin{equation}\label{C-2-thm-D1}
\Big|\int_\Omega w_\e\cdot G\, dx \Big|
\le C \e \| u_0\|_{H^2(\Omega)} \| v_0\|_{H^2(\Omega)}.
\end{equation}
Recall that $v_0\in H^1_0 (\Omega; \br^m)$ is a solution of
$\mathcal{L}_0 (v_0)=G$ in $\Omega$.
Since $\Omega$ is $C^{1,1}$,  it is known that $v_0\in H^2(\Omega; \br^m)$ and
 $\| v_0\|_{H^2(\Omega)}\le C\, \| G\|_{L^2(\Omega)}$
 (see e.g. \cite{Necas-2012}).
This, together with (\ref{C-2-thm-D1}), gives
$$
\Big|\int_\Omega w_\e\cdot G\, dx \Big|
\le C \e \| u_0\|_{H^2(\Omega)} \| G\|_{L^2(\Omega)}.
$$
By duality we obtain $\| w_\e\|_{L^2(\Omega)} \le C \e \| u_0\|_{H^2(\Omega)}$.
Thus
$$
\aligned
\| u_\e - u_0\|_{L^2(\Omega)}
 &\le C \e \| u_0\|_{H^2(\Omega)}
+ \| \e \chi(x/\e) \eta_\e S_\e^2 (\nabla u_0)\|_{L^2(\Omega)}\\
&\le C \e \| u_\e\|_{H^2(\Omega)},
\endaligned
$$
which completes the proof.
\end{proof}



\section{Convergence rates in $L^p$
 for Neumann problem}\label{section-c-5}

In this section we establish the $O(\e)$ convergence rate in $L^2$
 for the Neumann problem (\ref{N-C}) in a bounded $C^{1,1}$ domain $\Omega$, under the conditions that
 $A$ is 1-periodic and satisfies the Legendre condition (\ref{s-ellipticity}).
 With the additional symmetry condition, the $O(\e)$ rate is obtained in $L^p$,
with $p=\frac{2d}{d-1}$, in a bounded Lipschitz domain.

\begin{lemma}\label{lemma-c-5-0s}
Suppose that $A$ is 1-periodic and satisfies (\ref{s-ellipticity}).
Let $\Omega$ be a bounded Lipschitz domain.
Let $w_\varep$ be given by (\ref{w-c-2}),
where $u_\varep$ is the weak solution of (\ref{N-C}) and $u_0$ the homogenized solution.
Assume that $u_0\in W^{2, q}(\Omega; \br^m)$ for
$q=\frac{2d}{d+1}$. Then the inequality (\ref{c-4-0-0})
holds for any $\psi\in C_0^\infty (\br^d; \br^m)$.
\end{lemma}

\begin{proof}
The proof is exactly the same as that for Lemma \ref{lemma-c-4-0}.
\end{proof}

The symmetry condition is needed for the following lemma.

\begin{lemma}\label{lemma-c-5-1s}
Suppose that $A$ is 1-periodic and satisfies (\ref{s-ellipticity}).
Also assume that $A^*=A$.
Let $\Omega$ be a bounded Lipschitz domain.
Let $ u\in H^1(\Omega; \br^m)$ be a weak solution to the
Neumann problem: $\mathcal{L}_0 (u)=G$ in $\Omega$ and
$\frac{\partial u}{\partial \nu_0}=0$ on $\partial\Omega$,
where $G\in C^\infty(\br^d; \br^m)$ and
$\int_\Omega G\, dx =0$.
Then
\begin{equation}\label{Lip-n-20s}
\|\nabla u \|_{L^2(\Omega_t)}
\le C t^{1/2} \| G\|_{L^q(\Omega)},
\end{equation}
for $0<t<\text{\rm diam}(\Omega)$, and
\begin{equation}\label{Lip-n-21s}
\|\nabla u \|_{L^p(\Omega)}\le C\| G\|_{L^q(\Omega)},
\end{equation}
where $p=\frac{2d}{d-1}$,
$q=p^\prime=\frac{2d}{d+1}$,
and $C$ depends only on $\mu$ and $\Omega$.
\end{lemma}

\begin{proof}
The proof is similar to that of Lemma \ref{lemma-Lip-d-2}.
\end{proof}

\begin{thm}\label{theorem-c-5s}
Suppose that $A$ and $\Omega$ satisfy the same conditions as in Lemma \ref{lemma-c-5-1s}.
Let $u_\varep\in H^1(\Omega; \br^m)$ $(\varep\ge 0)$ be the weak solution of 
the Neumann problem (\ref{N-C}).
Then, if $u_0\in W^{2, q}(\Omega; \br^m)$,
\begin{equation}\label{c-5-0s}
\| u_\varep -u_0\|_{L^p(\Omega)} \le C \varep \| u_0\|_{W^{2, q}(\Omega)},
\end{equation}
where $q=p^\prime =\frac{2d}{d+1}$ and
$C$ depends only on $\mu$ and $\Omega$.
\end{thm}

\begin{proof}
The proof of (\ref{c-5-0}) is similar to that in the case of the Dirichlet boundary condition,
using Theorem  \ref{main-thm-2.3}, \ref{lemma-c-5-0s}, \ref{lemma-c-5-1s}, and a duality argument.

Fix $G\in C^\infty (\br^d; \br^m)$ with $\int_\Omega G\, dx =0$.
Let $v_\varep\in H^1(\Omega; \br^d)$ ($\varep\ge 0$)
be the weak solution to the Neumann problem,
\begin{equation}\label{c-5-1}
\mathcal{L}_\varep (v_\varep) =G \quad \text{ in } \Omega \quad 
\text{ and } \quad \frac{\partial v_\varep}{\partial \nu_\varep} =0\quad \text{ on } \partial\Omega,
\end{equation}
with $\int_\Omega v_\e\, dx =0$.
As in the case of Dirichlet problem, we may show that
$$
\Big|\int_\Omega w_\varep\cdot G\, dx \Big|
\le C\varep \| u_0\|_{W^{2, q}(\Omega)} \| G\|_{L^q(\Omega)},
$$
where $w_\varep =u_\varep -u_0 -\varep \chi(x/\varep)\eta_\varep S^2_\varep (\nabla u_0)$.
Using
$$
\|\varep \chi(x/\varep) \eta_\varep S_\varep^2 (\nabla u_0)\|_{L^p(\Omega)}
\le C \varep \| u_0\|_{W^{2, q}(\Omega)},
$$
we then obtain 
$$
\Big|\int_\Omega ( u_\varep -u_0) \cdot G\, dx \Big|
\le C\varep \| u_0\|_{W^{2, q}(\Omega)} \| G\|_{L^q(\Omega)}.
$$
Since $\int_\Omega (u_\e -u_0)\, dx =0$,
by duality, this gives the estimate (\ref{c-5-0s}).
\end{proof}

Without the symmetry condition, as in the case of Dirichlet problem,
an $O(\e)$ convergence rate is obtained in $L^2$ in a $C^{1, 1}$ domain.

\begin{thm}\label{C-2-thm-N}
Assume that $A$ is 1-periodic and satisfies (\ref{s-ellipticity}).
Let $\Omega$ be a bounded $C^{1,1}$ domain in $\br^d$.
Let $u_\e\in H^1(\Omega; \br^m)$ $(\e\ge 0)$ be the weak solution of 
(\ref{D-C}). Assume that $u_0\in H^2(\Omega; \br^m)$. Then
\begin{equation}\label{C-2-D-0}
\| u_\e -u_0\|_{L^2(\Omega)}
\le C \e\, \| u_0\|_{H^2(\Omega)},
\end{equation}
where $C$ depends only on $\mu$ and $\Omega$.
\end{thm}

\begin{proof}
The proof is similar to that of Theorem \ref{C-2-thm-D}.
The details are left to the reader.
\end{proof}


 
\section{Convergence rates for elliptic systems of elasticity}\label{elasticity-2}

In this section we study the convergence rates for elliptic systems of
linear elasticity. Since the elasticity condition (\ref{ellipticity}) implies
(\ref{weak-e-1})-(\ref{weak-e-2}) and the symmetry condition.
Results obtained in Sections \ref{section-c-2} and \ref{section-c-4}
for the Dirichlet problem hold for the system of elasticity. 

\begin{thm}\label{ECD}
Assume that $A\in E(\kappa_1, \kappa_2)$ and  is 1-periodic.
Let $\Omega$ be a bounded Lipschitz domain in $\br^d$.
Let $w_\e$ be defined by (\ref{w-c-2}), where $u_\e$ is the weak solution to the
Dirichlet problem, $\mathcal{L}_\e (u_\e)=F$ in $\Omega$ 
and $u_\e =f $ on $\partial\Omega$. Then, for 
$0<\e<1$,
\begin{equation}\label{ECD-0}
\| w_\e \|_{H^1(\Omega)}
\le C \sqrt{\e} \Big\{ \| F\|_{L^q(\Omega)} +\| f\|_{H^1(\partial\Omega)} \Big\},
\end{equation}
and
\begin{equation}\label{ECD-1}
\| u_\e -u_0\|_{L^p(\Omega)}
\le C\e \| u_0\|_{W^{1, q}(\Omega)},
\end{equation}
where $q=\frac{2d}{d+1}$, $p=\frac{2d}{d-1}$, and $C$ depends only on $\kappa_1$, $\kappa_2$, and $\Omega$.
\end{thm}

In the following we extend the results in Sections \ref{section-c-3} and \ref{section-c-5}
to the Neumann problem,
\begin{equation}\label{N-CE}
\left\{
\aligned
\mathcal{L}_\varep (u_\varep) & =F &  &  \text{ in } \Omega,\\
\frac{\partial u_\varep}{\partial \nu_\varep} & = g &  & \text{ on } \partial\Omega,\\
 u_\varep   \perp \mathcal{R}   & \text{ in }  L^2(\Omega; \br^d),
\endaligned
\right.
\end{equation}
where $F \in L^2(\Omega; \br^d)$ and
$g\in L^2(\partial\Omega; \br^d)$ satisfy the compatibility condition (\ref{compatibility}),
and $\mathcal{R}$ denotes the space of rigid displacements, given by (\ref{rigid}).

\begin{lemma}\label{lemma-c-3-1}
Let $u_\e$ be the solution of (\ref{N-CE})
and $u_0$ the homogenized solution. 
Let $w_\varep$ be defined as in (\ref{w-c-2}).
Then the inequality (\ref{c-2-1-00}) holds  for any $\psi\in H^1(\Omega; \br^d)$.
\end{lemma}

\begin{proof}
The proof is similar to that of Lemma \ref{lemma-c-3-1s}.
\end{proof}

Let
\begin{equation}\label{L-R}
L^p_{\mathcal{R}}(\partial\Omega)
=\Big\{ g\in L^p(\partial\Omega; \br^d): \, 
\int_{\partial\Omega} g\cdot \phi\, d\sigma=0
\text{ for any } \phi\in \mathcal{R} \Big\}.
\end{equation}

\begin{lemma}\label{lemma-Lip-n}
Let $u\in H^1(\Omega; \br^d)$ be a weak solution to the Neumann problem:
$\mathcal{L}_0 (u)=0$ in $\Omega$ and $\frac{\partial u}{\partial \nu_0}
=g$ on $\partial\Omega$,
where $g\in L^2_{\mathcal{R}}(\partial\Omega)$.
Suppose that $u\perp \mathcal{R}$ in $L^2(\Omega; \br^d)$.
Then 
\begin{equation}\label{Lip-d-n}
\|(\nabla u)^*\|_{L^2(\partial\Omega)}
\le C \| g\|_{L^2(\partial\Omega)},
\end{equation}
where $C$ depends only on $\kappa_1$, $\kappa_2$, and $\Omega$.
\end{lemma}

\begin{proof}
This was proved in \cite{DKV-1988}.
\end{proof}

\begin{thm}\label{main-thm-2.3e}
Let $\Omega$ be a bounded Lipschitz domain in $\mathbb{R}^d$.
Let $w_\varep$ be the same as in Lemma \ref{lemma-c-3-1}.
Then, for $0<\varep<1$
\begin{equation}\label{c-2-3-00}
\| w_\varep\|_{H^1(\Omega)}
\le C \sqrt{\varep}\,
\Big\{ \| F\|_{L^q(\Omega)} +\| g\|_{L^2(\partial\Omega)} \Big\},
\end{equation}
where $q=\frac{2d}{d+1}$  and
$C$ depends only on $\kappa_1$,  $\kappa_2$, and $\Omega$. 
\end{thm}

\begin{proof} 
It follows from Lemma \ref{lemma-c-3-1} that
$$
\aligned
&\|\nabla w_\varep + (\nabla w_\varep)^T \|^2_{L^2(\Omega)}\\
&\le C  \|\nabla w_\varep\|_{L^2(\Omega)}
\Big\{ \varep \|S_\varep (\nabla^2 u_0)\|_{L^2(\Omega\setminus \Omega_{2\varep})}
+\| \nabla u_0 -S_\varep (\nabla u_0)\|_{L^2(\Omega\setminus \Omega_{2\varep})}
+\|\nabla u_0\|_{L^2(\Omega_{5\varep})} \Big\}.
\endaligned
$$
We then apply the second Korn inequality,
\begin{equation}\label{Korn-3}
\| u\|_{H^1(\Omega)}
\le C \| \nabla u + (\nabla u)^T\|_{L^2(\Omega)}
+ C \sum_{k=1}^\ell \Big|\int_\Omega u\cdot \phi_k\, dx \Big|,
\end{equation}
where $\ell =\frac{d(d+1)}{2}$ and $\{ \phi_k: k=1, \dots, \ell \}$
forms an orthonormal basis for $\mathcal{R}$ as a subspace of $L^2(\Omega; \br^d)$.
This leads to
$$
\aligned
\|w_\varep\|_{H^1(\Omega)}
&\le C \Big\{ \varep \| S_\varep(\nabla^2 u_0)\|_{L^2(\Omega\setminus \Omega_{2\varep})}
+\| \nabla u_0 -S_\varep (\nabla u_0)\|_{L^2(\Omega\setminus \Omega_{2\varep})}
+\|\nabla u_0\|_{L^2(\Omega_{5\varep})} \Big\}\\
& \qquad\qquad
+C \sum_{k=1}^\ell 
\Big| \int_\Omega \varep \chi(x/\varep) \eta_\varep S_\varep (\nabla u_0)
\cdot \phi_k \, dx \Big|\\
&\le C \Big\{ \varep \|S_\varep( \nabla^2 u_0)\|_{L^2(\Omega\setminus \Omega_{2\varep})}
+\| \nabla u_0 -S_\varep (\nabla u_0)\|_{L^2(\Omega\setminus \Omega_{2\varep})}\\
&\qquad\qquad
+\|\nabla u_0\|_{L^2(\Omega_{5\varep})}  +\varep \| \nabla u_0\|_{L^2(\Omega)}\Big\},
\endaligned
$$
where we have used the assumptions that
$u_\varep, u_0 \perp \mathcal{R}$ in $L^2(\Omega; \br^d)$.

Next, we proceed as in the proof of Theorem \ref{main-thm-2.2}
by writing $u_0 =v_0 +\phi$, where $v_0$ is given by (\ref{v-0}).
The terms involving $v_0$ are handled exactly in the same manner as before.
Similarly, to control the terms involving $\phi$,
it suffices to bound the $L^2(\partial\Omega)$ norm of
$(\nabla \phi)^*$. To do this, we note that
 $\mathcal{L}_0(\phi)=0$ in $\Omega$ and
$\frac{\partial \phi}{\partial \nu_0}=g -\frac{\partial v_0}{\partial\nu_0}$ on $\partial\Omega$,
It follows by Lemma \ref{lemma-Lip-n}  that
$$
\|\mathcal{M}(\nabla \phi)\|_{L^2(\partial\Omega)}
\le C\| g\|_{L^2(\partial\Omega)} + C\|\nabla v_0\|_{L^2(\partial\Omega)}
+C \sum_{k=1}^\ell \Big|\int_\Omega \phi \cdot \phi_k\, dx \Big|.
$$
 Note that
 $
 \|\nabla v_0\|_{L^2(\partial\Omega)}\le C \|v\|_{W^{2, q}(\Omega)} \le C\, \| F\|_{L^q(\Omega)},
 $
 and
 $$
 \Big|\int_\Omega \phi \cdot \phi_k\, dx \Big|
 = \Big|\int_\Omega v_0 \cdot \phi_k\, dx \Big|
 \le C \| F\|_{L^q(\Omega)},
 $$
 where we have used the fact $u_0\perp \mathcal{R}$ in $L^2(\Omega; \br^d)$.
 Thus,
 $$
 \|(\nabla \phi)^*\|_{L^2(\partial\Omega)}
\le C\Big\{ \| g\|_{L^2(\partial\Omega)} + \| F\|_{L^q(\Omega)} \Big\}.
$$
This completes the proof.
\end{proof}

\begin{thm}\label{theorem-c-3}
Let $\Omega$ be a bounded Lipschitz domain in $\br^d$.
Let $u_\varep\in H^1(\Omega; \br^d)$ be the weak solution of
the Neumann problem (\ref{N-CE}).
Let $u_0$ be the homogenized solution.
Then, if $u_0\in W^{2, d}(\Omega; \br^d)$,
\begin{equation}\label{c-3-3-0}
\| u_\varep -u_0 -\varep \chi(x/\varep) \nabla u_0\|_{H^1(\Omega)}
\le C \sqrt{\varep}\,  \| u_0\|_{W^{2, d}(\Omega)}
\end{equation}
for any $0<\varep<1$,
where $C$ depends only on $\kappa_1$, $\kappa_2$, and $\Omega$.
Furthermore, if the corrector $\chi$ is bounded and $u_0\in H^2(\Omega; \br^d)$, then
\begin{equation}\label{c-3-3-1}
\| u_\varep -u_0 -\varep \chi(x/\varep) \nabla u_0\|_{H^1(\Omega)}
\le C \sqrt{\varep}\,  \| u_0\|_{H^2(\Omega)}
\end{equation}
for any $0<\varep<1$,
where $C$ depends only on $\kappa_1$, $\kappa_2$, $\|\chi\|_\infty$ and $\Omega$.
\end{thm}

\begin{proof}
The proof is exactly the same as that of Theorem \ref{theorem-d-c-2-1}, where
the Dirichlet boundary condition was never used.
\end{proof}

We now move to  the sharp convergence rate in $L^p$
with $p=\frac{2d}{d-1}$ for  (\ref{N-CE}).

\begin{lemma}\label{lemma-c-5-0}
Let $w_\varep$ be given by (\ref{w-c-2}),
where $u_\varep$ is the weak solution of (\ref{N-CE}) and $u_0$ the homogenized solution.
Assume that $u_0\in W^{2, q}(\Omega; \br^d)$ for
$q=\frac{2d}{d+1}$. Then the inequality (\ref{c-4-0-0})
holds for any $\psi\in C_0^\infty (\br^d; \br^d)$.
\end{lemma}

\begin{proof}
The proof is exactly the same as that for Lemma \ref{lemma-c-4-0}.
\end{proof}

\begin{lemma}\label{lemma-c-5-1}
Let $ u\in H^1(\Omega; \br^d)$ be a weak solution to the
Neumann problem: $\mathcal{L}_0 (u)=G$ in $\Omega$ and
$\frac{\partial u}{\partial \nu_0}=0$ on $\partial\Omega$,
where $G\in C^\infty(\br^d; \br^d)$ and
$G\perp \mathcal{R}$ in $L^2(\Omega; \br^d)$.
Assume that $u\perp \mathcal{R}$ in $L^2(\Omega; \br^d)$.
Then
\begin{equation}\label{Lip-n-20}
\|\nabla u \|_{L^2(\Omega_t)}
\le C t^{1/2} \| G\|_{L^q(\Omega)},
\end{equation}
\begin{equation}\label{Lip-n-21}
\|\nabla u \|_{L^p(\Omega)}\le C\| G\|_{L^q(\Omega)},
\end{equation}
where $p=\frac{2d}{d-1}$,
$q=p^\prime=\frac{2d}{d+1}$,
and $C$ depends only on $\kappa_1$, $\kappa_2$, and $\Omega$.
\end{lemma}

\begin{proof}
With Lemma \ref{lemma-Lip-n} at our disposal,
the proof is similar to that of Lemma \ref{lemma-Lip-d-2}.
\end{proof}

\begin{thm}\label{theorem-c-5}
Let $\Omega$ be a bounded Lipschitz domain in $\rd$.
Let $u_\varep\in H^1(\Omega; \br^d)$ $(\varep\ge 0)$ be the weak solution of 
the Neumann problem (\ref{N-CE}).
Then, if $u_0\in W^{2, q}(\Omega; \br^d)$,
\begin{equation}\label{c-5-0}
\| u_\varep -u_0\|_{L^p(\Omega)} \le C \varep \| u_0\|_{W^{2, q}(\Omega)},
\end{equation}
where $q=p^\prime =\frac{2d}{d+1}$ and
$C$ depends only on $\kappa_1$, $\kappa_2$, and $\Omega$.
\end{thm}

\begin{proof}
The proof of (\ref{c-5-0}) is similar to that in the case of the Dirichlet boundary condition,
using Theorem  \ref{main-thm-2.3e}, \ref{lemma-c-5-0}, \ref{lemma-c-5-1}, and a duality argument.

Fix $G\in C^\infty (\br^d; \br^d)$ such that $G\perp \mathcal{R}$ in $L^2(\Omega; \br^d)$.
Let $v_\varep\in H^1(\Omega; \br^d)$ ($\varep\ge 0$)
be the weak solution to the Neumann problem,
\begin{equation}\label{c-5-1e}
\mathcal{L}_\varep (v_\varep) =G \quad \text{ in } \Omega \quad 
\text{ and } \quad \frac{\partial v_\varep}{\partial \nu_\varep} =0\quad \text{ on } \partial\Omega,
\end{equation}
with the property that $v_\varep \perp \mathcal{R}$ in $L^2(\Omega; \br^d)$.
As in the case of Dirichlet problem, we have 
$$
\Big|\int_\Omega w_\varep\cdot G\, dx \Big|
\le C\varep \| u_0\|_{W^{2, q}(\Omega)} \| G\|_{L^q(\Omega)},
$$
where $w_\varep =u_\varep -u_0 -\varep \chi(x/\varep)\eta_\varep S^2_\varep (\nabla u_0)$.
Using
$$
\|\varep \chi(x/\varep) \eta_\varep S_\varep^2 (\nabla u_\varep)\|_{L^p(\Omega)}
\le C \varep \| u_0\|_{W^{2, q}(\Omega)},
$$
we then obtain 
$$
\Big|\int_\Omega ( u_\varep -u_0) \cdot G\, dx \Big|
\le C\varep \| u_0\|_{W^{2, q}(\Omega)} \| G\|_{L^q(\Omega)}.
$$
Since $u_\varep, u_0 \perp \mathcal{R}$ in $L^2(\Omega; \br^d)$,
by duality, this gives the estimate (\ref{c-5-0}).
\end{proof}


\section{Notes}

There is an extensive literature on the problem of convergence rates in periodic homogenization.
Early results, proved  under smoothness conditions on the correctors $\chi=(\chi_j^\beta)$,
may be found in classical books \cite{BLP-1978, JKO-1993, OSY-1992}.
The flux correctors, defined by (\ref{definition-of-F}), were already used in  \cite{JKO-1993}.

In \cite{Griso-2004, Griso-2006, Zhikov-2005, Zhikov-2006, Zhikov-P-2005,Pas-2006, OV-2007},
various $\e$-smoothing techniques were introduced to treat the case where $\chi$ may be 
unbounded. In particular, error estimates similar to (\ref{r-1/2}) were proved in \cite{Zhikov-P-2005}
for solutions of scalar second-order elliptic equations with the Dirichlet or Neumann boundary
conditions.  Further extensions were made in \cite{ P-Suslina-2012, Suslina-2013N}, where the error estimate
(\ref{r-1/2}) in $H^1$ was established for a broader class of elliptic systems.

The error estimate (\ref{rate-1/2-0}) in $H^1$ for two-scale expansions in Lipschitz domains
was proved in \cite{Shen-2017-APDE}.
It follows from (\ref{rate-1/2-0}) that if $\mathcal{L}_\e (u_\e)=0$ in
a bounded Lipschitz domain $\Omega$, then
\begin{equation}\label{Rellich-1}
\left\{
\aligned
\|\nabla u_\e \|_{L^2(\Omega_\e)}
& \le C \sqrt{\e} \| u_\e \|_{H^1(\partial\Omega)},\\
\|\nabla u_\e \|_{L^2(\Omega_\e)}
&\le C \sqrt{\e} \Big\| \frac{\partial u_\e}{\partial\nu_\e} \Big\|_{L^2(\partial\Omega)},
\endaligned
\right.
\end{equation}
where $C$ depends only on $\mu$ and $\Omega$.
The inequalities  in (\ref{Rellich-1}) should be regarded as large-scale Rellich estimates.
Such estimates play an essential role in the study of $L^2$ boundary value problems
for $\mathcal{L}_\e$ in Lipschitz domains.
See Chapter \ref{chapter-7}
and  \cite{KS-2011-H, KS-2011-L, Shen-2017-APDE, GSS-2017} for details.

The sharp convergence rate (\ref{c-100}) in $L^2$
was obtained in \cite{Griso-2004, Griso-2006, Suslina-2013, Suslina-2013N}.
 The duality method  was used first in \cite{Suslina-2013}.
 Also see related work in \cite{KLS-2012, Gu-2016, Xu-2016, Gu-2018,
 Shen-Zhuge-2017,Niu-Shen-Xu}.

The scale-invariant estimate (\ref{c-L-2}) is new. A weaker estimate,
$\| u_\e -u_0\|_{L^p(\Omega)} \le C \e \| u_0\|_{H^2(\Omega)}$,
where $p=\frac{2d}{d-1}$, was proved in \cite{Shen-2017-APDE}.

%
%
%


\chapter{Interior Estimates}\label{chapter-2}

In this chapter we establish interior H\"older ($C^{0, \alpha}$) estimates, $W^{1,p}$ estimates, and Lipschitz  ($C^{0,1}$) 
estimates, that are uniform in $\e>0$, for solutions of $\mathcal{L}_\e(u_\e)=F$,
where $\mathcal{L}_\e =-\text{\rm div} (A(x/\e)\nabla)$. As a result,
we obtain uniform size estimates of $\Gamma_\varep (x,y)$, $\nabla_x \Gamma_\varep (x,y)$,
$\nabla_y \Gamma_\varep (x,y)$,
and $\nabla_x\nabla_y \Gamma_\varep (x,y)$,
where $\Gamma_\varep (x,y)$ denotes the matrix of fundamental solutions for
$\mathcal{L}_\varep$ in $\br^d$.
This in turn allows us to derive asymptotic expansions, as $\varep\to 0$, of
$\Gamma_\varep (x,y)$, $\nabla_x\Gamma(x,y)$, $\nabla_y\Gamma_\varep(x,y)$,
 and $\nabla_x\nabla_y \Gamma_\varep (x,y)$.
 Note  that if $u_\e(x)=P_j^\beta (x) +\e \chi_j^\beta (x/\e)$,
 then $\nabla u_\varep=\nabla P_j^\beta +\nabla \chi_j^\beta (x/\e)$
 and $\mathcal{L}_\e (u_\e) =0$ in $\brd$. Thus
 no better uniform regularity beyond Lipschitz estimates 
 should be expected (unless $\text{\rm div} (A)=0$, which would imply $\chi_j^\beta=0$). 

In Section \ref{section-2.2} we use a compactness method to establish a Lipschitz estimate
down to the microscopic scale $\varep$ under the ellipticity and periodicity conditions.
This, together with a simple blow-up argument, leads to the full Lipschitz estimate under an additional
smoothness condition.
The compactness method, which originated from the study of the regularity theory
in the calculus of variations and minimal surfaces, was introduced to the study of homogenization
problems by M. Avellaneda and F. Lin \cite{AL-1987}.
It also will be used in Chapters \ref{chapter-3} to establish uniform boundary 
 estimates for solutions of $\mathcal{L}_\varep (u_\varep)=F$ with Dirichlet  conditions.

In Section \ref{real-variable-section} we present a real-variable method, which originated in a paper
by L. Caffarelli  and I. Peral \cite{CP-1998}. The method may be regarded as a dual and refined
version of the celebrated Calder\'on-Zygmund Lemma.
It is used in Section \ref{section-2.4} to study the $W^{1, p}$ estimates for
the elliptic system $\mathcal{L}_\varep (u_\varep)=\text{\rm div} (G)$, 
and reduces effectively the problem to certain reverse H\"older inequality for local solutions
of $\mathcal{L}_\varep (u_\varep)=0$.
In Section \ref{section-2.5} we investigate the asymptotic behaviors of fundamental solutions
and their derivatives. 

Throughout this chapter
 we will assume that the coefficient matrix $A=\big(a_{ij}^{\alpha\beta}\big)$,
 with $ 1\le \alpha, \beta\le m$ and $1\le i, j\le d$, is 1-periodic and satisfies the $V$-ellipticity condition
 (\ref{weak-e-1})-(\ref{weak-e-2}).
 This in particular includes the case of elasticity operators, where $m=d$ and $A\in E(\kappa_1, \kappa_2)$.
We will also need to impose some smoothness condition to ensure estimates at the microscopic scale.
We say $A\in \text{VMO}(\rd)$ if
\begin{equation}\label{VMO}
\lim_{ r\to 0} \sup_{x\in \rd} 
\average_{B(x,r)} \Big| A-\average_{B(x, r)} A \Big| =0.
\end{equation}
Observe that $A\in \text{VMO}(\rd)$ if $A$ is uniformly continuous in $\rd$.
The uniform H\"older and $W^{1, p}$ estimates will be proved under the condition (\ref{VMO}).
A stronger smoothness condition,
\begin{equation}\label{smoothness}
|A(x)-A(y)|\le \tau |x-y|^\lambda \quad
\text{ for any } x, y\in \br^d,
\end{equation}
where $\tau\ge 0$ and $\lambda\in (0,1]$,
will be imposed for uniform Lipschitz estimates.

A very important feature of the family of operators $\{ \mathcal{L}_\varep, \varep>0\}$
is the following rescaling property:
\begin{equation}\label{rescaling}
\aligned
& \text{if } \mathcal{L}_\varep (u_\varep\big)=F \text{ and } v(x)=u_\varep (rx),\\
&\text{then } 
\mathcal{L}_{\frac{\varep}{r}} (v)=G, \text{ where } G(x)=r^2F(rx).
\endaligned
\end{equation}
It plays an essential role in the compactness method as well as in numerous other rescaling
arguments in this monograph. 
The property of translation is also important to us:
\begin{equation}\label{translation}
\aligned
& \text{ if } -\text{div}\big(A(x/\varep)\nabla u_\varep)=F \text{ and } v_\varep(x)=u_\varep (x-x_0), \\
& \text{ then }
-\text{div} \big(\widetilde{A}(x/\varep)\nabla v_\varep\big)=\widetilde{F}, \\
& \text{ where }
\widetilde{A}(y)=A(y +\varep^{-1} x_0) \text{ and } \widetilde{F}(x)=F(x-x_0).
\endaligned
\end{equation}
Observe that the matrix $\widetilde{A}$ is 1-periodic and satisfies the same ellipticity condition as $A$.
It also satisfies the same smoothness condition that we will impose on $A$, 
uniformly in $\varep>0$ and $x_0\in \rd$.




\section[Lipschitz estimate]{Interior Lipschitz estimates}\label{section-2.2}

In this section we prove the interior Lipschitz estimate, using a compactness method.
As a corollary,
we also establish a Liouville property for entire solutions of $\mathcal{L}_1 (u)=0$ in $\br^d$.

\begin{thm}[Interior Lipschitz estimate at large scale]\label{theorem-2.2.1}
Assume  that $A$ satisfies (\ref{weak-e-1})-(\ref{weak-e-2}) and is 1-periodic.
Let $u_\varep \in H^1(B; \mathbb{R}^m)$ be a weak solution of
$\mathcal{L}_\varep (u_\varep)=F$ in $B$, where $B=B(x_0, R)$ for some $x_0\in \rd$ and $R>0$,
and  $F\in L^p (B; \mathbb{R}^m)$ for some $p>d$.
Suppose that $0<\varep<R$. Then
\begin{equation}\label{2.2-1}
\left(\average_{B(x_0, \varep)}
|\nabla u_\varep|^2 \right)^{1/2}
\le C_p \left\{  \left(\average_{B(x_0, R)} |\nabla u_\varep|^2\right)^{1/2}
+R \left(\average_{B(x_0, R)} |F|^p \right)^{1/p} \right\},
\end{equation}
where $C_p$ depends only on $\mu$ and $p$.
\end{thm}

Estimate (\ref{2.2-1}) should be regarded as a Lipschitz estimate down to the microscopic scale $\varep$.
In fact, under some smoothness condition on $A$, 
the full scale Lipschitz estimate follows readily from Theorem \ref{theorem-2.2.1}
by a blow-up argument. 

\begin{thm}[Interior Lipschitz estimate]\label{interior-Lip-theorem}
Suppose that $A$ satisfies (\ref{weak-e-1})-(\ref{weak-e-2}) and is 1-periodic.
Also assume that $A$ satisfies the smoothness condition (\ref{smoothness}).
Let $u_\varep \in H^1(B;\br^m)$ be a weak solution to
$\mathcal{L}_\varep (u_\varep)=F$ in $B$ for some ball $B=B(x_0,R)$,
where $F\in L^p(B; \mathbb{R}^m)$ for some $p>d$.
Then
\begin{equation}\label{estimate-2.2}
|\nabla u_\varep (x_0)|
\le C_p\left\{  \left( \average_{B} |\nabla u_\varep|^2\right)^{1/2}
+R \left(\average_{B} |F|^p \right)^{1/p} \right\},
\end{equation}
where $C_p$ depends only on $p$, $\mu$ and $(\lambda, \tau)$.
\end{thm}

\begin{proof} 
We give the proof of Theorem \ref{interior-Lip-theorem}, assuming Theorem \ref{theorem-2.2.1}.
By translation and dilation we may assume that $x_0=0$ and $R=1$.

The ellipticity condition (\ref{weak-e-1}-(\ref{weak-e-2}),
together with H\"older continuous assumption (\ref{smoothness}), allows us to use the 
following local regularity result: if
 $-\text{\rm div}(A(x) \nabla u)=F$ in $B(0,1)$,  where $F\in L^p(B(0,1); \mathbb{R}^m)$ for some $p>d$, then 
\begin{equation}\label{classical-Lip}
|\nabla u(0)|\le C_p \left\{ \left(\average_{B(0,1)} |\nabla u|^2\right)^{1/2} 
+\left(\average_{B(0,1)} |F|^p \right)^{1/p} \right\},
\end{equation}
where $C_p$ depends only on $\mu$, $p$,  $\lambda$ and $\tau$
 (see e.g. \cite{Gia-Ma-book}).
We may also assume that $0<\varep<(1/2)$, as the case $\varep\ge (1/2)$
follows directly from (\ref{classical-Lip}) and the observation that
 the coefficient matrix $A(x/\varep)$ is uniformly H\"older continuous
for $\varep\ge (1/2)$.

To handle the case $0<\varep<(1/2)$, we use a blow-up argument and estimate (\ref{2.2-1}).
Let $w(x)=\varep^{-1}  u_\varep (\varep x)$. 
Since  $\mathcal{L}_1 (w) =\varep F(\varep x)$ in $B(0,1)$, it follows again from (\ref{classical-Lip}) that
$$
\aligned
|\nabla w (0)|
& \le C \left\{ \left(\average_{B(0,1)} |\nabla w|^2\right)^{1/2}
+\left(\average_{B(0,1)} |\varep F(\varep x)|^p \, dx \right)^{1/p} \right\}\\
&\le C \left\{
\left(\average_{B(0, \varep)} |\nabla u_\varep|^2 \right)^{1/2}
+ \varep^{1-\frac{d}{p} }\left(\average_{B(0,1)} |F|^p \right)^{1/p} \right\},
\endaligned
$$
where $C$ depends only on $p$, $\mu$, $\lambda$ and $\tau$.
This, together with (\ref{2.2-1}) and the fact that $\nabla w(0) =\nabla u_\varep (0)$,
gives the estimate (\ref{estimate-2.2}).
\end{proof}

In the rest of this section we will assume that $A$ satisfies (\ref{weak-e-1})-(\ref{weak-e-2}) and is 1-periodic.
No smoothness condition on $A$ is needed. The proof uses only interior $C^{1,\alpha}$ estimates 
for elliptic systems with constant coefficients satisfying the Legendre-Hadamard condition. 

Recall that $P^\beta_j  (x)=x_j e^\beta$ and $e^\beta=(0, \dots, 1, \dots, 0)$ with $1$ in the $\beta^{th}$
position.

\begin{lemma}[One-step improvement]\label{step-2.2-1}
Let $0<\sigma<\rho<1$ and $\rho=1-\frac{d}{p}$.
There exist constants $\varep_0\in (0, 1/2)$ and $\theta\in (0,1/4)$, depending only on
 $\mu$, $\sigma$ and $\rho$, such that 
 \begin{equation}\label{2.2.1-0}
 \aligned
 & \left(\average_{B(0, \theta)} 
 \big|u_\varep (x) -\average_{B(0, \theta)} u_\varep
 -\left(P^\beta_j (x) + \varep \chi_j^\beta (x/\varep) \right)
 \average_{B(0, \theta)}
 \frac{\partial u^\beta_\varep}{\partial x_j}\big|^2\, dx \right)^{1/2}\\
& \qquad \le \theta^{1+\sigma}
 \max \left\{ \left(\average_{B(0, 1)} |u_\varep|^2\right)^{1/2}, \left(\average_{B(0, 1)} |F|^p \right)^{1/p} \right\},
 \endaligned
 \end{equation}
 whenever $0<\varep<\varep_0$, and $u_\varep\in H^1(B(0,1);\mathbb{R}^m)$ is a weak solution of
 \begin{equation}\label{2.2.1-1}
 \mathcal{L}_\varep (u_\varep) =F \quad \text{ in } B(0,1).
 \end{equation}
 \end{lemma}

\begin{proof}
Estimate (\ref{2.2.1-0}) is proved by contradiction, using Theorem \ref{theorem-1.3.4} and the following 
observation: for any $\theta\in (0,1/4)$,
\begin{equation}\label{2.2.1-3}
\aligned
& \sup_{|x|\le \theta}
\Big| u(x)-\average_{B(0,\theta)} u -x_j \average_{B(0,\theta)} \frac{\partial u}{\partial x_j} \Big|\\
&\qquad \le C\, \theta^{1+\rho} \|\nabla u\|_{C^{0, \rho}(B(0,\theta))}\\
& \qquad \le C\, \theta^{1+\rho}  \|\nabla u\|_{C^{0, \rho} (B(0,1/4))}\\
&\qquad
\le C_0 \,\theta^{1+\rho} \left\{ \left( \average_{B(0,1/2)} |u|^2\right)^{1/2}
+\left(\average_{B(0,1/2)} |F|^p\right)^{1/p} \right\},
\endaligned
\end{equation}
where $u$ is a solution of $-\text{\rm div} (A^0\nabla u)=F$ in $B(0,1/2)$
and $A^0$ is a constant matrix satisfying the Legendre-Hadamard condition (\ref{weak-eee}).
We mention that the last inequality in (\ref{2.2.1-3}) is a standard $C^{1, \rho}$ estimate
for second-order elliptic systems with  constant coefficients, and
that the constant $C_0$ depends only on $\mu$ and $\rho$.

Since $\sigma <\rho$, we may choose $\theta\in (0,1/4)$ so small that $2^{d+1} C_0\theta^{1+\rho} < \theta^{1+\sigma}$.
We claim that the estimate (\ref{2.2.1-0}) holds for this $\theta$ and some
$\varep_0\in (0,1/2)$, which depends only on $\mu$, $\sigma$ and $\rho$.

Suppose this is not the case.
Then there exist sequences $\{\varep_\ell\}\subset (0, 1/2)$,
$\{ A_\ell\}$ satisfying (\ref{weak-e-1})-(\ref{weak-e-2}) and (\ref{periodicity}), $\{ F_\ell \} \subset L^p(B(0, 1); \mathbb{R}^m)$,
and $\{ u_\ell \}\subset H^1(B(0,1);\br^m)$,
such that
$\varep_\ell \to 0$, 
$$
-\text{\rm div} \big(A_\ell (x/\varep_\ell)\nabla u_\ell \big) =F_\ell \quad \text{ in } B(0,1),
$$
\begin{equation}\label{2.2.1-4}
\left(\average_{B(0,1)} |u_\ell|^2\right)^{1/2}\le 1, \quad
\left(\average_{B(0,1)} |F_\ell |^p \right)^{1/p}\le 1,
\end{equation}
and 
\begin{equation}\label{2.2.1-5}
\left(\average_{B(0, \theta)} 
 \Big|u_\ell (x) -\average_{B(0, \theta)} u_\ell
 -\left(P^\beta_j (x) + \varep_\ell \chi_{\ell, j}^\beta (x/\varep_\ell) \right)
 \average_{B(0, \theta)}
 \frac{\partial u^\beta_\ell}{\partial x_j}\Big|^2\, dx \right)^{1/2}
>\theta^{1+\sigma},
\end{equation}
where $\chi_{\ell, j}^{ \beta}$ denote the correctors associated with the 1-periodic matrix $A_\ell$.
Observe that by (\ref{2.2.1-4}) and Caccioppoli's inequality,
the sequence $\{ u_\ell \}$ is bounded in $H^1(B(0,1/2); \mathbb{R}^m)$.
By passing to  subsequences,
we may assume that
$$
\left\{
\aligned
& u_\ell\rightharpoonup u  \text{ weakly in } L^2(B(0,1);\br^m),\\
& u_\ell\rightharpoonup u  \text{ weakly in } H^1 (B(0,1/2);\br^m),\\
& F_\ell \rightharpoonup F \text{ weakly in }L^p (B(0,1); \br^m),\\
& \widehat{A_\ell} \to A^0,
\endaligned
\right.
$$
where $\widehat{A_\ell}$ denotes the homogenized matrix for $A_\ell$.
Since $p>d$, $F_\ell \rightharpoonup F$ weakly in $L^p(B(0,1); \br^m)$ implies that
$F_\ell \to F$ strongly in $H^{-1} (B(0,1); \mathbb{R}^m)$.
It then follows by Theorem \ref{theorem-1.3.4} that $u\in H^1(B(0,1/2); \br^m)$ is a solution of
$-\text{\rm div} \big (A^0 \nabla u\big) =F$ in $B(0,1/2)$.

We now let $\ell \to \infty$ in (\ref{2.2.1-4}) and (\ref{2.2.1-5}). This leads to
\begin{equation}\label{2.2.1-6}
\left(\average_{B(0,1)} |u|^2\right)^{1/2}\le 1, \quad
\left(\average_{B(0,1)} |F |^p \right)^{1/p}\le 1,
\end{equation}
and 
\begin{equation}\label{2.2.1-7}
\left(\average_{B(0, \theta)} 
 \Big|u (x) -\average_{B(0, \theta)} u
 -x_j
 \average_{B(0, \theta)}
 \frac{\partial u}{\partial x_j}\Big|^2\, dx \right)^{1/2}
\ge \theta^{1+\sigma},
\end{equation}
where we have used the observation that $u_\ell \to u$ strongly in $L^2(B(0,1/2); \mathbb{R}^m)$.
Here we also have used the fact that the sequence $\{ \chi_{\ell, j}^\beta\}$ is bounded in $L^2(Y;\mathbb{R}^m)$.
Finally, we note that by (\ref{2.2.1-7}), (\ref{2.2.1-3}) and (\ref{2.2.1-6}),
$$
\theta^{1+\sigma} \le C_0 \theta^{1+\rho} 
\left\{ \left(\average_{B(0,1/2)} |u|^2\right)^{1/2}
+\left(\average_{B(0,1/2)} |F|^p\right)^{1/p} \right\}
< 2^{d+1} C_0 \theta^{1+\rho},
$$
which is in contradiction with the choice of $\theta$.
This completes the proof.
\end{proof}

\begin{remark}\label{remark-2.2.1}
{\rm 
Since
$$
\inf_{\alpha\in \brd} \average_E \big |f -\alpha\big|^2 =\average_E \Big|f -\average_E f\Big|^2
$$
for any $f\in L^2(E; \br^m)$, we may replace $\average_{B(0, \theta)} u_\varep$ in (\ref{2.2.1-0}) by the average
$$
\average_{B(0, \theta)} \left[ u_\varep -\left( P_j^\beta (x) +\varep\, \chi_j^\beta (x/\varep) \right)\average_{B(0, \theta)}
\frac{\partial u_\varep^\beta}{\partial x_j} \right]\, dx.
$$
The observation  will be used in the proof of the next lemma.
}
\end{remark}

\begin{lemma}[Iteration]\label{step-2.2-2}
Let $0<\sigma< \rho<1$ and $\rho=1-\frac{d}{p}$.
Let $(\varep_0, \theta)$ be the constants  given by Lemma \ref{step-2.2-1}.
Suppose that $0<\varep< \theta^{k-1} \varep_0$ for some $k\ge 1$ and
$u_\varep$ is a solution of $\mathcal{L}_\varep (u_\varep)=F$ in $B(0,1)$.
Then there exist constants $E(\varep, \ell) =\big (E_j^\beta (\varep, \ell)\big)\in \mathbb{R}^{m\times d}$
for $1\le \ell \le k$, such that if
$$
v_\varep =u_\varep -\left( P_j^\beta +\varep \chi_j^\beta (x/\varep) \right) E_j^\beta (\varep, \ell),
$$
then 
\begin{equation}\label{2.2.2-0}
\aligned
&\left(\average_{B(0, \theta^\ell)}
\Big| v_\varep
-\average_{B(0, \theta^\ell)} v_\varep \Big|^2 \right)^{1/2}\\
&\qquad \le \theta^{\ell (1+\sigma)}  \max \left\{ \left(\average_{B(0,1)} |u_\varep|^2\right)^{1/2},
\left(\average_{B(0,1)} |F|^p\right)^{1/p} \right\}.
\endaligned
\end{equation}
Moreover, the constants $E(\varep, \ell)$ satisfy 
\begin{equation}\label{2.2.2-00}
\aligned
|E(\varep, \ell)|
& \le C \max \left\{ \left(\average_{B(0,1)} |u_\varep|^2\right)^{1/2},
\left(\average_{B(0,1)} |F|^p\right)^{1/p} \right\},\\
|E(\varep, \ell+1)-E(\varep, \ell)|
& \le C\, \theta^{\ell\sigma}
 \max \left\{ \left(\average_{B(0,1)} |u_\varep|^2\right)^{1/2},
\left(\average_{B(0,1)} |F|^p\right)^{1/p} \right\},
\endaligned
\end{equation}
where $C$ depends only on $\mu$, $\sigma$ and $\rho$.
\end{lemma}

\begin{proof}
We prove (\ref{2.2.2-0})-(\ref{2.2.2-00}) by an induction argument on $\ell$. The case $\ell=1$ follows readily from 
Lemma \ref{step-2.2-1} and Remark \ref{remark-2.2.1}, with
$$
 E^\beta_j (\varep, 1)=\average_{B(0, \theta)} \frac{\partial u^\beta_\varep}{\partial x_j}
$$
(set $E(\varep, 0)=0$).
Suppose now that the desired constants $E(\varep, \ell)$ exist for all integers up to some
$\ell$, where $1\le \ell \le k-1$.
To construct $E(\varep, \ell +1)$,  consider the function
$$
\aligned
w(x)=u_\varep (\theta^\ell x)
& - \left\{ P^\beta_j (\theta^\ell x) +\varep \chi^\beta_j (\theta^\ell x/\varep) \right\}
 E^\beta_j (\varep, \ell)\\
 &-\average_{B(0, \theta^\ell)}
 \left[ u_\varep - \left( P_j^\beta +\varep \chi_j^\beta (y/\varep) \right)E_j^\beta (\varep, \ell) \right] dy.
 \endaligned
$$
By the rescaling property (\ref{rescaling}) and the equation (\ref{corrector-solution}) for correctors,
$$
\mathcal{L}_{\frac{\varep}{\theta^\ell}} (w)=F_\ell \quad \text{ in } B(0,1),
$$
where $F_\ell (x) =\theta^{2\ell} F(\theta^\ell x)$.
Since $\varep\theta^{-\ell} \le \varep \theta^{-k}\le \varep_0$, it follows from Lemma \ref{step-2.2-1} and
Remark \ref{remark-2.2.1}  that
\begin{equation}\label{2.2.2-1}
\aligned
& \bigg(\average_{B(0, \theta)} 
\Big| w  - \left\{ P^\beta_j  +  \varep \theta^{-\ell} \chi^\beta_j (\theta^\ell x /\varep)\right\}
\average_{B(0,\theta)} \frac{\partial w^\beta}{\partial x_j}\\
& \qquad\qquad
 -\average_{B(0, \theta)}
\left[ w -\left(P_j^\beta +\theta^{-\ell}\varep \chi_j^\beta (\theta^\ell y/\varep)\right) \average_{B(0, \theta)}
\frac{\partial w^\beta}{\partial x_j}\right] dy \Big|^2\, dx \bigg)^{1/2}\\
& \le \theta^{1+\sigma} 
\max \left\{ \left(\average_{B(0,1)} |w|^2\right)^{1/2},
\left(\average_{B(0,1)} |F_\ell|^p \right)^{1/p} \right\}.
\endaligned
\end{equation}
Observe that by the induction assumption,
\begin{equation}\label{2.2.2-2}
\left(\average_{B(0, 1)} |w|^2\right)^{1/2}
\le \theta^{\ell (1+\sigma)} 
\max \left\{ \left(\average_{B(0,1)} |u_\varep|^2\right)^{1/2},
\left(\average_{B(0,1)} |F|^p\right)^{1/p} \right\}.
\end{equation}
Also, since $0<\rho=1-\frac{d}{p}$,
\begin{equation}\label{2.2.2-3}
\left(\average_{B(0,1)} |F_\ell|^p\right)^{1/p}
\le \theta^{\ell (1+\rho)} \left(\average_{B(0,1)} |F|^p \right)^{1/p}.
\end{equation}
Hence, the RHS of (\ref{2.2.2-1}) is bounded by
$$
\theta^{(\ell+1) (1+\sigma)}
\max \left\{ \left(\average_{B(0,1)} |u_\varep|^2\right)^{1/2},
\left(\average_{B(0,1)} |F|^p\right)^{1/p} \right\}.
$$

Finally, note that the LHS of (\ref{2.2.2-1}) may be written as
$$
\aligned
& \bigg(\average_{B(0, \theta^{\ell+1})} 
\bigg| u_\varep  - \left\{ P^\beta_j  +  \varep  \chi^\beta_j (x /\varep)\right\}
E_j^\beta (\varep, \ell +1)\\
& \qquad\qquad\qquad
 -\average_{B(0, \theta^{\ell+1})}
\left[ u_\varep -\left(P_j^\beta +\varep \chi_j^\beta (y/\varep)\right) E_j^\beta(\varep, \ell+1)
\right] dy \bigg|^2\, dx \bigg)^{1/2}
\endaligned
$$
with
$$
E_j^\beta (\varep, \ell +1)
=E_j^\beta (\varep, \ell) +\theta^{-\ell} \average_{B(0, \theta)} \frac{\partial w}{\partial x_j}.
$$
By Caccioppoli's inequality,
$$
\aligned
|E(\varep, \ell+1)-E(\varep, \ell)|
&\le \theta^{-\ell} \left(\average_{B(0, \theta)} |\nabla w|^2\right)^{1/2}\\
&\le C \theta^{-\ell} 
\max \left\{ \left(\average_{B(0,1)} |w|^2\right)^{1/2},
\left(\average_{B(0,1)} |F_\ell|^p\right)^{1/2} \right\}\\
&\le C\, \theta^{\ell \sigma} 
\max \left\{ \left(\average_{B(0,1)} |u_\varep|^2\right)^{1/2},
\left(\average_{B(0,1)} |F|^p\right)^{1/p} \right\},
\endaligned
$$
where we have used estimates (\ref{2.2.2-2}) and (\ref{2.2.2-3}) for the last inequality.
Thus we have established the second inequality in (\ref{2.2.2-00}), from which the first  follows by summation.
This completes the induction argument and thus the proof.
\end{proof}

We now give the proof of Theorem \ref{theorem-2.2.1}.

\begin{proof}[\bf Proof of Theorem \ref{theorem-2.2.1}]
By translation and dilation we may assume that $x_0=0$ and $R=1$.
We may also assume that $0<\varep<\varep_0 \theta$, where $\varep_0, \theta$ are constants given by 
Lemma \ref{step-2.2-1}. The case $\varep_0\theta\le \varep<1$ is trivial.

Now suppose that $0<\varep< \varep_0 \theta$.
Choose $k\ge 2$ so that $\varep_0 \theta^k  \le \varep< \varep_0 \theta^{k-1}$.
It follows from Lemma \ref{step-2.2-2} that
$$
\left(\average_{B(0, \theta^{k-1})} 
\Big|u_\varep -\average_{B(0, \theta^{k-1})} u_\varep\Big|^2 \right)^{1/2}
\le C\, \theta^k \left\{ \left(\average_{B(0,1)} |u_\varep|^2\right)^{1/2}
+\left(\average_{B(0,1)} |F|^p\right)^{1/p} \right\}.
$$
This, together with Caccioppoli's inequality, gives
$$
\aligned
&\left(\average_{B(0, \varep)} |\nabla u_\varep|^2\right)^{1/2}
\le C \left(\average_{B(0, \varep_0 \theta^{k-1} )} |\nabla u_\varep|^2\right)^{1/2}\\
& \qquad \le C \left\{ 
\theta^{-k}\left(\average_{B(0, \theta^{k-1} )} \Big| u_\varep-\average_{B(0, \theta^{k-1})} u_\varep \Big|^2\right)^{1/2}
+\theta^k \left(\average_{B(0, \theta^{k-1})} |F|^2\right)^{1/2} \right\}\\
&\qquad \le 
C \left\{ \left(\average_{B(0,1)} |u_\varep|^2\right)^{1/2}
+\left(\average_{B(0,1)} |F|^p\right)^{1/p} \right\}.
\endaligned
$$
By replacing $u_\e$ with $u_\e-\average_{B(0,1)} u_\e$, we obtain 
$$
\left(\average_{B(0, \varep)} |\nabla u_\varep|^2\right)^{1/2}
\le C  \left\{ \left(\average_{B(0,1)} \Big| u_\varep -\average_{B(0,1)} u_\e \Big|^2\right)^{1/2}
+\left(\average_{B(0,1)} |F|^p\right)^{1/p} \right\},
$$
from which the estimate (\ref{2.2-1}) follows by Poincar\'e's inequality.
\end{proof}

\noindent{\bf A Liouville property}

\medskip

We end this section by proving a Liouville  property for entire solutions of elliptic  systems with 
periodic coefficients. This is done by using the $C^{1,\sigma}$ estimates on mesoscopic scales,
given by Lemma \ref{step-2.2-2}. No smoothness condition is imposed on the coefficient matrix $A$.

\begin{thm}\label{L-theorem}
Suppose that $A$ satisfies (\ref{weak-e-1})-(\ref{weak-e-2}) and is
1-periodic. Let $u\in H^1_{\loc} (\rd; \mathbb{R}^m)$ be a weak solution of
$\mathcal{L}_1 (u)=0$ in $\rd$. Assume that there exist constants $C_u>0$ and $\sigma \in (0,1)$ such that
\begin{equation}\label{L-0}
\left(\average_{B(0, R)} |u|^2\right)^{1/2} \le C_u \, R^{1+\sigma}
\end{equation}
for all $R>1$. Then
\begin{equation}\label{L-1}
u(x) =H + \big( P_j^\beta (x) +\chi_j^\beta (x) \big) E_j^\beta \qquad \text{ in } \rd
\end{equation}
for some constants $H\in \mathbb{R}^m$ and $E =(E_j^\beta)\in \mathbb{R}^{m\times d}$.
\end{thm}

\begin{proof}
We begin by choosing $\sigma_1$ and $\rho$ such that  $\sigma<\sigma_1<\rho<1$.
Let $\varep_0, \theta$ be the positive constants given by Lemma \ref{step-2.2-1} for 
$0<\sigma_1< \rho<1$. 
Suppose that $k\ge 1$ and $\theta^{k+1}<\varep_0$.
Let $v(x) =u(Rx)$, where $R=\theta^{-k-\ell}$ for some $\ell\ge 1$.
Note that $\mathcal{L}_\varep (v)=0$ in $B(0,1)$, where $\varep =R^{-1}$.
Since $\varep=\theta^{k+\ell} <\varep_0 \theta^{\ell-1}$,
in view of Lemma \ref{step-2.2-2}, there exist constants $H(\varep, \ell)\in \mathbb{R}^d$
and $E(\varep, \ell) =\big(E_j^\beta (\varep, \ell)\big)\in \mathbb{R}^{d\times d}$, 
such that
$$
\aligned
&\left(\average_{B(0, \theta^\ell )} 
\left| v (x)- \left( P_j^\beta (x) +\varep \chi_j^\beta (x/\varep) \right) 
E_j^\beta (\varep, \ell) - H(\varep, \ell) \right|^2\, dx\right)^{1/2}\\
&\qquad\qquad\qquad
\le \theta^{\ell (\sigma_1 +1)} \left(\average_{B(0,1)} |v|^2\right)^{1/2}.
\endaligned
$$
By a change of variables this gives
$$
\aligned
&\left(\average_{B(0, \theta^\ell  R)} 
\left| u (x)- \varep \left( P_j^\beta (x) + \chi_j^\beta (x) \right) 
E_j^\beta (\varep, \ell) - H(\varep, \ell) \right|^2\, dx\right)^{1/2}\\
&\qquad\qquad\qquad
\le \theta^{\ell (\sigma_1 +1)} \left(\average_{B(0,R)} |u|^2\right)^{1/2}.
\endaligned
$$
Since $R=\theta^{-k-\ell}$, it follows that
\begin{equation}\label{L-3}
\inf_{\substack{E=(E_j^\beta)\in \mathbb{R}^{d\times m}\\ H\in \mathbb{R}^m}}
\left(\average_{B(0, \theta^{-k})}
\left|u- \left( P_j^\beta +\chi_j^\beta\right) E_j^\beta -H \right|^2\right)^{1/2}
\le C_u \,\theta^{\ell (\sigma_1 +1)} \, \theta^{(-k-\ell)(\sigma+1)},
\end{equation}
where we have used assumption (\ref{L-0}).
Using the fact that $\sigma_1>\sigma$ and $\theta\in (0,1)$, we may let $\ell \to \infty$ in (\ref{L-3}) to conclude that
for each $k$ with $\theta^{k+1}<\varep_0$, there exist constants $H^k \in \mathbb{R}^m$
and $E^k = (E^{k\beta}_j) \in \mathbb{R}^{m\times d}$ such that
$$ 
u(x)= H^k + \left( P_j^\beta (x) + \chi_j^\beta (x) \right)E_j^{k\beta} \qquad \text{ in } B(0, \theta^{-k}).
$$

Finally, note that $\nabla u = \left(\nabla P_j^\beta +\nabla \chi_j^\beta \right) E_j^{k\beta}$ in $Y$.
Since $\int_Y \nabla \chi_j^\beta =0$, we obtain
$$
\int_Y (\nabla u)\, dx =\int_Y \nabla P_j^\beta \cdot E_j^{k\beta}\, dx.
$$
This implies that $E_j^{k_1\beta} =E_j^{k_2\beta}$ for any $k_1$, $k_2$ large.
As a consequence, we also obtain $H^{k_1}=H^{k_2}$ for any $k_1, k_2$ large.
Hence (\ref{L-1}) holds  in $\rd$ for some $H\in \mathbb{R}^m$ 
and $E=(E_j^\beta)\in \mathbb{R}^{m\times d}$.
\end{proof}

\begin{remark}\label{remark-L-1}
{\rm 
Let $u$ be a weak solution of $\mathcal{L}_1 (u)=0$ in $\rd$.
Suppose that there exist $C_u>0$ and $\sigma \in (0,1)$ such that
\begin{equation}\label{r-L-1}
\left(\average_{B(0,R)} |u|^2\right)^{1/2}
\le C_u\, R^\sigma
\end{equation}
for any $R>1$. By Theorem \ref{L-theorem} the solution $u$ is of form  (\ref{L-1}).
This, together with (\ref{r-L-1}), implies that $u$ is constant.
}
\end{remark}



\section{A real-variable method}\label{real-variable-section}

In this section we introduce a real-variable method for $L^p$ estimates.
The method, which will be used to establish $W^{1,p}$ estimates  for $\mathcal{L}_\e$,
 may be regarded as a refined and dual version of the celebrated 
Calder\'on-Zygmund Lemma.

\begin{definition}
{\rm
For $f\in L^1_{\loc} (\br^d)$, the Hardy-Littlewood maximal function $\mathcal{M} (f)$ is defined by
\begin{equation}\label{definition-Hardy-Littlewood}
\mathcal{M}(f) (x)
=\sup\left\{ \average_{B} |f|: \, B \text{ is a ball containing } x \right\}.
\end{equation}
}
\end{definition}

It is well known that the operator $\mathcal{M}$ is bounded on $L^p(\br^d)$ for $1<p\le \infty$, and is of weak type
$(1,1)$:
\begin{equation}\label{weak-type-1-1}
|\Big\{ x\in \br^d: \, \mathcal{M} (f) (x)>t \Big\} |
\le \frac{C}{t} \int_{\br^d} |f|\, dx \qquad \text{ for any } t>0,
\end{equation}
where $C$ depends only on $d$ (see e.g. \cite[Chapter 1]{Stein-1993} for a proof).
For a fixed ball $B$ in $\br^d$,  the localized Hardy-Littlewood maximal function $\mathcal{M}_B (f)$ is
defined by
\begin{equation}\label{local-max}
\mathcal{M}_{B} (f)(x)
=\sup\left\{ \average_{B^\prime} |f|: \ x\in B^\prime \text{ and } B^\prime \subset B\right\}.
\end{equation}
Since $\mathcal{M}_B (f) (x) \le \mathcal{M}(f\chi_B) (x)$ for any $x\in B$, 
it follows that $\mathcal{M}_B$ is bounded on $L^p(B)$ for
$1<p\le \infty$, and is of weak type $(1,1)$.

In the proof of Theorem \ref{real-variable-theorem}  we will perform a Calder\'on-Zygmund decomposition.
It will be convenient to work with  (open) cubes $Q$ in $\br^d$ with sides parallel to the coordinate planes. 
By $tQ$ we denote the cube  that has the same center and $t$ times the side length as $Q$.
We say $Q^\prime$ is a dyadic subcube of $Q$ if $Q^\prime$ may be obtained from $Q$ by
repeatedly bisecting the sides of $Q$. If $Q^\prime$ is obtained from $Q$ by bisecting each side of
$Q$ once, we will call $Q$ the dyadic parent of $Q^\prime$.

\begin{lemma}\label{C-Z}
Let $Q$ be a cube in $\rd$. Suppose that $E\subset  Q$ is open and $|E|<2^{-d} |Q|$. Then
there exists a sequence of disjoint dyadic subcubes $\{ Q_k\}$ of $Q$ such that 
(1) $Q_k\subset E$, (2) the dyadic parent of $Q_k$ in $Q$ is not
contained in $E$, and (3)
$
|E\setminus \cup_k Q_k|=0.
$
\end{lemma}

\begin{proof}
This is a dyadic version of the Calder\'on-Zygmund decomposition.
To prove the lemma, one simply collects all dyadic subcubes $Q^\prime$ of $Q$ 
with the property that $Q^\prime\subset E$ and its dyadic parent
is not contained in $E$; i.e. $Q^\prime$ is maximal among all dyadic subcubes of  $Q$ that are contained in $E$.
Note that since $E$ is open in $Q$,
 the set $E\setminus\cup_k Q_k$ is contained in the union $Z$ of boundaries of all dyadic subcubes of $Q$.
Since $Z$ has measure zero, one obtains 
$
|E\setminus \cup_k Q_k|=0.
$
\end{proof}

\begin{thm}\label{real-variable-theorem}
Let $B_0$ be a ball in $\br^d$ and $F\in L^2(4B_0)$.
Let $q>2$ and $f\in L^p(4B_0)$ for some $2<p<q$.
Suppose that for each ball $B\subset 2B_0$ with 
$|B|\le c_1 |B_0|$, there exist two measurable functions $F_B$ and $R_B$ on $2B$,
such that $|F|\le |F_B|+|R_B|$ on $2B$,
\begin{equation}\label{real-variable-1}
\aligned
\left(\average_{2B} |R_B|^q\right)^{1/q}
& \le N_1 \left\{
\left(\average_{4 B}|F|^2\right)^{1/2}
+\sup_{4B_0\supset B^\prime\supset B}
\left(\average_{B^\prime}
|f|^2\right)^{1/2}\right\},\\
\left(\average_{2B}|F_B|^2 \right)^{1/2}
& \le N_2 \sup_{4B_0\supset B^\prime\supset B} \left(\average_{B^\prime} |f|^2\right)^{1/2}
+\eta \left(\average_{4B} |F|^2\right)^{1/2},
\endaligned
\end{equation}
where $N_1, N_2>0$, $0<c_1<1$, and $\eta\ge 0$.
Then there exists $\eta_0>0$, depending only on $p, q, c_1, N_1, N_2$, with the property that
if $0\le \eta<\eta_0$, then
 $F\in L^p(B_0)$ and
\begin{equation}\label{real-variable-2}
\left(\average_{B_0}| F|^p\right)^{1/p}
\le C
\left\{ \left(\average_{4B_0} |F|^2\right)^{1/2}
+\left(\average_{4B_0} |f|^p\right)^{1/p} 
\right\},
\end{equation}
where $C$ depends only on 
$N_1$, $N_2$, $c_1$, $p$ and $q$.
\end{thm}

\begin{proof}
Let $Q_0$ be a cube such that $2Q_0\subset 2B_0$ and $|Q_0|\approx |B_0|$. We will show that
\begin{equation}\label{real-1}
\left(\average_{Q_0} |F|^p\right)^{1/p}
\le C \left\{ \left(\average_{4B_0} |F|^2\right)^{1/2}
+\left(\average_{4B_0} |f|^p\right)^{1/p} \right\},
\end{equation}
where $C$ depends only on 
$N_1$, $N_2$, $c_1$, $p$, $q$, and $|Q_0|/|B_0|$.
Estimate (\ref{real-variable-2}) follows from (\ref{real-1})
by covering $B_0$ with a finite number of non-overlapping $\overline{Q_0}$ of same size such that
$2Q_0\subset 2B_0$.
 
To prove (\ref{real-1}), let
$$
E(t)=\big\{ x\in Q_0: \, \mathcal{M}_{4B_0} (|F|^2)(x)>t \big\}\qquad \text{ for } t>0.
$$
We claim that if $0\le \eta<\eta_0$ and $\eta_0=\eta_0(p,q,c_1, N_1, N_2)$ is sufficiently small,
 it is possible to choose three constants $\delta, \gamma \in (0,1)$, and
$C_0>0$, depending only on $N_1$, $N_2$, $c_1$, $p$ and $q$,
such that
\begin{equation}\label{real-claim}
|E(\alpha t)|
\le \delta |E(t)| +|\big\{ x\in Q_0: \mathcal{M}_{4B_0} (|f|^2)(x)>\gamma t\big\}|
\end{equation}
for all $t>t_0$, where 
\begin{equation}\label{real-2}
\alpha =(2\delta)^{-2/p} \quad \text{ and } \quad
t_0 =C_0 \average_{4B_0} |F|^2.
\end{equation}
Assume the claim (\ref{real-claim}) for a moment.
We multiply both sides of (\ref{real-claim}) by $t^{\frac{p}{2}-1}$ and then
integrate the resulting inequality in $t$ over the interval $(t_0, T)$.
This leads to
\begin{equation}\label{real-3}
\int_{t_0}^T  t^{\frac{p}{2}-1} |E(\alpha t)|\, dt
\le \delta \int_{t_0}^T t^{\frac{p}{2}-1} |E(t)|\, dt
+C_\gamma \int_{4B_0} |f|^p\, dx,
\end{equation}
where we have used the boundedness of $\mathcal{M}_{4B_0}$
on $L^{p/2}(4B_0)$.
By a change of variables in the LHS of (\ref{real-3}), we may deduce that for any $T>0$,
\begin{equation}\label{real-5}
\alpha^{-\frac{p}{2}}
(1-\delta \alpha^{\frac{p}{2}})
\int_0^T t^{\frac{p}{2}-1} | E(t)|\, dt
\le C |Q_0| t_0^{\frac{p}{2}}
+C_\gamma \int_{4B_0} |f|^p\, dx.
\end{equation}
Note that $\delta \alpha^{p/2}=(1/2)$.
By letting $T\to \infty$ in (\ref{real-5}) we obtain
\begin{equation}\label{real-7}
\int_{Q_0} |F|^p\, dx
\le C |Q_0| t_0^{\frac{p}{2}} +C \int_{4B_0} |f|^p\, dx,
\end{equation}
which, in view of (\ref{real-2}), gives (\ref{real-1}).

It remains to prove (\ref{real-claim}).
To this end we first note that by the weak $(1,1)$ estimate for $\mathcal{M}_{4B_0}$,
$$
|E(t)|\le \frac{C_d}{t} \int_{4B_0} |F|^2\, dx,
$$
where $C_d$ depends only on $d$.
It follows that $|E(t)|< \delta |Q_0|$ for any $t>t_0$, if we choose 
$$
C_0=2\delta^{-1} C_d |4B_0|/|Q_0|
$$
in (\ref{real-2}) with $\delta\in (0,1)$ to be determined.
We now fix $t>t_0$.
Since $E(t)$ is open in $Q_0$, by Lemma \ref{C-Z},
$$
\cup_k Q_k \subset E(t) \quad \text{ and } \quad 
|E(t)\setminus \cup_k Q_k|=0,
$$
where $\{ Q_k\} $ are (disjoint) maximal dyadic subcubes of $Q_0$ contained in $E(t)$.
By choosing $\delta$ sufficiently small, we may assume that $|Q_k|<c_1|Q_0|$.
We will show that if $0\le \eta<\eta_0$ and $\eta_0$ is sufficiently small,
it is possible to choose $\delta, \gamma \in (0,1)$ so small that
\begin{equation}\label{real-claim-2}
|E(\alpha t)\cap Q_k|\le \delta |Q_k|,
\end{equation}
whenever
\begin{equation}\label{real-8}
\Big\{ x\in Q_k: \, \mathcal{M}_{4B_0} (|f|^2)(x) \le \gamma t \Big\} \neq \emptyset.
\end{equation}
It is not hard to see that (\ref{real-claim}) follows from (\ref{real-claim-2}) by summation.
Indeed, 
$$
\aligned
|E(\alpha t)| & =|E(\alpha t)\cap E(t)|\\
&= |E(\alpha t)\cap \cup_k Q_k|\\
&\le \sum_{k^\prime}
|E(\alpha t)\cap Q_{k^\prime}| +|\Big\{ x\in Q_0 :\ 
\mathcal{M}_{4B_0} (|f|^2) (x)>\gamma t \Big\}|\\
&\le \delta \sum_{k^\prime} |Q_{k^\prime}| 
+|\Big\{ x\in Q_0 :\ 
\mathcal{M}_{4B_0} (|f|^2) (x)>\gamma t \Big\}|\\
& \le \delta |E(t)| +|\Big\{ x\in Q_0 :\ 
\mathcal{M}_{4B_0} (|f|^2) (x)>\gamma t \Big\}|,
\endaligned
$$
where the summation is taken only over those $Q_{k^\prime}$ for which 
(\ref{real-8}) hold.

Finally, to see (\ref{real-claim-2}), we fix $Q_k$ that satisfies the condition (\ref{real-8}).
Observe that
\begin{equation}\label{real-9}
\mathcal{M}_{4B_0} (|F|^2)(x)
\le \max \Big\{ \mathcal{M}_{2B_k} (|F|^2)(x), C_d\, t\Big\}
\end{equation}
for any $x\in Q_k$,
where $B_k$ is the ball that has the same center and diameter as $Q_k$.
This is because $Q_k$ is maximal and so its dyadic parent  is not contained in $E(t)$, which in turn implies that
\begin{equation}\label{real-10}
\average_{B^\prime} |F|^2 \le C_d\, t
\end{equation}
for any ball $B^\prime \subset 4B_0$ 
 with the property that $B^\prime \cap B_k\neq \emptyset$ and diam$(B^\prime)\ge \text{\rm diam}(B_k)$.
Clearly we may assume $\alpha>C_d$ by choosing $\delta$ small.
In view of (\ref{real-9}) this implies that
\begin{equation}\label{real-11}
\aligned
|E(\alpha t)\cap Q_k|
&\le \big| \big\{ x\in Q_k: \, \mathcal{M}_{2B_k} (|F|^2)(x)>\alpha t\big\}\big|\\
&\le  \big| \left\{ x\in Q_k: \, \mathcal{M}_{2B_k}(|F_{B_k}|^2)(x)>\frac{\alpha t}{4}\right\}\big|\\
&\qquad
+
\big| \left\{ x\in Q_k: \, \mathcal{M}_{2B_k}(|R_{B_k}|^2)(x)>\frac{\alpha t}{4}\right\}\big|\\
& \le \frac{C_d}{\alpha t}\int_{2B_k} |F_{B_k}|^2\, dx
+\frac{C_{d, q}}{(\alpha t)^{\frac{q}{2}}}\int_{2B_k} |R_{B_k}|^q\, dx,
\endaligned
\end{equation}
where we have used the assumption $|F|\le |F_{B_k}| +|R_{B_k}|$ on $2B_k$ as well as the weak $(1,1)$
and weak $(q/2,q/2)$ bounds of $\mathcal{M}_{2B_k}$.

By the second inequality in the assumption (\ref{real-variable-1}), we have
\begin{equation}\label{real-12}
\aligned
\average_{2B_k} |F_{2B_k}|^2\, dx
& \le 2N^2_2  \sup_{4B_0\supset B^\prime\supset B_k} \average_{B^\prime} |f|^2 +2 \eta^2 \average_{2B_k} |F|^2\\
& \le 2N_2^2 \cdot \gamma t +2\eta^2 C_d\, t, 
\endaligned
\end{equation}
where the last inequality follows from (\ref{real-8}) and (\ref{real-10}).
Similarly, we may use the first inequality in (\ref{real-variable-1})  and (\ref{real-10}) to obtain
\begin{equation}\label{real-13}
\aligned
\average_{2B_k} |R_{B_k}|^q\, dx
&\le N_1^q \left\{ \left(\average_{4B_k} |F|^2\right)^{1/2} +(\gamma t)^{1/2} \right\}^q\\
& \le N_1^q C_{d,q} t^{q/2}.
\endaligned
\end{equation}
We now use (\ref{real-12}) and (\ref{real-13}) to bound the right side of (\ref{real-11}).
This yields
\begin{equation}\label{real-14}
\aligned
|E(\alpha t)\cap Q_k|
&\le |Q_k| \Big\{ C_d \cdot N_2^2 \cdot \gamma \cdot \alpha^{-1} 
+\eta^2 \cdot C_d\cdot \alpha^{-1}+C_{d,q} \cdot N_1^q \cdot \alpha^{-q/2}    \Big\} \\
&\le \delta |Q_k| \left\{ C_d \cdot N_2^2 \cdot \gamma\cdot  \delta^{\frac{2}{p}-1}
+C_d\cdot \eta_0^2 \cdot \delta^{\frac{2}{p}-1}
+C_{d,q} \cdot N_1^q\cdot  \delta^{\frac{q}{p}-1} \right\},
\endaligned
\end{equation}
where we have used the fact $\alpha=(2\delta)^{-\frac{2}{p}}$.
Note that since $p<q$, it is possible to choose $\delta\in (0,1)$ so small that
$$
C_{d,q} N_1^q \delta^{\frac{q}{p}-1} \le (1/4).
$$
After $\delta$ is chosen, we then choose $\gamma\in (0,1)$ and $\eta_0\in (0,1)$ so small that 
$$
C_d \cdot N_2^2 \cdot \gamma\cdot  \delta^{\frac{2}{p}-1}
+C_d\cdot \eta_0^2 \cdot \delta^{\frac{2}{p}-1}\le (1/4).
$$
This gives 
$$
\aligned
|E(\alpha t)\cap Q_k| &\le (1/2)\delta|Q_k|\\
&<\delta |Q_k|
\endaligned
$$
 and finishes the proof.
\end{proof}

\begin{remark}\label{real-remark-1}
{\rm
The fact that $L^2$ is a Hilbert space plays no role in the proof of Theorem \ref{real-variable-theorem}.
Consequently, one may replace the $L^2$ average in the assumption (\ref{real-variable-1})
by the $L^{p_0}$ average for any $1\le p_0<q$, and obtain 
$$
\left(\average_{B_0}| F|^p\right)^{1/p}
\le C
\left\{ \left(\average_{4B_0} |F|^{p_0}\right)^{1/p_0}
+\left(\average_{4B_0} |f|^p\right)^{1/p} 
\right\},
$$
for $p_0<p<q$, in the place of (\ref{real-variable-2}).
}
\end{remark}

An operator $T$ is called sublinear if there exists a constant $K$ such that
\begin{equation}\label{sublinear-2.3}
|T(f+g)|\le K \big\{ |T(f)| +|T(g)|\big\}.
\end{equation}

\begin{thm}\label{real-variable-operator-theorem}
Let $T$ be a bounded sublinear operator on $L^2(\br^d)$ with $\|T\|_{L^2\to L^2}\le C_0$. Let $q>2$.
Suppose that
\begin{equation}\label{real-operator-condition}
\left( \average_{B} |T(g)|^q\right)^{1/q}
\le N \left\{ \left( \average_{2B} |T(g)|^2\right)^{1/2}
+\sup_{B^\prime\supset B} 
\left(\average_{B^\prime} |g|^2\right)^{1/2} \right\}
\end{equation}
for any ball $B$ in $\br^d$ and for any $g\in C_0^\infty(\br^d)$ with supp$(g)\subset \br^d\setminus 4B$.
Then for any $f\in C_0^\infty(\br^d)$,
\begin{equation}\label{real-operator-conclusion}
\| T(f)\|_{L^p(\br^d)} \le C_p \, \| f\|_{L^p(\br^d)},
\end{equation}
where $2<p<q$ and $C_p$ depends at most on $p$, $q$, $C_0$, $N$, and $K$ in (\ref{sublinear-2.3}).
\end{thm}

\begin{proof}
Let $f\in C_0^\infty(\br^d)$ and $F=T(f)$. Suppose that supp$(f)\subset B(0, \rho)$. Let $B_0=B(0, R)$, where
$R>100 \rho$. For each ball $B\subset 2B_0$ with $|B|\le (100)^{-1} |B_0|$, we define
$$
F_B=K T(f\varphi_B)\quad \text{ and } \quad R_B = K  T(f(1-\varphi_B)),
$$
where $\varphi_B\in C_0^\infty(9B)$ such that $0\le \varphi_B\le 1$ and $\varphi_B =1$ in $8B$.
Clearly, by (\ref{sublinear-2.3}), $|F|\le |F_B| + |R_B|$ in $\br^d$. By the $L^2$ boundedness of $T$, we have
$$
\left(\average_{2B} |F_B|^2\right)^{1/2}
\le C \left(\average_{9B}|f|^2 \right)^{1/2}.
$$
In view of the assumption (\ref{real-operator-condition}) we obtain
$$
\aligned
\left(\average_{2B} |R_B|^q\right)^{1/q}
&\le C \left(\average_{4B} |R_B|^2\right)^{1/2}
+ C \sup_{4B_0\supset B^\prime\supset B}
\left(\average_{B^\prime} |f|^2\right)^{1/2}\\
&\le C \left( \average_{4B} |T(f)|^2\right)^{1/2}
+C \left(\average_{4B} |T(f\varphi_B)|^2\right)^{1/2} +
C\sup_{4B_0\supset B^\prime\supset B}
\left(\average_{B^\prime} |f|^2\right)^{1/2}\\
& \le C\left(\average_{4B} |F|^2\right)^{1/2}
+C\sup_{4B_0\supset B^\prime\supset B}
\left(\average_{B^\prime} |f|^2\right)^{1/2},
\endaligned
$$
where we have used the $L^2$ boundedness of $T$ in the last inequality.
It now follows from Theorem \ref{real-variable-theorem} (with $\eta=0$) that
\begin{equation}\label{real-operator-1}
\left(\int_{B_0} |T(f)|^p\, dx \right)^{1/p}
\le C |B_0|^{\frac{1}{p}-\frac12} \left(\int_{4B_0} |T(f)|^2\, dx\right)^{1/2}
+C \left(\int_{4B_0} |f|^p\, dx \right)^{1/p}
\end{equation}
for $2<p<q$. By letting $R\to \infty$ in (\ref{real-operator-1}) and using the fact that $T(f)\in L^2(\br^d)$,
we obtain the estimate (\ref{real-operator-conclusion}).
\end{proof}

We may use Theorem \ref{real-variable-operator-theorem} to show that the Calder\'on-Zygmund operators 
are bounded in $L^p(\br^d)$ for any $1<p<\infty$.
Indeed, suppose that $T$ is a Calder\'on-Zygmund operator; i.e. $T$ is 
a bounded linear operator on $L^2(\br^d)$ and is associated with a Calder\'on-Zygmund kernel
 $K(x,y)$ in the sense that
$$
T(f)(x)=\int_{\br^d} K(x,y) f(y)\, dy \quad \text{ for } x\notin \text{supp} (f),
$$
if $f$ is a bounded measurable function with compact support.
A measurable function $K(x,y)$ in $\br^d\times \br^d$ 
is called a Calder\'on-Zygmund kernel, if there exist $C>0$ and $\delta\in (0,1]$ such that
$$
|K(x,y)|\le C |x-y|^{-d}
$$
for any $x,y\in \br^d$, $x\neq y$, and 
$$
|K(x,y+h)-K(x,y)| +|K(x+h, y)-K(x,y)| \le \frac{C|h|^\delta}{|x-y|^{d+\delta}}
$$
for any $x, y, h\in \br^d$ and $|h|<(1/2)|x-y|$ (see e.g. \cite{Stein-1993}).
Suppose supp$(f)\subset \br^d\setminus 4B$. Then for any $x,z\in B$,
$$
\aligned
|T(f)(x)-T(f)(z)|
&\le \int_{\br^d} |K(x,y)-K(z,y)|\, | f(y)|\, dy\\
& \le C |x-z|^\delta \int_{\br^d\setminus 4B} \frac{|f(y)|\, dy}{|z-y|^{d+\delta}}\\
&\le C \sup_{B^\prime\supset B} \average_{B^\prime} |f|.
\endaligned
$$
This implies that
$$
\|T(f)\|_{L^\infty(B)} \le  \average_{B} |T(f)| +C \sup_{B^\prime\supset B} \average_{B^\prime} |f|.
$$
It follows that $T$ satisfies the condition (\ref{real-operator-condition}) for any $q>2$.
Since $T$ is bounded on $L^2(\br^d)$, by Theorem \ref{real-variable-operator-theorem}, we may conclude that
$T$ is bounded on $L^p(\br^d)$ for any $2<p<\infty$.
The boundedness of $T$ on $L^p(\br^d)$ for $1<p<2$ follows by duality.

The next two theorems extend Theorems \ref{real-variable-theorem}  and \ref{real-variable-operator-theorem}
to bounded Lipschitz domains.

\begin{thm}\label{real-variable-Lipschitz-theorem}
Let  $q>2$ and $\Omega$ be a bounded Lipschitz domain. Let $F\in L^2(\Omega)$
and $f\in L^p(\Omega)$ for some $2<p<q$.
Suppose that for each ball $B$ with the property that $|B|\le c_0|\Omega|$ and
either $4B\subset \Omega$ or $B$ is centered on $\partial\Omega$,
 there exist two measurable functions $F_B$ and $R_B$ on $\Omega\cap 2B$,
such that $|F|\le |F_B|+|R_B|$ on $\Omega\cap 2B$,
\begin{equation}\label{real-variable-Lipschitz-1}
\aligned
\left(\average_{\Omega\cap 2B} |R_B|^q\right)^{1/q}
& \le N_1 \left\{
\left(\average_{\Omega\cap 4 B}|F|^2\right)^{1/2}
+\sup_{4B_0\supset B^\prime\supset B}
\left(\average_{\Omega\cap B^\prime}
|f|^2\right)^{1/2}\right\},\\
\left(\average_{\Omega\cap 2B}|F_B|^2 \right)^{1/2}
& \le N_2 \sup_{4B_0\supset B^\prime\supset B} \left(\average_{\Omega\cap B^\prime} |f|^2\right)^{1/2}
+\eta \left(\average_{\Omega\cap 4B} |F|^2\right)^{1/2},
\endaligned
\end{equation}
where $N_1, N_2>0$ and $0<c_0<1$.
Then there exists $\eta_0>0$, depending only on $N_1$, $N_2$, $c_0$, $p$, $q$, and the Lipschitz 
character of $\Omega$, with the property that if $0\le \eta<\eta_0$, then $F\in L^p(\Omega)$ and
\begin{equation}\label{real-variable-Lipschitz-2}
\left(\average_{\Omega}| F|^p\right)^{1/p}
\le C
\left\{ \left(\average_{\Omega} |F|^2\right)^{1/2}
+\left(\average_{\Omega} |f|^p\right)^{1/p} 
\right\},
\end{equation}
where $C$ depends at most on 
$N_1$, $N_2$, $c_0$, $p$, $q$, and the Lipschitz character of $\Omega$.
\end{thm}

\begin{proof}
We will deduce the theorem from Theorem \ref{real-variable-theorem}.
Choose a ball $B_0$ so that $\Omega\subset B_0$ and diam$(B_0)=$diam$(\Omega)$.
Consider the functions $\widetilde{f}=f\chi_\Omega$ and $\widetilde{F}=F\chi_\Omega$ in $4B_0$.
Let $B=B(x_0,r) $ be a ball such that $B\subset 2B_0$ and $|B|\le (100)^{-1} c_1 |\Omega|$.
If $4B\subset \Omega$, we simply choose $\widetilde{F}_B=F_B$ and $\widetilde{R}_B =R_B$.
If $2B\cap\Omega=\emptyset$, we let $\widetilde{F}_B=\widetilde{R}_B=0$.
Note that in the remaining case, we have 
$$
\Omega\cap 4B\neq \emptyset \quad \text{ and } \quad 4B\cap (\br^d\setminus \overline{\Omega})\neq \emptyset.
$$
This implies that there exists $y_0\in \partial\Omega$ such that $4B\subset B(y_0,8r)$.
Let $B_1=B(y_0, 8r)$ and define $\widetilde{F}_B$, $\widetilde{R}_B$ by
$$
\widetilde{F}_B= F_{B_1}\chi_\Omega \quad \text{ and } \quad \widetilde{R}_B =R_{B_1} \chi_\Omega.
$$
It is not hard to verify that the functions $\widetilde{F}_B$ and $\widetilde{R}_B$ in all three cases satisfy the conditions
in Theorem \ref{real-variable-theorem}.
As a result the inequality (\ref{real-variable-Lipschitz-2}) follows from (\ref{real-variable-2}).
\end{proof}

\begin{thm}\label{real-variable-operator-Lipschitz-theorem}
Let $q>2$ and $\Omega$ be a bounded Lipschitz domain in $\br^d$.
Let $T$ be a bounded sublinear operator on $L^2(\Omega)$ with $\|T\|_{L^2\to L^2}\le C_0$. 
Suppose that for any ball $B$ in $\br^d$ with the property
that $|B|\le c_0 |\Omega|$ and either $2B\subset \Omega$ or $B$ is centered on $\partial\Omega$,
\begin{equation}\label{real-operator-Lipschitz-condition}
\left( \average_{\Omega\cap B} |T(g)|^q\right)^{1/q}
\le N \left\{ \left( \average_{\Omega\cap 2B} |T(g)|^2\right)^{1/2}
+\sup_{\Omega\supset B^\prime\supset B} 
\left(\average_{\Omega\cap B^\prime} |g|^2\right)^{1/2} \right\}
\end{equation}
where $g\in L^2(\Omega)$ and $g=0$ in $\Omega\cap 4B$.
Then, for any $f\in L^p(\Omega)$,
\begin{equation}\label{real-operator-Lipschitz-conclusion}
\| T(f)\|_{L^p(\Omega)} \le C_p \, \| f\|_{L^p(\Omega)},
\end{equation}
where $C_p$ depends at most on $p$, $q$, $c_0$, $C_0$, $N$, $K$ in (\ref{sublinear-2.3}), and the Lipschitz character of $\Omega$.
\end{thm}

Theorem \ref{real-variable-operator-Lipschitz-theorem} follows from Theorem \ref{real-variable-Lipschitz-theorem}
in the same manner as Theorem \ref{real-variable-operator-theorem} follows  from Theorem \ref{real-variable-theorem}.
We leave the details to the reader.



\section[$W^{1,p}$ estimates]{Interior $W^{1,p}$ estimates}\label{section-2.4}

In this section we establish interior $W^{1,p}$ estimates for solutions of
$\mathcal{L}_\varep (u_\varep) =F$ under the assumption that
$A$ satisfies the ellipticity condition (\ref{weak-e-1}-(\ref{weak-e-2}), is 1-periodic, and belongs to VMO$(\rd)$.
In order to quantify the smoothness,
we will impose the following condition:
\begin{equation}\label{VMO-1}
\sup_{x\in \rd} \average_{B(x,t)} \big| A-\average_{B(x,t)} A \big|\le \omega (t)\qquad \text{ for } 0<t \le 1,
\end{equation}
where $\omega$ is a nondecreasing continuous function on $[0, 1]$ with
$\omega (0)=0$. 

\begin{thm}[Inteior $W^{1,p}$ estimate]\label{interior-W-1-p-theorem}
Suppose that $A$ satisfies  (\ref{weak-e-1})-(\ref{weak-e-2}) and is 1-periodic.
Also assume that $A$ satisfies the VMO condition (\ref{VMO-1}).
Let $H\in L^p(2B; \mathbb{R}^m)$ and $G\in L^p(2B; \mathbb{R}^{m\times d})$
 for some $2<p<\infty$ and ball $B=B(x_0,r)$. 
Suppose that $u_\varep \in H^1(2B;\br^m)$ and $\mathcal{L}_\varep (u_\varep)=H+\text{\rm div} (G)$ in $2B$.
Then $|\nabla u_\varep| \in L^p (B)$ and
\begin{equation}\label{W-1-p-estimate}
\left(\average_{B} |\nabla u_\varep|^p\right)^{1/p}
\le C_p \left\{ \left(\average_{2B} |\nabla u_\varep|^2\right)^{1/2}
+ r\left(\average_{2B} |H|^p\right)^{1/p}
+\left(\average_{2B} |G|^p\right)^{1/p} \right\},
\end{equation}
where $C_p$ depends only on $\mu$, $p$, and the function $\omega(t)$ in (\ref{VMO-1}).
\end{thm}

Theorem \ref{interior-W-1-p-theorem} is proved by using the real-variable argument given in the previous section.
Roughly speaking, the argument reduces the $W^{1, p}$ estimate to a (weak) reverse H\"older
inequality for some exponent $q>p$.

\begin{lemma}\label{local-lemma-2.4}
Suppose that $A$ satisfies (\ref{weak-e-1})-(\ref{weak-e-2}) and  (\ref{VMO-1}).
Let $u\in H^1(2B; \mathbb{R}^m)$ be a weak solution of
$\mathcal{L}_1 (u)=0$ in $2B$ for some $B=B(x_0, r)$ with $0<r\le 1$.
Then $|\nabla u|\in L^p (B)$ for any $2<p<\infty$, and
\begin{equation}\label{local-estimate-2.4}
\left(\average_{B} |\nabla u|^p \right)^{1/p}
\le C_p \left(\average_{2B} |\nabla u|^2\right)^{1/2},
\end{equation}
where $C_p$ depends only on $\mu$, $p$, and $\omega(t)$.
\end{lemma}

\begin{proof}
This is a local $W^{1, p}$ estimate for second-order elliptic systems in divergence form with VMO coefficients.
We prove it by using Theorem \ref{real-variable-theorem} with $F=|\nabla u|$ and $f=0$.
Let $B^\prime\subset B$ with $|B^\prime|\le (1/8)^d |B|$.
Let $v\in H^1(2B^\prime; \mathbb{R}^m)$ be the weak solution to the Dirichlet problem:
$$
\text{div} \big(A^0\nabla v)=0 \quad \text{ in } 3B^\prime \quad \text{ and } \quad v=u \quad \text{ on } \partial (3B^\prime),
$$
where 
$$
A^0=\average_{3B^\prime} A
$$
is a constant matrix.
Recall that  (\ref{weak-e-1})-(\ref{weak-e-2}) implies (\ref{L-ellipticity}).
It follows that $A^0$ satisfies  the Legendre-Hadamard ellipticity condition. Let
$$
F_{B^\prime} =|\nabla (u-v)| \quad \text{ and } \quad R_{B^\prime} =|\nabla v|.
$$
We will show that $F_{B^\prime}$ and $R_{B^\prime}$ satisfy the conditions in Theorem \ref{real-variable-theorem}.
Clearly, $F=|\nabla u|\le F_{B^\prime} +R_{B^\prime}$ on $2B^\prime$.
By the interior Lipschitz estimates for elliptic systems with constant coefficients,
$$
\aligned
\max_{2B^\prime} R_{B^\prime}
=\max_{2B^\prime} |\nabla v| &\le C \left(\average_{3B^\prime} |\nabla v|^2 \right)^{1/2}\\
&\le C \left(\average_{3B^\prime} |\nabla u|^2\right)^{1/2}\\
&\le C
\left(\average_{4B^\prime} |F|^{2} \right)^{1/2},
\endaligned
$$
where $C$ depends only on $\mu$.
Since $\mathcal{L}_1 (u)=0$ in $4B^\prime\subset 2B$,
by the reverse H\"older inequality, 
$$
\left(\average_{3B^\prime} |\nabla u|^{q} \right)^{1/q}
\le C
\left(\average_{4B^\prime} |\nabla u|^{2} \right)^{1/2},
$$
where $C>0$ and $q>2$ depend only on $\mu$.
Note that  $u-v\in H^1_0(3B^\prime; \mathbb{R}^m)$ and
$$
\text{\rm div} \big( A^0 \nabla (u-v) \big) =\text{\rm div} \big( (A^0-A)\nabla u\big) \quad \text{ in } 3B^\prime.
$$
It follows that
$$
\aligned
 \mu \left(\average_{3B^\prime} |\nabla (u-v)|^2\right)^{1/2}  &\le  \left( \average_{3B^\prime} |(A-A^0)\nabla u|^2\right)^{1/2}\\
&\le  \left(\average_{3B^\prime} |A-A^0|^{2p_0^\prime}\right)^{1/2p_0^\prime}
\left(\average_{3B^\prime} |\nabla u|^{2p_0} \right)^{1/2p_0}\\
&\le  \eta(r) \left(\average_{4B^\prime} |\nabla u|^{2} \right)^{1/2},
\endaligned
$$
where $p_0=(q/2)>1$ and
$$
\eta (r)=\sup_{\substack{x\in \rd\\ 0<s<r}}
 \left(\average_{B(x,s)} \Big| A -\average_{B(x,s)} A \Big|^{2p_0^\prime} \right)^{1/2p_0^\prime}.
$$
Hence,
$$
\left(\average_{3B^\prime} |F_{B^\prime}|^2\right)^{1/2}
\le C\,  \eta(r) 
\left(\average_{4B^\prime} |F|^2\right)^{1/2}.
$$

Finally, we observe that by the John-Nirenberg inequality, if $A\in \text{VMO}(\rd)$,
we have $\eta (r)\to 0$ as $r\to 0$.
Moreover, given any $\eta_0\in (0,1)$, there exists $r_0>0$, depending only on $\mu$ and the function
$\omega (t)$ in (\ref{VMO-1}),
such that $0\le C\, \eta (r)<\eta_0$ if $0<r<r_0$.
As a result, if $0<r<r_0$, the functions $F_{B^\prime}$ and $R_{B^\prime}$ satisfy the conditions
in Theorem \ref{real-variable-theorem}. Consequently, the estimate (\ref{local-estimate-2.4})
holds for $0<r<r_0$.
By a simple covering argument it also holds for any $0<r\le 1$.
\end{proof}

\begin{lemma}\label{global-lemma-2.4}
Suppose that $A$ satisfies (\ref{weak-e-1})-(\ref{weak-e-2}) and (\ref{VMO-1}).
Also assume that $A$ is 1-periodic.
Let $u_\varep\in H^1(2B; \mathbb{R}^m)$ be a weak solution of
$\mathcal{L}_\varep (u_\varep)=0$ in $2B$ for some $B=B(x_0, r)$.
Then $|\nabla u_\varep|\in L^p (B)$ for any $2<p<\infty$, and
\begin{equation}\label{global-estimate-2.4}
\left(\average_{B} |\nabla u_\varep|^p \right)^{1/p}
\le C_p \left(\average_{2B} |\nabla u_\varep |^2\right)^{1/2},
\end{equation}
where $C_p$ depends only on $\mu$, $p$, and $\omega(t)$.
\end{lemma}

\begin{proof}
By translation and dilation we may assume  $B=B(0,1)$.
The case $\varep\ge (1/4)$ follows readily from Lemma \ref{local-lemma-2.4},
as the coefficient matrix $A(x/\varep)$ satisfies (\ref{VMO-1}) uniformly in $\varep\ge (1/2)$.
Suppose now that $0<\varep<(1/4)$, $p>2$, and $\mathcal{L}_\varep (u_\varep)=0$ in $B(0,2)$.
By a simple blowup argument we may deduce from Lemma \ref{local-lemma-2.4} that
$$
\left(\average_{B(y, \varep)} |\nabla u_\varep|^p \right)^{1/p}
\le C 
\left(\average_{B(y, 2\varep)} |\nabla u_\varep|^2 \right)^{1/2}
$$
for any $y\in B(0,1)$. This, together with Theorem \ref{theorem-2.2.1}, gives
$$
\left(\average_{B(y, \varep)} |\nabla u_\varep|^p \right)^{1/p}
\le C \left(\average_{B(y, 1)} |\nabla u_\varep|^2 \right)^{1/2}.
$$
It follows that
$$
\int_{B(y, \varep)} |\nabla u_\varep|^p\, dx \le C\, \varep^d\, \|\nabla u_\varep\|^{p/2}_{L^2(B(0,2))},
$$
which yields estimate (\ref{global-estimate-2.4}) by an integration in $y$ over $B(0,1)$.
\end{proof}

\begin{remark}\label{remark-H}
{\rm
Suppose that $A$ satisfies the same conditions as in Lemma \ref{global-lemma-2.4}.
Let $\mathcal{L}_\varep (u_\varep)=0$ in $B=B(x_0, r)$.
Then, for any $0<t<r$ and $\sigma \in (0,1)$,
\begin{equation}\label{H-0}
\left(\average_{B(x_0, t)} |\nabla u_\varep|^2 \right)^{1/2}
\le C_\sigma \left(\frac{t}{r}\right)^{\sigma-1}
\left(\average_{B(x_0, r)} |\nabla u_\varep|^2\right)^{1/2},
\end{equation}
where $C_\sigma$ depends only on $\sigma$, $\mu$ and $\omega(t)$ in (\ref{VMO-1}).
This follows from Lemma \ref{global-lemma-2.4} and H\"older's inequality.
}
\end{remark}

\begin{proof}[\bf Proof of Theorem \ref{interior-W-1-p-theorem}]

Suppose that $\mathcal{L}_\varep (u_\varep) =H+\text{\rm div}(G)$ in $2B_0$,
where  $B=B(x_0, r_0)$. By dilation we may assume that $r_0=1$.
We shall apply Theorem \ref{real-variable-theorem} with $q=p+1$, $\eta=0$,
$$
 F=|\nabla u_\varep| \quad \text{ and } \quad
f= |H| +|G|.
$$

For each ball $B^\prime $ such that $4B^\prime \subset 2B_0$, 
we write $u_\varep=v_\varep +w_\varep $ in $2B^\prime$, where $v_\varep\in H_0^1(4B^\prime;\br^m)$ is
 the weak solution to
$
\mathcal{L}_\varep (v_\varep ) =H+\text{\rm div}(G)
\text{ in }  4B^\prime
$ and $v_\varep=0$ on $\partial (4B^\prime)$.
Let 
$$
F_{B^\prime}=|\nabla v_\varep|\quad \text{ and  }\quad R_{B^\prime}=|\nabla w_\varep|.
$$
Clearly, $|F|\le F_{B^\prime} + R_{B^\prime}$ in $2B^\prime$. It is also easy to see that by Theorem \ref{theorem-1.1-2},
$$
\aligned
\left(\average_{4B^\prime} |F_{B^\prime}|^2 \right)^{1/2}
 & \le C\left(\average_{4B^\prime} \big( r^\prime |H|+|G| \big)^2\right)^{1/2}\\
&\le  C\left( \average_{4B^\prime} |f|^2\right)^{1/2},
\endaligned
$$
where $r^\prime$ is the radius of $B^\prime$.
To verify the remaining condition in (\ref{real-variable-1}), we note that
$w_\varep \in H^1(4B^\prime;\br^m)$ and $\mathcal{L}_\varep (w_\varep)=0$ in $4B^\prime$.
It then follows from Lemma \ref{global-lemma-2.4} that
$$
\aligned
\left(\average_{2B^\prime} |\nabla w_\varep|^q\right)^{1/q}
& \le C 
\left(\average_{4B^\prime} |\nabla w_\varep|^2\right)^{1/2}\\
& 
\le C
 \left(\average_{4B^\prime} |\nabla u_\varep|^2\right)^{1/2}
+ C\left(\average_{4B^\prime} |\nabla v_\varep|^2\right)^{1/2}\\
&\le C \left(\average_{4B^\prime} |F|^2\right)^{1/2}
+C \left(\average_{4B^\prime} | f|^2\right)^{1/2}.
\endaligned
$$
By Theorem \ref{real-variable-theorem} we obtain
\begin{equation}\label{w-1-p-2}
\left(\average_{B}|\nabla u_\varep|^p\right)^{1/p}
\le C \left(\average_{4B}|\nabla u_\varep|^2\right)^{1/2}
+C \left(\average_{4B}| f|^p \right)^{1/p}
\end{equation}
for any ball $B$ such that $4B\subset 2B_0$.
By a simple covering argument this implies 
\begin{equation}\label{w-1-p-3}
\aligned
\left(\average_{B_0}|\nabla u_\varep|^p\right)^{1/p}
&\le C \left(\average_{2B_0}|\nabla u_\varep|^2\right)^{1/2}
+C \left(\average_{2B_0}| f|^p \right)^{1/p}\\
&\le C \left(\average_{2B_0}|\nabla u_\varep|^2\right)^{1/2}
+C\left(\average_{2B_0} |H|^p\right)^{1/p}
+C \left(\average_{2B_0} |G|^p\right)^{1/p},
\endaligned
\end{equation}
where $C$ depends only on $\mu$, $p$ and $\omega(t)$ in (\ref{VMO-1}).
\end{proof}

Consider the homogeneous Sobolev space
$$
\dot{W}^{1,2}(\br^d, \br^m)
=\Big \{ u\in L_\loc^{2}(\br^d;\br^m): \ \nabla u\in L^2(\br^d;\br^{m\times d})\Big\}.
$$
Elements of $\dot{W}^{1,2}(\br^d;\br^m)$ are equivalent classes of functions under the relation that
$u\sim v$ if $u-v$ is constant.
It follows from the ellipticity condition (\ref{weak-e-1})-(\ref{weak-e-2}) and the Lax-Milgram Theorem that 
for any $f=(f_i^\alpha)\in L^2(\br^d;\br^{m\times d})$,
there exists a unique $u_\varep \in \dot{W}^{1,2}(\br^d;\br^m)$
such that $\mathcal{L}_\varep (u_\varep)=\text{div} (f)$ in $\br^d$.
Moreover, the solution satisfies the estimate
$$
\|\nabla u_\varep \|_{L^2(\br^d)}
\le C\, \| f\|_{L^2(\br^d)},
$$
where $C$ depends only on $\mu$.
The following theorem gives the $W^{1,p}$ estimate in $\br^d$.

\begin{thm}\label{W-1-p-theorem-2}
Suppose that $A$ satisfies (\ref{weak-e-1})-(\ref{weak-e-2}) and is 1-periodic.
Also assume that $A$ satisfies the VMO condition  (\ref{VMO-1}).
Let $f=(f_i^\alpha)\in C_0^\infty(\brd, \mathbb{R}^{m\times d})$ and $1<p<\infty$.
Then the unique solution in $\dot{W}^{1,2}(\br^d;\br^d)$ to $\mathcal{L}_\varep (u_\varep)
=\text{\rm div} (f)$ in $\br^d$ satisfies the estimate
\begin{equation}\label{estimate-2.3.3}
\| \nabla u_\varep\|_{L^p(\br^d)}
\le C_p \, \| f\|_{L^p(\br^d)},
\end{equation}
where $C_p$ depends only on $\mu$, $p$, and the function $\omega(t)$.
\end{thm}

\begin{proof}
We first consider the case $p>2$.
Let $B=B(0,r)$.
It follows from (\ref{W-1-p-estimate}) that
$$
\aligned
\| \nabla u_\varep\|_{L^p(B)}
& \le C\, |B|^{\frac{1}{p}-\frac12}\,  \|\nabla u_\varep\|_{L^2(2B)}
+C \, \| f\|_{L^p(2B)}\\
&\le 
C \, |B|^{\frac{1}{p}-\frac12} \, \|\nabla u_\varep\|_{L^2(\br^d)}
+C\, \| f\|_{L^p(\br^d)}.
\endaligned
$$
By letting $r\to \infty$, this gives the estimate (\ref{estimate-2.3.3}).

The case $1<p<2$ follows by a duality argument.
Indeed, suppose $\mathcal{L}_\varep (u_\varep)=\text{div}(f)$ in $\br^d$ and
$\mathcal{L}^*_\varep (v_\varep) =\text{div}(g)$ in $\br^d$,
where $u_\varep, v_\varep \in \dot{W}^{1,2}(\br^d;\br^m)$ and
$f, g\in C_0^\infty(\br^d;\br^{m\times d})$.
Then
$$
\int_{\br^d}
f\cdot \nabla v_\e\, dx
=-\int_{\br^d} A(x/\varep) \nabla u_\e \cdot \nabla v_\e \, dx
=\int_{\br^d}
g \cdot  \nabla u_\e\, dx.
$$
Since $\| \nabla v_\varep \|_{L^q(\br^d)} \le C\, \| g\|_{L^q(\br^d)}$ for $q=p^\prime>2$, we obtain
$$
\big|
\int_{\br^d}
g\cdot  \nabla u_\e\, dx
\big|
\le C \,\| f\|_{L^p(\br^d)} \| g\|_{L^q(\br^d)}.
$$
By duality this yields $\|\nabla u_\varep \|_{L^p(\br^d)} \le C \| f\|_{L^p(\br^d)}$.
\end{proof}

\begin{remark}\label{pert-r}
{\rm
Without the periodicity and VMO condition on $A$, the estimate (\ref{estimate-2.3.3})
holds if 
$$
\Big|\frac{1}{p}-\frac12\Big|< \delta,
$$
 where $\delta>0$ depends only on $\mu$.
 This result, which is due to N. Meyers \cite{Meyers-1963}, follows readily from the proof of Theorem \ref{W-1-p-theorem-2},
 using the reverse H\"older inequality (\ref{reverse-Holder-1.1}).
}
\end{remark}

The next theorem gives the interior H\"older estimates.
 
\begin{thm}[Interior H\"older estimate]\label{interior-Holder-theorem}
Suppose that $A$ satisfies  (\ref{weak-e-1})-(\ref{weak-e-2}) and (\ref{VMO-1}).
Also assume that $A$ is 1-periodic.
Let $u_\varep \in H^1(2B;\br^m)$ and 
$
\mathcal{L}_\varep (u_\varep) =F+\text{\rm div}(G) \text{ in } 2B
$
for some ball $B=B(x_0, r)$.
Then for $p>d$,
\begin{equation}\label{estimate-2.3.1}
|u_\varep (x) -u_\varep (y)|
 \le C_p \left(\frac{|x-y|}{r}\right)^{ 1-\frac{d}{p} }
\left\{  \left(\average_{2B} | u_\varep|^2\right)^{1/2} 
+r^2 \left(\average_{2B}|F|^p\right)^{1/p}
+r \left(\average_{2B} |G|^p\right)^{1/p}
  \right\}
\end{equation}
for any $x,y\in B$,
where $C_p$ depends only on $\mu$,  $p$, and $\omega(t)$.
\end{thm}

\begin{proof}
This follows from Theorem \ref{interior-W-1-p-theorem} by the Sobolev imbedding,
$$
|u_\varep (x) -u_\varep (y)|\le C_p\,  r\, \left(\frac{|x-y|}{r}\right)^{1-\frac{d}{p}}
\left(\average_{B} |\nabla u_\varep|^p\right)^{1/p},
$$
for any $x, y\in B$, where $p>d$.
\end{proof}

\begin{remark}\label{remark-2.3.1}
{\rm It follows from (\ref{estimate-2.3.1}) that if $\mathcal{L}_\varep (u_\varep)=\text{\rm div} (f)$ in $2B$, then
\begin{equation}\label{interior-L-infty-1} 
\| u_\varep \|_{L^\infty(B)}
\le C _p 
\bigg\{ \left(\average_{2B} |u_\varep|^2\right)^{1/2} 
+r \left(\average_{2B} |f|^{p}\right)^{1/p} \bigg\}
\end{equation}
for $p>d$.
}
\end{remark}

\begin{remark}\label{remark-2.3.2}
{\rm
The VMO assumption (\ref{VMO-1}) on the coefficient matrix $A$ is used only in Lemma \ref{local-lemma-2.4}
to establish local $W^{1,p}$ estimates for solutions of $\mathcal{L}_1(u)=0$.
Consequently, the smoothness condition in Theorems \ref{interior-W-1-p-theorem} and \ref{interior-Holder-theorem}
may be weakened if one is able to weaken the condition in Lemma \ref{local-lemma-2.4}.
We refer the reader to \cite{Krylov-2007, D-Kim-2010} and their references for local $W^{1, p}$ estimates for
elliptic and parabolic operators with partially VMO coefficients.
}
\end{remark}



\section[Fundamental Solutions]{Asymptotic expansions of fundamental solutions}\label{section-2.5}

Let $d\ge 3$.
A fundamental solution in $\br^d$ can be constructed for any scalar elliptic operator
in divergence form with real, bounded measurable coefficients
(see e.g. \cite{Gruter-1982}). In fact,
the construction in \cite{Gruter-1982} can be extended to any system of second-order elliptic operators in 
divergence form with complex, bounded measurable coefficients,
provided that solutions of the system and its adjoint satisfy the De Giorgi -Nash type local
H\"older continuity estimates; i.e., there exist $\sigma\in (0,1)$ and $ H_0>0$ such that
\begin{equation}\label{Holder-condition}
\left(\average_{B(x,t)} |\nabla u|^2\right)^{1/2}
\le H_0 \left(\frac{t}{r}\right)^{\sigma-1} \left(\average_{B(x,r)} |\nabla u|^2\right)^{1/2}, \quad \text{ for } 0<t<r<\infty,
\end{equation}
whenever  $\mathcal{L}(u)=0$ or $\mathcal{L}^* (u)=0$ in $B(x,r)$
 (see \cite{Hofmann-2007}).
As a result, in view of Remark \ref{remark-H},
if $A$ is 1-periodic and satisfies (\ref{weak-e-1})-(\ref{weak-e-2}) and (\ref{VMO-1}),
one may construct an $m\times m$ matrix $\Gamma_\varep (x,y)=\big(\Gamma^{\alpha\beta}_\varep (x,y)\big)$
such that for each $y\in \br^d$, $\nabla_x \Gamma_\varep (x,y)$ is locally integrable and
\begin{equation}\label{fundamental-representation}
\phi^\gamma (y)
=\int_{\br^d} a_{ij}^{\alpha\beta} (x/\varep)
\frac{\partial}{\partial x_j} \Big\{ \Gamma_\varep^{\beta\gamma} (x,y)\Big\} 
\frac{\partial \phi^\alpha}{\partial x_i}\, dx
\end{equation}
for $\phi=(\phi^\alpha)\in C_0^1(\br^d;\br^m)$.
Moreover, the matrix $\Gamma_\varep (x,y)$ satisfies the size estimate
\begin{equation}\label{fundamental-solution-size-estimate}
 |\Gamma_\varep (x,y)|
\le C\, |x-y|^{2-d} 
\end{equation}
for any $x, y\in\br^d$ and $x\neq y$, where the constant $C$ 
depends only on $\mu$ and the function $\omega(t)$.
Such matrix, which is unique, is called the matrix of fundamental solutions for $\mathcal{L}_\varep$.

If $m$=1, the periodicity condition  and VMO condition are not needed; constant $C$
in (\ref{fundamental-solution-size-estimate}) depends only on $\mu$.
This is also the case for $d=2$, where the estimate (\ref{fundamental-solution-size-estimate})
is replaced by $\|\Gamma_\e (\cdot, y)\|_{\text{BMO}(\br^2)} \le C$ and
\begin{equation}\label{2d-f}
\Big|\Gamma_\e (x, y)-\average_{B(x, 1)} \Gamma_\e (x, z)\, dz\Big|
\le C \Big\{ 1+ \big| \log |x-y|\big| \Big\}
\end{equation}
for any $x, y\in \br^2$ and $x\neq y$. 
It is also known that
\begin{equation}\label{2d-f-100}
|\Gamma_\e (x, y) -\Gamma_\e (x, z)|
\le \frac{C |y-z|^\sigma}{|x-y|^\sigma}
\end{equation}
for $x, y\in \br^2$ and $|y-z|<(1/2)|x-y|$,
where $\sigma$ depends only on $\mu$.
See \cite{Brown-2013}.

It can be shown that the matrix of fundamental solutions $\Gamma^*(x, y)=\big(\Gamma^{*\alpha\beta}(x, y)\big)$
for $\mathcal{L}_\varep^*$
is given by $\big( \Gamma_\varep (y, x))^T$, the matrix transpose of $ \Gamma_\varep (y,x)$; i.e.,
\begin{equation}\label{adjoint-fs}
\Gamma_{\varep}^{*\alpha\beta}( x, y)=\Gamma_{\varep}^{\beta\alpha} (y,x).
\end{equation}
By uniqueness and  the rescaling property of $\mathcal{L}_\varep$, one may deduce that
\begin{equation}\label{fundamental-solution-scaling}
\Gamma_\varep (x,y) =\varep^{2-d}\Gamma_1 (\varep^{-1} x, \varep^{-1} y).
\end{equation}

\begin{thm}\label{theorem-2.5-1}
Suppose that $A$ satisfies (\ref{weak-e-1})-(\ref{weak-e-2}) and
is 1-periodic. Also assume that $A$ satisfies the H\"older continuity  condition (\ref{smoothness}). Then,
for any $x,y\in\br^d$ and $x\neq y$,
\begin{equation}\label{fundamental-derivative-estimate}
|\nabla_x \Gamma_\varep (x,y)|
+|\nabla_y \Gamma_\varep (x,y)|
\le C\,  |x-y|^{1-d},
\end{equation}
and
\begin{equation}\label{fundamental-two-derivative-estimate}
|\nabla_y\nabla_x \Gamma_\varep (x,y)|
\le C\, |x-y|^{-d},
\end{equation}
where $C$ depends only on $\mu$, $\lambda$, and $\tau$.
\end{thm}

\begin{proof}
We first consider the case $d\ge 3$.
Fix $x_0, y_0\in \rd$ and $r=|x_0-y_0|>0$.
Let $u_\varep (x) =\Gamma^\beta_\varep (x, y_0)$, 
where
$\Gamma_\varep^\beta (x,y_0)
=(\Gamma_\varep^{1\beta} (x,y_0), \dots, \Gamma_\varep^{m\beta}(x,y_0))$.
Note that  $\mathcal{L}_\varep (u_\varep)=0$ in $B(x_0, r/2)$ and by (\ref{fundamental-solution-size-estimate}),
$$
\| u_\varep\|_{L^\infty(B(x_0, r/2))} \le Cr^{2-d}.
$$
It follows from the Lipschitz estimates in Theorem \ref{interior-Lip-theorem} that
$$
\|\nabla u_\varep\|_{L^\infty(B(x_0, r/4))} \le C r^{1-d}.
$$
This gives $|\nabla_x \Gamma_\varep (x_0, y_0)|\le C\, |x_0 -y_0|^{1-d}$.
In view of (\ref{adjoint-fs}), the same argument also yields the estimate $|\nabla_y \Gamma_\varep (x, y)|\le C\, |x-y|^{1-d}$.
To see (\ref{fundamental-two-derivative-estimate}), we apply the Lipschitz estimate to 
$$
v_\varep (x) = \Gamma_\varep(x,y_0)-\Gamma_\e (x, y_1),
$$ 
where $y_1\in B(y_0, r/4)$, and use the estimate (\ref{fundamental-derivative-estimate}). It follows that
$$
\aligned
|\nabla_x \Gamma_\e (x_0, y_0)-\nabla_x \Gamma_\e (x_0, y_1)|
&\le \frac{C}{r} \max_{x\in B(x_0, r/4)} 
|\Gamma_\e (x, y_0) -\Gamma_\e (x, y_1)|\\
&\le \frac{C|y_0-y_1|}{r^d},
\endaligned
$$
which yields (\ref{fundamental-two-derivative-estimate}).

Finally, we note that if $d=2$,  (\ref{fundamental-derivative-estimate}) is a consequence of the
interior Lipschitz estimate and the fact that $\|G_\e (\cdot, y)\|_{\text{BMO}(\br^2)} \le C$.
The estimate (\ref{fundamental-two-derivative-estimate}) follows from (\ref{fundamental-derivative-estimate})
by the same argument as in the case $d\ge 3$.
\end{proof}


In the remaining of this section we  study the asymptotic behavior, as $\varep\to 0$, of
$\Gamma_\varep(x,y)$, $\nabla_x \Gamma_\varep(x,y)$, $\nabla_y \Gamma_\varep(x,y)$,
and $\nabla_x\nabla_y \Gamma_\varep (x,y)$.
Let $\Gamma_0 (x,y)$ denote the matrix of fundamental solutions for the homogenized operator $\mathcal{L}_0$.
Since $\mathcal{L}_0$ is a second-order elliptic operator with constant coefficients,
we have $\Gamma_0 (x,y)=\widetilde{\Gamma}_0 (x-y)$ and
$$
|\nabla^k \widetilde{\Gamma}_0 (x)|\le C_k |x|^{2-d-k} \text{ for any integer $k\ge 1$}.
$$

\begin{thm}\label{fundamental-solution-theorem-1}
Suppose that $A$ satisfies the ellipticity condition (\ref{weak-e-1})-(\ref{weak-e-2})
and is 1-periodic.
If $m\ge 2$,  we also assume $A\in \text{\rm VMO}(\rd)$.
Then, if $d\ge 3$,
\begin{equation}\label{fundamental-solution-asym}
|\Gamma_\varep (x,y)
-\Gamma_0 (x,y)|\le C\, \varep\, |x-y|^{1-d}
\end{equation}
for any $x,y\in \br^d$ and $x\neq y$.
\end{thm}

The proof of Theorem \ref{fundamental-solution-theorem-1} is based on
a local $L^\infty$ estimate for $u_\varep -u_0$.

\begin{lemma}\label{lemma-2.4.1}
Assume that $A$ satisfies the same conditions as in Theorem \ref{fundamental-solution-theorem-1}.
Let $u_\varep \in H^1(2B;\br^m)$ and $u_0\in C^2(2B;\br^m)$ for some ball $B=B(x_0,r)$.
Suppose that $\mathcal{L}_\varep (u_\varep)=\mathcal{L}_0 (u_0)$ in $2B$.
Then
\begin{equation}\label{estimate-2.4.1}
\| u_\varep -u_0\|_{L^\infty(B)}
\le C \left\{ \average_{2B} |u_\varep -u_0|^2\right\}^{1/2}
+C\, \varep\, \| \nabla u_0\|_{L^\infty(2B)}
+C \,\varep\, r\, \|\nabla^2 u_0\|_{L^\infty(2B)},
\end{equation}
where $C$ depends only on $\mu$ and $\omega(t)$ in  (\ref{VMO-1}) (if $m\ge 2$).
\end{lemma}

\begin{proof}
By translation and dilation we may assume that $x_0=0$ and $r=1$.
Consider
\begin{equation}\label{definition-of-w-10}
w_\varep =u_\varep -u_0-\varep \chi_j^\beta (x/\varep) \frac{\partial u_0^\beta}{\partial x_j}
\end{equation}
 in $2B$.
Using $\mathcal{L}_\e (u_\e)=\mathcal{L}_0 (u_0)$,
we see that
\begin{equation}\label{new-comp-0}
\aligned
\mathcal{L}_\e (u_\e)
&=\text{\rm div} \big( \big(A(x/\e)-\widehat{A}\big)\nabla u_0\big)
+\text{\rm div} \Big( A(x/\e)\nabla \big(\e \chi_j^\beta (x/\e) \frac{\partial u_0^\beta}{\partial x_j}\big) \Big)\\
&=\text{\rm div} \big( B(x/\e)\nabla u_0\big)
+\e \, \text{\rm div} \Big( A(x/\e)\chi_j^\beta (x/\e) \frac{\partial }{\partial x_j} \nabla u^\beta_0\Big),
\endaligned
\end{equation}
where $B(y)=A(y) + A (y)\nabla \chi (y) -\widehat{A}$.
By (\ref{phi-identity-0}) we obtain 
\begin{equation}\label{right-hand-side}
\big( \mathcal{L}_\varep (w_\varep)\big)^\alpha
=-\varep\,
\frac{\partial}{\partial x_i}
\left\{
\left[
\phi_{jik}^{\alpha\gamma} (x/\varep) - a_{ij}^{\alpha\beta} (x/\varep) \chi_k^{\beta\gamma} (x/\varep)\right]
\frac{\partial^2 u_0^\gamma}{\partial x_j\partial x_k} \right\}.
\end{equation}
In view of the $L^\infty$ estimate (\ref{interior-L-infty-1}) this implies that
\begin{equation}\label{2.4.1-1}
\|w_\varep\|_{L^\infty(B)}
\le C \left\{ \left(\average_{2B} |w_\varep|^2 \right)^{1/2}
+ \varep\, \|\nabla^2 u_0\|_{L^\infty(2B)}\right\},
\end{equation}
from which the estimate (\ref{estimate-2.4.1}) follows easily.
We point out that under the assumptions on $A$ in Theorem \ref{fundamental-solution-theorem-1},
the correctors $\chi_j^{\alpha\beta}$ as well as the flux correctors $\phi_{jik}^{\alpha\beta}$
are bounded.
\end{proof}

\begin{proof}[\bf Proof of Theorem \ref{fundamental-solution-theorem-1}]
Fix $x_0, y_0\in \br^d$ and let $r=|x_0-y_0|/4$.
It suffices to consider the case $0<\varep<r$, since the estimate for the case $r \ge \varep$
is trivial and follows directly from the size estimate (\ref{fundamental-solution-size-estimate}).
Furthermore, by (\ref{fundamental-solution-scaling}), we may assume that $r=1$.

Let $f\in C_0^1(B(y_0, 1);\br^m)$,
$$
u_\varep (x) =\int_{\br^d} \Gamma_\varep (x,y) f(y)\, dy \quad \text{ and } \quad
u_0 (x)  =\int_{\br^d} \Gamma_0  (x,y) f(y)\, dy.
$$
By the Calder\'on-Zygmund estimates for singular integrals we have 
$$
\|\nabla^2 u_0\|_{L^p(\br^d)}
\le C_p\,  \| f\|_{L^p(\br^d)} \quad \text{ for any $1<p<\infty$}
$$
(see e.g. \cite[Chapter II] {Stein-1970}).
Also, the fractional integral estimates  give 
$$
\|\nabla u_0\|_{L^q(\br^d)}
\le C\, \| f\|_{L^p(\br^d)}
$$
 for 
$$
1<p<q<\infty \quad \text{ and } \quad \frac{1}{q}=\frac{1}{p}-\frac{1}{d}
$$
(see e.g. \cite[Chapter V]{Stein-1970}).
Let $w_\varep$ be defined by (\ref{definition-of-w-10}).
In view of (\ref{right-hand-side}) and the fact that $w_\varep (x)=O(|x|^{2-d})$
as $|x|\to \infty$, we obtain $|\nabla w_\varep| \in L^2(\br^d)$ and
$$
\|\nabla w_\varep\|_{L^2(\br^d)}
\le C \, \varep\, \|\nabla^2 u_0\|_{L^2(\br^d)}
\le C \,\varep \,\| f\|_{L^2(\br^d)}.
$$
By H\"older's inequality and Sobolev imbedding this implies that
$$
\|w_\varep\|_{L^2(B(x_0,1))}
\le C \,\| w_\varep \|_{L^{2^*}(\br^d)}
\le C \,\varep\, \| f\|_{L^2(\br^d)},
$$
where $2^*=\frac{2d}{d-2}$. Hence,
\begin{equation}\label{2.4.4-4}
\aligned
\|u_\varep -u_0 \|_{L^2(B(x_0,1))}
&\le C\,\varep\, \| f\|_{L^2(\br^d)}
+C\,\varep\, \|\nabla u_0\|_{L^2(B(x_0,1))}\\
& \le C\, \varep\, \| f\|_{L^2(B(y_0,1))}.
\endaligned
\end{equation}
Note that $\mathcal{L}_\varep (u_\varep)=\mathcal{L}_0 (u_0) =0$ in $B(x_0,3)$.
We may apply Lemma \ref{lemma-2.4.1} to obtain
\begin{equation}\label{2.4.4-6}
|u_\varep (x_0)-u_0(x_0)|
\le C\, \varep\, \| f\|_{L^2(B(y_0,1))},
\end{equation}
where we have used (\ref{2.4.4-4}) and the observation that
$$
\|\nabla u_0\|_{L^\infty(B(x_0,1))}
\le C\, \|\nabla u_0\|_{L^2(B(x_0,2))} \le C\, \| f\|_{L^2(B(y_0,1))}.
$$
By duality the estimate (\ref{2.4.4-6}) yields that
$$
\| \Gamma_\varep(x_0, \cdot) -\Gamma_0 (x_0, \cdot)\|_{L^2(B(y_0,1))}
\le C\, \varep.
$$

Finally, since $\mathcal{L}^*_\varep \big( \Gamma_\varep (x_0, \cdot)\big)
=\mathcal{L}^*_0 \big( \Gamma_0 (x_0, \cdot)\big)$ in $B(y_0,3)$,
we may invoke Lemma \ref{lemma-2.4.1} again to conclude that
$$
\aligned
|\Gamma_\varep (x_0, y_0)-\Gamma_0 (x_0, y_0)|
&\le C \|\Gamma_\varep (x_0, \cdot)-\Gamma_0 (x_0, \cdot)\|_{L^2(B(y_0,1))}\\
& \qquad \qquad +C\varep \|\nabla_y \Gamma_0 (x_0, \cdot)\|_{L^\infty(B(y_0,1))}\\
&\qquad \qquad + C\varep \|\nabla^2_y \Gamma_0 (x_0, \cdot)\|_{L^\infty(B(y_0,1))}\\
&\le C\,\varep.
\endaligned
$$
This completes the proof.
\end{proof}

\begin{remark}\label{2d-f-re}
{\rm
Let $d=2$. Suppose that $A$ is 1-periodic and satisfies (\ref{weak-e-1})-(\ref{weak-e-2}). Then
\begin{equation}\label{2d-f-1}
\| \Gamma_\e (x, \cdot) -\Gamma_0 (x, \cdot)- E\|_{L^\infty(B(y, r))}
\le \frac{C\e}{|x-y|},
\end{equation}
for any $x, y\in \br^2$ and $x\neq y$, where $r= (1/4)|x-y|$ and $E$ denotes the $L^1$ average of
$\Gamma_\e (x, \cdot)-\Gamma_0 (x, \cdot)$ over $B(x, r)$.
This follows from the proof of Theorem \ref{fundamental-solution-theorem-1}, with a few
modifications.

As in the case $d\ge 3$, we fix $x_0, y_0\in \br^2$ and
let $r= (1/4) |x_0-y_0|$.
We may assume that $r=1$ and $0<\e<1$.
Let $f\in C_0^1 (B(y_0, 1); \br^m)$ with $\int_{B(y_0, 1)} f\, dx =0$.
Define $u_\e$ and $u_0$ as in the proof of Theorem \ref{fundamental-solution-theorem-1}.
Using (\ref{2d-f-100}) and the assumption that $\int_{B(y_0, 1)} f\, dx =0$., we may show that
$$
u_\e (y)=O(|y|^{-\sigma}) \quad \text{ as } |y|\to \infty.
$$
By Caccioppoli's inequality this implies that $|\nabla w_\e| \in L^2(\br^2)$. Moreover, by Remark \ref{pert-r},
there exists some $q<2$ such that
$$
\|\nabla w_\e\|_{L^q(\br^2)}
\le C \e \|\nabla^2 u_0\|_{L^q(\br^2)} 
\le C \e \| f\|_{L^q(\br^2)}.
$$
It follows by H\"older's inequality and Sobolev inequality that
$$
\aligned
\|w_\e\|_{L^2(B(x_0, 1))}
 &\le C \|w_\e\|_{L^p(\br^2)}
\le C \|\nabla w_\e\|_{L^q(\br^2)}\\
&\le C \e \|f\|_{L^2(B(y_0, 1))},
\endaligned
$$
where $\frac{1}{p}=\frac{1}{q}-\frac12$.
As in the proof of Theorem \ref{fundamental-solution-theorem-1}, this leads to
$$
|u_\e (x_0)-u_0 (x_0)|\le C \e \| f\|_{L^2(B(y_0, 1))}.
$$
By duality we obtain 
$$
\|\Gamma_\e (x_0, \cdot) -\Gamma_0 (x_0, \cdot) -E\|_{L^2(B(y_0, 1))}
\le C \e,
$$
where $E$ denotes the $L^1$ average of $\Gamma_\e (x_0, y)-\Gamma_0 (x_0, y)$
over $B(y_0, 1)$.
The desired estimate now follows by Lemma \ref{lemma-2.4.1}.
}
\end{remark}

The next theorem gives an asymptotic expansion for $\nabla_x \Gamma_\varep (x,y)$.

\begin{thm}\label{fundamental-solution-theorem-2}
Suppose that $A$ is 1-periodic and satisfies (\ref{weak-e-1})-(\ref{weak-e-2}).
Also assume that $A$ satisfies  the H\"older continuity condition (\ref{smoothness}).
Then
$$
\big|\nabla_x \Gamma_\varep (x,y) -\nabla_x \Gamma_0(x,y)
-\nabla\chi (x/\varep) \nabla_x \Gamma_0 (x,y)\big|
\le  \frac{C \varep \ln [\varep^{-1}|x-y| +2]}{|x-y|^d}.
$$
More precisely,
\begin{equation}\label{fundamental-solution-asym-2}
\aligned
\big|
\frac{\partial}{\partial x_i}
\Big\{ \Gamma^{\alpha\beta}_\varep (x,y)\Big\}
-\frac{\partial}{\partial x_i}
\Big\{ \Gamma^{\alpha\beta}_0 (x,y)\Big\}
 & -\frac{\partial}{\partial x_i} \Big\{ \chi_j^{\alpha\gamma} \Big\}  (x/\varep)
\cdot \frac{\partial}{\partial x_j} \Big\{ \Gamma^{\gamma\beta}_0(x,y)\Big\}\big|\\
& \le \frac{C \varep \ln [\varep^{-1}|x-y| +2]}{|x-y|^d}
\endaligned
\end{equation}
for any $x,y\in \br^d$ and $x\neq y$,
where $C$ depends only on $\mu$, $\lambda$, and $\tau$.
\end{thm}

The proof of Theorem \ref{fundamental-solution-theorem-2} relies on the following 
Lipschitz estimate.

\begin{lemma}\label{interior-Lip-lemma}
Assume that $A$ satisfies the same conditions as in Theorem \ref{fundamental-solution-theorem-2}.
Suppose that $u_\varep \in H^1(4B;\br^m)$, $u_0\in C^{2,\eta}(4B;\br^m)$,
and $\mathcal{L}_\varep (u_\varep)=\mathcal{L}_0 (u_0)$ in $4B$ for some ball $B=B(x_0,r)$.
Then, if $0<\varep<r$,
\begin{equation}\label{estimate-2.4.5}
\aligned
\big\| \frac{\partial u_\varep^\alpha}{\partial x_i}
& -\frac{\partial u_0^\alpha}{\partial x_i}
-\frac{\partial }{\partial x_i} \Big\{ \chi_j^{\alpha\beta}\Big\}
(x/\varep)\cdot
\frac{\partial u_0^\beta}{\partial x_j} \big\|_{L^\infty(B)}\\
& \le C\, r^{-1} \left\{ \average_{4B}| u_\varep -u_0|^2\right\}^{1/2}
+C\,\varep r^{-1} \|\nabla u_0\|_{L^\infty(4B)}\\
& \qquad + C\varep \ln [\varep^{-1} r +2]
\|\nabla^2 u_0\|_{L^\infty(4B)}
+C\,\varep^{1+\lambda} \|\nabla^2 u_0\|_{C^{0, \lambda} (4B)},
\endaligned
\end{equation}
where $C$ depends only on $\mu$, $\lambda$, and $\tau$.
\end{lemma}

\begin{proof}
Let $ w_\varep =u_\varep -u_0 - \varep \chi_j^\beta (x/\varep)\frac{\partial u_0^\beta}{\partial x_j}$.
It suffices to show that $\|\nabla w_\varep\|_{L^\infty(B)}$
is bounded by the RHS of (\ref{estimate-2.4.5}).
Choose $\varphi\in C_0^\infty(3B)$ such that $0\le \varphi\le 1$, 
$\varphi=1$ in $2B$, and $|\nabla\varphi|\le C\,r^{-1}$, $|\nabla^2 \varphi|\le C\,r^{-2}$.
Note that
$$
\big(\mathcal{L}_\varep (w_\varep\varphi)\big)^\alpha
=
\big(\mathcal{L}_\varep (w_\varep)\big)^\alpha \varphi
-a_{ij}^{\alpha\beta}(x/\varep) \frac{\partial w_\varep^\beta}{\partial x_j} \frac{\partial\varphi}{\partial x_i}
-\frac{\partial}{\partial x_i}
\left\{ a_{ij}^{\alpha\beta} (x/\varep) w_\varep^\beta \frac{\partial\varphi}{\partial x_j} \right\}.
$$
This, together with (\ref{right-hand-side}), implies that
\begin{equation}\label{2.4.5-1}
\aligned
 w_\varep^\alpha (x)
= & -\varep \int_{3B} \frac{\partial}{\partial y_i} \Big\{\Gamma_\varep^{\alpha\beta} (x,y)\Big\} 
f_i^\beta (y)\varphi (y) \, dy\\
&\qquad
-\varep \int_{3B}  \Gamma_\varep^{\alpha\beta} (x,y)
f_i^\beta (y) \frac{\partial\varphi}{\partial y_i}\, dy\\ 
&\qquad -\int_{3B} \Gamma_\varep^{\alpha\beta} (x,y) a_{ij}^{\beta\gamma}(y/\varep)
\frac{\partial w_\varep^\gamma}{\partial y_j} \frac{\partial\varphi}{\partial y_i}\, dy\\
&\qquad  +\int_{3B} \frac{\partial}{\partial y_i} \Big\{  \Gamma_\varep^{\alpha\beta} (x,y) \Big\}
a_{ij}^{\beta\gamma}(y/\varep) w_\varep^\gamma \frac{\partial \varphi}{\partial y_j}\, dy\\
&=I_1 +I_2+I_3 + I_4
\endaligned
\end{equation}
for any $x\in B$, where $f=(f_i^\alpha)$ and
\begin{equation}\label{2.4.5-3}
f_i^\alpha (x)
=\left[ -\phi_{jik}^{\alpha\gamma} (x/\varep) +a_{ij}^{\alpha\beta}(x/\varep) \chi_k^{\beta\gamma}
(x/\varep) \right] \frac{\partial^2 u_0^\gamma}{\partial x_j \partial x_k}.
\end{equation}
Using (\ref{fundamental-derivative-estimate}) and (\ref{fundamental-two-derivative-estimate}), 
we see that for $x\in B$,
\begin{equation}\label{2.4.5-5}
\aligned
 |\nabla_x  & \{ I_2 +I_3+I_4\}|\\
&\le
\frac{C\varep}{r^d} \int_{3B} |f|
+\frac{C}{r^d}\int_{3B} |\nabla w_\varep|
+\frac{C}{r^{d+1}}\int_{3B} |w_\varep|\\
&\le 
\frac{C}{ r} \left\{ \average_{4B} |w_\varep|^2\right\}^{1/2}
+C\varep \left\{ \average_{4B} |\nabla^2 u_0|^2\right\}^{1/2}\\
& \le
\frac{C}{ r} \left\{ \average_{4B} |u_\varep -u_0|^2\right\}^{1/2}
+ \frac{C \varep} {r} \|\nabla u_0\|_{L^\infty(4B)}
+C\, \varep \,\|\nabla^2 u_0 \|_{L^\infty(4B)},
\endaligned
\end{equation}
where we have used  Caccioppoli's inequality for the second inequality.

Finally, to estimate $\nabla I_1$, we write
$$
I_1 =-\varep\int_{3B}
\frac{\partial}{\partial y_i} \Big\{ \Gamma_\varep^{\alpha\beta}(x,y)\varphi (y)\Big\}
\left\{ f_i^\beta (y)-f_i^\beta (x)\right\}\, dy
+\varep \int_{3B}
\Gamma_\varep^{\alpha\beta} (x,y) \frac{\partial\varphi}{\partial y_i} f_i^\beta (y)\, dy.
$$
It follows that
\begin{equation}\label{2.4.5-7}
\aligned
|\nabla_x I_1|
 &\le C\,\varep \int_{3B} \frac{|f(y)-f(x)|}{|x-y|^d}\, dy
+\frac{C\varep}{r^d} \int_{3B} |f(y)|\, dy\\
&\le C\,\varep \int_{B(x,\varep)}
\frac{|f(y)-f(x)|}{|x-y|^d}\, dy
+C\,\varep\, \ln[\varep^{-1} r +2] \| f\|_{L^\infty(3B)}.
\endaligned
\end{equation}
Observe that
$$
\| f\|_{C^{0, \lambda} (4B)}
\le C\, \varep^{-\lambda} \|\nabla^2 u_0\|_{L^\infty(4B)}
+C\, \| \nabla^2 u_0\|_{C^{0,\lambda}(4B)}.
$$
In view of (\ref{2.4.5-7}) we obtain
$$
|\nabla_x I_1|
\le C\, \varep \|\nabla^2 u_0\|_{L^\infty(4B)} 
+ C\, \varep^{1+\lambda} \| \nabla^2 u_0\|_{C^{0, \lambda}(4B)}
+C\, \varep\, \ln[\varep^{-1} r +2] \| \nabla^2 u_0\|_{L^\infty(4B)}.
$$
This, together with (\ref{2.4.5-5}), gives the estimate (\ref{estimate-2.4.5}).
\end{proof}

\begin{proof}[\bf Proof of Theorem \ref{fundamental-solution-theorem-2}]
Fix $x_0, y_0\in \br^d$ and let $r=|x_0-y_0|/8$.
We may assume that $\varep< r$, since the estimate for the case $\varep\ge r$ is trivial
and follows directly from (\ref{fundamental-derivative-estimate}).

Let $u_\varep (x)=\Gamma_\varep (x,y_0)$ and $u_0 =\Gamma_0(x,y_0)$.
Then $\mathcal{L}_\varep (u_\varep)=\mathcal{L}_0 (u_0)=0$
in $B(x_0,4r)$.
Note that by Theorem \ref{fundamental-solution-theorem-1}, we have
$$
\| u_\varep -u_0\|_{L^\infty(4B)} \le C\, \varep\, r^{1-d}
$$
for $d\ge 3$. If $d=2$, we may deduce from Remark \ref{2d-f-re} that
$$
\| u_\varep -u_0-E\|_{L^\infty(4B)} \le C\, \varep\, r^{-1},
$$
where $E$ denotes the $L^1$ average of $\Gamma_\e (\cdot, y_0)-\Gamma_0 (\cdot, y_0)$
over $B(x_0, r)$.
Also, since $\mathcal{L}_0$ is a second-order elliptic operator with constant coefficients,
$$
\|\nabla u_0\|_{L^\infty(4B)} \le C\, r^{1-d}, \quad
\|\nabla^2 u_0\|_{L^\infty (4B)} \le C\, r^{-d} \quad
\text{ and } \quad
\|\nabla^2 u_0\|_{C^{0, \lambda} (4B)} \le C\, r^{-d-\lambda}.
$$
It then follows from Lemma \ref{interior-Lip-lemma} that
$$
\big\|\frac{\partial u_\varep}{\partial x_i} 
-\frac{\partial u_0}{\partial x_i} 
-\frac{\partial \chi^\beta_j}{\partial x_i}
(x/\varep)
\frac{\partial u^\beta_0}{\partial x_j}
\big\|_{L^\infty(B)}
\le C\, \varep r^{-d}\, \ln [\varep^{-1} r +2].
$$
This finishes the proof.
\end{proof}

Using (\ref{adjoint-fs}),
we may deduce from Theorem \ref{fundamental-solution-theorem-2} that
\begin{equation}\label{fundamental-solution-asym-3}
\aligned
\big|
\frac{\partial}{\partial y_i}
\Big\{ \Gamma^{\alpha\beta}_\varep (x,y)\Big\}
-\frac{\partial}{\partial y_i}
\Big\{ \Gamma^{\alpha\beta}_0 (x,y)\Big\}
 & -\frac{\partial}{\partial y_i} \Big\{ \chi_j^{*\beta\gamma} \Big\}  (y/\varep)
\cdot \frac{\partial}{\partial y_j} \Big\{ \Gamma^{\alpha\gamma}_0(x,y)\big\}\Big|\\
& \le \frac{C \varep \ln [\varep^{-1}|x-y| +2]}{|x-y|^d},
\endaligned
\end{equation}
for any $x,y\in \br^d$ and $x\neq y$, where $\chi^*=(\chi_j^{*\beta\gamma})$ denotes the 
matrix of correctors for $\mathcal{L}_\e^*$.

The following theorem gives an asymptotic expansion of $\nabla_x\nabla_y \Gamma_\varep (x,y)$.

\begin{thm}\label{fundamental-solution-theorem-4}
Suppose that $A$ satisfies the same conditions
as in Theorem \ref{fundamental-solution-theorem-2}.
Then
\begin{equation}\label{fundamental-solution-asym-4}
\aligned
& \big|
\frac{\partial^2}{\partial x_i\partial y_j} \Big\{ \Gamma_\varep^{\alpha\beta} (x,y)\Big\}\\
& \qquad -
\frac{\partial}{\partial x_i} \Big\{ \delta^{\alpha\gamma} x_k +\varep \chi_k^{\alpha\gamma} (x/\varep) \Big\}
\frac{\partial^2}{\partial x_k \partial y_\ell}
\Big\{ \Gamma_0^{\gamma\sigma} (x, y)\Big\}
\frac{\partial}{\partial y_j}
\left\{ \delta^{\beta\sigma} y_\ell
+\varep \chi_\ell^{*\beta\sigma} (y/\varep)\right\}
 \big|\\
&\qquad\qquad\qquad
\le 
\frac{C\varep \ln [\varep^{-1} |x-y| +2]}{|x-y|^{d+1}}
\endaligned
\end{equation}
for any $x,y\in \br^d$ and $x\neq y$, where
$C$ depends only on $\mu$, $\lambda$, and $\tau$.
\end{thm}

\begin{proof}
Fix $x_0, y_0\in \br^d$ and let $r=|x_0-y_0|/8$.
Again, it suffices to consider the case $0<\varep<r$.
For each $1\le j\le d$ and $1\le \beta\le m$, let
$$
\left\{
\aligned
& u_\varep^\alpha (x) =\frac{\partial}{\partial y_j} \Big\{ \Gamma_\varep^{\alpha\beta}\Big\}  (x,y_0),\\
& u_0^\alpha (x) =\frac{\partial}{\partial y_\ell} \Big\{\Gamma_0^{\alpha\sigma}\Big\}
 (x, y_0)
\cdot \left\{ \delta^{\beta\sigma}\delta_{j\ell}
+\frac{\partial}{\partial y_j} \big(\chi_\ell^{*\beta\sigma}\big) (y_0/\varep) \right\}
\endaligned
\right.
$$
in $4B=B(x_0, 4r)$. In view of (\ref{fundamental-solution-asym-3}) we have
$$
\| u_\varep -u_0\|_{L^\infty(B(x_0,r))}
\le C\, \varep \, r^{-d} \ln [\varep^{-1} +2].
$$
Also note that $\|\nabla u_0\|_{L^\infty(4B)}\le C\, r^{-d}$,
$$
\|\nabla^2 u_0\|_{L^\infty(4B)} \le C\, r^{-d-1}
\quad \text{ and } \quad 
\|\nabla^2 u_0\|_{C^{0, \lambda}(4B)} \le C\, r^{-d-1-\lambda}.
$$
By Lemma \ref{interior-Lip-lemma} we obtain
$$
\big\|\frac{\partial u_\varep^\alpha}{\partial x_i}
-\left\{ \delta^{\alpha\gamma}\delta_{ik}
+
\frac{\partial }{\partial x_i} \Big\{ \chi_k^{\alpha\gamma}\Big\} (x/\varep) \right\}
\frac{\partial u_0^\gamma}{\partial x_k}\big\|_{L^\infty(B)}
\le \frac{C \varep \ln [\varep^{-1} r +2]}{r^{d+1}},
$$
which  gives the desired estimate.
\end{proof}



\section{Notes}

The compactness method used Section \ref{section-2.2} was introduced to the study of homogenization
problems by M. Avellaneda and F. Lin \cite{AL-1987, AL-1989}.
The uniform interior H\"older and Lipschitz estimates as well as boundary estimates
with Dirichlet condition were proved in \cite{AL-1987}.
In \cite{AL-1989-Lo} a general Liouville property for
$\mathcal{L}_1$ was proved under some additional smoothness condition on $A$.

Theorems \ref{real-variable-theorem} and \ref{real-variable-operator-theorem} are taken from
 \cite{Shen-2007-L, Shen-2005} by Z. Shen.
The real-variable argument in Section \ref{real-variable-section} is motivated by a paper \cite{CP-1998} 
of L. Caffarelli and I. Peral. 
Also see related work for $L^p$ estimates
 in \cite{Wang-2003, Byun-Wang-2004, Auscher-2004, BW-2005, Byun-2005, Auscher-2007,
Geng-2012}).

The matrices of fundamental solutions were studied in \cite{AL-1991} and used to prove
the $W^{1,p}$ estimates in Theorem \ref{W-1-p-theorem-2} as well as a weak $(1,1)$ estimate.  
Also see \cite{ ERS-2001} for $L^p$ bounds of Riesz transforms for second-order elliptic operators with periodic coefficients.
Earlier work on asymptotic behavior of fundamental solutions for scalar equations with periodic coefficients 
may be found in \cite{Sevostjanova-1982,Kozlov-1980}.
Asymptotic results in Section \ref{section-2.5} are stronger than those obtained  in \cite{AL-1991}. 
The approach was developed in \cite{KLS-2014}, where the asymptotic behaviors of
Green and Neumann functions in bounded domains were studied (see Chapter \ref{chapter-6}).

Related work on uniform interior estimates in periodic homogenization may be found in 
\cite{Gu-Shen-2015, Geng-S-2015, Chase-R-2017, Niu-Shen-Xu}

%
%
%
%
%

\chapter[Regularity for Dirichlet Problem]{Regularity for Dirichlet Problem}\label{chapter-3}

This chapter is devoted to the study of uniform boundary regularity estimates
for the Dirichlet problem
\begin{equation}\label{boundary-Dirichlet-problem}
\left\{
\aligned
\mathcal{L}_\varep (u_\varep) &=F \quad \, \text{ in } \Omega,\\
u_\varep & = g \quad \ \text{ on }\partial\Omega,
\endaligned
\right.
\end{equation}
where  $\mathcal{L}_\varep=-\text{div}(A(x/\e)\nabla)$.
Assuming that the coefficient matrix $A=A(y)$ is elliptic, periodic, and belongs to VMO$(\rd)$,
we establish uniform boundary H\"older and $W^{1,p}$ estimates
in $C^1$ domains $\Omega$.
We also prove uniform boundary Lipschitz estimates in $C^{1, \alpha}$ domains 
under the assumption that $A$ is elliptic, periodic, and H\"older continuous.
As in the previous chapter for interior estimates,
boundary H\"older and Lipschitz estimates are proved by a compactness method.
The boundary $W^{1, p}$ estimates are obtained by combining the boundary H\"older estimates
with the interior $W^{1, p}$ estimates, via the real-variable method introduced  in Section \ref{real-variable-section}.

We point out that the boundary  $W^{1, p}$ estimates may fail in Lipschitz domains 
for $p$ large ($p>3$ for $d\ge 3$ and $p>4$ for $d=2$), even for Laplace's equation \cite{Kenig-J-1995}.
Also, the $C^{1, \alpha}$ assumption on the domain $\Omega$ for boundary Lipschitz estimates is more or less sharp.
In general, one should not expect Lipschitz estimates in $C^1$ domains, even for harmonic functions.
The compactness method we use for the boundary Lipschitz estimates is similar to that in Section \ref{section-2.2} for the interior 
Lipschitz estimates. To control the influence of the boundary data, a Dirichlet  corrector is introduced.
The main step in this compactness scheme is to show that the corrector is uniformly Lipschitz.

Throughout this chapter we will assume that $A(y)=(a_{ij}^{\alpha\beta}(y))$,
with $1\le i, j\le d$ and $1\le \alpha, \beta\le m$,
is 1-periodic and satisfies the $V$-ellipticity condition (\ref{weak-e-1})-(\ref{weak-e-2}).
As in the case of interior estimates studied in the last chapter,
the results in this chapter hold for the systems of linear elasticity.



\section[Boundary localization]{Boundary localization in the periodic setting}\label{section-3.0}

In this section we make a few observations on the change of coordinate systems by translations and rotations and
their effects on the operator $\mathcal{L}_\varep$.

\begin{definition}
{\rm
Let $\Omega$ be a bounded domain in $\br^d$.
We say $\Omega$ is Lipschitz (resp. $C^1$) if there exist $r_0>0$, $M_0 >0$, and
$\{ z_k: k=1, 2, \dots, N_0\} \subset \partial\Omega$
such that 
$$
\partial\Omega\subset \bigcup _i B(z_k, r_0),
$$
 and for each $k$, there exist a Lipschitz  (resp. $C^1$) function $\psi_k$ in $\br^{d-1}$
and  a coordinate system, obtained from 
the standard Euclidean system through translation and rotation, so that
$z_k =(0,0)$ and
\begin{equation}\label{boundary-1}
B(z_k, C_0 r_0 )\cap \Omega
=B(z_k, C_0 r_0)
\cap \Big\{ (x^\prime, x_d)\in \br^{d}: \
x^\prime \in \br^{d-1} \text{ and } x_d>\psi_k(x^\prime)\Big\},
\end{equation}
where $C_0=10\sqrt{d}(M_0 +1)$ and $\psi_k$ satisfies
\begin{equation}\label{definition-of-psi}
\psi_ k(0)=0 \quad \text{ and } \quad
\|\nabla \psi_k\|_\infty \le M_0.
\end{equation}
}
\end{definition}

We will call $\Omega$ a $C^{1, \eta}$ domain for some $\eta\in 0,1]$, if each $\psi=\psi_k$ satisfies
\begin{equation}\label{3.0.1}
\psi (0)=0, \quad 
\|\nabla \psi\|_\infty\le M_0 \quad \text{ and } \quad
\|\nabla \psi\|_{C^{0, \eta}(\br^{d-1})}
\le M_0.
\end{equation}
Here we have used the notation 
\begin{equation}\label{Holder-norm}
\| u\|_{C^{0, \eta}(E)}
=\sup \left\{ \frac{|u(x)-u(y)|}{|x-y|^\eta}: \ x, y \in E \text{ and } x\neq y \right\}.
\end{equation}

Suppose that $\mathcal{L}_\varep (u_\varep)=F$ in $\Omega$, where $\Omega$ is a bounded Lipschitz domain.
Recall that if $v_\varep (x) =u_\varep (x+x_0)$ for some $x_0\in \br^d$,
then $-\text{div}\big(B(x/\varep)\nabla v_\varep\big)=G$ in $\widetilde{\Omega}$,
where $B(y)=A(y+(x_0/\varep))$, $G(x)=F(x+x_0)$, and $\widetilde{\Omega}
=\big\{ x: x+x_0 \in \Omega\big\}$.
It should be easy to verify that all of our assumptions on $A$ and $\Omega$ are translation invariant.

To handle the rotation we invoke a theorem,
 whose proof may be found in \cite{Schmutz-2008},
  on the approximation of real orthogonal matrices by orthogonal matrices with rational entries.

\begin{thm}\label{rational-matrix-theorem}
Let $O=(O_{ij}) $ be a $d\times d$ orthogonal matrix. For any $\delta>0$, there exists a $d\times d$ orthogonal
matrix $T=(T_{ij})$ with rational entries such that 
\begin{enumerate}

\item

 $\|O-T\|=\left(\sum_{i,j} |O_{ij}-T_{ij}|^2\right)^{1/2} <\delta$; 
 
 \item
 
each entry of $T$ has a denominator less than a constant depending only on $d$
and $\delta$.
\end{enumerate}
\end{thm}

We now fix $P\in \partial\Omega$, where $\Omega$ is Lipschitz.
By translation we may assume that $P$ is the origin.
There exist $r_0>0$ and an orthogonal matrix $O$ such that
$$
{\Omega}_1 \cap B(0, r_0)
=\big\{ (y^\prime, y_d)\in \br^d:\ y_d>\psi_1 (y^\prime) \big\} \cap B(0,r_0),
$$
where 
$$
{\Omega}_1 =\big\{ y\in \br^d: \ y=Ox \text{ for some } x\in \Omega\big\},
$$
 and
$\psi_1 $ is a Lipschitz function in $\br^{d-1}$ such that $\psi_1 (0)=0$ and $\|\nabla \psi_1 \|_\infty\le M$.
Observe that if $T$ is an orthogonal matrix such that $\|T-O\|_\infty<\delta$, where
$\delta>0$, depending only on $d$ and $M$, is sufficiently small, then
$$
{\Omega}_2 \cap B(0, r_0/2)
=\big\{ (y^\prime, y_d)\in \br^d:\ y_d>\psi_2(y^\prime) \big\} \cap B(0,r_0/2),
$$
where  
$$
{\Omega}_2 =\big\{ y\in \br^d: \ y=Tx \text{ for some } x\in \Omega\big\}
$$
 and $\psi_2 $ is a Lipschitz function in $\br^{d-1}$ such that $\psi_2(0)=0$ and $\|\nabla \psi_2\|_\infty\le 2M$.
Since 
$$
\| O-T\| \le C_d\, \| O^{-1} -T^{-1}\|,
$$
in view of Theorem \ref{rational-matrix-theorem}, the orthogonal matrix $T$ may be chosen 
in such a way that $T^{-1}$ is an orthogonal matrix with rational entries and 
$NT^{-1}$ is a matrix with integer entries, where $N$ is a large integer depending only on $d$ and $M$.

Let $w_\varep (y)=u_\varep (x)$, where $y=N^{-1} Tx$.
Then 
$$
-\text{div}_y \big( H (y/\varep) \nabla_y w_\varep\big) =\widetilde{F} (y)\quad \text{ in } \Omega_3,
$$
where $\widetilde{F}(y) =N^2F(NT^{-1} x)$, 
$$
\Omega_3 =\big\{ y\in \br^d:\ y=N^{-1} Tx \text{ for some } x\in \Omega\big\},
$$
and $H(y)=\big(h_{ij}^{\alpha\beta} (y)\big)$ with
$$
h_{ij}^{\alpha\beta} (y)
= T_{ik} T_{j\ell} a_{k\ell}^{\alpha\beta} (NT^{-1}y).
$$
Observe that since $NT^{-1}$ is a matrix with integer entries, $H(y)$ is 1-periodic if $A(y)$ is 1-periodic.
Also, the matrix $H$ satisfies the ellipticity condition (\ref{weak-e-1})-(\ref{weak-e-2}) with the same $\mu$ as for $A$.
Moreover, $H$ satisfies the same smoothness condition (VMO or H\"older) as $A$.
We further note that
$$
\Omega_3 \cap B(0, c_0 r_0)
=\big\{ (y^\prime, y_d)\in \br^d: \ y_d>\psi _3 (y^\prime)\big\} \cap B(0, c_0r_0)),
$$
where $\psi_3 (y^\prime)=N^{-1} \psi_2(N y^\prime)$ and $c_0$ depends only on $d$ and $M$.
As a result, in the study of uniform boundary estimates for $\mathcal{L}_\varep$,
one may localize the problem to a setting where $z\in \partial\Omega$ is the origin and $B(z, r_0)\cap \Omega$
is given by the region above the graph of a function $\psi$ in $\br^{d-1}$. 
More precisely, it suffices to consider solutions of $\mathcal{L}_\varep  (u_\varep) =F$
in $D_r$ with boundary data given on $\Delta_r$, where $D_r$ and $\Delta_r$
are defined by
\begin{equation}
\label{definition-of-Delta}
\aligned
D_r=& D(r, \psi)\\
=&\big\{ (x^\prime, x_d)\in \br^d: \
|x^\prime|<r  \text{ and } \psi (x^\prime)<x_d<\psi (x^\prime) +10\sqrt{d}(M_0+1)r\big\},\\
\Delta_r =& \Delta(r, \psi)
=\big\{ (x^\prime, \psi(x^\prime))\in \br^d: \ |x^\prime|<r\big\},
\endaligned
\end{equation}
and $\psi$ is a Lipschitz (or $C^1$, $C^{1, \eta}$) function in $\mathbb{R}^{d-1}$.
This localization procedure is used in the proof of several  boundary estimates in the monograph. 
We point out that even if $\Omega$ is smooth,
it may not be possible to choose a local coordinate system such that $\nabla \psi (0)=0$ and 
$A$ is periodic in the $x_d$ direction.

We end this section with a boundary Caccioppoli inequality.

\begin{thm}\label{D-Ca-lemma}
Let $\Omega$ be a bounded Lipschitz domain in $\rd$.
Suppose that $A$ satisfies (\ref{weak-e-1})-(\ref{weak-e-2}).
Let $u_\varep \in H^1(B(x_0, r)\cap\Omega; \mathbb{R}^m)$ be a weak solution of
$\mathcal{L}_\varep (u_\varep) =F +\text{\rm div}(f)$ in $B(x_0, r)\cap\Omega$ and
$u_\varep =g$ on $B(x_0, r)\cap \partial\Omega$ for some $x_0\in \partial\Omega$ and
$0<r<r_0$. Then
\begin{equation}\label{b-Ca}
\aligned
\int_{B(x_0, r/2)\cap \Omega} |\nabla u_\varep|^2\, dx
& \le  \frac{C}{r^2} \int_{B(x_0, r)\cap \Omega} |u_\varep|^2\, dx
+C r^2 \int_{B(x_0, r)\cap \Omega} |F|^2\, dx\\
&\qquad
+C \int_{B(x_0, r)\cap\Omega} |f|^2\, dx
+C\int_{B(x_0, r)\cap \Omega} |\nabla G|^2\, dx,
\endaligned
\end{equation}
where $G\in H^1(B(x_0, r)\cap \Omega; \mathbb{R}^m)$ and $G=g$ on $B(x_0, r)\cap\partial\Omega$.
\end{thm}

\begin{proof}
By considering the function $u_\varep -G$, one may reduce the general case to the special case $G=0$.
In the latter case the inequality is proved by using
$$
\mu \int_\Omega |\nabla (\varphi  u_\e)|^2\, dx
  \le \int_\Omega A(x/\e)\nabla (\varphi u_\e)\cdot \nabla (\varphi u_\e)\, dx
$$
and the equation 
$$
\aligned
A(x/\e)\nabla (\varphi u_\e) \cdot \nabla (\varphi u_\e)
= &A(x/\e)\nabla u_\e \cdot \nabla (\varphi^2 u_\e)
 +A(x/\e) (\nabla \varphi) u_\e \cdot \nabla (\varphi u_\e)\\
&-A(x/\e)\nabla (\varphi u_\e) \cdot (\nabla \varphi) u_\e
+A(x/\e)(\nabla \varphi) u_\e \cdot (\nabla \varphi) u_\e,
\endaligned
$$
where $\varphi \in C_0^1(B(x_0, r))$ is a cut-off function satisfying 
$0\le \varphi\le 1$, $\varphi=1$ on $B(x_0, r/2)$ and $\|\nabla \varphi\|_\infty\le C r^{-1}$.
The argument  is similar to that for the interior Caccioppoli's inequality (\ref{Cacciopoli-1.1}).
\end{proof}



\section{Boundary H\"older estimates}\label{section-3.1}

In this section we prove the following boundary H\"older estimate.

\begin{thm}\label{boundary-Holder-theorem}
Suppose that $A$ is 1-periodic and satisfies (\ref{weak-e-1})-(\ref{weak-e-2}).
Also assume that $A$ satisfies (\ref{VMO-1}).
Let $\Omega$ be a bounded $C^1$ domain and  $0<\rho<1$.
Suppose that $u_\varep\in H^1(B(x_0,r)\cap \Omega;\br^m)$ satisfies 
$$
\left\{
\aligned
\mathcal{L}_\varep (u_\varep) & =0&\quad &
\text { in } B(x_0,r)\cap \Omega,\\
u_\varep & =g &\quad & \text{ on } B(x_0, r)\cap\partial\Omega
\endaligned
\right.
$$
for some $x_0\in \partial\Omega$ and $0<r<r_0$, where $g\in C^{0,1}(B(x_0,r)\cap\partial\Omega)$.
Then
\begin{equation}\label{boundary-Holder-estimate}
\aligned
& \| u_\varep\|_{C^{0, \rho}(B(x_0,r/2)\cap\Omega)}\\
&\qquad
\le C r^{-\rho}
  \bigg\{ \left(\average_{B(x_0,r)\cap\Omega} |u_\varep|^2\right)^{1/2}
 +|g(x_0)|+r \| g\|_{C^{0,1} (B(x_0,r)\cap \partial\Omega)} \bigg\},
\endaligned
\end{equation}
where $C$ depends only on $\mu$, $\rho$, $\omega(t)$ in (\ref{VMO-1}),  and $\Omega$.
\end{thm}


Let $D_r$ and $\Delta_r$ be defined by (\ref{definition-of-Delta}),
where
$\psi:\br^{d-1}\to \br$ is a $C^{1}$ function. To quantify the $C^1$ condition we assume that
\begin{equation}\label{C-1}
\left\{
\aligned
& \psi (0)=0, \quad  \text{supp} (\psi)\subset \left\{ x^\prime\in \mathbb{R}^{d-1}: \ |x^\prime|\le 1\right\}, \quad
\|\nabla \psi\|_\infty \le M_0, \\
& |\nabla \psi(x^\prime)-\nabla \psi (y^\prime)|\le \omega_1 (|x^\prime-y^\prime|)
\quad \text{ for any } x^\prime, y^\prime\in \mathbb{R}^{d-1},
\endaligned
\right.
\end{equation}
where $\omega_1(t)$ is a (fixed) nondecreasing continuous function on $[0, \infty)$
with $\omega_1 (0)=0$.
We will use the following compactness result: 
if $\{ \psi_\ell\}$ is a sequence of $C^1$ functions satisfying (\ref{C-1}),
then there exists a subsequence $\{ \psi_{\ell^\prime}\}$ such that
$\psi_{\ell^\prime} \to \psi$ in $C^1(\mathbb{R}^{d-1})$ for some $\psi$ satisfying  (\ref{C-1}).

By a change of the coordinate system and Campanato's characterization of H\"older spaces,
it suffices to establish the following.

\begin{thm}\label{boundary-Holder-theorem-1}
Suppose that $A$ satisfies the same conditions as in Theorem \ref{boundary-Holder-theorem}.
Let $0<\rho<1$.
Suppose that $u_\varep \in H^1(D_r;\br^m)$ is a solution of
$\mathcal{L}_\varep (u_\varep)=0$ in $D_r$ with
$u_\varep =g$ on $\Delta_r$ for some $0<r\le 1$.
Also assume that $g(0)=0$.
Then
for any $ 0<t<r$,
\begin{equation}\label{boundary-Holder-estimate-1}
\left(\average_{D_t} |u_\varep|^2\right)^{1/2} \le
C \left(\frac{t}{r}\right)^\rho
\left\{ \left(\average_{D_r} |u_\varep|^2\right)^{1/2}
+r \| g\|_{C^{0,1}(\Delta_r)} \right\},
\end{equation}
where $C$ depends only on $\rho$,
$\mu$, $\omega (t)$ in (\ref{VMO-1}), and $(M_0, \omega_1(t))$ in (\ref{C-1}).
\end{thm}

Theorem \ref{boundary-Holder-theorem-1} is proved by a compactness argument,
which is similar to the argument used for the interior Lipschitz estimates in Section \ref{section-2.2}.
However, the correctors are not needed for H\"older estimates.

\begin{lemma}\label{compactness-lemma}
Let $\{\psi_\ell\}$ be a sequence of $C^1$ functions satisfying (\ref{C-1}).
Let $v_\ell \in L^2(D(r, \psi_\ell))$.
Suppose that $\psi_\ell\to \psi$ in $C^1(|x^\prime|<r)$
and $\{ \|v_\ell \|_{L^2(D(r,\psi_\ell))}\}$
 is bounded.
Then there exists a subsequence, which we still denote by $\{ v_\ell\}$,
and $v\in L^2(D(r, \psi))$ such that
$v_k\rightharpoonup v$ weakly in $L^2(\Omega)$ for any open set $\Omega$
such that $\overline{\Omega}\subset D(r, \psi)$.
\end{lemma}

\begin{proof}
Consider the function $w_\ell (x^\prime, x_d) =v_\ell (x^\prime, x_d+\psi_\ell(x^\prime))$, defined in
$$
D(r, 0)=\Big\{ (x^\prime, x_d): \ |x^\prime|<r \text{ and } 0<x_d<10(M_0+1)r\Big\}.
$$
Since $\{ \| w_\ell\|_{L^2(D(r, 0))}\}$ is bounded,
there exists a subsequence, which we still denote by $\{ w_\ell\}$, such that
$w_\ell\rightharpoonup w$ weakly in $L^2(D(r,0))$.
Let
$$
v (x^\prime, x_d) =w(x^\prime,x_d-\psi(x^\prime)).
$$
It is not hard to verify that $v\in L^2(D(r, \psi))$ and
$v_\ell \rightharpoonup v$ weakly in $L^2(\Omega)$
for any open set $\Omega$ such that $\overline{\Omega}\subset D(r, \psi)$.
\end{proof}

Next we prove a homogenization result for a sequence of domains.

\begin{lemma}\label{compactness-lemma-Dirichlet}
Let  $\{ A_\ell (y)\} $ be a sequence of 1-periodic matrices  satisfying the
ellipticity condition (\ref{weak-e-1})-(\ref{weak-e-2})
and $\{ \psi_\ell\}$ a sequence of $C^1$ functions satisfying (\ref{C-1}).
Suppose that 
$$
\left\{
\aligned
 \text{\rm div}(A_\ell(x/\varep_\ell)\nabla u_\ell) &=0  &\quad &  \text{ in } D(r, \psi_\ell),\\
u_\ell &  =g_\ell & \quad & \text{ on } \Delta(r, \psi_\ell),
\endaligned
\right.
$$ 
where
$\varep_\ell \to 0$, $g_\ell (0)=0$  and
\begin{equation}\label{3.1.1-1}
\| u_\ell \|_{H^1(D(r, \psi_\ell))} +\| g_\ell\|_{C^{0, 1}(\Delta(r,\psi_\ell))}\le C.
\end{equation}
Then there exist subsequences of $\{ A_\ell\}$, $\{ \psi_\ell\}$, $\{ u_\ell\}$ and $\{ g_\ell\}$,
which we still denote by the same notation, 
and a function $\psi$ satisfying (\ref{C-1}), $u\in H^1(D(r, \psi);\br^m)$,
$g\in C^{0,1}(\Delta(r, \psi);\br^m)$ and
a constant matrix $A$ such that
\begin{equation}\label{3.1.1-2}
\left\{
\aligned
&\widehat{A_\ell}\to A^0,\\
& \psi_\ell \to \psi\text{ in } C^1(\mathbb{R}^{d-1} ),\\
& g_\ell (x^\prime, \psi_\ell(x^\prime)) \to g(x^\prime, \psi(x^\prime))
\text{ uniformly for } |x^\prime|<r,\\
& u_\ell (x^\prime, x_d-\psi_\ell(x^\prime))
\rightharpoonup u(x^\prime, x_d-\psi(x^\prime)) \text{ weakly in }
H^1(D(r,0);\br^m),
\endaligned
\right.
\end{equation}
and
\begin{equation}\label{3.1.1-3}
\left\{
\aligned
\text{\rm div}(A^0\nabla u) & =0\quad  \text{ in } D(r, \psi),\\
  u& =g\quad \text{ on } \Delta(r, \psi).
  \endaligned
  \right.
  \end{equation}
Moreover, the constant matrix $A^0$ satisfies the ellipticity condition (\ref{weak-eee}).
\end{lemma}

\begin{proof}
Since $\{\widehat{A^\ell}\}$ satisfies (\ref{weak-eee}), by passing to a subsequence,
we may  assume that $\widehat{A_\ell}\to A^0$, which also satisfies (\ref{weak-eee}).
The remaining statements in (\ref{3.1.1-2}) follow readily from (\ref{C-1}) and (\ref{3.1.1-1})
by passing to subsequences. To prove (\ref{3.1.1-3}), we fix $\varphi\in C^1_0(D(r, \psi);\br^m)$.
Clearly, if $\ell$ is sufficiently large, $\varphi\in C_0^1(D(r,\psi_\ell);\br^m)$.
In view of Lemma \ref{compactness-lemma} we may assume that
$A_\ell(x/\varep_\ell)\nabla u_\ell$ converges weakly 
in $L^2(\Omega;\br^{m\times d})$,
where $\Omega$ is any open set such that
supp$(\varphi)\subset\Omega\subset \overline{\Omega}\subset D(r, \psi)$.
Note that $\{ u_\ell\}$  converges to $u$ strongly in $L^2(\Omega;\br^m)$ and weakly in $H^1(\Omega;\br^m)$.
By Theorem \ref{theorem-1.3.4}  we obtain
$$
\int_{D(r, \psi)} A^0\nabla u \cdot \nabla \varphi\, dx=0.
$$
Hence, div$(A^0\nabla u) =0$ in $D(r, \psi)$.

Finally, let $v_\ell(x^\prime, x_d)=u_\ell(x^\prime, x_d+\psi_\ell(x^\prime))$ and
$v(x^\prime, x_d) =u(x^\prime, x_d+\psi(x^\prime))$.
That $u=g$ on $\Delta(r,\psi)$ in the sense of trace follows from the fact that
$v_\ell\rightharpoonup v$  weakly in $H^1(D(r,0);\br^m)$
and $v_\ell=g_\ell(x^\prime, \psi_\ell (x^\prime))$ on $\Delta(r, 0)$.
\end{proof}

\begin{lemma}[One-step improvement]\label{step-1-3.1}
Suppose that $A$ satisfies (\ref{weak-e-1})-(\ref{weak-e-2}) and is 1-periodic.
Fix $0<\sigma<1$.
There exist constants $\varep_0\in (0,1/2)$ and $\theta\in (0,1/4)$, depending only on 
$\mu$, $\sigma$, and $(\omega_1 (t), M_0)$ in (\ref{C-1}), such that
\begin{equation}\label{estimate-3.1.5}
\average_{D(\theta)}
|u_\varep |^2\le \theta^{2\sigma},
\end{equation}
whenever $0<\varep<\varep_0$,
\begin{equation}\label{3.1.3-0}
\left\{
\aligned
\mathcal{L}_\varep (u_\varep) & =0\quad  \text{ in }D_1,\\
 u_\varep & =g\quad \text{ on } \Delta_1,
 \endaligned
 \right.
\end{equation}
and 
\begin{equation}\label{3.1.3-1}
\left\{ 
\aligned
& g(0)=0, \quad \| g\|_{C^{0,1}(\Delta_1)}\le 1,\\
& \average_{D_1} |u_\varep  |^2 \le 1.
\endaligned
\right.
\end{equation}
\end{lemma}

\begin{proof}
Let $\rho=(1+\sigma)/2$.
The proof uses the following observation: 
\begin{equation}\label{3.1.3-2}
\average_{D_r} |w|^2 \le C_0\,  r^{2\rho} \qquad
\text{ for } 0<r<\frac14,
\end{equation}
whenever
\begin{equation}\label{3.1.3-3}
 \left\{
 \aligned
 \text{\rm div}(A^0\nabla w) & =0 \quad  \text{ in } D_{1/2},\\
  w & =g\quad \text{ on } \Delta_{1/2},\\
  \| g\|_{C^{0,1}(\Delta_1)}  & \le 1,\quad g(0)=0,\\
\int_{D(1/2)} |w|^2  & \le |D_1|,
\endaligned
\right.
\end{equation}
where $A^0$ is a constant matrix satisfying the ellipticity condition (\ref{weak-eee}).
This follows from the boundary H\"older estimates in $C^1$ domains for second-order elliptic systems with 
constant coefficients:
$$
\aligned
\left(\average_{D_r} |w|^2\right)^{1/2}
& \le C\, r^\rho \| w\|_{C^{0, \rho}(D_r)}\\
&\le C\, r^\rho \| w\|_{C^{0, \rho}(D_{1/4})}\\
& \le C\, r^\rho \left\{ \| w\|_{L^2(D_{1/2})}
+\| g\|_{C^{0,1}(\Delta_{1/2})} \right\}.
\endaligned
$$
The constant $C_0$ in (\ref{3.1.3-2}) depends only on $\mu$,  
$\sigma$, and $(\omega_1 (t), M_0)$ in (\ref{C-1}).

We now choose $\theta\in (0,1/4)$ so small that $2C_0\theta^{2\rho}\le \theta^{2\sigma}$.
We claim that for this $\theta$, there exists $\varep_0>0$, depending only on 
$\mu$, $\sigma$, and $(\omega_1(t), M_0)$, such that
$(\ref{estimate-3.1.5})$ holds if
$0<\varep<\varep_0$ and $u_\varep$ satisfies (\ref{3.1.3-0})-(\ref{3.1.3-1}).

The claim is proved by contradiction.
Suppose that there exist sequences $\{\varep_\ell\}\subset \mathbb{R}_+$,
$\{ A_\ell \}$ satisfying (\ref{weak-e-1})-(\ref{weak-e-2}) and
(\ref{periodicity}), $\{\psi_\ell\}$ satisfying (\ref{C-1}), and
$\{ u_\ell\}\subset H^1(D(1,\psi_\ell);\br^m)$,
  such that
$\varep_\ell \to 0$, 
\begin{equation}\label{3.1.3-4}
\left\{
\aligned
\text{\rm div} \big(A_\ell (x/\varep_\ell)\nabla u_\ell\big) & =0 \quad \ \text{ in } D(1, \psi_\ell),\\
 u_\ell & =g_\ell \quad \text{ on } \Delta(1, \psi_\ell),\\
 \| g_\ell \|_{C^{0,1}(\Delta(1, \psi_\ell))} & \le 1, \quad g_\ell (0)=0,\\
 \average_{D(1, \psi_\ell)} |u_\ell |^2  & \le 1,
 \endaligned
 \right.
 \end{equation}
 and
 \begin{equation}\label{3.1.3-5}
 \average_{D(\theta, \psi_\ell)} |u_\ell|^2 > \theta^{2\sigma}.
 \end{equation}
 Note that by Caccioppoli's inequality (\ref{b-Ca}),
 the norm of $u_\ell$ in $H^1(D(1/2, \psi_\ell);\br^m)$ is uniformly bounded.
 This allows us to apply Lemma \ref{compactness-lemma-Dirichlet} and obtain 
  subsequences that satisfying (\ref{3.1.1-2}) and (\ref{3.1.1-3}).
It follows  that
$$
\aligned
\int_{D(1/2,\psi)} |u|^2   & =
\lim_{\ell\to \infty} \int_{D(1/2, \psi_\ell)} |u_\ell|^2\\
& \le \lim_{\ell\to \infty}| D(1, \psi_\ell)|\\
&=|D(1, \psi)|,
\endaligned
$$
and
$$
\average_{D(\theta, \psi)}
|u|^2
=\lim_{\ell\to\infty}
\average_{D(\theta, \psi_\ell)}
|u_\ell  |^2
\ge \theta^{2\sigma}.
$$
In view of (\ref{3.1.3-2})-(\ref{3.1.3-3}) we obtain
$\theta^{2\sigma}\le C_0 \,\theta^{2\rho}$, which is in contradiction with the choice of $\theta$.
This completes the proof.
\end{proof}

\begin{lemma}[Iteration]\label{step-2-3.1}
Assume  that $A$ satisfies the same conditions as in  Lemma \ref{step-1-3.1}.
Fix $0<\sigma<1$.
Let $\varep_0$ and $\theta$ be the positive constants given by Lemma \ref{step-1-3.1}.
Suppose that 
$$
\mathcal{L}(u_\varep)=0\quad \text{ in }
D(1,\psi) \quad  \text{ and } \quad u_\varep=g \quad  \text{ on }\Delta(1,\psi),
$$
where $g\in C^{0,1}(\Delta(1,\psi);\br^m)$ and $g(0)=0$.
Then, if $\varep<\theta^{k-1}\varep_0$ for some $k\ge 1$,
\begin{equation}\label{estimate-3.1.4}
\average_{D(\theta^k,\psi)} 
|u_\varep |^2
\le \theta^{2k\sigma}
\max \left\{ 
\average_{D(1, \psi)} |u_\varep  |^2,\ \  \| g\|^2_{C^{0,1}(\Delta(1,\psi ))}\right\}.
\end{equation}
\end{lemma}

\begin{proof}
The lemma is proved by an induction argument on $k$.
Note that the case $k=1$ is given by Lemma \ref{step-1-3.1}.
Suppose that the lemma holds for some $k\ge 1$.
Let $\varep<\theta^k \varep_0$.
We apply Lemma \ref{step-1-3.1} to the function $w(x)=u_\varep (\theta^k x)$
in $D(1, \psi_k)$, where $\psi_k (x^\prime)= \theta^{-k} \psi (\theta^k x^\prime)$.
Since
$$
\mathcal{L}_{\frac{\varep}{\theta^k}} (w) =0 \quad \text{ in } D(1, \psi_k)
$$
and $\theta^{-k}\varep<\varep_0$, we obtain
$$
\aligned
\average_{D(\theta^{k+1}, \psi)}
|u_\varep |^2
& =\average_{D(\theta, \psi_k)} |w |^2\\
& \le \theta^{2\sigma}
\max \left\{ \average_{D(1, \psi_k)} |w|^2, \ \ \| w\|^2_{C^{0,1}(\Delta(1, \psi_k))}  \right\}\\
& =\theta^{2\sigma} \max \left\{
\average_{D(\theta^{2k},\psi)} |u_\varep |^2, \ \ \theta^{2k} \| f\|^2_{C^{0,1} (\Delta(\theta^k, \psi))} \right\}\\
&\le \theta^{2(k+1)\sigma}\max \left\{
\average_{D(1, \psi)}
|u_\varep |^2,\ \  \| f\|^2_{C^{0,1}(\Delta(1, \psi))} \right\},
\endaligned
$$
where we have used the induction assumption in the last step.
The fact that $\{\psi_k\}$ satisfies (\ref{C-1}) uniformly in $k$ is essential here.
\end{proof}

We are now ready to give the proof of Theorem \ref{boundary-Holder-theorem-1}.

\begin{proof}[\bf Proof of Theorem \ref{boundary-Holder-theorem-1}]
By rescaling we may assume that $r=1$.
We may also assume that $0<\varep<\varep_0$,
since the case $\varep\ge \varep_0$ follows directly from the well-known boundary H\"older estimates
for second-order elliptic systems in divergence form in $C^1$ domains with VMO coefficients
\footnote{This result may be  proved by using the real-variable method in Section \ref{real-variable-section}.
The method reduces the problem to a local boundary $W^{1, p}$ estimate for second-order elliptic
systems with constant coefficients near a flat boundary.}.
We further assume that
$$
\| g\|_{C^{0,1}(\Delta (1,\psi))} \le 1 \quad \text{ and } \quad \int_{D(1,\psi)} |u_\varep|^2 \, dx \le 1.
$$
Under these assumptions we will show that
\begin{equation}\label{3.1.6-1}
\average_{D(t,\psi)} |u_\varep|^2 \le C t^{2\sigma}
\end{equation}
for any $t\in (0, 1/4)$.

To prove (\ref{3.1.6-1}) we first consider the case $t\ge (\varep/\varep_0)$.
Choose $k\ge 1$ so that $\theta^{k}\le t<\theta^{k-1}$.
Then $\varep\le \varep_0 t <\varep_0 \theta^{k-1}$.
It follows from Lemma \ref{step-2-3.1} that
$$
\aligned
\average_{D(t,\psi)} |u_\varep|^2 & \le
C \average_{D(\theta^{k-1},\psi )} |u_\varep|^2\\
&\le C\, \theta^{2k\sigma}\\
& \le C\, t^{2\sigma}.
\endaligned
$$

Next suppose that $0<t<(\varep/\varep_0)$.
Let $w(x)=u_\varep (\varep x)$.
Then $\mathcal{L}_1(w)=0 $ in $D(\varep^{-1}_0, \psi_\varep)$ and
$w(0)=0$, where
$\psi_\varep (x^\prime) =\varep^{-1}\psi (\varep x^\prime)$.
By the boundary H\"older estimates in $C^1$ domains
 for the elliptic operator $\mathcal{L}_1$, we obtain
$$
\average_{D(t\varep^{-1}, \psi_\varep) }|w|^2
\le C \left(\frac{t}{\varep}\right)^{2\sigma}
\left\{ \average_{D(\varep_0^{-1}, \psi_\varep)} |w|^2
+\| w\|^2_{C^{0,1}(\Delta(\varep_0^{-1}, \psi_\varep))} \right\}.
$$
Hence,
$$
\aligned
\average_{D(t,\psi)} 
|u_\varep|^2
&\le C \left(\frac{t}{\varep}\right)^{2\sigma}
\left\{ \average_{D(\varep/\varep_0, \psi)} |u_\varep|^2  +\varep^2 \right\}\\
&\le C t^{2\sigma},
\endaligned
$$
where we have used the estimate (\ref{3.1.6-1}) for the case $t=(\varep/\varep_0)$
in the last step. This finishes the proof.
\end{proof}

Theorem \ref{boundary-Holder-theorem} follows from Theorem \ref{boundary-Holder-theorem-1} 
by  Campanato's characterization of H\"older spaces:
if $\mathcal{O}$ is a bounded Lipschitz domain, then
\begin{equation}\label{Campanato}
\aligned
& 
\|u\|_{C^{0,\sigma} (\mathcal{O})}
=\sup\left\{\frac{|u(x)-u(y)|}{|x-y|^\sigma}:\ x, y\in \mathcal{O} \text{ and } x\neq y \right\}\\
&\approx
\sup
\left\{ r^{-\sigma}
\left(\average_{\mathcal{O}(x,r)} \big|u-\average_{\mathcal{O}(x,r)} u\big|^2\right)^{1/2}, \
x\in \mathcal{O} \text{ and } 0<r<\text{diam}(\mathcal{O})\right\},
\endaligned
\end{equation}
where $\mathcal{O}(x,r)=\mathcal{O}\cap B(x,r)$ (see e.g. \cite[pp.70-72]{Gia-Ma-book}).

\begin{proof}[\bf Proof of Theorem \ref{boundary-Holder-theorem}]
Suppose that $\mathcal{L}_\varep (u_\varep)=0$ in $B(x_0, 2r)\cap\Omega$,
$u_\varep =g $ on $B(x_0,2r)\cap \partial\Omega$,
and
$$
\left(\average_{B(x_0,2r)\cap\Omega} |u_\varep|^2 \right)^{1/2}
+|g(x_0)|+r \| g\|_{C^{0,1}(\Delta(x_0,2r)\cap\partial\Omega)} \le 1,
$$
where $x_0\in \partial\Omega$ and $0<r<r_0$.
By a change of the coordinate system we may deduce from Theorem \ref{boundary-Holder-theorem-1}
that
for any $y\in B(x_0, r)\cap \partial\Omega$ and $0<t<cr$,
$$
\average_{B(y,t)\cap\Omega)} |u_\varep -u_\varep(y)|^2 \le C \left(\frac{t}{r}\right)^{2\sigma}.
$$
It follows that
\begin{equation}\label{3.1.7-1}
\average_{B(y,t)\cap\Omega} \Big|u_\varep -\average_{B(y,t)\cap \Omega} u_\varep\Big|^2 \le C 
\left(\frac{t}{r}\right)^{2\sigma}
\end{equation}
for any $y\in B(x_0,r)\cap\partial\Omega$ and $0<t<c\,r$.
This, together with the interior H\"older estimate in Theorem \ref{interior-Holder-theorem}, implies that
estimate (\ref{3.1.7-1}) holds for any $y\in B(x_0,r)\cap \Omega$ and $0<t<c\,r$.
It then follows by the Campanato characterization of H\"older spaces that
$$
\| u_\varep\|_{C^{0, \sigma}(B(x_0,r)\cap\Omega)}
\le C,
$$
which gives the estimate (\ref{boundary-Holder-estimate}).
\end{proof}

\begin{remark}
{\rm 
We may use Lemma \ref{step-2-3.1} to establish a Liouville property for solutions in a half-space
with no smoothness condition on $A$.
Indeed, assume that $A$ satisfies (\ref{weak-e-1})-(\ref{weak-e-2}) and is 1-periodic.
Suppose that $u\in H^1_{\loc}(\mathbb{H}_n(a); \br^m)$, 
$$
\mathcal{L}_1 (u)=0 \quad \text{ in } \mathbb{H}_n(a) \quad \text{ and }
\quad u=0 \quad  \text{ on } \partial \mathbb{H}_n(a),
$$
where $a\in \mathbb{R}$ and
$$
\mathbb{H}_n(a)=\big\{ x\in \brd: \ x\cdot n <a\big\}
$$
is a half-space with outward unit normal $n\in \mathbb{S}^{d-1}$.
Under the growth condition 
$$
\left( \average_{B(0, R)\cap \mathbb{H}_n (a)} |u|^2\right)^{1/2}\le C R^\rho
$$
for $R\ge 1$, where $\rho\in (0,1)$, we may deduce that $u\equiv 0$ in $\mathbb{H}_n(a)$.

To see this, by translation, we may assume that $a=0$.
Let $ u_\e( x)= u(\e^{-1} x)$, where $\e=\theta^{k+\ell}$ and $\theta\in (0,1)$ is given by Lemma \ref{step-1-3.1}.
Then $\mathcal{L}_\e (u_\e)=0$ in $\mathbb{H}_n(0)$ and $u_\e =0$ on $\partial \mathbb{H}_n(0)$.
Choose $\sigma \in (\rho, 1)$.
It follows from Lemma \ref{step-2-3.1} that
$$
\average_{B(0, \theta^k)\cap \mathbb{H}_n (0)} |u_\e|^2
\le C \theta^{2k\sigma} \average_{B(0, 1)\cap \mathbb{H}_n (0)} |u_\e|^2,
$$
if $k$ and $\ell$ are sufficiently large. By a change of variables, this gives
$$
\aligned
\average_{B(0, \theta^{-\ell} )\cap \mathbb{H}_n (0)} |u|^2
&\le C \theta^{2k\sigma} \average_{B(0, \theta^{-\ell -k})\cap \mathbb{H}_n (0)} |u|^2\\
&\le C \theta^{2k(\sigma-\rho)-2\rho \ell},
\endaligned
$$
where we have used the growth condition for the last inequality. 
Since $\rho>\sigma$, we may let $k\to \infty$ to conclude that 
$u\equiv 0$ in $B(0, \theta^{-\ell} )\cap \mathbb{H}_n (0)$ for any $\ell>2$.
It follows that $u\equiv 0$ in $\mathbb{H}_n(0)$.
}
\end{remark}



\section{Boundary $W^{1,p}$ estimates}\label{section-3.00}

In this section we establish the uniform boundary $W^{1,p}$ estimates in $C^1$ domains under the assumption 
that $A$ is elliptic, periodic, and belongs to VMO$(\rd)$.
We will use  $B^{\alpha, p}(\partial\Omega)$ to denote the Besov space on $\partial\Omega$
with exponent $p\in (1, \infty)$ and of order $\alpha\in (0,1)$.
If $\Omega$ is Lipschitz, the space $B^{1-\frac{1}{p}, p} (\partial\Omega)$
may be identified as the set of functions that are traces  on $\partial\Omega$
of $W^{1,p}(\Omega)$ functions.

\begin{thm}\label{W-1-p-D-theorem}
Suppose that $A$ is 1-periodic and satisfies (\ref{weak-e-1})-(\ref{weak-e-2}).
Also assume that $A$ satisfies (\ref{VMO-1}).
Let $\Omega$ be a bounded $C^1$ domain in $\rd$ and $1<p<\infty$.
Let $u_\varep\in W^{1,p}(\Omega; \mathbb{R}^m)$ be the weak solution to the Dirichlet problem
\begin{equation}\label{DP-3.00}
\mathcal{L}_\varep (u_\varep) =F \quad \text{ in }\Omega \quad \text{ and } \quad u_\varep =g \quad \text{ on } 
\partial\Omega,
\end{equation}
where $F\in W^{-1, p}(\Omega; \mathbb{R}^m)$ and $g\in B^{1-\frac{1}{p}, p}(\partial\Omega; \mathbb{R}^m)$. Then
\begin{equation}\label{W-1-p-D}
\| u_\varep\|_{W^{1, p}(\Omega)} \le C_p\, \left\{ \| F\|_{W^{-1.p} (\Omega)} +\| g\|_{B^{1-\frac{1}{p}, p} (\partial\Omega)}
\right\},
\end{equation}
where $C_p$ depends only on $p$, $\mu$, $\omega(t)$ in (\ref{VMO-1}), and $\Omega$.
\end{thm}

By the real-variable method in Section \ref{real-variable-section},
to prove Theorem \ref{W-1-p-D-theorem},
it suffices to establish reverse H\"older estimates for local solutions of $\mathcal{L}_\varep (u_\varep)=0$.
Since the interior case is already settled  in Section \ref{section-2.4},
the remaining task is to establish the following.

\begin{lemma}\label{lemma-3.00-1}
Assume that $A$ and $\Omega$ satisfy the same conditions as in Theorem \ref{W-1-p-D-theorem}.
Let $u_\varep\in H^1(B(x_0, r)\cap \Omega; \mathbb{R}^m)$
 be a weak solution of $\mathcal{L}_\varep (u_\varep)=0$ in $B(x_0, r)\cap\Omega$
with $u_\varep=0$ on $B(x_0, r)\cap\partial\Omega$,
for some $x_0\in \partial\Omega$ and $0< r< r_0$.
Then, for any $2<p<\infty$,
\begin{equation}\label{3.00-1-0}
\left(\average_{B(x_0, r/2)\cap \Omega} |\nabla u_\varep|^p\right)^{1/p}
\le C_p \left(\average_{B(x_0, r)\cap\Omega} |\nabla u_\varep|^2\right)^{1/2},
\end{equation}
where $C_p$ depends only on  $p$, $\mu$, $\omega(t)$ in (\ref{VMO-1}), and $\Omega$.
\end{lemma}

\begin{proof}
We  prove (\ref{3.00-1-0}) by combining the interior $W^{1,p}$ estimates in Section \ref{section-2.4}
with the boundary H\"older estimates in Section \ref{section-3.1}.
Let $\delta(x)=\text{dist}(x, \partial\Omega)$.
It follows from Theorem \ref{interior-W-1-p-theorem} and  Cacciopoli's inequality that
\begin{equation}\label{3.00-1-1}
\average_{B(y, c\delta (y))} |\nabla u_\varep (x)|^p\, dx
\le C \average_{B(y, 2c\delta (y))} \left|\frac{u_\varep (x)}{\delta (x)} \right|^p\, dx,
\end{equation}
for any $y\in B(x_0, r/2)\cap\Omega$,
where $p>2$ and $c>0$ is sufficiently small.
Observe that if $|x-y|<c\, \delta (y)$ for some $c\in (0,1/2)$, then
$$
\delta(x) \le |x-y|+\delta (y) \le 2\, \delta (y).
$$
Also,  since $\delta (y)\le |x-y|+\delta (x)\le c\, \delta (y) +\delta (x)$, we have $\delta(y)\le 2\,\delta (x)$.
By integrating both sides of (\ref{3.00-1-1}) in $y$ over $B(x_0, r/2)\cap\Omega$,
we obtain 
\begin{equation}\label{3.00-1-2}
\int_{B(x_0, r/2)\cap\Omega} |\nabla u_\varep|^p\, dx
\le C \int_{B(x_0, 3r/4)\cap\Omega} \left|\frac{u_\varep (x)}{\delta (x)} \right|^p\, dx.
\end{equation}
Finally, since $u_\varep=0$ on $B(x_0, r)\cap \partial\Omega$,  the boundary
H\"older estimate in Theorem (\ref{boundary-Holder-theorem}) gives
\begin{equation}\label{3.00-1-3}
|u_\varep (x)|\le C_\sigma \left(\frac{\delta(x)}{r}\right)^\sigma \left(\average_{B(x_0, r)\cap\Omega}
|u_\varep|^2\right)^{1/2}
\end{equation}
for any $x\in B(x_0, 3r/4)\cap\Omega$, where $\sigma \in (0,1)$.
By choosing $\sigma$ close to $1$ so that $p(1-\sigma)<1$ and substituting estimate (\ref{3.00-1-3}) into
(\ref{3.00-1-2}), we see that
$$
\left(\average_{B(x_0, r/2)\cap\Omega} |\nabla u_\varep|^p\, dx\right)^{1/p}
\le \frac{C}{r} \left(\average_{B(x_0, r)\cap\Omega}
|u_\varep|^2 \right)^{1/2}.
$$
Since $u_\varep=0$ on $B(x_0, r)\cap\partial\Omega$,
by Poincar\'e inequality, this yields  (\ref{3.00-1-0}).
\end{proof}

\begin{proof}[\bf Proof of Theorem \ref{W-1-p-D-theorem}]
First, since $g\in B^{1-\frac{1}{p}, p} (\partial\Omega; \mathbb{R}^m)$,
there exists $G\in W^{1, p}(\Omega; \mathbb{R}^m)$ such that
$G=g$ on $\partial\Omega$ and 
$$
\| G\|_{W^{1, p}(\Omega; \mathbb{R}^m)}
\le C \, \| g\|_{B^{1-\frac{1}{p}, p} (\partial\Omega)}.
$$
By considering the function $u_\varep -G$, we may reduce the general case to the case $G=0$.

Next, note that if $u_\varep\in W^{1, p}_0(\Omega; \mathbb{R}^m)$ is a weak solution of
$\mathcal{L}_\varep (u_\varep)=F$ in $\Omega$, and
$\widetilde{u}_\varep\in W^{1, p^\prime}_0(\Omega; \mathbb{R}^m)$ is a weak solution of
$\mathcal{L}^*_\varep (\widetilde{u}_\varep)=\widetilde{F}$ in $\Omega$, 
then
$$
\langle F, \widetilde{u}_\varep \rangle_{W^{-1, p} (\Omega)\times W_0^{1, p^\prime} (\Omega)}
=\int_\Omega A(x/\varep)
\nabla u_\e \cdot \nabla \widetilde{u}_\varep\, dx
=\langle \widetilde{F}, u_\varep \rangle_{W^{-1, p^\prime} (\Omega)\times W_0^{1, p} (\Omega)}.
$$
Thus, by a duality argument, it suffices to prove the estimate (\ref{W-1-p-D}) 
with $g=0$ for $p>2$.

Finally, let $p>2$ and $F\in W^{-1, p}(\Omega; \mathbb{R}^m)$.
There exist $\{ f_0, f_1, \dots, f_d\} \subset L^p(\Omega; \mathbb{R}^m)$ such that
$$
F=f_0 +\frac{\partial f_i}{\partial x_i}  \quad \text{ and }  \sum_{i=0}^d
\| f_i\|_{L^p(\Omega)} \le C\, \| F\|_{W^{-1, p}(\Omega)}.
$$
We show that if $u_\varep \in W^{1,2}_0(\Omega; \mathbb{R}^m)$ 
and $\mathcal{L}_\varep  (u_\varep)=F$ in $\Omega$, then
\begin{equation}\label{3.00-2-1}
\|\nabla u_\varep\|_{L^p(\Omega)}
\le C \sum_{i=0}^d \| f_i\|_{L^p(\Omega)}.
\end{equation}
This will be done by applying the real-variable argument given by Theorem \ref{real-variable-Lipschitz-theorem}
(with $\eta=0$).
Let $q=p+1$. Consider two functions
$$
H(x)= |\nabla u_\varep (x)| \quad \text{ and } \quad
h(x) =\sum_{i=0}^d |f_i(x)| \quad \text{ in }\Omega.
$$
For each ball $B$ with the property that $|B|\le c_0 |\Omega|$ and either
$4B\subset \Omega$ or $B$ is centered on $\partial\Omega$, we need to construct 
two measurable functions $H_B$ and $R_B$ that satisfy 
$H\le | H_B| +|R_B| $ on $\Omega\cap 2B$ and condition (\ref{real-variable-Lipschitz-1}).
We will only deal with the case where $B$ is centered on $\partial\Omega$.
The other case is already treated in the proof of interior $W^{1,p}$ estimates in Section \ref{section-2.4}.

Let $B=B(x_0, r)$ for some $x_0 \in \partial\Omega$ and $0<r<r_0/16$.
Write $u_\varep = v_\varep  + w_\varep$ in $4B\cap\Omega$, where $v_\varep
\in W^{1,2}_0(4B\cap\Omega; \mathbb{R}^m)$ and 
$\mathcal{L}_\varep (v_\varep) =F$ in $4B\cap\Omega$.
Let 
$$
H_B =|\nabla v_\varep| \quad \text{ and } \quad R_B =|\nabla w_\varep|.
$$
Then $ H\le H_B +R_B$ in $2B\cap \Omega$, and by Theorem \ref{theorem-1.1-2},
$$
\aligned
\left(\average_{\Omega\cap 2B} |H_B|^2\right)^{1/2}
&\le C\left(\average_{\Omega\cap 4B} |\nabla v_\varep|^2 \right)^{1/2}\\
&\le C \left(\average_{\Omega\cap 4B} |h|^2 \right)^{1/2}.
\endaligned
$$
Note that $\mathcal{L}_\varep (w_\varep)=0$ in $\Omega\cap 4B$ and
$w_\varep=0$ on $4B\cap \partial\Omega$.
In view of Lemma \ref{lemma-3.00-1} we obtain
$$
\aligned
\left(\average_{\Omega\cap 2B} |R_B|^q \right)^{1/q}
&=\left(\average_{\Omega\cap 2B} |\nabla w_\varep|^q \right)^{1/q}
\le C \left(\average_{\Omega\cap 4B} |\nabla w_\varep|^2 \right)^{1/2}\\
&\le C \left(\average_{\Omega\cap 4B} |\nabla u_\varep|^2 \right)^{1/2}
+ C\left(\average_{\Omega\cap 4B} |\nabla v_\varep|^2 \right)^{1/2}\\
&\le C \left(\average_{\Omega\cap 4B} |H_B|^2 \right)^{1/2}
+ C\left(\average_{\Omega\cap 4B} |h|^2 \right)^{1/2}.
\endaligned
$$
Thus we have verified all conditions in Theorem \ref{real-variable-Lipschitz-theorem} with $\eta=0$.
Consequently, we obtain
$$
\aligned
\left(\int_\Omega |\nabla u_\varep|^p\right)^{1/p}
& \le C \left\{ \left(\int_\Omega |\nabla u_\varep|^2 \right)^{1/2}
+\left(\int_\Omega |h|^p \right)^{1/p} \right\}\\
&\le C \sum_{i=0}^d \| f_i\|_{L^p(\Omega)}\\
&\le C\, \| F\|_{W^{-1, p}(\Omega)}.
\endaligned
$$
This completes the proof.
\end{proof}

Let $\Omega$ be a bounded Lipschitz domain in $\br^d$.
Let $x_0\in \partial\Omega$ and $0<r<r_0$.
Suppose that $\mathcal{L}_\e (u_\e)=0$ in $B(x_0, r)\cap \Omega$ and $u_\e=0$ on $B(x_0, r)\cap \partial\Omega$.
It follows from the proof of Theorem \ref{D-Ca-lemma} that
\begin{equation}\label{b-Ca-30}
\left(\average_{B(x_0, sr)\cap \Omega} |\nabla u_\e|^2\right)^{1/2}
\le \frac{C}{(t-s) r} \left(\average_{B(x_0, tr)\cap \Omega} | u_\e|^2\right)^{1/2},
\end{equation}
where $(1/2)< s<t<1$ and $C$ depends only on $\mu$ and $\Omega$.
 By Sobolev-Poincar\'e inequality it follows that
\begin{equation}\label{b-Ca-31}
\left(\average_{B(x_0, sr)\cap \Omega} |\nabla u_\e|^2\right)^{1/2}
\le \frac{C}{(t-s) } \left(\average_{B(x_0, tr)\cap \Omega} | \nabla u_\e|^q\right)^{1/2},
\end{equation}
where $q=\frac{2d}{d+2}$ for $d \ge 3$, and $1<q<2$ for $d=2$.
Consequently, as in the interior case,
by a real-variable argument \cite{Gia-book},  there exists $\bar{p}>2$,
depending only on $\mu$ and $\Omega$, such that
\begin{equation}\label{b-Ca-32}
\left(\average_{B(x_0, r/4)\cap \Omega} |\nabla u_\e|^{\bar{p}}\right)^{1/\bar{p}}
\le {C} \left(\average_{B(x_0, r/2)\cap \Omega} | \nabla u_\e|^2\right)^{1/2}.
\end{equation}
By the proof of Theorem \ref{W-1-p-D-theorem}, the boundary reverse H\"older inequality
(\ref{b-Ca-32}) and its interior counterpart (\ref{reverse-Holder-1.1}) yields the following.

\begin{thm}\label{pert-D-33}
Suppose that $A$ satisfies (\ref{weak-e-1})-(\ref{weak-e-2}).
Let $\Omega$ be a bounded Lipschitz domain.
Then there exists $\delta\in (0,1/2)$, depending only on $\mu$ and $\Omega$,
such that for any $F\in W^{-1, p}(\Omega; \br^m)$ and $g\in B^{1-\frac{1}{p}, p}(\partial\Omega;\br^m)$
with $|\frac{1}{p}-\frac{1}{2}|<\delta$,
there exists a unique solution in $W^{1, p}(\Omega; \br^m)$ to the Dirichlet problem:
$\mathcal{L}_\e (u_\e)=F$ in $\Omega$ and $u_\e=g $ on $\partial\Omega$.
Moreover, the solution satisfies the estimate (\ref{W-1-p-D})
with constant $C$ depending only on $p$, $\mu$ and $\Omega$.
\end{thm}

Theorem \ref{pert-D-33} is due to N. Meyers \cite {Meyers-1963} in the case that $\Omega$ is smooth.
We point out that no smoothness or periodicity assumption on $A$ is needed.



\section[Dirichlet correctors]{Green functions and Dirichlet correctors}\label{section-3.2}

The $m\times m$ matrix $G_\varep (x,y) = \big( G_\varep^{\alpha\beta} (x,y)\big)$
of Green functions  for the operator $\mathcal{L}_\varep$ in $\Omega$
is defined, at least formally, by
\begin{equation}\label{definition-of-Green's-function}
\left\{
\aligned
\mathcal{L}_\varep \big( G^\beta_\varep (\cdot, y)\big)  & =e^\beta \delta_y (\cdot)  &\quad& \text{ in } \Omega,\\
G_\varep^\beta (\cdot,y)  & =0 &\quad&  \text { on } \partial\Omega,
\endaligned
\right.
\end{equation}
where $G^\beta_\varep (x,y) =(G_\varep^{1\beta}(x,y), \dots, G_\varep^{m\beta} (x,y))$,
$e^\beta =(0, \dots, 1, \dots, 0)$ with $1$ in the $\beta^{th}$ position, and 
$\delta_y (\cdot)$ denotes the Dirac delta function with pole at $y$.
More precisely, if $F\in C_0^\infty(\Omega; \br^m)$, then
$$
u_\e (x)=\int_\Omega G_\e (x, y) F(y)\, dy
$$
is the weak solution in $H^1_0 (\Omega; \br^m)$ to
the Dirichlet problem: $\mathcal{L}_\e (u_\e)=F$ in $\Omega$ and
$u_\e=0$ on $\partial\Omega$.

In the case $m=1$ or $d=2$, it is  known that if $\Omega$ is Lipschitz and
$A$ satisfies the ellipticity condition (\ref{weak-e-1})-(\ref{weak-e-2}),  the Green functions exist
and satisfy the estimate
\begin{equation}\label{Green's-function-size}
|G_\varep (x,y)|\le \left\{
\aligned
& C\, |x-y|^{2-d} &\quad&  \text{ if } d\ge 3,\\
& C \left\{ 1+\ln ( r_0 |x-y|^{-1})\right\} &\quad& \text{ if } d=2
\endaligned
\right.
\end{equation}
for any $x,y\in \Omega$ and $x\neq y$,
where $r_0=$ diam$(\Omega)$.  See e.g. \cite{Gruter-1982, Brown-2013}.
 If $d\ge 3$ and $m\ge 2$, the matrix of Green's functions can be constructed and 
$|G_\varep (x, y)|\le C\, |x-y|^{2-d}$ continues to hold as long as local solutions
of $\mathcal{L}_\varep (u_\varep)=0$ and $\mathcal{L}_\varep^* (v_\varep)=0$
satisfy the De Giorgi -Nash H\"older  estimates \cite{Hofmann-2007}.
As a result, in view of the interior and boundary H\"older estimates in 
Sections \ref{section-2.4} and \ref{section-3.1}, we obtain the following.

\begin{thm}\label{G-theorem}
Suppose that $A$ is 1-periodic and satisfies (\ref{weak-e-1})-(\ref{weak-e-2}).
Also assume that $A$ satisfies the VMO condition (\ref{VMO-1}).
Let $\Omega$ be a bounded $C^1$ domain in $\rd$.
Then the matrix $G_\varep (x,y)$ of Green functions exists and satisfies the estimate
(\ref{Green's-function-size}). Moreover,
\begin{equation}\label{Green's-function-Holder}
\left\{
\aligned
 |G_\varep (x,y)| &\le  \frac{ C [\delta (x)]^\sigma}{|x-y|^{d-2+\sigma}} & \quad & \text{ if }\ \  \delta(x)<\frac14 |x-y|,\\
 |G_\varep (x,y)| & \le \frac{ C [\delta (y)]^{\sigma_1}}{|x-y|^{d-2+\sigma_1}} & \quad & \text{ if }\ \  \delta(y )<\frac14 |x-y|,\\
|G_\varep (x,y)| &\le \frac{C [\delta (x)]^{\sigma} [\delta (y)]^{\sigma_1}}{|x-y|^{d-2 +\sigma +\sigma_1}} 
&\quad& \text{ if } \ \ \delta (x) <\frac14 |x-y| \text{ or } \delta(y) <\frac14 |x-y|,
\endaligned
\right.
\end{equation}
for any $x, y \in \Omega$ and $x\neq y$,
where $0<\sigma, \sigma_1<1$ and
$\delta(x)=\text{dist} (x, \partial\Omega)$. The constant
$C$ depends at most on $\mu$, $\sigma$, $\sigma_1$, $\omega(t)$ in
(\ref{VMO}), and $\Omega$.
\end{thm}

\begin{proof}
As we mentioned above, the existence of Green functions and estimate (\ref{Green's-function-size})
follow from the interior and boundary H\"older estimates in Theorems \ref{interior-Holder-theorem} and
\ref{boundary-Holder-theorem}, by the general results in \cite{Hofmann-2007, Brown-2013}.
Suppose that $d\ge 3$.
To see the first inequality in (\ref{Green's-function-Holder}), we fix $x_0, y_0\in \rd$ with
$\delta(x_0)<\frac12 |x_0-y_0|$, and consider
$u_\varep (x)=G_\varep^\beta (x, y_0)$.
Let $r=|x_0-y_0|$.
Since $\mathcal{L}_\varep (u_\varep)=0$ in $B(z_0, r/2)\cap\Omega$ and $u_\varep=0$ on $\partial\Omega$,
by Theorem \ref{boundary-Holder-theorem},
we obtain 
\begin{equation}\label{G-size-0}
|u_\varep (x)|\le C_\sigma \left(\frac{\delta(x)}{r} \right)^\sigma
 \left(\average_{B(z_0,r/2)\cap\Omega} |u_\varep|^2\right)^{1/2}
\end{equation}
for any $x\in B(z_0, r/4)\cap\Omega$, 
where $z_0\in \partial\Omega$ and $\delta (x_0)=|x_0-z_0|$. This gives
 $$
 |G_\varep (x_0, y_0)|\le C\, \left( \delta(x_0)\right)^\sigma |x_0-y_0|^{2-d-\sigma}.
 $$
 The second inequality in (\ref{Green's-function-Holder}) follows from the first
 and the fact that
 \begin{equation}\label{G-adjoint}
 G_{\varep}^{*\alpha\beta} (x, y) =G_{\varep}^{\beta\alpha} (y, x)
 \end{equation}
 for any $x, y\in \Omega$, where $G^*_{\varep} (x,y)=\big(G_\e^{*\alpha\beta}(x, y)\big)$ denotes the matrix of Green functions for
 the operator $\mathcal{L}_\varep^*$ in $\Omega$.
 To prove the third inequality,  we assume that $\delta(x)<\frac14 |x-y|$ and $\delta(y)<\frac14 |x-y|$.
 For otherwise the estimate follows from the first two.
 We repeat the argument for the first inequality, but using
 the second inequality to estimate the RHS of (\ref{G-size-0}).
 
 Finally, we note that if $d=2$, the argument given above works equally well, provided 
 the estimate $|G_\varep (x,y)|\le C$ holds in the case $\delta(x)<(1/2)|x-y|$.
 The latter estimate is indeed true in the general case (see \cite{Brown-2013})
\end{proof}

An argument similar to that used in the poof of Theorem \ref{G-theorem} gives 
\begin{equation} \label{G-0-1}
|G_\varep (x,y)-G_\varep (z,y) |\le \frac{C_\sigma\, |x-z|^\sigma}{|x-y|^{d-2+\sigma}} \qquad \text{ if } |x-z|<(1/2)|x-y|,
\end{equation}
and
\begin{equation} \label{G-0-2}
|G_\varep (x,y)-G_\varep (x,z) |\le \frac{C_\sigma\, |y-z|^\sigma}{|x-y|^{d-2+\sigma}} \qquad \text{ if } |y-z|<(1/2)|x-y|,
\end{equation}
for any $\sigma\in (0,1)$.

\begin{thm}\label{G-theorem-1}
Assume that $A$ and $\Omega$ satisfy the same conditions as in Theorem \ref{G-theorem}.
Then, for any $\sigma \in (0,1)$,
\begin{equation}\label{G-1-0}
\int_\Omega |\nabla_y G_\varep (x,y)| \big[ \delta (y) \big]^{\sigma-1}\, dy
\le C_\sigma \big[ \delta (x) \big]^{\sigma},
\end{equation}
where $C_\sigma$ depends only on $\sigma$, $\mu$, $\omega(t)$ in (\ref{VMO}), 
and $\Omega$.
\end{thm}

\begin{proof}
Fix $x\in \Omega$ and let $r=\delta(x)/2$.
By H\"older's inequality, Cacciopoli's inequality,  and estimate (\ref{Green's-function-Holder}),
$$
\aligned
\int_{B(x, r)} |\nabla G_\varep (x, y)|\, dy
& =\sum_{j=0}^\infty \int_{2^{-j-1}\le |y-x|< 2^{-j} r} |\nabla_y G_\varep (x,y)|\, dy\\
&\le C\sum_{j=0}^\infty\left(\average_{2^{-j-2} r \le |y-x|\le 2^{-j+1} r} |G_\varep (x,y)|^2\, dy\right)^{1/2}
(2^{-j} r)^{d-1}\\
&\le Cr.
\endaligned
$$
It follows that
\begin{equation}\label{G-1-2}
\int_{B(x,r)} |\nabla_y G_\varep (x,y)| \big[ \delta (y) \big]^{\sigma-1}\, dy
\le C\, r^{\sigma}.
\end{equation}

Next, to estimate the integral on $\Omega \setminus B(x, r)$, we observe that
if $Q$ is a cube in $\rd$ with the property that $3Q\subset \Omega\setminus \{ x\}$
and its side length $\ell (Q)\sim \text{dist} (Q, \partial \Omega)$,
then
\begin{equation}\label{G-1-3}
\aligned
\int_Q |\nabla_y G_\varep (x, y)| \big[\delta (y) \big]^{\sigma-1} \, dy
& \le C \, \big[ \ell (Q) \big]^{\sigma -1} |Q| \left(\average_Q |\nabla_y G_\varep(x,y)|^2\, dy \right)^{1/2}\\
&\le C \, \big[ \ell (Q) \big]^{\sigma -2} |Q| \left(\average_{2Q} | G_\varep(x,y)|^2\, dy \right)^{1/2},
\endaligned
\end{equation}
where we have used Cacciopoli's inequality for the last step. 
This, together with the third inequality in (\ref{Green's-function-Holder}), gives
\begin{equation}\label{G-1-4}
\aligned
\int_Q |\nabla_y  &G_\varep (x, y)| \big[\delta (y) \big]^{\sigma-1} \, dy\\
&\le C\, r^{\sigma_1} [\ell (Q)]^{\sigma+\sigma_2-2} |Q| \left(\average_{2Q} 
\frac{dy}{|x-y|^{2(d-2+\sigma_1 +\sigma_2)}} \right)^{1/2}\\
&\le C\, r^{\sigma_1} [\ell (Q)]^{\sigma+\sigma_2-2}  \int_{2Q} 
\frac{dy}{|x-y|^{(d-2+\sigma_1 +\sigma_2)}} \\
&\le C\, r^{\sigma_1} \int_Q \frac{[\delta (y)]^{\sigma+\sigma_2-2}}{|x-y|^{d-2+\sigma_1 +\sigma_2}}\, dy,
\endaligned
\end{equation}
where $0<\sigma_1, \sigma_2<1$,
and we have used the observation that $\delta(y)\sim \ell(Q)$ for $y\in 2Q$ and
$|x-y|\sim |x-z|$ for any $y,z\in 2Q$.

Finally, we perform a Whitney decomposition on $\Omega$ (see \cite{Stein-1970}).
This gives $\Omega =\cup_j Q_j$, where $\{ Q_j\}$ is a sequence of (closed) non-overlapping cubes with the property
that $4Q_j\subset \Omega$ and dist$(Q_j, \partial\Omega)\sim \ell (Q_j)$.
Let
$$
\mathcal{O} =\bigcup_{3Q_j\subset \Omega\setminus \{ x \} } Q_j.
$$
Note that if $y\in \Omega\setminus \mathcal{O}$, then $y\in Q_j$ for some $Q_j$ such that $x\in 3Q_j$.
It follows that $|y-x|\le C\, \ell(Q_j)\le C \delta (x)$.
Hence,
\begin{equation}\label{G-1-5}
\int_{\Omega\setminus \mathcal{O}} |\nabla_y G_\varep (x, y)| \big[\delta (y) \big]^{\sigma-1} \, dy
\le C\, r^{\sigma-1},
\end{equation}
by the proof of (\ref{G-1-2}).
By summation, the estimate(\ref{G-1-4}) leads to
\begin{equation}\label{G-1-6}
\int_{\mathcal{O}} |\nabla_y G_\varep (x, y)| \big[\delta (y) \big]^{\sigma-1} \, dy
 \le C\, r^{\sigma_1} \int_\Omega \frac{[\delta (y)]^{\sigma+\sigma_2-2}}{(|x-y|+r)^{d-2+\sigma_1 +\sigma_2}}\, dy.
\end{equation}
Since $\Omega$ is Lipschitz, the integral in the RHS of (\ref{G-1-6})
is bounded by
$$
\aligned
&C\int_0^\infty \int_{\mathbb{R}^{d-1}}
\frac{t^{\sigma+\sigma_2 -2}}{( |y^\prime| +|t-r| +r|)^{d-2 +\sigma_1 +\sigma_2}}\, dy^\prime dt\\
&\qquad\qquad \le C\int_0^\infty \frac{ t^{\sigma +\sigma_2-2}}{(|t-r| +r)^{\sigma_1 +\sigma_2 -1}}\, dt\\
&\qquad\qquad \le C\, r^{\sigma-\sigma_1} \int_0^\infty
\frac{t^{\sigma+\sigma_2 -2}}{(|t-1|+1)^{\sigma_1 +\sigma_2 -1}}\, dt\\
&\qquad\qquad  \le C\, r^{\sigma-\sigma_1},\\
\endaligned
$$
where we have chosen $\sigma_1,\sigma_2\in (0,1)$ so that $\sigma_1+\sigma_2>1$ and
$\sigma<\sigma_1<1$.
This, together with (\ref{G-1-6}) and (\ref{G-1-5}),
completes the proof.
\end{proof}

\begin{definition}
{\rm
The Dirichlet corrector  $\Phi_\varep =\big(\Phi_{\varep, j}^{\beta}\big)$ for the operator $\mathcal{L}_\varep$
in $\Omega$ is defined by
\begin{equation}\label{definition-of-Phi}
\left\{ 
\aligned 
\mathcal{L}_\varep (\Phi_{\varep, j}^\beta)  & =0\ \  \quad \text{ in } \Omega,\\
\Phi_{\varep, j}^\beta  & =P_j^\beta \quad \text{ on } \partial\Omega,
\endaligned
\right.
\end{equation}
where $P_j^\beta (x)=x_je^\beta$ for $1\le j\le d$ and $1\le \beta\le m$.
}
\end{definition}

In the study of boundary Lipschitz estimates for solutions with the Dirichlet condition,
the function $\Phi^\beta_{\varep,j} (x)-P_j^\beta(x) $ plays a similar role as $\varep\chi_j^\beta(x/\varep)$
for interior Lipschitz estimates. Note that  $\Phi_{\varep, j}^\beta-P_j^\beta\in H_0^1(\Omega)$
satisfies
$$
\mathcal{L}_\varep \big\{ \Phi_{\varep, j}^\beta -P_j^\beta\big\}
=\mathcal{L}_\varep \big\{ \varep\chi_j^\beta(x/\varep)\big\}\quad \text{ in } \Omega.
$$

Our goal in the rest of this section is to establish the following Lipschitz estimate of $\Phi_\varep$.

\begin{thm}\label{Dirichlet-corrector-theorem} 
Suppose that $A$ is 1-periodic and satisfies (\ref{weak-e-1})-(\ref{weak-e-2}).
Also assume that $A$ satisfies the H\"older continuity condition (\ref{smoothness}).
Let $\Omega$ be a bounded $C^{1,\eta}$ domain.
Then
\begin{equation}\label{Dirichlet-corrector-estimate}
\|\nabla \Phi_\varep\|_{L^\infty(\Omega)} \le C,
\end{equation}
where $C$ depends only on $\mu$, $\lambda$, $\tau$ and $\Omega$.
\end{thm}

The proof of Theorem \ref{Dirichlet-corrector-theorem} uses the estimate of Green functions
in Theorem \ref{G-theorem-1} and a blowup argument. 

\begin{lemma}\label{lemma-3.2-2}
Suppose that $A$ and $\Omega$ satisfy the same conditions as in Theorem \ref{G-theorem}.
For $g\in C^{0,1} (\Omega;\br^m)$, let $u_\varep\in H^1(\Omega;\br^m)$ be the solution of the
Dirichlet problem: $\mathcal{L}_\varep (u_\varep)=0$ in $\Omega$ and
$u_\varep =g$ on $\partial\Omega$.
Then for any $x_0\in \partial\Omega$ and $c\,\varep\le r<\text{\rm diam}(\Omega)$,
\begin{equation}
\label{estimate-3.2.2-1}
\left(\average_{B(x_0,r)\cap \Omega}
|\nabla u_\varep|^2 \right)^{1/2}
\le
C\, \Big\{ \| \nabla g\|_{L^\infty(\Omega)}
+\varep^{-1} \|g\|_{L^\infty(\Omega)}\Big\},
\end{equation}
where $C$ depends only on $\mu$, $\omega (t)$ in (\ref{VMO}), and $\Omega$.
\end{lemma}

\begin{proof}
Let $v_\varep=g\varphi_\varep$, where $\varphi_\varep$ is a cut-off function in $ C_0^\infty(\br^d)$
such that $0\le \varphi_\varep\le 1$,
$\varphi_\varep (x)=1$ if $\delta(x)\le (\varep/4)$,
$\varphi_\varep (x)=0$ if $\delta(x)\ge (\varep/2)$,
and $|\nabla\varphi_\varep|\le C\, \varep^{-1}$.
Note that
\begin{equation}\label{3.2.2-1}
\|\nabla v_\varep\|_{L^\infty(\Omega)}
\le 
C \Big\{ \| \nabla g\|_{L^\infty(\Omega)}
+\varep^{-1} \|g\|_{L^\infty(\Omega)}\Big\}.
\end{equation}
Thus it suffices to estimate $w_\varep =u_\varep -v_\varep$.

To this end, we observe that $\mathcal{L}_\varep (w_\varep)=-\mathcal{L}_\varep (v_\varep)$ in $\Omega$
and $w_\varep =0$ on $\partial\Omega$.
It follows that
$$
w_\varep^\alpha (x)
=-\int_\Omega 
\frac{\partial}{\partial y_i}
\Big\{ G_\varep^{\alpha\beta} (x,y)\Big\} \cdot a_{ij}^{\beta\gamma} (y/\varep)
\cdot \frac{\partial v_\varep^\gamma}{\partial y_j}\, dy.
$$
Hence,
$$
| w_\varep (x)|
\le C\, \|\nabla v_\varep\|_{L^\infty(\Omega)}
\int_{\delta(y)\le \varep} |\nabla_y G_\varep (x,y)|\, dy.
$$
By Theorem \ref{G-theorem-1},
\begin{equation}\label{3.2.2-3}
\int_{\delta(y)\le \varep}
|\nabla_y G_\varep (x,y)|\, dy
\le C_\sigma\,  [\delta (x)]^\sigma \varep^{1-\sigma}
\end{equation}
for any $\sigma\in (0,1)$.
This implies  that 
$$
\aligned
\|w_\varep\|_{L^\infty(B(Q,2r)\cap\Omega)} & 
\le C_\sigma  r^\sigma \varep^{1-\sigma}\|\nabla v_\varep\|_{L^\infty(\Omega)}\\
&\le C\, r \|\nabla v_\varep\|_{L^\infty(\Omega)},
\endaligned
$$
where we have used the assumption $r\ge c\,\varep$.
By Caccioppoli's inequality we obtain
$$
\aligned
\average_{B(Q,r)\cap \Omega}
|\nabla w_\varep|^2 
&\le \frac{C}{r^2} 
\average_{B(Q,2r)\cap \Omega} |w_\varep|^2 + C\, \|\nabla v_\varep\|^2_{L^\infty(\Omega)}\\
&\le C\, \|\nabla v_\varep\|^2_{L^\infty(\Omega)},
\endaligned
$$
which, in view of (\ref{3.2.2-1}), gives the desired estimate.
\end{proof}

\begin{proof}[\bf Proof of Theorem \ref{Dirichlet-corrector-theorem}]

Let $r_0=\text{diam}(\Omega)$.
The case $\varep\ge cr_0$ follows from the boundary Lipschitz estimates in $C^{1, \eta}$ domains
for elliptic systems in divergence form with H\"older continuous coefficients \cite{Gia-Ma-book}. Suppose that $0<\varep<c\, r_0$.
Fix $1\le j\le d$ and $1\le \beta\le m$, let
$$
u_\varep =\Phi_{\varep, j}^\beta (x)-P_j^\beta (x) -\varep \chi_j^\beta (x/\varep).
$$
Then $\mathcal{L}_\varep (u_\varep)=0$ in $\Omega$ and 
$u_\varep =-\varep\chi_j^\beta (x/\varep)$ on $\partial\Omega$.
Under the assumption that $A$ is H\"older continuous,
the corrector $\nabla \chi$ is H\"older  continuous.
It then follows from Lemma \ref{lemma-3.2-2} that
\begin{equation}\label{3.2.3-1}
\average_{B(x_0,r)\cap \Omega} |\nabla u_\varep|^2 \le C
\end{equation}
for any $x_0\in \partial\Omega$ and $c\, \varep< r < r_0$.
By the interior Lipschitz estimates in Section \ref{section-2.2}, this implies that
$|\nabla u_\varep (x)|\le C$ if $\delta(x)\ge \varep$.
Indeed, suppose $x\in \Omega$ and $\delta (x)\ge \varep$. Let $r=\delta (x) =|x-x_0|$, where $x_0\in \partial\Omega$.
Then
$$
\aligned
|\nabla u_\varep (x)| &\le C \left(\average_{B(x, r/2)} |\nabla u_\varep|^2\right)^{1/2}\\
&\le C \left(\average_{B(x_0, 2r)\cap\Omega} |\nabla u_\varep|^2\right)^{1/2}\\
&\le C 
\endaligned
$$
Consequently, we obtain 
$$
|\nabla \Phi_\varep (x)|\le C+C\, \|\nabla \chi\|_\infty \le C,
$$
 if $\delta(x)\ge \varep$.

The case $\delta(x)<\varep$ can be handled by a blow-up argument, using the observation
\begin{equation}\label{3.2.3-2}
\average_{B(x_0,2\varep )\cap \Omega} |\nabla \Phi_\varep |^2 \le C,
\end{equation}
which follows from (\ref{3.2.3-1}).
Without loss of generality, suppose $x_0=0\in \partial\Omega$.
Consider $w(x)=\varep^{-1}\Phi_{\varep, j}^\beta (\varep x)$ in $\Omega^\varep
=\{ x\in \br^d: \ \varep x\in \Omega\}$.
Then $\mathcal{L}_1 (w)=0$ in $\Omega^\varep$ and $w=P_j^\beta(x)$ on $\partial\Omega^\varep$.
It follows from the boundary Lipschitz estimates for $\mathcal{L}_1$
that
$$
\|\nabla w\|_{L^\infty(B(0,1)\cap\Omega^\varep)}
\le C \left\{ 1+\left(\average_{B(0,2)\cap \Omega^\varep} |\nabla w|^2\right)^{1/2}\right\}.
$$
By a change of variables, this yields 
$$
\aligned
\|\nabla\Phi_\varep\|_{L^\infty(B(0,\varep)\cap\Omega)}
&\le C \left\{
1+\left( \average_{B(0,2\varep)\cap \Omega} |\nabla \Phi_\varep|^2\right)^{1/2}
\right\}\\
&\le C,
\endaligned
$$
which completes the proof of Theorem \ref{Dirichlet-corrector-theorem}.
\end{proof}

\begin{remark}\label{remark-3.2-2}
{\rm
Let $\psi:\mathbb{R}^{d-1}\to \mathbb{R}$ be a function such that
\begin{equation}\label{psi-1}
\left\{
\aligned
& \psi (0)=0, \quad \text{supp} (\psi) \subset \{ x^\prime: \, |x^\prime|\le 3 \},\\
& \|\nabla \psi\|_\infty \le M_0 \quad \text{ and } \quad
\|\nabla \psi\|_{C^{0, \eta}(\mathbb{R}^{d-1})} \le M_0,
\endaligned
\right.
\end{equation}
where $\eta\in (0,1)$ and $M_0>0$ are fixed.
We construct a bounded $C^{1,\eta}$
domain $\Omega_\psi$ in $\br^d$ with the following property,
\begin{equation}\label{3.2.4-1}
\left\{
\aligned
& D (4, \psi) \subset \Omega_\psi
\subset \big\{
(x^\prime, x_d): \
|x^\prime|<8 \text{ and } |x_d|<100 (M_0 +1) \big\},\\
& \big\{ (x^\prime, \psi (x^\prime)): \
|x^\prime|<4 \big\}
\subset \partial\Omega_\psi,\\
& \Omega_\psi \setminus \{ (x^\prime, \psi(x^\prime)): \ |x^\prime|\le 4\}
\text{ depends only on $M_0$}.
\endaligned
\right.
\end{equation}
Let $\Phi_\varep =\Phi_\varep (x, \Omega_\psi, A)$
be the matrix of Dirichlet correctors for $\mathcal{L}_\varep$ in $\Omega_\psi$.
It follows from the proof of Theorem \ref{Dirichlet-corrector-theorem} that
$$
\|\nabla\Phi_\varep(\cdot, \Omega_\psi, A) \|_{L^\infty(\Omega_\psi)}\le C,
$$
where $C$ depends only on $\mu$, $\lambda$, $\tau$, and $(\eta, M_0)$
in (\ref{psi-1}). Also note  that by (\ref{thm-2.2-1-0}),
$$
\|\Phi^\beta_{\varep,j} (\cdot, \Omega_\psi, A)-P_j^\beta (\cdot) \|_{L^2(\Omega_\psi)}
\le C \sqrt{\varep}.
$$
}
\end{remark}



\section{Boundary Lipschitz estimates}\label{section-3.3}

In this section we establish the uniform boundary Lipschitz estimates in $C^{1, \eta}$ domains
for solutions with  Dirichlet conditions under the assumption that
$A$ is elliptic, periodic and H\"older continuous.

Let $\nabla_\ta g$  denote the tangential gradient of $g$ on the boundary.

\begin{thm}\label{Dirichlet-Lip-theorem}
Suppose that $A$ is 1-periodic and satisfies (\ref{weak-e-1})-(\ref{weak-e-2}).
Also assume that $A$ satisfies (\ref{smoothness}).
Let $\Omega$ be a bounded $C^{1,\eta}$ domain for some $\eta\in (0,1)$.
Suppose $u_\varep \in H^1(B(x_0,r)\cap\Omega;\br^m)$ is a solution of
$$
\left\{
\aligned
\mathcal{L}_\varep (u_\varep)& =0 \, \quad \text{ in } B(x_0,r)\cap\Omega,\\
 u_\varep & =g \quad \text{ on } B(x_0,r)\cap \partial\Omega
 \endaligned
 \right.
 $$
 for some $g\in C^{1, \eta} (B(x_0,r)\cap\partial\Omega;\br^m)$, 
 where $x_0\in \partial\Omega$ and $0<r<r_0$.
 Then
 \begin{equation}\label{Dirichlet-Lip-estimate}
 \aligned
 \|\nabla u_\varep\|_{L^\infty(B(x_0,r/2)\cap\Omega)}
 \le C   \bigg\{ r^{-1} & \left(\average_{B(x_0,r)\cap\Omega} |u_\varep|^2\right)^{1/2}
  + r^\eta \| \nabla_\ta g\|_{C^{0,\eta}(B(x_0,r)\cap \partial\Omega)}\\
  & +\|\nabla_\ta g\|_{L^\infty(B(x_0,r)\cap \partial\Omega)}
  +r^{-1} \|g\|_{L^\infty(B(x_0,r)\cap \partial\Omega)}
  \bigg\},
  \endaligned
 \end{equation}
 where $C$ depends only on $\mu$, $\lambda$, $\tau$, and $\Omega$.
 \end{thm}
 
 As in the case of interior Lipschitz estimates, Theorem \ref{Dirichlet-Lip-theorem}
 is proved by a compactness argument.
 However, we need to replace $\varep \chi (x/\varep)$
 by the function $\Phi_\varep(x) -P(x)$,
  where $\Phi_\varep$ is the matrix of Dirichlet correctors constructed in Section \ref{section-3.2}.
  
 Let $D_r=D(r, \psi)$ and $\Delta _r=\Delta (r, \psi)$ be defined by (\ref{definition-of-Delta}) with
 $\psi$ satisfying (\ref{psi-1}).
 
 \begin{lemma}[One-step improvement]\label{lemma-3.3-1}
 Let $\sigma=\eta/4$.
 Let $\Phi^\beta_{\varep,j}=\Phi^\beta_{\varep,j} (x, \Omega_\psi, A)$ be defined as in 
 Remark \ref{remark-3.2-2}.
 There exist constants $\varep_0\in (0,1/2)$ and $\theta\in (0,1/4)$,
 depending only on $\mu$, $\lambda$, $\tau$, and $(\eta, M_0)$, 
 such that
 \begin{equation}\label{estimate-3.3-1-1}
 \left(\average_{D_\theta}
 \big| u_\varep  -\Phi^\beta_{\varep, j} n_j(0)n_i(0) \average_{D_\theta} \frac{\partial u^\beta_\varep}{\partial x_i}
 \big |^2\right)^{1/2}
 \le \theta^{1+\sigma},
 \end{equation}
 whenever $0<\varep<\varep_0$, 
 $$
 \left\{
 \aligned
 \mathcal{L}_\varep (u_\varep) & =0\quad \text{ in } D_1,\\
  u & =g \quad  \text{ on } \Delta_1,
  \endaligned
  \right.
  $$
  and
 $$
 \left\{
 \aligned
&  \left(\average_{D_1} |u_\varep|^2\right)^{1/2} \le 1,\quad
 \| \nabla_\ta g\|_{C^{0,2\sigma}(\Delta_1)} \le 1, \\
 & \quad g(0)=0, \quad |\nabla_\ta g (0)|=0.
 \endaligned
 \right.
 $$
 \end{lemma} 
 
 \begin{proof}
 The lemma is proved by contradiction, using Lemma \ref{compactness-lemma-Dirichlet} and
 the following observation: if
 $\text{div}(A^0 \nabla w)=0$ in $D(1/2)$ and $w=g$ on $\Delta(1/2)$, where
 $A^0$ is a constant matrix satisfying (\ref{weak-eee})
 and $|g(0)|=|\nabla_\ta g(0)|=0$,
   then
 \begin{equation}\label{3.3.1-1}
 \aligned
 \big\| w- &x_j n_j (0)  n_i(0) \average_{D_r} \frac{\partial w}{\partial x_i} \big\|_{L^\infty(D_r)}\\
& \le C_0\, r^{1+2\sigma}
 \Big\{ \| w\|_{L^2(D_{1/2})} 
 +\|\nabla_\ta g\|_{C^{0,2\sigma}(\Delta_{1/2})}\Big\}
 \endaligned
 \end{equation}
 for any $r\in (0,1/4)$, where $C_0$ depends only on $\mu$ and $(\eta, M_0)$ in (\ref{psi-1}).
  To see (\ref{3.3.1-1}), we use the boundary $C^{1,2\sigma}$ estimates in $C^{1, \eta}$ domains
  for second-order elliptic systems with constant coefficients satisfying the Legendre-Hadamard ellipticity condition 
  (\ref{weak-eee}),
  $$
  \|\nabla w\|_{C^{0,2\sigma}(D_{1/4})}
  \le 
  C \Big\{ \| w\|_{L^2(D_{1/2})}
 +\|\nabla_\ta g\|_{C^{0,2\sigma}(\Delta_{1/2})}\Big\},
$$
to obtain
  \begin{equation}\label{3.3.1-2}
  \aligned
\big \| w-x_j \average_{D_r} \frac{\partial w}{\partial x_j} \big\|_{L^\infty(D_r)}
&\le C\, r^{1+2\sigma} \| \nabla w\|_{C^{0,2\sigma}(D_r)}\\
&\le C\, r^{1+2\sigma} \| \nabla w\|_{C^{0, 2\sigma} (D_{1/4})}\\
&  \le C r^{1+2\sigma}
 \Big\{ \| w\|_{L^2(D_{1/2})}
 +\|\nabla_\ta g\|_{C^{0,2\sigma}(\Delta_{1/2})}\Big\}.
 \endaligned
 \end{equation}
Since $\nabla_\ta w(0)=0$, we have
 $$
 n_i(0) \frac{\partial w}{\partial x_j} (0)-n_j(0) \frac{\partial w}{\partial x_i} (0)=0.
 $$
Hence, for $x\in D_r$,
 $$
 \aligned
 & \left| x_j \average_{D_r} \frac{\partial w}{\partial x_j} -
 x_j n_j (0) n_i (0) \average_{D_r} \frac{\partial w}{\partial x_i}\right|\\
& \qquad =\left | x_j n_i (0) \left\{ n_i (0) \average_{D_r} \frac{\partial w}{\partial x_j} -
n_j (0) \average_{D_r} \frac{\partial w}{\partial x_i} \right\}\right|\\
 &\qquad =\left| x_j n_i(0) \left\{ n_i (0) \average_{D_r} \left(\frac{\partial w}{\partial x_j}
 -\frac{\partial w}{\partial x_j} (0)\right)
 -n_j(0) \average_{D_r} \left( \frac{\partial w}{\partial x_i} -\frac{\partial w}{\partial x_i} (0) \right)\right\}\right|\\
 &\qquad \le C|x| \average_{D_r} |\nabla w -\nabla w(0)|\\
 &\qquad \le C\, r^{1+2\sigma} \| \nabla w\|_{C^{0, 2\sigma} (D_r)}.
 \endaligned
 $$
This, together with (\ref{3.3.1-2}), gives (\ref{3.3.1-1}).
 
 Next we choose $\theta\in (0,1/4)$, depending only on $\mu$, $\eta$ and $M_0$, such  that 
 $$
 2C_0 \theta^\sigma \{ 1+ |D(1, \psi)|^{1/2} \} \le 1.
 $$
 We will show that for this $\theta$, there exists $\varep_0>0$, depending only on $\mu$, $\lambda$,
 $\tau$, and $(\eta, M_0)$, such that the estimate (\ref{estimate-3.3-1-1}) holds.
 
 Suppose  this is not the case. Then there exist sequences $\{ \varep_\ell\} \subset \mathbb{R}_+$,
 $\{A^\ell\}$ satisfying (\ref{periodicity}), (\ref{weak-e-1})-(\ref{weak-e-2}) and (\ref{smoothness}),
 $\{ u_\ell\}$,
 $\{ g_\ell\}$ and $\{ \psi_\ell\}$, such that
 $\varep_\ell\to 0$, $\psi_\ell$ satisfies (\ref{psi-1}),
 $$
 \left\{
 \aligned
 \text{\rm div}(A^\ell(x/\varep_\ell)\nabla u_\ell) & =0 \quad  \ \text{ in } D(1, \psi_\ell),\\
 u_\ell  & =g_\ell \quad  \text{ on } \Delta (1,\psi_\ell),
 \endaligned
 \right.
 $$
 \begin{equation}\label{3.3.1-4}
  \left(\average_{D(\theta, \psi_\ell)}
  \big| u_\ell-\Phi_{\varep_\ell, j}^{\beta,\ell} n^\ell_j (0) n^\ell_i (0)  \average_{D(\theta, \psi_\ell)}
  \frac{\partial u^\beta_\ell}{\partial x_i} \big|^2 \right)^{1/2}
 >\theta^{1+\sigma},
 \end{equation}
  and
 \begin{equation}\label{3.3.1-3}
 \aligned
 & |g_\ell (0)|=|\nabla_\ta g_\ell (0)|=0,\\
 & \|u_\ell\|_{L^2(D(1, \psi_\ell))} \le 1, \quad
 \|\nabla_\ta g_\ell \|_{C^{0,2\sigma}(D(1, \psi_\ell ))} \le 1,
 \endaligned
 \end{equation}
 where $n^\ell =(n_1^\ell, \dots, n_d^\ell)$ denotes the outward unit normal to $\partial D(1, \psi_\ell )$ and 
 $$
 \Phi_{\varep_\ell}^{\beta, \ell} (x) =\Phi^\beta_{\varep_\ell} (x, \Omega_{\psi_\ell}, A^\ell).
 $$
  In view of Lemma \ref{compactness-lemma-Dirichlet},
 by passing to subsequences we may assume that as $\ell\to \infty$,
 \begin{equation}\label{3.3.1-5}
 \left\{
 \aligned
& \widehat{A^\ell} \to  A,\\
 &\psi_\ell  \to \psi \quad \text{ in } C^1(|x^\prime|<4),\\
 & g_\ell(x^\prime, \psi_\ell (x^\prime)) 
 \to g (x^\prime, \psi(x^\prime)) \quad \text{ in } C^1(|x^\prime|<1),\\
& u_\ell (x^\prime, x_d-\psi_\ell (x^\prime)) \to u (x^\prime, x_d-\psi(x^\prime)) 
 \text{ weakly in } L^2(D(1,0); \mathbb{R}^m),\\
& u_\ell (x^\prime, x_d-\psi_\ell (x^\prime)) \to u (x^\prime, x_d-\psi(x^\prime)) 
 \text{ weakly in } H^1(D(1/2,0); \mathbb{R}^m).
 \endaligned
 \right.
 \end{equation}
 Moreover, $u$ is a weak solution of $\text{div}\big(A\nabla w)=0$ in $ D(1/2, \psi)$ with
 $u=g$ on $\Delta ( 1/2,\psi)$, and $A$ is a constant matrix satisfying (\ref{weak-eee}).
  
   By Remark \ref{remark-3.2-2},
  \begin{equation}\label{claim-3.3}
  \| \Phi_{\varep_\ell, j}^{\beta,\ell} -P_j^\beta\|_{L^2(\Omega_{\psi_\ell})}
  \le C \sqrt{\varep_\ell}.
  \end{equation}
 This, together with (\ref{3.3.1-5}), allows us to take the limit 
 ($\ell \to \infty$) in 
 (\ref{3.3.1-4}) to arrive at
\begin{equation}\label{3.3.1-9}
 \left(\average_{D_\theta} \big| u-x_j n_j (0) n_i (0)  \average_{D_\theta} 
  \frac{\partial u}{\partial x_i}\big|^2 \, dx\right)^{1/2}
 \ge \theta^{1+\sigma},
\end{equation}
where $D_\theta=D(\theta, \psi)$.
Similarly, by (\ref{3.3.1-3}) and (\ref{3.3.1-5}), we obtain
 $$
 \| u\|_{L^2(D_1)}\le 1\quad \text{ and } \quad
\| \nabla_\ta g\|_{C^{0, 2\sigma} (\Delta_1)}\le 1.
$$
It then follows from (\ref{3.3.1-9}) and (\ref{3.3.1-1}) that
$$
\theta^{1+\sigma} \le C_0\, \theta^{1+2\sigma} \big\{ 1+ |D(1, \psi)|^{1/2} \big\},
$$
which is in contradiction with the choice of $\theta$.
This completes the proof.
 \end{proof}
 
 Let $\psi_k(x^\prime)=\theta^{-k} \psi (\theta^k x^\prime)$,
 where  $k\ge 1$ and $\psi:\mathbb{R}^{d-1}\to \mathbb{R} $ be a $C^{1, \eta}$ function satisfying (\ref{psi-1}).
 Let
 \begin{equation}\label{Pi}
 \Pi_{\varep, j}^{\beta, k} (x) =\theta^k \Phi_{\frac{\varep}{\theta^k}, j}^\beta 
(\theta^{-k} x, \Omega_{\psi_k}, A).
\end{equation}
Then
\begin{equation}\label{Pi-1}
\mathcal{L}_{\frac{\varep}{\theta^k}} \left (\Pi_{\varep, j}^{\beta, k} \right) =0 \quad
\text{ in } D(1, \psi) \quad \text{ and } \quad 
\Pi_{\varep, j}^{\beta, k} =P_j^\beta \quad \text{ on } \Delta(1, \psi).
\end{equation}
This will be used in the next lemma.

\begin{lemma}[Iteration]\label{lemma-3.3-2}
Let $\sigma$, $\varep_0$ and $\theta$ be constants given by Lemma \ref{lemma-3.3-1}.
Suppose that $\mathcal{L}_\varep (u_\varep) =0$ in $D(1, \psi)$
and $u_\varep =g$ on $\Delta(1, \psi)$,
where $g\in C^{1,2\sigma} (\Delta(1, \psi);\br^m)$ and
$|g(0)|=|\nabla_\ta g (0)|=0$.
Assume that $\varep<\theta^{\ell-1} \varep_0$ for some $\ell\ge 1$.
Then there exist constants $E_\varep^k =\big (E_{\varep}^{\beta, k}\big )\in \br^{m}$ for $k=0,\dots, \ell-1$, such that
\begin{equation}\label{estimate-3.3.2-1}
|E_\varep^k |\le C \theta^{-1} J
\end{equation}
and
\begin{equation}\label{estimate-3.3-2-2}
\left( \average_{D(\theta^\ell, \psi)}
\Big| u_\varep -\sum_{k=0}^{\ell-1} \theta^{\sigma k}
\Pi_{\varep, j}^{\beta, k}  (x)\, n_j (0)\,   E_{\varep }^{\beta, k} \Big|^2\, dx \right)^{1/2}
\le \theta^{\ell (1+\sigma)} J,
\end{equation}
where $\Pi_{\varep, j}^{\beta, k}$ is defined by (\ref{Pi}),
\begin{equation}\label{estimate-3.3-2-3}
J=\max \left\{
\| \nabla_\ta \|_{C^{0, 2\sigma} (\Delta(1, \psi))},
\left(\average_{D(1, \psi)} | u_\varep|^2\right)^{1/2}  \right\},
\end{equation}
and $\psi_k (x^\prime)=\theta^{-k} \psi (\theta^k x^\prime)$.
\end{lemma}

\begin{proof}
The lemma is proved by an induction argument on $\ell$.
The case $\ell=1$ follows by applying Lemma \ref{lemma-3.3-1}
to $u_\varep/J$, with
$$
E_\varep^{\beta, 0}=n_i( 0) \average_{D(\theta, \psi)}  \frac{\partial u_\varep^\beta}{\partial x_i}.
$$
Suppose that Lemma \ref{lemma-3.3-2}
holds for some $\ell\ge 1$.
Consider the function
$$
w(x)
=\theta^{-\ell}
\left\{ u_\varep (\theta^\ell x)
-\sum_{k=0}^{\ell-1}  \theta^{\sigma k}\, 
\Pi_{\varep, j}^{\beta, k}  (\theta^\ell x)\, n_j (0) \, E_\varep^{\beta, k}  \right\}
$$
on $D(1, \psi_\ell)$.
Note that $\mathcal{L}_{\theta^{-\ell} \varep} (w) =0$ in $D(1, \psi_\ell)$ and $w (0)=0$.
Also, since $\nabla_\ta u_\varep (0)=0$ and
$$
\nabla_\ta \Pi_{\varep, j}^{\beta, k} (0) \, n_j (0)
=\nabla_\ta P_j^\beta (0)\, n_j (0)=0,
$$
we see that $\nabla_\ta w (0)=0$.
It then follows from Lemma \ref{lemma-3.3-1} that
\begin{equation}\label{3.3-2-1}
\aligned
& \left(\average_{D(\theta, \psi_\ell )}
\Big| w -\Phi^\beta_{\frac{\varep}{\theta^\ell}, j}
(x,  \Omega_{\psi_\ell}, A)\, n_j (0)\, n_i(0)
 \average_{D(\theta, \psi_\ell)}
\frac{\partial w^\beta}{\partial x_i}
\Big|^2 \, dx\right)^{1/2}\\
& \qquad\qquad \le
\theta^{1+\sigma}
\max \left\{
\left(\average_{D(1, \psi_\ell)} | w |^2\right)^{1/2}, \
\|\nabla_\ta w\|_{C^{0, 2\sigma} (\Delta(1, \psi_\ell))} \right\}.
\endaligned
\end{equation}
By a change of variables this yields
\begin{equation}\label{3.3-2-3}
\aligned
& \left(\average_{D(\theta^{\ell+1}, \psi)}
 \Big| u_\varep -\sum_{k=0}^\ell \theta^{\sigma k}   \Pi_{\varep, j}^{\beta, k} (x) \, 
n_j (0)\, E_{\varep}^{\beta, k} \Big|^2\, dx \right)^{1/2}\\
& \qquad\qquad \le 
\theta^{\ell +1 +\sigma}
\max \left\{
\left(\average_{D(1, \psi_\ell)} | w |^2\right)^{1/2}, \
\|\nabla_\ta w\|_{C^{0, 2\sigma} (\Delta(1, \psi_\ell))} \right\},
\endaligned
\end{equation}
where
$$
E_{\varep}^{\beta, \ell} =\theta^{-\sigma \ell} 
n_i (0)\average_{D(\theta, \psi_\ell)}
\frac{\partial w^\beta}{\partial x_i}.
$$
Note that by the induction assumption,
\begin{equation}\label{3.3-2-5}
\left(\average_{D(1, \psi_\ell)} |w|^2\right)^{1/2}
 \le \theta^{\ell \sigma} J.
\end{equation}
Also, observe that
$$
\aligned
\|\nabla_\ta w\|_{C^{0, 2\sigma}(\Delta(1, \psi_\ell))}
& 
=\theta^{2\ell \sigma}
\|\nabla_\ta f -\sum_{k=0}^{\ell -1} \theta^{\sigma k}
\nabla_\ta \Pi_{\varep, j}^{\beta, k} \, n_j (0)\,  E_{\varep}^{\beta, k} \|_{C^{0, 2\sigma} (\Delta(\theta^\ell, \psi))}\\
&\le  J \theta^{2\ell\sigma} 
\left\{ 1+ C\theta^{-1} \sum_{k=0}^{\ell -1} \theta^{\sigma k} \|\nabla_\ta \Pi_{\varep, j}^{ k} 
n_j (0)\|_{C^{0, 2\sigma} (\Delta (\theta^\ell, \psi))}
\right\}\\
&\le \theta^{\ell \sigma} J \left\{
\theta^{\ell \sigma} + C \theta^{-1} \sum_{k=0}^{\ell -1}
\theta^{\sigma k} \|\nabla_\ta P_j n_j (0)\|_{C^{0, 2\sigma} (\Delta(\theta^{\ell}, \psi))} \right\}\\
&\le \theta^{\ell \sigma} J 
\left\{ \theta^{\sigma} +\frac{C \| n\|_{C^{0, 2\sigma}(\Delta(\theta^\ell, \psi))}}{\theta (1-\theta^\sigma)}\right\}.
\endaligned
$$
Since $\|\nabla \psi\|_{C^{0, 4\sigma} (\mathbb{R}^{d-1})}\le M_0$,
by making an initial dilation of the independent variables, if necessary, we may assume that $\psi$ is such that
$$
\theta^{\sigma} +\frac{C \| n\|_{C^{0, 2\sigma}(\Delta(1, \psi))}}{\theta (1-\theta^\sigma)}<1.
$$
This, together with (\ref{3.3-2-3}) and (\ref{3.3-2-5}), gives the estimate (\ref{estimate-3.3-2-2}).

Finally, we note that by Caccioppoli's inequality,
$$
\aligned
|E_\varep^\ell|
& \le C\, \theta^{-\sigma\ell-1} \left\{
\left(\average_{D(1, \psi_\ell )} |w|^2\right)^{1/2} + \| \nabla_\ta w\|_{L^\infty(D(1, \psi_\ell))} \right\}\\
&\le C\, \theta^{-\sigma\ell-1} \Big\{\theta^{\sigma} J +\| \nabla_\ta w \|_{C^{0, 2\sigma}(\Delta (1, \psi_\ell))} \Big\}\\
&\le C \theta^{-1} J,
\endaligned
$$
where we have used the fact $\nabla_\ta w(0)=0$.
The induction argument is now complete.
\end{proof}

\begin{lemma}\label{lemma-3.3-3}
Suppose that $\mathcal{L}_\varep (u_\varep)=0$ in $D_1$ and $u_\varep =g$ on $\Delta_1$.
Then
\begin{equation}\label{estimate-3.3-3}
\left(\average_{D_r}
|\nabla u_\varep|^2\right)^{1/2} \le C \left\{
\left(\average_{D_1}  |u_\varep |^2 \right)^{1/2}
+\| g\|_{C^{1,\eta} (\Delta_1)}\right\}
\end{equation}
for any $0<r<(1/2)$,
where $C$ depends only on $\mu$, $\lambda$, $\tau$, and $(\eta, M_0)$.
\end{lemma}

\begin{proof}
Let $\Phi_\varep =\Phi_\varep (x, \Omega_\psi, A)=(\Phi_{\varep, j}^{\beta} (x) )$
and $b_j^\beta =\frac{\partial}{\partial x_j} \big\{ g^\beta(x^\prime, \psi(x^\prime))\big\}(0)$.
By considering the function
$$
u _\varep (x) -\left\{ u_\varep (0) +\sum_{j=1}^{d} \Phi_{\varep, j}^{\beta} (x) b_j^\beta\right\},
$$
we may assume that $ |u_\varep(0)|=|\nabla_\ta u_\varep  (0)|=0$.
Under these assumptions we will show that
\begin{equation}\label{3.3-3-0}
\left(\average_{D_r} |u_\varep |^2 \right)^{1/2}
\le C\, r 
\left\{
\left(\average_{D_1} |u_\varep|^2 \right)^{1/2}
+\| \nabla_\ta g \|_{C^{0,\eta} (\Delta_1)}\right\}
\end{equation}
for $0<r<(1/2)$.
Estimate (\ref{estimate-3.3-3}) follows from (\ref{3.3-3-0}) by 
Caccioppoli's inequality.

Let $\sigma$, $\theta$ and $\varep_0$ be the constants given by Lemma \ref{lemma-3.3-2}.
Let $0<\varep<\theta\varep_0$ (the case $\varep\ge \theta\varep_0$ follows from
the Lipschitz estimates for elliptic systems with H\"older continuous coefficients).
Suppose that
$$
\theta^{i+1} \le \frac{\varep}{\varep_0} <\theta^i \quad \text{ for some } i\ge 1.
$$
We may assume that $0<r<\theta$ (the case $\theta\le r<1/2$ is trivial).
We first consider the case $\frac{\varep}{\varep_0} \le r<\theta$.
Since $\theta^{\ell+1}\le r <\theta^\ell$ for some $\ell=1, \dots, i$,
it follows from Lemma \ref{lemma-3.3-2} that
\begin{equation}\label{3.3-3-1}
\aligned
& \left(\average_{D_r} |u_\varep |^2 \right)^{1/2} 
\le C\left(\average_{D_{\theta^\ell}} |u_\varep|^2 \right)^{1/2} \\
& \le C 
\left(\average_{D_{\theta^\ell}}
\big| u_\varep -\sum_{k=1}^{\ell-1} \theta^{\sigma k} 
 \Pi_{\varep, j}^{\beta, k}\, n_j (0) E^{\beta, k}_\varep\big|^2 \right)^{1/2}
+C \sum_{k=0}^{\ell-1}
\theta^{\sigma k} |E_\varep^k|
\|\Pi_{\varep}^k\|_{L^\infty(D_{\theta^\ell})}\\
&
\le
\theta^{\ell (1+\sigma)}J 
+CJ \sum_{k=0}^{\ell-1}
\theta^{\sigma k} \| \Pi_\varep^k\|_{L^\infty(D_{\theta^\ell})},
\endaligned
\end{equation}
where 
$$
J=\max \big\{ \| u_\varep\|_{L^2(D_1)},\ \| g\|_{C^{1,\eta} (\Delta_1)}\big\}.
$$
Note that $\Pi_\varep^k (0)=0$ and by Remark \ref{remark-3.2-2},
$$
\|\nabla \Pi_\varep^k \|_{L^\infty(D_{\theta^k})} \le C.
$$
This implies that 
$$
\|\Pi_\varep^k \|_{L^\infty(D_{\theta^\ell})} \le C \,\theta^\ell \quad \text{  for } k<\ell.
$$
In view of (\ref{3.3-3-1}) we obtain estimate (\ref{3.3-3-0}) for any $r\in (\varep/\varep_0, 1)$.

Finally, to treat the case $0<r < (\varep/\varep_0)$, we use a familiar blow-up argument.
Let $w(x)=\varep^{-1} u_\varep (\varep x)$.
Then $\mathcal{L}_1 (w)=0$ in $D(2\varep_0^{-1}, \psi_\varep)$ 
and $w=\varep^{-1} g(\varep x)$ on $\Delta(2\varep_0^{-1}, \psi_\varep)$,
where $\psi_\varep (x^\prime)
=\varep^{-1}\psi(\varep x^\prime)$.
Since $|w(0)|=|\nabla_\ta w(0)|=0$,
by the boundary Lipschitz estimates for $\mathcal{L}_1$,
$$
\| w\|_{L^\infty(D(s, \psi_\varep))}
\le C\, s \left\{
\left(\average_{ D(2\varep_0^{-1}, \psi_\varep))} |w|^2\right)^{1/2}
+\|\nabla_\ta w \|_{C^{0, \eta} (\Delta (2\varep_0^{-1}, \psi_\varep))}\right\}
$$
for $0<s<(2/\varep_0)$. By a change of variables
this yields 
$$
\aligned
\| u_\varep\|_{L^\infty(D(r, \psi))}
& \le C\, r \left\{ \varep^{-1} \left(\average_{D(\frac{2\varep}{\varep_0}, \psi)} |u_\varep |^2\right)^{1/2}
+\|\nabla_\ta g\|_{C^{0, \eta}(\Delta(1, \psi))}\right\}\\
&\le C\, rJ,
\endaligned
$$
where we have used the estimate (\ref{3.3-3-0}) for the case $r=(2\varep/\varep_0)$
in the last inequality.
The proof is now complete.
\end{proof}

We are now in a position to give the proof of Theorem \ref{Dirichlet-Lip-theorem}.

\begin{proof}[\bf Proof of Theorem \ref{Dirichlet-Lip-theorem}]

By rescaling we may assume that $x_0=0$ and $r=1$.
By a change of the coordinate system we may deduce from Lemma \ref{lemma-3.3-3} 
that if $y\in \partial\Omega$, $|y|<(1/2)$ and $0<t<(1/4)$,
$$
\left(\average_{B(y,t)\cap\Omega} |\nabla u_\varep|^2\right)^{1/2}
\le C \left\{ \left(\average_{B(0,1)\cap \Omega }  |u_\varep |^2 \right)^{1/2}
+\|g\|_{C^{1, \eta}(B(0,1)\cap\partial\Omega)}\right\},
$$
where $C$ depends only on $\mu$, $\lambda$, $\tau$, $\eta$, and $\Omega$.
This, together with the interior Lipschitz estimate, gives (\ref{Dirichlet-Lip-estimate}).
Indeed, let $x\in B(0,1/2)\cap\Omega$ and $t=\text{dist}(x, \partial\Omega)$.
Choose $z\in B(0,1/2)\cap\partial\Omega$ such that $|x-z|\le C\, t$. Then
$$
\aligned
|\nabla u_\varep (x)| &\le C \left(\average_{B(x, t/2)} |\nabla u_\varep|^2 \right)^{1/2}\\
&\le C\left(\average_{B(z, (C+1) t)\cap \Omega} |\nabla u_\varep|^2 \right)^{1/2}\\
& \le C \left\{ \left(\average_{B(0,1)\cap \Omega }  |u_\varep |^2 \right)^{1/2}
+\|g\|_{C^{1, \eta}(B(0,1)\cap\partial\Omega)}\right\},
\endaligned
$$
which completes the proof.
\end{proof}



\section{Dirichlet problem in $C^1$ and $C^{1,\eta}$ domains}\label{section-3.4}

In this section we establish uniform H\"older estimates and Lipschitz estimates 
as well as the nontangential-maximal-function estimates
for the Dirichlet problem
\begin{equation}\label{DP-3}
\left\{
\aligned
\mathcal{L}_\varep (u_\varep) &=F \quad \text{ in } \Omega,\\
u_\varep & =g \quad \, \text{ on } \partial\Omega.
\endaligned
\right.
\end{equation}
The results for the $W^{1, p}$ estimates are already given in Section \ref{section-3.00}.

We begin with improved estimates for Green functions.

\begin{thm}\label{G-theorem-3}
Suppose that $A$ is 1-periodic and satisfies (\ref{weak-e-1})-(\ref{weak-e-2}).
Also assume that $A$ satisfies (\ref{smoothness}).
Let $\Omega$ be a bounded $C^{1, \eta}$ domain for some $\eta\in (0,1)$.
Then the matrix of Green functions satisfies 
\begin{equation}\label{Lip-estimate-Green-function-1}
|\nabla_x G_\varep (x,y)|
+
|\nabla_y G_\varep (x,y)|
\le C\, |x-y|^{1-d}
\end{equation}
\begin{equation}\label{Lip-estimate-Green-function-3}
|\nabla_x G_\varep (x,y)|  \le \frac{ C \delta(y)}{|x-y|^d}, \quad  \quad
|\nabla_y G_\varep (x,y)|  \le \frac{ C \delta(x)}{|x-y|^d},
\end{equation}
and
\begin{equation}\label{Lip-estimate-Green-estimate-2}
|\nabla_x\nabla_y G_\varep (x,y)|\le C\, |x-y|^{-d}
\end{equation}
 for any $ x, y\in \Omega$ and $ x\neq y$, where $C$ depends only on
 $\mu$, $\lambda$, $\tau$, and $\Omega$.
 \end{thm}
 
 \begin{proof}
Recall that if $d\ge 3$, $|G_\varep (x,y)|\le C |x-y|^{2-d}$.
Using  $\mathcal{L}_\varep \big( G_\varep (\cdot, y)\big)=0$ in $\Omega \setminus \{ y\}$
and $G_{\varep}^{*\alpha\beta} (x,y)= G_{\varep}^{\beta\alpha} (y, x)$,
estimates (\ref{Lip-estimate-Green-function-1}), (\ref{Lip-estimate-Green-function-3}),
and (\ref{Lip-estimate-Green-estimate-2})
follow readily from the interior and boundary Lipschitz estimates in Sections \ref{section-2.2}
and \ref{section-3.3}. 
If $d=2$ one should replace the size estimate on $|G_\varep (x, y)|$ by
$$
|G_\varep (x, y)-G_\varep (z,y)|\le C \quad \text{ if } |x-z|<(1/2)|x-y|,
$$ 
and apply the Lipschitz estimates to $u_\varep (x)= G_\varep (x, y)-G(z, y)$.
 \end{proof}
 
\begin{thm}[Lipschitz estimate]\label{Lip-estimate-theorem-3}
Suppose that $A$ and $\Omega$ satisfy the same conditions as in Theorem \ref{G-theorem-3}.
Let $g\in C^{1, \sigma}(\partial\Omega;\br^m)$ and $F\in L^p (\Omega;\br^m)$, 
where $\sigma \in (0, \eta)$ and $p>d$.
Then the unique solution in $C^{0,1}(\Omega;\br^m)$
to the Dirichlet problem (\ref{DP-3})
 satisfies the estimate
\begin{equation}\label{Lip-estimate-3}
\|\nabla u_\varep\|_{L^\infty(\Omega)}
\le C \Big\{ \|g\|_{C^{1, \sigma}(\partial\Omega)}
+\| F \|_{L^p(\Omega)}
\Big\},
\end{equation}
where $C$ depends only on $\mu$, $\lambda$, $\tau$, $\sigma$, $p$, and $\Omega$.
\end{thm}

\begin{proof}
Let
$$
v_\varep (x) =\int_\Omega G_\varep(x,y) F(y)\, dy.
$$
Then $\mathcal{L}_\varep (v_\varep)=F$ in $\Omega$ and $v_\varep=0$ on $\partial\Omega$.
Note that by (\ref{Lip-estimate-Green-function-1}),
\begin{equation}\label{Lip-v}
|\nabla v_\varep (x)|\le C \int_\Omega \frac{|F(y)|}{|x-y|^{d-1}}\, dy.
\end{equation}
By H\"older's inequality we obtain $\|\nabla v_\varep\|_{L^\infty(\Omega)}\le C\, \| F\|_{L^p (\Omega)}$
for $p>d$.
Thus, by subtracting $v_\varep$ from $u_\varep$,
we may now assume that $F=0$ in Theorem \ref{Lip-estimate-theorem-3}.
In this case, by covering $\Omega$ with balls of radius $c\, r_0$,
 we may deduce from Theorems \ref{Dirichlet-Lip-theorem}
and  \ref{interior-Lip-theorem} that
$$
\aligned
\|\nabla u_\varep\|_{L^\infty(\Omega)}
&\le C \Big\{
\|\nabla u_\varep\|_{L^2(\Omega)}
+\| g\|_{C^{1, \eta}(\partial\Omega)}\Big\}\\
&\le C\, \| g\|_{C^{1, \eta}(\partial\Omega)},
\endaligned
$$
where we have used the energy estimate
$$
\|\nabla u_\varep \|_{L^2(\Omega)}
\le C \,\| g\|_{H^{1/2}(\partial\Omega)}
\le C\, \| g\|_{C^1(\partial\Omega)}.
$$
\end{proof}

\begin{remark}\label{remark-Lip-1}
{\rm 
Let $v_\varep$ be the same as in the proof of Theorem \ref{Lip-estimate-theorem-3}.
Using the estimate (\ref{Lip-v}) and writing 
$$
\Omega=\bigcup_{j=j_0 }^\infty B(x, 2^{-j} r)\cap\Omega,
$$
where $2^{-j_0} r_0 \approx \text{diam}(\Omega)$,
it is not hard to deduce that
$$
\| v_\varep\|_{L^\infty(\Omega)}
\le C_\rho  \sup_{\substack{x\in \Omega\\ 0<r<r_0}} r^{1-\rho}\average_{B(x, r)\cap \Omega} |F|,
$$
where $\rho\in (0,1)$. This allows us to replace (\ref{Lip-estimate-3}) by
\begin{equation}\label{Lip-3-0}
\| \nabla u_\varep\|_{L^\infty(\Omega)}
\le C \left\{ \| g\|_{C^{1, \sigma}(\partial\Omega)}
+\sup_{\substack{x\in \Omega\\ 0<r<r_0}} r^{1-\rho}\average_{B(x, r)\cap \Omega} |F|\right\}
\end{equation}
for any $\sigma, \rho \in (0,1)$.
}
\end{remark}

\begin{remark}\label{remark-Lip-2}
{\rm
Consider the Dirichlet problem
\begin{equation}\label{DP-3-0}
\mathcal{L}_\varep (u_\varep) =\text{div}(f) \quad \text{ in } \Omega
\quad \text{ and } \quad u_\varep =g \quad \text{ on } \partial\Omega,
\end{equation}
where $f=(f_i^\beta)$.
Let 
\begin{equation}\label{representation-divergence}
w^\alpha_\varep (x) =-\int_\Omega \frac{\partial}{\partial y_i} \bigg\{ G_\varep^{\alpha\beta} (x, y) \bigg\}.
f_i^\beta (y)\, dy.
\end{equation}
Then $\mathcal{L}_\varep (w_\varep)=\text{div} (f)$ in $\Omega$ and $w_\varep =0$ on $\partial\Omega$.
Since
$$
w^\alpha_\varep (x) =-\int_\Omega \frac{\partial}{\partial y_i} \bigg\{ G_\varep^{\alpha\beta} (x, y) \bigg\}.
\big( f_i^\beta (y) -f_i^\beta (x) \big)\, dy,
$$
it follows that
$$
\aligned
|\nabla w_\varep (x)|
&\le C\int_\Omega |\nabla_x \nabla_y G_\varep (x,y)| \, | f(y)- f(x)|\, dy\\
&\le C\, \| f\|_{C^{0, \rho} (\Omega)}\int_\Omega \frac{dy}{|x-y|^{d-\rho}}\\
&\le C\, \| f\|_{C^{0, \rho} (\Omega)},
\endaligned
$$
where we have used the estimate (\ref{Lip-estimate-Green-estimate-2}).
As a result, we obtain the following estimate for the solutions of (\ref{DP-3-0}),
\begin{equation}\label{Lip-3-1}
\|\nabla u_\varep\|_{L^\infty(\Omega)}
\le C \, \Big\{ \| g\|_{C^{1, \sigma}(\partial\Omega)} +\| f\|_{C^\rho(\Omega)} \Big\},
\end{equation}
under the assumptions that $A$ is periodic, elliptic and H\"older continuous, and that $\Omega$ is $C^{1, \eta}$.
}
\end{remark}

 Recall that for a continuous function $u$ in $\Omega$, the nontangential maximal function  is defined by
\begin{equation}\label{definition-of-nontangential-max-function}
(u)^* (y)
=\sup \Big\{
|u(x)|: \ x\in \Omega \text{ and }
|x-y|< C_0 \, \text{dist}(x, \partial \Omega)\Big\}
\end{equation}
for $y\in \partial\Omega$,
where $C_0=C_0 (\Omega)> 1 $ is sufficiently large.

\begin{thm}[Nontangential-maximal-function estimates]\label{nontangential-max-theorem-3}
Suppose that $A$ and $\Omega$ satisfy the same conditions as in Theorem \ref{G-theorem-3}.
Let $1<p\le \infty$.
For $g\in L^p (\partial\Omega;\br^m)$,
let $u_\varep$ be the unique solution to the
Dirichlet problem $\mathcal{L}_\varep (u_\varep)=0$ in $\Omega$
and $u_\varep =g$ on $\partial\Omega$ with the 
property $(u_\varep)^* \in L^p(\partial\Omega)$.
Then
\begin{equation}\label{nontangential-max-estimate-3}
\| (u_\varep)^* \|_{L^p(\partial\Omega)}
\le C_p\, \| g\|_{L^p(\partial\Omega)},
\end{equation}
where $C_p$ depends only on $p$, $\mu$, $\lambda$, $\tau$, and $\Omega$.
\end{thm}

\begin{proof}
Write
$$
u_\varep (x)=\int_{\partial\Omega} P_\varep(x,y) g(y)\, dy,
$$
where the Poisson kernel $P_\varep(x,y)=(P_\varep^{\alpha\beta}(x,y))$
for $\mathcal{L}_\varep$ on $\Omega$ is given by
\begin{equation}\label{definition-of-Poisson}
P^{\alpha\beta}_\varep (x,y)
=-n_i(y) a_{ji}^{\gamma\beta} (y/\varep) \frac{\partial}{\partial y_j}
\Big\{ G_\varep^{\alpha\gamma} (x,y) \Big\}
\end{equation}
for $x\in \Omega$ and $y\in \partial\Omega$,
and $n=(n_1, \dots, n_d)$ denotes the outward unit normal to $\partial\Omega$.
In view of estimate (\ref{Lip-estimate-Green-function-3}), we have
\begin{equation}\label{3.4.3-1}
|P_\varep(x,y)|\le \frac{C \delta (x)}{|x-y|^d},
\end{equation}
and hence
\begin{equation}\label{3.4.3-2}
|u_\varep (x)|\le C\, \delta(x) \int_{\partial\Omega} \frac{|g(y)|}{|x-y|^d}\, dy\qquad \text{ for } x\in \Omega.
\end{equation}
It follows from (\ref{3.4.3-2}) that if $|x-z|< C_0\, \delta (x)$ for some $z\in \partial\Omega$,
\begin{equation}\label{3.4.3-3}
(u_\varep)^* (z)
\le C \mathcal{M}_{\partial\Omega} (g) (z),
\end{equation}
where
\begin{equation}\label{definition-of-HL-max-function}
\mathcal{M}_{\partial\Omega} (g) (z)
=\sup
\left\{ \average_{B(z,r)\cap\partial\Omega} |g|:\ 0<r<\text{diam}(\Omega)\right\}
\end{equation}
is the Hardy-Littlewood maximal function of $g$ on $\partial\Omega$.
Since
$$
\|\mathcal{M}_{\partial\Omega} (g)\|_{L^p(\partial\Omega)}
\le C_p\, \| g\|_{L^p(\partial\Omega)}
\qquad \text{ for } 1<p\le \infty,
$$
the desired estimate $\| \mathcal{M}(u_\varep)\|_{L^p(\partial\Omega)}
\le C_p\, \| g\|_{L^p(\partial\Omega)}$
follows readily from (\ref{3.4.3-3}).
\end{proof}

\begin{remark}[Agmon-Miranda maximum principle]\label{max-principle-remark}
{\rm
In the case $p=\infty$, Theorem \ref{nontangential-max-theorem-3} gives
\begin{equation}\label{max-principle}
\| u_\varep\|_{L^\infty(\Omega)}
\le C\, \| u_\varep \|_{L^\infty(\partial\Omega)},
\end{equation}
where $\mathcal{L}_\varep (u_\varep)=0$ in $\Omega$ and $C$ is independent of $\varep$.
In particular, let $u_\varep =\Phi_{\varep, j}^\beta -P_j^\beta-\varep\chi_j^\beta(x/\varep)$,
where $(\Phi_{\varep, j}^\beta)$ are the Dirichlet correctors for $\mathcal{L}_\varep$ in $\Omega$.
Then $\mathcal{L}_\varep (u_\varep)=0$ in $\Omega$
and $u_\varep =-\varep\chi_j^\beta(x/\varep)$
on $\partial\Omega$.
It follows from (\ref{max-principle}) that
$\| u_\varep\|_{L^\infty(\Omega)} \le C \,\varep$.
This yields that
\begin{equation}\label{Dirichlet-corrector-max}
\| \Phi_{\varep, j}^\beta -P_j^\beta \|_{L^\infty(\Omega)}
\le C \varep.
\end{equation}
}
\end{remark}

We end this section with some sharp H\"older estimates in $C^1$ domains
under the assumptions that $A$ is elliptic, periodic, and belongs to VMO$(\rd)$.

\begin{thm}[H\"older estimate]\label{Dirichlet-Holder-theorem}
Suppose that $A$ satisfies (\ref{ellipticity})-(\ref{periodicity}) and belongs to VMO$(\rd)$.
Let $\Omega$ be a bounded $C^1$ domain.
Let $u_\varep$  be the solution to the Dirichlet problem (\ref{DP-3}).
Then
\begin{equation}\label{estimate-3.4.4}
\|u_\varep\|_{C^\rho (\overline{\Omega})}
\le C \Big\{ \|g\|_{C^{\rho}(\partial\Omega)} +\| F\|_{W^{-1, p}(\Omega)} \Big\},
\end{equation}
where $\rho\in (0,1)$ and $\rho=1-\frac{d}{p}$.
\end{thm}

\begin{proof}
By supposition of solutions it suffices to consider two separate cases: (1) $g=0$;
(2) $F=0$.
In the first case we use the $W^{1, p}$ estimates in Theorem \ref{W-1-p-D-theorem} to
obtain $\|\nabla u_\varep\|_{L^p(\Omega)} \le C\, \| F\|_{W^{-1,p}(\Omega)}$.
By Sobolev imbedding this implies that 
$$
\| u_\varep\|_{C^\rho(\overline{\Omega})}
\le C\, \| F\|_{W^{-1, p}(\Omega)},
$$ 
where $\rho=1-\frac{d}{p}$.

To treat the second case, without loss of generality, we assume that $\| g\|_{C^\rho(\partial\Omega)}=1$. 
We need to show that
\begin{equation}\label{H-3-4}
|u_\varep (x) -u_\varep (y)|\le C\, | x-y|^\rho \qquad \text{ for any } x, y\in \Omega.
\end{equation}
To this end we choose a harmonic function $v$ in $\Omega$ such  that
$v=g$ on $\partial\Omega$.
Since $\Omega$ is $C^1$, it is known that such $v$ exists and is unique.
Moreover, $\| v\|_{C^\rho(\overline{\Omega})} \le C\, \| g\|_{C^\rho (\partial\Omega)}=C$.
Furthermore, by the boundary H\"older estimates for harmonic functions,
\begin{equation}\label{H-3-1}
|\nabla v(x)|\le C \big[ \delta (x) \big]^{\rho-1}.
\end{equation}
Let $w_\varep =u_\varep - v(x)$. Then $\mathcal{L}_\varep (w_\varep)=-\mathcal{L}_\varep (v_\varep)$
in $\Omega$ and $w_\varep =0$ on $\partial\Omega$.
Representing $w_\varep$ by
$$
w_\varep^\alpha (x)
=- \int_\Omega \frac{\partial}{\partial y_i} 
\bigg\{ G_\varep^{\alpha\beta} (x, y) \bigg\} a_{ij}^{\beta\gamma}(y/\varep)
\frac{\partial v^\gamma}{\partial y_j}\, dy,
$$
we obtain
\begin{equation}\label{H-3-2}
\aligned
| w_\varep (x)|
&\le C \int_\Omega |\nabla_y G_\varep (x, y) |\big[ \delta(y) \big]^{\rho-1}dy\\
&\le C \, \big[\delta(x)\big]^\rho,
\endaligned
\end{equation}
where we have used the estimate (\ref{H-3-1}) for the first inequality and Theorem \ref{G-theorem-1}
for the second. By Caccioppoli's inequality we also have
\begin{equation}\label{H-3-3}
\aligned
\left(\average_{B(x, \delta(x)/2)} |\nabla w_\varep|^2 \right)^{1/2}
& \le \frac{C}{\delta (x)}\left( \average_{B(x, 3\delta(x)/4)} |w_\varep|^2 \right)^{1/2}
+C \left(\average_{B(x, 3\delta(x)/4)} |\nabla v|^2\right)^{1/2}\\
&\le \big[\delta (x) \big]^{\rho-1},
\endaligned
\end{equation}
where we have used (\ref{H-3-2}) and (\ref{H-3-1}) for the last inequality.

Finally, to show (\ref{H-3-4}), 
we consider three subcases:
(i) $|y-x|\le \delta(x)/4$;
(ii) $|y-x|\le \delta(y)/4$;
(iii) $|y-x|> \delta(x)/4$ and $|y-x|>\delta(y)/4$.
In the first sub-case, we use  the interior H\"older estimates to obtain
$$
\aligned
|u_\varep (x)-u_\varep (y)|
& \le C\, r  \left(\frac{|x-y|}{r}\right)^\rho
\left(\average_{B(x, r)} |\nabla u_\varep|^2 \right)^{1/2}\\
&\le C\, |x-y|^\rho,
\endaligned
$$
where $r=\delta(x)/2$ and we have used (\ref{H-3-3}) and (\ref{H-3-1}).
The second case can be handled in the same manner.
In the third case we choose $x_0,y_0\in \partial\Omega$ so that
$|x-x_0|=\delta(x)$ and $|y-y_0|=\delta (y)$.
Note that
$$
|x_0-y_0|\le |x-x_0|+|x-y| +|y-y_0| \le 9\, |x-y|.
$$
Hence,
$$
\aligned
|u_\varep (x)- u_\varep (y)|
&\le |u_\varep(x)-g(x_0)| +|g(x_0)-g(y_0)| +|u_\varep (y)-g(y_0)|\\
&\le C\, \Big\{ |w_\varep(x)| +|v(x) -g(x_0)| +|x_0-y_0|^\rho +|w_\varep (y_0)| + |v(y_0)-g(y_0)|\Big\}\\
&\le C \, \Big\{\big[\delta (x)\big]^\rho +|x-y|^\rho +\big[\delta(y)\big]^\rho\Big\}\\
&\le C\, |x-y|^\rho,
\endaligned
$$
where we have used estimates (\ref{H-3-2}) and $\|v\|_{C^\rho(\overline{\Omega})}\le C$.
This finishes the proof.
\end{proof}



\section{Notes}

Under the assumption that $A$ is elliptic, periodic, and H\"older continuous,
the uniform boundary H\"older, $W^{1,p}$, and Lipschitz estimates with Dirichlet boundary conditions as well as 
nontangential-maximal-function estimates were due to M. Avellaneda and and F. Lin \cite{AL-1987}
(also see \cite{AL-1987-L-P, AL-1989, AL-1991}).
Our exposition on H\"older, Lipschitz, and nontangential-maximal-function estimates follows 
\cite{AL-1987} closely. The uniform $W^{1,p}$ estimates in $C^1$ domains for operators with periodic VMO coefficients
were established in \cite{Shen-2008}. See also \cite{GSS-2012, Geng-S-2015}.
Theorem \ref{G-theorem-1} on Green functions is taken from \cite{Shen-APDE-2015}.

%
%
%
%
%
%
%

\chapter{Regularity for Neumann Problem}\label{chapter-4}

In this chapter we study uniform regularity estimates for the Neumann boundary value problem,
\begin{equation}\label{Neumann-problem-4.0}
\left\{
\aligned
\mathcal{L}_\varep (u_\varep) &=F \quad \text{ in } \Omega,\\
\frac{\partial u_\varep}{\partial \nu_\varep} & =g \quad \ \text{ on } \partial\Omega,
\endaligned
\right.
\end{equation}
where $\mathcal{L}_\varep=-\text{div}\big(A(x/\varep)\nabla\big)$ and
$\frac{\partial u_\varep}{\partial \nu_\varep}$ denotes the conormal derivative of $u_\varep$, defined by
\begin{equation}\label{conormal-4.0}
\left( \frac{\partial u_\varep}{\partial \nu_\varep}\right)^\alpha
=n_i a_{ij}^{\alpha\beta} (x/\varep) \frac{\partial u^\beta_\varep}{\partial x_j}.
\end{equation}
Our approach is based on a general scheme for 
establishing regularity estimates at large scale in homogenization,
developed by S.N. Armstrong and C. Smart \cite{Armstrong-2016}
 in the study of stochastic homogenization.
Roughly speaking, the scheme states that if a function $u_\e$ is well approximated by $C^{1, \alpha}$
functions at every scale greater than $\e$, then $u_\e$ is Lipschitz continuous at every scale greater than $\e$.
The approach relies on certain (very weak) results 
on convergence rates and does not involve correctors in a direct manner.
In comparison with the compactness method used in Chapters \ref{chapter-2} and
\ref{chapter-3}, when applied to boundary Lipschitz estimates,
it does not require a-priori Lipschitz estimates for boundary correctors.

We start out in Section \ref{section-4.00} by establishing a result on the approximation of solutions of $\mathcal{L}_\e(u_\e)=F$
with (partial) Neumann data
by solutions of $\mathcal{L}_0(u)=F$  at large scale.
In Section \ref{section-4.1} we test the scheme in the simple case of
boundary H\"older estimates  in $C^1$ domains.
As in the case of Dirichlet condition,
boundary $W^{1,p}$ estimates in $C^1$ domains are
obtained in Section \ref{section-4.2} by combining the boundary H\"older estimates
with interior $W^{1,p}$ estimates in Chapter \ref{chapter-2}.
 In Section \ref{section-LN} we prove boundary Lipschitz estimates
 in $C^{1, \alpha}$ domains, assuming that $A$ is elliptic, periodic, and
H\"older continuous. Section \ref{section-4.4} is devoted to the study of the matrix of Neumann functions
$N_\e (x, y)$.
We obtain uniform size estimates for  $N_\varep (x,y)$ and
its derivatives $\nabla_x N_\varep (x,y)$, $\nabla_y N_\varep (x,y)$, $\nabla_x\nabla_y N_\varep (x,y)$.
 Throughout Sections \ref{section-4.00}-\ref{section-4.4} 
 we shall assume that $A$ satisfies the Legendre ellipticity condition 
 (\ref{s-ellipticity}).
 In Section \ref{N-EE} we discuss boundary estimates for elliptic systems
 of linear elasticity with Neumann conditions.

We will be working with local solutions with (partial) Neumann conditions.
Let $g\in L^2(B(x_0,r)\cap \partial\Omega;\br^m)$ and $F\in L^2(B(x_0, r)\cap\Omega; \br^m)$
for some $x_0\in \partial\Omega$ and $0<r<r_0$.
 We say $u_\varep \in H^1(B(x_0,r)\cap\Omega;\br^m)$ is a weak solution of 
 $$
 \mathcal{L}_\varep (u_\varep)=F \quad 
 \text{ in } B(x_0,r)\cap\Omega \quad
 \text{ and } \quad
 \frac{\partial u_\varep}{\partial \nu_\varep}=g \quad
 \text{ on } B(x_0,r)\cap \partial\Omega, 
 $$
 where $F\in L^2( B(x_0, r)\cap \Omega; \br^m)$ and $g\in L^2(B(x_0, r)\cap \partial\Omega; \br^m)$, if
 \begin{equation}\label{weak-solution-4.1}
 \int_{B(x_0,r)\cap \Omega}
 A(x/\varep) \nabla u_\e
\cdot  \nabla \varphi \, dx
 = \int_{B(x_0, r)\cap\Omega} F \cdot \varphi \, dx
 + \int_{B(x_0,r)\cap\partial\Omega} g\cdot \varphi\, d\sigma
 \end{equation}
 for any $\varphi\in C_0^\infty (B(x_0,r);\br^m)$.



\section{Approximation of solutions at large scale}\label{section-4.00}

Throughout this section we assume that $A$ is 1-periodic and satisfies (\ref{s-ellipticity}).
No smoothness condition on $A$ is needed. 
Let $D_r=D(r, \psi)$ and $\Delta_r=D(r, \psi)$ be defined by (\ref{definition-of-Delta}),
where $\psi: \br^{d-1}\to \br$ is a Lipschitz function such that $\psi(0)=0$ and $\|\nabla \psi\|_\infty
\le M_0$.

The goal of this section is to establish the following.

\begin{thm}[Approximation at large cale]\label{4.0-0-1}
Suppose that $A$ is 1-periodic and satisfies (\ref{s-ellipticity}).
Let $u_\e\in H^1(D_{2r}, \br^m)$ be a weak solution  to
$$
\mathcal{L}_\e (u_\e)=F\quad
\text{ in } D_{2r} \quad \text{ and } \quad
\frac{\partial u_\e}{\partial \nu_\e} =g \quad \text{ on } \Delta_{2r},
$$
where $F\in L^2(D_{2r}; \br^m)$ and $g\in L^2(\Delta_{2r}; \br^m)$.
Assume that $r\ge \e$.
Then there exists $w\in H^1(D_r; \br^m)$ such that
$$
\mathcal{L}_0 (w)=F\quad
\text{ in } D_r \quad \text{ and } \quad
\frac{\partial w}{\partial \nu_0} =g \quad \text{ on } \Delta_r,
$$
and
\begin{equation}\label{4.0-0-2}
\left(\average_{D_r} |u_\e -w|^2\right)^{1/2}
\le C \left(\frac{\e}{r} \right)^\alpha
\left\{ \left(\average_{D_{2r}} |u_\e|^2\right)^{1/2} 
+ r^2 \left(\average_{D_{2r}} |F|^2\right)^{1/2} 
+ r \left(\average_{\Delta_{2r}} |g|^2\right)^{1/2}\right\},
\end{equation}
where $C>0$ and $\alpha\in (0, 1/2)$ depend only on $\mu$ and $M_0$.
\end{thm}

We start with a Caccioppoli inequality for solutions with Neumann conditions.

\begin{lemma}[Caccioppoli's inequality]\label{c-n-lemma}
Suppose that $\mathcal{L}_\varep(u_\varep)=F$ in $D_{2r}$ and
$\frac{\partial u_\varep}{\partial\nu_\varep}=g$ on $\Delta_{2r}$.
Then
\begin{equation}\label{Cacciopoli-4.1}
\average_{D_{3r/2}}
|\nabla u_\varep|^2\, dx
\le \frac{C}{ r ^2}
\average_{D_{2r}} |u_\varep|^2\, dx +C\, r\average_{D_{2r}} |F|^2\, dx
+C\average_{\Delta_{2r}} |g|^2\, d\sigma,
\end{equation}
where $C$ depends only on $\mu$ and $M_0$.
\end{lemma}

\begin{proof}
By rescaling  we may assume that $r=1$.
We first consider the  special case $g=0$,
for which the proof is similar to that of (\ref{Cacciopoli-1.1}).
Choose $\psi \in C_0^\infty(\br^d)$ such that $0\le \psi \le1$,
$\psi=1$ in $D_{3/2}$ and $\psi=0$ in $D_2\setminus D_{7/4}$.
Note that
$$
\int_{D_2} A(x/\e)\nabla u_\e \cdot \nabla (\psi^2 u_\e)\, dx=\int_{D_2} F \cdot \psi^2 u_\e\, dx.
$$
This, together with the equation (\ref{Ca-Id}), leads to 
$$
\int_{D_{3/2}} |\nabla u_\e|^2\, dx \le C \int_{D_2} |u_\e|^2\, dx + C \int_{D_2} |F|^2\, dx,
$$
by Cauchy inequality (\ref{Cauchy}).
The general case may be reduced to the special case by considering $u_\varep -v_\varep$, where
$\mathcal{L}_\varep (v_\varep)=F$ in $D_2$, $\frac{\partial v_\varep}{\partial \nu_\varep} =\widetilde{g}$
on $\partial D_2$, and $\int_{D_2} v_\varep \, dx=0$.
Here $\widetilde{g}=g$ on $\Delta_2$ and $\widetilde{g}$ is a constant on $\partial D_2\setminus \Delta_2$,
chosen so that $\int_{\partial D_2} \widetilde{g} \, d\sigma +\int_{D_2} F\, dx=0$.
Note that by Theorem \ref{s-NP-theorem},
 $$
 \aligned
 \|v_\varep\|_{H^1(D_2)}  &\le C\,\Big\{  \| \widetilde{g}\|_{L^2(\partial D_2)}
 +\| F\|_{L^2(D_2)} \Big\}\\
 & \le C\Big\{ \| g\|_{L^2(\Delta_2)} + \|F\|_{L^2(D_2)}\Big\}.
 \endaligned
 $$
 It follows that
 $$
 \aligned
 \int_{D_{3/2}} |\nabla u_\e|^2 
 &\le 2 \int_{D_{3/2}} |\nabla v_\e|^2 + 2 \int_{D_{3/2}} |\nabla (u_\e -v_\e)|^2\\
 &\le C \| g\|^2_{L^2(\Delta_2)} 
 + C \int_{D_2} |u_\e -v_\e|^2 + C \| F\|^2_{L^2(D_2)}\\
 &\le C \int_{D_2} |u_\e|^2 +  C\| F\|^2_{L^2(D_2)} + C \| g\|^2_{L^2(\Delta_2)},
 \endaligned
 $$
 where we have used (\ref{Cacciopoli-4.1}) for the special case.
 \end{proof}
 
 \begin{remark}\label{c-n-remark}
 {\rm 
 Suppose that $\mathcal{L}_\e (u_\e)=0$ in $D_{2r}$ and $\frac{\partial u_\e}{\partial \nu_\e}=0$
 on $\Delta_{2r}$.
 It follows from the proof of Lemma \ref{c-n-lemma} that
 $$
 \int_{D_{sr}} |\nabla u_\e|^2\, dx \le \frac{C}{(t-s)^2 r^2}\int_{D_{tr}}
 | u_\e -E |^2\, dx,
 $$
 where $E\in \br^m$, $1<s<t<2$, and $C$ depends only on $\mu$ and $M_0$.
 By Poincar\'e-Sobolev inequality we obtain 
 $$
 \left(\average_{D_{sr}} |\nabla u_\e|^2\right)^{1/2}
 \le \frac{C}{(t-s)^2} 
 \left(\average_{D_{tr}} |\nabla u_\e|^q\right)^{1/q},
 $$
 where $\frac{1}{q}=\frac12 +\frac{1}{d}$.
 As in the interior case, by a real-variable argument (see  \cite[Theorem 6.38]{Gia-Ma-book}),
this leads to  the reverse H\"older inequality,
 \begin{equation}\label{RH-N}
  \left(\average_{D_{r}} |\nabla u_\e|^{\bar{p}}\right)^{1/\bar{p}}
 \le C
 \left(\average_{D_{2r}} |\nabla u_\e|^2\right)^{1/2},
 \end{equation}
 where $C>0$ and $\bar{p}>2$ depend only on $\mu$ and $M_0$.
 Note that the periodicity of $A$ is not needed for (\ref{Cacciopoli-4.1}) and (\ref{RH-N}).
 It is also not needed for the next lemma, which gives a Meyers estimate \cite{Meyers-1963}.
 }
 \end{remark}
 
 \begin{lemma}\label{4.0-1-00}
 Let $\Omega=D_r$ for some $1\le r\le 2$.
Let $u_\e\in H^1(\Omega; \br^m)$ $(\e\ge 0)$ be a weak solution to the Neumann problem,
$\mathcal{L}_\e (u_\e)=F$ in $\Omega$ and $\frac{\partial u_\e}{\partial \nu_\e}=g$ on 
$\partial\Omega$, where $F\in L^2(\Omega; \br^m)$, $g\in L^2(\partial\Omega; \br^m)$
and $\int_\Omega F\, dx +\int_{\partial\Omega} g\, d\sigma =0$.
Then there exists $p>2$, depending only on $\mu$ and $M_0$, such that
\begin{equation}\label{4.0-1-001}
\|\nabla u_\e\|_{L^p(\Omega)}
\le C \Big\{ \| F\|_{L^2(\Omega)} + \| g\|_{L^2(\partial \Omega)} \Big\},
\end{equation}
where $C>0$ depend only on $\mu$ and $M_0$.
 \end{lemma}
  
  \begin{proof}
  Consider the Neumann problem: $\mathcal{L}_\e (v_\e)=\text{div} (f) $ in $\Omega$ and
  $\frac{\partial v}{\partial \nu_\e} =0$ on $\partial \Omega$, where $f\in C_0^1(\Omega; \br^{m\times d})$.
  Clearly, $\|\nabla v_\e \|_{L^2(\Omega)} \le C \| f\|_{L^2(\Omega)}$.
  In view of the reverse H\"older inequality (\ref{RH-N}),
  using the real-variable argument in Section \ref{real-variable-section},
  one may deduce that
  \begin{equation}\label{per-N0}
  \|\nabla v_\e \|_{L^p(\Omega)} \le C \| f\|_{L^p(\Omega)},
  \end{equation}
  for any $2<p<\bar{p}$.
  Since this is also true for the adjoint operator $\mathcal{L}_\e^*$, by a duality argument 
  (see Remark \ref{W-1-p-100}), it follows that 
  \begin{equation}\label{4.0-1-002}
  \|\nabla u_\e\|_{L^p(\Omega)}
  \le C \Big\{ \| F\|_{L^q(\Omega)}
  +\| g\|_{B^{-1/p, p}(\partial\Omega)} \Big\},
  \end{equation}
  for $2<p<\bar{p}$, where $\frac{1}{q}=\frac{1}{p}+\frac{1}{d}$ and
  $B^{-1/p, p}(\partial\Omega)$ denotes the dual of the Besov space $B^{1/p, p^\prime}(\partial\Omega)$.
  By choosing $p\in (2, \bar{p})$  close to $2$ so that $L^2(\Omega)\subset L^q(\Omega)$ and
  $L^2(\partial\Omega)\subset B^{-1/p, p}(\partial\Omega)$, we see that the estimate (\ref{4.0-1-001})
  follows from (\ref{4.0-1-002}).
  \end{proof}
  
  It follows from Theorem \ref{main-thm-2.3} that 
  \begin{equation}\label{sr-n}
  \| u_\e -u_0\|_{L^2(\Omega)}
  \le C \sqrt{\e} 
  \Big\{ \| F \|_{L^2(\Omega)} + \| g\|_{L^2(\partial\Omega)} \Big\},
  \end{equation}
  where $\Omega$ is a bounded Lipschitz domain.
  However, this was proved under the symmetry condition $A^*=A$.
  The next lemma  provides a very weak rate of convergence without the symmetry condition.
  
\begin{lemma}\label{4.0-1-0}
Let $u_\e\in H^1(\Omega; \br^m)$ $(\e\ge 0)$ be the weak solution to the Neumann problem 
(\ref{Neumann-problem-4.0}) with $\int_\Omega u_\e\, dx =0$,
where $\Omega=D_r$ for some $1\le r\le 2$.
Then 
\begin{equation}\label{4.0-1-1}
\| u_\e -u_0\|_{L^2(\Omega)}
\le C \e^\sigma \Big\{ 
 \| F\|_{L^2(\Omega)} +\| g\|_{L^2(\partial\Omega)} \Big\}
\end{equation}
 for  any $0< \e<2$, where $\sigma>0$ and $C>0$ depend only on $\mu$ and $M_0$.
\end{lemma}

\begin{proof}
It follows from (\ref{thm-2.3-1a}) that 
\begin{equation}\label{4.0-1-2}
\| u_\e -u_0\|_{L^2(\Omega)}
\le C \Big\{ \e \|\nabla^2 u_0\|_{L^2(\Omega\setminus \Omega_\e)}
+\e \|\nabla u_0\|_{L^2(\Omega)}
+\|\nabla u_0\|_{L^2(\Omega_{5\e})} \Big\},
\end{equation}
where $\Omega_t=\{ x\in \Omega: \text{\rm dist}(x, \partial\Omega)<t \}$.
To bound the RHS of (\ref{4.0-1-2}), we first use  interior estimates for $\mathcal{L}_0$
to obtain
\begin{equation}\label{4.0-1-3}
\average_{B(x, \delta(x)/8)} |\nabla^2 u_0|^2
\le \frac{C}{ [\delta(x)]^2}
\average_{B(x, \delta(x)/4)} |\nabla u_0|^2
+ C \average_{B(x, \delta(x)/4)} |F|^2
\end{equation}
for any $x\in \Omega$, where $\delta(x)=\text{dist}(x, \partial\Omega)$.
We then integrate both sides of (\ref{4.0-1-3}) over the set $\Omega\setminus\Omega_\e$.
Observe that if $|x-y|<\delta(x)/4$, then
$|\delta (x)-\delta (y)|\le |x-y|<\delta(x)/4$, which gives
$$
(4/5) \delta (y)< \delta (x) < (4/3) \delta (y).
$$
It follows that
$$
\aligned
\int_{\Omega\setminus \Omega_\e} |\nabla^2 u_0(y)|^2\, dy
 &\le  C\int_{\Omega\setminus \Omega_{\e/2}} [\delta(y)]^{-2} |\nabla u_0 (y)|^2\, dy
+ C\int_\Omega |F(y)|^2\, dy\\
&\le \left(\int_\Omega |\nabla u_0|^{2s}\, dy\right)^{1/s}
\left(\int_{\Omega\setminus \Omega_{\e/2}} [\delta (y)]^{-2s^\prime}\, dy\right)^{1/s^\prime}
+ C\int_\Omega |F(y)|^2\, dy\\
&\le \e^{-1-\frac{1}{s}} \|\nabla u_0\|_{L^{2s}(\Omega)}^2
++ C\int_\Omega |F(y)|^2\, dy,
\endaligned
$$
where $s>1$ and we have used H\"older's inequality for the second step.
Let $p=2s>2$. We see that
\begin{equation}\label{4.0-1-4}
\aligned
 \e \|\nabla^2 u_0\|_{L^2(\Omega\setminus \Omega_\e)}
 & \le C \e^{\frac12-\frac{1}{p}} \| \nabla u_0\|_{L^p (\Omega)}
+C \e \| F\|_{L^2(\Omega)}\\
&\le C \e^{\frac{1}{2}-\frac{1}{p}} 
\Big\{ \| F\|_{L^2(\Omega)} +\| g\|_{L^2(\partial\Omega)} \Big\},
\endaligned
\end{equation}
where we have used (\ref{4.0-1-001})  for the last inequality.
Also, by H\"older's inequality, 
\begin{equation}\label{4.0-1-5}
\aligned
\|\nabla u_0\|_{L^2(\Omega_{5\e})}
 & \le C \e^{\frac12-\frac{1}{p}} \|\nabla u_0\|_{L^p(\Omega)}\\
 & \le C \e^{\frac12-\frac{1}{p}} 
 \Big\{ \| F\|_{L^2(\Omega)} +\| g\|_{L^2(\partial\Omega)} \Big\}.
\endaligned
\end{equation}
The inequality (\ref{4.0-1-1}) with $\sigma=\frac12-\frac{1}{p}>0$
  follows from (\ref{4.0-1-2}), (\ref{4.0-1-4}) and (\ref{4.0-1-5}).
\end{proof}

We are now in a position to give the proof of Theorem \ref{4.0-0-1}.

\begin{proof}[\bf Proof of Theorem \ref{4.0-0-1}]

By rescaling we may assume that $r=1$.
It follows from the Caccioppoli inequality (\ref{Cacciopoli-4.1}) and the co-area formula that
there exists some $t\in (1, 3/2)$ such that
\begin{equation}\label{4.0-2-0}
\int_{\partial D_t \setminus \Delta_2} 
|\nabla u_\e|^2 \, d\sigma \le C \int_{D_2} |u_\e|^2\, dx
+C \int_{D_2} |F|^2\, dx
+ C \int_{\Delta_2} |g|^2\, d\sigma.
\end{equation}
For otherwise, we could integrate the reverse inequality,
$$
\int_{\partial D_t \setminus \Delta_2} 
|\nabla u_\e|^2 \, d\sigma > C \int_{D_2} |u_\e|^2\, dx
+C \int_{D_2} |F|^2\, dx
+ C \int_{\Delta_2} |g|^2\, d\sigma,
$$
in $t$ over the interval $(1, 3/2)$ to obtain 
$$
\int_{D_{3/2}\setminus D_1}
|\nabla u_\e|^2 \, dx > C \int_{D_2} |u_\e|^2\, dx
+C \int_{D_2} |F|^2\, dx
+ C \int_{\Delta_2} |g|^2\, d\sigma,
$$
which is in contradiction with (\ref{Cacciopoli-4.1}).
We now let $w$  be the unique solution of the Neumann problem:
$$
\mathcal{L}_0 (w)=F \quad  \text{ in } D_t \quad\text{ and } \quad
\frac{\partial w}{\partial \nu_0} =\frac{\partial u_\e}{\partial \nu_\e} \quad \text{ on } \partial D_t,
$$
with $\int_{D_t} w\, dx=\int_{D_t} u_\e \, dx$.
Note that $\frac{\partial w}{\partial\nu_0} =g$ on $\Delta_1$, and that
by Lemma \ref{4.0-1-0},
$$
\aligned
\| u_\e -w\|_{L^2(D_1)}
&\le \| u_\e -w\|_{L^2(D_t)}\\
&\le C \e^\sigma \Big\{ \| F\|_{L^2(D_2)}
+ \| g\|_{L^2(\Delta_2)}
+\| \nabla u_\e\|_{L^2(\partial D_t\setminus \Delta_2)} \Big\},
\endaligned
$$
which, together with (\ref{4.0-2-0}), yields (\ref{4.0-0-2}).
\end{proof}



\section{Boundary H\"older estimates}\label{section-4.1}

In this section we establish uniform boundary H\"older estimates in $C^1$ domains
for $\mathcal{L}_\varep$ with  Neumann conditions. 
Throughout this section we assume that $D_r=D(r, \psi)$, $\Delta_r=\Delta (r, \psi)$, and
$\psi:\br^{d-1}\to \br$ is a $C^1$ function satisfying (\ref{C-1}).

\begin{thm}[H\"older estimate at large scale]\label{bh-n-l}
Suppose that $A$ is  1-periodic and satisfies the ellipticity condition (\ref{s-ellipticity}).
Let $u_\e \in H^1(D_2; \br^m)$ be a weak solution of $\mathcal{L}_\e (u_\e)=0$ in $D_2$ and
$\frac{\partial u_\e}{\partial \nu_\e}=g$ on $\Delta_2$. Then, for any $\rho\in (0,1)$ and
$\e\le r\le 1$,
\begin{equation}\label{4.1-1-0} 
\left(\average_{D_r} |\nabla u_\e|^2\right)^{1/2}
\le C r^{\rho-1}
\left\{ \left(\average_{D_2} |u_\e|^2\right)^{1/2} +\| g\|_\infty \right\},
\end{equation}
where $C$ depends only on $\rho$, $\mu$, and
$(\omega_1 (t), M_0)$ in (\ref{C-1}).
\end{thm}

\begin{proof}
Fix $\gamma\in (\rho, 1)$.
For each $r\in [\e, 1]$, let $w=w_r$ be the function given by Theorem \ref{4.0-0-1}.
By boundary $C^\gamma $ estimates in $C^1$ domains for the operator $\mathcal{L}_0$,
$$
\aligned
\inf_{q\in \br^m} \left(\average_{D_{\theta r}} |w-q|^2\right)^{1/2}
& \le  (\theta r)^\gamma \| w\|_{C^{0, \gamma}(D_{\theta r})}\\
&\le  (\theta r)^\gamma \| w\|_{C^{0, \gamma}(D_{r/2})}\\
&\le C_0 \theta^\gamma
\inf_{q\in \br^m} \left(\average_{D_{ r}} |w-q|^2\right)^{1/2} + C r \| g\|_\infty
\endaligned
$$
for any $\theta\in (0,1/2)$,
where $C_0$ depends only on $\mu$, $\gamma$ and $(\omega_1(t), M_0)$.
This, together with Theorem \ref{4.0-0-1}, gives
$$
\aligned
\inf_{q\in \br^m} \left(\average_{D_{\theta r}} |u_\e-q|^2\right)^{1/2}
&\le  \inf_{q\in \br^m} \left(\average_{D_{\theta r}} |w-q|^2\right)^{1/2}
+\left(\average_{D_{\theta r}} |w_\e -w|^2\right)^{1/2}\\
&\le C_0 \theta^\gamma
\inf_{q\in \br^m} \left(\average_{D_{ r}} |w-q|^2\right)^{1/2}
+C_\theta \left(\average_{D_r} |u_\e -w|^2\right)^{1/2} + Cr \| g\|_\infty\\
&\le C_0 \theta^\gamma
\inf_{q\in \br^m} \left(\average_{D_{ r}} |u_\e -q|^2\right)^{1/2}
+C_\theta \left(\average_{D_r} |u_\e -w|^2\right)^{1/2} + Cr \| g\|_\infty\\
&\le C_0 \theta^\gamma
\inf_{q\in \br^m} \left(\average_{D_{ r}} |u_\e -q|^2\right)^{1/2}
+C_\theta \left(\frac{\e}{r}\right)^\alpha
\left(\average_{D_{2r}} |u_\e|^2\right)^{1/2} + Cr \| g\|_\infty,
\endaligned
$$
where the constant $C_\theta$ also depends on $\theta$.
Replacing $u_\e$ with $u_\e-q$, we obtain 
\begin{equation}\label{4.1-1-1}
\phi(\theta r) \le C_0 \theta^{\gamma-\rho} \phi (r)
+ C_\theta \left(\frac{\e}{r}\right)^\alpha \phi (2r) + C_\theta r^{1-\rho} \| g\|_\infty
\end{equation}
for any $r\in [\e, 1]$, where
$$
\phi (r)=r^{-\rho} \inf_{q\in \br^m}
\left(\average_{D_r} |u_\e -q|^2\right)^{1/2}.
$$
We now choose $\theta\in (0,1/2)$ so small that 
$$
C_0 \theta^{\gamma-\rho} <(1/4),
$$
which is possible, since $\gamma-\rho>0$.
With $\theta$ fixed, we choose $N>1$ so large that
$$
C_\theta  N^{-\alpha} <(1/4).
$$
It follows that if $1\ge r\ge N\e$,
\begin{equation}\label{4.1-1-2}
\phi(\theta r) \le \frac{1}{4} \big\{ \phi (r) +\phi(2r)\big\} + Cr^{1-\rho} \| g\|_\infty.
\end{equation}

Finally, we divide both sides of (\ref{4.1-1-2})  by $r$ and integrate the resulting inequality in $r$ over the interval $(N\e, 1)$.
This gives
$$
\aligned
\int_{\theta N\e}^\theta \phi(r)\frac{dr}{r}
& \le \frac14 \int_{N\e}^1 \phi (r) \frac{dr}{r}
+\frac14 \int_{2N\e}^2 \phi (r) \frac{dr}{r} + C \| g\|_\infty\\
&\le \frac12 \int_{\theta N\e}^2 \phi (r) \frac{dr}{r} + C \| g\|_\infty.
\endaligned
$$
It follows that
\begin{equation}\label{4.1-1-3}
\int_{\theta N\e}^\theta \phi(r)\frac{dr}{r}
 \le 2 \int^{2 }_\theta \phi(r)\frac{dr}{r} + C \| g\|_\infty.
\end{equation}
Using the observation that $\phi(r)\le C \phi (t)$ for $t\in [r, 2r]$,  we may deduce from (\ref{4.1-1-3}) that
$$
\phi( r)\le  C \phi (2) + C \| g\|_\infty
$$
 for any $r\in [\e, 2]$.
The estimate (\ref{4.1-1-0}) now follows by Caccioppoli's inequality.
\end{proof}

No smoothness condition on $A$ is needed for (\ref{4.1-1-0}). 
Under the additional VMO condition (\ref{VMO-1}), we obtain the following.

\begin{thm}[Boundary H\"older estimate]\label{bh-n}
Suppose that $A$ is 1-periodic and satisfies (\ref{s-ellipticity}).
Also assume that $A$ satisfies  (\ref{VMO-1}).
Let $\Omega$ be a bounded $C^1$ domain and $0<\rho<1$.
Let $u_\varep \in H^1(B(x_0,r)\cap\Omega;\br^m)$ be a solution of
$\mathcal{L}_\varep (u_\varep)=0$ in $B(x_0, r)\cap\Omega$
with $\frac{\partial u_\varep}{\partial\nu_\varep} =g$
on $B(x_0,r)\cap\partial\Omega$
for some $x_0\in \partial\Omega$ and $0<r<r_0$. 
Then, for any $x,y\in B(x_0, r/2)\cap \Omega$,
\begin{equation}\label{local-holder-Neumann-estimate}
|u_\varep (x)-u_\varep (y)|
\le
C \left(\frac{|x-y|}{r}\right)^\rho
\left\{ \left(\average_{B(x_0,r)\cap\Omega} |u_\varep|^2\right)^{1/2} + r \| g\| \right\},
\end{equation}
where $C$ depends at most on $\rho$,
$\mu$, $\omega (t)$ in (\ref{VMO-1}), and $\Omega$.
\end{thm}

\begin{proof}
By a change of the coordinate system it suffices to show that
\begin{equation}
\label{boundary-holder-Neumann-estimate}
|u_\varep (x)-u_\varep (y)|
\le C \left(\frac{|x-y|}{r}\right)^\rho
\left\{ \left(\average_{D_{2r}}
|u_\varep |^2\right)^{1/2} + r\| g\|_\infty \right\}
\end{equation}
for any $x, y\in D_r$, where $0<r\le 1$, $\mathcal{L}_\e (u_\e)=0$ in $D_{2r}$ and
$\frac{\partial u_\e}{\partial\nu_\e}=g$ on $\Delta_{2r}$.

 By rescaling we may assume that $r=1$. We may also assume that
$\varep< 1$, since the case $\varep\ge 1$ follows readily
from  boundary H\"older estimates in $C^1$ domains for elliptic systems with VMO coefficients
(see e.g. \cite{AP-2002, BW-2005}).
 Under these assumptions we will show that
\begin{equation}\label{4.1.3-10}
\left(\average_{D_t}
\big|u_\varep-\average_{D_t} u_\varep\big |^2\right)^{1/2} \le C\, t^{\rho} 
\Big\{ \| u_\e\|_{L^2(D_2)} +\| g\|_\infty \Big\}
\end{equation}
for any $t\in (0,1/4)$. As in the case of the Dirichlet  condition,
the desired estimate  follows from 
the interior H\"older estimate  and (\ref{4.1.3-10}),
using Campanato's characterization of H\"older spaces.

To prove (\ref{4.1.3-10}) we first consider the case $t\ge \varep$.
It follows from Poincar\'e inequality and Theorem \ref{bh-n-l} that
$$
\aligned
\left(\average_{D_t}
\big|u_\varep -\average_{D_t } u_\varep \big|^2\right)^{1/2}
& \le  C t\left( \average_{D_t}
|\nabla u_\e|^2\right)^{1/2}\\
&\le C\, t^{\rho} \Big\{  \| u_\e\|_{L^2(D_2)} +\| g\|_\infty \Big\}.
\endaligned
$$

Next, suppose that $t<\e$.
Let $w(x)=u_\varep (\varep x)$. Then
$\mathcal{L}_1 (w)=0$ in $D(1, \psi_\varep)$, where
$\psi_\varep (x^\prime)=\varep^{-1} \psi(\varep x^\prime)$.
By the boundary H\"older estimates in $C^1$ domains for $\mathcal{L}_1$, we obtain
$$
\aligned
\bigg(\average_{D(t,\psi)}   \big|u_\varep -\average_{D(t,\psi)} u_\varep\big |^2\bigg)^{1/2}
&=\left(\average_{D(\frac{t}{\varep}, \psi_\varep)}
\big|w-\average_{D(\frac{t}{\varep}, \psi_\varep)} w\big|^2\right)^{1/2}\\
&\le C\left(\frac{t}{\varep}\right)^{\rho}
\left\{\left(\average_{D(1, \psi_\varep)}
\big|w-\average_{D(1, \psi_\varep)} w\big|^2\right)^{1/2} +\| g\|_\infty \right\} \\
& = C\left(\frac{t}{\varep}\right)^{\rho}
\left\{ \left(\average_{D(\e, \psi)}
\big|u_\varep-\average_{D(\e, \psi)} u_\varep\big|^2\right)^{1/2} +\| g\|_\infty \right\}\\
&\le  C\left(\frac{t}{\varep}\right)^{\rho}
\e^\rho \Big\{ \| u_\e\|_{L^2(D_2)} +\| g\|_\infty \Big\}\\
& =C\,  t^{\rho}  \Big\{ \| u_\e\|_{L^2(D_2)} + \| g\|_\infty \Big\},
\endaligned
$$
where the last inequality follows from the previous case with $t=\e $.
This finishes the proof of (\ref{4.1.3-10}) and thus of Theorem \ref{bh-n}.
\end{proof}



\section{Boundary $W^{1,p}$ estimates}\label{section-4.2}

In this section we establish the uniform $W^{1,p}$ estimates for the Neumann problem:
\begin{equation}\label{Neumann-problem-4.2}
\left\{
\aligned
\mathcal{L}_\varep (u_\varep) & = \text{\rm div} (f) +F &\ &\text{ in } \Omega,\\
\frac{\partial u_\varep}{\partial \nu_\varep} & =g-n\cdot f &\ &  \text{ on } \partial\Omega.
\endaligned
\right.
\end{equation}
Throughout the section we assume that $\Omega$ is $C^1$ and that
$A$ satisfies (\ref{s-ellipticity}), (\ref{periodicity}) and (\ref{VMO-1}).

Let $B^{-\frac{1}{p}, p}(\partial\Omega; \mathbb{R}^m)$ 
denote the dual space of Besov space $B^{\frac{1}{p}, p^\prime}(\partial\Omega; \mathbb{R}^m)$ for $1<p<\infty$.

\begin{definition}
{\rm
We call $u_\varep\in W^{1,p}(\Omega;\br^m)$ a weak solution of (\ref{Neumann-problem-4.2}), if
\begin{equation}\label{weak-solution-Neumann}
 \int_\Omega A(x/\e)\nabla u_\e \cdot \nabla \phi\, dx \\
=\int_\Omega
\left\{ -f_i^\alpha \frac{\partial \phi^\alpha}{\partial x_i} +F^\alpha\phi^\alpha\right\}\, dx
+\langle g, \phi\rangle_{B^{-1/p, p}(\partial\Omega)\times B^{1/p, p^\prime}(\partial\Omega)}
\end{equation}
for any $\phi=(\phi^\alpha)\in C_0^\infty(\br^d;\br^m)$, where $f=(f_i^\alpha)$ and $F=(F^\alpha)$.
}
\end{definition}

\begin{thm}\label{W-1-p-Neumann-theorem}
Let $1<p<\infty$.
Let $g=(g^\alpha)\in B^{-1/p, p}(\partial\Omega;\br^m)$, $f=(f_i^\alpha)\in L^p(\Omega;\br^{m\times d})$, and
$F=(F^\alpha)\in L^q(\Omega;\br^m)$, where $q=\frac{pd}{p+d}$ for $p>\frac{d}{d-1}$, 
$q>1$ for $p=\frac{d}{d-1}$, and $q=1$ for $1<p<\frac{d}{d-1}$. Then, if $F$ and $g$ satisfy the compatibility condition 
\begin{equation}\label{compatibility-Neumann}
\int_\Omega F\cdot b \, dx +\langle g, b\rangle_{B^{-1/p,p}(\partial\Omega)\times B^{1/p, p^\prime}(\partial\Omega)}
=0 
\end{equation}
for any $b\in \br^m$,
the Neumann problem (\ref{Neumann-problem-4.2})
has a unique (up to constants) weak solution. Moreover, the solution 
satisfies the estimate,
\begin{equation}\label{estimate-4.2}
\|\nabla u_\varep \|_{L^p(\Omega)}
\le C_p \, \Big\{ \| f\|_{L^p(\Omega)} +\| F \|_{L^q(\Omega)} + \| g\|_{B^{-1/p, p}(\partial\Omega)}\Big\},
\end{equation}
where $C_p>0$ depends only on $p$, $\mu$, $\omega(t)$ in (\ref{VMO-1}), and $\Omega$.
\end{thm}

The proof of Theorem \ref{W-1-p-Neumann-theorem} is divided into several steps.

\begin{lemma}\label{lemma-4.2-1}
Let $f\in L^p(\Omega;\br^{m\times d})$ for some $1<p<\infty$. Then there exists a unique (up to constants)  $u_\varep$
in $W^{1,p}(\Omega;\br^m)$ such that $\mathcal{L}_\varep (u_\varep)=\text{\rm div}(f)$ and $\frac{\partial u_\varep}{\partial\nu_\varep}
=-n\cdot f$ on $\partial\Omega$.
Moreover, the solution satisfies the estimate,
\begin{equation}\label{estimate-4.2-1}
\|\nabla u_\varep\|_{L^p(\Omega)} \le C_p\, \| f\|_{L^p(\Omega)},
\end{equation}
where $C_p$ depends only on $p$, $\mu$, $\omega (t)$ in (\ref{VMO-1}), and $\Omega$.
\end{lemma}

\begin{proof}
The proof is parallel to that of the same estimate for the Dirichlet problem (see Section \ref{section-3.00}).
Let $p>2$.
We first observe that if $\mathcal{L}_\varep (u_\varep)=\Omega\cap 2B$, 
$\frac{\partial u_\varep}{\partial\nu_\varep}=0$ in $\partial\Omega\cap 2B$,
where $B$ is a ball in $\br^d$ with the property that $|B|\le c_0|\Omega|$ and either
$2B\subset \Omega$ or $B$ is centered on $\partial\Omega$, then
\begin{equation}\label{4.2-1-1}
\left(\average_{\Omega\cap B} |\nabla u_\varep|^p\right)^{1/p}
\le C \left(\average_{\Omega\cap 2B} |\nabla u_\varep|^2\right)^{1/2}.
\end{equation}
The interior case where
$2B\subset\Omega$ follows directly from the interior $W^{1,p}$ estimate in Section \ref{section-2.4}.
To handle the boundary case where $x_0 \in \partial\Omega$,
one uses a line of argument similar to that used in the Dirichlet problem,
by combining  the interior $W^{1,p}$ estimates with the boundary H\"older estimate in Theorem \ref{bh-n}.
More precisely, let $B=B(x_0, r)$, where $x_0\in \partial\Omega$ and $0<r<r_0$.
Note that if $x\in B(x_0, r)$,
\begin{equation}\label{r-h-n-0}
\aligned
\average_{B(x,\delta(x)/8)} |\nabla u_\e|^p
& \le \frac{C}{[\delta(x)]^p}  \left(\average_{B(x,\delta(x)/4)} \Big| u_\e
-\average_{B(x, \delta(x)/4)} u_\e \Big|^2\right)^{p/2}\\
&\le \frac{C}{[\delta(x)]^p} 
\left(\frac{\delta (x)}{r}\right)^{\rho p} 
\left(\average_{\Omega\cap 2B} |u_\e|^2 \right)^{p/2},
\endaligned
\end{equation}
where $\rho\in (0,1)$ is close to $1$ so that $p(1-\rho)<1$.
By integrating both sides of (\ref{r-h-n-0}) over the set $B\cap \Omega$, we obtain 
$$
\int_{\Omega\cap B} |\nabla u_\e|^p
\le C r^{d-p} \left(\average_{\Omega\cap 2B} |u_\e|^2 \right)^{p/2}.
$$
Let  $E$ be the $L^1$ average of $u_\e$ over the set $\Omega\cap 2B$.
By replacing $u_\e$ with $u_\e -E$ in the inequality above and applying Poincar\'e inequality,
we obtain (\ref{4.2-1-1}).

With the reverse H\"older inequality (\ref{4.2-1-1}) at our disposal,  we may deduce the estimate 
(\ref{estimate-4.2-1}) by Theorem \ref{real-variable-operator-Lipschitz-theorem}.
 Indeed, for $f\in L^2(\Omega;\br^{m\times d})$, let $T_\varep (f)=\nabla u_\varep$, 
 where $u_\varep\in W^{1,2}(\Omega;\br^m)$
 is the unique weak solution to $\mathcal{L}_\varep (u_\varep)=\text{div}(f)$ in $\Omega$,
 $\frac{\partial u_\varep}{\partial\nu_\varep}=-n\cdot f$ on $\partial\Omega$, and $\int_\Omega u_\varep \, dx =0$.
 Clearly, 
 $$
 \|T_\varep (f)\|_{L^2(\Omega)} \le C\, \| f\|_{L^2(\Omega)},
 $$
  where $C$ depends only on $\mu$
 and $\Omega$.
Suppose now that $f=0$ in $\Omega\setminus  2B$, where $|B|\le c_0|\Omega|$ and either $2B\subset \Omega$ or
$B$ is centered on $\partial\Omega$. Then $\mathcal{L}_\varep(u_\varep)=0$ in $\Omega\cap 2B$ and
$\frac{\partial u_\varep}{\partial\nu_\varep}=0$ on $\partial\Omega\cap 2B$.
In view of (\ref{4.2-1-1}) we obtain
\begin{equation}\label{4.2-1-5}
\left(\average_{\Omega\cap B} | T_\varep (f)|^p\right)^{1/p}
\le C \left(\average_{\Omega\cap 2B} | T_\varep (f)|^2\right)^{1/2}.
\end{equation}
As a result, the operator $T_\varep$ satisfies the assumptions in Theorem \ref{real-variable-operator-Lipschitz-theorem}
with constants depending at most on $\mu$, $p$, $\omega (t)$ in (\ref{VMO-1}), and $\Omega$. 
This gives (\ref{estimate-4.2-1}) for $2<p<\infty$.
Note that the uniqueness for $p>2$ follows from the uniqueness for $p=2$.

The case $1<p<2$ may be handled by duality. Let $g=(g_i^\alpha)
\in C_0^1 (\Omega;\br^{m\times d})$ and $v_\varep$ be a weak solution in $W^{1,2}(\Omega;\br^m)$
of  $\mathcal{L}_\varep^* (v_\varep) =\text{div}(g)$ in $\Omega$ and $\frac{\partial v_\varep}{\partial\nu^*_\varep}=0$ on
$\partial\Omega$, where $\mathcal{L}_\varep^*$ denotes the adjoint of $\mathcal{L}_\varep$.
Since $A^*$ satisfies the same conditions as $A$ 
and $p^\prime>2$, we have 
$$
\|\nabla v_\varep\|_{L^{p^\prime}(\Omega)}
\le C\, \| g\|_{L^{p^\prime}(\Omega)}.
$$
Also, note that if $f=(f_i^\alpha)\in C_0^1(\Omega;\br^{m\times d})$ 
and $u_\varep$ is a weak solution in $W^{1,2}(\Omega;\br^m)$
of $\mathcal{L}_\varep (u_\varep) =\text{div}(f)$ in $\Omega$ with
 $\frac{\partial u_\varep}{\partial \nu_\varep}=0$ on $\partial\Omega$, then
\begin{equation}\label{4.2-1-7}
\int_\Omega f_i^\alpha \cdot \frac{\partial v_\varep^\alpha}{\partial x_i}\, dx
=\int_\Omega A(x/\e)\nabla u_\e \cdot \nabla v_\e\, dx
=\int_\Omega g_i^\alpha \cdot \frac{\partial u_\varep^\alpha}{\partial x_i}\, dx.
\end{equation}
It follows from (\ref{4.2-1-7}) by duality that 
$$
\|\nabla u_\varep\|_{L^p(\Omega)} \le C\, \| f\|_{L^p(\Omega)}.
$$
By a density argument this gives the existence of solutions in $W^{1,p}(\Omega;\br^m)$ for general $f$ in 
$L^p(\Omega;\br^{m\times d})$.
Observe that the duality argument above in fact shows that any solution
 in $W^{1,p}(\Omega;\br^m)$ with data $f\in L^p(\Omega;\br^{d\times m})$
satisfies (\ref{estimate-4.2-1}).
As a consequence we obtain the uniqueness for $1<p<2$.
\end{proof}

\begin{lemma}\label{lemma-4.2-2}
Let $1<p<\infty$ and  $g\in B^{-1/p, p}(\partial\Omega;\br^m)$ with
\begin{equation}\label{c-d}
\langle g, b\rangle_{B^{-1/p, p}(\partial\Omega)\times B^{1/p, p^\prime}(\partial\Omega)}
=0
\end{equation}
for any $b\in \br^m$.
  Then there exists a unique (up to constants) $u_\varep\in W^{1,p}(\Omega;\br^m)$ such that
$\mathcal{L}_\varep (u_\varep)=0$ in $\Omega$ and $\frac{\partial u_\varep}{\partial\nu_\varep} =g$ on $\partial\Omega$.
Moreover, the solution satisfies 
\begin{equation}\label{estimate-4.2-2}
\|\nabla u_\varep\|_{L^p(\Omega)}
\le C\, \| g\|_{B^{-1/p, p}(\partial\Omega)},
\end{equation}
where $C$ depends only on $\mu$, $p$, $\omega(t)$ in (\ref{VMO-1}), and $\Omega$.
\end{lemma}

\begin{proof}
The uniqueness is contained in Lemma \ref{lemma-4.2-1}. To establish the existence as well as the
estimate (\ref{estimate-4.2-2}),
we first assume that 
$$
g\in B^{-1/p, p}(\partial\Omega;\br^m)\cap B^{-1/2, 2}(\partial\Omega;\br^m)
$$
 and
 show that the estimate (\ref{estimate-4.2-2}) holds for  solutions in 
$W^{1,2}(\Omega;\br^m)$, given by Theorem \ref{s-NP-theorem}.
By a density argument this gives the existence of solutions in $W^{1,p}(\Omega;\br^m)$
for general $g\in B^{-1/p, p}(\partial\Omega;\br^m)$ satisfying (\ref{c-d}).

Let $f=(f_i^\alpha)\in C_0^1(\Omega;\br^{m\times d})$ and $v_\varep$ be a weak solution in $W^{1,2}(\Omega;\br^m)$
 to 
 $$
 \mathcal{L}_\varep^* (v_\varep)=\text{div}(f) \quad 
\text{ in }
\Omega \quad \text{  and  }
\frac{\partial v_\varep}{\partial \nu_\varep^*}=0 \quad \text{ on }\partial\Omega.
$$
Since $A^*$ satisfies the same conditions as $A$, by Lemma \ref{lemma-4.2-1}, we have
$$
\|\nabla v_\varep\|_{L^{p^\prime}(\Omega)} \le C\, \| f\|_{L^{p^\prime}(\Omega)}.
$$
Note that
\begin{equation}\label{4.2-2-1}
\aligned
\int_\Omega f_i^\alpha \cdot \frac{\partial u_\varep^\alpha}{\partial x_i}\, dx
& =-\int_\Omega A(x/\e)\nabla u_\e \cdot \nabla v_\e\, dx\\
& =-\langle g, v_\varep\rangle_{B^{-1/p, p}(\partial\Omega)\times B^{1/p, p^\prime}(\partial\Omega)}.
\endaligned
\end{equation}
Let $E$ denote the $L^1$ average of $v_\varep$ over $\Omega$.
Then
\begin{equation}\label{4.2-2-3}
\aligned
|\langle g, v_\varep\rangle_{B^{-1/p, p}(\partial\Omega)\times B^{1/p, p^\prime}(\partial\Omega)}|
& =|\langle g, v_\varep-E\rangle_{B^{-1/p, p}(\partial\Omega)\times B^{1/p, p^\prime}(\partial\Omega)}|\\
&\le \| g\|_{B^{-1/p, p}(\partial\Omega)} \| v_\varep-E\|_{B^{1/p, p^\prime}(\partial\Omega)}\\
& \le C\, \| g\|_{B^{-1/p, p}(\partial\Omega)} \| v_\varep-E\|_{W^{1,p^\prime}(\Omega)}\\
&\le C \, \| g\|_{B^{-1/p, p}(\partial\Omega)} \|\nabla v_\varep\|_{L^{p^\prime}(\Omega)}\\
& \le C\, \| g\|_{B^{-1/p, p}(\partial\Omega)} \| f\|_{L^{p^\prime} (\Omega)},
\endaligned
\end{equation}
where we have used a trace theorem for the second inequality and Poincar\'e inequality for the third.
The estimate (\ref{estimate-4.2-2}) follows from (\ref{4.2-2-1})-(\ref{4.2-2-3}) by duality.
\end{proof}

Let $1<q<d$ and $\frac{1}{p}=\frac{1}{q}-\frac{1}{d}$. 
In the proof of the next lemma we will need the following Sobolev inequality:
\begin{equation}\label{Sobolev-4.2}
\left(\int_\Omega |u|^p\, dx \right)^{1/p}
\le C \left(\int_\Omega |\nabla u|^q\, dx \right)^{1/q},
\end{equation}
where $u\in W^{1,q}(\Omega)$ and $\int_{\partial\Omega} u=0$. 
Note that by the Sobolev imbedding, (\ref{Sobolev-4.2})
 also holds for $q>d$ and $p=\infty$. If $q=d$, it holds for any $1<p<\infty$.
 To see (\ref{Sobolev-4.2}), by Poincar\'e-Sobolev inequality, it suffices to show that
 \begin{equation}\label{S-4.2-0}
 \Big|\int_\Omega u \, dx \Big|
 \le C \Big|\int_{\partial\Omega} u \, d\sigma \Big|
 + C \left(\int_\Omega |\nabla u|^q\, dx \right)^{1/q}
 \end{equation}
 for $q>1$, which may be done by using a proof by contradiction.

\begin{lemma}\label{lemma-4.2-3}
Let $1<p<\infty$ and $1\le q<\infty$, where 
\begin{equation}\label{pq-100}
q=\left\{
\aligned
&\frac{pd}{p+d} & \quad & \text{ if } p>\frac{d}{d-1},\\
 & q>1 & \quad & \text{ if  } p=\frac{d}{d-1}, \\
&q=1 &\quad & \text{ if } 1<p<\frac{d}{d-1}.
\endaligned
\right.
\end{equation}
Then, for any $F\in L^q(\Omega;\br^m)$,  there exists a unique (up to constants) solution $u_\varep$ in $W^{1,p}(\Omega;\br^m)$
to $\mathcal{L}_\varep(u_\varep)=F$ in $\Omega$ and $\frac{\partial u_\varep}{\partial\nu_\varep}
=-b$ on $\partial\Omega$, where $b=\frac{1}{\partial\Omega} \int_\Omega F$.
Moreover, the solution satisfies 
\begin{equation}\label{estimate-4.2-3}
\|\nabla u_\varep\|_{L^p(\Omega)} \le C\, \| F\|_{L^q(\Omega)}.
\end{equation}
\end{lemma}

\begin{proof} The uniqueness is contained in Lemma \ref{lemma-4.2-1}. To establish the existence as well as the estimate 
(\ref{estimate-4.2-3}), we first assume $F\in C_0^1(\Omega;\br^m)$. 
By Theorem \ref{theorem-1.1-3} there exists a unique (up to constants) solution in $W^{1,2}(\Omega;\br^m)$
to $\mathcal{L}_\varep (u_\varep)=F$ in $\Omega$ and $\frac{\partial u_\varep}{\partial\nu_\varep}=-\frac{1}{|\partial\Omega}
\int_\Omega F$ on $\partial\Omega$. We will show the solution satisfies (\ref{estimate-4.2-3}).
By a density argument this would give the existence of solutions in $W^{1,p}(\Omega;\br^m)$ for general $F\in L^q(\Omega;\br^m)$.

Let $f=(f_i^\alpha)\in C_0^1(\Omega;\br^m)$ and $v_\varep\in W^{1,2}(\Omega;\br^m)$ be a weak solution to the Neumann
problem: $\mathcal{L}_\varep^* (v_\varep)=
\text{div}(f)$ in $\Omega$ and $\frac{\partial v_\varep}{\partial\nu_\varep}=0$ on $\partial\Omega$.
By Lemma \ref{lemma-4.2-1}, we have 
\begin{equation}\label{4.2-3-00}
\|\nabla v_\varep\|_{L^{p^\prime}(\Omega)}
\le C \, \| f\|_{L^{p^\prime}(\Omega)}.
\end{equation}
 Note that
\begin{equation}\label{4.2-3-1}
\aligned
\int_\Omega \frac{\partial u_\varep^\alpha}{\partial x_i} \cdot f_i^\alpha \, dx
&=\int_\Omega A(x/\e)\nabla u_\e\cdot \nabla v_\e\, dx\\
&=\int_\Omega F \cdot v_\varep\, dx -\int_{\partial\Omega} b \cdot v_\varep\, d\sigma\\
&=\int_\Omega F (v_\varep -E)\, dx,
\endaligned
\end{equation}
where $b=\frac{1}{|\partial\Omega|} \int_\Omega F$ and
$E$ is the $L^1$ average of $v_\varep$ over $\partial\Omega$.
Note that 
$$
\| v_\varep-E\|_{L^{q^\prime}(\Omega)} \le C\, \|\nabla v_\varep\|_{L^{p^\prime}(\Omega)}
\le C\, \| f\|_{L^{p^\prime}(\Omega)}
$$
by (\ref{Sobolev-4.2}) and (\ref{4.2-3-00}). 
In view of (\ref{4.2-3-1}) we obtain
$$
\aligned
\Big|\int_\Omega \frac{\partial u_\varep^\alpha}{\partial x_i} \cdot f_i^\alpha \, dx\Big|
&\le \| F\|_{L^q(\Omega)} \| v_\varep -E\|_{L^{q^\prime}(\Omega)}\\
& \le C\, \| F\|_{L^q(\Omega)} \|\nabla v_\varep \|_{L^{p^\prime}(\Omega)}\\
&\le C\,  \|F\|_{L^q(\Omega)} \| f\|_{L^{p^\prime}(\Omega)}.
\endaligned
$$
By duality this gives the estimate (\ref{estimate-4.2-3}).
\end{proof}

We are now in a position to give the proof of Theorem \ref{W-1-p-Neumann-theorem}.

\begin{proof}[\bf Proof of Theorem \ref{W-1-p-Neumann-theorem}]
The uniqueness is contained in Lemma \ref{lemma-4.2-1}. To establish the existence as well as the estimate
(\ref{estimate-4.2}), we let $v_\varep\in W^{1,p}(\Omega;\br^m)$ be a weak solution to $\mathcal{L}_\varep (v_\varep)=\text{div}(f)$ in $\Omega$ and
$\frac{\partial v_\varep}{\partial\nu_\varep}=-n\cdot f$ on $\partial\Omega$. Also, let $w_\varep$ be a weak solution
to $\mathcal{L}_\varep (w_\varep)=F$ in $\Omega$ and $\frac{\partial w_\varep}{\partial\nu_\varep}
=b$ on $\partial\Omega$, where $b=-\frac{1}{|\partial\Omega|}\int_\Omega F$, and $z_\varep\in W^{1,p}(\Omega;\br^m)$ be a weak solution to
$\mathcal{L}_\varep (z_\varep)=0$ in $\Omega$ and $\frac{\partial z_\varep}{\partial \nu_\varep}
=g-b$ on $\partial\Omega$. Let $u_\varep=v_\varep+w_\varep+z_\varep$.
Note that $\mathcal{L}_\varep (u_\varep)=\text{div}(f) +F$ in $\Omega$ and $\frac{\partial u_\varep}{\partial\nu_\varep}
=g-n\cdot f$ on $\partial\Omega$. In view of Lemmas \ref{lemma-4.2-1}, \ref{lemma-4.2-2}, and \ref{lemma-4.2-3}, we obtain
$$
\aligned
\|\nabla u_\varep \|_{L^p(\Omega)}
&\le \|\nabla v_\varep\|_{L^p(\Omega)} +\| \nabla w_\varep\|_{L^p(\Omega)} +\| \nabla z_\varep\|_{L^p(\Omega)}\\
& \le C \Big\{ \| f\|_{L^p(\Omega)} +\| F\|_{L^q(\Omega)} +\| g\|_{B^{-1/p, p}(\partial\Omega)} \Big\},
\endaligned
$$
where $q=\frac{pd}{p+d}$ for $p>\frac{d}{d-1}$, $q>1$ for $p=\frac{d}{d-1}$, and $q=1$ for $1<p<\frac{d}{d-1}$.
\end{proof}

\begin{remark}\label{W-1-p-100}
{\rm 
Let $\Omega$ be a fixed Lipschitz domain in $\br^d$. We also fix $2<p<\infty$.
Suppose that for any $f\in C_0^1(\Omega; \br^{m\times d})$, weak solutions 
in $W^{1,2}(\Omega; \br^m)$ to the Neumann problem,
$$
\mathcal{L}_\e (v_\e)=\text{\rm div}(f) \quad  \text{ in }
\Omega \quad  \text{ and } \quad 
\frac{\partial v_\e}{\partial \nu_\e}=0 \quad
\text{ on }
\partial\Omega,
$$
 satisfy the $W^{1, p}$ estimate
$$
\| \nabla v_\e \|_{L^p(\Omega)} \le C_0 \| f\|_{L^p(\Omega)}
$$
for some $C_0>0$.
Then, for any $g\in B^{-\frac{1}{p}, p} (\partial\Omega; \br^m)$ satisfying the compatibility condition (\ref{c-d}),
 weak solutions to the Neumann problem,
 $$
 \mathcal{L}_\e (u_\e) =0 \quad \text{ in } \Omega \quad \text{ and } \quad 
 \frac{\partial u_\e}{\partial \nu_\e} =g \quad \text{ on } \partial\Omega,
 $$
 satisfy the estimate
 $$
 \|\nabla u_\e\|_{L^p(\Omega)} \le C C_0 \| g\|_{B^{-1/p, p}(\partial\Omega)},
 $$
 where $C$ depends only on $\Omega$.
 This follows from the duality argument used in the proof of Lemma \ref{lemma-4.2-2}.
 Similarly, by the duality argument used in the proof of Lemma \ref{lemma-4.2-3},
 for any $F\in L^q(\Omega; \br^m)$, where $\frac{1}{q}=\frac{1}{p} +\frac{1}{d}$,
 weak solutions in $W^{1,2}(\Omega; \br^m)$ to the Neumann problem,
 $$
  \mathcal{L}^*_\e (u_\e) =F \quad \text{ in } \Omega \quad \text{ and } \quad 
 \frac{\partial u_\e}{\partial \nu^*_\e} =b \quad \text{ on } \partial\Omega,
 $$
satisfy the $W^{1, p}$ estimate
$$
\|\nabla u_\e\|_{L^p(\Omega)} \le C C_0 \| F\|_{L^q(\Omega)},
$$
where $b=-\frac{1}{|\partial\Omega|} \int_\Omega F$ and
$C$ depends only on $\Omega$.
}
\end{remark}

\begin{remark}\label{W-1-p-200}
{\rm
Let $B=B(x_0, r)$ for some $x_0\in \partial\Omega$ and $0<r<r_0$.
Let $u_\e\in H^1(2B\cap \Omega; \br^m)$ be a weak solution 
of $\mathcal{L}_\e (u_\e)=0$ in $ 2B\cap \Omega$ with
$\frac{\partial u_\e}{\partial\nu_\e}=g$ on $2B\cap \partial\Omega$.
Then for $2<p<\infty$,
\begin{equation}\label{W-1-p-201}
\left(\average_{B\cap \Omega} |\nabla u_\e|^p\, dx \right)^{1/p}
\le \frac{C}{r}
\left(\average_{2B\cap \Omega} |u_\e|^2\, dx \right)^{1/2}
+ C \left(\average_{2B\cap \partial\Omega} |g|^t \, d\sigma \right)^{1/t},
\end{equation}
where $t=p(d-1)/d$.
To see this we  apply the estimate  (\ref{estimate-4.2}) to the function $\psi u_\e $, 
where $\psi\in C_0^\infty (2B)$ is a cut-off function such that
$\psi=1$ in $B$ and $|\nabla \psi|\le C r^{-1}$.
A bootstrap argument  as well as the Sobolev imbedding $B^{1/p, p^\prime}(\partial\Omega)
\subset L^s(\partial\Omega)$, where $\frac{1}{s}=\frac{1}{p^\prime}-
\frac{1}{p(d-1)}$, is also needed.
}
\end{remark}



\section{Boundary Lipschitz estimates}\label{section-LN}

The goal of this section is to establish uniform boundary Lipschitz estimates in $C^{1, \eta}$ domains
for solutions with Neumann conditions.
Throughout the section we assume that $D_r=D(r, \psi)$, $\Delta_r =\Delta (r, \psi)$, and
$\psi:\br^{d-1}\to \br$ is a $C^{1, \eta}$ function satisfying the condition (\ref{psi-1}).

\begin{thm}[boundary Lipschitz estimate at large scale]\label{bl-l}
Suppose that $A$ is 1-periodic and satisfies the ellipticity condition (\ref{s-ellipticity}).
Let $ u_\e \in H^1(D_2; \br^m)$ be a weak solution of $\mathcal{L}_\e (u_\e)=F$ in $D_2$
with $\frac{\partial u_\e}{\partial \nu_\e}=g$ on $\Delta_2$,
where $F\in L^p(D_2; \br^m)$,
$g\in C^{ \rho}(\Delta_2; \br^m)$ for some  $p>d$ and $\rho\in (0,\eta)$.
Then, for $\e \le  r\le 1$,
\begin{equation}\label{bl-l-0}
\left(\average_{D_r} |\nabla u_\e|^2\right)^{1/2}
\le C \left\{ \left(\average_{D_2} |\nabla u_\e|^2\right)^{1/2}
+ \| F\|_{L^p(D_2)} +\| g\|_{L^\infty(\Delta_2)}
+ \| g\|_{C^{0, \rho}(\Delta_2)} \right\},
\end{equation}
where $C$ depends only on $\mu$, $p$, $\rho$, and $(M_0, \eta)$ in (\ref{psi-1}).
\end{thm}

No smoothness condition on $A$ is needed for the estimate (\ref{bl-l-0}).
Under the additional H\"older continuity condition (\ref{smoothness}), we may deduce the 
full-scale Lipschitz estimate from Theorem \ref{bl-l}.

\begin{thm}[boundary Lipschitz estimate]\label{NP-Lip-theorem-4.0}
Suppose that $A$ satisfies (\ref{s-ellipticity}), (\ref{periodicity}) and (\ref{smoothness}).
Let $\Omega$ be a bounded $C^{1, \eta}$ domain.
Let $u_\varep\in H^1(B(x_0, r)\cap\Omega; \mathbb{R}^m)$ be a weak solution 
of $\mathcal{L}_\varep (u_\varep) =F$ in $B(x_0, r)\cap\Omega$ with 
$\frac{\partial u_\varep}{\partial \nu_\varep}=g $ on $B(x_0, r)\cap \partial\Omega$,
for some $x_0\in \partial\Omega$ and $0<r<r_0$.
Then
\begin{equation}\label{NP-Lip-1}
\aligned
\|\nabla u_\varep\|_{L^\infty (B(x_0, r/2)\cap\Omega)}
& \le C\Bigg\{
\left(\average_{B(x_0, r)\cap\Omega} |\nabla u_\varep|^2\right)^{1/2}
+r\left(\average_{B(x_0, r)\cap\Omega} |F|^p\right)^{1/p}\\
&\qquad\qquad\qquad
 +\| g\|_{L^\infty(B(x_0, r)\cap\partial\Omega)}
+r^\rho \| g\|_{C^{0, \rho}(B(x_0, r)\cap\partial\Omega)} \Bigg\},
\endaligned
\end{equation}
where $\rho \in (0, \eta)$, $p>d$,  and $C$ depends only on $\rho$, $p$,  $\mu$, $(\lambda,\tau)$ in 
(\ref{smoothness}), and $\Omega$.
\end{thm}

\begin{proof}
We give the proof of Theorem \ref{NP-Lip-theorem-4.0}, assuming Theorem \ref{bl-l}.
By a change of the coordinate system it suffices to prove that if $p>d$ and $\rho\in (0,\eta)$,
\begin{equation}\label{NL-0}
\|\nabla u_\varep\|_{L^\infty (D_r)}
 \le C\Bigg\{
\left(\average_{D_{2r}} |\nabla u_\varep|^2\right)^{1/2}
+r\left(\average_{D_{2r}} |F|^p\right)^{1/p}
 +\| g\|_{L^\infty(\Delta_{2r})}
+r^\rho \| g\|_{C^{0, \rho}(\Delta_{2r})} \Bigg\},
\end{equation}
for $0<r\le 1$,
where $\mathcal{L}_\e (u_\e)=F$ in $D_{2r}$,
$\frac{\partial u_\e}{\partial \nu_\e}=g$ on $\Delta_{2r}$,
and $C$ depends only on $\mu$, $p$, $\rho$ and $(M_0, \eta)$ in (\ref{psi-1}).
By rescaling we may assume that $r=1$.
Note that if $\e\ge1$, the matrix $A(x/\e)$ is uniformly H\"older continuous in $\e$.
Consequently,  the case $\e \ge1$ follows from the standard boundary Lipschitz estimates  in $C^{1, \eta}$ domains 
for elliptic systems with H\"older continuous coefficients.

We thus assume that $r=1$ and $0<\e<1$.
Let $w(x)=\e^{-1} u_\e (\e x)$.
Then $\mathcal{L}_1(w)=\widetilde{F}$ in $\widetilde{D}_2$,
 $\frac{\partial w}{\partial \nu_1} =\widetilde{g}$ on $\widetilde{\Delta}_2$, where
 $\widetilde{F}(x)=\e F(\e x)$, $\widetilde{g} (x)=g(\e x)$, and 
 $$
 \widetilde{D}_r = D(r, \widetilde{\psi}), \quad
 \widetilde{\Delta}_r = \Delta (r, \widetilde{\psi}), \quad
 \widetilde{\psi} (x^\prime)
 =\e^{-1} \psi(\e x^\prime).
 $$
 Since $0<\e<1$, the function $\widetilde{\psi}$ satisfies the condition (\ref{psi-1})
 with the same $(M_0, \eta)$.  
 It follows from  boundary Lipschitz estimates for the operator $\mathcal{L}_1$ that
 $$
 \|\nabla w\|_{L^\infty(\widetilde{D}_1)}
 \le C \left\{ \|\nabla w\|_{L^2(\widetilde{D}_2)}
 + \| \widetilde{F}\|_{L^p(\widetilde{D}_2)}
 +\|\widetilde{g}\|_{L^\infty(\widetilde{\Delta}_2)}
 +\|\widetilde{g}\|_{C^{0, \rho}(\widetilde{\Delta}_2)}\right\}.
 $$
By a change of variables this leads to 
\begin{equation}\label{NL-1}
\aligned
\|\nabla u_\varep\|_{L^\infty (D_\e)}
  &\le C\Bigg\{
\left(\average_{D_{2\e}} |\nabla u_\varep|^2\right)^{1/2}
+\e \left(\average_{D_{2\e}} |F|^p\right)^{1/p}
 +\| g\|_{L^\infty(\Delta_{2\e})}
+\e^\rho \| g\|_{C^{0, \rho}(\Delta_{2\e})} \Bigg\}\\
&\le C\Bigg\{
\left(\average_{D_{2\e}} |\nabla u_\varep|^2\right)^{1/2}
+ \| F\|_{L^p(D_2)}
 +\| g\|_{L^\infty(\Delta_{2})}
+ \| g\|_{C^{0, \rho}(\Delta_{2})} \Bigg\}\\
 &  \le C\Bigg\{
\left(\average_{D_{2}} |\nabla u_\varep|^2\right)^{1/2}
+ \|F\|_{L^p(D_2)}
 +\| g\|_{L^\infty(\Delta_{2})}
+ \| g\|_{C^{0, \rho}(\Delta_{2})} \Bigg\},
\endaligned
\end{equation}
where we have used the fact that $p>d$ and $\e<1$ for the second inequality and
(\ref{bl-l-0}) for the last.

Using (\ref{NL-1}) and translation,
 we may bound $|\nabla u_\e (x)|$ by the RHS of (\ref{NL-1}) for any $x\in D_1$
with $\text{dist}(x, \Delta_1)\le c\e$.
Similarly, by combining interior Lipschitz estimates for $\mathcal{L}_1$ with (\ref{bl-l-0}), 
we may dominate $|\nabla u_\e (x)|$ by the RHS of (\ref{NL-1})
for any $x\in D_1$ with $\text{dist}(x, \Delta_1)\ge c\e$.
This finishes the proof.
\end{proof}

\begin{cor}\label{NPL-0}
Suppose that $A$ satisfies (\ref{periodicity}), (\ref{s-ellipticity}) and (\ref{smoothness}).
Let $\Omega$ be a bounded $C^{1, \eta}$ domain in $\br^d$ for some $\eta\in (0,1)$.
Let $u_\e\in H^1(\Omega; \br^m)$ be a weak solution to the Neumann problem: 
\begin{equation}\label{NPL-1}
\mathcal{L}_\e (u_\e)=F \quad \text{ in } \Omega \quad
\text{ and } \quad \frac{\partial u_\e}{\partial \nu_\e}=g \quad \text{ on } \partial\Omega,
\end{equation}
where $F\in L^p(\Omega; \br^m)$,
$g\in C^{\rho}(\partial\Omega; \br^m)$ for some $p>d$ and $\rho\in (0,\eta)$, and
$\int_\Omega F \, dx +\int_{\partial\Omega} g\, d\sigma=0$.
Then 
\begin{equation}\label{NPL-2}
\|\nabla u_\e \|_{L^\infty(\Omega)}
\le C \Big\{ \| F\|_{L^p(\Omega)} + \| g\|_{C^\rho(\partial\Omega)} \Big\},
\end{equation}
where $C$ depends only on $p$, $\rho$, $\mu$, $(\lambda, \tau)$ and $\Omega$.
\end{cor}

\begin{proof}
By covering $\partial\Omega$ with a finite number of balls $\{ B(x_\ell, r_0/4)\}$,
where $x_\ell \in \partial \Omega$, we may deduce from Theorem \ref{NP-Lip-theorem-4.0}
and the interior Lipschitz estimate in Theorem \ref{interior-Lip-theorem} that
$$
\|\nabla u_\e \|_{L^\infty(\Omega)}
\le C\Big\{ \|\nabla u_\e\|_{L^2(\Omega)}
+\| F\|_{L^p(\Omega)} +\| g\|_{C^\rho(\partial\Omega)} \Big\},
$$
which, together with the energy estimate for $\|\nabla u_\e\|_{L^2(\Omega)}$,
gives (\ref{NPL-2}).
\end{proof}

The rest of this section is devoted to the proof of Theorem \ref{bl-l}.

\begin{lemma}\label{lemma-4.0-3}
Suppose that $\mathcal{L}_0 (w)=F$ in $D_{r}$ and $\frac{\partial w}{\partial \nu_0} =g $ on $\Delta_{r}$
for some $0<r\le 1$.
Let
$$
\aligned
I(t)= &\frac{1}{t}\inf_{\substack{E\in \mathbb{R}^{m\times d}\\ q\in \mathbb{R}^m}}
\bigg\{ \left(\average_{D_t} |w-E x -q|^2\right)^{1/2}
+t^2 \left(\average_{D_t} |F|^p \right)^{1/p}
+ t \,\big\| \frac{\partial}{\partial \nu_0} \big( w-Ex \big)\big\|_{L^\infty(\Delta_t)}\\
& \qquad\qquad\qquad\qquad
+t^{1+\rho} \big\| \frac{\partial}{\partial \nu_0} \big( w-Ex \big)\big\|_{C^{0,\rho}(\Delta_t)} \bigg\}
\endaligned
$$
for $0<t\le r$, where $p>d$ and $0<\rho < \min \big\{ \eta, 1-\frac{d}{p} \big\}$.
Then there exists $\theta\in (0,1/4)$, depending only on $\rho$, $p$, 
$\mu$ and $(\eta, M_0)$, such
that
\begin{equation}\label{NP-Lip-9}
I (\theta r) \le (1/2) I(r).
\end{equation}
\end{lemma}

\begin{proof}
The proof uses  boundary $C^{1, \rho}$ estimates in $C^{1, \eta}$ domains with Neumann conditions
for second-order elliptic systems with constant coefficients.
By rescaling we may assume  $r=1$.
By choosing $q=w(0)$ and $E=\nabla w (0)$, we see that for any $\theta\in (0,1/4)$,
$$
\aligned
I(\theta) &  \le C\, \theta^{\rho} \| \nabla w\|_{C^{0, \sigma}(D_\theta)} + C \theta^{1-\frac{d}{p}} \left(\average_{D_1} |F|^p\right)^{1/p}\\
&\le C\, \theta^{\rho} \left\{ \| \nabla w \|_{C^{0, \sigma}(D_{1/4})} 
+ \left(\average_{D_1} |F|^p\right)^{1/p} \right\},
\endaligned
$$
where we have used the assumption  $\rho< 1-\frac{d}{p}$.
It follows from  the boundary $C^{1,\rho}$ estimates for $\mathcal{L}_0$ that
$$
\| \nabla w\|_{C^{0, \rho}(D_{1/4})}
\le C \left\{ \left(\average_{D_1} |w|^2\right)^{1/2}
+\left(\average_{D_1} |F|^p\right)^{1/p}
+\big\| \frac{\partial w}{\partial\nu_0}\big\|_{L^\infty (\Delta_1)}
+\big\| \frac{\partial w}{\partial\nu_0}\big\|_{C^{0, \rho} (\Delta_1)} \right\}.
$$
Hence, for any $\theta \in (0,1/4)$,
$$
I(\theta) \le C\, \theta^{\rho} 
 \left\{ \left(\average_{D_1} |w|^2\right)^{1/2}
+\left(\average_{D_1} |F|^p\right)^{1/p}
+\big\| \frac{\partial w}{\partial\nu_0}\big\|_{L^\infty (\Delta_1)}
+\big\| \frac{\partial w}{\partial\nu_0}\big\|_{C^{0, \rho} (\Delta_1)} \right\},
$$
where $C$ depends only on $\mu$, $\rho$, $p$ and $(\eta, M_0)$ in (\ref{psi-1}).
Finally, since $\mathcal{L}_0 (w- Ex-q)=F$  in $D_2$ for any $E\in \mathbb{R}^{m\times d}$
and $q\in \mathbb{R}^m$, the inequality above gives
$$
I(\theta)\le C\, \theta^\rho I(1).
$$
The estimate (\ref{NP-Lip-9}) follows by choosing $\theta\in (0,1/4)$ so small that
$C\theta^\rho \le (1/2)$.
\end{proof}

\begin{lemma}\label{lemma-4.0-4}
Suppose that $\mathcal{L}_\varep (u_\varep)=F$ in $D_2$ and 
$\frac{\partial u_\varep}{\partial\nu_\varep}=g$ on $\Delta_2$, where $0<\varep< 1$.
Let
$$
\aligned
H(t)= &\frac{1}{t}\inf_{\substack{E\in \mathbb{R}^{m\times d}\\ q\in \mathbb{R}^m}}
\bigg\{ \left(\average_{D_t} |u_\varep-E x -q|^2\right)^{1/2}
+t^2\left(\average_{D_t} |F|^p\right)^{1/p}\\
& \qquad\qquad\qquad\qquad\qquad
+ t \,\big\| g- \frac{\partial}{\partial \nu_0} \big(Ex \big)\big\|_{L^\infty(\Delta_t)}
+t^{1+\sigma} \big\| g-\frac{\partial}{\partial \nu_0} \big(Ex \big)\big\|_{C^{0,\sigma}(\Delta_t)} \bigg\},
\endaligned
$$
where $0<t\le 1$ and $0<\rho<\min \big\{ \eta, 1-\frac{d}{p}\big\}$.
Then, for $\varep<t\le 1$,
\begin{equation}\label{NP-Lip-10}
H(\theta t)\le \frac{1}{2} H(t)
+ C \left(\frac{\varep}{t} \right)^{\alpha}
\left\{ \frac{1}{t} \inf_{q\in \mathbb{R}^m} 
\left(\average_{D_{2t}} |u_\varep -q|^2 \right)^{1/2}
+t\left(\average_{D_{2t}} |F|^p\right)^{1/p}
+\| g\|_{L^\infty(\Delta_{2t})} \right\},
\end{equation}
where $\theta\in (0, 1/4)$ is given by Lemma \ref{lemma-4.0-3},
$\alpha\in (0,1/2)$ is given by Theorem \ref{4.0-0-1},  and
$C$ depends only on $\rho$, $p$, $\mu$, and $(\eta, M_0)$.
\end{lemma}

\begin{proof}
For each $t\in (\varep, 1]$,
let $w=w_t$ be the solution of $\mathcal{L}_0 (w)=F$ in $D_t$ with $\frac{\partial w}{\partial \nu_0}
=g$ on $\Delta_t$, given by Theorem \ref{4.0-0-1}.
Using
$$
\left(\average_{D_{\theta t}} |u_\varep -Ex -q|^2 \right)^{1/2}
\le \left(\average_{D_{\theta t}} |u_\varep -w|^2\right)^{1/2}
+\left(\average_{D_{\theta t}} |w -Ex -q|^2 \right)^{1/2}
$$
for any $E\in \mathbb{R}^{m\times d}$ and $q\in \mathbb{R}^m$,
we may deduce that
\begin{equation}\label{NP-Lip-11}
H(\theta t) \le I (\theta t) + \frac{1}{\theta t}\left(\average_{D_{\theta t}} |u_\varep -w|^2\right)^{1/2}.
\end{equation}
Similarly, since
$$
\left(\average_{D_t} |w -Ex -q|^2 \right)^{1/2}
\le \left(\average_{D_t} |u_\varep -w|^2\right)^{1/2}
+\left(\average_{D_t} |u_\varep -Ex -q|^2 \right)^{1/2},
$$
we obtain
$$
I( t) \le H ( t) + \frac{1}{t} \left(\average_{D_t} |u_\varep -w|^2\right)^{1/2}.
$$
This, together with (\ref{NP-Lip-11}) and the estimate $I(\theta t)\le (1/2) I(t)$ in Lemma \ref{lemma-4.0-3},
gives
\begin{equation}\label{NP-Lip-12}
H(\theta t) \le \frac{1}{2} H(t) +\frac{C}{t}
\left(\average_{D_t} |u_\varep -w|^2\right)^{1/2},
\end{equation}
which, by Theorem \ref{4.0-0-1}, yields (\ref{NP-Lip-10}).
\end{proof}

The proof of the next lemma will be given at the end of this section.

\begin{lemma}\label{g-lemma}
Let $H(r)$ and $h(r)$ be two nonnegative, continuous functions on the interval 
$(0, 1]$. Let $0<\e< (1/4)$. 
Suppose that there exists a constant $C_0$ such that
\begin{equation}\label{g-c-1}
\max _{r\le t\le 2r} H(t) \le C_0 H(2r)
\quad \text{ and } \quad
\max_{r\le t, s\le 2r} | h(t)-h(s)|
\le C_0 H(2r)
\end{equation}
for any $r\in [\e, 1/2]$.
We further assume that
\begin{equation}\label{g-c-2}
H(\theta r)\le \frac12 H(r) +C_0\,  \beta (\e/r) \big\{ H(2r) +h(2r) \big\}
\end{equation}
for any $r\in [\e, 1/2]$, where $\theta\in (0, 1/4)$ and $\beta(t)$ is a nonnegative, nondecreasing function
on $[0, 1]$ such that $\beta(0)=0$ and
\begin{equation}\label{g-c-3}
\int_0^1 \frac{\beta (t)}{t}\, dt<\infty.
\end{equation}
Then
\begin{equation}\label{g-c-4}
\max_{\e\le r\le 1} \big\{ H(r) + h(r)\big\}
\le C \big\{ H(1) + h(1) \big\},
\end{equation}
where $C$ depends only on $C_0$, $\theta$ and the function $\beta (t)$.
\end{lemma}

We now give the proof of Theorem \ref{bl-l}, using Lemmas \ref{lemma-4.0-4}
and \ref{g-lemma}.

\begin{proof}[\bf Proof of Theorem \ref{bl-l}]

Let $u_\e$ be a solution of $\mathcal{L}_\e (u_\e)=F$ in $D_2$ with
$\frac{\partial u_\e}{\partial \nu_\e}=g$ on $\Delta_2$.
We define the function $H(t)$ by (\ref{NP-Lip-10}).
It is not hard to see that 
\begin{equation}\label{NP0-1}
H(t)\le C H(2r) \quad \text{ if } t\in [r, 2r]
\end{equation}

Next, we define $h(t)=|E_t|$, where $E_t$ is the $m\times d$ matrix such that
$$
\aligned
H(t)= &\frac{1}{t}\inf_{q\in \mathbb{R}^m}
\bigg\{ \left(\average_{D_t} |u_\varep-E_t x -q|^2\right)^{1/2}
+t^2\left(\average_{D_t} |F|^p\right)^{1/p}\\
& \qquad\qquad\qquad\qquad\qquad
+ t \,\big\| g- \frac{\partial}{\partial \nu_0} \big(E_tx \big)\big\|_{L^\infty(\Delta_t)}
+t^{1+\rho} \big\| g-\frac{\partial}{\partial \nu_0} \big(E_t x \big)\big\|_{C^{0,\rho}(\Delta_t)} \bigg\}.
\endaligned
$$
Let $t, s\in [r, 2r]$. Using 
$$
\aligned
|E_t -E_s|
&\le \frac{C}{r} \inf_{q\in \br^m}
\left(\average_{D_r} |(E_t -E_s)x -q|^2\right)^{1/2}\\
& =\frac{C}{r} \inf_{q_1\in \br^m, q_2\in \br^m}
\left(\average_{D_r} |(E_t -E_s)x -q_1 +q_2 |^2\right)^{1/2}\\
&\le \frac{C}{t}
\inf_{q\in \br^m}
\left(\average_{D_t} |u_\e-E_t x -q|^2\right)^{1/2}
+\frac{C}{s} \inf_{q\in \br^m}
\left(\average_{D_s} |u_\e- E_sx -q|^2\right)^{1/2}\\
&\le C \big\{ H(t) +H(s)\big\}\\
&\le C H(2r),
\endaligned
$$
we obtain 
\begin{equation}\label{NP0-2}
\max_{r\le t, s\le 2r} |h(t)-h(s)|\le C H(2r).
\end{equation}
Furthermore, by (\ref{NP-Lip-10}), 
\begin{equation}\label{NP0-3}
H(\theta r) \le \frac12 H(r) +C \left(\frac{\e}{r} \right)^\alpha \Phi (2r)
\end{equation}
for $r\in [\e, 1]$, where $\alpha\in (0,1/2)$ and 
$$
\Phi (t)=\left\{ \frac{1}{t} \inf_{q\in \mathbb{R}^m} 
\left(\average_{D_{t}} |u_\varep -q|^2 \right)^{1/2}
+t\left(\average_{D_{t}} |F|^p\right)^{1/p}
+\| g\|_{L^\infty(\Delta_{t})} \right\}.
$$
It is easy to see that
$$
\Phi(t) \le C \big\{ H(t) +h(t) \big\},
$$
which, together with (\ref{NP0-3}), leads to
\begin{equation}\label{NP0-4}
H(\theta r) \le \frac12 H(r) +C \left(\frac{\e}{r} \right)^\alpha \big\{ H(2r) +h(2r) \big\}.
\end{equation}
Thus the functions $H(r)$ and $h(r)$ satisfy the conditions (\ref{g-c-1}), (\ref{g-c-2}) and
(\ref{g-c-3}). As a result, we obtain that for $r\in [\e, 1]$,
$$
\aligned
\inf_{q\in \br^m} \frac{1}{r} \left(\average_{D_r} |u_\e -q|^2\right)^{1/2}
&\le C \big\{ H(r) +h(r)\big\}\\
&\le C \big\{ H(1) +h(1)\big\}.
\endaligned
$$

By taking $E=0$ and $q=0$,  we see that
$$
H(1)  \le C \left\{ 
\left(\average_{D_1} |u_\e |^2\right)^{1/2}
+\| F\|_{L^p(D_1)} +\| g\|_{L^\infty(\Delta_1)} +\| g\|_{C^{0, \rho} (\Delta_1)}\right\}.
$$
Also, note that
$$
\aligned
h(1) & \le C\inf_{q\in \mathbb{R}^m}
\left(\average_{D_{1}} | E_1 x +q|^2\right)^{1/2}\\
 &\le C \left\{ H(1)+\left(\average_{D_1} |u_\varep|^2 \right)^{1/2} \right\}.
\endaligned
$$
Hence we have proved that for $\e\le r\le 1$,
\begin{equation}\label{NP-Lip-20}
\inf_{q\in \mathbb{R}^m}
\frac{1}{r} \left(\average_{D_r} |u_\varep -q|^2\right)^{1/2}
\le C \left\{ \left(\average_{D_2} |u_\varep|^2 \right)^{1/2} 
+\|F\|_{L^p(D_2)}
+\| g\|_{L^\infty(\Delta_2)} 
+\| g\|_{C^{0, \rho} (\Delta_2)} \right\}.
\end{equation}
Replacing $u_\e$ by $u_\e-\average_{D_2} u_\e$ in the estimate above and using Poincar\'e inequality,
we obtain 
\begin{equation}\label{NP0-20}
\inf_{q\in \mathbb{R}^m}
\frac{1}{r} \left(\average_{D_r} |u_\varep -q|^2\right)^{1/2}
\le C \left\{ \left(\average_{D_2} |\nabla u_\varep|^2 \right)^{1/2} 
+\|F\|_{L^p(D_2)}
+\| g\|_{L^\infty(\Delta_2)} 
+\| g\|_{C^{0, \rho} (\Delta_2)} \right\}.
\end{equation}
This, together with Caccioppoli's inequality (\ref{Cacciopoli-4.1}), gives (\ref{bl-l-0}).
\end{proof}

We end this section with the proof of Lemma \ref{g-lemma}.

\begin{proof}[\bf Proof of Lemma \ref{g-lemma}]

It follows from the second inequality in (\ref{g-c-1}) that
$$
h(r)\le h(2r) + C_0 H(2r)
$$ 
for any $r\in [\e, 1/2]$.
Hence,
$$
\aligned
\int_a^{1/2} \frac{h(r)}{r} dr
&\le \int_a^{1/2} \frac{h(2r)}{r} dr +C_0 \int_a^{1/2}\frac{H(2r)}{r} dr\\
&=\int_{2a}^{1} \frac{h(r)}{r} dr +C_0 \int_{2a}^{1}\frac{H(2r)}{r} dr,
\endaligned
$$
where $a\in [\e, 1/4]$.
This implies that 
$$
\aligned
\int_a^{2a} \frac{h(r)}{r} dr
&\le \int^1_{1/2} \frac{h(2r)}{r} dr +C_0 \int_a^{1/2}\frac{H(2r)}{r} dr\\
&\le C_0 \big\{ H(1) + h(1) \big\} +C_0 \int_a^{1/2}\frac{H(2r)}{r} dr,
\endaligned
$$
which, together with (\ref{g-c-1}), leads to
$$
H(a)+h(a) \le C \Big\{ H(2a) + h(1) + H(1) +\int_{2a}^1 \frac{H(r)}{r} dr \Big\}
$$
for any $a\in [\e, 1/4]$. By the first inequality in (\ref{g-c-1}) we  may further deduce that 
\begin{equation}\label{g-c-40}
H(a) +h(a) \le C \Big\{ H(1) + h(1) +\int_a^1 \frac{H(r)}{r} dr \Big\}
\end{equation}
for any $a\in [\e, 1]$.

To bound the integral in the RHS of (\ref{g-c-40}),
we use (\ref{g-c-2}) and (\ref{g-c-40}) to obtain 
$$
H(\theta r) \le \frac12 H(r) + C \beta(\e/r) \Big\{ H(1) + h(1) \Big\}
+C \beta(\e/r) \int_r^1 \frac{H(t)}{t}\, dt
$$
for $r\in [\e, 1/2]$.
It follows that
\begin{equation}\label{g-c-41}
\int_{\alpha \theta \e}^\theta \frac{H(r)}{r} dr
\le 
\frac12 \int_{\alpha\e}^1 \frac{H(r)}{r} dr
+ C \big\{ H(1) + h(1) \Big\}
+ C \int_{\alpha \e}^1 \beta(\e/r) \left\{ \int_r^1 \frac{H(t)}{t} dt \right\} \frac{dr}{r},
\end{equation}
where $\alpha>1$ and we have used the condition (\ref{g-c-3}) on $\beta (t)$ for 
$$
\int_{\alpha\e}^1 \beta(\e/r) \frac{dr}{r}
=\int_\e^{\frac{1}{\alpha} }\frac{\beta(t)}{t} dt
\le \int_0^1 \frac{\beta(t)}{t} dt<\infty.
$$
Note that by Fubini's Theorem,
$$
\aligned
\int_{\alpha \e}^1 \beta(\e/r) \left\{ \int_r^1 \frac{H(t)}{t} dt \right\} \frac{dr}{r}
&=\int_{\alpha \e}^1 \left\{ \int_{\alpha \e}^t  \beta(\e/r) \frac{dr}{r} \right\} \frac{H(t)}{t} dt\\
&=\int_{\alpha \e}^1 H(t) \left\{ \int_{\e/t}^{1/\alpha} \frac{\beta (s)}{s} ds \right\} \frac{dt}{t}\\
& \le \int_0^{1/\alpha}  \frac{\beta(s)}{s} ds \int_{\alpha \e}^1 \frac{H(t)}{t} dt\\
&\le \frac{1}{4C} \int_{\alpha \e}^1 \frac{H(t)}{t} dt,
\endaligned
$$
if $\alpha>1$, which only depends on $C_0$ and the function $\beta$, is sufficiently large.
In view of (\ref{g-c-41}) this gives
$$
\int_{\alpha \theta \e}^\theta \frac{H(r)}{r} dr
\le 
\frac12 \int_{\alpha\e}^1 \frac{H(r)}{r} dr
+ C \big\{ H(1) + h(1) \Big\}
+\frac{1}{4} \int_{\alpha \e}^1 \frac{H(t)}{t} dt.
$$
It follows that
$$
\int_{\alpha \theta \e}^\theta \frac{H(r)}{r} dr \le  C \big\{ H(1) + h(1) \Big\},
$$
which, by (\ref{g-c-1}) and (\ref{g-c-40}), yields
$$
\aligned
H(r) +h(r)
& \le C \left\{ H(1) + h(1) +\int_r^1 \frac{H(t)}{t} dt \right\}\\
&\le C \big\{ H(1) + h(1) \big\}
\endaligned
$$
for any $r\in [\e, 1]$.
\end{proof}



\section{Matrix of Neumann functions}\label{section-4.4}

Assume that $A$ is 1-periodic and satisfies the ellipticity condition (\ref{s-ellipticity}) and the VMO condition (\ref{VMO-1}).
Suppose that either $\mathcal{L}_\e (u_\e)=0$ or $\mathcal{L}^*_\e (u_\e)=0$ in $2B=B(x_0, 2r)$. 
It follows from interior H\"older estimate (\ref{estimate-2.3.1}) that
$$
\| u_\e\|_{C^{0, \rho} (B)} \le C r^{-\rho} \left(\average_{2B} |u_\e|^2\right)^{1/2}
$$
for any $\rho\in (0,1)$, where $C$ depends only on $\mu$, $\rho$, and the function $\omega(t)$ in (\ref{VMO-1}).
This allows one to construct an $m\times m$ matrix of Neumann functions in a bounded Lipschitz domain $\Omega$
in $\br^d$, 
$$
N_\e (x, y)= \big( N_\e^{\alpha \beta} (x, y) \big),
$$
with the following properties:

\begin{itemize}

\item 

For $d\ge 3$, one has
$$
|N_\e (x, y)|\le  C |x-y|^{2-d},
$$
for  $x, y\in \Omega$ with $ |x-y|< \frac12 \text{dist}(y, \partial\Omega)$.
If  $d=2$, then
$$
|N_\varep (x,y)|\le C \big\{ 1+\ln [r_0 |x-y|^{-1}]\big\},
$$
for any $x, y\in \Omega$, where $r_0=\text{diam}(\Omega)$. 

\item

 If $F\in L^p(\Omega; \br^m)$ for some $p>\frac{d}{2}$ and $g\in L^2(\partial\Omega; \br^m)$
satisfy the compatibility condition  $\int_\Omega F\, dx +\int_{\partial\Omega} g=0$, then
\begin{equation}\label{N-R-1}
u_\e (x)=\int_\Omega N_\e (x, y) F(y) \, dy
+\int_{\partial\Omega} N_\e (x, y) g(y)\, d\sigma (y)
\end{equation}
is the unique weak solution in $H^1(\Omega; \br^m)$ of the Neumann problem,
$
\mathcal{L}_\e (u_\e)=F \text{ in } \Omega \text{ and } 
\frac{\partial u_\e}{\partial \nu_\e} =g  \text{ on }\partial\Omega,
$
with $\int_{\partial\Omega} u_\e\, d\sigma =0$.

\item 

Let $N^*_\e (x, y)$ denote the matrix of Neumann functions for the adjoint operator $\mathcal{L}_\e^*$ in $\Omega$. Then
\begin{equation}\label{a-Neumann}
N_\e^* (x, y)=\big(N_\e (y, x)\big)^T \quad \text{ for any } x, y\in \Omega.
\end{equation}

\end{itemize}

\begin{thm}\label{NF-10}
Suppose that $A$ is 1-periodic and satisfies (\ref{s-ellipticity}) and (\ref{VMO-1}).
Let $\Omega$ be a bounded $C^1$ domain in $\br^d$.
Then, if $d\ge 3$,
\begin{equation}\label{NF-11}
|N_\e (x, y)|\le C |x-y|^{2-d},
\end{equation}
for any $x, y\in \Omega$ and $x\neq y$. Moreover, if $d\ge 2$,
\begin{equation}\label{NF-12}
|N_\e (x, y)-N(z, y)|  + |N_\e(y, x)-N_\e (y, z)|
\le \frac{C |x-z|^\rho}{|x-y|^{d-2+\rho}}
\end{equation}
for any $x, y, z\in \Omega$ such that $|x-z|<\frac12 |x-y|$.
The constant  $C$ depends at most on $\mu$, $\rho$, $\omega(t)$ and $\Omega$.
\end{thm}

\begin{proof}
Since $\Omega$ is $C^1$, solutions of $\mathcal{L}_\e (u_\e)=0$ in $B(x_0, r)\cap \Omega$ with 
$\frac{\partial u_\e}{\partial\nu_\e}=0$ on $B(x_0, r)\cap \partial\Omega$, where
$x_0\in \partial\Omega$ and $0<r<r_0$, satisfies the estimate
\begin{equation}\label{NL-13}
\| u_\e\|_{C^{0, \rho} (B(x_0, r/2)\cap \Omega)} \le C r^{-\rho} \left(\average_{B(x_0, r)\cap\Omega} |u_\e|^2\right)^{1/2}
\end{equation}
for any $\rho \in (0,1)$, where $C$ depends only on $\rho$, $\mu$, $\omega(t)$ and $\Omega$.
This is a consequence of Theorem \ref{bh-n}, which also gives (\ref{NL-13}) for solutions of
$\mathcal{L}^*_\e (u_\e)=0$ in $B(x_0, r)\cap \Omega$ with
$\frac{\partial u_\e}{\partial \nu_\e^*}=0$ on $B(x_0, r)\cap\partial\Omega$.
The estimates (\ref{NF-11})-(\ref{NF-12}) now follow from  general results in \cite{Brown-2013} for $d= 2$
and in \cite{CHoi-Kim-2013} for $d\ge 3$.
\end{proof}

Using interior and boundary Lipschitz estimates,
stronger estimates may be proved in $C^{1, \eta}$ domains under the assumption that $A$ is H\"older 
continuous.

\begin{thm}\label{NF-L}
Assume that $A$ is 1-periodic and satisfies conditions (\ref{s-ellipticity}) and (\ref{smoothness}).
Let $\Omega$ be a bounded $C^{1, \eta}$ domain in $\br^d$ for some $\eta>0$.
Then for any $x, y\in \Omega$ and $x\neq y$, 
\begin{equation}\label{NFL-1}
|\nabla_x N_\varep(x,y)|+|\nabla_y N_\varep(x,y)|\le C |x-y|^{1-d},
\end{equation}
and 
\begin{equation}\label{NFL-2}
|\nabla_y\nabla_x N_\varep (x,y)|\le C |x-y|^{-d},
\end{equation}
where $C$ depends only on $\mu$, $(\lambda, \tau)$ and $\Omega$.
\end{thm}

\begin{proof}
Suppose $d\ge 3$.
Fix $x_0, y_0\in \Omega$ and let $r=|x_0 -y_0|/8$.
Let $u_\e (x) =N_\e (x, y_0)$.
Then $\mathcal{L}_\e (u_\e)=0$ in $B(x_0, 4r)\cap\Omega$ and
$$
\frac{\partial u_\e}{\partial \nu_\e}= -\frac{1}{|\partial \Omega|} I_{m\times m} \quad \text{ on } B(x_0, 4r)\cap\partial\Omega,
$$
where $I_{m\times m }$ denotes the $m\times m$ identity matrix.
By the boundary Lipschitz estimate (\ref{NP-Lip-1}) it follows that
$$
\aligned
|\nabla u_\e (x_0)|
&\le \frac{C}{r} \left(\average_{B(x_0, 4r)\cap\Omega} |u_\e|^2 \right)^{1/2}
+ C\\
&\le \frac{C}{r^{d-1}},
\endaligned
$$
where we have used the size estimate (\ref{NF-11}) for the last inequality.
This gives $|\nabla_x N_\e (x_0, y_0)|\le C r^{1-d}$.
Thus we have proved that $|\nabla_x N_\e (x, y)|\le C |x-y|^{1-d}$.
Since the same argument also yields $|\nabla_x N^*_\e (x, y)|\le C |x-y|^{1-d}$,
in  view of (\ref{a-Neumann}), we obtain 
$$
|\nabla_y N_\e (x, y)|=|\nabla_y N^*_\e (y, x)|\le C |x-y|^{1-d}.
$$

To see (\ref{NFL-2}), we note that the boundary Lipschitz estimate  gives
$$
\aligned
|\nabla_x  N_\varep (x_0,y_1)-\nabla_x N_\varep (x_0,y_2)\big\} | & \le  Cr^{-1}\max_{z\in B(x_0,r)\cap\Omega} 
|N_\varep (z,y_1)-N_\varep (z, y_2)|\\
& \le \frac{C|y_1 -y_2|}{r^d},
\endaligned
$$
where $y_1, y_2\in B(y_0, r)$.
It follows that $|\nabla_y\nabla _x N_\e (x_0, y_0)|\le C r^{-d}$.

Finally, in the case $d=2$, we  apply the Lipschitz estimate to 
$u_\e (x)=N_\e (x, y_0)-N_\e(x_0, y_0)$ and use the fact that
$|N_\e (x, y_0)-N_\e (x_0, y_0)|\le C$ if $|x-x_0|<(1/2)|x_0-y_0|$.
\end{proof}


\section{Elliptic systems of linear elasticity}\label{N-EE}

A careful inspection of the proof in the previous sections in this chapter shows that all results, with 
a few minor modifications, hold for the elliptic system of elasticity.
In particular, we obtain the uniform boundary H\"older and $W^{1, p}$ estimate  in $C^1$ domains as well as
the uniform Lipschitz estimate in $C^{1, \eta}$ domains.

We start with a Caccioppoli inequality for elliptic systems of elasticity with Neumann conditions.

\begin{lemma}[Caccioppoli's inequality]\label{Ca-E-N}
Suppose that $A\in E(\kappa_1, \kappa_2)$.
Let $u_\e\in H^1(D_{2r}; \br^d)$ be a solution of $\mathcal{L}_\e (u_\e) =F$ in $D_{2r}$
with $\frac{\partial u_\e}{\partial \nu_\e} =g$ on $\Delta_{2r}$.
Then the estimate (\ref{Cacciopoli-4.1}) holds with constant $C$ depending only on
$\kappa_1$, $\kappa_2$ and $M_0$.
\end{lemma}

\begin{proof}
We follow the same line of argument as in the proof of Lemma \ref{c-n-lemma}.
The special case where $g=0$ may be handled  in the same manner with help of the second 
Korn inequality. To deal with the general case, in view of the proof of Lemma \ref{c-n-lemma},
it suffices to construct a function $\widetilde{g}\in L^2(\partial D_2; \br^d)$ such that
$\widetilde{g}=g$ on $\Delta_2$, 
$$
\|\widetilde{g}\|_{L^2(\partial D_2)} \le C \| g\|_{L^2(\Delta_2)} +\| F\|_{L^2(D_2)},
$$
 and 
$\widetilde{g}$ satisfies the compatibility condition,
\begin{equation}\label{Ca-E-10}
\int_{\partial D_2} \widetilde{g} \cdot \phi\, d\sigma
+\int_{D_2} F \cdot \phi \, dx =0
\end{equation}
for any $\phi\in \mathcal{R}$.
To this end,  we let
$$
\widetilde{g}=\alpha_1 \phi_1 +\alpha_2 \phi_2 +\cdots +\alpha_N \phi_N \quad \text{ on } \partial D_2 \setminus \Delta_2,
$$
 where  $N=d(d+1)/2$, 
 $(\alpha_1, \alpha_2, \dots, \alpha_N)\in \br^N$ is to be determined, and  $\{\phi_1, \phi_2, \dots, \phi_N\}$ is an orthonormal basis
 of $\mathcal{R}$ in $L^2(D_2; \br^d)$. To determine $(\alpha_1, \alpha_2, \dots, \alpha_N)$, 
 we solve the $N\times N$ system of linear equations,
 $$
 \int_{\partial D_2 \setminus \Delta_2}
 \big( \alpha_1 \phi_1 +\alpha_2 \phi_2 +\cdots +\alpha_N \phi_N \big) \cdot \phi_j \, d\sigma
 =-\int_{D_2} F \cdot \phi_j \, dx
 -\int_{\Delta_2} g \cdot \phi_j \, d\sigma
 $$
 for $  j=1, 2, \dots, N$.  The linear system is uniquely solvable, provided that
 \begin{equation}\label{com-50}
 \det \left(\int_{\partial D_2 \setminus \Delta_2} \phi_i \cdot \phi_j \, d\sigma \right)
 \neq 0.
 \end{equation}
To see (\ref{com-50}), let's assume that it is not true. Then there exists $(\beta_1, \beta_2, \dots, \beta_N)\in \br^N\setminus \{ 0\}$
such that
$$
\int_{\partial D_2 \setminus \Delta_2} |\beta_i \phi_i|^2\, d\sigma =
\beta_i\beta_j \int_{\partial D_2\setminus \Delta_2} \phi_i \cdot \phi_j \, d\sigma=0,
$$
which implies that $\beta_i \phi_i=0$ on $\partial D_2 \setminus \Delta_2$.
Since $\beta_i\phi_i$ is a linear function and $\partial D_2 \setminus \Delta_2$ cannot be a hyperplane,
we may conclude that $\beta_i \phi_i\equiv 0$ in $\br^d$. Consequently, $\beta_1=\beta_2 =\cdots =\beta_N=0$,
which gives us a contradiction. 
\end{proof}

The next theorem is an analogous of Theorem \ref{4.0-0-1}.

\begin{thm}\label{A-E-10}
Suppose that $A\in E(\kappa_1, \kappa_2)$ and is 1-periodic.
Let $u_\e \in H^1(D_{2r}, \br^d)$ be a weak solution of $\mathcal{L}_\e (u_\e)=F$ in
$D_{2r}$ with $\frac{\partial u_\e}{\partial \nu_\e}=g$ on $\Delta_{2r}$,
where $F\in L^2(D_{2r}; \br^d)$ and $g\in L^2(\Delta_{2r}; \br^d)$.
Assume that $r\ge \e$.
Then there exists $w\in H^1(D_r; \br^d)$ such that
$\mathcal{L}_0 (w)=F$ in  $D_r$, 
$\frac{\partial w}{\partial \nu_0} =g$ on  $ \Delta_r$, 
and
\begin{equation}\label{A-E-11}
\left(\average_{D_r} |u_\e -w|^2\right)^{1/2}
\le C \left(\frac{\e}{r} \right)^{1/2}
\left\{ \left(\average_{D_{2r}} |u_\e|^2\right)^{1/2} 
+ r^2 \left(\average_{D_{2r}} |F|^2\right)^{1/2} 
+ r \left(\average_{\Delta_{2r}} |g|^2\right)^{1/2}\right\},
\end{equation}
where $C$ depends only on $\mu$ and $M_0$.
\end{thm}

\begin{proof}
The proof is similar to that of Theorem \ref{4.0-0-1}, using Lemma \ref{Ca-E-N} and 
estimate (\ref{4.0-1-1}).
Note that by (\ref{c-2-3-00}), the estimate (\ref{4.0-1-1}) holds for $\sigma=1/2$,
where $\mathcal{L}_\e (u_\e)=\mathcal{L}_0 (u_0)$ in $\Omega$,
$\frac{\partial u_\e}{\partial \nu_\e} =\frac{\partial u_0}{\partial \nu_0}$ on $\partial\Omega$, and
$$
\int_\Omega (u_\e -u_0)\cdot \phi \, dx =0,
$$
for any $\phi \in \mathcal{R}$.
\end{proof}

\begin{thm}[H\"older estimate]\label{E-H}
Suppose that $A\in E(\kappa_1, \kappa_2)$ is 1-periodic and satisfies (\ref{VMO-1}).
Let $\Omega$ be a bounded $C^1$ domain in $\br^d$ and $0<\rho<1$.
Let $u_\varep \in H^1(B(x_0,r)\cap\Omega;\br^d)$ be a solution of
$\mathcal{L}_\varep (u_\varep)=0$ in $B(x_0, r)\cap\Omega$
with $\frac{\partial u_\varep}{\partial\nu_\varep} =g$
on $B(x_0,r)\cap\partial\Omega$
for some $x_0\in \partial\Omega$ and $0<r<r_0$. 
Then estimate (\ref{local-holder-Neumann-estimate}) holds for any $x, y\in B(x_0, r/2)\cap\Omega$,
where the constant $C$ depends only on $\mu$, $\rho$, $\kappa_1$, $\kappa_2$, and $\Omega$.
\end{thm}

\begin{proof}
With Theorem \ref{A-E-10} at our disposal,
the proof is the same as that for Theorem \ref{bh-n}.
\end{proof}

\begin{thm}[$W^{1, p}$ estimate] \label{E-W-1-p}
Suppose that $A\in E(\kappa_1, \kappa_2)$ is 1-periodic and satisfies (\ref{VMO-1}).
Let $\Omega$ be a bounded $C^1$ domain in $\br^d$ and $1<p<\infty$.
Let $F\in L^q(\Omega; \br^d)$, $f\in L^p(\Omega; \br^{d\times d})$ and $g\in B^{-1/p, p}(\partial\Omega; \br^d)$
satisfy the compatibility condition (\ref{compatibility}), where $q$ given by (\ref{pq-100}).
Then the Neumann problem (\ref{Neumann-problem-4.2}) has a solution $u_\e$  in 
$W^{1, p}(\Omega; \br^d)$ such that the estimate (\ref{estimate-4.2}) holds
for some constant $C_p$ depending only on $p$, $\kappa_1$, $\kappa_2$, 
$\omega(t)$ in (\ref{VMO-1}), and $\Omega$.
The solution is unique in $W^{1, p}(\Omega; \br^d)$, up to an element of $\mathcal{R}$.
\end{thm}

\begin{proof}
The proof follows the same line of argument used for Theorem \ref{W-1-p-Neumann-theorem}.
Because of the compatibility condition (\ref{compatibility}),  some modifications are needed.

{\bf Step One.}
Consider the Neumann problem:
\begin{equation}\label{com-9}
\mathcal{L}_\e (u_\e) =\text{\rm div} (f) \quad \text{ in } \Omega
\quad \text{ and } \quad \frac{\partial u_\e}{\partial \nu_\e} =-n\cdot f
\quad \text{ on } \partial\Omega,
\end{equation}
which has a unique solution in $H^1(\Omega; \br^d)$ such that
$u_\e \perp \mathcal{R}$ in $L^2(\Omega; \br^d)$, provided that 
$f=(f_i^\alpha)\in L^2(\Omega; \br^{d\times d})$ satisfies the compatibility condition
$$
\int_\Omega f_i^\alpha \frac{\partial\phi^\alpha}{\partial x_i} \, dx=0
$$
for any $\phi=(\phi^\alpha)\in \mathcal{R}$. The condition is equivalent to
\begin{equation}\label{com-11}
\int_\Omega \big (f_i^\alpha -f_\alpha^i) \, dx=0.
\end{equation}
As a result, Theorem \ref{real-variable-operator-Lipschitz-theorem} can not be applied directly.
Rather, we use Theorem \ref{real-variable-Lipschitz-theorem} to show that
\begin{equation}\label{com-12}
\|\nabla u_\e\|_{L^p(\Omega)} \le C_p \| f\|_{L^p(\Omega)}
\end{equation}
for $1<p<\infty$.

To this end we fix a ball $B$ with $|B|\le c_0|\Omega|$.
Assume that either $4B\subset \Omega$  or $B$ is centered on $\partial\Omega$.
We write $u_\e =v_\e +w_\e$ in $\Omega$, where $v_\e$ is the unique solution in $H^1(\Omega; \br^d)$ to the 
Neumann problem,
\begin{equation}\label{com-13}
\mathcal{L}_\e (v_\e) =\text{\rm div} \big( (f-E) \chi_{4B\cap \Omega} \big) \quad \text{ in } \Omega
\quad \text{ and } \quad \frac{\partial v_\e}{\partial \nu_\e}=-n \cdot (f-E)\chi_{4B\cap \Omega} \quad \text{ on } \partial\Omega,
\end{equation}
such that $v_\e \perp \mathcal{R}$ in $L^2(\Omega; \br^d)$,
where $E=(E_i^\alpha)$ is a constant with 
$$
E_i^\alpha =\frac12 \average_{4B\cap \Omega} \big( f_i^\alpha -f_\alpha^i)
$$
It is easy to verify  that the function $(f-E)\chi_{4B\cap\Omega}$ satisfies the compatibility condition (\ref{com-11}). Thus,
$$
\|\nabla v_\e\|_{L^2(\Omega)} \le C \|  f-E\|_{L^2(4B\cap \Omega)}
\le C \| f\|_{L^2(4B\cap \Omega)},
$$
which leads to
\begin{equation}\label{com-14}
\left(\average_{4B\cap \Omega} |\nabla v_\e|^2 \right)^{1/2}
\le C \left(\average_{4B\cap \Omega} |f|^2\right)^{1/2}.
\end{equation}

To estimate $\nabla w_\e$ in $2B\cap \Omega$,
we first consider the case where $4B\subset \Omega$.
 Since $\mathcal{L}_\e (w_\e)=0$ in $4B$, we may use the interior 
$W^{1, p}$ estimate in Theorem \ref{interior-W-1-p-theorem} to obtain 
$$
\aligned
\left(\average_{2B} |\nabla w_\e|^p \right)^{1/p}
& \le C \left(\average_{4B} |\nabla w_\e|^2 \right)^{1/2}\\
& \le C \left(\average_{4B} |\nabla u_\e|^2 \right)^{1/2} +
\left(\average_{4B} |\nabla v_\e|^2 \right)^{1/2}\\
&\le C  \left(\average_{4B} |\nabla u_\e|^2 \right)^{1/2} +
C \left(\average_{4B} |f|^2\right)^{1/2},
\endaligned
$$
for $2<p<\infty$, where we have used (\ref{com-14}) for the last step.
If $B$ is centered on $\partial\Omega$, we observe that $w_\e$ satisfies 
$$
\mathcal{L}_\e (w_\e) =0 \quad \text{ in } 4B\cap \Omega \quad 
\text{ and } \quad \frac{\partial w_\e}{\partial \nu_\e} =n \cdot E \quad \text{ on } 4B \cap \partial\Omega.
$$
It follows by Theorem \ref{E-H} that 
$$
|w_\e (x) -w_\e (y) |
\le C r \left(\frac{|x-y|}{r} \right)^\rho
\left\{ \left(\average_{4B\cap \Omega} |\nabla w_\e|^2\right)^{1/2} +|E| \right\}
$$
for any $x, y\in 2B\cap \Omega$, where $0<\rho<1$.
As in the proof of Theorem \ref{W-1-p-Neumann-theorem},
 this, together with the interior $W^{1, p}$ estimate, yields that
$$
\aligned
\left(\average_{2B\cap \Omega} |\nabla w_\e|^p \right)^{1/p}
& \le C \left\{ \left(\average_{4B\cap \Omega} |\nabla w_\e|^2 \right)^{1/2}
+|E| \right\}\\
& \le C \left\{ \left(\average_{4B\cap \Omega} |\nabla u_\e|^2 \right)^{1/2}
+ \left(\average_{4B\cap \Omega} |\nabla v_\e|^2 \right)^{1/2}
+|E| \right\}\\
&
\le C \left\{ \left(\average_{4B\cap \Omega} |\nabla u_\e|^2 \right)^{1/2}
+ \left(\average_{4B\cap \Omega} |f|^2\right)^{1/2} \right\}.
\endaligned
$$
By Theorem \ref{real-variable-Lipschitz-theorem} it follows that
$$
\left(\average_\Omega | \nabla u_\e|^p\right)^{1/p}
\le C\left(\average_\Omega | \nabla u_\e|^2\right)^{1/2}
+ C \left(\average_\Omega | f |^p\right)^{1/p}
$$
for any $2<p<\infty$.
Using
$$
\|\nabla u_\e\|_{L^2(\Omega)} \le C \| f\|_{L^2(\Omega)} \le C \| f\|_{L^p(\Omega)},
$$
we obtain  (\ref{com-12}) for $2< p<\infty$.

{\bf Step Two.} 
To prove the estimate (\ref{com-12}) for $1<p<2$, we use a duality argument.
Let $f=(f_i^\alpha)\in L^2(\Omega; \br^{d\times d})$,
$g=(g_i^\alpha)\in L^{p^\prime}(\Omega; \br^{d\times d})$ satisfy the compatibility condition 
(\ref{com-11}). Let $u_\e, v_\e$ be  weak solutions of (\ref{com-9}) with data $f, g$, respectively. 
Then
$$
\int_\Omega g_i^\alpha \frac{\partial u_\e^\alpha}{\partial x_i}\, dx
=\int_\Omega A(x/\e)\nabla u_\e \cdot \nabla v_\e\, dx
=\int_\Omega f_i^\alpha \frac{\partial v_\e^\alpha}{\partial x_i}\, dx.
$$
Assume that $v_\e\perp \mathcal{R}$ in $L^2(\Omega; \br^{d\times d})$.
By Step One,  $\| \nabla v_\e\|_{L^{p^\prime}(\Omega)} \le C \| g\|_{L^{p^\prime}(\Omega)}$.
It follows that
$$
\Big| \int_\Omega g_i^\alpha \frac{\partial u_\e^\alpha}{\partial x_i}\, dx \Big|
\le C \, \| g\|_{L^{p^\prime}(\Omega)} \| f\|_{L^p(\Omega)}.
$$
By duality this implies that
\begin{equation}\label{com-20}
\|\nabla u_\e\|_{L^p(\Omega)}
\le C \| f\|_{L^p(\Omega)} 
+ C \Big| \int_\Omega \big( \nabla u_\e -(\nabla u_\e)^T\big) \Big|.
\end{equation}
We may eliminate the last term in the inequality above by subtracting an element of $\mathcal{R}$ from $u_\e$.
The duality argument above also gives the uniqueness in $W^{1, p}(\Omega; \br^d)$,
up to an element in $\mathcal{R}$.
As a result,  by a density argument,  for any $ f \in L^p(\Omega; \br^d)$ satisfying (\ref{com-12}),
there exists a weak solution $u_\e$  of (\ref{com-9}) in $W^{1, p}(\Omega; \br^d)$
such that $\|\nabla u_\e\|_{L^p(\Omega)} \le C \| f\|_{L^p(\Omega)}$.
The solution is unique in $W^{1, p}(\Omega; \br^d)$, up to an element of $\mathcal{R}$.

{\bf Step Three.} Let $1<p<\infty$.
Consider the Neumann problem
\begin{equation}\label{com-21}
\mathcal{L}_\e (u_\e) =0 \quad \text{ in } \Omega \quad \text{ and } \quad 
\frac{\partial u_\e}{\partial \nu_\e} =g \quad \text{ on } \partial\Omega,
\end{equation}
where $g\in B^{-1/p, p}(\partial \Omega; \br^d)$ satisfies the compatibility condition
\begin{equation}\label{com-22}
\langle g, \phi \rangle _{B^{-1/p, p}(\partial\Omega) \times B^{1/p, p^\prime}(\partial\Omega)} =0
\end{equation}
for any $\phi\in \mathcal{R}$. Then there exists a weak solution of (\ref{com-21}) in $W^{1, p}(\Omega; \br^d)$,
unique up to an element of $\mathcal{R}$, such that 
\begin{equation}\label{com-23}
\| \nabla u_\e\|_{L^p (\Omega)}
\le C \| g\|_{B^{-1/p, p}(\partial\Omega)}.
\end{equation}
This follows from Steps One and Two by a duality argument,
 similar to that in the proof of Lemma \ref{lemma-4.2-2}.

{\bf Step Four}. Let $1<p<\infty$.
Consider the Neumann problem
\begin{equation}\label{com-24}
\mathcal{L}_\e (u_\e) =F \quad \text{ in } \Omega \quad \text{ and } \quad 
\frac{\partial u_\e}{\partial\nu_\e}=h \quad \text{ on } \partial\Omega,
\end{equation}
where $F\in L^q(\Omega; \br^d)$, $q$ is give by (\ref{pq-100}), and $h\in \mathcal{R}$  satisfies  the compatibility condition
\begin{equation}\label{com-25}
\int_\Omega F\cdot \phi\, dx +\int_{\partial\Omega} h  \cdot \phi\, d\sigma =0
\end{equation}
for any $\phi \in \mathcal{R}$.
We first show that for any $F\in L^1(\Omega; \br^d)$, there exists a unique $h\in \mathcal{R}$ such that
$h$ satisfies (\ref{com-25}) and
\begin{equation}\label{com-26}
\| h\|_{C^1(\Omega)} \le C \| F\|_{L^1(\Omega)},
\end{equation}
where $C$ depends only on $\Omega$.
Indeed, let $\{ \phi_1, \phi_2, \cdots, \phi_N \}$ be an orthonormal basis of $\mathcal{R}$ in 
$L^2(\Omega; \br^d)$, where $N=\frac{d(d+1)}{2}$, and
$$
h=\alpha _1 \phi_1 +\alpha_2 \phi_2 +\dots \alpha_N \phi_N,
$$
 where $(\alpha_1, \alpha_2, \dots, \alpha_N)\in \br^N$.
 To find $(\alpha_1, \alpha_2, \dots, \alpha_N)$, we solve the $N\times N$ system
 of linear equations
 $$
 \alpha_i \int_{\partial\Omega} \phi_i \cdot \phi_j \, d\sigma=-\int_\Omega F\cdot \phi_j\, dx, \quad 
 j=1,2, \dots, N,
 $$
 which is uniquely solvable, provided that
 \begin{equation}\label{com-27}
 \det \left(\int_{\partial\Omega} \phi_i \cdot \phi_j\, d\sigma  \right) \neq 0.
 \end{equation}
 The proof of  (\ref{com-27}) is similar to that of (\ref{com-50}).
  Suppose  it is not true.
 Then there exists $(\beta_1, \beta_2, \cdots, \beta_N)\in \br^N$ such that
 $$
 \beta_i \int_{\partial\Omega} \phi_i \cdot \phi_j \, d\sigma =0.
 $$
 Let $v=\beta_1 \phi_1 +\cdots \beta_N \phi_N\in \mathcal{R}$.
 Then 
 $$
 \int_{\partial \Omega} |v|^2\, d\sigma =0,
 $$
which implies that $ v=0$ in $\partial\Omega$.
Since $v$ is a linear function and $\partial\Omega$ is not a hyperplane,
we obtain $v\equiv 0$ in $\br^d$. By linearly independence of $\phi_1,\phi_2, \dots, \phi_N$
in $L^2(\Omega; \br^d)$,
this leads to $\beta_1=\beta_2\cdots =\beta_N=0$ and gives us a contradiction.
 
 With the construction of $h$,
 the argument in  the proof of Lemma \ref{lemma-4.2-2} shows that 
  there exists a weak solution of (\ref{com-24}) in $W^{1, p}(\Omega; \br^d)$,
unique up to an element of $\mathcal{R}$, such that 
\begin{equation}\label{com-28}
\| \nabla u_\e\|_{L^p (\Omega)}
\le C \| F\|_{L^q(\Omega)}.
\end{equation}

{\bf Step Five.}
Theorem \ref{E-W-1-p} follows from steps 1-5 by writing 
$$
u_\e =u_\e^{(1)} +u_\e^{(2)} + u_\e^{(3)},
$$
where $u_\e^{(1)}$ is a solution of (\ref{com-9}) with data $f$,
$u_\e^{(2)}$ is a solution of (\ref{com-23}) with data $g-h$,
and $u_\e^{(3)}$ is a solution of (\ref{com-24}) with data $F$ and $h$.
In view of (\ref{com-25}), the boundary $g-h$ satisfies the compatibility condition 
(\ref{com-22}), and by (\ref{com-28}),
$$
\| g-h\|_{B^{-1/p, p}(\partial\Omega)}\le C \| g\|_{B^{-1/p, p}(\partial\Omega)} +\| F\|_{L^1(\Omega)}.
$$
It follows that
$$
\aligned
\|\nabla u_\e \|_{L^p(\Omega)}
&\le \|\nabla u_\e^{(1)}\|_{L^p(\Omega)}
+\|\nabla u_\e^{(2)}\|_{L^p(\Omega)}
+\|\nabla u_\e^{(3)}\|_{L^p(\Omega)}\\
&\le C \Big\{ \| f\|_{L^p(\Omega)} +\| F\|_{L^q(\Omega)} +\| g\|_{B^{-1/p, p}(\partial\Omega)} \Big\}.
\endaligned
$$
This completes the proof of Theorem \ref{E-W-1-p}.
\end{proof}

The next theorem gives the boundary Lipschitz estimate for the Neumann problem in a
$C^{1, \eta}$ domain.

\begin{thm}[Lipschitz estimate]\label{E-L-T}
Suppose that $A\in E(\kappa_1, \kappa_2)$ is 1-periodic and satisfies (\ref{smoothness}).
Let $\Omega$ be a bounded $C^{1, \eta}$ domain for some $\eta\in (0,1)$.
Let $u_\e\in H^1(B(x_0, r)\cap \Omega; \br^d)$ be a weak solution of
$\mathcal{L}_\e (u_\e)=F$ in $B(x_0, r)\cap\Omega; \br^d)$ with 
$\frac{\partial u_\e}{\partial \nu_\e} =g$ on $B(x_0, r)\cap \partial\Omega$, for some $x_0\in \partial\Omega$
and $0<r<r_0$.
Then the estimate (\ref{NP-Lip-1}) holds for $\rho\in (0, 1)$ and $p>d$,
where $C$ depends only on $\rho$, $p$, $\kappa_1$, $\kappa_2$, $(\lambda, \tau)$ in (\ref{smoothness}),
and $\Omega$.
\end{thm}

\begin{proof} 
With Theorem \ref{A-E-10},
the proof is the same as that of Theorem \ref{NP-Lip-theorem-4.0}.
\end{proof}

\begin{cor}\label{cor-E-L}
Assume that $A$ and $\Omega$ satisfy the same conditions as in Theorem \ref{E-L-T}.
Let $p>d$ and $\rho \in (0,1)$.
Let $u_\e \in H^1(\Omega; \br^d)$ be a weak solution to the Neumann problem,
$\mathcal{L}_\e (u_\e)=F$ in $\Omega$ and $\frac{\partial u_\e}{\partial \nu_\e}=g$ on $\partial\Omega$,
where $F\in L^p(\Omega; \br^d)$ and $g\in C^\rho(\partial\Omega; \br^d)$ satisfy the compatibility condition 
(\ref{compatibility}). Assume that $u_\e \perp \mathcal{R}$.
Then
\begin{equation}\label{E-L-00}
\|\nabla u_\e\|_{L^\infty(\Omega)}
\le C \Big\{ \| F\|_{L^p(\Omega)} + \| g\|_{C^\rho(\partial\Omega)} \Big\},
\end{equation}
where $C$ depends only on $p$, $\rho$, $\kappa_1$, $\kappa_2$, $(\lambda, \tau)$ and $\Omega$.
\end{cor}

\begin{proof}
The proof is the same as that of Corollary \ref{NPL-0}.
\end{proof}


\section{Notes}

The boundary H\"older and $W^{1, p}$ estimates for solutions with Neumann conditions 
were proved by C. Kenig, F. Lin, and Z. Shen in \cite{KLS-2013}.
Under the additional symmetry condition $A^*=A$, the boundary Lipschitz estimate for  Neumann problems 
was also established  in \cite{KLS-2013}.
This was achieved  by using a compactness method, similar to that used
in Chapter \ref{chapter-3}  for boundary Lipschitz estimates for solutions with Dirichlet conditions.
As we pointed out earlier, the compactness method reduces to  the problem to the Lipschitz estimate for the
correctors. In the case of Neumann correctors, the Lipschitz estimate was
 obtained in \cite{KLS-2013} by utilizing the nontangential-maximal-function estimates  in \cite{KS-2011-H, KS-2011-L}
 by C. Kenig and Z. Shen
 for solutions to Neumann problems with $L^2$ boundary data (see Chapter \ref{chapter-7}).
 As results in \cite{KS-2011-H, KS-2011-L} were proved in Lipschitz domains under the symmetry condition,
so did the main estimates in \cite{KLS-2013}. 

The symmetry condition was removed by S. N. Armstrong and Z. Shen in \cite{A-Shen-2016},
where the interior and boundary Lipschitz estimates were established for elliptic systems with
uniformly almost-periodic coefficients.
The approach used in \cite{A-Shen-2016} was developed by S. N. Armstrong and C. Smart  in \cite{Armstrong-2016}
for the study of large-scale regularity theory in stochastic homogenization. 
Also see related work in \cite{Gloria-2015, AM-2016, AKM-2016,  Armstrong-2017} and their references. 
 
The presentation in this chapter follows closely \cite{Shen-2017-APDE} by Z. Shen, where the boundary regularity estimates
were studied for elliptic systems of elasticity with periodic coefficients. In particular, Lemma \ref{g-lemma}, which 
improves the analogous results in \cite{Armstrong-2016, A-Shen-2016},
is taken from \cite{Shen-2017-APDE}.



\chapter{Convergence Rates, Part II}\label{chapter-6}

In Chapter \ref{chapter-C} we establish  the $O(\sqrt{\e})$ error estimates for some two-scale expansions
 in $H^1$ and the $O(\varep)$ convergence rate for solutions $u_\e$ in $L^2$. The results are obtained without any smoothness
assumption on the coefficient matrix $A$.
In this chapter we return to the problem of convergence rates and prove various results 
under some additional smoothness assumptions, using uniform regularity estimates obtained 
in Chapters \ref{chapter-2}-\ref{chapter-4}. We shall be mainly interested in the sharp $O(\varep)$ or near
sharp rates of convergence.

We start out in Section \ref{section-7.1} with an error estimate in $H^1$
for a two-scale expansion involving boundary correctors. 
 The result is used in Section \ref{section-7.6} to study the convergence rates 
of the Dirichlet eigenvalues for $\mathcal{L}_\varep$. 
 In Section \ref{section-6.2} we derive asymptotic 
expansions, as $\varep \to 0$, of Green functions $G_\varep (x,y)$ 
as well as their derivatives $\nabla_xG_\varep (x,y)$, $\nabla_y G_\varep (x,y)$,
and $\nabla_x\nabla_y G_\varep(x,y)$, using the Dirichlet correctors.
As a corollary, we also obtain an asymptotic expansion of the Poisson kernel $P_\e (x, y)$ for
$\mathcal{L}_\e$ in a $C^{2, \alpha}$ domain $\Omega$.
Analogous expansions are obtained for Neumann functions 
$N_\varep (x,y)$ and their derivatives in Section \ref{section-7.3}.
Results in Sections \ref{section-6.2} and \ref{section-7.3} are used in Section \ref{section-7.4} to establish
convergence rates of $u_\varep -u_0$ in $L^p$ and $u_\varep -u_0 -v_\varep$ in $W^{1,p}$
for $p\neq 2$, where $v_\varep$ is a first-order corrector. 



\section{Convergence rates in $H^1$ and $L^2$}\label{section-7.1}

For solutions of $\mathcal{L}_\varep (u_\varep)=F$ in $\Omega$
subject to Dirichlet condition $u_\varep=f$ or 
Neumann condition $\frac{\partial u_\varep}{\partial \nu_\varep} =g$ on $\partial\Omega$,
 it is proved in Chapter \ref{chapter-C} that
\begin{equation}\label{7.0-0}
\| u_\varep -u_0 -\varep \chi(x/\varep)\nabla u_0\|_{H^1(\Omega)}
\le C\sqrt{\varep}\, \| u_0\|_{W^{2, d}(\Omega)}, 
\end{equation}
where $\Omega$ is a bounded Lipschitz domain in $\rd$ .
If $d\ge 3$ and the corrector $\chi$ is bounded,
the estimate (\ref{7.0-0}) is improved to
\begin{equation}\label{7.0-1}
\| u_\varep -u_0 -\varep \chi(x/\varep)\nabla u_0\|_{H^1(\Omega)}
\le C\sqrt{\varep}\, \| u_0\|_{H^2(\Omega)}.
\end{equation}
See Theorems \ref{theorem-d-c-2-1} and \ref{theorem-c-3s}
Furthermore,
the  sharp $O(\varep)$ rate in $L^2(\Omega)$ is established in Sections \ref{section-c-4}
and \ref{section-c-5},
\begin{equation}\label{7.0-1a}
\| u_\varep -u_0\|_{L^2(\Omega)} \le C\, \varep\, \| u_0\|_{H^2(\Omega)},
\end{equation}
if  $\Omega$ is $C^{1,1}$.
 See Theorems \ref{C-2-thm-D} and \ref{C-2-thm-N}.
 
 Recall that the Dirichlet corrector 
 $\Phi_\varep =\big( \Phi_{\varep, j}^\beta \big)$, with $1\le j\le d$ and $1\le \beta\le m$, for $\mathcal{L}_\varep$
in $\Omega$ is defined by
\begin{equation}\label{DC-1.5}
\mathcal{L}_\varep \big( \Phi_{\varep, j}^\beta \big)  =0  \quad  \text{ in } \Omega \quad \text{ and } \quad
\Phi_{\varep, j}^\beta  =P_j^\beta \quad  \text{ on } \partial\Omega,
\end{equation}
where $P_j^\beta (x)=x_j e^\beta$. 
 The following theorem gives the $O(\varep)$ rate of convergence in $H^1_0(\Omega)$
 for the Dirichlet problem.

\begin{thm}\label{rate-1.5.1}
Suppose that $A$ is 1-periodic and satisfies (\ref{weak-e-1})-(\ref{weak-e-2}).
Let $\Omega$ be a bounded Lipschitz domain in $\rd$.
Assume that $\chi=(\chi_j^\beta)$ is H\"older continuous and that
 $\Phi_{\varep} =\big( \Phi_{\varep, j}^\beta\big)$ is bounded (if $m\ge 2$).
Let $u_\varep\in H^1(\Omega; \br^m)$ $(\varep\ge 0)$ be the weak solution of the Dirichlet problem:
$\mathcal{L}_\e (u_\e)=F$ in $\Omega$ and $u_\e=f$ on $\partial\Omega$.
Then,  if $u_0\in H^2(\Omega; \mathbb{R}^m)$,
\begin{equation}\label{1.5.1-0}
\big\| u_\varep -u_0 - \left\{ \Phi_{\varep, j}^\beta -P_j^\beta \right\} \frac{\partial u_0^\beta}{\partial x_j} \big\|_{H^1_0(\Omega)}
\le C \, \Big \{ \varep + \|\Phi_{\varep} -P \|_\infty \Big\} \|\nabla^2 u_0\|_{L^2(\Omega)},
\end{equation}
where $P=(P_j^\beta)$ and $C$ depends only on $A$ and $\Omega$.
\end{thm}

The proof of Theorem \ref{rate-1.5.1} uses energy estimates and  the formula in the following lemma.

\begin{lemma}\label{lemma-1.4.2}
Suppose that $u_\varep\in H^1(\Omega;\br^m)$ and $ u_0\in H^2(\Omega;\br^m)$.
Let
\begin{equation}\label{definition-of-w}
w_\varep (x) =u_\varep (x)-u_0 (x) -\left\{ V_{\varep, j}^\beta (x) -P_j^\beta (x)\right\}
 \frac{\partial u_0^\beta}{\partial x_j},
\end{equation}
where $V_{\varep, j}^\beta= (V_{\varep,j}^{1\beta}, \dots,
V_{\varep, j}^{m\beta})
 \in H^1(\Omega; \br^m)$ and $\mathcal{L}_\varep \big(V_{\varep, j}^\beta\big)
=0$ in $\Omega$ for each $1\le j\le d$ and $1\le \beta\le m$.
Assume that $\mathcal{L}_\e(u_\e)=\mathcal{L}_0(u_0)$ in $\Omega$.
Then
\begin{equation}\label{formula-1.4}
\aligned
\big(\mathcal{L}_\varep (w_\varep)\big)^\alpha
=&
 -\varep\, \frac{\partial}{\partial x_i}
\left\{ \phi_{jik}^{\alpha\gamma} \left({x}/{\varep}\right)
\frac{\partial^2 u_0^\gamma}{\partial x_j\partial x_k}\right\}\\
&+\frac{\partial}{\partial x_i}
\left\{ a_{ij}^{\alpha\beta}\left({x}/{\varep}\right)
\left[ V_{\varep, k}^{\beta\gamma}(x) 
-x_k \delta^{\beta\gamma}\right]
\frac{\partial^2 u_0^\gamma}{\partial x_j\partial x_k}\right\}\\
&
+ a_{ij}^{\alpha\beta} \left({x}/{\varep}\right)
\frac{\partial}{\partial x_j}
\left[ V_{\varep, k}^{\beta\gamma}(x)
-x_k \delta^{\beta\gamma}
-\varep \chi_k^{\beta\gamma}\left({x}/{\varep}\right) \right]
\frac{\partial^2 u_0^\gamma}{\partial x_i\partial x_k},
\endaligned
\end{equation}
where 
$\phi =\big( \phi_{kij}^{\alpha\beta} (y)\big)$ is the flux corrector, given by Lemma \ref{lemma-1.4.1}.
\end{lemma}

\begin{proof}
Note that
$$
\aligned
a_{ij}^{\alpha\beta}\left(\frac{x}{\varep}\right)
 \frac{\partial w^\beta_\varep}{\partial x_j}
=a_{ij}^{\alpha\beta} \left(\frac{x}{\varep}\right)
\frac{\partial u^\beta_\varep}{\partial x_j}
-& a_{ij}^{\alpha\beta} \left(\frac{x}{\varep}\right)
\frac{\partial u^\beta_0}{\partial x_j}
-a_{ij}^{\alpha\beta} \left(\frac{x}{\varep}\right)
\frac{\partial}{\partial x_j}
\left\{ V_{\varep, k}^{\beta\gamma}-x_k\delta^{\beta\gamma}
\right\}
\cdot \frac{\partial u_0^\gamma}{\partial x_k}\\
&-a_{ij}^{\alpha\beta}\left(\frac{x}{\varep}\right) 
\left\{ V_{\varep, k}^{\beta\gamma}-x_k\delta^{\beta\gamma}\right\}
 \frac{\partial^2 u^\gamma_0}{\partial x_k
\partial x_j},
\endaligned
$$
and
$$
\aligned
\big(\mathcal{L}_\varep (w_\varep)\big)^\alpha
=& \big(\mathcal{L}_\varep (u_\varep)\big)^\alpha - \big(\mathcal{L}_0 (u_0)\big)^\alpha
-\frac{\partial}{\partial x_i} \left\{ 
\left[ \hat{a}_{ij}^{\alpha\beta} -a_{ij}^{\alpha\beta} \left({x}/{\varep}\right)\right]
\frac{\partial u_0^\beta}{\partial x_j}\right\}\\
&\quad +\big\{ \mathcal{L}_\varep (V_{\varep,k}^\gamma -P_k^\gamma) \big\}^\alpha 
\cdot\frac{\partial u_0^\gamma}
{\partial x_k}
+{a}_{ij}^{\alpha\beta} \left( {x}/{\varep}\right) 
\frac{\partial}{\partial x_j} \left\{ V_{\varep, k}^{\beta\gamma}
-x_k \delta^{\beta\gamma} \right\} \cdot \frac{\partial^2 u_0^\gamma}
{\partial x_i \partial x_k}\\
&\quad  +\frac{\partial}{\partial x_i}
\left\{a_{ij}^{\alpha\beta}\left({x}/{\varep}\right) 
\left[ V_{\varep, k}^{\beta\gamma}-x_k\delta^{\beta\gamma}\right]
 \frac{\partial^2 u^\gamma_0}{\partial x_k
\partial x_j}\right\}.
\endaligned
$$
Using
$$
\mathcal{L}_\varep \big( V_{\varep, k}^\gamma -P_k^\gamma\big)
=-\mathcal{L}_\varep \big(P_k^\gamma\big)
=\mathcal{L}_\varep \big\{ \varep \chi_k^\gamma ({x}/{\varep})\big\},
$$
and $\mathcal{L}_\e (u_\e)=\mathcal{L}_0 (u_0)$, we obtain
\begin{equation}\label{1.4.2-1}
\aligned
\big(\mathcal{L}_\varep (w_\varep)\big)^\alpha
= &\frac{\partial}{\partial x_i} \left\{ b_{ij}^{\alpha\beta} \left(
{x}/{\varep}\right) \frac{\partial u_0^\beta}{\partial x_j}\right\}\\
& +a_{ij}^{\alpha\beta} \left({x}/{\varep}\right)
\frac{\partial}{\partial x_j} \left\{ V_{\varep, k}^{\beta\gamma} (x)
-x_k\delta^{\beta\gamma} - \varep \chi_k^{\beta\gamma}
 \left({x}/{\varep}\right)\right\}
 \cdot \frac{\partial^2 u_0^\gamma}
{\partial x_i \partial x_k}\\
& 
+\frac{\partial}{\partial x_i}
\left\{a_{ij}^{\alpha\beta}\left({x}/{\varep}\right) 
\left[ V_{\varep, k}^{\beta\gamma}-x_k\delta^{\beta\gamma}\right]
\cdot \frac{\partial^2 u^\gamma_0}{\partial x_k
\partial x_j}\right\},
\endaligned
\end{equation}
where $b_{ij}^{\alpha\beta} (y)$ is defined by (\ref{definition-of-b}).
The formula (\ref{formula-1.4}) now follows from the identity 
\begin{equation}\label{phi-identity}
\frac{\partial}{\partial x_i}
\left\{ b_{ij}^{\alpha\beta} (x/\varep)\,
\frac{\partial u_0^\beta}{\partial x_j} \right\}
=-\varep\, \frac{\partial}{\partial x_i} \left\{ \phi_{kij}^{\alpha\beta} 
\left({x}/{\varep}\right)  \frac{\partial^2 u_0^\beta}{\partial x_k\partial x_j}\right\},
\end{equation}
which is a consequence of (\ref{phi-identity-0}).
\end{proof}

\begin{proof}[\bf Proof of Theorem \ref{rate-1.5.1}]

Let 
$$
w_\varep=u_\varep -u_0 - \left\{ \Phi_{\varep, j}^\beta -P_j^\beta \right\} \frac{\partial u_0^\beta}{\partial x_j}.
$$
Since $\Phi_\varep$ is bounded, $\mathcal{L}_\varep (\Phi_\varep)=0$ and $u_0\in H^2(\Omega; \mathbb{R}^m)$,
one may use the argument in the proof of Lemma \ref{lemma-c-2-3}
 to show that $ |\nabla \Phi_\varep|  |\nabla u_0|\in L^2(\Omega)$.
This implies that $w_\varep \in H_0^1(\Omega; \br^m)$.
It follows from Lemma \ref{lemma-1.4.2} that
$$
\aligned
\big( \mathcal{L}_\varep (w_\varep)\big)^\alpha
&= -\varep\,  \frac{\partial}{\partial x_i} \left\{ \phi_{jik}^{\alpha\gamma} (x/\varep) \frac{\partial^2 u_0^\gamma}
{\partial x_j\partial x_k}\right\}\\
&\qquad +\frac{\partial}{\partial x_i}\left\{a_{ij}^{\alpha\beta} (x/\varep) \left[ 
\Phi_{\varep, k}^{\beta\gamma} (x) -x_k \delta^{\beta\gamma} \right]
\frac{\partial u_0^\gamma}{\partial x_j\partial x_k} \right\}\\
&\qquad +a_{ij}^{\alpha\beta}(x/\varep)
\frac{\partial}{\partial x_j}
\left[ \Phi_{\varep, k}^{\beta\gamma} (x) -x_k \delta^{\beta\gamma} 
-\varep \chi_k^{\beta\gamma} (x/\varep) \right] \frac{\partial^2 u_0^\gamma}
{\partial x_i \partial x_k}.
\endaligned
$$
Hence,
\begin{equation}\label{1.5.1-1}
\aligned
\int_\Omega |\nabla w_\varep|^2\, dx
&\le C\, \varep \int_\Omega |\phi (x/\varep)|\, |\nabla^2 u_0|\, |\nabla w_\varep|\, dx
+ C \int_\Omega |\Phi_\varep -P|\, |\nabla^2 u_0|\, |\nabla w_\varep|\, dx\\
&\qquad\qquad
+C \int_\Omega |\nabla \big( \Phi_\varep -P -\varep \chi(x/\varep)\big)|\, |w_\varep|\, |\nabla^2 u_0|\, dx,
\endaligned
\end{equation}
where $\phi =\big (\phi_{jik}^{\alpha\beta} \big)$.
Since $\chi$ is H\"older continuous,  $\phi$ is bounded.
Note that $\mathcal{L}_\varep \big( \Phi_\varep -P -\varep\chi(x/\varep) \big)=0$ in $\Omega$.
Thus, by (\ref{Ca-0}),
$$
\int_\Omega |\nabla \big( \Phi_\varep -P -\varep \chi(x/\varep)\big)|^2 |w_\varep|^2\, dx
\le C \int_\Omega |\Phi_\varep-P -\varep \chi(x/\varep)|^2 |\nabla w_\varep|^2\, dx.
$$
This, together with (\ref{1.5.1-1}) and Cauchy inequality (\ref{Cauchy}), gives  (\ref{1.5.1-0}).
\end{proof}

\begin{cor}\label{cor-1.5-1}
Let $m=1$ and $\Omega$ be a bounded Lipschitz domain in $\rd$.
Let $u_\varep$ $(\varep\ge 0)$ be the same as in Theorem \ref{rate-1.5.1}.
Then 
\begin{equation}\label{1.5.2-0}
\Big\| u_\varep -u_0 - \Big\{ \Phi_{\varep, j} -x_j \Big\} \frac{\partial u_0}{\partial x_j} \Big\|_{H^1_0(\Omega)}
\le C \, \varep\, \|\nabla^2 u_0\|_{L^2(\Omega)}.
\end{equation}
Consequently, 
\begin{equation}\label{1.5.2-1}
\| u_\varep -u_0 \|_{L^2(\Omega)} \le C\, \varep\, \| u_0\|_{H^2(\Omega)},
\end{equation}
where $C$ depends only on $\mu$ and $\Omega$.
\end{cor}

\begin{proof}
In the scalar case $m=1$, the corrector $\chi$ is H\"older continuous.
Also note that if 
$$
v_\varep =\Phi_{\varep, j}  (x) -x_j -\varep \chi_j (x/\varep),
$$ 
then $\mathcal{L}_\varep (v_\varep)=0$ in $\Omega$
and $v_\varep =-\varep \chi_j(x/\varep)$ on $\partial\Omega$.
Thus, by the maximum principle, 
$$
\| v_\varep\|_{L^\infty(\Omega)} \le \| v_\varep\|_{L^\infty(\partial\Omega)}
\le C\, \varep,
$$
which gives
\begin{equation}\label{1.5.2-2}
\| \Phi_{\varep, j} -x_j \|_{L^\infty(\Omega)} \le C\, \varep.
\end{equation}
Thus the estimate (\ref{1.5.2-0}) follows from (\ref{1.5.1-0}).
Finally, since 
$$
\Big\| \big( \Phi_{\varep, j} -x_j \big) \frac{\partial u_0}{\partial x_j}\Big \|_{L^2(\Omega)}
\le C \, \varep\, \|\nabla u_0\|_{L^2(\Omega)},
$$
the estimate (\ref{1.5.2-1}) follows readily from (\ref{1.5.2-0}).
\end{proof}

\begin{cor}\label{theorem-7.1-1}
Suppose that $m\ge 2$ and $\Omega$ is $C^{1, \eta}$ for some $\eta\in (0,1)$.
Assume that $A$ is H\"older continuous.
Let $u_\varep$ $(\varep\ge 0)$ be the same as in Theorem \ref{rate-1.5.1}. 
Then, 
\begin{equation}\label{estimate-7.1-1-2}
\Big\| u_\varep -u_0 - \left\{ \Phi_{\varep, j}^\beta
-P_j^\beta\right\} \frac{\partial u^\beta_0}{\partial x_j}\Big\|_{H^1_0(\Omega)}
\le C\, \varep\,  \| \nabla^2 u_0 \|_{L^2(\Omega)},
\end{equation}
where $C$ depends only on $A$ and $\Omega$.
\end{cor}

\begin{proof} 
Under the assumptions that $A$ is H\"older continuous and
$\Omega$ is $C^{1, \eta}$, it follows by the Agmon-Miranda maximum principle in Remark \ref{max-principle-remark} that
\begin{equation}
\| \Phi_{\varep, j}^\beta -P_j^\beta\|_{L^\infty(\Omega)} \le C\, \varep,
\end{equation}
which, together with (\ref{1.5.1-0}), gives the estimate (\ref{estimate-7.1-1-2}).
\end{proof}

\begin{remark}\label{remark-7.1-2}
{\rm
Since
$$
\aligned
&\frac{\partial }{\partial x_i} \left\{ u_\varep -u_0 -\left\{ \Phi_{\varep, j}^\beta 
-P_j^\beta \right\} \frac{\partial u_0^\beta}{\partial x_j} \right\}\\
&\qquad
=\frac{\partial u_\varep}{\partial x_i} -\frac{\partial}{\partial x_i} \left\{ \Phi_{\varep, j}^\beta \right\} \cdot
\frac{\partial u_0^\beta}{\partial x_j}
-\left\{ \Phi_{\varep, j}^\beta -P_j^\beta\right\} \frac{\partial^2 u_0^\beta}{\partial x_i\partial x_j},
\endaligned
$$
it follows from (\ref{1.5.2-0}) and (\ref{estimate-7.1-1-2}) that
\begin{equation}\label{7.1-2-1}
\Big\| \frac{\partial u_\varep}{\partial x_i} -\frac{\partial}{\partial x_i}
\left\{ \Phi_{\varep, j}^\beta \right\} \cdot \frac{\partial u_0^\beta}{\partial x_j} \Big\|_{L^2(\Omega)}
\le C\, \varep\,  \| \nabla^2 u_0 \|_{L^2(\Omega)},
\end{equation}
where $C$ depends only on $A$ and $\Omega$.
}
\end{remark}



\section{Convergence rates of eigenvalues}\label{section-7.6}

In this section we study the convergence rate for Dirichlet  eigenvalues for the operator $\mathcal{L}_\varep$.
Throughout the section we assume that $A$ is 1-periodic and satisfies the ellipticity condition  (\ref{weak-e-1})-(\ref{weak-e-2}) and
the symmetry condition $A^*=A$.

For  $f\in L^2(\Omega; \br^m)$, under
the ellipticity condition (\ref{weak-e-1})-(\ref{weak-e-2}), 
the elliptic system $\mathcal{L}_\varep (u_\varep) =f$ in $\Omega$
has a unique (weak) solution in $H^1_0(\Omega; \br^m)$. Define $T^D_\varep (f)=u_\varep$.  Note that
\begin{equation}\label{bi-linear-form}
\langle T^D_\varep(f), f\rangle
=\langle u_\varep, f\rangle
=\int_\Omega A(x/\e)\nabla u_\e\cdot \nabla u_\e\, dx
\end{equation}
(if $\varep =0$, $A(x/\varep)$ is replaced by $\widehat{A}$),
where $\langle\, , \, \rangle$ denotes the inner product in $L^2(\Omega; \br^m)$.
Since $\|u_\varep\|_{H^1_0(\Omega)} \le C\, \| f\|_{L^2(\Omega)}$, where $C$ depends only on
$\mu$ and $\Omega$, the linear operator $T^D_\varep$ is bounded, positive, and compact  on $L^2(\Omega; \br^m)$.
With the  symmetry condition $A^*=A$, the operator $T_\varep$ is also self-adjoint. Let 
\begin{equation}\label{eigenvalue}
\sigma_{\varep, 1}\ge \sigma_{\varep, 2}\ge \cdots \ge \sigma_{\varep, k} \ge \cdots >0
\end{equation}
denote the sequence of eigenvalues in a decreasing order of $T^D_\varep$.
Recall that $\lambda_\varep$ is called a Dirichlet eigenvalue for $\mathcal{L}_\varep$ in $\Omega$ if
there exists a nonzero $u_\varep \in H^1_0(\Omega;\br^m)$ such that
$\mathcal{L}_\varep (u_\varep) =\lambda_\varep u_\varep$ in $\Omega$.
Thus, if $A$ is elliptic and symmetric, for each $\varep>0$,
$\{ \lambda_{\varep, k} =(\sigma_{\varep, k})^{-1}\}$ forms the sequence of Dirichlet eigenvalues in an increasing order 
for $\mathcal{L}_\varep$ in $\Omega$.

By the mini-max principle,
\begin{equation}\label{mini-max}
\sigma_{\varep, k}
=\min_{\substack{f_1, \cdots, f_{k-1}\\ \in L^2(\Omega; \br^m)}}\
\max_{\substack{ \| f\|_{L^2(\Omega)}=1\\ f\perp f_i\\ i=1, \dots, k-1}}
 \langle T^D_\varep (f), f\rangle.
\end{equation}
Let $\{ \phi_{\varep, k} \}$ be an orthonormal basis of $L^2(\Omega; \br^m)$, where each $\phi_{\varep, k}$ is 
an eigenfunction associated with $\sigma_{\varep, k}$.
Let $V_{\varep, 0}=\{ 0\}$ and $V_{\varep, k}$ be the subspace of $L^2(\Omega; \br^m)$
 spanned by $\{\phi_{\varep, 1}, \dots, \phi_{\varep, k}\}$ for $k\ge 1$.
Then
\begin{equation}\label{7.6-1}
\sigma_{\varep, k}=\max_{\substack{f\perp V_{\varep, k-1}\\ \| f\|_{L^2(\Omega)}=1}}
\langle T^D_\varep (f), f\rangle.
\end{equation}

\begin{lemma}\label{lemma-7.6-1}.
For any $\e>0$,
$$
|\sigma_{\varep, k} -\sigma_{0, k}|
\le \max \left\{
\max_{\substack{f\perp V_{0, k-1} \\ \| f\|_{L^2(\Omega)}=1}}
\langle(T^D_\varep-T^D_0 )f, f\rangle|,\, 
\max_{\substack{f\perp V_{\varep, k-1} \\ \| f\|_{L^2(\Omega)}=1}}
\langle(T^D_\varep-T^D_0 )f, f\rangle|\right\}.
$$
\end{lemma}

\begin{proof}
It follows from (\ref{mini-max}) that
$$
\aligned
\sigma_{\varep, k} & \le \max_{\substack{f\perp V_{0, k-1}\\ \| f\|_{L^2(\Omega)}=1}}
\langle T^D_\varep (f), f\rangle\\
&\le \max_{\substack{f\perp V_{0, k-1}\\ \| f\|_{L^2(\Omega)}=1}}
\langle (T^D_\varep-T^D_0)( f), f\rangle
+
\max_{\substack{f\perp V_{0, k-1}\\ \| f\|_{L^2(\Omega)}=1}}
\langle T^D_0 (f), f\rangle\\
& =\max_{\substack{f\perp V_{0, k-1}\\ \| f\|_{L^2(\Omega)}=1}}
\langle (T^D_\varep-T^D_0) (f), f\rangle
+\sigma_{0,k},
\endaligned
$$
where we have used (\ref{7.6-1}). Hence,
\begin{equation}\label{7.6-1-1}
\sigma_{\varep, k}-\sigma_{0, k}
\le \max_{\substack{f\perp V_{0, k-1}\\ \| f\|_{L^2(\Omega)}=1}}
\langle (T^D_\varep-T^D_0) (f), f\rangle.
\end{equation}
Similarly, one can show that
\begin{equation}\label{7.6-1-3}
\sigma_{0, k}-\sigma_{\varep, k}
\le \max_{\substack{f\perp V_{\varep, k-1}\\ \| f\|_{L^2(\Omega)}=1}}
\langle (T^D_0-T^D_\varep) (f), f\rangle.
\end{equation}
The desired estimate follows from (\ref{7.6-1-1}) and (\ref{7.6-1-3}).
\end{proof}

By Theorem \ref{theorem-1.5.7},
$$
\| T_\e^D (f) -T_0^D (f)\|_{L^2(\Omega)}
\le C\, \e\,  \| T_\e^D (f)\|_{H^2(\Omega)},
$$
where $\Omega$ is a bounded Lipschitz domain.
If $\Omega$ is $C^{1,1}$ (or convex in the case $m=1$), the $H^2$ estimate
$$
\| T^D_0(f)\|_{H^2(\Omega)} \le C \| f\|_{L^2(\Omega)}
$$
holds for $\mathcal{L}_0$.
 It follows that
\begin{equation}\label{operator-norm-estimate}
\| T_\varep^D - T_0^D\|_{L^2\to L^2} \le C \, \varep.
\end{equation}
In view of Lemma \ref{lemma-7.6-1}, this gives
$$
| \sigma_{\e. k} -\sigma_{0, k} |\le C\, \e,
$$
where $C$ depends only on $\Omega$ and $\mu$, which leads to 
\begin{equation}\label{rough-estimate}
|\lambda_{\varep, k} -\lambda_{0, k}|\le C \, \varep (\lambda_{0,k})^2.
\end{equation}
We will see that the convergence estimate in $H^1_0(\Omega)$ in Theorem \ref{theorem-7.1-1}
allows us to improve the estimate (\ref{rough-estimate}) by a factor of $(\lambda_{0,k})^{1/2}$.
Note that the smoothness condition on $A$ is not needed in the scalar case $m=1$ in the following theorem.

\begin{thm}\label{theorem-7.6-2}
Suppose that $A$ is 1-periodic, symmetric, and satisfies the ellipticity condition (\ref{weak-e-1})-(\ref{weak-e-2}).
If $m\ge 2$, we also assume that $A$ is H\"older continuous.
Let $\Omega$ be a bounded $C^{1,1}$ domain or convex domain in the case $m=1$.
Then 
\begin{equation}\label{better-estimate}
|\lambda_{\varep, k}-\lambda_{0,k}|\le C\, \varep (\lambda_{0,k})^{3/2},
\end{equation}
where $C$ is independent of $\varep$ and $k$.
\end{thm}

\begin{proof}
We will use Lemma \ref{lemma-7.6-1}, Corollaries  \ref{cor-1.5-1} and
 \ref{theorem-7.1-1} to show that
\begin{equation}\label{7.6-4-1}
|\sigma_{\varep, k}-\sigma_{0, k}|\le C \, \varep\,  (\sigma_{0, k})^{1/2},
\end{equation}
where $C$ is independent of $\varep $ and $k$.
Since $\lambda_{\varep, k}=(\sigma_{\varep, k})^{-1}$ for $\varep\ge 0$ and
$\lambda_{\varep, k}\approx \lambda_{0, k}$,
this gives the desired estimate.

Let $u_\varep =T^D_\varep (f)$ and
$u_0=T^D_0(f)$, where $\| f\|_{L^2(\Omega)}=1$ and $f\perp V_{0,k-1}$.
In view of (\ref{7.6-1}) for $\varep=0$, we have $\langle u_0, f\rangle \le \sigma_{0, k}$.
Hence,
$$
c\, \|\nabla u_0\|^2_{L^2(\Omega)} \le \langle u_0, f\rangle  \le \sigma_{0, k},
$$
where $c$ depends only on $\mu$. It follows that
\begin{equation}\label{7.6-4-3}
\| f\|_{H^{-1}(\Omega)}
\le C \, \|\nabla u_0\|_{L^2(\Omega)} \le C\, (\sigma_{0, k})^{1/2}.
\end{equation}
Now, write
$$
\langle u_\varep-u_0, f\rangle
=\Big\langle u_\varep -u_0 - \big\{ \Phi_{\varep, \ell}^\beta -P_\ell^\beta\big\} \frac{\partial u_0^\beta}{\partial x_\ell}, f\Big>
+\Big<\big\{ \Phi_{\varep, \ell}^\beta -P_\ell^\beta\big\} \frac{\partial u_0^\beta}{\partial x_\ell}, f\Big>.
$$
This implies that for any $f\perp V_{0, k-1}$ with $\| f\|_{L^2(\Omega)}=1$,
\begin{equation}\label{7.6-4-5}
\aligned
|\langle u_\varep-u_0,f\rangle |
&\le  \Big\|u_\varep-u_0-
\Big\{ \Phi_{\varep, \ell}^\beta -P_\ell^\beta\big\} \frac{\partial u_0^\beta}{\partial x_\ell}\Big\|_{H^1_0(\Omega)}
\| f\|_{H^{-1}(\Omega)}\\
&\qquad\qquad 
+\Big\| \big\{ \Phi_{\varep, \ell}^\beta -P_\ell^\beta\big\} \frac{\partial u_0^\beta}{\partial x_\ell}\Big\|_{L^2(\Omega)}
\| f\|_{L^2(\Omega)}\\
& \le C \varep \| f\|_{L^2(\Omega)} \| f\|_{H^{-1}(\Omega)}
+ C\varep \|\nabla u_0\|_{L^2(\Omega)} \| f\|_{L^2(\Omega)}\\
&\le C\varep \| \nabla u_0\|_{L^2(\Omega)}\\
& \le C \varep (\sigma_{0,k})^{1/2},
\endaligned
\end{equation}
where we have used  Corollaries  \ref{cor-1.5-1} and
 \ref{theorem-7.1-1} as well as  the estimate 
$\|\Phi_{\varep, \ell}^\beta-P_\ell^\beta\|_\infty \le C\varep$
for  the second inequality, and (\ref{7.6-4-3}) for the third and fourth.

Next we consider the case $f\perp V_{\varep, k-1}$ and $\| f\|_{L^2(\Omega)}=1$.
In view of (\ref{7.6-1}) we have $\langle u_\varep, f\rangle \le \sigma_{\varep, k}$.
Hence, $c\|\nabla u_\varep\|_{L^2(\Omega)}^2 \le \langle u_\varep, f\rangle \le \sigma_{\varep, k}$.
It follows that
\begin{equation}\label{7.6-4-6}
\| f\|_{H^{-1}(\Omega)}
\le C \|\nabla u_\varep\|_{L^2(\Omega)} \le C (\sigma_{\varep, k})^{1/2}
\end{equation}
and
\begin{equation}\label{7.6-4-7}
\|\nabla u_0\|_{L^2(\Omega)}
\le C \, \| f\|_{H^{-1}(\Omega)}
\le C \, (\sigma_{\varep, k})^{1/2},
\end{equation}
where $C$ depends only on $\mu$.
As before, this implies that
for any $f\perp V_{\varep, k-1}$ with $\|f\|_{L^2(\Omega)}=1$,
\begin{equation}\label{7.6-4-8}
\aligned
|\langle u_\varep-u_0,f\rangle |
&\le  \big\|u_\varep-u_0-
\Big\{ \Phi_{\varep, \ell}^\beta -P_\ell^\beta\big\} \frac{\partial u_0^\beta}{\partial x_\ell}\Big\|_{H^1_0(\Omega)}
\| f\|_{H^{-1}(\Omega)}\\
&\qquad\qquad 
+\Big\| \big\{ \Phi_{\varep, \ell}^\beta -P_\ell^\beta\big\} \frac{\partial u_0^\beta}{\partial x_\ell}\Big\|_{L^2(\Omega)}
\| f\|_{L^2(\Omega)}\\
& \le C \varep \| f\|_{H^{-1}(\Omega)}+
C \varep \|\nabla u_0\|_{L^2(\Omega)}\\
&\le C \varep (\sigma_{\varep, k})^{1/2}\\
&\le C \varep (\sigma_{0, k})^{1/2},
\endaligned
\end{equation}
where we have used the fact $\sigma_{\varep, k}\approx \sigma_{0, k}$.
In view of Lemma \ref{lemma-7.6-1}, the estimate (\ref{7.6-4-1}) follows from (\ref{7.6-4-5}) and (\ref{7.6-4-8}).
\end{proof}


\section{Asymptotic expansions of Green functions}\label{section-6.2}

Assume that $A$ satisfies the ellipticity condition (\ref{weak-e-1})-(\ref{weak-e-2}).
Let $G_\varep(x,y)=\big(G_\varep^{\alpha\beta}(x,y)\big)$ denote the $m\times m$
matrix of Green functions
 for $\mathcal{L}_\varep$ in $\Omega$. 
 Recall that in the scalar case $m=1$,
 \begin{equation}\label{7.2-0}
 |G_\varep (x, y)|\le 
 \left\{
 \aligned
 & C\,|x-y|^{2-d} & \quad& \text{ if } d\ge 3,\\
 &C\, \big\{ 1+ \ln \big(r_0 |x-y|^{-1} \big) \big\} &\quad& \text{ if } d=2
 \endaligned
 \right.
 \end{equation}
  for any $x,y\in \Omega$ and $x\neq y$,
 where $\Omega$ is a bounded Lipschitz domain in $\rd$,
 $r_0 =\text{diam}(\Omega)$, and $C$ depends only on $\mu$ and $\Omega$. 
 The estimate in (\ref{7.2-0}) for $d=2$ also holds for
 $m\ge 2$. No smoothness or periodicity condition is needed in both cases.
 If $A$ is 1-periodic and belongs to VMO$(\rd)$, it follows from the interior and boundary
 H\"older estimates that the estimate (\ref{7.2-0}) for $d\ge 3$ holds
 if $m\ge 2$ and $\Omega$ is $C^1$.
 Furthermore,
 if $A$ is H\"older continuous and
 $\Omega$ is $C^{1,\eta}$,
 it is proved in Chapter \ref{chapter-3} that 
 \begin{equation}\label{7.2-1}
 \aligned
 |\nabla_x G_\varep (x,y)|+|\nabla_y G_\varep (x,y)|&\le C\, |x-y|^{1-d},\\
 |\nabla_x\nabla_y G_\varep (x,y)|& \le C\, |x-y|^{-d}
 \endaligned
 \end{equation}
 for any $x,y\in \Omega$ and $x\neq y$, where $C$ depends only on $\mu$, $(\lambda, \tau)$, and $\Omega$.
 
 In this section we study the asymptotic behavior, as $\varep\to 0$, of
 $G_\varep(x,y)$, $\nabla_x G_\varep(x,y)$, $\nabla_y G_\varep (x,y)$, and $\nabla_x\nabla_y G_\varep (x,y)$.
 We shall use $G_0(x,y)=\big(G_0^{\alpha\beta}(x,y)\big)$ to denote the $m \times m$ matrix of Green functions for
 the homogenized operator $\mathcal{L}_0$ in $\Omega$.
 
 We begin with a size estimate of $|G_\varep(x,y)-G_0(x,y)|$. 
  
 \begin{thm}\label{theorem-7.2-1}
 Suppose that $A$ satisfies  (\ref{weak-e-1})-(\ref{weak-e-2}) and is 1-periodic.
 If  $m\ge 2$, we also assume that $A$ is H\"older continuous.
 Let $\Omega$ be a bounded $C^{1,1}$ domain. Then
 \begin{equation}\label{estimate-7.2-1}
 |G_\varep (x,y)-G_0(x,y)|\le \frac{C\, \varep}{|x-y|^{d-1}} \quad \text{ for any } x,y\in \Omega \text{ and } x\neq y,
 \end{equation}
 where $C$ depends only on $\mu$, $\Omega$ as well as $(\lambda, \tau)$ (if $m\ge 2$).
  \end{thm}

 The proof of Theorem \ref{theorem-7.2-1}, which follows the same line of argument for
 the estimate of $|\Gamma_\varep (x,y)-\Gamma_0(x,y)|$ in Section \ref{section-2.4},
 relies on some boundary $L^\infty$ estimates.
 Let
 \begin{equation}\label{definition-of-Omega}
 T_r=T(x,r)=B(x,r)\cap\Omega \quad \text{ and } \quad
 I_r=I(x,r)=B(x,r)\cap \partial\Omega
 \end{equation}
 for some $x\in \overline{\Omega}$ and $0<r<r_0=c_0\, \text{diam}(\Omega)$.
 
  \begin{lemma}\label{lemma-7.2-1}
  Suppose that $A$ satisfies the same conditions as in Theorem \ref{theorem-7.2-1}.
  Assume that $\Omega$ is Lipschitz if $m=1$, and $C^{1, \eta}$ if $m\ge 2$. Then
  \begin{equation}\label{estimate-lemma-7.2-1}
  \| u_\varep\|_{L^\infty(T_r)}\le C \| f\|_{L^\infty(I_{3r})}
  +C \average_{T_{3r}} |u_\varep |,
  \end{equation}
  where $\mathcal{L}_\varep (u_\varep)=0$ in $T_{3r}$ and $u_\varep =f$ on $I_{3r}$.
  \end{lemma}
   
   \begin{proof}
   By rescaling we may assume that $r=1$.
      If $f=0$, the estimate is a consequence of (\ref{boundary-Holder-estimate}).
   To treat the general case, let $v_\varep$ be the solution to $\mathcal{L}_\varep (v_\varep)=0$ in
   $\widetilde{\Omega}$ with the Dirichlet condition $v_\varep =f$ on $\partial\widetilde{\Omega}\cap \partial \Omega$
    and $v_\varep=0$ on $\partial\widetilde{\Omega}\setminus \partial\Omega$,
   where $\widetilde{\Omega}$ is a $C^{1, \eta}$ domain such that $T_{2}\subset \widetilde{\Omega}
   \subset T_{3}$.
   By the Agmon-Miranda maximum principle in
    Remark \ref{max-principle-remark},
    $\| v_\varep\|_{L^\infty(\widetilde{\Omega})} \le C \, \| f\|_{L^\infty(I_{3})}$.
    This, together with
   $$
   \aligned
   \|u_\varep -v_\varep\|_{L^\infty(T_1)}
   &\le C\average_{T_2} |u_\varep -v_\varep|\\
&   \le C\average_{T_{3}}  |u_\varep|+ C\,  \| f\|_{L^\infty(I_{3})},
  \endaligned
   $$
    gives (\ref{estimate-lemma-7.2-1}) for the case $m\ge 2$.
Finally, we observe that if $m=1$, the $L^\infty$ estimate and the maximum principle used above
hold for Lipschitz domains without smoothness (and periodicity) condition on $A$.
      \end{proof}
 
 \begin{lemma}\label{lemma-7.2-2}
 Assume that $A$ and $\Omega$ satisfy the same conditions as in Lemma \ref{lemma-7.2-1}.
 Let $u_\varep \in H^1(T_{4r};\br^m)$ and $u_0 \in W^{2, p}(T_{4r};\br^m)$ for some $d<p<\infty$.
 Suppose that 
 $$
 \mathcal{L}_\varep (u_\varep) =\mathcal{L}_0 (u_0) \quad \text{ in } T_{4r}\quad
 \text{ and } \quad u_\varep =u_0 \quad \text{ on }  I_{4r}.
 $$
 Then,
 \begin{equation}\label{estimate-lemma-7.2-2}
 \aligned
 \| u_\varep -u_0\|_{L^\infty (T_r)}
 &\le C \average_{T_{4r}} |u_\varep -u_0|
 +C\, \varep\, \|\nabla u_0\|_{L^\infty(T_{4r})}\\
 &\qquad \qquad + C_p \, \varep\, r^{1-\frac{d}{p}} \| \nabla^2 u_0\|_{L^p(T_{4r})}.
 \endaligned
 \end{equation}
 \end{lemma}
 
 \begin{proof}
 By rescaling we may assume that $r=1$.
 Choose a domain $\widetilde{\Omega}$, which is Lipschitz for $m=1$ and $C^{1, \eta}$ for $m\ge 2$, such that
 $T_3\subset \widetilde{\Omega} \subset T_4$.
 Consider
 $$
 w_\varep =u_\varep -u_0 -\varep \chi_j^\beta (x/\varep) \frac{\partial u_0^\beta}{\partial x_j}
 =w_\varep^{(1)} +w_\varep^{(2)} \quad \text{ in } \widetilde{\Omega},
 $$
 where
 \begin{equation} \label{7.2-2-1}
 \mathcal{L}_\varep \big(w_\varep^{(1)} \big)
 =\mathcal{L}_\varep (w_\varep) \quad \text{ in } \widetilde{\Omega} \quad \text{ and } \quad
 w_\varep^{(1)}\in H_0^1 (\widetilde{\Omega};\br^m),
 \end{equation}
 and
 \begin{equation}\label{7.2-2-2}
 \mathcal{L}_\varep \big(w_\varep^{(2)} \big) =0 \quad \text{ in } \widetilde{\Omega} \quad \text{ and } 
 \quad w_\varep^{(2)} =w_\varep \quad \text{ on }\quad  \partial \widetilde{\Omega}.
 \end{equation}
 Since $w_\varep^{(2)}=w_\varep =-\varep\chi(x/\varep)\nabla u_0$ on $I_{3}$ and
 $\|\chi\|_\infty\le C$, it follows from Lemma \ref{lemma-7.2-1} that
 $$
 \aligned
 \| w_\varep^{(2)}\|_{L^\infty(T_1)}
 &\le C\, \varep\, \| \nabla u_0\|_{L^\infty(I_3)}
 +C \average_{T_3} |w_\varep^{(2)} |\\
 &\le C\, \varep\, \| \nabla u_0\|_{L^\infty(I_3)}
 +C \average_{T_3} |w_\varep |
 +C \average_{T_3} |w_\varep^{(1)} |\\
 &\le C\average_{T_3} |u_\varep -u_0|
 +C\, \varep\, \| \nabla u_0\|_{L^\infty(T_3)}
 +C\, \| w_\varep^{(1)}\|_{L^\infty(T_3)}.
 \endaligned
 $$
 This gives
 \begin{equation}\label{7.2-2-3}
 \| u_\varep -u_0\|_{L^\infty(T_1)}
 \le C\average_{T_3} |u_\varep -u_0|
 + +C\, \varep\, \| \nabla u_0\|_{L^\infty(T_3)}
 +C\, \| w_\varep^{(1)}\|_{L^\infty(T_3)}.
\end{equation}

To estimate $w_\varep^{(1)}$ on $T_3$, we use the Green function representation 
$$
w_\varep^{(1)} (x)
=\int_{\widetilde{\Omega}} \widetilde{G}_\varep (x,y) \mathcal{L}_\varep (w_\varep) (y)\, dy,
$$
where $\widetilde{G}_\varep (x,y)$ denotes the matrix of Green functions for $\mathcal{L}_\varep$ in 
$\widetilde{\Omega}$.
In view of (\ref{right-hand-side}), we obtain
$$
w_\varep^{(1)} (x)
=\varep \int_{\widetilde{\Omega}}
\frac{\partial}{\partial y_i}
\Big\{ \widetilde{G}_\varep (x,y)\Big\} \cdot
\Big[\phi_{jik} (y/\varep) -a_{ij}(y/\varep) \chi_k (y/\varep)\Big]
\cdot \frac{\partial^2 u_0}{\partial y_j \partial y_k}\, dy,
$$
where we have suppressed the subscripts for notational simplicity.
Since $\|\phi_{jik}\|_\infty\le C$ and $p>d$, it follows that
$$
\aligned
|w_\varep^{(1)} (x)|
&\le C\, \varep \int_{\widetilde{\Omega}}
|\nabla_y \widetilde{G}_\varep (x,y)|\, |\nabla^2 u_0(y)|\, dy\\
&\le C \, \varep \, \|\nabla^2 u_0\|_{L^p(T_4)}
\left(\int_{\widetilde{\Omega}} |\nabla_y \widetilde{G}_\varep (x,y)|^{p^\prime}\, dy\right)^{1/p^\prime}\\
&\le C_p\, \varep\, \|\nabla^2 u_0\|_{L^p(T_4)},
\endaligned
$$
where we have used  H\"older's inequality and the observation 
\begin{equation}\label{7.2-2-4}
\|\nabla_y \widetilde{G}_\varep (x, \cdot)\|_{L^{p^\prime}(\widetilde{\Omega})}\le C.
\end{equation}
This, together with (\ref{7.2-2-3}), gives the estimate (\ref{estimate-lemma-7.2-2}).
We point out that the estimate (\ref{7.2-2-4})
follows from the size estimate (\ref{7.2-0}) and Cacciopoli's inequality 
by decomposing $\Omega$ as a union of $\Omega \cap \{ y: |y-x|\sim 2^{-\ell}\}$.
 \end{proof}
 
 \begin{proof}[\bf Proof of Theorem \ref{theorem-7.2-1}]
 We first note that under the assumptions on $A$ and $\Omega$ in the theorem,
 the size estimate (\ref{7.2-0}) and
 $|\nabla_x G_0(x,y)|\le C |x-y|^{1-d}$ hold for any $x,y\in \Omega$ and $x\neq y$.
 We now fix $x_0, y_0\in \Omega$ and $r=|x_0-y_0|/8>0$.
For $F\in C_0^\infty( T(y_0,r); \br^m)$, let
$$
u_\varep (x)=\int_\Omega G_\varep (x,y) F(y)\, dy
\quad \text{ and } \quad
u_0 (x)=\int_\Omega G_0 (x,y) F(y)\, dy.
$$ 
 Then $\mathcal{L}_\varep (u_\varep)=\mathcal{L}_0 (u_0)=F$ in $\Omega$ and
 $u_\varep =u_0 =0$ on $\partial\Omega$.
 Note that since $\Omega$ is $C^{1,1}$,
 \begin{equation}\label{7.2-3-1}
 \aligned
 \|\nabla^2 u_0\|_{L^p(\Omega)} & \le C_p\, \| F\|_{L^p(\Omega)} \quad \text{ for } 1<p<\infty,\\
 \|\nabla u_0\|_{L^\infty(\Omega)} & \le C_p \, r^{1-\frac{d}{p}} \| F\|_{L^p(T(y_0,r))}
 \quad \text{ for } p>d.
 \endaligned
 \end{equation}
 The first inequality in (\ref{7.2-3-1}) is the $W^{2,p}$ estimate in $C^{1,1}$ domains
 for second-order elliptic systems with constant coefficients, while the second
 follows from the estimate $|\nabla_x G_0 (x,y)|\le C\, |x-y|^{1-d}$ by H\"older's 
 inequality.
 
 Next, let
 $$
 w_\varep =u_\varep -u_0 -\varep \chi_j^\beta (x/\varep) \frac{\partial u_0^\beta}{\partial x_j}
 =\theta_\varep (x) +z_\varep (x),
 $$
 where $\theta_\varep \in H^1_0(\Omega; \br^m)$ and $\mathcal{L}_\varep (\theta_\varep)=\mathcal{L}_\varep (w_\varep)$
 in $\Omega$.
 Observe that by the formula  (\ref{right-hand-side}) for $\mathcal{L}_\e(w_\e)$,
 $$
 \|\nabla \theta_\varep \|_{L^2(\Omega)}
\le C \, \varep\, \| \nabla^2 u_0\|_{L^2(\Omega)}
\le C\, \varep\, \| F\|_{L^2(T(y_0,r))},
$$
where we have used the fact that $\chi$ and $\phi$ are bounded.
By H\"older and Sobolev inequalities, this implies that if $d\ge 3$,
\begin{equation}\label{7.2-3-3}
\aligned
\|\theta_\varep\|_{L^2(T(x_0,r))}
& \le C \, r\, \| \theta_\varep \|_{L^q(\Omega)}\\
&\le C \, r\, \| \nabla \theta_\varep\|_{L^2(\Omega)}\\
& \le C\, \varep\, r^{1+\frac{d}{2} -\frac{d}{p}}\, \| F\|_{L^p(T(y_0,r))},
\endaligned
\end{equation}
where $\frac{1}{q}=\frac{1}{2}-\frac{1}{d}$ and $p>d$.
We point out that if $d=2$, one has
$$
\|\theta_\varep\|_{L^2(T(x_0,r))}
\le C\e\,  r\, \| F\|_{L^2(T(y_0,r))}.
$$
in place of (\ref{7.2-3-3}).
To see this, we use the fact that the $W^{1, p}$ estimate holds
for $\mathcal{L}_\e$ for $p$ close to $2$, even without the smoothness assumption on $A$
(see Remark \ref{pert-r}).
Thus there exists some $\bar{p}<2$ such that
$$
\|\nabla \theta_\e \|_{L^{\bar{p} }(\Omega)}
\le C \e \|\nabla^2 u_0\|_{L^{\bar{p}} (\Omega)}
\le C\, \varep\, \| F\|_{L^{\bar{p}}(T(y_0,r))},
$$
which, by H\"older's inequality and Sobolev inequality, leads to
\begin{equation}\label{7.2-3-5}
\aligned
\|\theta_\e \|_{L^2(T(x_0, r))}
&\le C r^{1-\frac{2}{q}} \| \theta_\e\|_{L^q(T(x_0, r))}
\le C r^{1-\frac{2}{q}} \| \nabla \theta_\e \|_{L^{\bar{p}} (\Omega)}\\
&\le C r^{2-\frac{2}{\bar{p}}} \| \nabla \theta_\e\|_{L^{\bar{p}}(\Omega)}
\le C \e\,  r^{2-\frac{2}{\bar{p}}} \| F \|_{L^{\bar{p}}(T(y_0, r))}\\
&\le C\e  \, r \| F \|_{L^2(T(y_0, r))},
\endaligned
\end{equation}
where $\frac{1}{q}=\frac{1}{\bar{p}} -\frac12$.

Observe that since $\mathcal{L}_\varep (z_\varep)=0$ in $\Omega$ and $z_\varep =w_\varep$ on $\partial\Omega$,
by the maximum principle (\ref{max-principle}),
\begin{equation}\label{7.2-3-5a}
\|z_\varep\|_{L^\infty(\Omega)}
\le C\, \| z_\varep \|_{L^\infty(\partial\Omega)}
\le C \, \varep\, \| \nabla u_0\|_{L^\infty(\partial\Omega)}.
\end{equation}
In view of (\ref{7.2-3-1})-(\ref{7.2-3-5a}), we obtain
$$
\aligned
\| u_\varep -u_0\|_{L^2(T(x_0,r))}
&\le \| \theta_\varep\|_{L^2(T(x_0,r))} +\|z_\varep\|_{L^2(T (x_0,r))}
+ C\, \varep\, r^{\frac{d}{2}}\, \| \nabla u_0\|_{L^\infty(\Omega)}\\
&\le \|\theta_\varep\,  \|_{L^2(T(x_0, r))}
+ C \, \varep\, r^{\frac{d}{2}}\,  \|\nabla u_0\|_{L^\infty(\Omega)}\\
&\le C  \, \varep\,  r^{1+\frac{d}{2}-\frac{d}{{p}}}\,  \| f\|_{L^{{p}}(T(y_0, r))},
\endaligned
$$
where ${p}>d$.
This, together with Lemma \ref{lemma-7.2-2} and (\ref{7.2-3-1}), gives
$$
|u_\varep (x_0)-u_0(x_0)|\le C\,\varep\,  r^{1-\frac{d}{p}} \| f\|_{L^p(T(y_0, r))}.
$$
It then follows by duality that
$$
\left(\int_{T(y_0,r)} 
|G_\varep (x_0, y)-G_0(x_0, y)|^{p^\prime}\, dy\right)^{1/p^\prime}
\le C_p\, \varep\, r^{1-\frac{d}{p}} \quad \text{ for any } p>d.
$$

Finally, since $\mathcal{L}_\varep^* \big(G_\varep (x_0, \cdot)\big)=
\mathcal{L}_0^* \big(G_0(x_0, \cdot)\big)=0$ in $T(y_0,r)$, we may invoke Lemma \ref{lemma-7.2-2}
again to conclude that
$$
\aligned
|G_\varep (x_0, y_0)- G_0 (x_0, y_0)|
&\le \average_{T(y_0,r)}
|G_\varep (x_0, y)-G_0 (x_0, y)|\, dy\\
&\qquad \qquad+C \,\varep\,  \|\nabla_y G_0(x_0, \cdot)\|_{L^\infty(T(y_0,r))}\\
&\qquad \qquad +C_p \, \varep \, r^{1-\frac{d}{p}} \, \|\nabla_y^2 G_0 (x_0,\cdot)\|_{L^p(T(y_0,r))}\\
&\le C\, \varep\, r^{1-d},
\endaligned
$$
where we have used 
$$
\aligned
\left(\average_{T(y_0,r)} |\nabla_y^2 G_0 (x_0, y)|^p\, dy\right)^{1/p}
& \le C_p \, r^{-2}
\|G_0(x_0, \cdot)\|_{L^\infty(T(y_0, 2r))}\\
&\le C_p \, r^{-d},
\endaligned
$$
obtained by using the boundary $W^{2, p}$estimates on $C^{1,1}$ domains for $\mathcal{L}^*_0$.
 \end{proof}
 
 The next theorem gives an asymptotic expansion of $\nabla_x G_\varep (x,y)$.
 Recall that $\big(\Phi_{\varep, j}^{\alpha\beta} (x)\big)$ denotes the matrix of Dirichlet correctors for 
$\mathcal{L}_\varep$ in $\Omega$.
 
 \begin{thm}\label{theorem-7.2-2}
 Suppose that $A$ is 1-periodic and satisfies (\ref{weak-e-1})-(\ref{weak-e-2}).
 Also assume that $A$ is H\"older continuous.
 Let $\Omega$ be a bounded $C^{2, \eta}$ domain for some $\eta\in (0,1)$.
 Then
 \begin{equation}\label{estimate-7.2-2}
 \aligned
 \big| \frac{\partial}{\partial x_i} \Big\{ G_\varep^{\alpha\gamma} (x,y)\Big\}
 & -\frac{\partial}{\partial x_i}
 \left\{ \Phi_{\varep, j}^{\alpha\beta} (x)\right\} \cdot
 \frac{\partial }{\partial x_j} \left\{ G_0^{\beta\gamma} (x,y)\right\} \big|\\
& \le \frac{ C\, \varep \ln [\varep^{-1} |x-y| +2]}{|x-y|^d}
 \endaligned
 \end{equation}
 for any $x,y\in \Omega$ and $x\neq y$, where $C$ depends only on $\mu$, $(\lambda,\tau)$, and $\Omega$.
  \end{thm}
  
The proof of Theorem \ref{theorem-7.2-2} relies on a boundary Lipschitz estimate.
The argument is similar to that for the estimate of
$\nabla_x \Gamma_\varep (x,y)-\nabla \chi (x/\e) \cdot \nabla_x \Gamma_\varep (x, y)$
in Section \ref{section-2.4}.

\begin{lemma}\label{lemma-7.2-3}
Suppose that $A$ and $\Omega$ satisfy the same conditions as in Theorem \ref{theorem-7.2-2}.
Let $u_\varep \in H^1 (T_{4r};\br^m)$ and $u_0 \in C^{2,\rho}(T_{4r};\br^m)$
for some $0<\rho<\eta$.
Assume that $\mathcal{L}_\varep (u_\varep)=\mathcal{L}_0 (u_0)$ in $T_{4r}$ and
$u_\varep =u_0$ on $I_{4r}$. Then, if $0<\varep<r$,
\begin{equation}\label{estimate-lemma-7.2-3}
\aligned
& \big\| \frac{\partial u^\alpha_\varep}{\partial x_i} -\frac{\partial}{\partial x_i} \left\{ \Phi_{\varep, j}^{\alpha\beta}\right\} \cdot
\frac{\partial u_0^\beta}{\partial x_j}\big\|_{L^\infty(T_r)}\\
&\qquad
\le \frac{C}{r} \average_{\Omega_{4r}} |u_\varep -u_0|
+C\, \varep\, r^{-1} \| \nabla u_0\|_{L^\infty(T_{4r})}\\
&\qquad\qquad
+ C\, \varep \ln [\varep^{-1} r +2] \|\nabla^2 u_0\|_{L^\infty(T_{4r})}
+C\, \varep \, r^\rho\, \|\nabla^2 u_0\|_{C^{0, \rho}(T_{4r})}.
\endaligned
\end{equation}
\end{lemma}
 
 \begin{proof}
 We start out  by choosing a $C^{2,\eta}$ domain $\widetilde{\Omega}$ such that $T_{3r}\subset\widetilde{\Omega}
 \subset T_{4r}$.
 Let
 $$
 w_\varep =u_\varep - u_0 -\left\{ \Phi_{\varep, j}^\beta -P_j^\beta \right\} \frac{\partial u_0^\beta}{\partial x_j}.
 $$
 Note that $w_\varep =0$ on $I_{4r}$.
 Write $w_\varep =\theta_\varep +z_\varep$ in $\widetilde{\Omega}$, where
 $\theta_\varep \in H_0^1(\widetilde{\Omega};\br^m)$ and
 $\mathcal{L}_\varep \big(\theta_\varep\big) =\mathcal{L}_\varep (w_\varep)$ in $\widetilde{\Omega}$.
 Since $\mathcal{L}_\varep \big(z_\varep\big)=0$ in $\widetilde{\Omega}$
 and $z_\varep=w_\varep=0$ on $I_{3r}$, 
 it follows from the boundary Lipschitz estimate (\ref{Dirichlet-Lip-estimate}) that
 $$
 \aligned
 \|\nabla z_\varep\|_{L^\infty (T_r)}
 &\le \frac{C}{r} \average_{T_{2r}} |z_\varep |\\
&\le  \frac{C}{r} \average_{T_{2r}} |w_\varep| + C r^{-1} \| \theta_\varep\|_{L^\infty(T_{2r})}\\
&\le \frac{C}{r} \average_{T_{2r}} |u_\varep -u_0|
+C \, \varep\, r^{-1} \|\nabla u_0\|_{L^\infty(T_{2r})}
+C r^{-1} \| \theta_\varep \|_{L^\infty(T_{2r})},
\endaligned
$$
where we have used the estimate $\|\Phi_{\varep, j}^\beta -P_j^\beta\|_{L^\infty(\Omega)}\le C\, \varep$.
This implies that
$$
\|\nabla w_\varep\|_{L^\infty(T_r)}
\le \frac{C}{r} \average_{T_{2r}} |u_\varep -u_0|
+C \, \varep\, r^{-1} \|\nabla u_0\|_{L^\infty(T_{2r})}
+C  \| \nabla \theta_\varep \|_{L^\infty(T_{2r})},
$$
where we have used $\| \theta_\varep \|_{L^\infty(T_{2r})} \le C \, r \| \nabla \theta_\varep\|_{L^\infty(T_{2r})}$.
Thus, 
\begin{equation}\label{7.2-4-1}
\aligned
\Big\|\frac{\partial u^\alpha_\varep}{\partial x_i}
& -\frac{\partial}{\partial x_i} \Big\{ \Phi_{\varep, j}^{\alpha\beta}\Big\}
\cdot \frac{\partial u_0^\beta}{\partial x_j}\Big\|_{L^\infty (T_r)}\\
& \le 
\frac{C}{r}
\average_{T_{2r}} |u_\varep -u_0|\, dx
+C\varep r^{-1} \|\nabla u_0\|_{L^\infty(T_{2r})}\\
&\quad\quad\quad +C\,\varep\, \| \nabla^2 u_0\|_{L^\infty(T_{2r})}
+C\, \|\nabla \theta_\varep\|_{L^\infty(T_{2r})}.
\endaligned
\end{equation}

It remains to estimate $\nabla \theta_\varep$ on $T_{2r}$.
To this end we use the Green function representation
$$
\theta_\varep (x)
=\int_{\widetilde{\Omega}} \widetilde{G}_\varep (x,y) \mathcal{L}_\varep (w_\varep) (y)\, dy,
$$
where
$\widetilde{G}_\varep (x,y)$ is the matrix of Green functions for $\mathcal{L}_\varep$ in 
the $C^{2,\eta}$ domain $\widetilde{\Omega}$.
Let 
$$f_i (x) =-\varep \phi_{kij} \left({x}/{\varep}\right) \frac{\partial^2 u_0}
{\partial x_j\partial x_k}
+a_{ij}\left({x}/{\varep}\right) \big[
\Phi_{\varep, k}-P_k\big] \cdot \frac{\partial^2 u_0}{\partial x_j\partial x_k},
$$
where we have suppressed the superscripts for notational simplicity. In view of (\ref{formula-1.4}),
we obtain
$$
\aligned
\theta_\varep (x)
=& -\int_{\widetilde{\Omega}} \frac{\partial}{\partial y_i}
\Big\{ \widetilde{G}_\varep (x,y)\Big\} \cdot \big\{ f_i (y)-f_i (x)\big\}\, dy\\
&+\int_{\widetilde{\Omega}}
\widetilde{G}_\varep (x,y) a_{ij}\left({y}/{\varep}\right)
\frac{\partial}{\partial y_j} \Big[ \Phi_{\varep, k} -P_k -\varep \chi_k \left({y}/{\varep}\right)
\Big] \cdot \frac{\partial^2 u_0}{\partial y_i \partial y_k}\, dy.
\endaligned
$$
It follows that
\begin{equation}\label{7.2-4-3}
\aligned
|\nabla \theta_\varep(x)|
&\le \int_{\widetilde{\Omega}} |\nabla_x\nabla_y \widetilde{G}_\varep (x,y)|\, |f(y)-f(x)|\, dy\\
&
+C\,\| \nabla^2 u_0\|_{L^\infty (\Omega_{4r})}
\int_{\widetilde{\Omega}} |\nabla_x \widetilde{G}_\varep (x,y)|\, \big|\nabla_y \big[
\Phi_{\varep} -P -\varep \chi \left({y}/{\varep}\right)\big]\big|\, dy.
\endaligned
\end{equation}
To handle the first term in the RHS of (\ref{7.2-4-3}), we use
$|\nabla_x\nabla_y \widetilde{G}_\varep (x,y)|\le C|x-y|^{-d}$ and the observation that
$$
\aligned
\| f\|_{L^\infty(T_{4r})} & \le C \, \varep\, \|\nabla^2 u_0\|_{L^\infty (T_{4r})},\\
 | f(x)-f(y)| & \le C\, |x-y|^\rho\,
\Big\{ \varep^{1-\rho} \|\nabla^2 u_0\|_{L^\infty (T_{4r})}
+\varep\, \|\nabla^2 u_0\|_{C^{0,\rho} (T_{4r})}\Big\}.
\endaligned
$$
This yields that
$$
\aligned
& \int_{\widetilde{\Omega}} |\nabla_x\nabla_y \widetilde{G}_\varep (x,y)| |f(y)-f(x)|\, dy\\
& \qquad \le C\, \varep \|\nabla^2 u_0\|_{L^\infty (T_{4r})}
\int_{\widetilde{\Omega}\setminus B(x,\varep)} \frac{dy}{|x-y|^d}\\
& \qquad \qquad + C\, \Big\{ \varep^{1-\rho} \|\nabla^2 u_0\|_{L^\infty (T_{4r})}
+\varep \, \|\nabla^2 u_0\|_{C^{0,\rho} (T_{4r})} \Big\}
\int_{\widetilde{\Omega}\cap B(x,\varep)} 
\frac{dy}{|x-y|^{d-\rho}}\\
&\qquad \le C\, \varep \ln [\varep^{-1} r +2]\|\nabla^2 u_0\|_{L^\infty (T_{4r})}
+C\,\varep^{1+\rho} \|\nabla^2 u_0\|_{C^{0,\rho} (T_{4r})}.
\endaligned
$$

Finally, using the estimates
$$
|\nabla_x \widetilde{G}_\varep (x,y)|\le C\, \text{dist} (y, \partial\widetilde{\Omega})\, |x-y|^{-d}
$$
and $|\nabla_x \widetilde{G}_\varep (x,y)|
\le C\, |x-y|^{1-d}$ as well as the observation
\begin{equation}\label{Lip-corrector-estimate}
\big|\nabla \Big\{ \Phi_{\varep, j} -P_j -\varep\chi_j (x/\varep) \Big\} \big|
\le C\, \min\Big\{ 1, \varep\, \big[\text{dist}(x, \partial\widetilde{\Omega}) \big]^{-1} \Big\},
\end{equation}
 we may bound the second term in the RHS of (\ref{7.2-4-3})  by
$$
\aligned
& C \|\nabla^2 u_0\|_{L^\infty (T_{4r})} 
\left\{ \varep \int_{\widetilde{\Omega}\setminus B(x,\varep)} 
\frac{dy}{|x-y|^d}
+\int_{\widetilde{\Omega}\cap B(x,\varep)} \frac{dy}{|x-y|^{d-1}}\right\}\\
& \le C\, \varep \ln [\varep^{-1} r +2] \|\nabla^2 u_0\|_{L^\infty (T_{4r})}.
\endaligned
$$
We remark that the inequality  (\ref{Lip-corrector-estimate}) follows from the
estimate 
$$
\|\Phi_{\varep, j} -P_j -\varep \chi_j(x/\varep) \|_{L^\infty(\Omega)}\le C \, \varep
$$
and the interior Lipschitz estimate for $\mathcal{L}_\varep$.
As a result, we have proved that
$$
\|\nabla \theta_\varep\|_{L^\infty(T_{3r})}
\le C\, \varep \ln [\varep^{-1} r +2]\|\nabla^2 u_0\|_{L^\infty (T_{4r})}
+C\varep^{1+\rho} \|\nabla^2 u_0\|_{C^{0,\rho} (T_{4r})}.
$$
This, together with (\ref{7.2-4-1}), completes the proof of (\ref{estimate-lemma-7.2-3}).
\end{proof}

\begin{proof}[\bf Proof of Theorem \ref{theorem-7.2-2}]
Fix $x_0$, $y_0\in \Omega$ and $r=|x_0-y_0|/8$.
We may assume that $0<\varep < r$, since the case $\varep\ge r$ is trivial and
follows directly from the size
estimates of $|\nabla_x G_\varep (x,y)|$, $|\nabla_xG_0(x,y)|$  and
$\|\nabla \Phi_\varep\|_{L^\infty(\Omega)} \le C$.

Let $u_\varep (x)=G_\varep (x,y_0)$ and $u_0 (x)=G_0(x,y_0)$.
Observe that  $\mathcal{L}_\varep (u_\varep)=\mathcal{L}_0 (u_0)=0$ in $T_{4r}=T(x_0,4r)$
and $u_\varep=u_0=0$ on $I_{4r} =I(x_0, 4r)$.
By Theorem \ref{theorem-7.2-1}, 
$$
\| u_\varep -u_0\|_{L^\infty (T_{4r})}
\le C\, \varep\, r^{1-d}.
$$ 
Also, since $\Omega$ is $C^{2,\eta}$, we have
$\|\nabla u_0\|_{L^\infty (T_{4r})} \le C r^{1-d}$,
$$
\|\nabla^2 u_0\|_{L^\infty(T_{4r})} \le C r^{-d}
\quad \text{ and } \quad 
\|\nabla^2 u_0\|_{C^{0, \rho}(T_{4r})} \le Cr^{-d-\rho}.
$$
Hence, by Lemma \ref{lemma-7.2-3}, we obtain 
$$
\big\|\frac{\partial u^\alpha_\varep}{\partial x_i}
-\frac{\partial}{\partial x_i} \Big\{ \Phi_{\varep, j}^{\alpha\beta}\Big\}
\cdot \frac{\partial u_0^\beta}{\partial x_j}\big \|_{L^\infty(T_r)}
\le C\, \varep\, r^{-d} \ln [ \varep^{-1} r +2].
$$
This finishes the proof.
\end{proof}

Let $G_\varep^* (x,y)=\big( G_\varep^{*\alpha\beta} (x,y)\big)_{m\times m}$ 
denote the matrix of Green's functions for $\mathcal{L}_\varep^*$,
the adjoint of $\mathcal{L}_\varep$. 
Since $A^*$ satisfies the same conditions as $A$, by Theorem \ref{theorem-7.2-2},
\begin{equation}\label{adjoint-estimate-1}
\aligned
\big| \frac{\partial}{\partial x_i} \Big\{ G_\varep^{*\alpha\gamma}(x,y)\Big\}
 &-\frac{\partial}{\partial x_i} \Big\{ \Phi_{\varep, j}^{*\alpha\beta} (x) \Big\} \cdot
\frac{\partial }{\partial x_j}\Big\{ G_0^{*\beta\gamma} (x,y)\Big\} \big|\\
&\le \frac{C \, \varep \ln [ \varep^{-1}|x-y| +2]}{|x-y|^d},
\endaligned
\end{equation}
where $\Phi_\varep^* =\big(\Phi_{\varep, j}^{*\alpha\beta} (x) \big)_{m\times m}$ 
denotes the matrix of Dirichlet correctors for $\mathcal{L}_\varep^*$
in $\Omega$.
Using $G_\varep^{*\alpha\beta} (x,y) =G_\varep^{\beta\alpha} (y,x)$, we obtain
\begin{equation}\label{adjoint-estimate-2}
\aligned
\big| \frac{\partial}{\partial y_i} \Big\{ G_\varep^{\gamma\alpha}(x,y)\Big\}
&-\frac{\partial}{\partial y_i} \Big\{ \Phi_{\varep, j}^{*\alpha\beta} (y) \Big\} \cdot
\frac{\partial }{\partial y_j}\Big\{ G_0^{\gamma\beta} (x,y)\Big\} \big|\\
&\le \frac{C \, \varep \ln [ \varep^{-1}|x-y| +2]}{|x-y|^d}.
\endaligned
\end{equation}
This leads to an asymptotic expansion of the Poisson kernel for $\mathcal{L}_\varep$
on $\Omega$.

 Let $(h^{\alpha\beta} (y))$ denote the inverse matrix of 
$\big( n_i(y)n_j(y)\hat{a}_{ij}^{\alpha\beta}\big)_{m\times m}$.

\begin{thm}\label{Poisson-kernel-theorem}
Suppose that $A$ satisfies the same conditions as in Theorem \ref{theorem-7.2-2}.
Let $P_\varep (x,y) =\big( P_\varep^{\alpha\beta}(x,y) \big)_{m\times m}$ denote the Poisson kernel
for $\mathcal{L}_\varep$ in a bounded $C^{2,\eta}$ domain $\Omega$.
Then 
\begin{equation}\label{Poisson-kernel-estimate}
P_\varep^{\alpha\beta} (x,y)=P_0^{\alpha\gamma}(x,y) \omega_\varep^{\gamma\beta}(y)
+R_\varep^{\alpha\beta} (x,y),
\end{equation}
where
\begin{equation}\label{definition-of-omega}
\omega_\varep^{\gamma\beta} (y)
=h^{\gamma \sigma} (y) \cdot \frac{\partial}{\partial n(y)}
\Big\{ \Phi_{\varep, k}^{*\rho\sigma} (y) \Big\} \cdot n_k(y) \cdot
n_i(y)n_j(y) a_{ij}^{\rho\beta} (y/\varep),
\end{equation}
and
\begin{equation}\label{remainder-estimate}
|R_\varep^{\alpha\beta} (x,y)|
\le \frac{C \, \varep \ln [ \varep^{-1}|x-y| +2]}{|x-y|^d}
\qquad \text{ for any }
x\in \Omega \text{ and } y\in \partial\Omega.
\end{equation}
The constant  $C$ depends
only on $\mu$, $\lambda$, $\tau$, and $\Omega$.
\end{thm}

\begin{proof}
Note that for $x\in \Omega$ and $y\in \partial\Omega$,
\begin{equation}\label{7.2-5-1}
\aligned
P_\varep^{\alpha\beta} (x,y)
& =-n_i(y) a_{ji}^{\gamma\beta} (y/\varep) \frac{\partial}{\partial y_j}
\Big\{ G_\varep^{\alpha\gamma} (x,y)\Big\}\\
& =-\frac{\partial}{\partial n(y)} \Big\{ G_\varep^{\alpha\gamma} (x,y)\Big\}
\cdot n_i(y)n_j(y) a_{ij}^{\gamma\beta} (y/\varep),
\endaligned
\end{equation}
where the second equality follows from the fact $G_\varep (x, \cdot)=0$ on $\partial\Omega$.
 By (\ref{adjoint-estimate-2}),
we obtain
\begin{equation}\label{7.2-5-3}
\aligned
& \big| P_\varep^{\alpha\beta} (x,y)
+\frac{\partial}{\partial n(y)}
\Big\{ G_0^{\alpha \sigma} (x,y)\Big\}\cdot
\frac{\partial}{\partial n(y)} \Big\{ \Phi_{\varep, k}^{*\gamma \sigma} (y)\Big\}
\cdot 
n_i(y)n_j(y) a_{ij}^{\gamma\beta} (y/\varep) n_k (y)\big|\\
&\qquad\qquad \le \frac{C \, \varep \ln [ \varep^{-1}|x-y| +2]}{|x-y|^d}.
\endaligned
\end{equation}
In view of (\ref{7.2-5-1}) (with $\varep=0$), we have
$$
P_0^{\alpha\beta} (x,y) h^{\beta\sigma} (y) =-\frac{\partial}{\partial n(y)} 
\Big\{ G_0^{\alpha \sigma} (x,y)\Big\}.
$$
This, together with (\ref{7.2-5-3}), gives
$$
|P_\varep^{\alpha\beta} (x,y)- P_0^{\alpha\gamma} (x,y) \omega_\varep^{\gamma\beta} (y)|
\le 
\frac{C \, \varep \ln [ \varep^{-1}|x-y| +2]}{|x-y|^d},
$$
for any $x\in \Omega$ and $y\in \partial\Omega$,
where $\omega_\varep (y)$ is defined by (\ref{definition-of-omega}).
\end{proof}

We end this section with an asymptotic expansion for $\nabla_x\nabla_y G_\varep (x,y)$.

\begin{thm}\label{theorem-7.2-6}
Let $A$ satisfy the same conditions as in Theorem \ref{theorem-7.2-2}.
Let $\Omega$ be a bounded $C^{3,\eta}$ domain for some
$\eta\in (0,1)$.
Then
\begin{equation}\label{estimate-7.2-6}
\aligned
\big| \frac{\partial^2}{\partial x_i\partial y_j}\Big\{ G_\varep^{\alpha\beta} (x,y)\Big\}
& -\frac{\partial}{\partial x_i}
\Big\{ \Phi_{\varep, k}^{\alpha\gamma} (x)\Big\}\cdot
 \frac{\partial^2}{\partial x_k\partial y_\ell}\Big\{ G_0^{\gamma \sigma} (x,y)\Big\}
\cdot
\frac{\partial}{\partial y_j} \Big\{ \Phi_{\varep, \ell}^{*\beta \sigma} (y)\Big\}
\big|\\
&\le \frac{C\varep \ln \big[ \varep^{-1} |x-y|+2\big]}
{|x-y|^{d+1}}
\endaligned
\end{equation}
for any $x,y\in \Omega$ and $x\neq y$, where $C$ depends only on $\mu$, $\lambda$, $\tau$,
and $\Omega$.
\end{thm}

\begin{proof}
Fix $x_0, y_0\in \Omega$.
Let $r=|x_0-y_0|/8$.
Since $|\nabla_x\nabla_y G_\varep (x,y)|\le C |x-y|^{-d}$, it suffices to consider the case
$0<\varep<r$. Fix $1\le \beta\le m$ and $1\le j\le d$, let
$$
\left\{
\aligned
u^\alpha_\varep (x) & =\frac{\partial G_\varep^{\alpha\beta}}{\partial y_j}  (x, y_0),\\
u_0^\alpha (x) & = \frac{\partial}{\partial y_j} \Big\{ \Phi_{\varep, \ell}^{*\beta \sigma}\Big\} (y_0)
\cdot \frac{\partial  G_0^{\alpha \sigma}}{\partial y_\ell} (x, y_0)
\endaligned
\right.
$$
in $ T_{4r}=\Omega\cap B(x_0, 4r)$.
Observe that $\mathcal{L}_\varep (u_\varep)=\mathcal{L}_0 (u_0)=0$ in $T_{4r}$
and $u_\varep =u_0=0$ on $I_{4r}$.
It follows from  (\ref{adjoint-estimate-2}) that
\begin{equation}\label{7.2-6-1}
\| u_\varep -u_0\|_{L^\infty (T_{4r})}
\le C\, \varep\, r^{-d} \ln \big[\varep^{-1} r +2\big].
\end{equation}
Since $\Omega$ is $C^{3,\eta}$, we have $\|\nabla u_0\|_{L^\infty(T_{4r})} \le Cr^{-d}$, 
\begin{equation}\label{7.2-6-3}
\|\nabla^2 u_0\|_{L^\infty (T_{4r})} \le Cr^{-d-1}
\quad \text{ and }\quad 
\|\nabla^2 u_0\|_{C^{0,\eta}(T_{4r})} \le Cr^{-d-1-\eta}.
\end{equation}
By Lemma \ref{lemma-7.2-3}, estimates (\ref{7.2-6-1}) and (\ref{7.2-6-3}) imply that
$$
\Big\|\frac{\partial u_\varep^\alpha}{\partial x_i}
-\frac{\partial}{\partial x_i} \Big\{ \Phi_{\varep, k}^{\alpha\gamma}\Big\}
\cdot \frac{\partial u_0^\gamma}{\partial x_k}
\Big\|_{L^\infty (T_r)}
\le \frac{C\,\varep \ln \big[\varep^{-1} r+2\big]}{r^{d+1}}.
$$
This gives the desired estimate (\ref{estimate-7.2-6}).
\end{proof}



\section{Asymptotic expansions of Neumann functions}\label{section-7.3}

Throughout this section we will assume that $A$ is 1-periodic and satisfies the ellipticity condition
(\ref{s-ellipticity}) and the H\"older continuity condition (\ref{smoothness}).
Let $N_\varep (x,y)=\big(N^{\alpha\beta}_\varep (x,y)\big)$ denote the $m\times m$ 
matrix of Neumann functions for
$\mathcal{L}_\varep$ in $\Omega$.
Under the assumption that  $\Omega$ is $C^{1,\eta}$,
it is proved in Chapter \ref{chapter-4}  that if $d\ge 3$,
\begin{equation}\label{7.3-1}
 |N_\varep (x,y)|  \le C |x-y|^{2-d},
 \end{equation}
and 
\begin{equation}\label{7.3-2}
\aligned
 |\nabla_x N_\varep (x,y)|+|\nabla_y N_\varep (x,y)|&\le C |x-y|^{1-d},\\
 |\nabla_x\nabla_y N_\varep (x,y)|& \le C |x-y|^{-d},
 \endaligned
 \end{equation}
 for any $x,y\in \Omega$ and $x\neq y$, where $C$ depends only on $\mu$, $(\lambda, \tau)$, and $\Omega$.
 In the case $d=2$, estimates (\ref{7.3-2}) continue to hold, while (\ref{7.3-1}) is replaced by
 \begin{equation}\label{7.3-0}
 |N_\varep (x, y)| \le C \Big\{ 1+\ln (r_0 |x-y|^{-1}) \Big\},
 \end{equation}
 where $r_0=\text{diam}(\Omega)$.
In this section we investigate the asymptotic behavior, as $\varep\to 0$,
of $N_\varep (x,y)$, $\nabla_x N_\varep(x,y)$, $\nabla_y N_\varep(x,y)$, and $\nabla_x\nabla_y N_\varep (x,y)$.
We shall use $N_0(x,y)=\big(N_0^{\alpha\beta}(x,y)\big)$ to denote the $m\times m$ matrix of Neumann
functions for $\mathcal{L}_0$ in $\Omega$.

\begin{thm}\label{theorem-7.3-1}
Let $\Omega$ be a bounded $C^{1,1}$ domain in $\br^d$. Then
\begin{equation}\label{estimate-theorem-7.3-1}
|N_\varep (x,y)-N_0(x,y)|
\le \frac{ C\, \varep \ln [\varep^{-1} |x-y| +2]}{|x-y|^{d-1}}
\end{equation}
for any $x,y\in \Omega$ and $x\neq y$, where $C$ depends only on $\mu$, $(\lambda, \tau)$, and $\Omega$.
\end{thm}

As in the case of Green functions, Theorem \ref{theorem-7.3-1} is a consequence of an $L^\infty$
estimate for local solutions. Recall that  $T_r =B(x_0, r)\cap\Omega$ and 
$I(r)=B(x_0,r)\cap \partial\Omega$, where $x_0\in \overline{\Omega}$
and $0<r<c_0\,\text{diam}(\Omega)$.

\begin{lemma}\label{lemma-7.3-1}
Let $\Omega$ be a bounded $C^{1,\eta}$ domain for some $\eta\in (0,1)$.
Let $u_\varep\in H^1(T_{3r};\br^m)$ and $ u_0\in W^{2,p}(T_{3r};\br^m)$ for some $p>d$.
Suppose that 
$$
\mathcal{L}_\varep (u_\varep)=\mathcal{L}_0(u_0)\quad \text{ in }  T_{3r} \quad 
\text{ and } \quad \frac{\partial u_\varep}{\partial\nu_\varep}=\frac{\partial u_0}{\partial\nu_0}\quad
\text{ on } I_{3r}.
$$
Then,  if $0<\varep<(r/2)$,
\begin{equation}\label{estimate-lemma-7.3-1}
\aligned
\| u_\varep -u_0\|_{L^\infty(T_r)}
&\le C \average_{T_{3r}} |u_\varep -u_0|
+C\varep \ln [\varep^{-1}r+2] \|\nabla u_0\|_{L^\infty(T_{3r})}\\
&\qquad\qquad
+C_p\, \varep\, r^{1-\frac{d}{p}} \|\nabla^2 u_0\|_{L^p(T_{3r})}.
\endaligned
\end{equation}
\end{lemma}

\begin{proof}
By rescaling we may assume that $r=1$.
Choose a $C^{1,\eta}$ domain $\widetilde{\Omega}$ 
such that $T_2\subset \widetilde{\Omega}\subset T_3$.
Let 
$$
w_\varep =u_\varep (x)-u_0(x)- \varep \chi_j^\beta (x/\varep) \frac{\partial u_0^\beta}
{\partial x_j}.
$$
Recall that
\begin{equation}\label{7.3-1-1}
\left(\mathcal{L}(w_\varep) \right)^\alpha =-\varep 
\frac{\partial}{\partial x_i} \left\{ \left[ \phi_{jik}^{\alpha\beta} (x/\varep)
 -a_{ij}^{\alpha\gamma} (x/\varep) \chi_k^{\gamma\beta}(x/\varep)\right]
\frac{\partial^2 u_0^\beta}{\partial x_j\partial x_k}\right\}.
\end{equation}
Using (\ref{definition-of-F}), one may verify that
\begin{equation}\label{7.3-1-2}
\aligned
\left(\frac{\partial w_\varep}{\partial \nu_\varep}\right)^\alpha
= & \left(\frac{\partial u_\varep}{\partial \nu_\varep}\right)^\alpha
-\left(\frac{\partial u_0}{\partial \nu_0}\right)^\alpha
-
\frac{\varep}{2}
\left( n_i\frac{\partial}{\partial x_j} -n_j\frac{\partial}{\partial x_i}\right)
\left\{ \phi_{jik}^{\alpha\gamma}(x/\varep) \frac{\partial u_0^\gamma}
{\partial x_k}\right\}\\
& \qquad
+\varep n_i \left[ \phi_{jik}^{\alpha\beta} (x/\varep) -a_{ij}^{\alpha\gamma} (x/\varep) \chi_k ^{\gamma\beta} (x/\varep)\right]
\frac{\partial^2 u_0^\beta}{\partial x_j\partial x_k}.
\endaligned
\end{equation}
We point out  that $ n_i\frac{\partial}{\partial x_j} -n_j\frac{\partial}{\partial x_i}$ is a tangential derivative on the boundary and possesses 
the following property (integration by parts on $\partial\Omega$),
\begin{equation}\label{tang-p}
\int_{\partial\Omega} 
\left( n_i\frac{\partial}{\partial x_j} -n_j\frac{\partial}{\partial x_i}\right)u \cdot v\, d\sigma
=-\int_{\partial\Omega}
u\cdot \left( n_i\frac{\partial}{\partial x_j} -n_j\frac{\partial}{\partial x_i}\right)v \, d\sigma
\end{equation}
for any $u, v\in C^1(\overline{\Omega})$. The equality (\ref{tang-p}) may be proved by using the divergence theorem 
and a simple approximation argument.

Let $w_\varep =\theta_\varep  +z_\varep$, where
\begin{equation}\label{7.3-1-3}
\big( \theta_\varep (x)\big)^\alpha
=\varep \int_{\widetilde{\Omega}} \frac{\partial}{\partial y_i}
\Big\{ \wN^{\alpha\beta}_\varep (x,y)\Big\}\cdot
\left[ \phi_{jik}^{\beta\gamma} (y/\varep) -a_{ij}^{\beta\sigma} (y/\varep) \chi_k^{\sigma \gamma} (y/\varep)\right]
\frac{\partial^2 u_0^\gamma}
{\partial y_j\partial y_k} \, dy
\end{equation}
and $\widetilde{N}_\varep (x,y)$
denotes the matrix of Neumann functions for $\mathcal{L}_\varep$
in $\widetilde{\Omega}$.
Since $|\nabla_y \wN_\varep (x,y)|\le C |x-y|^{1-d}$, it follows by H\"older's
inequality that 
\begin{equation}\label{7.3-1-4-1}
\|\theta_\varep \|_{L^\infty (T_2)} \le C_p \, \varep \|\nabla^2 u_0\|_{L^p(T_3)}
\qquad \text{ for any } p>d.
\end{equation}

To estimate $z_\varep$, we observe that
$\mathcal{L}_\varep (z_\varep ) =0$ in $\widetilde{\Omega}$ and
\begin{equation}\label{7.3-1-4}
\left(\frac{\partial z_\varep}{\partial \nu_\varep}\right)^\alpha
= \left(\frac{\partial u_\varep}{\partial \nu_\varep}\right)^\alpha
-\left(\frac{\partial u_0}{\partial \nu_0}\right)^\alpha
-
\frac{\varep}{2}
\left( n_i\frac{\partial}{\partial x_j} -n_j\frac{\partial}{\partial x_i}\right)
\left\{ \phi_{jik}^{\alpha\gamma}(x/\varep) \frac{\partial u_0^\gamma}
{\partial x_k}\right\}
\end{equation}
on $\partial\widetilde{\Omega}$. Let $z_\varep =z_\varep^{(1)} +z_\varep^{(2)}$, where
\begin{equation}\label{7.3-1-5}
\big( z_\varep^{(1)}\big)^\alpha (x)
=\frac{\varep}{2}
\int_{\partial\widetilde{\Omega}} 
\left( n_i\frac{\partial}{\partial y_j} -n_j\frac{\partial}{\partial y_i}\right)
\Big\{ \wN^{\alpha\beta}_\varep (x,y)\Big\}
\cdot \phi_{jik}^{\beta\gamma}(y/\varep) \frac{\partial u_0^\gamma}
{\partial y_k}\, d\sigma (y).
\end{equation}
For each $x\in \widetilde{\Omega}$, choose 
$\hat{x}\in\partial \widetilde{\Omega}$ such that $|\hat{x}-x|=\text{dist}(x, \partial\widetilde{\Omega})$.
Note that for $y\in \partial\widetilde{\Omega}$, 
$$
|y-\hat{x}|
\le |y-x| +|x-\hat{x}|\le 2|y-x|.
$$
 Hence,
$|\nabla_y \wN_\varep (x,y)|\le C |y-\hat{x}|^{1-d}$ and
$$
\aligned
|\big( z_\varep^{(1)}\big)^\alpha (x)|
& = \frac{\varep}{2}
\left|\int_{\partial\widetilde{\Omega}} 
\left( n_i\frac{\partial}{\partial y_j} -n_j\frac{\partial}{\partial y_i}\right)
\Big\{ \wN^{\alpha\beta}_\varep (x,y)\Big\}
\cdot \big\{ f_{ji}^\beta (y)-f_{ji}^\beta (\hat{x})\big\} \, d\sigma (y)\right|\\
&
\le C\varep \int_{\partial\widetilde{\Omega}}
\frac{|f(y)-f(\hat{x})|}{|y-\hat{x}|^{d-1}}\, d\sigma(y),
\endaligned
$$
where $f(y)=(f_{ji}^\beta (y)) =\big(\phi_{jik}^{\beta\gamma}(y/\varep) \frac{\partial u_0^\gamma}
{\partial y_k} (y)\big)$.
Since 
$$
\| f\|_{L^\infty(T_3)} \le C \|\nabla u_0\|_{L^\infty(T_3)}
$$
 and
$$
|f(y)-f(\hat{x})|
\le C\varep^{-1} |y-\hat{x}| \|\nabla u_0\|_{L^\infty(T_3)}
+C \, |y-\hat{x}|^\rho \|\nabla u_0\|_{C^{0,\rho}(T_3)},
$$
where $0<\rho<\eta$ and we have used the fact $\|\phi_{jik}^{\beta\gamma}\|_{C^1 (Y)}\le C$, it follows that
$$
\aligned
|z_\varep^{(1)} (x)|
& \le C\varep \| \nabla u_0\|_{L^\infty(T_3)}
\int_{\partial\widetilde{\Omega}\setminus B(\hat{x}, \varep)} |\hat{x}-y|^{1-d} \, d\sigma(y)\\
&\qquad \qquad+ C\|\nabla u_0\|_{L^\infty(T_3)}\int_{B(\hat{x}, \varep)\cap \partial\widetilde{\Omega}} 
|y-\hat{x}|^{2-d}\, d\sigma (y)\\
&\qquad \qquad+C\varep \|\nabla u_0\|_{C^{0,\rho}(T_3)}
\int_{B(\hat{x}, \varep)\cap \partial\widetilde{\Omega}} 
 |y-\hat{x}|^{1-d+\rho}\, d\sigma (y)\\
& \le C\varep \ln[\varep^{-1}+2] \|\nabla u_0\|_{L^\infty(T_3)}
+C\varep^{1+\rho} \| \nabla u_0\|_{C^{0,\rho}(T_3)}.
\endaligned
$$ 
By Sobolev imbedding, this implies that
\begin{equation}\label{7.3-1-6}
\|z_\varep^{(1)}\|_{L^\infty(T_2)}
 \le C\varep \ln[\varep^{-1}+2] \|\nabla u_0\|_{L^\infty(T_3)}
+ C_p\, \varep \|\nabla^2 u_0\|_{L^p(T_3)}
\end{equation}
for any $p>d$.

Finally, to estimate $z_\varep^{(2)}$, we note that
$\mathcal{L}_\varep (z_\varep^{(2)}) =0$ in $T_2$ and
$$
\frac{\partial}{\partial\nu_\varep}\Big\{ z_\varep^{(2)}\Big\}=
\frac{\partial u_\varep}{\partial\nu_\varep} -
\frac{\partial u_0}{\partial\nu_0} =0\quad \text{ on } I_2.
$$
It follows from the boundary H\"older estimate (\ref{boundary-holder-Neumann-estimate}) 
as well as interior estimates that
$$
\aligned
\| z_\varep^{(2)}\|_{L^\infty(T_1)}
&\le C \average_{T_2} |z_\varep^{(2)}| \\
& \le C \average_{T_2} |u_\varep-u_0|
+C\varep \|\nabla u_0\|_{L^\infty(T_2)}
+C \|\theta_\varep\|_{L^\infty(T_2)}
+C \|z_\varep^{(1)} \|_{L^\infty(T_2)}.
\endaligned
$$
Hence,
$$
\|u_\varep -u_0\|_{L^\infty(T_1)}
\le  C \average_{T_2} |u_\varep-u_0|
+C\varep \|\nabla u_0\|_{L^\infty(T_2)}
+C \|\theta_\varep\|_{L^\infty(T_2)}
+C \|z_\varep^{(1)}\|_{L^\infty(T_2)}.
$$
This, together with (\ref{7.3-1-4-1}) and (\ref{7.3-1-6}), gives the estimate (\ref{estimate-lemma-7.3-1}).
\end{proof}

\begin{lemma}\label{lemma-7.3-2}
Let $\Omega$ be a bounded Lipschitz domain in $\br^d$. Let
$$
u(x)=\int_\Omega \frac{g(y) }{|x-y|^{d-1}}\, dy \quad
\text{ and } \quad v(x)=\int_{\partial\Omega} \frac{f(y)}{|x-y|^{d-1}}\, d\sigma (y).
$$
Then
\begin{equation}\label{frac-1}
\aligned
\| u\|_{L^2(\partial\Omega)} & \le C \| g\|_{L^q(\Omega)},\\
\| v\|_{L^p(\Omega)} & \le C \| f\|_{L^2(\partial\Omega)},
\endaligned
\end{equation}
where $p=\frac{2d}{d-1}$, $q=p^\prime=\frac{2d}{d+1}$, and $C$ depends only on $\Omega$.
\end{lemma}

\begin{proof}
It follows from (\ref{bl-estimate-1}) that
$$
\| u\|_{L^2(\partial \Omega)}
\le C \| u\|_{W^{1, q}(\Omega)}
\le C \| g\|_{L^q(\Omega)}.
$$
The estimate for $v$ follows from that for $u$ by duality.
\end{proof}

\begin{proof}[\bf Proof of Theorem \ref{theorem-7.3-1}]
By rescaling we may assume that diam$(\Omega)= 1$.
Fix $x_0,y_0\in \Omega$ and let $r=|x_0-y_0|/8$. 
Since $|N_\varep (x_0,y_0)|\le Cr^{2-d}$, we may also assume that
$0<\varep<r$.
For $g\in C_0^\infty(T (y_0,r);\br^m)$ and $\varep\ge 0$,
let
$$
u_\varep (x)=\int_\Omega N_\varep (x,y) g(y)\, dy.
$$
Then $\mathcal{L}_\varep (u_\varep)=g$ in $\Omega$,
$\frac{\partial u_\varep}{\partial\nu_\varep}
=-\frac{1}{|\partial\Omega|}\int_\Omega g$ on $\partial\Omega$
and $\int_{\partial\Omega} u_\varep =0$.
It follows that $\mathcal{L}_\varep (u_\varep)=\mathcal{L}_0 (u_0)$
in $\Omega$ and $\frac{\partial u_\varep}{\partial\nu_\varep}=
\frac{\partial u_0}{\partial\nu_0}$ on $\partial\Omega$.
Let
$
w_\varep =u_\varep (x) -u_0(x) -\varep \chi_j^\beta (x/\varep)
\frac{\partial u_0^\beta}{\partial x_j}.
$
As in the proof of Lemma \ref{lemma-7.3-1},
$\mathcal{L}(w_\varep)$ and $\frac{\partial w_\varep}{\partial \nu_\varep}$
 are given by (\ref{7.3-1-1}) and (\ref{7.3-1-2}), respectively.
Now, write $w_\varep =\theta_\varep +z_\varep +\rho$, where
$\mathcal{L}_\varep (z_\varep)=\mathcal{L}_\varep (w_\varep)$ in $\Omega$,
$\int_\Omega z_\varep=0$,
$$
\frac{\partial z_\varep}{\partial\nu_\varep}
=-\varep n_i\left[ F_{jik}^{\alpha\beta}(x/\varep) +a_{ij}^{\alpha\gamma} (x/\varep) \chi_k^{\gamma\beta}(x/\varep)\right]
\frac{\partial^2 u_0^\beta}{\partial x_j\partial x_k}
\qquad \text{ on } \partial\Omega,
$$
and $\rho=\average_{\partial\Omega} \big\{ w_\varep -z_\varep\big\}$ 
is a constant.
Note that 
\begin{equation}\label{ee70}
\|\nabla z_\varep\|_{L^2(\Omega)} \le C \varep \| \nabla^2 u_0\|_{L^2(\Omega)}
\le C\varep \| g\|_{L^2(\Omega)}.
\end{equation}
Since $\int_\Omega z_\varep =0$, by the Poincar\'e inequality, we obtain
$\| z_\varep\|_{L^p(\Omega)} \le C\varep \| g\|_{L^2(\Omega)}$,
where $d\ge 3$ and $p=\frac{2d}{d-2}$.
It follows by H\"older's inequality that
\begin{equation}\label{7.3-3-1}
\| z_\varep\|_{L^2(T(x_0,r))}
\le Cr^{\frac{d}{2}-\frac{d}{p}} \| z_\varep\|_{L^p(T(x_0, r))}
\le C\varep r\| g\|_{L^2(T(y_0, r))}.
\end{equation}
For the case $d=2$, in the place of (\ref{ee70})-(\ref{7.3-3-1}),
we use a Meyers type estimate to obtain 
\begin{equation}\label{ee71}
\| \nabla z_\e\|_{L^{\bar{p}}(\Omega)}
\le C \e \|\nabla^2 u_0\|_{L^{\bar{p}}(\Omega)} \le C \e \| g\|_{L^{\bar{p}}(\Omega)}
\end{equation}
for some $\bar{p}<2$. See .
This gives
$$
\aligned
\|z_\e \|_{L^2(T(x_0, r))}
&\le C r^{1-\frac{2}{q}} \| z_\e\|_{L^q(\Omega)} 
\le C r^{1-\frac{2}{q}} \|\nabla z_\e \|_{L^{\bar{p}}(\Omega)}\\
&\le C \e r^{1-\frac{2}{q}} \| g\|_{L^{\bar{p}}(\Omega)}
\le C \e r \| g\|_{L^2(T(y_0, r))},
\endaligned
$$
where $\frac{1}{q}=\frac{1}{\bar{p}} -\frac12$.

Next, to estimate $\theta_\varep$, we observe that $\mathcal{L}_\varep (\theta_\varep)
=0$ in $\Omega$, $\int_{\partial\Omega} \theta_\varep =0$ and
$$
\frac{\partial \theta_\varep}{\partial \nu_\varep}
=-\frac{\varep}{2}
\left( n_i\frac{\partial}{\partial x_j} -n_j\frac{\partial}{\partial x_i}\right)
\left\{ F_{jik}^{\alpha \gamma}(x/\varep) \frac{\partial u_0^\gamma}
{\partial x_k}\right\}.
$$
It follows that
$$
\theta^\alpha_\e (x)=-\frac{\e}{2}\int_{\partial\Omega}
\left( n_i\frac{\partial}{\partial y_j} -n_j\frac{\partial}{\partial y_i}\right) N^{\alpha\beta}_\e (x, y)
\left\{ F_{jik}^{\beta\gamma}(y/\varep) \frac{\partial u_0^\gamma}
{\partial y_k}\right\} d\sigma (y).
$$
Using the estimate $|\nabla_y N_\e (x, y)|\le C |x-y|^{1-d}$, we obtain
$$
|\theta_\e (x)|\le C \e \int_{\partial\Omega} \frac{|\nabla u_0 (y)|}{|x-y|^{d-1}} \, d\sigma (y).
$$
In view of Lemma \ref{lemma-7.3-2} we see that
$$
\|\theta_\varep\|_{L^p(\Omega)}
\le C\varep \| \nabla u_0\|_{L^2(\partial\Omega)},
$$
where $p=\frac{2d}{d-1}$.
By H\"older's inequality, this
gives
\begin{equation}\label{7.3-3-3}
\| \theta_\varep\|_{L^2(T(x_0,r))}
\le C\varep r^{\frac12}\|\nabla u_0\|_{L^2(\partial\Omega)}.
\end{equation}
Since 
$$
|\nabla u_0 (x)|\le C \int_\Omega \frac{|g(y)|\, dy}{|x-y|^{d-1}},
$$
we may invoke Lemma \ref{lemma-7.3-2} again  to claim  that 
$$
\|\nabla u_0\|_{L^2(\partial\Omega)}
\le C \| g\|_{L^q(T(y_0, r))}
\le C r^{1/2} \| g\|_{L^2(T(y_0, r))},
$$
where $q=\frac{2d}{d+1}$.
 In view of (\ref{7.3-3-3})
we obtain 
$$
\| \theta_\varep\|_{L^2(T(x_0,r))} \le C\varep r \| g\|_{L^2(T(y_0, r))}.
$$
This, together with (\ref{7.3-3-1}) and the observation
$$
|\rho|\le C \int_{\partial\Omega} 
\big\{ \varep |\nabla u_0| +|z_\varep|\big\} d\sigma \le C \varep \| g\|_{L^2(T(y_0, r))},
$$
gives 
$$
\| w_\varep \|_{L^2(T(x_0,r))} \le C\varep r \| g\|_{L^2(T(y_0, r))}.
$$
It follows that
$$
\left\{\average_{T(x_0,r)} |u_\varep -u_0|^2\right\}^{1/2}
\le C \varep r^{\frac{2-d}{2}} \| g\|_{L^2(T(y_0, r))}.
$$
Since $\|\nabla u_0\|_{L^\infty (T(x_0,r))}
\le C r^{\frac{2-d}{2}} \| g\|_{L^2(\Omega)}$
and $\|\nabla^2 u_0\|_{L^p (\Omega)} \le C \| g\|_{L^p(\Omega)}$,
by Lemma \ref{lemma-7.3-1}, we obtain
$$
|u_\varep (x_0)-u_0 (x_0)|\le C_p\, \varep r^{1-\frac{d}{p}} \ln [\varep^{-1} r+2]
\| g\|_{L^p(T(y_0, r))},
$$
where $p>d$.
By duality this gives
$$
\left\{\average_{T(y_0,r)}
|N_\varep (x_0,y)-N_0(x_0,y)|^{p^\prime}\, dy\right\}^{1/p^\prime}
\le C_p\,  \varep r^{1-d} \ln [\varep^{-1} r+2].
$$

Finally, since 
$$ 
\frac{\partial}{\partial \nu_\varep (y)}
\big\{ N_\varep (x,y)\big\}
=\frac{\partial}{\partial \nu_0 (y)}
\big\{ N_0 (x,y)\big\}=-\frac{I_{m\times m} }{|\partial\Omega|} \qquad \text{ on }\partial\Omega,
$$
$|\nabla_y N_0 (x,y)|\le C|x-y|^{1-d}$ and
$\| \nabla^2_y N_0(x_0, y)\|_{L^p(T(y_0,r))}
\le Cr^{\frac{d}{p}-d}$,
we may invoke Lemma \ref{lemma-7.3-1} again to obtain
$$
\aligned
|N_\varep (x_0, y_0)-N_0 (x_0,y_0)|
& \le \average_{T(y_0,r)}
|N_\varep (x_0,y)-N_0(x_0,y)|\, dy\\
&\qquad\qquad
+C\varep r^{1-d} \ln [\varep^{-1} r+2]\\
&\le 
C\varep r^{1-d} \ln [\varep^{-1} r+2].
\endaligned
$$
This completes the proof.
\end{proof}

\begin{definition}
{\rm
For $1\le j\le d$ and $1\le \beta\le m$, the Neumann corrector 
$\Psi_{\e, j}^\beta =\big( \Psi_{\e, j}^{\alpha\beta} \big)$ is defined to be a weak solution of
the Neumann problem,
\begin{equation}\label{definition-of-Psi}
\left\{
\aligned
\mathcal{L}_\e \big(\Psi_{\e, j}^\beta\big)  & =0 & \quad & \text{ in } \Omega,\\
\frac{\partial }{\partial\nu_\e} \big( \Psi_{\e, j}^\beta\big)
&=\frac{\partial }{\partial\nu_0} \big( P_ j^\beta\big)& \quad & \text{ on } \partial\Omega,
\endaligned
\right.
\end{equation}
where $P_j^\beta=x_j e^\beta$.
}
\end{definition}

Since solutions of (\ref{definition-of-Psi}) are unique up to a constant in $\br^m$,
to fix the corrector, we assume that $\Psi_{\e, j}^\beta (x_0)=P_j^\beta (x_0)$ for some $x_0\in \Omega$.
By the boundary Lipschitz estimate in Corollary \ref{NPL-0}, we see that
\begin{equation}\label{Psi-ee}
\|\nabla \Psi_{\e, j}^\beta\|_{L^\infty(\Omega)} \le C,
\end{equation}
where $C$ depends only on $\mu$, $(\lambda, \tau)$, and $\Omega$.

\begin{lemma}\label{N-C-E0}
Suppose that $A$ is 1-periodic and satisfies (\ref{s-ellipticity}) and (\ref{smoothness}).
Let $\Omega$ be a bounded $C^{1, \eta}$ domain in $\br^d$.
Then for any $x\in \Omega$,
\begin{equation}\label{N-C-E1}
|\nabla \big\{  \Psi_{\e, j}^\beta (x)- P_j^\beta (x) -\e \chi_j^\beta (x/\e)\big\} |\le \frac{C \e}{\delta (x)},
\end{equation}
where $\delta (x)=\text{\rm dist}(x, \partial\Omega)$ and $C$ depends only on $\mu$, $(\lambda, \tau)$, and
$\Omega$. Moreover, 
\begin{equation}\label{N-C-E2}
\| \Psi_{\e, j}^\beta -P_j^\beta \|_{L^\infty (\Omega)}
\le C \e \ln \big[ \e^{-1} r_0 +2\big],
\end{equation}
where $r_0=\text{diam}(\Omega)$.
\end{lemma}

\begin{proof}
Fix $1\le j\le d $ and $1\le \beta \le m$.
Let 
$$
w_\e = \Psi_{\e, j}^\beta -P_j^\beta -\e \chi_j^\beta (x/\e).
$$
Then $\mathcal{L}_\e (w_\e)=0$ in $\Omega$.
Using  the formula (\ref{7.3-1-2}) with $u_\e=\Psi_{\e, j}^\beta$ and $u_0=P_j^\beta$, we obtain 
$$
\left(\frac{\partial w_\e}{\partial \nu_\e}\right)^\alpha
=-\frac{\e}{2} \left(n_i\frac{\partial }{\partial x_j} -n_j \frac{\partial}{\partial x_i} \right)
\left\{ \phi_{jik}^{\alpha\gamma} (x/\e) \frac{\partial u_0^\gamma}{\partial x_k} \right\} \quad \text{ on } \partial\Omega.
$$
It follows that for any $x, y\in \Omega$,
$$
w^\alpha_\e (x)-w^\alpha_\e (y)
=\frac{\e}{2} \int_{\partial\Omega}
\left(n_i\frac{\partial }{\partial z_j} -n_j \frac{\partial}{\partial z_i} \right)
\big( N_\e^{\alpha\beta} (x, z) -N^{\alpha\beta}_\e (y, z) \big)
\left\{ \phi_{jik}^{\beta\gamma} (z/\e) \frac{\partial u_0^\gamma}{\partial z_k} \right\}.
d\sigma (z)
$$
Thus, if $|x-y|<(1/2)\delta (x)$, 
$$
\aligned
|w_\e (x) -w_\e (y)|
 &\le C\e  \int_{\partial\Omega} |\nabla_z N_\e (x, y) -N_\e (y, z)|\, d\sigma (z) \\
 & \le C \e |x-y| \int_{\partial\Omega} \frac{d\sigma (z)}{|x-y|^d}\\
 &\le C \e |x-y| [\delta (x)]^{-1},
 \endaligned
 $$
 where we have used the estimate $|\nabla_x \nabla_y N_\e (x, y)|\le C |x-y|^{-d}$.
 This gives the estimate (\ref{N-C-E1}).
 
 To see (\ref{N-C-E2}),  note that $\|\nabla w_\e\|_{L^\infty(\Omega)} \le C$ and thus 
 $$
 |\nabla w_\e(x)|\le C \min \Big\{ \frac{\e}{\delta(x)} , 1\Big\}.
 $$
 This, together with the fact that $|w_\e (x_0)|= \e |\chi(x_0/\e)|$, yields (\ref{N-C-E2})
 by a simple integration.
\end{proof}

The next theorem gives an asymptotic expansion of $\nabla_x N_\varep (x,y)$.

\begin{thm}\label{theorem-7.3-2}
Suppose that $A$ is 1-periodic and satisfies (\ref{s-ellipticity}) and (\ref{smoothness}).
Let $\Omega$ be a bounded $C^{2, \eta}$ domain for some $\eta \in (0,1)$.
Then, for any $\rho\in (0,1)$,
\begin{equation}\label{estimate-7.3-2}
\big| \frac{\partial}{\partial x_i} \left\{ N_\varep^{\alpha\beta} (x,y) \right\}
-\frac{\partial}{\partial x_i} \left\{ \Psi_{\varep, j}^{\alpha\gamma} (x) \right\}
\cdot \frac{\partial}{\partial x_j} \left\{ N_0^{\gamma \beta} (x,y) \right\} \big|
\le \frac{C_\rho \, \varep^{1-\rho} \ln [\varep^{-1} r_0 +2]}{|x-y|^{d-\rho}}
\end{equation}
for any $x,y\in \Omega$ and $x\neq y$,
where  $r_0=\text{\rm diam}(\Omega)$ and
$C_\rho$ depends only on $\mu$, $\rho$, $(\lambda,\tau)$ and $\Omega$.
\end{thm}

We need two lemmas before we carry out the proof of Theorem \ref{theorem-7.3-2}.

\begin{lemma}\label{lemma-7.3-3}
Let $\Omega$ be a bounded $C^{1,\eta}$ domain for some $\eta\in (0,1]$
and
$$
u_\varep (x)=
\int_\Omega \frac{\partial}{\partial y_j}
\big\{ N_\varep (x,y)\big\} f(y)\, dy
$$
for some $1\le j \le d$.
Then
\begin{equation}\label{estimate-lemma-7.3-3}
\|\nabla u_\varep\|_{L^\infty(\Omega)}
\le C\left\{ \ln [\varep^{-1}r_0 +2] + r_0^\eta\right \} \| f\|_{L^\infty(\Omega)}
+ C \varep^\eta H_{\varep, \eta} (f),
\end{equation}
where $r_0=\text{\rm diam} (\Omega)$ and
$$
H_{\varep, \eta} (f) =\sup \left\{ \frac{|f(x)-f(y)|}{|x-y|^\eta}: \ x,y\in \Omega
\text{ and } 0<|x-y|<\varep\right\}.
$$
\end{lemma}

\begin{proof}
For $x\in \Omega$, choose $\hat{x}\in \partial\Omega$
such that $|x-\hat{x}|=\text{\rm dist}(x, \partial\Omega)$.
Note that
$$
\aligned
\frac{\partial u_\e }{\partial x_i} 
= &\int_\Omega 
\frac{\partial^2}{\partial x_i\partial y_j}
\big\{ N_\varep (x,y)\big\}  \cdot \{  f(y)- f(x)\} \, dy\\
 &\quad +
\int_{\partial\Omega}
\big\{ n_j(y)-n_j (\hat{x})\big\} \cdot \frac{\partial}{\partial x_i}
\big\{ N_\varep (x,y) \big\} \cdot f(x)\, d\sigma (y),
\endaligned
$$
where we have used the fact $\int_{\partial\Omega} N_\varep (x,y)\, d\sigma (y)=0$.
This, together with the estimates  $|\nabla_x N_\varep (x,y)|\le C |x-y|^{1-d}$
and $|\nabla_x\nabla_y N_\varep (x,y)|\le C |x-y|^{-d}$, gives
$$
\aligned
 |\nabla u_\varep  (x)|
&\le C \int_\Omega \frac{|f(y)-f(x)|}{|x-y|^d}\, dy
+C \| f\|_{L^\infty(\Omega)}
\int_{\partial\Omega} \frac{d\sigma(y)}{|y-\hat{x}|^{d-1-\eta}}\\
&\le C\| f\|_{L^\infty(\Omega)}
\int_{\Omega\setminus B(x, \varep)} \frac{dy}{|x-y|^d}
+CH_{\varep, \eta} (f) \int_{|y-x|<\varep} 
\frac{dy}{|y-x|^{d-\eta}}
+ C\| f\|_{L^\infty(\Omega)} r_0^\eta\\
&\le C \| f\|_{L^\infty(\Omega)} \ln [\varep^{-1} r_0+2]
+C \varep^\eta H_{\varep, \eta} (f) 
+ C\| f\|_{L^\infty(\Omega)} r_0^\eta.
\endaligned
$$
This completes the proof.
\end{proof}

\begin{lemma}\label{lemma-7.3-4}
Let $\Omega$ be a bounded $C^{2,\eta}$ domain for some $\eta\in (0,1)$.
Suppose that $u_\varep \in H^1(T_{3r})$ and $u_0\in C^{2, \eta}(T_{3r})$ for some
$0<r<c_0 r_0$, where $r_0 =\text{\rm diam}(\Omega)$.
Assume that
$$
\mathcal{L}_\varep (u_\varep)=\mathcal{L}_0(u_0) \quad \text{ in } T_{3r} \quad \text{ and } \quad
\frac{\partial u_\varep}{\partial \nu_\varep}
=\frac{\partial u_0}{\partial \nu_0} \quad \text{ on } I_{3r}.
$$
Then, if $0<\varep<(r/2)$,
\begin{equation}\label{estimate-lemma-7.3-4}
\aligned
& \big\|\frac{\partial u_\varep^\alpha}{\partial x_i}
-\frac{\partial}{\partial x_i}
\big\{ \Psi_{\varep, j}^{\alpha\beta}\big\} 
\cdot \frac{\partial u_0^\beta}{\partial x_j}\big \|_{L^\infty(T_r)}\\
&\le \frac{C}{r}
\average_{T_{3r}} |u_\varep -u_0|
+C\varep r^{-1}\ln [\varep^{-1} r_0+2] \|\nabla u_0\|_{L^\infty(T_{3r})}\\
&\qquad +C\varep^{1-\rho} r^\rho \ln [\varep^{-1}r_0+2] \|\nabla^2 u_0\|_{L^\infty(T_{3r})}
+C\varep r^\rho \ln [\varep^{-1}r_0 +2] \|\nabla^2 u_0\|_{C^{0, \rho}(T_{3r})}
\endaligned 
\end{equation}
for any $0<\rho<\min (\eta, \tau)$.
\end{lemma}

\begin{proof}
By rescaling and translation we may assume that $r=1$ and $0\in T_1$.
Let 
$$
w_\varep =u_\varep -u_0- \big\{ \Psi_{\varep, j}^\beta -P_j^\beta\big\} \cdot 
\frac{\partial u_0^\beta}{\partial x_j}.
$$ 
Choose a $C^{2, \eta}$ domain $\widetilde{\Omega}$ such that $T_{2}\subset \widetilde{\Omega}\subset T_{3}$.
 We now write
$$
w_\varep (x)
=\int_{\widetilde{\Omega}} \wN_\varep (x,y) \mathcal{L}_\varep (w_\varep)\, dy
+\int_{\partial\widetilde{\Omega}} \wN_\varep (x,y) \frac{\partial w_\varep}{\partial\nu_\varep}\, 
d\sigma(y)
+\average_{\partial\widetilde{\Omega}} w_\varep
$$
for $x\in T_2$,
where $\wN_\varep (x,y)$ denotes the matrix of Neumann functions
 for $\mathcal{L}_\varep$ in $\widetilde{\Omega}$.
In view of the formula for $\mathcal{L}_\varep (w_\varep)$ in Lemma \ref{lemma-1.4.2} and
formula (\ref{7.3-1-2}) for $\frac{\partial w_\varep}{\partial\nu_\varep}$, we have
$w_\varep =w_\varep^{(1)} +w_\varep^{(2)}+c$, 
where $c=\average_{\partial\widetilde{\Omega}} w_\varep$,
\begin{equation}\label{7.3-4-1}
\aligned
w_\varep^{(1)} (x)
=& -\varep \int_{\widetilde{\Omega}} \frac{\partial}{\partial y_i}
\big\{ \wN_\varep (x,y)\big\} \cdot
\big\{ F_{jik}(y/\varep)\big\} \cdot
\frac{\partial^2 u_0}{\partial y_j\partial y_k}\, dy\\
&-\int_{\widetilde{\Omega}} \frac{\partial}{\partial y_i}
\big\{ \wN_\varep (x,y)\big\} \cdot
a_{ij}(y/\varep) \big\{ \Psi_{\varep, k}(y) -P_k(y) \big\}
\cdot \frac{\partial^2 u_0}{\partial y_j\partial y_k}\, dy\\
&+\int_{\widetilde{\Omega}} \wN_{\varep} (x,y) \cdot a_{ij}(y/\varep)
\frac{\partial}{\partial y_j} \big\{ \Psi_{\varep, k}(y) -P_k (y)
 -\varep \chi_k(y/\varep)\big\}
\cdot \frac{\partial^2 u_0}{\partial y_i\partial y_k}\, dy
\endaligned
\end{equation}
and
\begin{equation}\label{7.3-4-3}
\aligned
w^{(2)}_\varep (x)= &
\varep \int_{\partial\widetilde{\Omega}}
\wN_\varep (x,y) \cdot n_i (y)  F_{jik} (y/\varep) \cdot 
\frac{\partial^2 u_0}{\partial y_j\partial y_k}\, d\sigma(y)\\
&
\qquad +\int_{\partial\widetilde{\Omega}}\wN_\varep (x,y) \cdot \left\{ \frac{\partial u_\varep}
{\partial \nu_\varep} -\frac{\partial u_0}{\partial\nu_0}\right\}\, d\sigma(y)
\endaligned
\end{equation}
(we have supressed all superscripts for notational simplicity).

To estimate $w_\varep^{(2)}$ in $T_1$,  observe that
$\mathcal{L}_\varep (w_\varep^{(2)})=0$ in $\widetilde{\Omega}$ and
$$
\frac{\partial} {\partial \nu_\varep} \big\{ w_\varep^{(2)}\big\}
=\varep n_i F_{jik} (x/\varep) \frac{\partial^2 u_0}{\partial x_j\partial x_k}
-\varep \average_{\partial\widetilde{\Omega}}
n_i F_{jik}(x/\varep) \frac{\partial^2 u_0}{\partial x_j\partial x_k}
\qquad \text{ on } I_{2},
$$
where we have used $\frac{\partial u_\varep}{\partial\nu_\varep}
=\frac{\partial u_0}{\partial\nu_0}$ on $I_{2}$.
Since 
$$
\|\frac{\partial} {\partial \nu_\varep} 
\big\{ w_\varep^{(2)}\big\}\|_{L^\infty(I_{2})}
\le C \varep \| \nabla^2 u_0\|_{L^\infty(T_{3})}
$$
 and
$$
\| \frac{\partial} {\partial \nu_\varep} \big\{ w_\varep^{(2)}\big\}
\|_{C^{0,\rho} (I_{2})}
\le C \varep^{1-\rho} \|\nabla^2 u_0\|_{L^\infty T_{2})}
+C \varep \|\nabla^2 u_0\|_{C^{0, \rho}(T_{2})},
$$
it follows from the boundary Lipschitz estimates in Section \ref{section-4.4}  that
$$
\|\nabla w_\varep^{(2)}\|_{L^\infty(T_1)}
\le C \varep^{1-\rho} \|\nabla^2 u_0\|_{L^\infty(T_2)}
+C \varep \|\nabla^2 u_0\|_{C^{0, \rho}(T_2)}
+C\int_{T_2} |w_\varep^{(2)}-c|
$$
for any constant $c$. This leads to
$$
\aligned
\|\nabla w_\varep \|_{L^\infty(T_1)}
&\le \|\nabla w_\varep^{(1)}\|_{L^\infty(T_1)}
+ \|\nabla w_\varep^{(2)}\|_{L^\infty(T_1)}\\
& \le C\int_{\Omega_2} |w_\varep|\, dx
+C \|\nabla w_\varep^{(1)} \|_{L^\infty(T_2)}
+C \varep^{1-\rho} \|\nabla^2 u_0\|_{L^\infty(T_3)}\\
&\qquad\qquad
+C\varep \|\nabla^2 u_0\|_{C^{0,\rho}(T_3)}.
\endaligned
$$
Since $|\Psi_{\varep, j}^\beta -P_j^\beta|\le C \varep \ln [\varep^{-1}M +2]$ by Lemma
\ref{N-C-E0},
we obtain
\begin{equation}\label{7.3-4-5}
\aligned
& \big\|\frac{\partial u_\varep^\alpha}{\partial x_i}
-\frac{\partial}{\partial x_i}
\big\{ \Psi_{\varep, j}^{\alpha\beta}\big\} 
\cdot \frac{\partial u_0^\beta}{\partial x_j}\big\|_{L^\infty(T_1)}\\
&\le C \int_{\Omega_2} |u_\varep -u_0|
+C \varep \ln [\varep^{-1}M+2] \| \nabla u_0\|_{L^\infty(T_3)}\\
& \qquad +C \varep \ln [\varep^{-1}M+2] \| \nabla^2 u_0\|_{L^\infty(T_3)}\\
& \qquad
+C \varep^{1-\rho} \|\nabla^2 u_0\|_{L^\infty(T_3)}
+C\varep \|\nabla^2 u_0\|_{C^{0,\rho}(T_3)}
+ C\|\nabla w_\varep^{(1)}\|_{L^\infty(T_2)}.
\endaligned
\end{equation}

It remains to estimate $\nabla w_\varep^{(1)}$ on $T_2$.
The first two integrals in the right hand side of (\ref{7.3-4-1}) may be handled
by applying Lemma \ref{lemma-7.3-3} on $\widetilde{\Omega}$. Indeed, let
$$
f(x)=-\varep F_{jik} (x/\varep)\cdot \frac{\partial^2 u_0}{\partial x_j\partial x_k}
-a_{ij}(x/\varep)
\big\{ \Psi_{\varep, k}(x) -P_k (x)\big\}
\frac{\partial^2 u_0}{\partial x_j\partial x_k}.
$$
Note that $\| f\|_{L^\infty(\widetilde{\Omega})}
\le C \varep \ln [\varep^{-1} r_0 +2] \|\nabla^2 u_0\|_{L^\infty(T_3)}$ and
$$
H_{\varep, \rho} (f)
\le
C \varep^{1-\rho} \ln [\varep^{-1} r_0+2] \|\nabla^2 u\|_{L^\infty(T_3)}
+C \varep \ln [\varep^{-1}r_0+2] \|\nabla^2 u_0\|_{C^{0, \rho}(T_3)}.
$$
It follows by Lemma \ref{lemma-7.3-3} that the first two integrals
in the right hand side of (\ref{7.3-4-1}) are bounded by
$$
C \varep \ln [\varep^{-1}r_0 +2]
\big\{ \varep^{-\rho} \| \nabla^2 u_0\|_{L^\infty(T_3)}
+\|\nabla^2 u_0\|_{C^{0, \rho}(T_3)}\big\}.
$$

Finally, the third integral in (\ref{7.3-4-1}) is bounded by
\begin{equation}\label{7.3-4-7}
C \|\nabla^2 u_0\|_{L^\infty(T_3)}
\int_{\widetilde{\Omega}} \frac{|\nabla_y \big\{ \Psi_{\varep, k} (y)-P_y (y) -\varep \chi_k (y/\varep)\big\}|}
{|x-y|^{d-1}}\, dy.
\end{equation}
Using the estimate 
$$
|\nabla_y \{ \Psi_{\varep, k} (y) -P_k (y) -\varep \chi_k (y/\varep)\}|
\le C \min \big(1, \varep [\text{dist}(y, \partial\Omega)]^{-1} \big),
$$
one may show that the integral in (\ref{7.3-4-7}) is bounded by
$C \varep \big[ \ln (\varep^{-1} +2)\big]^2$.
As a result,  we have proved that
$$
\|w_\varep^{(1)}\|_{L^\infty(T_2)}
\le C \varep^{1-\rho} \ln [\varep^{-1}M+2] \|\nabla^2 u_0\|_{L^\infty(T_3)}
+C \varep \ln [\varep^{-1} r_0 +2] \|\nabla^2 u_0\|_{C^{0, \rho}(T_3)}.
$$
This, together with (\ref{7.3-4-5}), yields the desired estimate.
\end{proof}

\begin{proof}[\bf Proof of Theorem \ref{theorem-7.3-2}]

Since $|\nabla_x N_\varep (x,y)|\le C |x-y|^{1-d}$ and $|\nabla\Psi_{\varep, j}^\beta|\le C$,
we may assume that $\varep<|x-y|$ and $\rho$ is small.
Fix $x_0, y_0\in \Omega$, $1\le \gamma \le d$ and let $r=|x_0-y_0|/8$.
Let $u^\alpha_\varep(x)= N^{\alpha\gamma}_\varep (x,y_0)$ and $u_0^{\alpha} (x)
=N_0^{\alpha\gamma} (x,y_0)$.
Observe that 
$$
\left\{
\aligned
\mathcal{L}_\varep (u_\varep) & =\mathcal{L}_0 (u_0) =0 \quad \text{ in } T(x_0, r),\\
\left( \frac{\partial u_\varep}{\partial\nu_\varep}\right)^\alpha
&=\left(\frac{\partial u_0}{\partial\nu_0}\right)^\alpha
=-|\partial\Omega|^{-1}\delta^{\alpha\gamma}\quad \text{ on } I(x_0,r).
\endaligned
\right.
$$ 
Also, note that $\|\nabla u_0\|_{L^\infty(T(x_0, r))} \le Cr^{1-d}$, 
$$
\|\nabla^2 u_0\|_{L^\infty(T(x_0,r))}\le C r^{-d}\quad
\text{ and }\quad
\| \nabla^2 u_0\|_{C^{0, \rho}(T(x_0,r))} \le C r^{-d-\rho}.
$$
Furthermore, it follows from Theorem \ref{theorem-7.3-1} that
$$
\|u_\varep -u_0\|_{L^\infty(T(x_0,r))} \le C \varep r^{1-d} \ln [\varep^{-1} r+2].
$$
Thus, by Lemma \ref{lemma-7.3-4}, we obtain
$$
\big\|\frac{\partial u_\varep^\alpha}{\partial x_i}
-\frac{\partial}{\partial x_i} \big\{ \Psi_{\varep, j}^{\alpha\beta}\big\} \cdot
\frac{\partial u_0^\beta}{\partial x_j}\big\|_{L^\infty (T(x_0, r/3))}
\le C \varep^{1-\rho} r^{\rho-d} \ln [\varep^{-1} r_0 +2].
$$
This completes the proof.
\end{proof}

\begin{remark}\label{remark-7.3-1}
{\rm
Using
$N_\varep^{*\alpha\beta} (x,y)=N_\varep^{\beta\alpha}(y,x)$.
Thus, it follows by Theorem \ref{theorem-7.3-2}  that for any $x,y\in \Omega$ and $x\neq y$,
\begin{equation}\label{estimate-remark-7.3-1}
\big| \frac{\partial }{\partial y_j} \left\{ N^{\alpha\beta}_\varep (x,y)\right\}
-\frac{\partial }{\partial y_j}
\big\{ \Psi_{\varep, \ell}^{*\beta \sigma} (y)\big\}
\cdot \frac{\partial}{\partial y_\ell}
\left\{ N_0^{\alpha \sigma}(x,y)\right\} \big|
\le \frac{C\varep^{1-\rho}\ln [\varep^{-1}r_0 +2]}
{|x-y|^{d-\rho}},
\end{equation}
where $\Psi_{\e, j}^{*\alpha\beta}$ denotes the Neumann corrector for $\mathcal{L}_\e^*$ in $\Omega$.
Fix $\beta$ and $j$.
Let 
$$
u^\alpha_\varep (x) =\frac{\partial }{\partial y_j} \left\{ N^{\alpha\beta}_\varep (x,y)\right\}
\text{ and }
u_0^\alpha (x)
=\frac{\partial }{\partial y_j}
\big\{ \Psi_{\varep, \ell}^{*\beta \sigma} (y)\big\}
\cdot \frac{\partial}{\partial y_\ell}
\left\{ N_0^{\alpha \sigma}(x,y)\right\}.
$$
Note that $\mathcal{L}_\varep (u_\varep)=\mathcal{L}_0 (u_0)=0
$ in $\Omega\setminus \{ y\}$
and $\frac{\partial u_\varep}{\partial\nu_\varep}
=\frac{\partial u_0}{\partial\nu_0}=0$ on $\partial\Omega$.
We may use Lemma \ref{lemma-7.3-4} and estimate (\ref{estimate-remark-7.3-1})
to deduce that if $\Omega$ is $C^{3, \eta}$ for some $\eta\in (0,1)$,
\begin{equation}\label{estimate-remark-7.3-1-1}
\aligned
\big|
\frac{\partial^2}{\partial x_i \partial y_j}
\big\{ N_\varep^{\alpha\beta} (x,y)\big\}
& -\frac{\partial}{\partial x_i}
\big\{ \Psi_{\varep, k}^{\alpha\gamma} (x) \big\}
\cdot
\frac{\partial^2 }{\partial x_k \partial y_\ell}
\big\{ N_0^{\gamma \sigma} (x,y) \big\}
\cdot
\frac{\partial}{\partial y_j}
\big\{ \Psi_{\varep, \ell}^{*\beta \sigma}(y) \big\}
\big|\\
&\le 
\frac{C_\rho \,\varep^{1-\rho}\ln [\varep^{-1}r_0+2]}{|x-y|^{d+1-\rho}}
\endaligned
\end{equation}
for any $x,y\in \Omega$ and $\rho\in (0,1)$, 
where $C_\rho$ depends only on $\mu$, $\lambda$, $\tau$, $\rho$, and $\Omega$.
}
\end{remark}



\section{Convergence rates in $L^p$ and $W^{1,p}$}\label{section-7.4}

In this section we establish the convergence rates of solutions $u_\varep$ in $L^p(\Omega)$ and
error estimates in $W^{1, p}$ for two-scale expansions, using
the asymptotic expansions of Green and Neumann functions obtained in Sections \ref{section-6.2} and \ref{section-7.3}.

We begin with the Dirichlet boundary condition.

\begin{thm}\label{theorem-7.4-0}
Suppose that $A$ satisfies the ellipticity condition (\ref{weak-e-1})-(\ref{weak-e-2}) and
is 1-periodic. If $m\ge 2$, we also assume that $A$ is H\"older continuous.
Let $\Omega$ be a bounded $C^{1,1}$ domain.
For $F\in L^2(\Omega;\br^m)$ and $\varep\ge 0$, let $u_\varep\in H^1_0(\Omega;\br^m)$ be the solution of
$\mathcal{L}_\varep (u_\varep)=F$ in $\Omega$. 
Then the estimate 
\begin{equation}\label{estimate-7.4-0}
\| u_\varep -u_0\|_{L^q(\Omega)} \le C_p \, \varep\, \| F\|_{L^p(\Omega)}
\end{equation}
holds if $1<p<d$ and $\frac{1}{q}=\frac{1}{p}-\frac{1}{d}$, or $p>d$ and $q=\infty$.
Moreover,
\begin{equation}\label{estimate-7.4-0-0}
\| u_\varep -u_0\|_{L^\infty(\Omega)}
\le C\, \varep \left[ \ln (\varep^{-1} r_0 +2)\right]^{1-\frac{1}{d}} \| F\|_{L^d(\Omega)},
\end{equation}
where $r_0=\text{\rm diam}(\Omega)$.
\end{thm}

\begin{proof}
This theorem is a corollary of Theorem \ref{theorem-7.2-1}. Indeed, 
by the Green function representation and estimate (\ref{estimate-7.2-1}),
$$
|u_\varep (x)-u_0(x)|\le C\, \varep \int_\Omega \frac{|F(y)|}{|x-y|^{d-1}}\, dy,
\quad \text{ for any } x\in \Omega.
$$
This leads to (\ref{estimate-7.4-0}) for $1<p<d$ and $\frac{1}{q}=\frac{1}{p}
-\frac{1}{d}$ by the well known estimates for fractional integrals.
The case of $p>d$ and $q=\infty$ follows readily from H\"older's inequality.
To see (\ref{estimate-7.4-0-0}) for $\ge 3$, we bound $|G_\varep (x,y)-G_0(x,y)|$ by $ C|x-y|^{2-d}$ 
when $|x-y|<\varep$, and by $C\varep |x-y|^{1-d}$ when $|x-y|\ge \varep$.
By H\"older's inequality, this gives
$$
\aligned
|u_\varep (x)-u_0(x)| &\le C\int_{\Omega\cap B(x,\varep)} \frac{|F(y)|}{|x-y|^{d-2}}\, dy
+ C\varep \int_{\Omega\setminus B(x,\varep)} \frac{|F(y)|}{|x-y|^{d-1}}\, dy\\
& \le C\varep \| F\|_{L^d(\Omega)}
+C\varep \big[\ln \left( \varep^{-1} r_0+2\right)\big]^{1-\frac{1}{d}} \| F\|_{L^d(\Omega)}\\
&\le C\varep \big[\ln \left( \varep^{-1}r_0 +2\right)\big]^{1-\frac{1}{d}}
 \| F\|_{L^d(\Omega)}.
\endaligned
$$
If $d=2$, we bound $|G_\e (x, y)-G_0(x, y)|$ by $C( 1+|\ln |x-y||)$ when $|x-y|<\e$. The rest is the same
as in the case $d\ge 3$.
\end{proof}

\begin{thm} \label{theorem-7.4-1}
Suppose that $A$ is 1-periodic and satisfies (\ref{weak-e-1})-(\ref{weak-e-2}) and (\ref{smoothness}).
 Let $\Omega$ be a bounded $C^{2, \eta}$ domain and $1<p<\infty$.
For $F\in L^p(\Omega;\br^m)$ and $\varep\ge 0$,
let $u_\varep \in W^{1,p}_0(\Omega;\br^m)$ be the weak solution of $\mathcal{L}_\varep (u_\varep)=F$ in $\Omega$.
Then
\begin{equation}\label{estimate-7.4-1}
\big\| u_\varep -u_0 - \big\{ \Phi_{\varep, j}^\beta -P_j^\beta\big\}\frac{\partial u_0^\beta}
{\partial x_j} \big\|_{W_0^{1,p}(\Omega)}\\
\le C_p\, \varep \, \big\{ \ln [\varep^{-1}r_0 +2]\big\}^{4|\frac12-\frac{1}{p}|}
 \| F\|_{L^p(\Omega)},
\end{equation}
where $r_0=\text{\rm diam}(\Omega)$ and $C_p$ depends only
on $p$, $\mu$, $\lambda$, $\tau$, and $\Omega$.
\end{thm}

\begin{proof}
The case $p=2$ is contained in Theorem \ref{theorem-7.1-1}. 
We will show that for any $1\le p\le \infty$,
\begin{equation}\label{7.4-1-1}
\big\|\frac{\partial u^\alpha_\varep}{\partial x_i}
-\frac{\partial }{\partial x_i} \big\{ \Phi_{\varep, j}^{\alpha\beta}\big\}
\cdot \frac{\partial u_0^\beta}{\partial x_j} \big\|_{L^p(\Omega)}
\le C \varep \left\{ \ln [ \varep^{-1} r_0 +2]\right\}^{4|\frac12-\frac{1}{p}|}
 \| F\|_{L^p(\Omega)}.
\end{equation}
This, together with $\|\Phi_{\varep, j}^\beta -P_j^\beta\|_\infty \le C \, \varep$ and the estimate 
$\|\nabla^2 u_0\|_{L^p(\Omega)}\le C \| F\|_{L^p(\Omega)}$ for $1<p<\infty$,
gives (\ref{estimate-7.4-1}). 

To see (\ref{7.4-1-1}), we use
Theorem \ref{theorem-7.2-2} for $|x-y|\ge \varep$ as well as
estimates on $|\nabla_x G_\varep(x,y)|$  and $|\nabla \Phi_\varep|$ to deduce that
$$
\big|\frac{\partial u^\alpha_\varep}{\partial x_i}
-\frac{\partial }{\partial x_i} \big\{ \Phi_{\varep, j}^{\alpha\beta}\big\}
\cdot \frac{\partial u_0^\beta}{\partial x_j}\big|
\le C \int_\Omega K_\varep (x,y) |f(y)|\, dy,
$$
where
\begin{equation}\label{definition-of-K}
K_\varep (x,y)= \left\{
\aligned
& \varep |x-y|^{-d} \ln \big[ \varep^{-1} |x-y|+2], \quad \text{ if } |x-y|\ge \varep,\\
&  |x-y|^{1-d}, \qquad\qquad\qquad\quad  \qquad \text{ if } |x-y|<\varep.
\endaligned
\right.
\end{equation}
It is not hard to show that
$$
\sup_{x\in \Omega} \int_\Omega K_\varep (x,y)\, dy 
+\sup_{y\in \Omega} \int_\Omega K_\varep (x,y)\, dx 
\le C\varep \big\{ \ln [\varep^{-1}r_0 +2]\big\}^2.
$$
This gives (\ref{7.4-1-1}) in the case $p=1$ or $\infty$.
Since the case $p=2$ is contained in Theorem \ref{theorem-7.1-1},
the proof is finished by using the M. Riesz interpolation theorem.
 \end{proof}

Next we turn to the Neumann boundary conditions.

\begin{thm}\label{theorem-7.4-3}
Suppose that $A$ is 1-periodic and satisfies (\ref{s-ellipticity}) and (\ref{smoothness}).
Let $\Omega$ a bounded $C^{1,1}$ domain and $1<p<\infty$.
For $\varep\ge 0$ and $F\in L^p(\Omega;\br^m)$ with $\int_\Omega F=0$, 
let $u_\varep\in W^{1,p}(\Omega;\br^m)$ be the solution to
the Neumann problem:
$\mathcal{L}_\varep (u_\varep)=F$ in $\Omega$, $\frac{\partial u_\varep}{\partial\nu_\varep}
=0$ on $\partial\Omega$, and $\int_{\partial \Omega} u_\varep =0$.
Then
\begin{equation}\label{estimate-7.4-3}
\| u_\varep -u_0\|_{L^q (\Omega)}
\le C \varep \ln [\varep^{-1} r_0 +2] \| F\|_{L^p(\Omega)}
\end{equation}
holds if $1<p<d$ and $\frac{1}{q}=\frac{1}{p}-\frac{1}{d}$, or $p>d$ and $q=\infty$,
where $r_0=\text{\rm diam} (\Omega)$.
Moreover,
\begin{equation}\label{estimate-theorem-7.4-3-1}
 \| u_\varep -u_0\|_{L^\infty (\Omega)}
\le C \varep \big[\ln (\varep^{-1} r_0 +2)\big]^{{2-\frac{1}{d}}}
 \| F\|_{L^d(\Omega)}.
\end{equation}
\end{thm}

\begin{proof}
This theorem is a corollary of Theorem \ref{theorem-7.3-1}.
Note that by the estimate (\ref{estimate-theorem-7.3-1}),
$$
\aligned
|u_\varep (x)-u_0(x)|
& \le \int_{\Omega} |N_\varep (x,y)-N_0 (x,y)|| F(y)|\, dy\\
& \le C \varep \ln (\varep^{-1} r_0+2) \int_\Omega
\frac{|F(y)| dy}{|x-y|^{d-1}}.
\endaligned
$$
The rest of the proof is the same as that of Theorem \ref{theorem-7.4-0}.
\end{proof}

\begin{thm}\label{theorem-7.4-4}
Suppose that $A$ is 1-periodic and satisfies (\ref{s-ellipticity}) and (\ref{smoothness}).
Let $\Omega$ be a bounded $C^{2, \eta}$ domain and $1<p<\infty$.
For $\varep\ge 0$ and $F\in L^p(\Omega;\br^m)$ with $\int_\Omega F=0$, let
$u_\varep\in W^{1,p}(\Omega;\br^m)$ be the solution of the Neumann problem:
$\mathcal{L}_\varep (u_\varep)=F$ in $\Omega$, $\frac{\partial u_\varep}
{\partial\nu_\varep}=0$ on $\partial\Omega$, and $\int_{\partial\Omega} u_\varep =0$.
Then
\begin{equation}\label{estimate-7.4-4}
\big \| u_\varep -u_0 - \big\{ \Psi^\beta_{\varep, j} -P^\beta_j\big\}
\frac{\partial u_0^\beta}{\partial x_j} \big\|_{W^{1,p}(\Omega)}
\le C_t\, \varep^t  \| F\|_{L^p(\Omega)}
\end{equation}
for any $t\in (0,1)$, where $\big(\Psi_{\varep, j}^\beta\big)$ is a matrix of Neumann correctors for
$\mathcal{L}_\varep$ in $\Omega$ such that $\Psi_{\varep, j}^\beta (x_0) =P_j^\beta (x_0)$ for some
$x_0\in \Omega$, and 
$C_t$ depends only on $\mu$, $\lambda$, $\tau$, $t$, $p$
and $\Omega$.
\end{thm}

\begin{proof}
Since $\|\Psi_{\varep, j}^\beta-P_k^\beta\|_{L^\infty(\Omega)}
\le C\varep \ln [\varep^{-1}r_0 +2]$ and  $\|u_0\|_{W^{2,p}(\Omega)}\le C \| F\|_{L^p(\Omega)}$, 
in view of Theorem \ref{theorem-7.4-3},
it suffices to prove that
\begin{equation}\label{7.4-4-1}
\big\|\frac{\partial u^\alpha_\varep}{\partial x_i}
-\frac{\partial }{\partial x_i} \big\{ \Psi_{\varep, j}^\beta\big\}
\cdot \frac{\partial u_0^\beta}{\partial x_j}\big \|_{L^p(\Omega)}
\le C_t \, \varep^t \| F\|_{L^p(\Omega)}.
\end{equation}
We will prove that the estimate (\ref{7.4-4-1}) holds
for any $1\le p\le \infty$.
To this end,
note that $u_\varep (x)=\int_\Omega N_\varep (x,y) F(y)\, dy$, 
by Theorem \ref{theorem-7.3-2},
\begin{equation}\label{7.4-4-3}
\big| 
\frac{\partial u^\alpha_\varep}{\partial x_i}
-\frac{\partial }{\partial x_i} \big\{ \Psi_{\varep, j}^\beta\big\}
\cdot \frac{\partial u_0^\beta}{\partial x_j}
\big|
\le C_\rho\, \varep^{1-\rho} \ln [\varep^{-1}r_0+2]
\int_\Omega \frac{|F(y)|\, dy}{|x-y|^{d-\rho}}
\end{equation}
for any $\rho\in (0,1)$.
Note that if $\rho<1-t$, then  $\varep^{1-\rho}\ln [\varep^{-1}r_0+2]
\le C \varep^{t}$. Estimate (\ref{7.4-4-1}) follows readily from (\ref{7.4-4-3}).
\end{proof}




\section{Notes}\label{section-7.7}

The material in Sections \ref{section-7.1} and \ref{section-7.6} is taken from  \cite{KLS-2013} by
C. Kenig, F. Lin, and Z. Shen.

Material in Sections \ref{section-6.2}, \ref{section-7.3} and \ref{section-7.4} is mostly taken from \cite{KLS-2014}
by C. Kenig, F. Lin, and Z. Shen.
Some modifications are made to cover the two dimensional case and to remove the symmetry condition for the Neumann problems.
Earlier work on  asymptotic expansions for Greens functions and Poisson kernels may be found in \cite{AL-1989-P}
by M. Avellaneda and F. Lin. 


%
%
%
%
%
%
%
%
%
%
%

\chapter{$L^2$ Estimates in Lipschitz Domains}\label{chapter-7}

In this chapter  we study the $L^2$ boundary value problems 
for $\mathcal{L}_\varep (u_\varep)=0$ in a bounded Lipschitz domain $\Omega$.
More precisely, we shalll be interested in uniform estimates
for the $L^2$ Dirichlet problem
\begin{equation}\label{L-2-DP}
\left\{
\begin{array}{ll}
 \mathcal{L}_\varep (u_\varep)  =0 &   \quad  \text{ in } \Omega,\\
 u_\varep  =f\in L^2(\partial\Omega) &  \quad  \text{ on } \partial\Omega,\\
 (u_\varep)^*  \in L^2(\partial\Omega),&
\end{array}
\right.
\end{equation}
the $L^2$ Neumann problem
\begin{equation}\label{L-2-NP}
\left\{
\begin{array}{ll}
 \mathcal{L}_\varep (u_\varep)  =0 \quad & \text{ in } \Omega,\\
 \frac{\partial u_\varep}{\partial \nu_\varep}  =g\in L^2(\partial\Omega)  \quad & \text{ on } \partial\Omega,\\
 (\nabla u_\varep)^*  \in L^2(\partial\Omega),
\end{array}
\right.
\end{equation}
as well as the $L^2$ regularity problem 
\begin{equation}\label{L-2-RP}
\left\{
\begin{array}{ll}
 \mathcal{L}_\varep (u_\varep)  =0    \quad & \text{ in } \Omega,\\
 u_\varep  =f\in H^1(\partial\Omega)   \quad  &\text{ on } \partial\Omega,\\
 (\nabla u_\varep)^* \in L^2(\partial\Omega),& 
\end{array}
\right.
\end{equation}
where $(u_\varep)^*$ denotes the nontangential maximal function of $u_\varep$, defined by 
(\ref{definition-of-nontangential-max-function}).
We will call the coefficient matrix $A\in \Lambda (\mu, \lambda, \tau)$ if 
 $A$ is 1-periodic and satisfies the Legendre ellipticity condition 
(\ref{s-ellipticity}) and the H\"older continuity condition (\ref{smoothness}).
 Under the assumptions that $A\in \Lambda (\mu, \lambda, \tau)$ and  $A^*=A$,
we will show that
 the solutions to (\ref{L-2-DP}), (\ref{L-2-NP}) and
(\ref{L-2-RP}) satisfy the estimates
\begin{equation}\label{L-2-estimate}
\aligned
\| (u_\varep)^*\|_{L^2(\partial\Omega)}
& \le C\, \| f\|_{L^2(\partial\Omega)},\\
\| (\nabla u_\varep)^*\|_{L^2(\partial\Omega)}
& \le C\,  \| g\|_{L^2(\partial\Omega)},\\
\| (\nabla u_\varep)^*\|_{L^2(\partial\Omega)}
& \le C\, \| f\|_{H^1(\partial\Omega)},
\endaligned
\end{equation}
respectively, 
where $C$ depends only on $\mu$, $\lambda$, $\tau$, and the Lipschitz character of $\Omega$.
 Since the Lipschitz character of $\Omega$ is scale-invariant,
 the constant $C$ in (\ref{L-2-estimate}) is independent of diam$(\Omega)$.
 Moreover,  in view of the rescaling property (\ref{rescaling}) for $\mathcal{L}_\e$,
these estimates are scale-invariant and it suffices to consider the case $\e=1$.

The estimates in (\ref{L-2-estimate})  are established by the method of layer potentials --the classical method of integral
equations. 
 In Section \ref{section-5.1} we introduce the nontangential convergence in Lipschitz domains and
formulate the $L^p$ boundary value problems for $\mathcal{L}_\e$.
Sections \ref{section-5.2}-\ref{section-5.4} are devoted to the study of mapping properties of singular
integral operators  associated with
the single and double layer potentials for $\mathcal{L}_\e$.
The basic insight is to approximate  the fundamental solution $\Gamma_1 (x, y)$ and its derivatives
 by freezing the coefficients when $|x-y|\le 1$.
For $|x-y|>1$, we use the asymptotic expansions of $\Gamma_1 (x, y)$ and its derivatives, established in Chapter \ref{chapter-2},
to bound the operators by the corresponding operators associated with $\mathcal{L}_0$.

The crucial  step in the use of layer potentials to solve $L^2$ boundary value problems in Lipschitz domains 
is to establish the following Rellich estimates,
\begin{equation}\label{R-00}
\|\nabla u_\e \|_{L^2(\partial\Omega)}
\le C \big\|\frac{\partial u_\e}{\partial \nu_\e} \big\|_{L^2(\partial\Omega)}
\quad \text{ and } \quad
\|\nabla u_\e \|_{L^2(\partial\Omega)}
\le C \big\|\nabla_{\tan} u_\e \big\|_{L^2(\partial\Omega)},
\end{equation}
for (suitable) solutions of $\mathcal{L}_\e (u_\e)=0$ in $\Omega$,
where $\nabla_{\tan} u_\e$ denotes the tangential derivatives of $u_\e$ on $\partial\Omega$.
In Section \ref{section-5.5} we give the proof of (\ref{R-00}) and solve the $L^2$ Dirichlet, Neumann, and regularity problems
in a Lipschitz domain in the case $\mathcal{L}=-\Delta$.
This is a classical result, due to B. Dahlberg, D. Jerison, C. Kenig, and C. Verchota 
\cite{D-1977, D-1979-P, Kenig-1981, Verchota-1984}, which will be needed 
to deal with the general case by a method of continuity.
In Section \ref{section-5.6} we show that for a general operator $\mathcal{L}_\e$,
the Rellich estimates in (\ref{R-00}) are more or less equivalent to the well-posedness of
(\ref{L-2-NP}) and (\ref{L-2-RP}). 

In Section \ref{section-5.7} we establish the estimates (\ref{R-00}) in the small scale where 
diam$(\Omega)\le \e$. This is a local result and the periodicity  of $A$ is not needed.
 If coefficients of $A$ are Lipschitz, the estimates follow by using Rellich identities.
The case where $A$ is only H\"older continuous is more involved and use a three-step approximate 
argument.  Rellich estimates (\ref{R-00}) for the large scale where diam$(\Omega)>\e$
is proved in Section \ref{section-5.8}.
To do this we first use the small scale estimates to reduce the problem
to the $L^2$ estimate of $\nabla u_\e$ on a boundary layer
$\{ x\in \Omega: \text{\rm dist}(x, \partial\Omega)< \e \}$.
The latter is then handled by using the $O(\sqrt{\e})$ error estimates in $H^1(\Omega)$
we proved in Chapter \ref{chapter-C} for a two-scale expansion.

The proof for (\ref{L-2-estimate}) is given in Section \ref{section-5.9}, assuming $\partial\Omega$ 
is connected and $d\ge 3$. The case of arbitrary Lipschitz domains in $\br^d$, $d\ge 2$,
 is treated in Section \ref{section-5.11}. 
 Finally, in Section \ref{section-5.10}, we prove the square function estimate as well as the $H^{1/2}$ estimate
 for solutions of the Dirichlet problem (\ref{L-2-DP}).



\section[Lipschitz domains]
{Lipschitz domains and nontangential convergence}\label{section-5.1}

Roughly speaking,
the class of Lipschitz domains is a class of domains which satisfy the uniform interior and exterior cone 
conditions. This makes it possible to extend the notion of nontangential convergence and
nontangential maximal functions from the upper half-space to a general Lipschitz domain.

Let $\psi:\br^{d-1}\to \br$ be a Lipschitz (continuous) function; i.e., there exists a nonnegative
constant $M$ such that
\begin{equation}\label{definition-of-psi-5.1}
|\psi(x^\prime)-\psi (y^\prime)|\le M |x^\prime-y^\prime|
\qquad \text{ for all } x^\prime, y^\prime\in \br^{d-1}.
\end{equation}
It is known that any Lipschitz function is almost everywhere (a.e.) differentiable in the ordinary 
sense that for a.e. $x^\prime\in \br^{d-1}$, there exists
a vector $\nabla\psi(x^\prime)\in \br^{d-1} $ such that
$$
\lim_{|y^\prime|\to 0}
\frac{|\psi(x^\prime+y^\prime)-\psi(x^\prime)-<\nabla\psi(x^\prime), y^\prime>|}{|y^\prime|}
=0.
$$
Furthermore, one has $\|\nabla\psi\|_\infty\le M$.
This is a classical result due to Denjoy, Rademacher and Stepanov.
Recall that a  bounded domain $\Omega$ in $\br^d$ is called a {\it Lipschitz domain}
if there exists $r_0>0$ such that for each point $z\in \partial\Omega$,
there is a new coordinate system of $\br^d$, obtained from the standard Euclidean
coordinate system through translation and rotation, so that
$z=(0,0)$ and
\begin{equation}\label{local-coordinate}
B(z,r_0)\cap \Omega
=B(z,r_0)\cap \big\{ (x^\prime, x_d)\in \br^d:\
x^\prime\in \br^{d-1} \text{ and } x_d>\psi(x^\prime) \big\},
\end{equation}
where $\psi: \br^{d-1}\to \br$ is a Lipschitz function and $\psi(0)=0$.

Let $\Omega$ be a Lipschitz domain.
Then $\partial\Omega$ possesses a tangent plane
and a unit outward normal $n$ at a.e.\, $z\in \partial\Omega$ with respect to the surface measure $d\sigma$
on $\partial\Omega$.
In the local coordinate system for which (\ref{local-coordinate}) holds, one has
$$
\aligned
d\sigma & =\sqrt{|\nabla\psi(x^\prime)|^2+1} \, dx_1\cdots dx_{d-1},\\
n& =\frac{(\nabla\psi(x^\prime), -1)}{\sqrt{|\nabla \psi(x^\prime)|^2+1}}.
\endaligned
$$
Note that the Lipschitz function $\psi$ in (\ref{local-coordinate}) may be taken to have
compact support.

Let $r, h>0$.
We call $Z$ a cylinder of radius $r$ and height $2h$ if $Z$ may be obtained from the set
$\{ (x^\prime, x_d)\in \br^d: \ |x^\prime|<r \text{ and } -h<x_d<h \}$
through translation and rotation.
We will use $tZ$ to denote the dilation of $Z$ about its center by a factor of $t$.
It follows that
 for each $z\in \partial\Omega$,
 we may find a Lipschitz function $\psi$ on $\br^{d-1}$,
 a cylinder $Z$ centered at $z$ with radius $r$ and height $2ar$, and a new coordinate system of
 $\br^d$ with $x_d$ axis containing the axis of $Z$, in which
 \begin{equation}\label{local-coordinate-cylinder}
 10Z\cap \Omega
 =10Z \cap \big\{ (x^\prime, x_d)\in \br^d:\ x^\prime\in \br^{d-1} \text{ and } x_d>\psi(x^\prime)\big\}
 \end{equation}
 and $z=(0, \psi(0))=(0,0)$,
 where $a=10 (\|\nabla\psi\|_\infty +1)$.
 We will call such cylinder $Z$ a {\it coordinate cylinder} and the pair $(Z,\psi)$ a {\it coordinate pair}.
 Since $\partial\Omega$ is compact, there exist $M>0$ and a finite number of
 coordinate pairs 
 $$
 \Big\{ (Z_i, \psi_i): i=1, \dots, N\Big\}
 $$
 with the same radius $r_0$ and $\|\nabla\psi_i\|_\infty\le M$, such that $\partial\Omega$ is covered by
 $Z_1, Z_2,\dots, Z_N$.
 Such $\Omega$ is said to belong to $\Xi(M,N)$
 \footnote{Let $\Omega$ be a bounded $C^1$ domain.
 Then for any $\delta>0$, there exists $N=N_\delta$ such that
 $\Omega\in \Xi (\delta, N)$.}.
 The constants $M$ and $N$, which are translation and dilation invariant, describe the {\it Lipschitz character} of
 $\Omega$.
 We shall say that a positive constant $C$ depends on the Lipschitz character of $\Omega$, if
 there exist $M$ and $N$ such that $\Omega\in \Xi(M,N)$
 and the constant can be made uniform for all Lipschitz domains in 
 $\Xi(M,N)$.
 For example, there is a constant $C$ depending only on $d$ and the Lipschitz character of $\Omega$
 such that $|\partial\Omega|\le C r_0^{d-1}$.
 By the isoperimetric inequality, we also have
 $|\Omega|\le Cr_0^d$, where $C$ depends only on $d$ and the Lipschitz character of $\Omega$.

Let $h, \beta>0$. We call $\varGamma$ a (two-component) cone of height $2h$ and aperture $\beta$
if $\varGamma$ may be obtained from the set 
$$
\Big\{ (x^\prime, x_d)\in \br^d:\ |x^\prime|<\beta x_d \text{ and } -h<x_d<h\Big\}
$$
through translation and rotation.
Let $(Z, \psi)$ be a coordinate pair and $r$ the radius of $Z$.
For each $z\in 8Z\cap \partial\Omega$, it is not had to see that the cone with vertex at $z$,
axis parallel to that of $Z$, height $r$ and aperture $\{ \|\nabla \psi\|_\infty +1\}^{-1}$
has the property that one component is in $\Omega$ and
the other in $\br^d\setminus \overline{\Omega}$.
By a simple geometric observation, the cone with vertex at $z$, 
axis parallel to that of $Z$, height $r/2$ and aperture $\{2\|\nabla\psi\|_\infty\}^{-1}$,
is contained in the set
\begin{equation}\label{deinition-of-gamma}
\gamma_\alpha (z)
=\Big\{ x\in \br^d\setminus \partial\Omega: \ 
|x-z|< \alpha\,  \text{dist}(x, \partial\Omega) \Big\},
\end{equation}
if $\alpha=\alpha(d, M)>1$
is sufficiently large.
By compactness it is possible to choose $\alpha>1$, depending only on $d$ and
the Lipschitz character of $\Omega$, so that
for any $z\in \partial\Omega$,
$\gamma_\alpha (z)$ contains a cone of some fixed height and aperture,
with vertex at $z$, one component in $\Omega$ and the other
in $\br^d\setminus \overline{\Omega}$.
Such $\{\gamma_\alpha (z): \, z\in \partial\Omega\}$ will be called a family of {\it nontangential 
approach regions} for $\Omega$.

\begin{definition}\label{definition-of-nontangential-max}
{\rm
Let $\Omega$ be a Lipschitz domain
and $\{ \gamma_\alpha (z): \, z\in \partial\Omega\}$
a family of nontangential approach regions for $\Omega$.
For $u\in C(\Omega)$, the {\it nontangential maximal function} of $u$ on $\partial\Omega$
is defined by
\begin{equation}\label{nontengential-max}
(u)^*_\alpha (z)
=\sup\big\{ |u(x)|: \ x\in \Omega \text{ and } x\in \gamma_\alpha (z) \big\}.
\end{equation}
We say that $u=f$ on $\partial\Omega$ in the sense of nontangential convergence,
written as $u=f$ n.t.\,on $\partial\Omega$, if
\begin{equation}\label{nontangential-converegence}
\lim_{\substack{x\to z\\ x\in \Omega\cap \gamma_\alpha (z)}} u (x)=f(z)
\qquad \text{ for a.e. } z\in \partial\Omega.
\end{equation}
}
\end{definition}

The nontangential maximal function $(u)^*_\alpha$ depends on the parameter $\alpha$.
However, the following proposition shows that the $L^p$ norms of $(u)^*_\alpha$
are equivalent for different $\alpha$'s.

\begin{prop}\label{max-equivalent-prop}
Let $\Omega$ be a Lipschitz domain and $1<\alpha<\beta$.
Then, for any $u\in C(\Omega)$ and any $t>0$,
\begin{equation}\label{distribution-equivalence}
|\Big\{ z\in \partial\Omega:\ (u)^*_\beta (P)> t\Big\}|
\le
C \, |\Big\{ z\in \partial\Omega:\ (u)^*_\alpha (P)> t\Big\}|,
\end{equation}
where $C$ depends only on $\alpha$, $\beta$, and the Lipschitz character of $\Omega$.
Consequently, 
$$
\| (u)^*_\alpha\|_{L^p(\partial\Omega)} \le \|(u)^*_\beta\|_{L^p(\partial\Omega)}
\le C\, \|(u)^*_\alpha \|_{L^p(\partial\Omega)}
$$
for any $0<p\le \infty$.
\end{prop}

\begin{proof}
Fix $t>0$ and let
$$
E=\Big\{ z\in \partial\Omega: \ (u)^*_\alpha (z)>t\Big\}.
$$
We will show that there exists a constant $c>0$, depending only on $\alpha$, $\beta$, and
the Lipschitz character of $\Omega$, such that
\begin{equation}\label{5.0.1-1}
\Big\{ z\in \partial\Omega: \
(u)^*_\beta (z)>t\Big\}
\subset
\Big\{ z\in \partial\Omega: \
\mathcal{M}_{\partial \Omega} (\chi_E) (z)\ge c \Big\},
\end{equation}
where $\mathcal{M}_{\partial\Omega}$ 
denotes the Hardy-Littlewood maximal operator on $\partial\Omega$.
Since $\mathcal{M}_{\partial\Omega}$ is of weak type $(1,1)$,
the estimate (\ref{distribution-equivalence}) follows readily from (\ref{5.0.1-1}).

To prove (\ref{5.0.1-1}), we suppose that
$z_0\in \partial\Omega$ and $(u)_\beta^* (z_0)>t$.
Then there exists $x\in \Omega$ such that
$|u(x)|>t$ and $|x-z_0|<\beta \, \text{dist}(x, \partial\Omega)$.
Choose $y_0\in \partial\Omega$ so that
$r=|x-y_0|=\text{dist}(x, \partial\Omega)$.
Observe that if $y\in \partial\Omega$ and $|y-y_0|
<(\alpha-1) r$, we have
$$
|x-y|\le |x-y_0|+|y-y_0|<\alpha r
=\alpha\, \text{dist}(x, \partial\Omega).
$$
This implies that $(u)^*_\alpha (y)\ge |u(x)|>t$. Thus,
$$
B(y_0, (\alpha-1)r)\cap \partial\Omega \subset E.
$$

Finally, we note that  $|z_0-y_0|\le |z_0-x|+|x-y_0|< (1+\beta) r$.
It follows that
$$
B(z_0, (\alpha+\beta)r)\cap E \supset  B(y_0, (\alpha-1)r)\cap E.
$$
Hence,
\begin{equation}\label{5.0.1-3}
\frac{|B(z_0, (\alpha+\beta)r)\cap E|}{|B(z_0, (\alpha+\beta)r)\cap\partial\Omega|}
\ge
\frac{|B(y_0, (\alpha-1) r)\cap \partial\Omega|}
{|B(z_0, (\alpha+\beta)r)\cap \partial\Omega|}
\ge c,
\end{equation}
where $c>0$ depends only on $\alpha$, $\beta$, and the Lipschitz character of $\Omega$.
The last inequality in (\ref{5.0.1-3})
follows from the fact that
$|B(z,r)\cap \partial\Omega|\approx r^{d-1}$
for any $z\in \partial\Omega$ and
$0<r<\text{diam}(\Omega)$.
From (\ref{5.0.1-3}) we conclude that
$\mathcal{M}_{\partial\Omega} (\chi_E) (z_0)\ge c$.
This finishes the proof of (\ref{5.0.1-1}).
\end{proof}

Because of Proposition \ref{max-equivalent-prop}, from now on,
we shall fix a family of nontangential  approach regions for $\Omega$,
and suppress the parameter $\alpha$ in the notation of nontangential maximal functions.

The next proposition will enable us to control $(u)^*$ by $(\nabla u)^*$.

\begin{prop}\label{control-prop}
Let $\Omega$ be a Lipschitz domain in $\br^d, d\ge 2$.
Suppose that $u\in C^1(\Omega)$ and $(\nabla u)^*\in L^p(\partial\Omega)$ for some
$p\ge 1$. Then $(\nabla u)^*\in L^q(\partial\Omega)$, where
$$
\left\{
\alignedat3
& 1<q<\frac{d-1}{d-2} &\quad & \text{ if } p=1;\\
&  q=\frac{p(d-1)}{d-1-p} & \quad & \text{ if } 1<p<d-1;\\
& 1<q<\infty & \quad & \text{ if } p=d-1;\\
& q=\infty &\quad & \text{ if } p>d-1.
\endalignedat
\right.
$$
\end{prop}

\begin{proof}
Let $K =\{ x\in \Omega: \, \text{dist}(x, \partial\Omega)\ge \delta\, r_0\}$, where
$\delta>0$ is sufficiently small.
We claim that for any $z\in \partial\Omega$,
\begin{equation}\label{5.0.4-1}
\aligned
(u)^* (z)  & \le \sup_{K} |u| +C_\delta \int_{\partial\Omega}
\frac{(\nabla u)^*(y)}{|z-y|^{d-2}}\, d\sigma (y) \quad \text{ if } d\ge 3,\\
(u)^* (z)  & \le \sup_{K} |u| +C_{\delta, \eta} \int_{\partial\Omega}
\frac{(\nabla u)^*(y)}{|z-y|^{\eta}}\, d\sigma (y) \quad \text{ if } d=2
\endaligned
\end{equation}
for any $\eta\in (0,1)$.
The desired estimate for $(u)^*$ follows readily from the estimates for the fractional
integral
$$
I_t (f)(z)=\int_{\partial\Omega} \frac{f(y)}{|z-y|^{d-1-t}}\, d\sigma (y)
$$
on $\partial\Omega$.
Indeed, it is known that $\|I_t (f)\|_{L^q(\partial\Omega)}
\le C \| f\|_{L^p(\partial\Omega)}$,
where $1<p<\frac{d-1}{t}$ and $\frac{1}{q}=\frac{1}{p}-\frac{t}{d-1}$.
If $f\in L^1(\partial\Omega)$, we have $I_t (f)\in L^{q, \infty}(\partial\Omega)
\subset L^{q_1}(\partial\Omega)$, where  $q=\frac{d-1}{d-1-t}$ and $q_1<q$.
We refer the reader to \cite[pp.118-121]{Stein-1970}
for a proof of the $(L^p, L^q)$ estimates in the case of $\br^d$. The results
 extend readily to $\partial\Omega$ by a simple localization
argument.

To prove (\ref{5.0.4-1}), we fix $z\in \partial\Omega$.
By translation and rotation we may assume that $P=0$ and
(\ref{local-coordinate}) holds.
Let $x=(x^\prime, x_d)\in \gamma_\alpha (z)$ and $x\notin K$.
Note that if $z=(x^\prime, s)$,
$$
\aligned
|\nabla u(z)| &\le \inf \Big\{ (\nabla u)^*(y):\ |z-y|<\alpha \, \text{dist}(z, \partial\Omega)\Big\}\\
&\le \frac{C}{s^{d-1}}
\int_{\substack{y\in \partial\Omega\\ |y|\le cs}}
(\nabla u)^* (y)\, d\sigma (y).
\endaligned
$$
Choose $a\in \br$ so that $(x^\prime, a)\in K$.
It follows that
\begin{equation}\label{5.0.4-2}
\aligned
|u(x)| & \le |u(x^\prime, a)| +\int_{x_d}^a |\nabla u(x^\prime, s)|\, ds\\
&\le \sup_K |u|
+ C \int_{x_d}^a \left\{ 
\int_{\substack{y\in \partial\Omega\\ |y|\le cs}}
(\nabla u)^* (y)\, d\sigma (y)\right\} \frac{ds}{s^{d-1}}\\
& \le \sup_K |u|
+C \int_{\partial\Omega} \frac{(\nabla u)^* (y)}{|y|^{d-2}}\, d\sigma (y),
\endaligned
\end{equation}
if $d\ge 3$. 
This gives the first estimate in (\ref{5.0.4-1}).
In the case $d=2$, an inspection of the argument above shows that
one needs to replace $\frac{1}{|y|^{d-2}}$ in (\ref{5.0.4-2})
 by $|\ln |y| |+1$, which is bounded
by $C_\eta |y|^{-\eta}$.
\end{proof}

\begin{remark}\label{nontangential-limit-remark}
{\rm
Let $u\in C^1(\Omega)$.
Suppose that $(\nabla u)^*\in L^1(\partial\Omega)$.
Then $(\nabla u)^*(z)$ is finite for a.e.\,$z\in \partial\Omega$.
By the mean value theorem and Cauchy criterion,
it follows that $u$ has nontangential limit a.e.\,on $\partial\Omega$.
}
\end{remark}

It is often necessary to approximate a given Lipschitz domain $\Omega$
by a sequence of $C^\infty$ domains $\{\Omega_j\}$
in such a manner that estimates on $\Omega_j$ with bounding constants
depending on the Lipschitz characters may be extended to $\Omega$
by a limiting argument.
The following theorem, whose proof may be found in \cite{Verchota-thesis} (also see \cite{Verchota-1984}),
 serves this purpose.

\begin{thm}\label{approximation-theorem}
Let $\Omega$ be a bounded Lipschitz domain in $\br^d$.
Then there exist constants $C$ and $c$, depending only on the Lipschitz
character of $\Omega$, and
 an increasing sequence of $C^\infty$ domains
$\Omega_j$, $j=1,2,\dots$ with the following properties.

\begin{enumerate}

\item

For all $j$, $\overline{\Omega_j}\subset \Omega$.

\item

There exists a sequence of homeomorphisms $\varLambda_j: \partial\Omega\to
\partial\Omega_j$ such that
$$
\sup_{z\in \partial\Omega} 
|z-\varLambda_j(z)| \to 0 \text{ as } j\to \infty
$$
and $\varLambda_j(z)\in \gamma_\alpha (P)$ for all $j$ and all $z\in \partial\Omega$, where $\{\gamma_\alpha (z)\}$
is a family of nontangential approach regions.

\item

The unit outward normal to $\partial\Omega_j$, $n(\varLambda(z))$, converges to $n(z)$
for a.e.\, $z\in \partial\Omega$.

\item

There are positive functions $\omega_j$ on $\partial\Omega$ such that
$0<c\le \omega_j\le C$ uniformly in $j$,
$\omega_j \to 1$ a.e. as $j\to\infty$, and
$$
\int_E \omega_j \, d\sigma =\int_{\varLambda(E)} d\sigma_j
\quad \text{ for any measurable set } E\subset \partial\Omega.
$$

\item

There exists a smooth vector field $h\in C_0^\infty(\br^d, \br^d)$ such that
$$
\langle h(z), n(z)\rangle \ge c>0 \quad \text{ for all } z\in \partial\Omega_j \text{ and all } j.
$$

\item

There exists a finite covering of $\partial\Omega$ by coordinate cylinders so that for each
coordinate cylinder $(Z, \psi)$ in the covering, $10Z\cap \partial\Omega_j$ is given by the graph of
a $C^\infty$ function $\psi_j$,
$$
10Z\cap \partial\Omega_j
=10Z \cap \big\{ (x^\prime, \psi_j(x^\prime)):\ x^\prime\in \br^{d-1}\big\}.
$$
Furthermore, one has $\psi_j \to \psi$ uniformly,
$\nabla\psi_j \to \nabla\psi$ a.e. as $j\to\infty$, and
$\|\nabla\psi_j\|\le \|\nabla\psi\|_\infty$.

\end{enumerate}

\end{thm}

We will use $\Omega_j\uparrow \Omega$ to denote an approximation
sequence with properties 1-6.
We may also approximate $\Omega$ by a decreasing sequence $\{\Omega_j\}$
from outside so that $\overline{\Omega}\subset \Omega_j$
and properties 2-6 hold.
Such sequence will be denoted by $\Omega_j\downarrow\Omega$.

Observe that if $(u)^*\in L^1(\partial\Omega)$ and $u=f$ n.t. on $\partial\Omega$, then
$$
\int_{\partial\Omega} u(\varLambda_j(y))\, d\sigma (y) \to \int_{\partial\Omega} f\, d\sigma
$$
by the Lebesgue's dominated convergence theorem.
This, in particular, allows us to
extend the divergence theorem
to Lipschitz domains for functions with nontangential limits.

\begin{thm}\label{divergence-theorem}
Let $\Omega$ be a bounded Lipschitz domain in $\br^d$
and $u\in C^1(\Omega; \br^d)$.
Suppose that $\text{\rm div} (u)\in L^1(\Omega)$
and $u$ has nontangential limit a.e.\,on $\partial\Omega$.
Also assume that $(u)^*\in L^1(\partial\Omega)$.
Then
\begin{equation}\label{divergence-formula}
\int_\Omega \text{\rm div} (u)\, dx
=\int_{\partial\Omega}
\langle u, n\rangle \, d\sigma.
\end{equation}
\end{thm}

\begin{proof}
Let $\Omega$ be a sequence of smooth domains such that $\Omega_j\uparrow\Omega$.
By the divergence theorem,
\begin{equation}\label{divergence-smooth}
\int_{\Omega_j} \text{\rm div} (u)\, dx
=\int_{\partial\Omega_j}
\langle u, n\rangle \, d\sigma.
\end{equation}
Since div$(u)\in L^1(\Omega)$, the integral of div$(u)$ on $\Omega_j$
converges to its integral on $\Omega$.
By a change of variables, one may write the right hand side of (\ref{divergence-smooth})
as
$$
\int_{\partial\Omega} \langle u(\varLambda_j(y)), n(\varLambda_j(y))\rangle  \omega_j(y)\, d\sigma (y),
$$
where $\varLambda_j$ and $\omega_j$ are given by Theorem \ref{approximation-theorem}.
Note that
$$
\langle u(\varLambda_j(y)), n(\varLambda_j(y))\rangle \omega_j(y)
\to \langle u(y), n(y)\rangle \quad \text{ for a.e. } y\in \partial\Omega
$$
and
$$
|\langle u(\varLambda_j(y)), n(\varLambda_j(y))\rangle \omega_j(y)|\le C\, (u)^* (y).
$$
Thus it follows by Lebesgue's dominated convergence theorem that
the right hand side of (\ref{divergence-smooth}) converges to the integral in the right
hand side of (\ref{divergence-formula}).
This completes the proof.
\end{proof}

\begin{remark}\label{divergence-theorem-remark}
{\rm
Theorem \ref{divergence-theorem} also holds for the unbounded domain
$\Omega_-=\br^d\setminus \overline{\Omega}$, under the additional
assumption that $u(x)=o(|x|^{1-d})$ as $|x|\to \infty$.
Since $n$ points away from $\Omega$,
in the case $\Omega_-$, (\ref{divergence-formula}) is replaced by
\begin{equation}\label{divergence-formula-1}
\int_{\Omega_-}
\text{\rm div} (u)\, dx
=-\int_{\partial\Omega} \langle u, n\rangle \, d\sigma.
\end{equation}
}
\end{remark}

We now formulate the $L^p$ Dirichlet and Neumann problems
for $\mathcal{L}_\varep (u_\varep)=0$ in a Lipschitz domain $\Omega$.

\begin{definition}[$L^p$ Dirichlet problem]\label{definition-of-L-p-Dirichlet-problem}
{\rm
Let $\Omega$ be a bounded Lipschitz domain.
For $1<p< \infty$, the $L^p$ Dirichlet problem for
$\mathcal{L}_\varep (u_\varep)=0$ in $\Omega$ is said to be well-posed,
 if for any $f\in L^p(\partial\Omega;\br^m)$, there exists a unique  function
$u_\varep$ in $W^{1,2}_{\loc}(\Omega;\br^m)$ such that $\mathcal{L}_\varep (u_\varep)=0$ in $\Omega$,
$(u_\varep)^*\in L^p(\partial\Omega)$, and 
$u_\varep=f$ n.t. on $\partial\Omega$.
Moreover, the solution satisfies 
$$
\|(u_\varep )^*\|_{L^p(\partial\Omega)}
\le C\, \| f\|_{L^p(\partial\Omega)}.
$$
}
\end{definition}

\begin{definition}[$L^p$ Neumann problem]\label{definition-of-L-p-Neumann-problem}
{\rm
Let $\Omega$ be a bounded Lipschitz domain.
For $1<p<\infty$, the $L^p$ Neumann problem for $\mathcal{L}_\varep (u_\varep)=0$
in $\Omega$ is said to be well-posed, if for any $g\in L^p(\partial\Omega;\br^m)$
with $\int_{\partial \Omega} g \, d\sigma=0$,
there exists a function $u_\varep$ in $W^{1,2}_\loc (\Omega;\br^m)$, unique up to constants,
such that $\mathcal{L}_\varep (u_\varep)=0$ in $\Omega$,
$(\nabla u_\varep)^*\in L^p(\partial\Omega)$, and
$\frac{\partial u_\varep}{\partial \nu_\varep}
=g$ on $\partial\Omega$ in the sense of nontangential convergence:
\begin{equation}\label{nontangential-conormal}
\lim_{\substack{x\to z\\
x\in \Omega\cap\gamma(z)}}
n_i(z)a_{ij}^{\alpha\beta} (z/\varep) \frac{\partial u_\varep^\beta}{\partial x_j }(x)
=g^\alpha (z)
\qquad \text{ for a.e. } z\in \partial\Omega.
\end{equation}
Moreover, the solutions satisfy 
$$
\|(\nabla u_\varep)^*\|_{L^p(\partial\Omega)}
\le C\, \| g\|_{L^p(\partial\Omega)}.
$$
}
\end{definition}

To formulate the $L^p$ regularity problem, we first give the definition of
$W^{1,p}(\partial\Omega)$.

\begin{definition}\label{definition-of-W-1-p}
{\rm
For $1\le p< \infty$, we say $f\in W^{1,p}(\partial\Omega)$ if $f\in L^p(\partial\Omega)$ and there exist
functions $g_{jk}\in L^p(\partial\Omega)$ so that for all $\varphi\in C_0^\infty(\br^d)$
and $1\le j,k\le d$,
$$
\int_{\partial\Omega} f \left(n_j\frac{\partial}{\partial x_k}
-n_k \frac{\partial}{\partial x_j} \right)\varphi\, d\sigma
=-\int_{\partial\Omega} g_{jk} \varphi\, d\sigma.
$$
}
\end{definition}

By a partition of unity, one may show that $f\in W^{1,p}(\partial\Omega)$
if and only if $f\eta\in W^{1,p}(\partial\Omega)$ for any $\eta\in C_0^\infty(B(z, r_0))$
with $z\in \partial\Omega$.
Thus, in a local coordinator system where (\ref{local-coordinate}) holds,
$f\in W^{1,p}(\partial\Omega)$ means that $f(x^\prime, \psi(x^\prime))$, as a function
of $x^\prime$, is a $W^{1,p}$ function on the set 
$\{ x^\prime\in \br^{d-1}: \, |x^\prime|<c\, r_0\}$.
The space $W^{1,p}(\partial\Omega)$ is a Banach space with the (scale-invariant) norm
\begin{equation}\label{definition-of W-1-p-norm}
\|f\|_{W^{1,p}(\partial\Omega)}
=|\partial\Omega|^{\frac{1}{1-d}}\| f\|_{L^p(\partial\Omega)}
+\sum_{1\le j,k\le d}
\|g_{jk}\|_{L^p(\partial\Omega)}.
\end{equation}
If $1\le p<\infty$, the set $\left\{ f|_{\partial\Omega}: \ f\in C_0^\infty(\br^d)\right\}$
is dense in $W^{1,p}(\partial\Omega)$.
Note that $g_{jk}=-g_{kj}$ and $g_{jj}=0$.
If $f\in C_0^1(\br^d)$ and $j\neq k$, then
\begin{equation}\label{g-j-k}
g_{jk}=n_j\frac{\partial f}{\partial x_k}-n_k \frac{\partial f}{\partial x_j}
\end{equation}
and
$$
\sum_{1\le j<k\le d}
|g_{jk}|^2
=|\nabla f|^2 -\big|\frac{\partial f}{\partial n}\big|^2 =|\nabla_{\rm tan} f|^2,
$$
where $\nabla_{\tan} f=\nabla f-\langle n,\nabla f\rangle  n$.
One may also deduce from (\ref{g-j-k}) that
\begin{equation}\label{compatibility-W-1-p}
n_\ell g_{jk}=n_k g_{jk}-n_j g_{kl}
\qquad \text{ for } 1\le j,k,\ell\le d.
\end{equation}
By a simple density argument, the compatibility condition (\ref{compatibility-W-1-p})
holds for any $f\in W^{1,p}(\partial\Omega)$.

If the Dirichlet data $f$ is taken from $W^{1,p}(\partial\Omega)$ instead of $L^p(\partial\Omega)$,
one should expect the solution to have one order higher regularity.
This is the so-called regularity problem (for the Dirichlet problem).

\begin{definition}[$L^p$ regularity problem]\label{definition-of-L-p-regularity-problem}
{\rm
Let $\Omega$ be a bounded Lipschitz domain.
For $1<p<\infty$, the $L^p$ regularity problem for $\mathcal{L}_\varep (u_\varep)=0$
in $\Omega$ is said to be well-posed,
if for any $f\in W^{1,p}(\partial\Omega;\br^m)$, there exists a unique function $u_\varep\in W^{1,2}_{loc}(\Omega;\br^m)$
such that $\mathcal{L}_\varep (u_\varep)=0$ in $\Omega$,
$(\nabla u_\varep)^*\in L^p(\partial\Omega)$, and
$u_\varep =f$ n.t.\,on $\partial\Omega$.
Moreover, the solution satisfies 
$$
\|(\nabla u_\varep)^*\|_{L^p(\partial\Omega)}
\le C \, \| f\|_{W^{1,p}(\partial\Omega)}.
$$
}
\end{definition}

The rest of this chapter is devoted to the study
of the $L^2$ Dirichlet, Neumann, and regularity problems
for $\mathcal{L}_\varep(u_\varep)=0$ in a bounded Lipschitz
domain $\Omega$ under the assumption that
$A$ is elliptic, periodic, symmetric, and H\"older continuous.
 Our goal is to establish the uniform estimates in (\ref{L-2-estimate}) with constants $C$ depending at most on
$\mu$, $\lambda$, $\tau$, and the Lipschitz character of $\Omega$.

%
%
%
%
%
%
%

\section{Estimates of fundamental solutions}\label{section-5.2}

Throughout Sections \ref{section-5.2}-\ref{section-5.9}, with the exception of Section \ref{section-5.5}, we assume that $d\ge 3$.
Let $A\in \Lambda(\mu, \lambda, \tau)$ and
$$
\Gamma(x,y)=\Gamma (x,y;A)
=\big(\Gamma^{\alpha\beta}(x,y;A) \big)_{m\times m}
$$
denote the matrix of fundamental solutions for the operator
$\mathcal{L}=-\text{div}\big(A \nabla \big)$
in $\br^d$, with pole at $y$.
Then
\begin{equation}\label{size-estimate-5}
\left\{
\aligned
|\Gamma (x,y)| & \le C |x-y|^{2-d},\\
|\nabla_x \Gamma (x,y)|+|\nabla_y \Gamma (x,y)| &\le C |x-y|^{1-d},\\
|\nabla_x\nabla_y \Gamma(x,y)| & \le C |x-y|^{-d}
\endaligned
\right.
\end{equation}
for any $x,y\in \br^d$ and $x\neq y$,
where $C$ depends only on $\mu$, $\lambda$ and $\tau$
(see Section \ref{section-2.5};
note that $\Gamma_1(x,y)=\Gamma(x,y; A)$
and $\Gamma_0(x,y)=\Gamma(x,y;\widehat{A})$).
It follows from (\ref{fundamental-solution-asym-2}) and (\ref{fundamental-solution-asym-3}) that
\begin{equation}\label{asymp-5-1}
\aligned
\Big|
\frac{\partial}{\partial x_i}
\big\{ \Gamma^{\alpha\beta} (x,y;A)\big\}
 & -\left\{ \delta_{ij}\delta^{\alpha\gamma}
+\frac{\partial}{\partial x_i} \left\{ \chi_j^{\alpha\gamma} (x)\right\}\right\}
\frac{\partial}{\partial x_j} \left\{ \Gamma^{\gamma\beta}
(x,y;\widehat{A}) \right\}\Big|\\
& \le C |x-y|^{-d} \ln \big[|x-y|+2\big]
\endaligned
\end{equation}
and 
\begin{equation}\label{asymp-5-2}
\aligned
\big|
\frac{\partial}{\partial y_i}
\big\{ \Gamma^{\alpha\beta} (x,y;A)\big\}
 & -\left\{ \delta_{ij}\delta^{\beta\gamma}
+\frac{\partial}{\partial y_i} \left\{ \chi_j^{*\beta\gamma} (y)\right\}\right\}
\frac{\partial}{\partial y_j} \left\{ \Gamma^{\alpha\gamma}
(x,y;\widehat{A}) \right\}\big|\\
& \le C |x-y|^{-d} \ln \big[|x-y|+2\big]
\endaligned
\end{equation}
for $x,y\in \br^d$ with $|x-y|\ge 1$,
where $\widehat{A}$ is the (constant)  matrix of homogenized coefficients.
Let $I$ denote the identity matrix.
For brevity estimates (\ref{asymp-5-1}) and (\ref{asymp-5-2}) may be written as
\begin{equation}\label{asymp-5-3}
\big| \nabla_x \Gamma (x,y; A)
-(I +\nabla \chi (x)) \nabla_x \Gamma (x,y; \widehat{A})\big|
\le C |x-y|^{-d} \ln \big[|x-y|+2\big]
\end{equation}
and
\begin{equation}\label{asymp-5-4}
\big| \nabla_y \big(\Gamma (x,y; A)\big)^*
-\big(I +\nabla \chi^*(y)\big) \nabla_y \big(\Gamma (x,y; \widehat{A})\big)^*\big|
\le C |x-y|^{-d} \ln \big[|x-y|+2\big],
\end{equation}
where $(\Gamma(x,y;A))^*$ denotes the matrix adjoint of $\Gamma(x,y;A)$.
These two inequalities give us the asymptotic behavior of $\nabla_x\Gamma(x,y;A)$
and $\nabla_y \Gamma(x,y;A)$ when $|x-y|$ is large.

For a function $F=F(x,y,z)$, we will use the notation
$$
\nabla_1 F(x,y,z)=\nabla_x F(x,y,z)
\quad \text{ and } \quad
\nabla_2 F(x,y,z)=\nabla_y F(x,y,z).
$$
The following lemma describes the local behavior of $\Gamma(x,y;A)$.
We emphasize that $\Gamma(x,y;A(x))$ denotes the matrix of fundamental solutions
for the operator $-\text{div}(E\nabla)$, where $E$ is the constant matrix given by
$A(x)$.

\begin{lemma}\label{fundamental-local-lemma}
Let $A\in \Lambda (\mu, \lambda, \tau)$.
Then 
\begin{equation}\label{fundamental-local-estimate}
\aligned
|\Gamma(x,y;A)-\Gamma (x,y;A(x))| & \le C |x-y|^{2-d+\lambda},\\
|\nabla_1 \Gamma(x,y;A)-\nabla_1 \Gamma (x,y;A(x))| & \le C |x-y|^{1-d+\lambda},\\
|\nabla_1 \Gamma(x,y;A)-\nabla_1 \Gamma (x,y;A(y))| & \le C |x-y|^{1-d+\lambda}
\endaligned
\end{equation}
for any $x,y\in \br^d$ and $x\neq y$,
 where
$C$ depends only on $\mu$, $\lambda$, and $\tau$.
\end{lemma}

\begin{proof}
Let $B=(b_{ij}^{\alpha\beta})\in \Lambda(\mu, \lambda, \tau)$. Then for any $x,y\in \br^d$,
\begin{equation}\label{Fundamental-difference}
\aligned
\Gamma^{\alpha\sigma}(x,y;A)
&-\Gamma^{\alpha\sigma} (x,y;B)\\
&
=\int_{\br^d}
\frac{\partial}{\partial z_i}
\Big\{ \Gamma^{\alpha\beta}(x,z;A)\Big\}
\Big\{ a_{ij}^{\beta\gamma} (z)-b_{ij}^{\beta\gamma} (z)\Big\}
\frac{\partial}{\partial z_j}
\Big\{ \Gamma^{\gamma\sigma}(z,y;B)\Big\}\, dz.
\endaligned
\end{equation}
In view of  (\ref{size-estimate-5}) we obtain
\begin{equation}\label{5.1.1-1}
|\Gamma(x,y; A)-\Gamma(x,y;B)|
\le C \int_{\br^d}
\frac{|A(z)-B(z)|}{|z-x|^{d-1} |z-y|^{d-1}}\, dz
\end{equation}
and
\begin{equation}\label{5.1.1-2}
|\nabla_1\Gamma(x,y; A)-\nabla_1 \Gamma(x,y;B)|
\le C \int_{\br^d}
\frac{|A(z)-B(z)|}{|z-x|^{d} |z-y|^{d-1}}\, dz.
\end{equation}

To show the first inequality in (\ref{fundamental-local-estimate}), we fix $x\in \br^d$ and
let $B=A(x)$. Since 
$$
|A(z)-A(x)|\le \tau |z-x|^\lambda,
$$
 it follows from (\ref{5.1.1-1})
that
$$
\aligned
|\Gamma(x,y;A)-\Gamma(x,y;A(x)|
& \le C \int_{\br^d} 
\frac{dz}{|z-x|^{d-1-\lambda}|z-y|^{d-1}}\\
&\le C |x-y|^{2-d+\lambda}.
\endaligned
$$
The second inequality in (\ref{fundamental-local-estimate}) follows from (\ref{5.1.1-2})
in the same manner.

To prove the third inequality in (\ref{fundamental-local-estimate}), note that
if $E_1, E_2$ are two constant matrices satisfying the ellipticity condition (\ref{ellipticity}), then
\begin{equation}\label{5.1.1-3}
|\nabla_1^N\Gamma (x,y;E_1)-\nabla_1^N \Gamma(x,y;E_2)|\le C |E_1-E_2| |x-y|^{2-d-N},
\end{equation}
where $C$ depends only on $\mu$ and $N$. 
This follows from (\ref{5.1.1-1}).
By taking $E_1=A(x)$ and $E_2=A(y)$, we obtain 
\begin{equation}\label{5.1.1-5}
|\nabla_1 \Gamma(x,y; A(x))-\nabla_1 \Gamma(x,y;A(y))|\le C |x-y|^{1-d+\lambda}.
\end{equation}
The third inequality in (\ref{fundamental-local-estimate}) follows from the second and (\ref{5.1.1-5}).
\end{proof}

\begin{remark}\label{remark-5.1-1}
{\rm
It follows from (\ref{5.1.1-1}) that if $A, B\in \Lambda(\mu, \lambda, \tau)$,
\begin{equation}\label{5.1.1-7}
|\Gamma(x,y; A)-\Gamma(x,y; B)|\le C \|A-B\|_\infty |x-y|^{2-d}.
\end{equation}
Also, if we fix $y\in \br^d$ and let $B=A(y)$,
the same argument as in the proof of Lemma \ref{fundamental-local-lemma} yields that
\begin{equation}\label{fundamental-local-estimate-1}
\aligned
|\Gamma(x,y;A)-\Gamma (x,y;A(y))| & \le C |x-y|^{2-d+\lambda},\\
|\nabla_2 \Gamma(x,y;A)-\nabla_2 \Gamma (x,y;A(y))| & \le C |x-y|^{1-d+\lambda},\\
|\nabla_2 \Gamma(x,y;A)-\nabla_2 \Gamma (x,y;A(x))| & \le C |x-y|^{1-d+\lambda}
\endaligned
\end{equation}
for any $x,y\in \br^d$ and $x\neq y$.
}
\end{remark}

In the rest of this section we will be concerned with
the estimate of $\Gamma(x,y; A)-\Gamma(x,y;B)$
when $A$ is close to $B$ in the H\"older space $C^{\lambda}(\br^d)$.
The results are local and will be used in an approximation 
argument for domains $\Omega$ with diam$(\Omega)\le 1$
in Section \ref{section-5.3}.

\begin{lemma}\label{lemma-5.1}
Let $R\ge 1$ and $A, B\in \Lambda(\mu, \lambda, \tau)$.
Then
\begin{equation}\label{estimate-5.1-1}
\aligned
|\nabla_x \Gamma(x,y;A)-\nabla_x \Gamma(x,y;B)| &\le C_R \|A-B\|_{C^\lambda(\br^d)} |x-y|^{1-d},\\
|\nabla_y \nabla_x \Gamma(x,y;A)-\nabla_y \nabla_x \Gamma(x,y;B)| & \le C_R \|A-B\|_{C^\lambda(\br^d)} |x-y|^{-d}
\endaligned
\end{equation}
for any $x,y\in \br^d$ with $0<|x-y|\le R$,
where $C_R$ depends on $\mu$, $\lambda$, $\tau$, and $R$.
\end{lemma}

\begin{proof}
Estimate (\ref{estimate-5.1-1}) follows from (\ref{5.1.1-7})
by (local) interior $C^{1, \lambda}$  estimates.
Indeed, fix $x_0,y_0\in \br^{d}$ with $r=|x_0-y_0|\le R$,
and consider
$u(x)=\Gamma(x,y_0; A)-\Gamma(x,y_0;B)$ in $\Omega=B(x_0, r/2)$.
Let $w(x)=-\Gamma(x,y_0;B)$.
Then
$$
\text{div}(A\nabla u)=\text{div}(A\nabla w)
=\text{div}\big( (A-B)\nabla w\big) \quad \text{ in } \Omega.
$$
It follows that
$$
\aligned
|\nabla u(x_0)|
& \le Cr^{-1} \| u\|_{L^\infty(\Omega)}
+ Cr^\lambda \|(A-B)\nabla w\|_{C^{0, \lambda} (\Omega)}\\
& 
\le C r^{1-d} \|A-B\|_{C^\lambda (\br^d)},
\endaligned
$$
where $C$ may depend on $R$.
This gives the first inequality in (\ref{estimate-5.1-1}).
The second inequality follows in the same manner by considering 
$v(x)=\nabla_y \Gamma(x,y; A)-\nabla_y \Gamma(x,y;B)$.
\end{proof}

For $A\in \Lambda(\mu, \lambda, \tau)$, define
\begin{equation}\label{definition-of-Pi-5}
\Pi (x,y;A)=\nabla_1 \Gamma (x,y;A)-\nabla_1\Gamma (x,y;A(x)).
\end{equation}

\begin{thm}\label{theorem-5.1}
Let $A, B\in \Lambda(\mu, \lambda, \tau)$ and $R\ge 1$.
Then
\begin{equation}\label{estimate-of-pi-5}
|\Pi(x,y;A)-\Pi(x,y;B)|
\le C_R\,  \|A-B\|_{C^{\lambda}(\br^d)} |x-y|^{1-d+\lambda}
\end{equation}
for any $x,y\in \br^d$ with $|x-y|<R$,
where $C_R$ depends on $\mu$, $\lambda$, $\tau$, and $R$.
\end{thm}

\begin{proof}
Let $\Omega=B(x_0,2R)$.
It follows from integration by parts that for $x,y\in B(x_0,R)$,
\begin{equation}\label{5.1.3-1}
\aligned
 \Gamma^{\alpha\delta} & (x,y;A)-\Gamma^{\alpha\delta}(x,y; A(P))\\
&=\int_\Omega
\frac{\partial}{\partial z_i}
\Big\{ \Gamma^{\alpha\beta} (x,z;A)\big\}
\Big\{ a_{ij}^{\beta\gamma} (P)-a_{ij}^{\beta\gamma}(z) \Big\}
\frac{\partial}{\partial z_j} \Big\{ \Gamma^{\delta\gamma}(z,y; A(P))\Big\}\, dz\\
&\quad
+\int_{\partial\Omega}
\frac{\partial}{\partial z_i} \Big\{ \Gamma^{\alpha\beta}(x,z;A)\Big\}
a_{ij}^{\beta\gamma} (z) n_j (z) \Gamma^{\gamma\delta}(z,y;A(P))\, d\sigma(z)\\
&\quad
-\int_{\partial\Omega}
\Gamma^{\alpha\beta} (x,z;A) n_i(z) a_{ij}^{\beta \gamma} (P)
\frac{\partial}{\partial z_j}
\Big\{ \Gamma^{\gamma\delta} (z,y;A(P))\Big\}\, d\sigma(z),
\endaligned
\end{equation}
where $P\in B(x_0,R)$ and
$n=(n_1,\dots, n_d)$ denotes the unit outward normal to $\partial\Omega$.
By taking the derivative with respect to $x$ on the both side of (\ref{5.1.3-1})
and then setting $P=x$,
we obtain
\begin{equation}\label{5.1.3-2}
\aligned
\Pi(x,y;A)
=& 
\int_\Omega \nabla_x \nabla_z \Gamma (x,z;A) \big\{ A(x)-A(z)\big\}
\nabla_1 \Gamma (z,y; A(x))\, dz\\
& +\int_{\partial\Omega}
\nabla_z \nabla_x \Gamma(x,z;A) A(z) n(z) \Gamma(z,y;A(x))\, d\sigma(z)\\
&-\int_{\partial\Omega}
\nabla_x \Gamma(x,z; A) n(z) A(x) \nabla_z \Gamma(z,y;A(x))\, d\sigma (z).
\endaligned
\end{equation}
To estimate $\Pi(x,y;A)-\Pi(x,y;B)$, we split its solid integrals as
$I_1+I_2+I_3$, where
$$
\aligned
& I_1=\int_\Omega
\big\{ \nabla_x\nabla_z \Gamma(x,z;A)-\nabla_x\nabla_z \Gamma(x,z;B) \big\}
\big\{ A(x)-A(z) \big\}
\nabla_1 \Gamma(z,y;A(x))\, dz,\\
& 
I_2=\int_\Omega \big\{ \nabla_x\nabla_z \Gamma(x,z;B)\big\}
\big\{ A(x)-A(z)-(B(x)-B(z))\big\}
\nabla_1 \Gamma(z,y;A(x))\, dz,\\
&I_3=\int_\Omega
\big\{\nabla_x\nabla_z \Gamma(x,z;B)\big\}
\big\{ B(x)-B(z)\big\}
\big\{ \nabla_1 \Gamma(z,y;A(x))-\nabla_1 \Gamma(z,y;B(x))\big\}\, dz.
\endaligned
$$
It follows from (\ref{estimate-5.1-1}) that
$$
\aligned
|I_1| & \le C_R \|A-B\|_{C^\lambda(\br^d)}
\int_\Omega \frac{dz}{|x-z|^{d-\lambda} |z-y|^{d-1}}\\
&\le C_R \|A-B\|_{C^\lambda(\br^d)}|x-y|^{1-d+\lambda}.
\endaligned
$$
Similarly, by estimates (\ref{size-estimate-5}) and (\ref{5.1.1-3}),
$$
\aligned
|I_2|\le & C \,\|A-B\|_{C^{0, \lambda}(\br^d)} |x-y|^{1-d+\lambda},\\
|I_3|\le & C \, \|A-B\|_\infty |x-y|^{1-d+\lambda}.
\endaligned
$$

Finally, we split the surface integrals in $\Pi(x,y;A)-\Pi(x,y;B)$
in a similar manner.
Using the fact that $|x-z|\ge R$ and $|y-z|\ge R$ for $z\in \partial\Omega$, we may show that
they are bounded by $C_R \|A-B\|_{C^\lambda(\br^d)}$.
\end{proof}

\begin{remark}\label{remark-5.1-2}
{\rm 
For $A\in \Lambda(\mu, \lambda, \tau)$, define
\begin{equation}\label{definition-of-Theta}
\Theta (x,y; A)=\nabla_2 \Gamma(x,y;A)-\nabla_2 \Gamma(x,y;A(y)).
\end{equation}
Let $A,B\in \Lambda(\mu, \lambda,\tau)$ and $R\ge 1$. Then
\begin{equation}\label{estimate-of-Theta-5}
|\Theta(x,y;A)-\Theta(x,y;B)|\le C_R \|A-B\|_{C^{\lambda}(\br^d)}
|x-y|^{1-d+\lambda}
\end{equation}
for any $x, y\in \br^d$ with $|x-y|\le R$,
where $C_R$ depends only on $\mu$, $\lambda$, $\tau$, and $R$.
The proof is similar to that of Theorem \ref{theorem-5.1}.
}
\end{remark}

%
%
%
%
%
%
%

\section{Estimates of singular integrals}\label{section-5.3}

Let $\Omega$ be a bounded Lipschitz domain.
For $A\in \Lambda(\mu, \lambda, \tau)$,
consider two maximal singular integral operators
$T_A^{1, *}$ and $T_A^{2,*}$ on $\partial\Omega$,
defined by
\begin{equation}\label{definition-of-T*}
\aligned
T_A^{1,*} (f) (P)
& =\sup_{r>0}
\big|
\int_{ \substack{y\in \partial \Omega\\ |y-P|>r}} 
\nabla_1 \Gamma(P,y; A) f(y)\, d\sigma(y) \big|,\\
T_A^{2,*} (f) (P)
& =\sup_{r>0}
\big|
\int_{ \substack{ y\in \partial \Omega\\  |y-P|>r}} 
\nabla_2 \Gamma(P,y; A) f(y)\, d\sigma(y) \big|
\endaligned
\end{equation}
for $P\in \partial\Omega$.

\begin{thm}\label{max-singular-theorem}
Let $f\in L^p(\partial \Omega;\br^m)$ for some $1<p<\infty$.
Then
\begin{equation}\label{max-singular-estimate}
\|T^{1,*}_A (f)\|_{L^p(\partial \Omega)}
+\|T^{2,*}_A (f)\|_{L^p(\partial \Omega)}
\le C_p \, \| f\|_{L^p(\partial \Omega)},
\end{equation}
where $C_p$ depends only on $\mu$, $\lambda$, $\tau$, $p$, and the Lipschitz character of $\Omega$.
\end{thm}

 To establish the $L^p$ boundedness of $T_A^{1,*}$,
we shall approximate its integral kernel $\nabla_1 \Gamma (P,y;A)$
by $\nabla_1 \Gamma(P,y; A(P))$ when $|P-y|\le 1$, and by
$(I +\nabla \chi (P))\nabla_1 \Gamma(P,y;\widehat{A})$ when
$|P-y|\ge 1$. The operator $T_A^{2,*}$ can be handled in a similar manner.

Let $\mathcal{M}_{\partial \Omega}$ denote the Hardy-Littlewood maximal operator on $\partial \Omega$.

\begin{lemma}\label{sin-1}
For each $P\in \partial \Omega$, 
\begin{equation}\label{5.2.1-1}
\aligned
T_A^{1,*} (f)(P) 
&\le C \mathcal{M}_{\partial \Omega} (f)(P)
+2\sup_{r>0} \big|
\int_{\substack{y\in \partial \Omega\\ |y-P|>r}}
\nabla_1 \Gamma(P,y;A(P)) f(y)\, d\sigma (y)\big|\\
& \qquad\qquad +C \sup_{r>0}
\big|
\int_{\substack{y\in \partial \Omega\\ |y-P|>r}}
\nabla_1 \Gamma(P,y;\widehat{A}) f(y)\, d\sigma (y)\big|,\\
T_A^{2,*} (f)(P) 
&\le C \mathcal{M}_{\partial \Omega} (f)(P)
+2\sup_{r>0} \big|
\int_{\substack{y\in \partial \Omega\\ |y-P|>r}}
\nabla_2 \Gamma(P,y;A(y)) f(y)\, d\sigma (y)\big|\\
&\qquad\qquad +C \sup_{r>0}
\big|
\int_{\substack{y\in \partial \Omega\\ |y-P|>r}}
\nabla_2 \Gamma(P,y;\widehat{A}) g(y)\, d\sigma (y)\big|,\\
\endaligned
\end{equation}
where $g$ is a function satisfying $|g|\le C |f|$ on $\partial \Omega$,
 and $C$ depends only on $\mu$, $\lambda$, $\tau$, and the Lipschitz character of $\Omega$.
 \end{lemma}
 
 \begin{proof}
To prove (\ref{5.2.1-1}), we fix $P\in \partial \Omega$ and $r>0$.
If $r\ge 1$, we use the estimate (\ref{asymp-5-3})  and boundedness of $\nabla\chi$ to obtain
$$
\aligned
 &\big|\int_{|y-P|>r} \nabla_1 \Gamma(P,y;A) f(y)\, d\sigma (y)\big|\\
 &\quad
 \le C \big|\int_{|y-P|>r} \nabla_1 \Gamma(P,y; \widehat{A}) f(y)\, d\sigma (y) \big|
 +C  \int_{|y-P|>r}
 \frac{\ln [|P-y|+1]}{|P-y|^d} |f(y)|\, d\sigma (y)\\
 &\quad
 \le C \sup_{t>0}
 \big|
\int_{|y-P|>t} \nabla_1 \Gamma(P,y; \widehat{A}) f(y)\, d\sigma (y) \big|
+C \mathcal{M}_{\partial \Omega} (f) (P).
\endaligned
$$
If $0<r<1$, we split the set $\{y\in \partial \Omega: |y-x|>r\}$
as 
$$
\Big\{y\in \partial \Omega: |y-P|>1\Big\} \cup
\Big\{ y\in \partial \Omega: 1\ge |y-P|>r\Big\}.
$$
The integral of $\nabla_1 \Gamma(P,y; A) f(y)$ on 
$\{ y\in \partial \Omega: |y-P|>1\}$
may be treated as above.
To handle the integral on $\{ y\in \partial \Omega: 1\ge |y-P|>r\}$,
we use the estimate (\ref{fundamental-local-estimate}) to obtain
$$
\aligned
& \big|\int_{1\ge |y-P|>r} \nabla_1 \Gamma(P,y;A) f(y)\, d\sigma (y)\big|\\
&\quad
\le \big|\int_{1\ge |y-P|>r}
\nabla_1 \Gamma(P,y; A(P)) f(y)\, d\sigma (y)\big|
+ C \int_{|y-P|\le 1}
\frac{|f(y)|}{|y-P|^{d-1-\lambda}}\, d \sigma (y)\\
&\quad \le 2 \sup_{t>0}
\big|
\int_{|y-P|>t}
\nabla_1 \Gamma(P,y;A(P)) f(y)\, d\sigma (y)\big|
+ C \mathcal{M}_{\partial \Omega} (f) (P).
\endaligned
$$
This gives the desired estimates for $T_A^{1,*}(f)$.
The estimate for $T_A^{2,*}(f)$  in (\ref{5.2.1-1}) follows from 
(\ref{asymp-5-2}) and (\ref{fundamental-local-estimate-1}) in the same manner.
The details are left to the reader.
\end{proof}

\begin{lemma}\label{sin-2}
Let $K(x, y)$ be odd in $x\in \br^d$ and homogenous of degree $1-d$ in $x\in \br^d$; i.e.,
$$
K(-x, y)=-K(x, y) \quad \text{ and } \quad
K(tx, y)=t^{1-d} K(x, y)
$$
for $x\in \br^d\setminus \{ 0\}$, $ y\in \br^d$ and $t>0$.
Assume that for all $0\le N\le N(d)$, where $N(d)>1$ is sufficiently large,
$\nabla_x^N K(x, y)$ is continuous on $\mathbb{S}^{d-1}\times \br^d$ and
$|\nabla_x^N K(x, y)|\le C_0$ for $x\in \mathbb{S}^{d-1}$ and $y\in \br^d$.
Let $f\in L^p(\partial\Omega)$ for some $1<p<\infty$.
Define
$$
\aligned
S^1 (f) (x) &=\text{\rm p.v.} \int_{\partial\Omega} K(x-y, x) f(y)\, d\sigma (y)\\
 : & =\lim_{r \to 0} \int_{\substack{ y\in \partial\Omega\\ |y-x|>r}} K(x-y, x) f(y)\, d\sigma (y) ,\\
S^2 (f) (x) &=\text{\rm p.v.} \int_{\partial\Omega} K(x-y, y) f(y)\, d\sigma (y),\\
S^{1, *} (f) (x) &=\sup_{r>0} \Big|
 \int_{\substack{ y\in \partial\Omega\\ |y-x|>r}} K(x-y, x) f(y)\, d\sigma (y)\Big|,\\
S^{2, *} (f) (x) &=\sup_{r>0} \Big|
 \int_{\substack{ y\in \partial\Omega\\ |y-x|>r}} K(x-y, y) f(y)\, d\sigma (y)\Big|.
\endaligned
$$
Then $S^1(f) (x)$ and $S^2(f) (x)$ exist for a.e. $x\in \partial\Omega$, and
$$
\| S^{1, *} (f)\|_{L^p(\partial\Omega)}
+\|S^{2, *} (f) \|_{L^p(\partial\Omega)}
\le  C C_0 \| f\|_{L^p(\partial\Omega)},
$$
where $C$ depends only on $p$ and the Lipschitz character of $\Omega$.
\end{lemma}

\begin{proof}
By considering $C_0^{-1} K(x, y)$, one may assume that $C_0=1$.
In the special case where $K(x, y)$ is independent of $y$, the 
result is a consequence of the $L^p$ boundedness of Cauchy integrals on Lipschitz curve
\cite{Cauchy-I}.
The general case may be deduced from the special case by using the spherical harmonic decomposition 
(see e.g. \cite{Mitrea-1999}).
Note that only the continuity condition in the variable $y$ is needed for $\nabla_x^N K(x, y)$.
\end{proof}

\begin{proof}[\bf Proof of Theorem \ref{max-singular-theorem}]
This follows readily from Lemmas \ref{sin-1} and \ref{sin-2}.
Note that if $B$ is a constant matrix satisfying (\ref{weak-eee}), the fundamental
solution $\Gamma (x, y; B)=\Gamma(x-y, 0; B)$ and 
$\Gamma (x, 0; B)$ is even in $x$ and homogeneous of degree $2-d$ in $x$.
Also recall that $\mathcal{M}_{\partial\Omega}$ is bounded on $L^p(\partial\Omega)$
for $1<p\le \infty$.
\end{proof}

Next we consider two singular integral operators:
\begin{equation}\label{definition-of-T}
\aligned
T_A^1 (f)(x)
&=\text{\rm p.v.}
\int_{\partial \Omega}
\nabla_1 \Gamma(x,y;A) f(y)\, d\sigma (y),\\
T_A^2 (f)(x)
&=\text{\rm p.v.}
\int_{\partial \Omega}
\nabla_2 \Gamma(x,y;A) f(y)\, d\sigma (y).
\endaligned
\end{equation}

\begin{thm}\label{pointwise-limit-theorem}
Let $A\in \Lambda(\mu, \lambda, \tau)$ and
$\Omega$ be a bounded Lipschitz domain.
Let $f\in L^p(\partial \Omega;\br^m)$ for some $1<p<\infty$.
Then $T_A^1 (f)(x)$ and $T_A^2 (f)(x)$ exist
for a.e. $x\in \partial \Omega$, and
$$
\|T_A^1(f)\|_{L^p(\partial\Omega)} +\|T_A^2(f)\|_{L^p(\partial\Omega)}
\le C_p\,  \| f\|_{L^p(\partial\Omega)},
$$
where $C_p$ depends only on $\mu$, $\lambda$, $\tau$, $p$, and the Lipschitz character of $\Omega$.
\end{thm}

\begin{proof}
Let $f\in C_0^\infty(\br^d;\br^m)$.
Write
\begin{equation}\label{5.2-2-1}
\aligned
\text{\rm p.v.}
\int_{\partial \Omega} \nabla_1  \Gamma(x,y;A) & f(y)  \, d\sigma (y)
=\text{\rm p.v.} 
\int_{\partial \Omega} \nabla_1 \Gamma(x,y; A(x)) f(y)\, d\sigma (y)\\
&
+\text{p.v}\int_{\partial \Omega}
\big[ \nabla_1 \Gamma(x,y;A)-\nabla_1 \Gamma(x,y;A(x))\big] f(y)\, d\sigma (y).
\endaligned
\end{equation}
Since 
$$
|\nabla_1\Gamma(x,y;A)-\nabla_1 \Gamma(x,y;A(x))|\le C |P-y|^{1-d+\lambda},
$$
the second term in the RHS of (\ref{5.2-2-1})
exists for any $x\in \partial \Omega$.
By Lemma \ref{sin-2}, the first term in the RHS of (\ref{5.2-2-1})
exists for a.e. $x\in \partial \Omega$.
This shows that if $f\in C_0^\infty(\br^d;\br^m)$,
$T_A^1(f)(x)$ exists for a.e. $x\in \partial \Omega$.
Since $C_0^\infty(\br^d;\br^m)$ is dense in
$L^p(\partial \Omega;\br^m)$ and the maximal singular integral
operator $T_A^{1,*}$ is bounded on $L^p(\partial \Omega)$,
we conclude that if $f\in L^p(\partial \Omega;\br^m)$,
$T_A^1(f)(P)$ exists for a.e. $x\in \partial \Omega$.
Since $|T_A^1(f)(P)|\le T_A^{1,*}(f)(P)$,
it follows by Theorem \ref{max-singular-theorem}
that
$\|T_A^1(f)\|_{L^p(\partial\Omega)}
\le C_p \, \| f\|_{L^p(\partial\Omega)}$.
The case of $T_A^2(f)$ may be handled in a similar manner.
\end{proof}

\begin{thm}\label{theorem-5.2-2}
Let $T_A^1, T_A^2, T_B^1$, and $T_B^2$ be defined by (\ref{definition-of-T}), 
where $A,B\in \Lambda(\mu, \lambda, \tau)$.
Let $\Omega$ be a bounded Lipschitz domain with $\text{\rm diam}(\Omega)\le 10$.
Then, for $1<p<\infty$,
\begin{equation}\label{estimate-5.2-2}
\aligned
\| T_A^1(f)-T_B^1 (f)\|_{L^p(\partial\Omega)}
& \le C_p\,  \| A-B\|_{C^{\lambda}(\br^d)} \| f\|_{L^p(\partial\Omega)},\\
\| T_A^1(f)-T_B^1 (f)\|_{L^p(\partial\Omega)}
& \le C_p\  \| A-B\|_{C^{\lambda}(\br^d)} \| f\|_{L^p(\partial\Omega)},
\endaligned
\end{equation}
where $C_p$ depends only on $\mu$, $\lambda$, $\tau$, $p$,
and the Lipschitz character of $\Omega$.
\end{thm}

\begin{proof}
Recall that $\Pi(x,y;A)=\nabla_1 \Gamma(x,y;A)-\nabla\Gamma(x,y, A(x))$.
Write
$$
\aligned
&\nabla_1 \Gamma(x,y;A)-\nabla_1 \Gamma(x,y;B)\\
&=\Pi(x,y;A)-\Pi (x,y;B)
+\big\{ \nabla_1 \Gamma(x-y,0;A(x))-\nabla_1 \Gamma(x-y,0;B(x))\big\}.
\endaligned
$$
Since diam$(\Omega)\le 10$,
it follows from Theorem \ref{theorem-5.1} that the norm of the operator 
with integral kernel $\Pi(x,y;A)-\Pi(x,y;B)$ on $L^p(\partial\Omega)$
is bounded by $C \|A-B\|_{C^\lambda(\br^d)}$
for $1\le p\le \infty$. 
By Lemma \ref{sin-2}, the norm of the operator with integral kernel
$$
\nabla_1 \Gamma(x-y,0;A(x))-\nabla_1\Gamma(x-y,0;B(x))
$$
on $L^p(\partial \Omega)$ for $1<p<\infty$
is bounded by $C \|A-B\|_\infty$.
This gives the desired estimate for
$\|T_A^1(f)-T_B^1 (f)\|_{L^p(\partial\Omega)}$.
The estimate for
$\|T_A^2(f)-T_B^2(f)\|_{L^p(\partial\Omega)}$
follows from Remark \ref{remark-5.1-2} and Lemma \ref{sin-2} in the same manner.
\end{proof}

We end this section with some notation and
a theorem on nontangential maximal functions.
Since we will be dealing with functions on domains
$$
\Omega_+=\Omega \quad \text{ and } \quad \Omega_-=\br^d\setminus \overline{\Omega}
$$
simultaneously, we introduce the following notation.
For a continuous function $u$ in $\Omega_\pm$, the nontangential maximal function
$(u)^*_\pm$ is defined by
\begin{equation}\label{nontangential-max-5-2}
(u)^*_\pm (y)  =\sup \big\{ |u(x)|: \ x\in \Omega_\pm \text{ and } x\in \gamma (y)\big\}
\end{equation}
for $y\in \partial\Omega$.
We will use $u_\pm (y)$ to denote the nontangential limit at $y$, if
it exists, taken from $\Omega_\pm$ respectively.
If $u$ is continuous in $\br^d\setminus \partial\Omega$, we define
$$
(u)^*(P)=\max \big\{ (u)^*_+ (P), \, (u)^*_-(P) \big\}.
$$
We also use $(u)^*$ to denote $(u)^*_+$ or $(u)^*_-$ if
there is no possibility of confusion.

For $f\in L^p(\partial\Omega;\br^m)$, consider the following two functions:
\begin{equation}\label{definition-of-v-w}
\aligned
v(x)& =\int_{\partial\Omega}
\nabla_1\Gamma(x,y;A) f(y)\, d\sigma (y),\\
w(x)& =\int_{\partial\Omega}
\nabla_2\Gamma(x,y;A) f(y)\, d\sigma (y),
\endaligned
\end{equation}
defined in $\br^d\setminus \partial\Omega$.

\begin{thm}
\label{theorem-5.2-3}
Let $\Omega$ be a bounded Lipschitz domain.
Let $v$ and $w$ be defined by (\ref{definition-of-v-w}),
where $A\in \Lambda (\mu, \lambda, \tau)$.
Then, for $1<p<\infty$,
\begin{equation}\label{estimate-5.2-3}
\|(v)^*\|_{L^p(\partial\Omega)}
+\|(w)^* \|_{L^p(\partial \Omega)}
\le C_p \| f\|_{L^p(\partial\Omega)},
\end{equation}
where $C_p$ depends only on $\mu$, $\lambda$, $\tau$, $p$, and the Lipschitz character of
$\Omega$.
\end{thm}

\begin{proof}
We first consider the case that $A=E$ is a constant matrix in $\Lambda(\mu, \lambda,\tau)$.
Let
$$
u(x)=\int_{\partial\Omega} \nabla_1 \Gamma(x,y;E) f(y)\, d\sigma (y).
$$
Fix $z\in \partial \Omega$.
Let $x\in \gamma(z)$ and $r=|x-z|$.
Write $u(x)=I_1 +I_2+I_3$, where
$$
\aligned
I_1 &=\int_{\substack{y\in \partial\Omega\\ |y-z|\le 4r}}
\nabla_1 \Gamma(x,y;E) f(y)\, d\sigma (y),\\
I_2 &=\int_{\substack{y\in \partial\Omega\\ |y-z|>4r}}
 \nabla_1 \Gamma(z,y;E) f(y)\, d\sigma (y),\\
I_3&=
\int_{\substack{y\in \partial\Omega\\ |y-z|>4r}}
\big\{ \nabla_1 \Gamma(x,y;E)-\nabla_1 \Gamma(z,y;E)\big\} f(y)\, d\sigma (y).
\endaligned
$$
Using $|\nabla_x \Gamma(x,0;E)|\le C|x|^{1-d}$ and
$|\nabla_x^2 \Gamma(x,0;E)|\le C |x|^{-d}$, we obtain
$$
\aligned
|I_1|+|I_3| & \le
\frac{C}{r^{d-1}}
\int_{\substack{y\in \partial\Omega\\ |y-z|\le 4r}}
|f(y)|\, d\sigma (y)
+Cr \int_{\substack{y\in \partial\Omega\\ |y-z|>4r}}
\frac{|f(y)|}{|y-z|^d}\, d\sigma (y)\\
&\le
C \mathcal{M}_{\partial\Omega} (f) (z).
\endaligned
$$
It follows that
\begin{equation}\label{5.2-3-1}
(u)^*(z)\le C \mathcal{M}_{\partial\Omega} (f)(z)
+\sup_{r>0}
\big|\int_{\substack{y\in \partial\Omega\\ |y-z|>r}}
\nabla_1 \Gamma(z,y;E) f(y)\, d\sigma (y)\big|.
\end{equation}
In view of Theorem \ref{max-singular-theorem}, this gives
$\|(u)^*\|_{L^p(\partial\Omega)} \le C \| f\|_{L^p(\partial\Omega)}$.

We now return to the functions $v$ and $w$ for the general case
$A\in \Lambda(\mu, \lambda, \tau)$.
We claim that for any $z\in \partial\Omega$,
\begin{equation}\label{5.2-3-3}
\aligned
(v)^*(z) \le & C\mathcal{M}_{\partial\Omega} (f) (z)
+C \sup_{t>0}
\big|\int_{\substack{y\in \partial\Omega\\ |y-z|>t}}
\nabla_1 \Gamma(z,y, A(z)) f(y)\, d\sigma (y)\big|\\
&\qquad + C \sup_{t>0}
\big|\int_{\substack{y\in \partial\Omega\\ |y-P|>t}}
\nabla_1 \Gamma(z,y; \widehat{A}) f(y)\, d\sigma (y)\big|,\\
(w)^*(z) & \le C\mathcal{M}_{\partial\Omega} (f)(z)
+C \sup_{t>0}
\big|\int_{\substack{y\in \partial\Omega\\ |y-z|>t}}
\nabla_2 \Gamma(z,y, A(y)) f(y)\, d\sigma (y)\big|\\
&\qquad + C \sup_{t>0}
\big|\int_{\substack{y\in \partial\Omega\\ |y-z|>t}}
\nabla_2 \Gamma(z,y; \widehat{A}) g(y)\, d\sigma (y)\big|,
\endaligned
\end{equation}
where $|g (y)|\le C |f(y)|$, and $C$ depends only on $\mu$, $\lambda$, $\tau$, and
the Lipschitz character of $\Omega$.
Estimate (\ref{estimate-5.2-3}) follows from (\ref{5.2-3-3})
by Lemma \ref{sin-2}.
We will give the proof for $(v)^*$;
the estimate for $(w)^*$ may be carried out in the same manner.

Fix $z\in \partial\Omega$.
Let $x\in \gamma (z)$ and $r=|x-z|$.
If $r\ge 1$, it follows from (\ref{asymp-5-3}) that
$$
\aligned
|v(x)-(I+\nabla\chi(x))\nabla U(x)|
 & \le C \int_{\partial\Omega}
\frac{\ln [|x-y|+2]}{|x-y|^d} |f(y)|\, d\sigma(y)\\
&\le C \mathcal{M}_{\partial\Omega} (f)(z),
\endaligned
$$
where 
$$
U(x)=\int_{\partial\Omega} \Gamma(x,y;\widehat{A}) f(y)\, d\sigma (y).
$$
Hence,
$$
\aligned
|u(x)|
&\le C \mathcal{M}_{\partial\Omega} (f)(z)
+ C (\nabla U)^* (z)\\
&\le C \mathcal{M}_{\partial\Omega} (f)(z)
+C \sup_{t>0}
\big|\int_{\substack{y\in \partial\Omega\\ |y-z|>t}}
\nabla_1 \Gamma(z,y;\widehat{A}) f(y)\, d\sigma (y)\big|,
\endaligned
$$
where we have used (\ref{5.2-3-1}).

Next suppose that $r=|x-z|<1$.
We write $u(x)=J_1+J_2+J_3$,
where $J_1$, $J_2$, and $J_3$ denote the integral of
$\nabla_1\Gamma(x,y;A) f(y)$
over 
$$
\aligned
E_1 & =\{ y\in \partial\Omega: \ |y-z|<r\},\\
E_2 & =\{ y\in \partial\Omega: \ r\le |y-z|\le 1\},\\ 
E_3 &=\{ y\in \partial\Omega: \ |y-z|>1\},
\endaligned
$$
respectively.
Clearly, $|J_1|\le C \mathcal{M}_{\partial\Omega} (f)(P)$.
For $J_2$, we use (\ref{fundamental-local-estimate}) to obtain
$$
\aligned
|J_2|
&\le \big|\int_{E_2} \nabla_1 \Gamma(x,y; A(y)) f(y)\, d\sigma (y)\big|
+C \int_{E_2} \frac{|f(y)|}{|x-y|^{d-1-\lambda}}\, d\sigma (y)\\
& 
\le \big|\int_{E_2} \nabla_1 \Gamma(z,y; A(z)) f(y)\, d\sigma (y)\big|
+C \mathcal{M}_{\partial\Omega} (f)(z)
+C \int_{E_2} \frac{|f(y)|}{|y-z|^{d-1-\lambda}}\, d\sigma (y)\\
&\le
2\sup_{t>0}
\big|\int_{\substack{y\in \partial\Omega\\ |y-z|>t}} 
\nabla_1 \Gamma(z,y; A(z)) f(y)\, d\sigma (y)\big|
+C \mathcal{M}_{\partial\Omega} (f)(z).
\endaligned
$$
In view of (\ref{asymp-5-3}), we have
$$
\aligned
|J_3|
&\le C \int_{E_3} \frac{\ln [|x-y|+2]}{|x-y|^d} |f(y)|\, d\sigma (y)
+C \big| \int_{E_3} \nabla_1 \Gamma(x,y;\widehat{A}) f(y)\, d\sigma (y)\big|\\
&\le C \mathcal{M}_{\partial\Omega} (f) (z)
+C \sup_{t>0}
\big| \int_{\substack{y\in \partial\Omega\\ |y-z|>t}} \nabla_1 \Gamma(x,y;\widehat{A}) f(y)\, d\sigma (y)\big|.
\endaligned
$$
This, together with the estimates of $J_1$ and $J_2$,
yields the desired estimate for $(v)^*(z)$.
\end{proof}

%
%
%
%
%
%
%
%

\section{Method of layer potentials}\label{section-5.4}

In this section we fix $A=(a_{ij}^{\alpha\beta}(x))\in \Lambda(\mu, \lambda, \tau)$
and let
$\mathcal{L}_\varep=-\text{div}(A(x/\varep)\nabla )$.

\begin{definition}
{\rm
Let $\Omega$ be a bounded Lipschitz domain and
$f=(f^\alpha)\in L^p(\partial\Omega; \br^m)$ with $1<p<\infty$.
The single layer potential $\mathcal{S}_\varep (f)=(\mathcal{S}^\alpha_\varep (f))$
is defined by
\begin{equation}\label{definition-of-single}
\mathcal{S}_\varep^\alpha(f) (x)=\int_{\partial\Omega}
\Gamma_\varep^{\alpha\beta}(x,y) f^\beta (y)\, d\sigma (y),
\end{equation}
where $\Gamma_\varep (x,y)=(\Gamma_\varep^{\alpha\beta}(x,y))$
is the matrix of fundamental solutions for
$\mathcal{L}_\varep$ in $\br^d$.
The double layer potential 
$\mathcal{D}_\varep (f)
=(\mathcal{D}_\varep^\alpha (f))$ is defined by
\begin{equation}
\label{definition-of-doule}
\mathcal{D}_\varep^\alpha (f)(x)
=\int_{\partial \Omega}
n_j(y) a_{ij}^{\beta\gamma}(y/\varep)
\frac{\partial}{\partial y_i}
\Big\{ \Gamma_\varep^{\alpha\beta} (x,y)\Big\}
f^\gamma (y)\, d\sigma (y).
\end{equation}
}
\end{definition}

Observe that both $\mathcal{S}_\varep (f)$ and $\mathcal{D}_\varep (f)$
are solutions of
$\mathcal{L}_\varep (u)=0$ in $\br^d\setminus \partial\Omega$.
Since
\begin{equation}\label{adjoint-relation}
\Gamma^{\alpha\beta}(x,y; A_\varep^*)=\Gamma^{\beta\alpha}(y,x;A_\varep)
=\Gamma_\e^{\beta\alpha} (y, x),
\end{equation}
 where
$A_\varep(x)=A(x/\varep)$,
the double layer potential may be written as
\begin{equation}\label{definition-of-double-1}
\mathcal{D}_\varep^\alpha (f) (x)
=\int_{\partial\Omega}
\left( \frac{\partial}{\partial\nu_\varep^*}
\Big\{ \Gamma^\alpha(y,x; A_\varep^*)\Big\}\right)^\gamma f^\gamma (y)\, d\sigma (y),
\end{equation}
where $\frac{\partial}{\partial\nu_\varep^*}$
denotes the conormal derivative associated with $\mathcal{L}_\varep^*$.
The definitions of single and double layer potentials are motivated by the following
Green representation formula.

\begin{prop}\label{Green's-representation-prop}
Suppose that $\mathcal{L}_\varep (u_\varep)=F$ in $\Omega$,
where $F\in L^p(\Omega; \mathbb{R}^m)$ for some $p>d$.
Also assume that $(\nabla u_\varep)^*\in L^1(\partial\Omega)$
and $u_\varep, \nabla u_\varep$ have nontangential limits a.e.
on $\partial\Omega$.
Then for any $x\in \Omega$,
\begin{equation}\label{Green's-representation-formula}
u_\varep (x)
=\mathcal{S}_\varep \left(\frac{\partial u_\varep}{\partial\nu_\varep}\right) (x)
-\mathcal{D}_\varep (u_\varep) (x)
+\int_\Omega \Gamma _\varep (x, y) F(y)\, dy.
\end{equation}
\end{prop}

\begin{proof}
Fix $x\in \Omega$ and choose $r>0$ so small that $B(x,4r)\subset \Omega$.
Let $\varphi\in C_0^\infty(B(x,2r))$ be such that $\varphi=1$ in $B(x,r)$.
It follows from (\ref{fundamental-representation}) that
$$
\aligned
u_\varep^\gamma (x)
&=(u_\varep\varphi)^\gamma (x)\\
& =\int_\Omega 
a_{ji}^{\beta\alpha}(y/\varep)
\frac{\partial}{\partial y_j}
\Big\{ \Gamma^{\beta\gamma}(y,x;A_\varep^*)\Big\}
\cdot \frac{\partial}{\partial y_i} (u_\varep^\alpha\varphi)\, d y\\
&=\int_\Omega
a_{ji}^{\beta\alpha} (y/\varep)
\frac{\partial}{\partial y_j}
\Big\{ \Gamma^{\beta\gamma}(y,x;A_\varep^*)\Big\}
\cdot \frac{\partial }{\partial y_i} \Big\{ u_\varep^\alpha (\varphi-1)\Big\} \, dy\\
& \quad
+\int_\Omega 
a_{ji}^{\beta\alpha}(y/\varep)
\frac{\partial}{\partial y_j}
\Big\{ \Gamma^{\beta\gamma}(y,x;A_\varep^*)\Big\}
\cdot \frac{\partial u_\varep^\alpha}{\partial y_i} \, dy\\
&=I_1^\gamma +I_2^\gamma.
\endaligned
$$
Using the divergence theorem and $\mathcal{L}_\varep^* \big\{ \Gamma^\gamma(\cdot, x;A_\varep^*)\big\}
=0$ in $\br^d\setminus \{ x\}$, we obtain
$$
\aligned
I_1^\gamma
&=\int_{\partial\Omega} n_i(y) a_{ji}^{\beta\alpha} (y/\varep)
\frac{\partial}{\partial y_j}\Big\{ \Gamma_\varep^{\beta\gamma} (y, x; A_\varep^*) \Big\}
u_\varep^\alpha (\varphi-1)\, d\sigma (y)\\
& =-\mathcal{D}_\varep^\gamma (u_\varep) (x),
\endaligned
$$
where we also used the fact that $\varphi-1=0$ in $B(x, r)$ and $\varphi-1 =-1$ on $\partial\Omega$.
Similarly,  since $\mathcal{L}(u_\varep)=F$ in $\Omega$,
it follows from the divergence theorem  that
$$
\aligned
I_2^\gamma
&=\int_{\partial\Omega} \Gamma_\varep^{\beta\gamma} (y, x; A_\varep^*) 
n_j(y) a_{ji}^{\beta\alpha} (y/\varep) \frac{\partial u_\varep^\alpha}{\partial y_i} \, d\sigma (y)
+\int_\Omega \Gamma_\varep^{\beta\gamma} (y, x; A_\varep^*) F^\beta (y)\, dy\\
&= \mathcal{S}_\varep^\gamma \left(\frac{\partial u_\varep}{\partial \nu_\varep} \right) (x)
+\int_\Omega \Gamma_\varep^{\gamma\beta} (x,y) F^\beta (y)\, dy.
\endaligned
$$
Hence,
$$
\aligned
u^\gamma (x) &=I_1^\gamma (x) +I_2^\gamma (x)\\
&=-\mathcal{D}_\varep^\gamma (u_\varep) (x) +
\mathcal{S}_\varep^\gamma \left(\frac{\partial u_\varep}{\partial \nu_\varep} \right) (x)
+\int_\Omega \Gamma_\varep^{\gamma\beta} (x,y) F^\beta (y)\, dy.
\endaligned
$$
We point out that since $F\in L^p(\Omega; \mathbb{R}^m)$ for some $p>d$,
 the solution $u_\varep \in C^1(\Omega; \mathbb{R}^m)$.
 This, together with the assumption that $(\nabla u_\varep)^*\in L^1(\partial\Omega)$
 and $u_\varep$, $\nabla u_\varep$ have nontangential limits,
 allows us to apply the divergence theorem on $\Omega$.
\end{proof}

\begin{thm}\label{nontangential-max-theorem-layer-potential}
Let $A\in \Lambda(\mu, \lambda, \tau)$ and $\Omega$ be a bounded Lipschitz domain.
Let $u_\varep (x)=\mathcal{S}_\varep (f)(x)$ and $v_\varep=\mathcal{D}_\varep(f)(x)$.
Then  for $1<p<\infty$,
\begin{equation}\label{estimate-max-layer-potential}
\| (\nabla u_\varep )^*\|_{L^p(\partial\Omega)}
+\| (v_\varep)^*\|_{L^p(\partial\Omega)}
\le C_p\,  \| f\|_{L^p(\partial\Omega)},
\end{equation}
where $C_p$ depends only on $\mu$, $\lambda$, $\tau$, $p$,
and the Lipschitz character of $\Omega$.
\end{thm}

\begin{proof}
Recall that
\begin{equation}\label{fundamental-rescaling}
\Gamma_\varep (x,y)= \Gamma(x,y; A_\varep)=
\varep^{2-d} \Gamma(\varep^{-1} x, \varep^{-1} y; A),
\end{equation}
where $A_\varep (x)=A(x/\varep)$.
Thus, by rescaling, it suffices to prove the theorem for the case $\varep=1$.
However, in this case, the estimate (\ref{estimate-max-layer-potential})
follows readily from Theorem \ref{theorem-5.2-3}.
Here we have used the observation that the rescaled domain $\{ x: \, \varep x\in \Omega\}$
and $\Omega$ have the "same" Lipschitz characters.
\end{proof}

The next theorem gives the nontangential limits of $\nabla \mathcal{S}_\varep (f)$ on $\partial\Omega$.
Recall that for a function $w$ in $\br^d\setminus \partial\Omega$,
we use $w_+$ and $w_-$ to denote its nontangential limits on $\partial\Omega$,
taken from inside $\Omega$ and outside $\overline{\Omega}$, respectively.

\begin{thm}\label{nontangential-limit-theorem}
Let $u_\varep=\mathcal{S}_\varep (f)$ for some $f\in L^p(\partial\Omega;\br^m)$ and
$1<p<\infty$.
Then for a.e. $x\in \partial\Omega$,
\begin{equation}\label{nontangential-limit}
\aligned
\left(\frac{\partial u_\varep^\alpha}{\partial x_i}\right)_\pm (x)
=& \pm \frac12 n_i(x) b^{\alpha\beta}_\varep (x) f^\beta(x)\\
& \quad +\text{\rm p.v.}
\int_{\partial\Omega}
\frac{\partial}{\partial x_i}
\Big\{ \Gamma_\varep^{\alpha\beta} (x, y)\Big\}
f^\beta (y)\, d\sigma(y),
\endaligned
\end{equation}
where $\big(b_\varep^{\alpha\beta}(x)\big)_{m\times m}$
is the inverse matrix of $\big( n_i(x)n_j(x)a_{ij}^{\alpha\beta}(x/\varep)\big)_{m\times m}$.
\end{thm}

\begin{proof}
By (\ref{fundamental-rescaling}) and rescaling it suffices to consider the case $\varep=1$.
By Theorem \ref{nontangential-max-theorem-layer-potential} 
we may assume that  $f\in C_0^\infty(\br^d;\br^m)$.
We will also use the fact that if $f\in C_0^\infty(\br^d;\br^m)$, 
there exists a set $F\subset \partial\Omega$ such that
$\sigma (\partial\Omega\setminus E)=0$ and
the trace formula (\ref{nontangential-limit}) holds
for any $x\in E$ and for any constant matrix $A$ satisfying the ellipticity condition (\ref{weak-eee})
(see \cite{D-Verchota-1990, Gao-1991}).

Now fix $z\in E$ and consider 
$$
w^\alpha (x)=\int_{\partial\Omega} \Gamma^{\alpha\beta} (x,y; A(z)) f^\beta(y)\, d\sigma (y),
$$
the single layer potential for the (constant coefficient) operator $-\text{div}(A(z)\nabla)$.
Note that by (\ref{fundamental-local-estimate}) and (\ref{5.1.1-5}),
$$
\aligned
|\nabla_1  &\Gamma(x,y;A)-\nabla_1 \Gamma(x,y;A(z))|\\
& \le |\nabla_1\Gamma(x,y;A)-\nabla_1\Gamma(x,y;A(y))|
+|\nabla_1 \Gamma(x,y;A(y))-\nabla_1 \Gamma(x,y;A(z))|\\
& \le C |x-y|^{1-d+\lambda}
+C |x-y|^{1-d} |y-z|^{\lambda}\\
& \le C |z-y|^{1-d+\lambda}
\endaligned
$$
for $x\in \gamma(z)$ and $y\in \partial\Omega$.
By the Lebesgue's dominated convergence theorem, this implies that
$$
\aligned
(\nabla u^\alpha)_\pm (z)
&=(\nabla w^\alpha)_\pm (z)
+\int_{\partial\Omega}
\big\{ \nabla_1 \Gamma^{\alpha\beta} (z,y;A)
-\nabla_1 \Gamma^{\alpha\beta}(z,y;A(z))\big\} f^\beta (y)\, d\sigma (y)\\
&=\pm \frac12 n(z) b^{\alpha\beta} (z) f^\beta (z)
+\text{\rm p.v.}
\int_{\partial\Omega}
\nabla_1 \Gamma^{\alpha\beta}
(z,y;A) f^\beta (y)\, d\sigma (y).
\endaligned
$$
The proof is complete.
\end{proof}

It follows from (\ref{nontangential-limit}) that if $u_\varep=\mathcal{S}_\varep (f)$,
\begin{equation}\label{tangential-limit}
n_j \left(\frac{\partial u_\varep^\alpha}{\partial x_i}\right)_+
-n_i \left(\frac{\partial u_\varep^\alpha}{\partial x_j}\right)_+
=
n_j \left(\frac{\partial u_\varep^\alpha}{\partial x_i}\right)_-
-n_i \left(\frac{\partial u_\varep^\alpha}{\partial x_j}\right)_-;
\end{equation}
i.e., $(\nabla_{\rm tan}u_\varep)_+
=(\nabla_{\rm tan} u_\varep)_-$ on $\partial\Omega$.
Moreover, 
\begin{equation}\label{conormal-layer-potential}
\left(\frac{\partial u_\varep}{\partial\nu_\varep}\right)_\pm
= \left( \pm \frac12 I + \mathcal{K}_{\varep,A} \right) (f),
\end{equation}
where $I$ denotes the identity operator, 
\begin{equation}\label{definition-of-K-5.4}
\left( \mathcal{K}_{\varep, A} (f)(x)\right)^\alpha
=\text{\rm p.v.}
\int_{\partial\Omega}
K_{\varep, A}^{\alpha\beta} (x,y) f^\beta (y)\, d\sigma (y),
\end{equation}
with
\begin{equation}\label{definition-of-kernel-K-5.4}
K_{\varep, A}^{\alpha\beta} (x,y)
=n_i(x) a_{ij}^{\alpha\gamma} (x/\varep) \frac{\partial}{\partial x_j}
\Big\{ \Gamma_\varep^{\gamma\beta} (x,y)\Big\}.
\end{equation}
In particular, we have the so-called  jump relation
\begin{equation}\label{jump-relation}
f=\left(\frac{\partial u_\varep}{\partial\nu_\varep}\right)_+
-
\left(\frac{\partial u_\varep}{\partial\nu_\varep}\right)_-.
\end{equation}
Note that by Theorems \ref{nontangential-max-theorem-layer-potential} and
\ref{nontangential-limit-theorem},
\begin{equation}\label{opertaor-estimate}
\|\mathcal{K}_{A, \varep} (f)\|_{L^p(\partial\Omega)}
\le C _p\,  \| f\|_{L^p(\partial\Omega)} \qquad \text{ for } 1<p<\infty,
\end{equation}
where $C_p$ depends at most on $\mu$, $\lambda$, $\tau$, $p$,
and the Lipschitz character of $\Omega$.
Let
\begin{equation}\label{definition-of-L-p-0}
L^p_0(\partial\Omega; \br^m)
=\left\{ f\in L^p(\partial\Omega; \br^m): \ \int_{\partial\Omega} f \, d\sigma =0\right\}.
\end{equation}
It follows from $\mathcal{L}_\varep (u_\varep)=0$ in $\Omega$ that
$$
\int_{\partial\Omega} \left(\frac{\partial u_\varep}{\partial \nu_\varep}\right)_+ \, d\sigma =0.
$$
Thus,  for $1<p<\infty$, the operator 
\begin{equation}\label{operator-on-L-p}
(1/2) I +\mathcal{K}_{\varep, A} :
L_0^p(\partial\Omega; \br^m)
\to L_0^p(\partial\Omega; \br^m)
\end{equation}
is bounded; its operator norm is bounded by a constant independent of $\varep>0$.
Also, by Theorem \ref{nontangential-max-theorem-layer-potential} and (\ref{tangential-limit}),
if $f\in L^p(\partial\Omega; \br^m)$ for some $1<p<\infty$,
then $\mathcal{S}_\varep (f)\in W^{1,p}(\partial\Omega; \br^m)$ and
\begin{equation}\label{single-W-1-p-estimate}
\| \mathcal{S}_\varep (f)\|_{W^{1,p}(\partial\Omega)}
\le C_p\, \| f\|_{L^p(\partial\Omega)},
\end{equation}
where $C_p$ depends only on $\mu$, $\lambda$, $\tau$, $p$,
and the Lispchitz character of $\Omega$.

The next theorem gives the trace of the double layer potentials on $\partial\Omega$.

\begin{thm}\label{double-layer-potential-trace-theorem}
Let $w=\mathcal{D}_\varep (f)$, where $f\in L^p(\partial\Omega;\br^m)$ and $1<p<\infty$.
Then
\begin{equation}\label{double-layer-trace-formula}
w_\pm =\Big( \mp (1/2) I + \mathcal{K}^*_{\varep, A^*} \Big) (f) \quad \text{ on } \partial \Omega,
\end{equation}
where $\mathcal{K}^*_{\varep, A^*}$ is the adjoint operator of
$\mathcal{K}_{\varep, A^*}$, defined by (\ref{definition-of-K-5.4}) and (\ref{definition-of-kernel-K-5.4}) 
(with $A$ replaced by $A^*$).
\end{thm}

\begin{proof}
By rescaling we may assume $\varep=1$.
Note that by (\ref{fundamental-local-estimate}) and (\ref{fundamental-local-estimate-1}),
$$
\aligned
|\nabla_1 \Gamma(x,y;A)-\nabla_1 \Gamma(x,y;A(y))| & \le C\, |x-y|^{1-d+\lambda},\\
|\nabla_2 \Gamma(x,y;A)-\nabla_2 \Gamma(x,y;A(y))| & \le C \, |x-y|^{1-d+\lambda}.
\endaligned
$$
This, together with the observation
 $\nabla_1 \Gamma(x,y;A(y))=-\nabla_2 \Gamma(x,y;A(y))$, gives
$$
\big|\frac{\partial}{\partial x_i} \Big\{ \Gamma(x,y;A)\Big\}
+\frac{\partial}{\partial y_i}\Big\{ \Gamma(x,y;A)\Big\}\big|
\le C\, |x-y|^{1-d+\lambda}.
$$
Hence, by the Lebesgue's dominated convergence theorem,
$$
\aligned
w_\pm^\alpha (x)
 =& -v_\pm^\alpha (x)\\
& +\int_{\partial\Omega}
n_j(y) a_{ij}^{\beta\gamma} (y) \left\{
\frac{\partial}{\partial y_i}
\big\{ \Gamma^{\alpha\beta}
(x,y;A)\big\}
+\frac{\partial}{\partial x_i}
\big\{ \Gamma^{\alpha\beta}(x,y;A)\big\} \right\} f^\gamma (y)\, d\sigma (y),
\endaligned
$$
where
$$
v^\alpha (x)
=\frac{\partial}{\partial x_i}
\int_{\partial\Omega}
n_j (y) a_{ij}^{\beta\gamma} (y)
\Gamma^{\alpha\beta} (x,y;A)
f^\gamma (y)\, d\sigma (y).
$$
In view of the trace formula (\ref{nontangential-limit}), we have
$$
\aligned
v_\pm^\alpha (x)
=& \pm (1/2) n_i (x) b^{\alpha\beta}(x)\cdot n_j (x) a_{ij}^{\beta\gamma}(P) f^\gamma (P)\\
&\qquad
 +\text{p.v.}
\int_{\partial\Omega} n_j (y) a_{ij}^{\beta\gamma} (y)
\frac{\partial}{\partial P_i} \Big\{ \Gamma^{\alpha\beta} (x,y;A)\Big\}
f^\gamma (y)\, d\sigma (y)\\
&=\pm (1/2) f^\alpha (x)
+\text{\rm p.v.}
\int_{\partial\Omega}
n_j(y) a_{ij}^{\beta\gamma} (y) 
\frac{\partial}{\partial P_i} \Big\{ \Gamma^{\alpha\beta} (x,y;A)\Big\}
f^\gamma (y)\, d\sigma (y).
\endaligned
$$
Hence,
$$
\aligned
w^\alpha_\pm (x)
=& \mp (1/2) f^\alpha (x)
+\text{\rm p.v.}
\int_{\partial\Omega}
n_j(y) a_{ij}^{\beta\gamma}(y) \frac{\partial}{\partial y_i}
\Big\{ \Gamma^{\alpha\beta}(x,y;A)\Big\} f^\gamma (y)\, d\sigma (y)\\
=& \mp (1/2) f^\alpha (x)
+\text{\rm p.v.}
\int_{\partial\Omega}
K_{1, A^*}^{\beta\alpha} (y,x) f^\beta (y) \, d\sigma (y),
\endaligned
$$
where $K_{1,A^*}^{\beta\alpha} (y,x)$ is defined by
(\ref{definition-of-kernel-K-5.4}), but with $A$ replaced by $A^*$.
\end{proof}

\begin{remark}\label{D-S-remark}
{\rm
Let $ f\in L^p(\partial\Omega;\br^m)$ and $u_\varep =\mathcal{S}_\varep (f)$.
It follows from the Green representation formula (\ref{Green's-representation-formula}),
$$
u_\varep (x)=\mathcal{S}_\varep (g) (x)
-\mathcal{D}_\varep (u_\varep) (x)
$$
for any $x\in \Omega$, where $g=\left(\partial u_\varep/\partial\nu_\varep\right)_+$.
By letting $x\to z\in \partial\Omega$ nontangentially, we obtain
$$
\mathcal{S}_\varep(f)=\mathcal{S}_\varep \left( \left( (1/2)I+\mathcal{K}_{\varep, A} \right)(f)\right)
-(-(1/2)I +\mathcal{K}^*_{\varep, A^*} ) \mathcal{S}_\varep (f) \qquad \text{ on } \partial\Omega.
$$
This gives
\begin{equation}\label{K-S-relation}
\mathcal{S}_\varep \mathcal{K}_{\varep, A} (f) =\mathcal{K}^*_{\varep, A^*} \mathcal{S}_\varep (f)
\end{equation}
for any $f\in L^p(\partial\Omega)$ with $1<p<\infty$.
}
\end{remark}

In summary, if $1<p<\infty$ and $f\in L^p(\partial\Omega;\br^m)$,
 the single layer potential $u_\varep =\mathcal{S}_\varep (f)$ is a solution to the $L^p$
Neumann problem for $\mathcal{L}_\varep (u_\varep)=0$
 in $\Omega$ with boundary data
$((1/2)I+\mathcal{K}_{\varep, A}) f$, and the estimate
$\|(\nabla u_\varep)^*\|_{L^p(\partial\Omega)}
\le C_p\, \| f\|_{L^p(\partial\Omega)}$ holds.
Similarly, the double layer potential 
 $v_\varep =\mathcal{D}(f)$ is a solution to the $L^p$ 
Dirichlet problem in $\Omega$ with boundary data
$(-(1/2)I +\mathcal{K}^*_{\varep, A^*}) f$, and the estimate 
$\|(v_\varep)^*\|_{L^p(\partial\Omega)} \le C_p\, \| f\|_{L^p(\partial\Omega)}$ holds.
As a result, one may establish the existence of solutions as well as the nontangential-maximal-function estimates  in the
$L^p$ Neumann and Dirichlet problems in $\Omega$
by showing that the operators 
$(1/2)I +\mathcal{K}_{\varep, A}$ and
$-(1/2) I +\mathcal{K}^*_{\varep, A^*}$ are invertible
on $L^p_0(\partial\Omega; \br^m)$ and
$L^p(\partial\Omega;\br^m)$ (modulo possible
finite-dimensional subspaces), respectively,
and by proving that the operator norms of their inverses 
are bounded by constants independent of $\varep>0$.
This approach to the boundary value problems is often 
referred  as the method of layer potentials.

In the following sections we will show that if $\Omega$ is a bounded 
Lipschitz domain with connected boundary,  the operators 
\begin{equation}\label{operator-on-L-2}
\aligned
(1/2)I + \mathcal{K}_{\varep, A}&: L^2_0(\partial\Omega, \br^m)
\to L^2_0(\partial\Omega,\br^m),\\
-(1/2)I + \mathcal{K}_{\varep, A}&: L^2(\partial\Omega, \br^m)
\to L^2(\partial\Omega,\br^m)
\endaligned
\end{equation}
are isomorphisms, under the assumptions that
$A\in \Lambda(\mu, \lambda, \tau)$ and $A^*=A$. More importantly,
we obtain estimates 
\begin{equation}\label{inverse-operator-norm}
\aligned
\|((1/2)I +\mathcal{K}_{\varep, A})^{-1} \|_{L^2_0\to L^2_0}  &\le C,\\
\|(-(1/2)I +\mathcal{K}_{\varep, A})^{-1} \|_{L^2\to L^2}  &\le C,
\endaligned
\end{equation}
where $C$ depends only on $\mu$, $\lambda$, $\tau$, and 
the Lipschitz character of $\Omega$.

We end this section with a simple observation. 

\begin{thm}\label{injection-theorem-L-2}
Let $\Omega$ be a bounded Lipschitz domain with connected
boundary and $A\in\Lambda(\mu, \lambda, \tau)$.
Then the operators in (\ref{operator-on-L-2}) are injective.
\end{thm}

\begin{proof}
Suppose that $f\in L^2_0(\partial\Omega; \br^m)$
and $((1/2)I +\mathcal{K}_{\varep, A})(f)=0$.
Let $u_\varep =\mathcal{S}_\varep (f)$.
It follows from integration by parts that
\begin{equation}\label{energy-inside}
\int_\Omega A(x/\e)\nabla u_\e \cdot \nabla u_\e\, dx
=\int_{\partial \Omega}
\left(\frac{\partial u_\varep}{\partial\nu_\varep}\right)_+ \cdot u_\varep \, d\sigma.
\end{equation}
Since $\big(\frac{\partial u_\varep}{\partial\nu_\varep}\big)_+=((1/2)I +\mathcal{K}_{\varep, A}) (f)=0$ on
$\partial\Omega$, 
we may deduce from (\ref{energy-inside}) that $\nabla u_\varep =0$ in $\Omega$ and hence
$u_\varep=b$ is constant in $\Omega$.
Recall that for $d\ge 3$,
we have $|\Gamma_\varep (x, y)|\le C\, |x-y|^{2-d}$ and $|\nabla_x \Gamma_\varep (x,y)|
+|\nabla_y \Gamma_\varep (x,y)|\le C\, |x-y|^{1-d}$.
It follows that $u_\varep (x)=O(|x|^{2-d})$ and $\nabla u_\varep (x) =O(|x|^{1-d})$, as $|x|\to\infty$.
As a result, we may use integration by parts in $\Omega_-$ to obtain
\begin{equation}\label{energy-outside}
\int_{\Omega_-} A(x/\e)\nabla u_\e \cdot \nabla u_\e \, dx
=-\int_{\partial \Omega}
\left(\frac{\partial u_\varep}{\partial\nu_\varep}\right)_- \cdot u_\varep \, d\sigma,
\end{equation}
where $\Omega_-=\br^d\setminus \overline{\Omega}$.
Note that
$$
\aligned
\int_{\partial \Omega}
\left(\frac{\partial u_\varep}{\partial\nu_\varep}\right)_- \cdot u_\varep \, d\sigma
& =b\cdot \int_{\partial\Omega} \left(\frac{\partial u_\varep}{\partial\nu_\varep}\right)_- d\sigma\\
&=-b\cdot \int_{\partial\Omega} f\, d\sigma =0,
\endaligned
$$
where we have used the jump relation (\ref{jump-relation}) and $\int_{\partial\Omega} f\, d\sigma =0$.
In view of (\ref{energy-outside}) this implies that $\nabla u_\varep=0$ in $\Omega_-$.
Since $\Omega_-$ is connected and $u_\varep(x)=O(|x|^{2-d})$ as $|x|\to \infty$,
 $u_\varep$ must be zero in $\Omega_-$.
It follows that $\big(\frac{\partial u_\varep}{\partial\nu_\varep}\big)_-=0$ on $\partial\Omega$.
By the jump relation we obtain $f=0$.
Thus we have proved that $(1/2)I +\mathcal{K}_{\varep, A}$ is injective on $L^2_0(\partial\Omega; \br^m)$.
That $-(1/2)I +\mathcal{K}_{\varep, A}$ is injective on $L^2(\partial\Omega; \br^m)$
may be proved in a similar manner.
\end{proof}

%
%
%
%
%
%
%

\section{Laplace's equation}\label{section-5.5}

In this section we establish the estimates in (\ref{inverse-operator-norm})
and solve the $L^2$ Dirichlet, Neumann, and regularity problems in Lipschitz domains
in the case $\mathcal{L}_\varep =-\Delta$.
This not only illustrates the crucial role of Rellich identities  in the study of
boundary value problems in nonsmooth domains
in the simplest setting,
but the results of this section will be used to handle the operator $\mathcal{L}_\varep$
in the general case.
Furthermore, the argument presented in this section
extends readily to the case of second order elliptic systems with constant coefficients satisfying (\ref{s-ellipticity}).

Throughout this section we will assume that $\Omega$ is a bounded Lipschitz domain in
$\br^d$, $d\ge 2$,
with connected boundary.
By rescaling we may also assume diam$(\Omega)=1$.

\begin{lemma}\label{Rellich-lemma-5.5-1}
Suppose that $\Delta u\in L^2(\Omega)$ and
$(\nabla u)^*\in L^2(\partial\Omega)$.
Also assume that $\nabla u$ has nontangential limit a.e. on $\partial\Omega$.
Then
\begin{equation}\label{Rellich-identity-5.5-1}
\aligned
\int_{\partial\Omega}
h_i n_i |\nabla u|^2\, d\sigma
=& 2 \int_{\partial\Omega} h_i \frac{\partial u}{\partial x_i}\cdot  \frac{\partial u}{\partial n}\, d\sigma
+\int_\Omega \text{\rm div} (h) |\nabla u|^2\, dx\\
& \qquad
-2\int_\Omega \frac{\partial h_i}{\partial x_j} \cdot \frac{\partial u}{\partial x_i}\cdot
\frac{\partial u}{\partial x_j}\, dx
-2\int_\Omega h_i \frac{\partial u}{\partial x_i} \, \Delta u\, dx,
\endaligned
\end{equation}
where $h=(h_1, \dots, h_d) \in C_0^1(\br^d; \br^d)$, $n$ denotes the outward unit normal to $\partial\Omega$ and
$\frac{\partial u}{\partial n} =\nabla u\cdot n$.
\end{lemma}

\begin{proof}
We begin  by choosing a sequence of smooth domains $\{ \Omega_\ell\}$
so that $\Omega_\ell \uparrow \Omega$.
By the divergence theorem we have
\begin{equation}\label{5.5.1-1}
\aligned
\int_{\Omega_\ell} h_i n_i |\nabla u|^2\, d\sigma
=& \int_{\Omega\ell} \frac{\partial}{\partial x_i} 
\left\{ h_i \frac{\partial u}{\partial x_j}\frac{\partial u}{\partial x_j}\right\}\, dx\\
=& \int_{\Omega_\ell}
\text{\rm div}(h) |\nabla u|^2\, dx
+2\int_{\Omega_\ell} h_i \frac{\partial^2 u}{\partial x_i\partial x_j}\cdot 
\frac{\partial u}{\partial x_j}\, dx.
\endaligned
\end{equation}
Using integration by parts, we see that the last term in (\ref{5.5.1-1}) equals to
$$
-2\int_{\Omega_\ell} h_i \frac{\partial u}{\partial x_i} \, \Delta u\, dx
-2\int_{\Omega_\ell}
\frac{\partial h_i}{\partial x_j}\cdot \frac{\partial u}{\partial x_i}\cdot
\frac{\partial u}{\partial x_j}\, dx
+2\int_{\partial\Omega_\ell}
h_i \frac{\partial u}{\partial x_i} \cdot \frac{\partial u}{\partial n}\, d\sigma.
$$
This gives the identity (\ref{Rellich-identity-5.5-1}), but with $\Omega$ replaced by
$\Omega_\ell$.
Finally we let $\ell\to \infty$.
Since $(\nabla u)^*\in L^2(\partial\Omega)$ and $\nabla u$ has nontangential limit a.e. on $\partial\Omega$,
the identity for $\Omega$ follows by the Lebesgue's dominated convergence theorem.
\end{proof}

\begin{lemma}\label{Rellich-lemma-5.5-2}
Under the same assumptions as in Lemma \ref{Rellich-lemma-5.5-1}, we have
\begin{equation}\label{Rellich-identity-5.5-2}
\aligned
\int_{\partial\Omega}
h_i n_i |\nabla u|^2\, d\sigma
=& 2\int_{\partial\Omega} h_i\frac{\partial u}{\partial x_j} \left\{ n_i\frac{\partial u}{\partial x_j}
-n_j\frac{\partial u}{\partial x_i} \right\}\, d\sigma\\
& \qquad
-\int_\Omega \text{\rm div} (h) |\nabla u|^2\, dx
+2\int_\Omega \frac{\partial h_i}{\partial x_j} \cdot \frac{\partial u}{\partial x_i}\cdot
\frac{\partial u}{\partial x_j}\, dx\\
&\qquad-2\int_\Omega h_i \frac{\partial u}{\partial x_i} \, \Delta u\, dx.
\endaligned
\end{equation}
\end{lemma}

\begin{proof}
Let $I$ and $J$ denote the left and right hand sides of (\ref{Rellich-identity-5.5-1}) respectively.
Identity (\ref{Rellich-identity-5.5-2})
follows from (\ref{Rellich-identity-5.5-1}) by writing
$J$ as $2I-J$.
\end{proof}

Formulas (\ref{Rellich-identity-5.5-1}) and (\ref{Rellich-identity-5.5-2})
are referred to as the Rellich identities for Laplace's equation.
They also hold on $\Omega_-=\rd\setminus \overline{\Omega}$ under the assumption that
$\Delta u\in L^2(\Omega_-)$,
$(\nabla u)^*\in L^2(\partial\Omega)$ and
$\nabla u$ has nontanegntial limit a.e. on $\partial\Omega$.
The use of divergence theorem on the unbounded domain $\Omega_-$
is justified, as the vector field $h$ has compact support.
We should point out that in the case of $\Omega_-$, all solid integrals in (\ref{Rellich-identity-5.5-1})
and (\ref{Rellich-identity-5.5-2}) need to change signs.
This is because $n$ always denotes the outward unit normal to $\partial\Omega$.

The following lemma will be used to handle solid integrals in (\ref{Rellich-identity-5.5-1})
and (\ref{Rellich-identity-5.5-2}).
Recall that $(\nabla u)^*\in L^1(\partial\Omega)$
implies that $u$ has nontangential limit a.e.\,on $\partial\Omega$.

\begin{lemma}\label{solid-integral-lemma}

\begin{enumerate}

\item

 Suppose that $\Delta u=0$ in $\Omega$ and $(\nabla u)^*\in L^2(\partial\Omega)$.
Also assume that $\nabla u$ has nontangential limit a.e.\,on $\partial\Omega$. Then
\begin{equation}\label{solid-integral-inside}
\int_\Omega |\nabla u|^2 \, dx
\le C\,  \big\|\frac{\partial u}{\partial n}\big\|_{L^2(\partial\Omega)}
\|\nabla_{\rm tan} u\|_{L^2(\partial\Omega)}.
\end{equation}

\item

Suppose that $\Delta u=0$ in $\Omega_-$ and $(\nabla u)^*\in L^2(\partial\Omega)$.
Also assume that $\nabla u$ has nontangential limit a.e.\,on $\partial \Omega$ and
that as $|x|\to \infty$, $|u(x)|=O(|x|^{2-d})$ for $d\ge 3$
and $|u(x)|=o(1)$ for $d=2$. Then
\begin{equation}\label{solid-integral-outside}
\int_{\Omega_-}
|\nabla u|^2\, dx
\le C\, \big\|\frac{\partial u}{\partial n}\big\|_{L^2(\partial\Omega)}
\|\nabla_{\rm tan} u\|_{L^2(\partial\Omega)}
+\big|\int_{\partial\Omega} \frac{\partial u}{\partial n}\, d\sigma\big|\,
\big| \average_{\partial\Omega} u  \big|.
\end{equation}
\end{enumerate}

\end{lemma}

\begin{proof}
For part (1) we use the divergence theorem and $\int_{\partial \Omega} u\, d\sigma =0$
to obtain
$$
\int_\Omega |\nabla u|^2\, dx 
 =\int_{\partial\Omega} \frac{\partial u}{\partial n} u\, d\sigma
 =\int_{\partial\Omega} \frac{\partial u}{\partial n} \left\{ u-\average_{\partial \Omega} u\right\}\, d\sigma.
$$
Estimate (\ref{solid-integral-inside}) follows from this by applying
the Cauchy inequality and then the Poincar\'e inequality.
Part (2) may be proved in a similar manner.
By interior estimates for harmonic functions and the decay assumption 
at infinity, we see that
 $|\nabla u(x)|=O(|x|^{1-d})$ as $|x|\to \infty$.
 Hence,
$$
\int_{|x|=R} |\nabla u |\, |u|\, d\sigma \to 0 \quad \text{ as } R\to \infty.
$$
 This is enough to justify the integration by parts in $\Omega_-$.
\end{proof}

\begin{thm}\label{Rellich-estimate-inside-theorem}
Suppose that $\Delta u=0$ in $\Omega$ and
$(\nabla u)^*\in L^2(\partial\Omega)$.
Also assume that $\nabla u$ has nontangential limit a.e.\,on $\partial\Omega$.
Then
\begin{equation}\label{Rellich-estimate-1}
\|\nabla u\|_{L^2(\partial\Omega)}
 \le C\, \big\|\frac{\partial u}{\partial n}\big\|_{L^2(\partial\Omega)} \quad 
 \text{ and } \quad
\|\nabla u\|_{L^2(\partial\Omega)}
 \le C\, \|\nabla_{\rm tan} u\|_{L^2(\partial\Omega)},
\end{equation}
where $C$ depends only on the Lipschitz character of $\Omega$.
\end{thm}

\begin{proof}
Let $h\in C_0^\infty(\br^d; \br^d)$ be a vector field such that
$h\cdot n\ge c_0>0$ on $\partial\Omega$.
It follows from (\ref{Rellich-identity-5.5-1}) that
$$
\aligned
\int_{\partial\Omega}
|\nabla u|^2\, d\sigma 
& \le C\, \|\nabla u\|_{L^2(\partial\Omega)}
\big\|\frac{\partial u}{\partial n}\big\|_{L^2(\partial\Omega)}
+C \int_\Omega |\nabla u|^2\, dx\\
& 
\le C\, \|\nabla u\|_{L^2(\partial\Omega)}
\big\|\frac{\partial u}{\partial n}\big\|_{L^2(\partial\Omega)}
+C\,\big\|\frac{\partial u}{\partial n}\big\|_{L^2(\partial\Omega)}
\|\nabla_{\rm tan} u\|_{L^2(\partial\Omega)},
\endaligned
$$
where we have used (\ref{solid-integral-inside}) for the second inequality.
Since $|\nabla_{\rm tan} u|\le |\nabla u|$, this gives the first estimate in (\ref{Rellich-estimate-1}).
The second estimate in (\ref{Rellich-estimate-1})
follows from  the formula (\ref{Rellich-identity-5.5-2}) and (\ref{solid-integral-inside})
in the same manner.
\end{proof}

Theorem \ref{Rellich-estimate-inside-theorem} shows that for (suitable) harmonic functions in a Lipschitz domain $\Omega$,
the $L^2$ norms of the normal derivative $\frac{\partial u}{\partial n}$
and the tangential derivatives $\nabla_{\rm tan} u$ on $\partial\Omega$ are equivalent.
This is also true for harmonic functions in $\Omega_-$, modulo some
linear functionals.
We leave the proof of the following theorem to the reader.

\begin{thm}\label{Rellich-estimate-outside-theorem}
Suppose that $\Delta u=0$ in $\Omega_-$ and
$(\nabla u)^*\in L^2(\partial\Omega)$.
Also assume that  $\nabla u$ has nontangential limit a.e.\,on $\partial\Omega$ and that
as $|x|\to \infty$, $|u(x)|=O(|x|^{2-d})$ for $d\ge 3$
and $|u(x)|=o(1)$ for $d=2$.
Then
\begin{equation}\label{Rellich-estimate-outside}
\aligned
\|\nabla u\|^2_{L^2(\partial\Omega)}
 & \le C\big \|\frac{\partial u}{\partial n}\big\|^2_{L^2(\partial\Omega)}
+C \, \big|\int_{\partial\Omega} \frac{\partial u}{\partial n}\, d\sigma \big|\, 
\big|\average_{\partial\Omega} u \big|,\\
\|\nabla u\|^2_{L^2(\partial\Omega)}
 & \le C\,  \|\nabla_{\rm tan} u\|^2_{L^2(\partial\Omega)}
+C \big|\int_{\partial\Omega} \frac{\partial u}{\partial n}\, d\sigma \big|\,
\big|\average_{\partial\Omega} u \big|,
\endaligned
\end{equation}
where $C$ depends only on the Lipschitz character of $\Omega$.
\end{thm}

Let $\omega_d$ denote the surface area of the unit sphere in
$\br^d$.
The fundamental solution for $\mathcal{L}=-\Delta$ with pole at the origin is given by
\begin{equation}\label{fundamental-solution-Laplacian}
\left\{
\aligned
\Gamma (x) &=\frac{1}{(d-2) w_d |x|^{d-2}}&\quad& \text{ for } d\ge 3,\\
\Gamma(x)&=-\frac{1}{2\pi} \ln |x|  & \quad& \text{ for } d=2.
\endaligned
\right.
\end{equation}
For $f\in L^p(\partial\Omega)$ with $1<p<\infty$, let
\begin{equation}\label{single-layer-potential-Laplacian}
u(x)=\mathcal{S}(f)(x)=\int_{\partial\Omega} \Gamma(x-y) f(y)\, d\sigma (y)
\end{equation}
 be the single layer potential for
$\mathcal{L}=-\Delta$. It follows from Section \ref{section-7.4} that
$$
\|(\nabla u)^*\|_{L^p(\partial\Omega)} \le C_p\,  \| f\|_{L^p(\partial\Omega)}
$$
 and 
\begin{equation}
\left(\frac{\partial u}{\partial n}\right)_+
=\big((1/2)I +\mathcal{K}\big) f
\quad \text{ and }\quad
\left(\frac{\partial u}{\partial n}\right)_-
=\big(-(1/2)I +\mathcal{K}\big) f
\end{equation}
on $\partial\Omega$,
where $\mathcal{K}$ is a bounded singular integral
operator on $L^p(\partial\Omega)$.
Also recall that $(\nabla_{\rm tan} u)_+
=(\nabla_{\rm tan} u)_-$ on $\partial\Omega$.
As indicated at the beginning of this section, we will show
that $(1/2)+\mathcal{K}$ and $-(1/2)I+\mathcal{K}$ are isomorphisms on $L^2_0(\partial\Omega)$
and $L^2(\partial\Omega)$, respectively.

\begin{lemma}\label{invertibility-lemma-5.5}
Let $f\in L^2(\partial\Omega)$. Then
\begin{equation}\label{invertibility-estimate-5.5}
\aligned
&\|f\|_{L^2(\partial\Omega)}
\le C \left\{ \|((1/2)I+\mathcal{K}) f\|_{L^2(\partial\Omega)}
+\big|\int_{\partial\Omega} f\, d\sigma \big|\right\},\\
& \| f\|_{L^2(\partial\Omega)}
\le C \, \| (-(1/2)I +\mathcal{K}) f\|_{L^2(\partial\Omega)},
\endaligned
\end{equation}
where $C$ depends only on the Lipschitz character of $\Omega$.
\end{lemma}

\begin{proof}
We first consider the case $f\in L^2_0(\partial\Omega)$.
Let $u=\mathcal{S}(f)$.
The additional assumption on $f$ implies that
$u(x)=O(|x|^{1-d})$ as $|x|\to\infty$,
which assures that estimates (\ref{Rellich-estimate-1}) and (\ref{Rellich-estimate-outside})
hold for all $d\ge 2$. By the jump relation (\ref{jump-relation}),
it also implies that the mean value of $\big(\frac{\partial u}{\partial n}\big)_-$
is zero.
Thus we may deduce from (\ref{Rellich-estimate-outside}) that
$$
\aligned
\|(\nabla u)_-\|_{L^2(\partial \Omega)}
& \le C\, \|(\nabla_{\rm tan} u)_-\|_{L^2(\partial\Omega)}
=C\,\|(\nabla_{\rm tan} u)_+\|_{L^2(\partial\Omega)}\\
& \le C \, \big\|\left(\frac{\partial u}{\partial n}\right)_+\big\|_{L^2(\partial \Omega)}.
\endaligned
$$
By the jump relation it follows that
$$
\aligned
\| f\|_{L^2(\partial\Omega)}
 & \le \big\|\left(\frac{\partial u}{\partial n}\right)_+\big\|_{L^2(\partial\Omega)}
+\big\|\left(\frac{\partial u}{\partial n}\right)_-\big\|_{L^2(\partial\Omega)}\\
& \le C\, \big\|\left(\frac{\partial u}{\partial n}\right)_+\big\|_{L^2(\partial\Omega)}\\
& =C\, \| ((1/2)I +\mathcal{K}) f\|_{L^2(\partial\Omega)}.
\endaligned
$$
This gives the first inequality in (\ref{invertibility-estimate-5.5})
for the case $f\in L^2_0(\partial\Omega)$.
The general case follows by considering $f-E$, where $E$ is the mean value of $f$ on $\partial\Omega$.

Similarly, to establish the second inequality in (\ref{invertibility-estimate-5.5}),
we use (\ref{Rellich-estimate-1}) and (\ref{Rellich-estimate-outside}) to obtain
$$
\aligned
\|(\nabla u)_+\|_{L^2(\partial\Omega)}
& \le C\, \|(\nabla_{\rm tan} u)_+\|_{L^2(\partial\Omega)}
=C\,  \|(\nabla_{\rm tan} u)_-\|_{L^2(\partial\Omega)}\\
&\le C\, \big\|\left(\frac{\partial u}{\partial n}\right)_-\big\|_{L^2(\partial\Omega)},
\endaligned
$$
where we also used $\int_{\partial \Omega} (\partial u/\partial n)_-\, d\sigma=0$.
By the jump relation this gives the second inequality in (\ref{invertibility-estimate-5.5})
for the case $f\in L^2_0(\partial\Omega)$.
As before, it follows that for any $f\in L^2(\partial\Omega)$,
$$
\| f\|_{L^2(\partial\Omega)}
\le C\, \| (-(1/2)I +\mathcal{K}) f\|_{L^2(\partial\Omega)}
+ C |E|,
$$
where $E$ is the mean value of $f$ on $\partial\Omega$.
To finish the proof we simply observe that by the jump relation, $E$ is also
the mean value of $(-(1/2)I +\mathcal{K}) f$ on $\partial\Omega$.
Hence, 
$$|E|\le C\, \| (-(1/2)I +\mathcal{K}) f\|_{L^2(\partial\Omega)},
$$
which completes the proof.
\end{proof}

It follows readily from Lemma \ref{invertibility-lemma-5.5} that $(1/2)I +\mathcal{K}$
and $-(1/2)I +\mathcal{K}$ are injective on $L^2_0(\partial\Omega)$
and $L^2(\partial\Omega)$, respectively.
The lemma also implies that the ranges of the operators are closed in $L^2(\partial\Omega)$.
This is a consequence of the next theorem, whose proof is left to the reader.

\begin{thm}\label{closeness-theorem}
Let $T: X\to Y$ be a bounded linear operator, where $X,Y$ are normed linear spaces.
Suppose that 

\begin{enumerate}

\item

 $X$ is Banach; 
 
 \item
 
  the dimension of the null space $\{ f\in X: T(f)=0\}$
is finite; 

\item

 for any $f\in X$,
$$
\| f\|_X \le C \, \|Tf\|_Y + C\sum_{j=1}^\ell\|S_j f\|_{Y_j},
$$
where $S_j : X\to Y_j$, $j=1,\dots, \ell$ are compact operators.

\end{enumerate}

Then the range of $T$ is closed in $Y$.
\end{thm}

We will use a continuity method to show that 
$\pm (1/2)I +\mathcal{K}$ are surjective on $L^2_0(\partial\Omega)$
and $L^2(\partial\Omega)$, respectively.
To this end we consider a family of operators
\begin{equation}\label{definition-of-T-lambda}
T_\lambda =\lambda I +\mathcal{K},
\end{equation}
where $\lambda\in \br$.

\begin{lemma}\label{T-lambda-lemma-1}
If $|\lambda|>1/2$,
 the operator $T_\lambda: L^2(\partial\Omega)\to L^2(\partial\Omega)$
is injective.
\end{lemma}

\begin{proof}
Let $f\in  L^2(\partial\Omega)$ and $u=\mathcal{S}(f)$.
Suppose that $T_\lambda (f)=0$.
Then
$$
\left(\frac{\partial u}{\partial n}\right)_+=((1/2) -\lambda) f \quad \text{ and }
\left(\frac{\partial u}{\partial n}\right)_+=(-(1/2) -\lambda) f.
$$
Note that the first equation above implies $f\in L^2_0(\partial\Omega)$.
It follows that $u(x)=O(|x|^{1-d})$ as $|x|\to \infty$.
By the divergence theorem,
\begin{equation}\label{5.5.8-1}
\int_{\Omega_\pm}
|\nabla u|^2\, dx
=\pm \int_{\partial\Omega}
u \left(\frac{\partial u}{\partial n}\right)_\pm \, d\sigma
=\left(\frac12 \mp \lambda\right) \int_{\partial\Omega} u f\, d\sigma.
\end{equation}
It follows that
$$
\aligned
\int_{\Omega_+} |\nabla u|^2\, dx +\int_{\Omega_-} |\nabla u|^2\, dx
& =\int_{\partial\Omega} uf \, d\sigma,\\
\int_{\Omega_+} |\nabla u|^2\, dx -\int_{\Omega_-} |\nabla u|^2\, dx
& =-2\lambda \int_{\partial\Omega} uf \, d\sigma.
\endaligned
$$
This implies that
$$
2|\lambda|\,  \big| \int_{\partial\Omega} uf \, d\sigma |
\le \big| \int_{\partial\Omega} uf \, d\sigma |.
$$
Since $2|\lambda|>1$, we obtain $\int_{\partial\Omega} uf \, d\sigma =0$.
Hence, by (\ref{5.5.8-1}), $ |\nabla u|=0$ in $\Omega_\pm$.
Consequently, by the jump relation (\ref{jump-relation}), we get $f=0$.
\end{proof}

Let $h=(h_1, \dots, h_d) \in C_0^1(\br^d;\br^d)$.
It follows from  the trace formula (\ref{nontangential-limit}) that if $u=\mathcal{S} (f)$,
$$
h_i \left(\frac{\partial u}{\partial x_i}\right)_\pm 
=\pm \frac12 \langle h,n\rangle  +\mathcal{K}_h (f),
$$
where 
$$
\mathcal{K}_h (f)(P)=\text{\rm p.v.}
\int_{\partial\Omega} \frac{\langle y-P, h(P)\rangle }{\omega_d |Q-y|^d} f (y)\, d\sigma (y).
$$
Observe that
$$
(\mathcal{K}_h + \mathcal{K}^*_h )(f) (P)
=\text{\rm p.v.}
\int_{\partial\Omega}
\frac{\langle y-P, h(P)-h(y)\rangle}{\omega_d |P-y|^d} f(y)\, d\sigma (y).
$$
Since $|h(P)-h(y)|\le C\, |P-y|$, the integral kernel of $\mathcal{K}_h +\mathcal{K}^*_h$ is
bounded by $C |y-P|^{2-d}$.
This implies that the operator $\mathcal{K}_h +\mathcal{K}^*_h$ is compact
on $L^p(\partial\Omega)$ for $1\le p\le \infty$.

\begin{lemma}\label{lemma-5.5.9}
Let $|\lambda|>1/2$.
Then for any $f\in L^2(\partial\Omega)$,
$$
\| f\|_{L^2(\partial\Omega)}
\le C_\lambda \left\{
\|T_\lambda f\|_{L^2(\partial\Omega)}
+\|(\mathcal{K}_h +\mathcal{K}_h^*) f\|_{L^2(\partial\Omega)}
+\|\mathcal{S} (f)\|_{L^2(\partial\Omega)}\right\},
$$
where $C_\lambda$ depends on $\lambda$.
\end{lemma}

\begin{proof}
Let $f\in L^2(\partial\Omega)$ and $u=\mathcal{S}(f)$.
Then
$$
\aligned
|(\nabla u)_+|^2
&\ge |\left(\frac{\partial u}{\partial n}\right)_+|^2
=|(\frac12I +\mathcal{K}) f|^2
=|(\frac12-\lambda) f + T_\lambda(f)|^2\\
&\ge (\frac12 -\lambda)^2 |f|^2 +(1-2\lambda) f \, T_\lambda (f).
\endaligned
$$
This, together with the Rellich identity (\ref{Rellich-identity-5.5-1}),
gives
$$
\aligned
\big(\frac12 & -\lambda\big)\int_{\partial\Omega} \langle h,n\rangle |f|^2\, d\sigma\\
&\le  (2\lambda-1) \int_{\partial\Omega} \langle h,n\rangle  f\, T_\lambda (f)\, d\sigma
+C \int_\Omega |\nabla u|^2\, dx\\
&\qquad
+2\int_{\partial\Omega}
\left\{ \frac12 \langle h,n\rangle f +\mathcal{K}_h (f)\right\}
\left\{ (\frac12-\lambda) f+ T_\lambda (f)\right\}\, d\sigma.
\endaligned
$$
It follows that
$$
\aligned
& (\lambda^2 -\frac14)\int_{\partial\Omega}
\langle h,n\rangle  |f|^2\, d\sigma\\
&\ \ \le C_\lambda \left\{
\| f\|_{L^2(\partial\Omega)} \| T_\lambda f\|_{L^2(\partial\Omega)}
+\big|\int_{\partial\Omega} \mathcal{K}_h (f) \cdot f\, d\sigma \big|
+\int_\Omega |\nabla u|^2\, dx \right\}.
\endaligned
$$
The desired estimate follows from this and the observation that
$$
2\int_{\partial\Omega} \mathcal{K}_h (f)\cdot f\, d\sigma
=\int_{\partial\Omega} (\mathcal{K}_h +\mathcal{K}_h^*) (f)\cdot f \, d\sigma
$$
and
$$
\int_\Omega |\nabla u|^2\, dx
=\int_{\partial\Omega} \frac{\partial u}{\partial n} \, u\, d\sigma
\le C \| f\|_{L^2(\partial\Omega)} \|\mathcal{S}(f)\|_{L^2(\partial\Omega)}.
$$
\end{proof}

\begin{lemma}\label{lemma-5.5-9}
If $|\lambda|\ge 1/2$, the operator
$T_\lambda: L^2(\partial\Omega)\to L^2(\partial\Omega)$
has a closed range.
\end{lemma}

\begin{proof}
The case $|\lambda|=1/2$ follows from Theorem \ref{closeness-theorem}
and Lemma \ref{invertibility-lemma-5.5}. 
To deal with the case $|\lambda|>1/2$,
we recall that $\mathcal{S}: L^2(\partial\Omega) \to W^{1,2}(\partial\Omega)$
is bounded.
Since the imbedding $W^{1,2}(\partial\Omega)\subset L^2(\partial\Omega)$
is compact, it follows that the operator $\mathcal{S}$ is compact on $L^2(\partial\Omega)$.
Since $\mathcal{K}_h +\mathcal{K}^*_h$ is also compact on $L^2(\partial\Omega)$,
in view of Lemma \ref{lemma-5.5.9} and Theorem \ref{closeness-theorem},
we may conclude that $T_\lambda: L^2(\partial\Omega)\to L^2(\partial\Omega)$
has a closed range.
\end{proof}

\begin{lemma}\label{continuity-lemma}
Let $\{ E(\lambda): \lambda\in I\}$ be a continuous family of bounded linear operators
from $X\to Y$, where $X,Y$ are Banach spaces and $I\subset \mathbb{C}$ is connected.
Suppose that (1) for each $\lambda\in I$, $E(\lambda)$ is injective and its range is closed;
(2) $E(\lambda_0)$ is an isomorphism for some $\lambda_0\in I$.
Then $E(\lambda)$ is an isomorphism for all $\lambda\in I$.
\end{lemma}

\begin{proof}
Let 
$$
J=\big\{ \lambda\in I: \, E(\lambda) \text{ is an isomorphism}\big\}.
$$
Since $E(\lambda)$ is continuous,
it is easy to see that $J\neq \emptyset$ is relative open in $I$.
Since $I$ is connected, we only need to show that $J$
is also relative closed in $I$.

To this end, let
$$
C(\lambda)=\sup \left\{ \frac{\|g\|}{\|E(\lambda)g\|}: \ g\in X \text{ and } g\neq 0\right\}.
$$
By the uniform boundedness theorem, $C(\lambda)$ is finite for all $\lambda\in I$.
Suppose that $\lambda_j\in J$ and $\lambda_j \to \lambda\in I$.
We will show that $\lambda\in J$.
Since
$$
\aligned
\|g\| & \le C(\lambda) \|E(\lambda) g\|
\le C(\lambda) \big\{
\|E(\lambda)g -E(\lambda_j)g\| +\|E(\lambda_j)g\|\big\}\\
& \le C(\lambda)
\|E(\lambda_j)-E(\lambda)\| \| g\|
+C(\lambda) \|E(\lambda_j) g\|,
\endaligned
$$
we obtain
$$
(1-C(\lambda) \|E(\lambda)-E(\lambda_j)\|) \| g\|
\le C(\lambda) \|E(\lambda_j) g\|.
$$
It follows that if $1-C(\lambda)\|E(\lambda_j)-E(\lambda)\|<1$,
$$
C(\lambda_j) \le \frac{C(\lambda_0)}{1-C(\lambda)\|E(\lambda_j)-E(\lambda)||}.
$$
This shows that $\{ C(\lambda_j)\}$ is bounded in $\br$.
Now let $f\in Y$.
Since $E(\lambda_j)$ is an isomorphism,
there exists $g_j\in X$ such that $E(\lambda_j) =f$.
Note that $\|g_j\|\le C(\lambda_j)\| f\|$ and thus $\{ \| g_j\|\}$ is bounded.
Also observe that
$$
\aligned
\|g_i -g_j\| & \le C(\lambda_j) \|E(\lambda_j)g_i -E(\lambda_j) g_j\|\\
&\le
C(\lambda_j) \|E(\lambda_j)g_i -E(\lambda_i)g_i\|\\
&\le C(\lambda)
\|E(\lambda_j)-E(\lambda_i)\| \| g_i\|.
\endaligned
$$
Hence, $\{ g_j\}$ is a Cauchy sequence in $X$.
Suppose that $g_j\to g\in X$.
It is not hard to see that $E(\lambda)g=f$.
This shows that $E(\lambda)$ is surjective and thus an isomorphism.
\end{proof}

We are now ready to prove the main theorem of this section.

\begin{thm}\label{Laplace-theorem}
Let $\Omega$ be a bounded Lipschitz domain in $\br^d$, $d\ge 2$, with
connected boundary.
Then operators $(1/2)I +\mathcal{K}: L^2_0(\partial\Omega)
\to L^2_0(\partial\Omega)$ and
$-(1/2)I +\mathcal{K}: L^2(\partial\Omega)
\to L^2(\partial\Omega)$ are isomorphisms.
Moreover,
\begin{equation}\label{Laplace-estimate}
\aligned
& \| f\|_{L^2(\partial\Omega)}
\le C \| ((1/2)I +\mathcal{K}) f\|_{L^2(\partial\Omega)}
\quad \ \ \text{ for } f\in L^2_0(\partial\Omega),\\
& \| f\|_{L^2(\partial\Omega)}
\le C \| (-(1/2)I +\mathcal{K}) f\|_{L^2(\partial\Omega)}
\quad \text{ for } f\in L^2(\partial\Omega),
\endaligned
\end{equation}
where $C$ depends only on the Lipschitz character of $\Omega$.
\end{thm}

\begin{proof}
By rescaling we may assume diam$(\Omega)=1$.
Note that the estimates in  (\ref{Laplace-estimate}) 
follow from  Lemma \ref{invertibility-lemma-5.5}.
To show $(1/2)+\mathcal{K}$ is an isomorphism on $L^2_0(\partial\Omega)$,
we apply Lemma \ref{continuity-lemma} to 
$E(\lambda)=\lambda I +\mathcal{K}$ for
$\lambda\in I =[1/2, \infty)$.
Observe that $E(\lambda)$ is a bounded operator on $L^2_0(\partial\Omega)$
for any $\lambda\in \br$.
Clearly, $E(\lambda)$ is an isomorphism on $L^2_0(\partial\Omega)$
if $\lambda$ is greater than the operator norm of $\mathcal{K}$
on $L^2_0(\partial\Omega)$.
By Lemma \ref{T-lambda-lemma-1} (for $\lambda>1/2$)
and Lemma \ref{invertibility-lemma-5.5} (for $\lambda=1/2$),
$\lambda  I+\mathcal{K}$ is injective for all $\lambda\in I$.
That the range of $\lambda I +\mathcal{K}$ is closed
was proved in Lemma \ref{lemma-5.5-9}.
It now follows from Lemma \ref{continuity-lemma}
that $\lambda I +\mathcal{K}$ is an isomorphism on
$L^2_0(\partial\Omega)$
for all $\lambda\ge 1/2$.

To show $-(1/2)I +\mathcal{K}$ is an isomorphism on $L^2(\partial\Omega)$,
we consider $E(\lambda)=\lambda I +\mathcal{K}$
on $L^2(\partial\Omega)$ for $\lambda\in (-\infty, -1/2]$.
The argument is similar to that for the case $(1/2)I+\mathcal{K}$.
The details are left to the reader.
\end{proof}

As we mentioned before, the invertibility of $\pm (1/2) I +\mathcal{K}$ on $L^2$
leads to the existence of solutions to the $L^2$
Neumann and Dirichlet problems.

\begin{thm}[$L^2$ Neumann problem]\label{Laplace-L-2-Neumann-theorem}
Let $\Omega$ be a bounded Lipschitz domain in $\br^d$, $d\ge 2$,
with connected boundary.
Then, given any $g\in L^2_0(\partial\Omega)$,
there exists a unique (up to constants) harmonic functions in 
$\Omega$ such that $\frac{\partial u}{\partial n}=g$ n.t.\, on $\partial\Omega$.
Moreover, the solution may be represented by a single layer potential $\mathcal{S}(h)$
with $\|h\|_{L^2(\partial\Omega)}\le C\, \| g\|_{L^2(\partial\Omega)}$ and
satisfies the estimate
$\|(\nabla u)^*\|_{L^2(\partial\Omega)}
\le C \, \| g\|_{L^2(\partial\Omega)}$,
where $C$ depends only on the Lipschitz character of $\Omega$.
\end{thm}

\begin{proof}
Let $g\in L^2_0(\partial\Omega)$.
By Theorem \ref{Laplace-theorem} there exits $h\in L^2_0(\partial\Omega)$ such that
$((1/2)I+\mathcal{K})h=g$ on $\partial\Omega$ and 
$\|h\|_{L^2(\partial\Omega)} \le C\, \| g\|_{L^2(\partial\Omega)}$.
Then $u=\mathcal{S}(h)$ is a solution to the $L^2$ Neumann problem
for $\Delta u=0$ in $\Omega$ with boundary data $g$.
Moreover, 
$$
\|(\nabla u)^*\|_{L^2(\partial\Omega)}\le C\,  \| h\|_{L^2(\partial\Omega)}
\le C \, \|g\|_{L^2(\partial\Omega)},
$$
where $C$ depends only on the Lipschitz character of $\Omega$.

The uniqueness follows from the Green's identity.
Indeed, suppose that $\Delta u=0$ in $\Omega$, $(\nabla u)^*\in L^2(\partial\Omega)$ and
$\frac{\partial u}{\partial n}=0$ n.t. on $\partial\Omega$.
Note that by Proposition \ref{control-prop},
 $(\nabla u)^* \in L^2 (\partial\Omega)$ implies that $(u)^*\in L^2(\partial\Omega)$.
Let $\Omega_j\uparrow\Omega$.
By using the Greeen's identity in $\Omega_j$ and then the dominated convergence theorem,
we may deduce that
$$
\int_\Omega |\nabla u|^2\, dx =\int_{\partial\Omega} \frac{\partial u}{\partial n}
\, u\, d\sigma=0.
$$
Hence $u$ is constant in $\Omega$.
\end{proof}

\begin{thm}[$L^2$ Dirichlet problem]\label{Laplace-L-2-Dirichlet-theorem}
Let $\Omega$ be a bounded Lipschitz domain in $\br^d$, $d\ge 2$, with connected
boundary. Then, given any $f\in L^2(\partial\Omega)$,
there exists a unique harmonic function $u$ in $\Omega$
such that $u=f$ n.t.\,on $\partial\Omega$
and $(u)^*\in L^2(\partial\Omega)$.
Moreover, the solution may be represented by a double layer potential $\mathcal{D}(h)$
with $\| h\|_{L^2(\partial\Omega)}\le C\, \| f\|_{L^2(\partial\Omega)}$ and
satisfies the estimate
$\|(u)^*\|_{L^2(\partial\Omega)}
\le C\, \| f\|_{L^2(\partial\Omega)}$, where
$C$ depends only on the Lipschitz character of $\Omega$.
\end{thm}

\begin{proof}
By Theorem \ref{Laplace-theorem}, the operator $-(1/2)I+\mathcal{K}$ is an isomorphism 
on $L^2(\partial\Omega)$.
By duality the operator $-(1/2)I+\mathcal{K}^*$ is also an isomorphism on 
$L^2(\partial\Omega)$.
Thus, given any $f\in L^2(\partial\Omega)$,
there exists $h\in L^2(\partial\Omega)$ such that
$(-(1/2)I+\mathcal{K}^*) h=f$ on $\partial\Omega$.
It follows that $u=\mathcal{D}(h)$ is a solution to the $L^2$ Dirichlet problem
in $\Omega$ with boundary data $f$.
Moreover, 
$$
\|(u)^*\|_{L^2(\partial\Omega)}
\le C\, \| h\|_{L^2(\partial\Omega)}\le C\, \| f\|_{L^2(\partial\Omega)},
$$
where $C$ depends only on the Lipschitz character of $\Omega$.

To establish the uniqueness, we will show that if $\Delta u=0$ in $\Omega$,
$(u)^*\in L^2(\partial\Omega)$ and $u=f$ n.t.\,on $\partial\Omega$, then
\begin{equation}\label{5.5.14-1}
\int_\Omega |u|^2\, dx \le C \int_{\partial\Omega} |f|^2\, d\sigma.
\end{equation}
Consequently, $f=0$ on $\partial\Omega$ impies that $u=0$ in $\Omega$.

To prove (\ref{5.5.14-1}), we let $\{\Omega_j\}$ be a sequence of smooth domains
such that $\Omega_j\uparrow \Omega$.
Let $F\in C_0^\infty(\Omega_j)$ and $w$ be the solution to the Dirichlet problem:
$\Delta w=F$ in $\Omega_j$ and $w=0$ on $\partial\Omega_j$.
Since $\Omega_j$ and $F$ are smooth, we have $w \in C^\infty(\overline{\Omega_j})$.
Note that by the Cauchy and Poincar\'e inequalities,
$$
\aligned
\int_{\Omega_j} |\nabla w|^2\, dx
&=-\int_{\Omega_j} F w\, dx
\le \| F\|_{L^2(\Omega_j)} \| w\|_{L^2(\Omega_j)}\\
& \le C\, \| F\|_{L^2(\Omega_j)}
\| \nabla w\|_{L^2(\Omega_j)},
\endaligned
$$
where $C$ does not depend on $j$.
It follows that $\|\nabla w\|_{L^2(\Omega_j)}
\le C\, \|F\|_{L^2(\Omega_j)}$.

Next we observe that by Green's identity,
\begin{equation}\label{5.5.14-3}
\int_{\Omega_j} u\, F\, dx
=\int_{\Omega_j} u \, \Delta w\, dx
=\int_{\partial\Omega_j}
u\, \frac{\partial w}{\partial n}\, d\sigma,
\end{equation}
where we have used the assumption $\Delta u=0$ in $\Omega$.
Also, since $w=0$ on $\partial\Omega_j$, we may use the Rellich identity (\ref{Rellich-identity-5.5-2}) to 
obtain
$$
\aligned
\int_{\partial\Omega_j}
|\nabla w|^2\, d\sigma
 &\le C \left\{ \int_{\Omega_j} |\nabla w|\, |F|\, dx +\int_{\Omega_j} |\nabla w|^2\, dx \right\}\\
 &\le C \int_{\Omega_j} |F|^2\, dx.
\endaligned
$$
This, together with (\ref{5.5.14-3}), yields
$$
\aligned
\big|\int_\Omega u\, F\, dx\big|
 &\le \| u\|_{L^2(\partial\Omega_j)}
\| \nabla w\|_{L^2(\partial\Omega_j)}\\
& \le C \, \| u\|_{L^2(\partial\Omega_j)}
\| F\|_{L^2(\Omega_j)}.
\endaligned
$$
It follows by duality that
\begin{equation}\label{5.5.14-5}
\int_{\Omega_j} |u|^2\, dx \le C\, \int_{\partial\Omega_j} |u|^2\, d\sigma.
\end{equation}
This gives the estimate (\ref{5.5.14-1}) by letting $j\to \infty$.
\end{proof}

Finally we consider the $L^2$ regularity problem.

\begin{thm}[$L^2$ regularity problem]\label{Laplace-L-2-regularity-theorem}
Let $\Omega$ be a bounded Lipschitz domain in $\br^d$, $d\ge 2$,
with connected boundary.
Given any $f\in W^{1,2}(\partial\Omega)$,
there exists a unique harmonic function $u$ in $\Omega$
such that $(\nabla u)^*\in L^2(\partial\Omega)$ and
$u=f$ n.t.\,on $\partial\Omega$.
Furthermore, the solution satisfies
\begin{equation}\label{Laplace-L-2-regularity-estimate}
\|(\nabla u)^*\|_{L^2(\partial\Omega)}
+|\partial\Omega|^{\frac{1}{d-1}}\| (u)^*\|_{L^2(\partial\Omega)}
\le C \, \| f\|_{W^{1,2}(\partial\Omega)},
\end{equation}
where $C$ depends only on the Lipschitz character of $\Omega$.
\end{thm}

\begin{proof}
Since $(\nabla u)^*\in L^2(\partial\Omega)$ implies $(u)^*\in L^2(\partial\Omega)$,
the uniqueness follows from that of the $L^2$ Dirichlet problem.
One may also prove the uniqueness by using the Green's identity, as in the case of
the $L^2$ Neumann problem.

To establish the existence, we normalize $\Omega$ so that $|\partial\Omega|=1$.
First, let us consider the case where $f=F|_{\partial\Omega}$,
where $F\in C_0^\infty(\br^d)$.
Choose a sequence of smooth domains $\{\Omega_j\}$
such that $\Omega_j\downarrow\Omega$ with homomorphism $\varLambda_j: \partial\Omega
\to \partial\Omega_j$,
given by Theorem \ref{approximation-theorem}.
Let $w_j$ be the solution of the Dirichlet problem:
$\Delta w_j=0$ in $\Omega_j$ and $w_j=F$ on $\partial\Omega_j$.
It follows from Theorem \ref{Laplace-L-2-Dirichlet-theorem} that
\begin{equation}\label{5.5.15-1}
\aligned
\|(\nabla w_j)^*\|_{L^2(\partial\Omega_j)}
+\|(w_j)^*\|_{L^2(\partial\Omega_j)}
&\le C \left\{
\| \nabla w_j\|_{L^2(\partial\Omega_j)}
+\|w_j\|_{L^2(\partial\Omega_j)}\right\} \\
& \le C\, \| F\|_{W^{1,2}(\partial\Omega_j)},
\endaligned
\end{equation}
where we have used Theorem \ref{Rellich-estimate-inside-theorem} for the second inequality.
We emphases that the constant $C$ in (\ref{5.5.15-1})
depends only on the Lipschitz character of $\Omega$.

Next we observe that by (\ref{5.5.15-1}),
the sequence $\{ w_j\}$ is bounded in $W^{1,2}(\Omega)$.
Thus, by passing to a subsequence, we may assume that $w_j$ converges 
to $w$ in $L^2(\Omega)$.
It follows from the mean value property and interior estimates
for harmonic functions that $w_j \to w$ and $\nabla w_j\to \nabla w$
uniformly on any compact subset of $\Omega$.
Moreover, $w$ is harmonic in $\Omega$.
For $Q\in \partial\Omega$ and $\delta>0$, let
$$
\mathcal{M}_\delta^2(u) (Q)
=\sup \big\{ |u(x)|: \
x\in \gamma_\alpha (Q) \text{ and dist}(x,\partial\Omega)\ge \delta\big\}.
$$
Note that by (\ref{5.5.15-1}), if $j$ is large,
\begin{equation}\label{5.5.15-3}
\|\mathcal{M}_\delta^2 (\nabla w_j)\|_{L^2(\partial\Omega)}
\le C\, \|(\nabla w_j)^*\|_{L^2(\partial\Omega_j)}
\le C\, \| F\|_{W^{1,2}(\partial\Omega_j)}.
\end{equation}
Since $\nabla w_j \to \nabla w$ uniformly on any compact subset of $\Omega$,
we see that $\mathcal{M}_\delta^2 (\nabla w_j)$ converges to $\mathcal{M}_\delta^2 (\nabla w)$ uniformly on $\partial\Omega$.
Thus, by letting $j\to \infty$ in (\ref{5.5.15-3}), we obtain
$$
\| \mathcal{M}_\delta^2 (\nabla w) \|_{L^2(\partial\Omega)}
\le C\, \| f\|_{W^{1,2}(\partial\Omega)}.
$$
We now let $\delta\to 0$.
By the monotone convergence theorem, this gives
$\|(\nabla w)^*\|_{L^2(\partial\Omega)}
\le C\, \|f\|_{W^{1,2}(\partial\Omega)}$.
Let $u$ be the solution of the $L^2$ Dirichlet problem in $\Omega$ with boundary data $f$.
Using 
$$
\|(w_j -u)^*\|_{L^2(\partial\Omega)}\le C \, \| w_j -f\|_{L^2(\partial\Omega)}
\le C\, \| w_j -F_j\|_{L^2(\partial\Omega)}
+C\, \| F_j -F\|_{L^2(\partial\Omega)}
\to 0,
$$
where $F_j(P)=F(\varLambda_j(P))$,
we see that $w_j \to u$ in $L^2(\Omega)$.
As a result, $u=w$ in $\Omega$.
This shows that $w=f$ n.t.\,on $\partial\Omega$.

Finally suppose that $f\in W^{1,2}(\partial\Omega)$.
We choose a sequences of smooth functions $\{f_j\}$ in $\br^d$
such that $\| f_j -f\|_{W^{1,2}(\partial\Omega)}
\to 0$ as $j\to \infty$.
Let $u_j$ be the solution of the $L^2$ Dirichlet problem in $\Omega$
with boundary data $f_j$.
We have proved above that
\begin{equation}\label{5.5.15-5}
\|(u_j-u_k)^*\|_{L^2(\partial\Omega)}
+\|(\nabla u_j -\nabla u_k)^*\|_{L^2(\partial\Omega)}
\le C\, \| f_j -f_k \|_{W^{1,2}(\partial\Omega)}.
\end{equation}
It follows that $u_j $ converges to $u$ uniformly on any compact subset of $\Omega$ and
$u$ is harmonic in $\Omega$.
Let $k\to \infty$ in $(\ref{5.5.15-5})$.
We obtain
$$
\|\mathcal{M}_\delta^2 (u_j -u)\|_{L^2(\partial\Omega)}
+\|\mathcal{M}_\delta^2 (\nabla u_j -\nabla u) \|_{L^2(\partial\Omega)}
\le C\, \| f_j -f\|_{W^{1,2}(\partial\Omega)}.
$$
As before, this leads to
\begin{equation}\label{5.5.15-7}
\|(u_j-u)^*\|_{L^2(\partial\Omega)}
+\|(\nabla u_j -\nabla u)^*\|_{L^2(\partial\Omega)}
\le C\, \| f_j -f\|_{W^{1,2}(\partial\Omega)},
\end{equation}
by the monotone convergence theorem.
Using (\ref{5.5.15-7}) and
$$
\limsup_{\substack{x\to P\\ x\in \Omega\cap\gamma_\alpha (P)}}
|u(x)-f(P)|
\le (u-u_j)^*(P) +|f_j (P)-f(P)|,
$$
we may conclude that $u=f$ n.t.\,on $\partial\Omega$.
To finish the proof,
we observe that $\|(\nabla u)^*\|_{L^2(\partial\Omega)}
\le C\, \|f\|_{W^{1,2}(\partial\Omega)}$.
This follows from the estimate
$\|(\nabla u_j)^*\|_{L^2(\partial\Omega)}
\le C\, \| f_j\|_{L^2(\partial\Omega)}$ by
 the same argument as in the proof of (\ref{5.5.15-7}).
\end{proof}

The next theorem shows that if $u$ is a solution of the $L^p$ regularity problem in $\Omega$,
then $\nabla u$ has nontangential limit a.e.\,on $\partial\Omega$.

\begin{thm}\label{regularity-problem-limit-theorem}
Let $u$ be harmonic in a bounded Lipschitz domain $\Omega$.
Suppose that $(\nabla u)^*\in L^p(\partial\Omega)$ for some $1<p<\infty$.
Then $u$ and $\nabla u$ have nontangential limits a.e.\,on $\partial\Omega$.
Furthermore, one has $u|_{\partial\Omega}\in W^{1,p}(\partial\Omega)$ and
$\|(\nabla u)^*\|_{L^p(\partial\Omega)}\le C_p\, \| \nabla u\|_{L^p(\partial\Omega)}$,
where $C_p$ depends only on $p$ and the Lipschitz character of $\Omega$.
\end{thm}

\begin{proof}
Recall  that $(\nabla u)^*\in L^p(\partial\Omega)$ implies that
$(u)^*\in L^p(\partial\Omega)$ and $u$ has nontangential limit a.e.\,on
$\partial\Omega$.
Let $\{\Omega_\ell\}$ be a sequence of smooth domains
such that $\Omega_\ell\uparrow \Omega$
and $\varLambda_\ell : \partial\Omega \to \partial\Omega_\ell$
the homomorphisms, given by Theorem \ref{approximation-theorem}.
Since $(\nabla u)^*\in L^p(\partial\Omega)$,
it follows that $\{\frac{\partial u}{\partial x_j}\circ \varLambda_\ell\}$
is bounded in $L^p(\partial\Omega)$.
Hence, by passing to a subsequence, we may assume that
$\frac{\partial u}{\partial x_j}\circ \varLambda_\ell $ converges weakly to $g_j$ in $L^p(\partial\Omega)$
as $\ell\to\infty$.
It follows by a limiting argument that
$$
\int_{\partial \Omega}
u \left( n_j \frac{\partial\varphi}{\partial x_k}
-n_k \frac{\partial \varphi}{\partial x_j}\right)\, d\sigma
=-\int_{\partial\Omega}
(n_j g_k-n_k g_j) \varphi\, d\sigma
$$
for any $\varphi\in C_0^1(\br^d)$.
By definition this implies that $u|_{\partial\Omega}\in W^{1,p}(\partial\Omega)$
and 
$$
\|\nabla_{\rm tan} u\|_{L^p(\partial\Omega)}
\le C \, \sum_j \| g_j\|_{L^p(\partial\Omega)}
\le C\, \|(\nabla u)^*\|_{L^p(\partial\Omega)}.
$$

Next, to show that $\nabla u$ has nontangential limit and
$\|(\nabla u)^*\|_{L^p(\partial\Omega)}\le C_p\, \|\nabla u\|_{L^p(\partial\Omega)}$,
we use the Green's representation formula in $\Omega_\ell$:
\begin{equation}\label{5.5.16-1}
u(x)=\int_{\partial\Omega_\ell} \Gamma(x-y)\frac{\partial u}{\partial n}\, d\sigma (y)
-\int_{\partial\Omega_\ell}
\frac{\partial}{\partial n(y)} 
\Big\{ \Gamma(x-y)\Big\} u(y)\,d\sigma (y).
\end{equation}
It is easy to see that the second integral in the right hand side of (\ref{5.5.16-1})
converges to $\mathcal{D} (u)(x)$ as $\ell\to\infty$.
To handle the first integral we write it as
$$
\int_{\partial\Omega_\ell}
\Gamma(x-\varLambda_\ell (y)) h_\ell (y)\, d\sigma (y),
$$
where $h_\ell (y)=\langle \nabla u(\varLambda_\ell (y)), n(\varLambda_\ell  (y))\rangle \omega_\ell (y)$.
Since $\{ h_\ell\}$ is bounded in $L^p(\partial\Omega)$,
by passing to a subsequence, we may assume that
$h_\ell \rightharpoonup h$ weakly in $L^p(\partial\Omega)$.
It follows that the first integral in the right hand side of (\ref{5.5.16-1}) converges to
$\mathcal{S}(h)(x)$ and hence, for $x\in \Omega$
\begin{equation}\label{5.5.16-3}
u(x)=\mathcal{S}(h)(x) -\mathcal{D}(u)(x).
\end{equation}
We claim that if $f\in W^{1,p}(\partial\Omega)$ and $1<p<\infty$, then $\nabla \mathcal{D}(f)$ has
nontangential limit a.e.\,on $\partial\Omega$ and
\begin{equation}\label{claim-5.5.16}
\|(\nabla \mathcal{D}(f))^*\|_{L^p(\partial\Omega)}
\le C_p \, \| \nabla_{\rm tan} f\|_{L^p(\partial\Omega)},
\end{equation}
where $C_p$ depends only on $p$ and the Lipchitz character of $\Omega$.
Assume the claim for a moment.
We may deduce from (\ref{5.5.16-3}) that
$\nabla u$ has nontangential limit a.e.\,on $\partial\Omega$ and thus $h=\langle n, \nabla u\rangle $.
Moreover,
$$
\|(\nabla u)^*\|_{L^p(\partial\Omega)}
\le C_p \, \big\{ \| h\|_{L^p(\partial\Omega)}
+\|\nabla_{\rm tan} u\|_{L^p(\partial\Omega)} \big\}
\le C_p\, \|\nabla u\|_{L^p(\partial\Omega)}.
$$

It remains to prove the claim.
Let $f\in W^{1,p}(\partial\Omega)$ and $w=\mathcal{S}(f)$.
Observe that
$$
w(x)=\int_{\partial\Omega} \frac{\partial}{\partial n(y)}
\Big\{ \Gamma(x-y) \Big\} f(y)\, d\sigma (y)
=-\frac{\partial}{\partial x_k}
\int_{\partial\Omega}
\Gamma(x-y) n_k(y) f(y)\, d\sigma (y).
$$
It follows that for $x\in \Omega$,
$$
\aligned
\frac{\partial w}{\partial x_j}  &=-\frac{\partial^2}{\partial x_j\partial x_k}
\int_{\partial\Omega}
\Gamma(x-y) n_k(y) f(y)\, d\sigma (y)\\
& =\frac{\partial}{\partial x_k}
\int_{\partial\Omega}
\frac{\partial}{\partial y_j} \Big\{ \Gamma(x-y)\Big\} n_k(y) f(y)\, d\sigma (y)\\
&=\frac{\partial}{\partial x_k}
\int_{\partial\Omega}
\left\{ n_k \frac{\partial}{\partial y_j} -n_j \frac{\partial}{\partial y_k} \right\} \Big\{ \Gamma(x-y)\Big\} f(y)\, d\sigma (y)\\
& =\frac{\partial}{\partial x_k}
\int_{\partial\Omega}
\Gamma(x-y) g_{jk}(y)\, d\sigma (y),
\endaligned
$$
where $g_{jk}=\frac{\partial f}{\partial t_{jk}}$
and we have used the fact $\Delta_y \{ \Gamma(x-y)\big\}=0$
for $y\neq x$ in the third equality.
By Theorem \ref{nontangential-limit-theorem} we may conclude that
$\nabla w$ has nontangential limit a.e.\,on $\partial\Omega$ and
$$
 \|(\nabla w)^*\|_{L^p(\partial\Omega)}
\le C_p \, \sum_{j,k} \| g_{jk}\|_{L^p(\partial\Omega)}
\le C_p \, \|\nabla_{\rm tan} f\|_{L^p(\partial\Omega)}.
$$
This completes the proof.
\end{proof}

Recall that $\mathcal{S}: L^p(\partial\Omega) \to W^{1,p}(\partial\Omega)$ is bounded for
$1<p<\infty$.

\begin{thm}\label{Laplace-single-layer-invertibility-theorem}
Let $\Omega$ be a bounded Lipschitz domain in $\br^d$, $d\ge 3$,
with connected boundary.
Then $\mathcal{S}: L^2(\partial\Omega)\to W^{1,2}(\partial\Omega)$
is an isomorphism.
Furthermore, we have
$$
\| g\|_{L^2(\partial\Omega)}\le C\, \|\mathcal{S}(g)\|_{W^{1,2}(\partial\Omega)}
$$
for any $g\in L^2(\partial\Omega)$,
where $C$ depends only on the Lipschitz character of $\Omega$.
Consequently, the unique solution of the $L^2$ regularity problem
in $\Omega$ with data $f\in W^{1,2}(\partial\Omega)$ 
may be represented by a single layer potential $\mathcal{S}(g)$,
where $g\in L^2(\partial\Omega)$ and $\| g\|_{L^2(\partial\Omega)}
\le C \, \| f\|_{W^{1,2}(\partial\Omega)}$.
\end{thm}

\begin{proof}
By dilation we may assume that $|\partial\Omega|=1$.
Let $g\in L^2(\partial\Omega)$ and $u=\mathcal{S} (g)$.
Then $u(x)=O(|x|^{2-d})$ as $|x|\to \infty$.
By the jump relation (\ref{jump-relation}) and
Theorems \ref{Rellich-estimate-inside-theorem} and \ref{Rellich-estimate-outside-theorem},
$$
\aligned
\| g\|_{L^2(\partial\Omega)}
& \le \big\|\left( \frac{\partial u}{\partial n}\right)_+\big\|_{L^2(\partial\Omega)} +
\big\|\left( \frac{\partial u}{\partial n}\right)_-\big\|_{L^2(\partial\Omega)}\\
& \le C\, \Big\{
\|\nabla_{\rm tan} u\|_{L^2(\partial\Omega)}
+\| u\|_{L^2(\partial\Omega)} \Big\}\\
&\le C\, 
\| \mathcal{S} (g)\|_{W^{1,2}(\partial\Omega)},
\endaligned
$$
where $C$ depends only on the Lipschitz character of $\Omega$.

It remains to show that $\mathcal{S}: L^2(\partial\Omega)
\to W^{1,2}(\partial\Omega)$ is surjective.
To this end, let $f\in W^{1,2}(\partial\Omega)$ and $u$ be the unique solution of the $L^2$
regularity problem in $\Omega$ with data $f$, given by Theorem \ref{Laplace-L-2-regularity-theorem}.
Since $(\nabla u)^* \in L^2(\partial\Omega)$,
by Theorem \ref{regularity-problem-limit-theorem},
$\nabla u$ has nontangential limit a.e.\,on $\partial\Omega$.
Thus $u$ may be regarded as a solution to the $L^2$ Neumann problem with 
data $\langle n,\nabla u\rangle $.
It follows from Theorem \ref{Laplace-L-2-Neumann-theorem} that $u=\mathcal{S}(h_1)+\beta$
for some $h_1\in L^2_0(\partial\Omega)$ and $\beta\in \br$.

Finally we recall that the range of the operator $(1/2)I +\mathcal{K}$ on $ L^2(\partial\Omega)$
is $L^2_0(\partial\Omega)$.
It follows that  there exists $h_2 \in L^2(\partial\Omega)$ such that
$h_2\neq 0$ and $\big( (1/2)I+\mathcal{K}\big) (h_2)=0$.
Let $v=\mathcal{S}(h_2)$.
Then $v$ is a nonzero constant in $\Omega$.
Hence there exists $\alpha \in \br$ such that
$\beta =\mathcal{S}(\alpha h_2)$ in $\Omega$.
As a result we obtain $u=\mathcal{S}(h_1+\alpha h_2)$ in $\overline{\Omega}$.
This finishes the proof.
\end{proof}

%
%
%
%
%
%
%
%
%
%
%
%

\section{Rellich property}\label{section-5.6}

\begin{definition}
{\rm
Let $\mathcal{L}=-\text{\rm div}\big(A(x)\nabla\big)$ and
$\Omega$ be a bounded Lipschitz domain with connected boundary.
We say that the operator $\mathcal{L}$ has the {\it Rellich property in $\Omega$ with
constant $C=C(\Omega)$} if
\begin{equation}\label{Rellich-property}
\|\nabla u\|_{L^2(\partial\Omega)}\le C\,
 \big\|\frac{\partial u}{\partial\nu}\big\|_{L^2(\partial\Omega)}
 \quad \text{ and }\quad
\|\nabla u\|_{L^2(\partial\Omega)}
\le C\, \|\nabla_{\rm tan} u\|_{L^2(\partial \Omega)},
\end{equation}
whenever $u$ is a solution of
$\mathcal{L} (u)=0$ in $\Omega$ such that
$(\nabla u)^*\in L^2(\partial\Omega)$ and
$\nabla u$ exists n.t.\,on $\partial\Omega$.
}
\end{definition}

In the previous section we proved that the Laplace operator $\mathcal{L}=-\Delta$ has the Rellich property
in any Lipschitz domain $\Omega$ with constant $C(\Omega)$ depending only on the 
Lipschitz character of $\Omega$. This was 
used to establish the invertibility of
$\pm (1/2)I +\mathcal{K}$ on $L^2(\partial\Omega)$ and consequently solve
the $L^2$ Dirichlet and Neumann problems.
We will see in this section that for second-order elliptic operators with variable coefficients,
the well-posedness of the $L^2$ Neumann, Dirichlet, and regularity problems
in Lipschitz domains
may be reduced to the boundary Rellich estimates in (\ref{Rellich-property}) by the method of
Layer potentials and a localization argument.

The following two theorems are the main results of this section.
The first theorem treats the well-posedness in small scale; the constants $C$ in the 
nontangential-maximal-function estimates in (\ref{estimate-5.6-1}) may depend on diam$(\Omega)$, if
diam$(\Omega)\ge 1$.
The assumptions and conclusions in the second theorem are scale invariant.
As a result, by rescaling, they lead to uniform $L^2$ estimates in a Lipschitz domain
for the family of elliptic operators $\{\mathcal{L}_\varep\}$.

\begin{thm}\label{theorem-5.6.1}
Let $d\ge 3$ and $\mathcal{L}=-\text{\rm div}(A(x)\nabla)$ with 
$A\in \Lambda(\mu, \lambda, \tau)$.
Let $R\ge 1$.
Suppose that for any Lipschitz domain $\Omega$ with {\rm diam}$(\Omega)\le (1/4)$
and connected boundary, there exists $C(\Omega)$ depending on the Lipschitz
character of $\Omega$ such that
for each $s\in (0,1]$, the operator
$$
\mathcal{L}^s=-\text{\rm div}\big((sA +(1-s)I)\nabla\big)
$$
has the Rellich property in $\Omega$ with constant $C(\Omega)$.
Then for any Lipschitz domain $\Omega$ with {\rm diam}$(\Omega)\le R$ and connected
boundary, the $L^2$ Neumann and regularity problems for
$\mathcal{L}(u)=0$ in $\Omega$ are well-posed and the solutions
satisfy the estimates
\begin{equation}\label{estimate-5.6-1}
\|(\nabla u)^*\|_{L^2(\partial\Omega)}
\le C \,\big\|\frac{\partial u}{\partial \nu}\big\|_{L^2(\partial\Omega)}
\quad \text{ and } \quad
\|(\nabla u)^*\|_{L^2(\partial\Omega)}
\le C \, \|\nabla_{\rm tan} u\|_{L^2(\partial\Omega)},
\end{equation}
where $C$ depends only on $\mu$, $\lambda$, $\tau$, the Lipschitz character of $\Omega$,
and $R$ (if \text{\rm diam}$(\Omega)\ge 1$).
Furthermore, the $L^2$ Dirichlet problemfor $\mathcal{L}^* (u)=0$
 in $\Omega$ is well-posed with the estimate
$\|(u)^*\|_{L^2(\partial\Omega)} \le C \, \|u\|_{L^2(\partial\Omega)}$.
\end{thm}

\begin{thm}\label{theorem-5.6.2}
Let $d\ge 3$ and $\mathcal{L}=-\text{\rm div} (A(x)\nabla)$ with $A\in  \Lambda(\mu, \lambda, \tau)$.
Suppose that for any Lipschitz domain $\Omega$ with connected boundary,
there exists $C(\Omega)$ depending on the Lipschitz character of $\Omega$ such that
for each $s\in (0,1]$, the operator 
$\mathcal{L}^s =-\text{div} \big((sA+(1-s)I)\nabla\big)$
has the Rellich property in $\Omega$ with constant $C(\Omega)$.
Then for any Lipshcitz domain $\Omega$ with connected boundary,
the $L^2$ Neumann and regularity problems for $\mathcal{L}(u)=0$ in $\Omega$
are well-posed and the solutions satisfy the estimates in (\ref{estimate-5.6-1})
with a constant $C$ depending only on $\mu$, $\lambda$, $\tau$, and the Lipschitz character of $\Omega$.
Furthermore, the $L^2$ Dirichlet problem for $\mathcal{L}^* (u)=0$ in $\Omega$ is well-posed
with the estimate $\|(u)^*\|_{L^2(\partial\Omega)}
\le C \, \| u\|_{L^2(\partial\Omega)}$.
\end{thm}

The uniqueness for the $L^2$ Neumann and regularity problems follows readily from the Green's identity,
\begin{equation}\label{Green's-identity-5.6}
\int_\Omega A(x)\nabla u
\cdot \nabla u\, dx
=\int_{\partial\Omega}
\frac{\partial u}{\partial \nu} \cdot u \, d\sigma.
\end{equation}
We will use the method of layer potentials to establish the existence
of solutions in Theorems \ref{theorem-5.6.1} and
\ref{theorem-5.6.2}.

Recall that $\Omega_+=\Omega$ and $\Omega_-=\br^d\setminus \overline{\Omega}$.

\begin{lemma}\label{lemma-5.6-1}
Let $\Omega$ be a bounded Lipschitz domain with $r_0=\text{\rm diam}(\Omega)\le R$.
Suppose that $\mathcal{L}(u)=0$ in $\Omega_\pm$,
$(\nabla u)^*\in L^2(\partial\Omega)$, and $(\nabla u)_\pm $ exists n.t.\,on $\partial\Omega$.
Under the same assumptions on $A$ as in Theorem \ref{theorem-5.6.1}, we have
\begin{equation}\label{estimate-5.6.1}
\aligned
\int_{\partial\Omega}
|(\nabla u)_\pm |^2\, d\sigma
& \le C \int_{\partial\Omega} \big|\left(\frac{\partial u}{\partial\nu}\right)_\pm \big|^2\, d\sigma
+\frac{C}{r_0} \int_{N_\pm} |\nabla u|^2\, dx,\\
\int_{\partial\Omega}
|(\nabla u)_\pm |^2\, d\sigma
& \le C \int_{\partial\Omega} \big|\left(\nabla_{\rm tan} u \right)_\pm \big|^2\, d\sigma
+\frac{C}{r_0} \int_{N_\pm} |\nabla u|^2\, dx,
\endaligned
\end{equation}
where 
$$
N_\pm =\Big\{x\in \Omega_\pm: \ \text{\rm dist} (x,\partial\Omega)\le r_0 \Big\},
$$
and $C$ depends on the Lipschitz character of $\Omega$ and $R$ (if $r_0>1$).
\end{lemma}

\begin{proof}
The proof uses a localization argument.
Let $\psi:\br^{d-1}\to \br$ be a Lipschitz function such that
$\psi (0)=0$ and $\|\nabla\psi\|_\infty\le M$.
Let 
\begin{equation}\label{5.6.4-1}
\aligned
Z_r=& \big\{ (x^\prime, x_d)\in \br^d:\
|x^\prime|<r \text{ and } \psi(x^\prime)<x_d<10\sqrt{d}\, (M+1)r\big\},\\
\Delta_r= & \big\{(x^\prime, \psi(x^\prime))\in \br^d:\ |x^\prime|<r\big\}.
\endaligned
\end{equation}
Suppose that $\mathcal{L}(u)=0$ in $\Omega_0=Z_{3r}$,
$(\nabla u)^*\in L^2(\partial\Omega_0)$, and
$\nabla u$ exists n.t.\,on $\partial\Omega_0$.
Assume that diam$(Z_{2r})\le (1/4)$.
Then for any $t\in (1,2)$,
$\mathcal{L}$ has the Rellich property in the Lipschitz domain $Z_{tr}$
with constant $C_0(Z_{tr})$ depending only on $M$.
It follows that
\begin{equation}\label{5.6.4-3}
\aligned
\int_{\Delta_r} |\nabla u|^2\, d\sigma
&\le \int_{\partial Z_{tr}}
|\nabla u|^2\, d\sigma\\
&\le C_0 \int_{\Delta_{2r}}
\big|\frac{\partial u}{\partial \nu}\big|^2\, d\sigma
+ CC_0 \int_{\partial Z_{tr}\setminus \Delta_{tr}} |\nabla u|^2\, d\sigma,
\endaligned
\end{equation}
where $C$ depends only on $\mu$.
We now integrate both sides of (\ref{5.6.4-3}) with respect to $t$ over the interval $(1,2)$
to obtain
\begin{equation}\label{5.6.4-5}
\int_{\Delta_r}
|\nabla u|^2\, d\sigma
\le C\int_{\Delta_{2r}}
\big|\frac{\partial u}{\partial \nu}\big|^2\, d\sigma
+\frac{C}{r}
\int_{Z_{2r}}
|\nabla u|^2\, dx.
\end{equation}

Finally we choose $r=c(M) r_0$ if $r_0\le 1$ and $r=c(M)$ if $r_0>1$.
The first inequality in (\ref{estimate-5.6-1})
follows by covering $\partial\Omega$ with $\{\Delta_i\}$,
each of which may be obtained from $\Delta_r$ through translation and rotation.
The proof  for the second inequality in (\ref{estimate-5.6.1}) is similar.
\end{proof}

\begin{remark}\label{remark-5.6-1}
{\rm
Under the same conditions on $A$ as in Theorem \ref{theorem-5.6.2},
the estimates in (\ref{estimate-5.6.1})
hold with constant $C$ independent of $R$.
This is because we may choose $r=c(M)r_0$ for any $\Omega$.
}
\end{remark}

We will use $\mathcal{S}$, $\mathcal{D}$ and $\mathcal{K}_A$
to denote $\mathcal{S}_\varep$, $\mathcal{D}_\varep$ and $\mathcal{K}_{A,\varep}$,
respectively, when $\varep=1$.

\begin{lemma}\label{lemma-5.6.2}
Let $R\ge 1$ and $\Omega$ be a bounded Lipschitz domain with connected boundary and
{\rm diam}$(\Omega)\le R$.
Under the same conditions on $A$ as in Theorem \ref{theorem-5.6.1},
the operators $(1/2)I +\mathcal{K}_A$ and
$-(1/2)I+\mathcal{K}_A$ are isomorphisms
on $L^2_0(\partial\Omega; \br^m)$
and $L^2(\partial\Omega; \br^m)$, respectively.
Moreover,
\begin{equation}\label{estimate-5.6-2}
\|\big((1/2)I +\mathcal{K}_A\big)^{-1}\|_{L^2_0\to L^2_0}
\le C 
\quad \text{ and } \quad
\|\big(-(1/2)I +\mathcal{K}_A\big)^{-1}\|_{L^2\to L^2}
\le C,
\end{equation}
where $C$ depends on the Lipschitz character of $\Omega$ and $R$ (if {\rm diam}$(\Omega)\ge 1$).
\end{lemma}

\begin{proof}
Let $f\in L^2_0(\partial\Omega;\br^m)$ and $u=\mathcal{S} (f)$.
Then $\mathcal{L} (u)=0$ in $\br^d\setminus \partial\Omega$,
$(\nabla u)^*\in L^2(\partial\Omega)$ and
$(\nabla u)_\pm $ exists n.t.\,on $\partial\Omega$.
Also recall that $(\nabla_{\tan} u)_+
=(\nabla_{\rm tan} u)_-$ on $\partial\Omega$.
Note that
$$
|u(x)| +|x||\nabla u(x)|=O(|x|^{2-d})
$$
as $|x|\to \infty$, for $d\ge 3$.
It follows from integration by parts that
\begin{equation}\label{5.6.6-1}
\int_{\Omega_-} A(x)\nabla u\cdot \nabla u\, dx
=-\int_{\partial\Omega} \left(\frac{\partial u}{\partial \nu}\right)_- \cdot u\, d\sigma.
\end{equation}
By the jump relation (\ref{jump-relation}), 
$$
\int_{\partial\Omega} \left(\frac{\partial u}{\partial\nu}\right)_-\, d\sigma
=-\int_{\partial\Omega} f\, d\sigma =0.
$$
As in the case of Laplacian,
using Poincar\'e's inequality on $\partial\Omega$ and
the Green's identities (\ref{5.6.6-1}) and (\ref{Green's-identity-5.6}), we obtain
\begin{equation}\label{Green's-estimate}
\int_{\Omega_\pm}
|\nabla u|^2\, dx
\le C r_0 \big\|\left(\frac{\partial u}{\partial\nu}\right)_\pm \big\|_{L^2(\partial\Omega)}
\|\nabla_{\rm tan} u\|_{L^2(\partial\Omega)}.
\end{equation}
By combining (\ref{estimate-5.6.1}) with (\ref{Green's-estimate}) and then using the Cauchy inequality,
we see that
\begin{equation}\label{5.6.6-5}
\|(\nabla u)_\pm \|_{L^2(\partial\Omega)}
\le C\, \| \nabla_{\tan } u\|_{L^2(\partial\Omega)}
\quad \text{ and } \quad
\|(\nabla u)_\pm \|_{L^2(\partial\Omega)}
\le C\, \big\|\left(\frac{\partial u}{\partial\nu}\right)_\pm \big\|_{L^2(\partial\Omega)}.
\end{equation}
It follows that
$$
\aligned
\big\|\left(\frac{\partial u}{\partial\nu}\right)_\pm \big\|_{L^2(\partial\Omega)}
 & \le C \, \|\nabla_{\rm tan} u\|_{L^2(\partial\Omega)}
\le C\, \|(\nabla u)_\mp\|_{L^2(\partial\Omega)}\\
&\le C\, \big\|\left(\frac{\partial u}{\partial\nu}\right)_\mp \big\|_{L^2(\partial\Omega)}.
\endaligned
$$
By the jump relation this gives
\begin{equation}\label{5.6.6-6}
\aligned
\| f\|_{L^2(\partial\Omega)}
 &\le C\, \big\|\left(\frac{\partial u}{\partial\nu}\right)_\pm \big\|_{L^2(\partial\Omega)}\\
& =C\, \|\big(\pm (1/2) I  +\mathcal{K}_A\big) f\|_{L^2(\partial\Omega)}
\endaligned
\end{equation}
for any $f\in L^2_0(\partial\Omega;\br^m)$.
By considering $f-\average_{\partial\Omega} f$, as in the case of Laplace's equation,
we may deduce from (\ref{5.6.6-6}) that
\begin{equation}\label{5.6.6-7}
\| f\|_{L^2(\partial\Omega)}
\le
C\, \|\big(-(1/2) I  +\mathcal{K}_A\big) f\|_{L^2(\partial\Omega)}
\end{equation}
for any $f\in L^2(\partial\Omega; \br^m)$.

Thus, to complete the proof,
it suffices to show that
$(1/2)I +\mathcal{K}_A: L_0^2(\partial\Omega; \br^m)\to L^2_0(\partial\Omega;\br^m)$
and 
$-(1/2)I +\mathcal{K}_A: L^2(\partial\Omega;\br^m)\to L^2(\partial\Omega;\br^m)$
are surjective.
This may be done by a continuity method. Indeed,
let us consider a family of matrices $A^s=sA +(1-s)I$ in $\Lambda(\mu, \lambda, \tau)$,
where $0\le s\le 1$.
Note that by Theorem \ref{Laplace-theorem}, $\pm (1/2)I +\mathcal{K}_{{A}^0}$
are isomorphisms on $L^2_0(\partial\Omega;\br^m)$ and
$L^2(\partial\Omega;\br^m)$, respectively.
Also observe that for each $s\in [0,1]$,
the matrix $A^s$ satisfies the same conditions as $A$.
Hence,
$$
\aligned
\|f\|_{L^2(\partial\Omega)}
&\le C\, \|\big((1/2) I+\mathcal{K}_{A^s}\big) f\|_{L^2(\partial\Omega)} \qquad \ \text{ for any } 
f\in L^2_0(\partial\Omega; \br^m),\\
\|f\|_{L^2(\partial\Omega)}
&\le C\, \|\big(-(1/2) I+\mathcal{K}_{A^s}\big) f\|_{L^2(\partial\Omega)} \quad \text{ for any } 
f\in L^2(\partial\Omega; \br^m),
\endaligned
$$
where $C$ is independent of $s$.
Since 
$$
\|A^{s_1}-A^{s_2}\|_{C^\lambda(\br^d)} \le C |s_1-s_2|\|A\|_{C^\lambda(\br^d)},
$$
it follows from Theorem \ref{theorem-5.2-2} that
$\{ (1/2)I +\mathcal{K}_{A^s}: \, 0\le s\le 1\}$
and $\{ -(1/2)I +\mathcal{K}_{A^s}:\ 0\le s\le 1\}$
are  continuous families of bounded operators
on $L^2_0(\partial\Omega; \br^m)$
and $L^2(\partial\Omega; \br^m)$,
respectively.
In view of Lemma \ref{continuity-lemma} we may conclude that
$\pm (1/2)+\mathcal{K}_A$ are isomorphisms on $L^2_0(\partial\Omega; \br^m)$
and $L^2(\partial\Omega; \br^m)$, respectively.
This finishes the proof.
\end{proof}

\begin{remark}\label{remark-5.6-2}
{\rm
Suppose $d\ge 3$.
Under the same assumptions on $A$ and $\Omega$ as in Lemma \ref{lemma-5.6.2},
the operator $\mathcal{S}: L^2(\partial\Omega; \br^m)
\to W^{1,2}(\partial\Omega; \br^m)$ is an isomorphism 
and
\begin{equation}\label{5.6.7-0}
\| \mathcal{S}^{-1}\|_{W^{1,2}\to L^2}
\le C.
\end{equation}
To see this, we let $f\in W^{1,2}(\partial\Omega; \br^m)$ and $u=\mathcal{S} (f)$.
Note that if $d\ge 3$, we have
$|u(x)| =O(|x|^{2-d})$ and $|\nabla u(x)|=O(|x|^{1-d})$, as $|x|\to \infty$,
which allow us to use the Green's identity on $\Omega_-$.
It  follows from the proof of Lemma \ref{lemma-5.6.2} that
$$
\|(\nabla u)_-\|_{L^2(\partial\Omega)}
\le C \, \|\nabla_{\rm tan} u\|_{L^2(\partial\Omega)}
+ C r_0^{-1} \| u\|_{L^2(\partial\Omega)}.
$$
This, together with $\|(\nabla u)_+\|_{L^2(\partial\Omega)}
\le C \, \|\nabla_{\rm tan} u\|_{L^2(\partial\Omega)}$ and the jump relation, gives
\begin{equation}\label{5.6.7-1}
\aligned
\| f\|_{L^2(\partial\Omega)}
 & \le C\, \Big\{ \|\nabla_{\rm tan} \mathcal{S} (f)\|_{L^2(\partial\Omega)}
+r_0^{-1} \|\mathcal{S}(f)\|_{L^2(\partial\Omega)}\Big\} \\
& \le C \, \|\mathcal{S}(f)\|_{W^{1,2}(\partial\Omega)}.
\endaligned
\end{equation}
Hence, $\mathcal{S}: L^2(\partial\Omega; \br^m)\to W^{1,2}(\partial\Omega; \br^m)$
is injective.
A continuity argument similar to that in the proof of Lemma \ref{lemma-5.6.2}
shows that the operator is in fact an isomorphism.
}
\end{remark}

\begin{remark}\label{remark-5.6-3}
{\rm
Under the same assumptions on $A$ as in Theorem \ref{theorem-5.6.2},
the estimates in (\ref{estimate-5.6-2}) and (\ref{5.6.7-0}) 
(for $d\ge 3$) hold with a constant depending on the Lipschitz character of $\Omega$.
}
\end{remark}

We are now in a position to give the proof of Theorems \ref{theorem-5.6.1}
and \ref{theorem-5.6.2}.

\begin{proof}[\bf Proof of Theorems \ref{theorem-5.6.1} and \ref{theorem-5.6.2}]
The existence of solutions to the $L^2$ Neumann and regularity problems for
$\mathcal{L}(u)=0$ in $\Omega$
is a direct consequence of the invertibility of $(1/2)I +\mathcal{K}_A$ on 
$L^2_0(\partial\Omega; \br^m)$
and that of $\mathcal{S}: L^2(\partial\Omega; \br^m)\to W^{1,2}(\partial\Omega; \br^m)$,
respectively.
Since $-(1/2)I +\mathcal{K}_A$ is invertible on $L^2(\partial\Omega; \br^m)$,
it follows by duality that $-(1/2)I +\mathcal{K}^*_A$ is also invertible on $L^2(\partial\Omega; \br^m)$
and
$$
\|\big(-(1/2)I +\mathcal{K}_A\big)^{-1}\|_{L^2\to L^2}
=\|\big(-(1/2)I +\mathcal{K}^*_A\big)^{-1}\|_{L^2\to L^2}.
$$
This gives the existence of solutions to the $L^2$ Dirichlet problem 
for $\mathcal{L}^* (u)=0$ in $\Omega$.
Note that under the conditions in Theorem \ref{theorem-5.6.2},
the operator norms of
$\big(\pm (1/2)I +\mathcal{K}_A\big)^{-1}$ and $\mathcal{S}^{-1}$
are bounded by constants depending on the Lipschitz character of $\Omega$.
As a result the constants $C$ in (\ref{estimate-5.6-1}) and  in $\|(u)^*\|_{L^2(\partial\Omega)}
\le C \, \| u\|_{L^2(\partial\Omega)}$ depend on the Lipschitz character of $\Omega$,
not on diam$(\Omega)$.

As we mentioned earlier, the uniqueness for the $L^2$ Neumann and regularity 
problems follows readily from the Green's identity (\ref{Green's-identity-5.6}).
To establish the uniqueness for the $L^2$ Dirichlet problem for $\mathcal{L}^*(u)=0$ in $\Omega$,
we construct a matrix of Green's functions $\big( G^{\alpha\beta}(x,y)\big)$ for $\Omega$,
where 
$$
G^{\alpha\beta}(x,y)=\Gamma^{\alpha\beta}(x,y;A)-W^{\alpha\beta}(x,y),
$$
and for each $\beta$ and $y\in \Omega$,
$W^\beta(\cdot, y)=\big(W^{1\beta}(\cdot, y), \dots, W^{m\beta} (\cdot, y)\big)$ is the 
solution to the $L^2$ regularity problem for $\mathcal{L}(u)=0$ in $\Omega$ with boundary
data
$$
\Gamma^\beta(\cdot, y)
=\big( \Gamma^{1\beta} (\cdot, y), \dots, \Gamma^{m\beta} (\cdot, y)\big) \quad
\text{ on } \partial\Omega.
$$
Since $|\nabla_x \Gamma (x, y; A)|\le C\, |x-y|^{1-d}$, by the well-posedness of the $L^2$
regularity problem,
\begin{equation}\label{max-W}
\Big( \nabla_x W^\beta (\cdot, y) \Big)^* \in L^2(\partial\Omega).
\end{equation}

Suppose now that $\mathcal{L}^*(u)=0$ in $\Omega$,
$(u)^*\in L^2(\partial\Omega)$ and $u=0$ n.t.\,on $\partial\Omega$.
For $\rho>0$ small, choose $\varphi=\varphi_\rho$ so that
$\varphi=1$ in $\{x\in \Omega: \, \text{\rm dist}(x, \partial\Omega)\ge 2\rho\}$,
$\varphi=0$ in $\{ x\in \Omega:\, \text{\rm dist}(x, \partial\Omega)\le \rho\}$, 
and $|\nabla\varphi |\le C \rho^{-1}$.
Fix $y\in \Omega$ so that dist$(y, \partial\Omega)\ge 2 \rho$.
It follows that
$$
\aligned
u^\gamma (y) & =u^\gamma (y) \varphi(y)
=\int_{\Omega}
a_{ij}^{\alpha\beta} (x) \frac{\partial}{\partial x_j}
\Big\{ G^{\beta\gamma} (x,y) \Big\} \frac{\partial}{\partial x_i}
\big( u^\alpha \varphi\big)\, dx\\
&=-\int_\Omega
a_{ij}^{\alpha\beta} (x)
G^{\beta\gamma}(x,y) \frac{\partial u^\alpha}{\partial x_i}
\cdot \frac{\partial\varphi}{\partial x_j}\, dx\\
&\qquad
+\int_\Omega a_{ij}^{\alpha\beta} (x)
\frac{\partial} {\partial x_j}
\Big\{ G^{\beta\gamma}(x,y)\Big\} u^\alpha \frac{\partial\varphi}{\partial x_i}\, dx,
\endaligned
$$
where we have used integration by parts and 
$\mathcal{L}^* (u)=0$ in $\Omega$ for the last equality.
This gives
\begin{equation}\label{5.6.8-1}
|u(y)|
\le \frac{C}{\rho}
\int_{E_\rho} |G(x,y)| \, |\nabla u(x)|\, dx
+\frac{C}{\rho}
\int_{E_\rho}
|\nabla_x G(x,y)|\, |u(x)|\, dx,
\end{equation}
where
$E_\rho=\{ x\in \br^d: \, \rho\le \text{dist}(x, \partial\Omega)\le 2\rho\}$.
Using $G(\cdot, y)=u=0$ on $\partial\Omega$ n.t.\,on $\partial\Omega$
as well as the interior Lipschitz estimate
of $u$, we may deduce from (\ref{5.6.8-1}) that
\begin{equation}\label{5.6.8-3}
|u(y)|
\le C \int_{\partial\Omega}
\mathcal{M}_{3\rho}^1(\nabla_x G(\cdot,y)\big ) \cdot \mathcal{M}_{3\rho}^1(u)\, d\sigma,
\end{equation}
where 
\begin{equation}\label{5.6.8-5}
\mathcal{M}^1_t(u) (z)
=\sup\Big\{ |u(x)|: \ x\in \gamma(z)\cap\Omega \text{ and } \text{dist}(x, \partial\Omega)< t\Big\}.
\end{equation}

Finally, we note that by (\ref{max-W}), $\mathcal{M}_\delta^1 (\nabla_x G(\cdot, y))\in L^2(\partial\Omega)$
for $\delta=\text{dist}(y, \partial\Omega)/9$. This, together with the assumption $(u)^*\in L^2(\partial\Omega)$, shows that
$$\mathcal{M}_\delta^1 (\nabla_x G(\cdot,y))\cdot (u)^*\in L^1(\partial\Omega).$$
Also, observe that as $\rho\to 0$,  $\mathcal{M}_\rho^1(u)(z)\to 0$ for a.e. $z\in \partial\Omega$.
It follows by the Lebesgue dominated convergence theorem that the right hand side of
(\ref{5.6.8-3}) goes to zero as $\rho\to 0$.
This yields that  $u(y)=0$ for any $y\in \Omega$,  and completes the proof.
\end{proof}

%
%
%
%
%
%
%
%
%

\section{Well-posedness for small scales}\label{section-5.7}

In this section we establish the well-posedness of the $L^2$ Dirichlet, Neumann, and regularity
problems for $\mathcal{L}(u)=0$ in a Lipschitz domain
$\Omega$ in $\rd$, $d\ge 3$, 
under the ellipticity condition (\ref{s-ellipticity}), the smoothness condition (\ref{smoothness}), and
the symmetry condition $A^*=A$.
The periodicity condition is not needed.
However, the constants $C$ may depend on diam$(\Omega)$ if diam$(\Omega)$ is large.

\begin{thm}\label{local-solvability}
Let $d\ge 3$ and $\mathcal{L}=-\text{\rm div}(A(x)\nabla)$.
Assume that $A=A(x)$ satisfies
the ellipticity condition (\ref{s-ellipticity}),
the smoothness condition (\ref{smoothness}), and the symmetry condition $A^*=A$.
Let $R\ge 1$.
Then for any bounded Lipschitz domain $\Omega$ with connected boundary and
{\rm diam}$(\Omega)\le R$, the $L^2$ Neumann and regularity problems for
$\mathcal{L}(u)=0$ in $\Omega$ are well-posed and
the solutions satisfy the estimates in (\ref{estimate-5.6-1})
with constant $C$ depending only on $\mu$, $\lambda$, $\tau$,
the Lipschitz character of $\Omega$, and $R$ (if $\text{\rm diam}(\Omega)>1$).
Furthermore, the $L^2$ Dirichlet problem for
$\mathcal{L}(u)=0$ in $\Omega$ is well-posed
with the estimate $\|(u)^*\|_2 \le C\, \| u\|_2$.
\end{thm}

\begin{remark}\label{periodicity-remark}
{\rm
Note that the periodicity of $A$ is not needed in Theorem \ref{local-solvability}.
This is because we may reduce the general case to the case of periodic 
coefficients. 
Indeed, by translation, we may assume that $0\in \Omega$.
If diam$(\Omega)\le (1/4)$, we construct $\widetilde{A}\in 
\Lambda (\mu, \lambda, \tau_0)$ so that
$\widetilde{A}=A$ on $[-3/8,3/8]^d$, where $\tau_0$ depends on 
$\mu$ and $\tau$.
The boundary value problems for
$\text{div}(\widetilde{A} \nabla u)=0$ in $\Omega$
are the same as those for $\text{div}(A\nabla u)=0$ in $\Omega$.
Suppose now that $r_0=\text{diam}(\Omega)> (1/4)$. 
By rescaling, the boundary value problems
for $\text{div}(A\nabla u)=0$ in $\Omega$
are equivalent to that of
$\text{div}(B\nabla u)=0$ in $\Omega_1$,
where $B(x)=A(4r_0x)$ and 
$\Omega_1 =\{ x\in \mathbb{R}^d: 4r_0x\in \Omega\}$.
Since $\text{diam}(\Omega_1)= (1/4)$, we have reduced the case to the
previous one. Note that in this case the bounding constants $C$ may depend on $r_0$.
}
\end{remark}

By Remark \ref{periodicity-remark} it is enough to prove Theorem \ref{local-solvability}
under the additional assumption that $\text{diam}(\Omega)\le (1/4)$
and $A$ is 1-periodic (thus $A\in \Lambda(\mu, \lambda, \tau)$).
Furthermore, in view of Theorem \ref{theorem-5.6.1},
it suffices to show that if $\mathcal{L}=-\text{div}(A\nabla )$
with $A\in \Lambda(\mu,\lambda, \tau)$ and $A^*=A$ and if
$\Omega$ is a Lipschitz domain with $\text{diam}(\Omega)\le (1/4)$, then
$\mathcal{L}$ has the Rellich property in $\Omega$
with constant $C(\Omega)$ depending only on $\mu$, $\lambda$, $\tau$,
and the Lipschitz character of $\Omega$.
We point out that if $A$ is Lipschitz continuous, 
the Rellich property follows readily from the Rellich type identities (see Lemma \ref{lemma-5.7.1}),
as in the case of Laplace's equation.
However,
 the proof for operators with H\"older continuous coefficients is more
involved.

As we pointed out above,
Theorem \ref{local-solvability} is a consequence of the following.

\begin{thm}\label{Holder-continuous-Rellich}
Let $\mathcal{L}=-\text{\rm div}(A\nabla )$ with $A\in\Lambda(\mu, \lambda,\tau)$
and $A^*=A$.
Let $\Omega$ be a bounded Lipschitz domain with $\text{\rm diam}(\Omega)\le (1/4)$
and connected boundary.
Then $\mathcal{L}$ has the Rellich property
in $\Omega$ with constant $C(\Omega)$ depending only on 
$\mu$, $\lambda$, $\tau$, and the Lipschitz character of $\Omega$.
\end{thm}

By translation we may assume that $0\in \Omega$ and thus $\Omega\subset [-1/4,1/4]^d$.
We divide the proof of Theorem \ref{Holder-continuous-Rellich} into three steps.

\begin{quote}
{\bf Step 1.} Establish the invertibility of $\pm (1/2)I +\mathcal{K}_A$
under the additional assumption that
\begin{equation}\label{additional-assumption-1}
\left\{ \aligned
& A \in C^1 ([-1/2, 1/2]^d\setminus \partial\Omega),\\
& |\nabla A(x)|\le C_0  \big\{ \text{dist}(x, \partial\Omega)\big\}^{\lambda_0-1}
\text{ for any } x\in [-1/2, 1/2]^d \setminus \partial\Omega,
\endaligned
\right.
\end{equation}
where $\lambda_0\in (0,1]$.
\end{quote}
Clearly, if $A\in C^1([-1/2,1/2]^d)$, then it satisfies (\ref{additional-assumption-1}).

We start out with two Rellich type identities for
the system $\mathcal{L}(u)=0$.

\begin{lemma}\label{Rellich-lemma-5.7}
Let $\Omega$ be a bounded Lipschitz domain.
Let $A\in \Lambda(\mu, \lambda,\tau)$ be such that $A^*=A$ and the condition
(\ref{additional-assumption-1}) holds.
Suppose that $u\in C^1(\overline{\Omega})$ and $\mathcal{L}(u)=0$ in $\Omega$.
Then
\begin{equation}\label{Rellich-identity-5.7-1}
\aligned
\int_{\partial\Omega} \langle h, n\rangle\, a_{ij}^{\alpha\beta} \frac{\partial u^\alpha}{\partial x_i}
\cdot \frac{\partial u^\beta}{\partial x_j}\, d\sigma
=&2\int_{\partial\Omega}
\left(\frac{\partial u}{\partial \nu}\right)^\alpha \cdot \frac{\partial u^\alpha}{\partial x_k} h_k\, d\sigma\\
& +\int_\Omega \text{\rm div}(h) \cdot
a_{ij}^{\alpha\beta} \, \frac{\partial u^\alpha}{\partial x_i}
\cdot \frac{\partial u^\beta}{\partial x_j}\,  dx\\
& +\int_\Omega h_k \frac{\partial}{\partial x_k} \Big\{ a_{ij}^{\alpha\beta}\Big\}\cdot
 \frac{\partial u^\alpha}{\partial x_i}
\cdot \frac{\partial u^\beta}{\partial x_j}\, dx\\
&-2\int_\Omega \frac{\partial h_k}{\partial x_i} \cdot
a_{ij}^{\alpha\beta} \, \frac{\partial u^\alpha}{\partial x_k}
\cdot \frac{\partial u^\beta}{\partial x_j}\, dx
\endaligned
\end{equation}
and
\begin{equation}
\label{Rellich-identity-5.7-2}
\aligned
\int_{\partial\Omega} \langle h, n\rangle\,  a_{ij}^{\alpha\beta} \frac{\partial u^\alpha}{\partial x_i}
\cdot \frac{\partial u^\beta}{\partial x_j}\, d\sigma
= & 2\int_{\partial\Omega}
\left\{ n_k \frac{\partial}{\partial x_i} -n_i \frac{\partial}{\partial x_k}\right\} u^\alpha
\cdot a_{ij}^{\alpha\beta} \frac{\partial u^\beta}{\partial x_j} h_k \, d\sigma\\
& -\int_\Omega \text{\rm div}(h) \cdot
a_{ij}^{\alpha\beta} \frac{\partial u^\alpha}{\partial x_i}
\cdot \frac{\partial u^\beta}{\partial x_j}\,  dx\\
& -\int_\Omega h_k \frac{\partial}{\partial x_k} \Big\{ a_{ij}^{\alpha\beta}\Big\}\cdot
 \frac{\partial u^\alpha}{\partial x_i}
\cdot \frac{\partial u^\beta}{\partial x_j}\, dx\\
& +2\int_\Omega \frac{\partial h_k}{\partial x_i} \cdot
a_{ij}^{\alpha\beta} \frac{\partial u^\alpha}{\partial x_k}
\cdot \frac{\partial u^\beta}{\partial x_j}\, dx,
\endaligned
\end{equation}
where $h=(h_1, \dots, h_d) \in C_0^1(\br^d; \br^d)$.
\end{lemma}

\begin{proof}
Using the assumption that $\mathcal{L}(u)=0$ in $\Omega$ and $A^*=A$,
one may verify that
$$
\aligned
& \frac{\partial}{\partial x_k}
\left\{ h_k a_{ij}^{\alpha\beta} \frac{\partial u^\alpha}{\partial x_i}\cdot \frac{\partial u^\beta}{\partial x_j}\right\}
-2\frac{\partial}{\partial x_i}
\left\{ h_k a_{ij}^{\alpha\beta} \frac{\partial u^\alpha}{\partial x_k}\cdot \frac{\partial u^\beta}{\partial x_j}\right\}\\
&
\quad =\text{\rm div} (h)\cdot a_{ij}^{\alpha\beta}
\frac{\partial u^\alpha}{\partial x_i}\cdot \frac{\partial u^\beta}{\partial x_j}
+h_k \frac{\partial}{\partial x_k}
\Big\{ a_{ij}^{\alpha\beta} \Big\} \cdot \frac{\partial u^\alpha}{\partial x_i}\cdot \frac{\partial u^\beta}{\partial x_j}\\
&\qquad\qquad 
-2 \frac{\partial h_k}{\partial x_i}\cdot 
a_{ij}^{\alpha\beta}
\frac{\partial u^\alpha}{\partial x_k}\cdot \frac{\partial u^\beta}{\partial x_j}.
\endaligned
$$
This gives (\ref{Rellich-identity-5.7-1}) by the divergence theorem.
Let $I$ and $J$ denote the left and right hand sides of (\ref{Rellich-identity-5.7-1}),
respectively.
The identity (\ref{Rellich-identity-5.7-2})
follows by rewritting $I=J$ as $I=2I-J$.
\end{proof}

By an approximation argument, Rellich identities
(\ref{Rellich-identity-5.7-1}) and (\ref{Rellich-identity-5.7-2})
continue to hold under the assumption that
$\mathcal{L}(u)=0$ in $\Omega$, $(\nabla u)^*\in L^2(\partial\Omega)$,
and $\nabla u$ exists n.t.\,on $\partial\Omega$.

\begin{lemma}\label{lemma-5.7.1}
Let $\Omega$ be a bounded Lipschitz domain with connected boundary.
Suppose that $0\in \Omega$ and $r_0=\text{\rm diam}(\Omega)\le (1/4)$.
Let $A\in \Lambda(\mu, \lambda,\tau)$ be such that $A^*=A$ and the condition
(\ref{additional-assumption-1}) holds.
Assume that $\mathcal{L}(u)=0$ in $\Omega$,
$(\nabla u)^*\in L^2 (\partial\Omega)$ and $\nabla u$ exists
n.t. on $\partial\Omega$.
Then
\begin{equation}\label{estimate-5.7.1}
\aligned
& \int_{\partial\Omega} |\nabla u|^2\, d\sigma
\le C\int_{\partial\Omega} \big|\frac{\partial u}{\partial\nu}\big|^2\, d\sigma
+C\int_\Omega \big(|\nabla A|+r_0^{-1}\big) |\nabla u|^2\, dx,\\
& \int_{\partial\Omega} |\nabla u|^2\, d\sigma
\le C\int_{\partial\Omega} |\nabla_{\rm tan} u|^2\, d\sigma
+C\int_\Omega \big(|\nabla A|+r_0^{-1}\big) |\nabla u|^2\, dx,
\endaligned
\end{equation}
where $C$ depends only on $\mu$ and  the Lipschitz character of $\Omega$.
\end{lemma}

\begin{proof}
Let ${h}\in C_0^1(\br^d; \br^d)$ be a vector field
such that supp$({h})\subset \{x: \text{dist}(x, \partial\Omega)<c r_0\}$,
$|\nabla {h}|\le Cr_0^{-1}$, and
$\langle{h}, n\rangle \ge c>0$ on $\partial\Omega$.
It follows from (\ref{Rellich-identity-5.7-1}) and (\ref{Rellich-identity-5.7-2}) that
\begin{equation}\label{local-Rellich}
\aligned
& \int_{\partial\Omega}
\langle {h}, n\rangle (A\nabla u\cdot \nabla u)\, d\sigma
=2 \int_{\partial\Omega} \langle {h}, \nabla u^\alpha\rangle
\left(\frac{\partial u}{\partial \nu}\right)^\alpha\, d\sigma + I_1,\\
&
\int_{\partial\Omega}
\langle {h}, n\rangle (A\nabla u\cdot \nabla u)\, d\sigma
=2\int_{\partial\Omega}
h_k a_{ij}^{\alpha \beta} \frac{\partial u^\beta}{\partial x_j}
\left(n_k \frac{\partial}{\partial x_i}
-n_i \frac{\partial}{\partial x_k}\right) u^\alpha\, d\sigma
+I_2,
\endaligned
\end{equation}
where 
$$
|I_1| +|I_2|
\le C \int_{\Omega}\big\{
 |\nabla {h}| + |{h}| |\nabla A|\big\} |\nabla u|^2\, dx,
$$
and $C$ depends only on $\mu$.
Estimates in (\ref{estimate-5.7.1}) follow from (\ref{local-Rellich})
by the Cauchy inequality (\ref{Cauchy}).
\end{proof}

\begin{remark}\label{local-exterior-remark}
{\rm 
Let $\mathcal{L}(u)=0$ in $(-1/2,1/2)^d\setminus \overline{\Omega}$.
Suppose that $(\nabla u)^*\in L^2(\partial\Omega)$ and $\nabla u$ exists
n.t. on $\partial\Omega$. Under the same conditions on $\Omega$ and $A$ as in
Lemma \ref{lemma-5.7.1}, we have
\begin{equation}\label{local-exterior-estimate}
\aligned
& \int_{\partial\Omega} |(\nabla u)_-|^2\, d\sigma
\le C\int_{\partial\Omega} |\left(\frac{\partial u}{\partial\nu}\right)_-|^2\, d\sigma
+C\int_{\Omega_-\cap [-1/2,1/2]^d}
 (|\nabla A|+r_0^{-1}) |\nabla u|^2\, dx,\\
& \int_{\partial\Omega} |(\nabla u)_-|^2\, d\sigma
\le C\int_{\partial\Omega} |(\nabla_{\rm tan} u)_-|^2\, d\sigma
+C\int_{\Omega_-\cap [-(1/2),1/2]^d} (|\nabla A|+r_0^{-1}) |\nabla u|^2\, dx,
\endaligned
\end{equation}
where $C$ depends only on $\mu$ and the Lipschitz character of $\Omega$.
The proof is similar to that of Lemma \ref{lemma-5.7.1}.
}
\end{remark}

\begin{lemma}\label{local-interior-lemma-1}
Under the same assumptions as in Lemma \ref{lemma-5.7.1}, we have
\begin{equation}\label{local-interior-estimate-1}
\aligned
\int_{\partial\Omega}
|\nabla u|^2\, d\sigma 
& \le C \Big\{ 1+ r_0^{2\lambda_0} \rho^{2\lambda_0-2}\Big\}
\int_{\partial\Omega}
\big|\frac{\partial u}{\partial\nu}\big|^2\, d\sigma
+C (\rho r_0 )^{\lambda_0} \int_{\partial\Omega} |(\nabla u)^*|^2\, d\sigma,\\
\int_{\partial\Omega}
|\nabla u|^2\, d\sigma 
& \le C \Big\{ 1+ r_0^{2\lambda_0} \rho^{2\lambda_0-2}\Big\}
\int_{\partial\Omega}
\big|\nabla_{\rm tan} u|^2\, d\sigma
+C (\rho r_0)^{\lambda_0} \int_{\partial\Omega} |(\nabla u)^*|^2\, d\sigma,
\endaligned
\end{equation}
where $0<\rho<1$ and $C$ depends only on $\mu$, 
the Lipschitz character of $\Omega$, and
$(\lambda_0, C_0)$ in (\ref{additional-assumption-1}).
\end{lemma}

\begin{proof} Write $\Omega=E_1 \cup E_2$, 
where 
$$
\aligned
&E_1=\{ x\in \Omega: \text{dist}(x,\partial\Omega)\le \rho r_0\},\\
& E_2=\{ x \in \Omega: \text{dist}(x,\partial\Omega)>\rho r_0\}.
\endaligned
$$
Using the condition (\ref{additional-assumption-1}), we obtain
\begin{equation}
\aligned
\int_\Omega |\nabla A ||\nabla u|^2\, dx
&\le C_0 \int_{E_1} \big\{ \text{dist}(x,\partial\Omega)\big\}^{\lambda_0-1} |\nabla u|^2\, dx
+C_0 (\rho r_0)^{\lambda_0-1} \int_{E_2} |\nabla u|^2\, dx\\
& \le C(\rho r_0)^{\lambda_0} \int_{\partial\Omega} |(\nabla u)^*|^2\, d\sigma
+C_0 (\rho r_0)^{\lambda_0-1} \int_\Omega |\nabla u|^2\, dx.
\endaligned
\end{equation}
This, together with (\ref{estimate-5.7.1}) and (\ref{Green's-estimate})
for $\Omega_+$,
gives
\begin{equation}\label{I-5}
\aligned
\|\nabla u\|_{L^2(\partial\Omega)}^2
\le C\, \big \|\frac{\partial u}{\partial\nu}\big\|_{L^2(\partial\Omega)}^2
& +C (1+r_0^{\lambda_0} \rho^{\lambda_0-1})\big\|\frac{\partial u}{\partial\nu}\big\|_{L^2(\partial\Omega)}
\|\nabla_{\rm tan} u\|_{L^2(\partial\Omega)}\\
&+C (\rho r_0)^{\lambda_0} \|(\nabla u)^*\|_{L^2(\partial\Omega)}^2.
\endaligned
\end{equation}
The first inequality in (\ref{local-interior-estimate-1})
follows from (\ref{I-5}) by the Cauchy inequality (\ref{Cauchy}).
The proof of the second inequality in (\ref{local-interior-estimate-1})
is similar.
\end{proof}

\begin{remark}
{\rm
Let $\mathcal{L}(u)=0$ in $\Omega_-$.
Suppose that $(\nabla u)^*\in L^2 (\partial\Omega)$, $(\nabla u)_-$
exists n.t. on $\partial\Omega$, and
$|u(x)|=O(|x|^{2-d})$ as $|x|\to\infty$.
In view of Remark \ref{local-exterior-remark} and (\ref{Green's-estimate})
for $\Omega_-$, 
the same argument as in the proof of 
Lemma \ref{local-interior-lemma-1} shows that
\begin{equation}\label{local-exterior-estimate-1}
\aligned
\int_{\partial\Omega}|(\nabla u)_-|^2 \, d\sigma 
& \le C \Big\{ 1+ r_0^{2\lambda_0} \rho^{2\lambda_0-2}\Big\}
\big\| \left( \frac{\partial u}{\partial\nu}\right)_-\big\|_2^2
+C (\rho r_0 )^{\lambda_0} \|(\nabla u)^*\|_2^2\\
& \qquad\qquad \qquad 
+ C (\rho r_0)^{\lambda_0 -1} \left|\average_{\partial\Omega} u \right|
\left| \int_{\partial \Omega} \left( \frac{\partial u}{\partial \nu}\right)_-\, d\sigma\right|,\\
\int_{\partial\Omega}|(\nabla u)_-|^2\, d\sigma 
& \le C \Big\{ 1+ r_0^{2\lambda_0} \rho^{2\lambda_0-2}\Big\}
\| \big( \nabla_{\rm tan} u)_-\|_2^2
+C (\rho r_0 )^{\lambda_0} \|(\nabla u)^*\|_2^2\\
& \qquad\qquad\qquad
+ C (\rho r_0)^{\lambda_0 -1} \left|\average_{\partial\Omega} u \right|
\left| \int_{\partial \Omega} \left( \frac{\partial u}{\partial \nu}\right)_-\, d\sigma\right|,
\endaligned
\end{equation}
for any $0<\rho<1$.
}
\end{remark}

The following theorem completes Step 1.

\begin{thm}{\label{step-one-theorem}}
Suppose that $\Omega$ and $A$ satisfy the conditions 
in Theorem \ref{Holder-continuous-Rellich}.
We further assume that $0\in \Omega$ and
$A$ satisfies (\ref{additional-assumption-1}).
Then $(1/2)I +\mathcal{K}_A$ and $-(1/2)I +\mathcal{K}_A $ are invertible
on $L^2_0(\partial\Omega; \mathbb{R}^m)$ and $L^2(\partial\Omega; \mathbb{R}^d)$,
respectively, and the estimates in (\ref{estimate-5.6-2}) 
hold with a constant $C$ depending only on
$\mu$, $\lambda$, $\tau$, the Lipschitz character of $\Omega$,
 and $(C_0, \lambda_0)$ in (\ref{additional-assumption-1}).
\end{thm}

\begin{proof}
Let  $u=\mathcal{S}(f)$ for some $f\in L^2_0(\partial\Omega; \mathbb{R}^m)$.
Recall that $\int_{\partial\Omega} \left(\frac{\partial u}{\partial\nu}\right)_-d\sigma=0$,
$(\nabla_{\rm tan} u)_- =(\nabla_{\rm tan} u)_+$, and
$\|(\nabla u)^*\|_{L^2(\partial\Omega)} \le C\,\| f\|_{L^2(\partial\Omega)}$.
Thus, by the second inequality in (\ref{local-exterior-estimate-1}), we obtain
\begin{equation}\label{I-5.6}
\| (\nabla u)_-\|_{L^2(\partial\Omega)}
 \le C\rho_1^{\lambda_0-1} \|(\nabla u)_+\|_{L^2(\partial\Omega)} +C\rho_1^{\lambda_0/2}
\| f\|_{L^2(\partial\Omega)}
\end{equation}
for any $0<\rho_1<1$.
Similarly, by the first inequality in (\ref{local-interior-estimate-1}),
\begin{equation}\label{I-5.7}
\|(\nabla u)_+\|_{L^2(\partial\Omega)} 
\le C\rho_2^{\lambda_0-1} \big\| \left(\frac{\partial u}{\partial \nu}\right)_+\big\|_{L^2(\partial\Omega)}
+C\rho_2^{\lambda_0/2} \| f\|_{L^2(\partial\Omega)}
\end{equation}
for any $0<\rho_2<1$. It follows from the jump relation (\ref{jump-relation}),
(\ref{I-5.6}) and (\ref{I-5.7}) that
\begin{equation}\label{I-5.8}
\aligned
\| f\|_{L^2(\partial\Omega)}
& \le \big\| \left(\frac{\partial u}{\partial \nu}\right)_+\big\|_{L^2(\partial\Omega)}
+\| \left(\frac{\partial u}{\partial \nu}\right)_-\|_{L^2(\partial\Omega)}\\
& \le C \rho_1^{\lambda_0-1}\rho_2^{\lambda_0-1}
\big\| \left(\frac{\partial u}{\partial \nu}\right)_+\big\|_{L^2(\partial\Omega)}
+ C\left\{ \rho_1^{\lambda_0-1}\rho_2 ^{\lambda_0/2} +\rho_1^{\lambda_0/2}\right\}
\| f\|_{L^2(\partial\Omega)}.
\endaligned
\end{equation}
We now choose $\rho_1\in (0,1)$ and then $\rho_2\in (0,1)$ so that
$$
C\,\big\{ \rho_1^{\lambda_0-1}\rho_2 ^{\lambda_0/2} +\rho_1^{\lambda_2/2}\big\}
\le (1/2).
$$
This gives
\begin{equation}\label{I-5.9}
\| f\|_2 \le C\,  \big\| \left(\frac{\partial u}{\partial \nu}\right)_+\big\|_{L^2(\partial\Omega)}
=C \, \| \big((1/2)I +\mathcal{K}_A \big) f\|_{L^2(\partial\Omega)}
\end{equation}
for any $f\in L^2_0(\partial\Omega; \mathbb{R}^m)$.
The same argument also shows that for any $f\in L^2_0(\partial\Omega, \mathbb{R}^m)$,
\begin{equation}\label{I-5.10}
\| f\|_{L^2(\partial\Omega)}
 \le C\,  \big\| \left(\frac{\partial u}{\partial \nu}\right)_-\big\|_{L^2(\partial\Omega)}
=C \, \| \big(-(1/2)I +\mathcal{K}_A \big) f\|_{L^2(\partial\Omega)}.
\end{equation}
The rest of the proof is the same as that of Lemma \ref{lemma-5.6.2}.
\end{proof}

\begin{remark}\label{local-single-layer-invertibility-remark-1}
{\rm
Let $f\in L^2(\partial\Omega;\br^m)$ and $u=\mathcal{S}(f)$. It follows from 
(\ref{local-interior-estimate-1}) and (\ref{local-exterior-estimate-1}) that
\begin{equation}\label{I-5.13}
\aligned
\| (\nabla u)_+\|_{L^2(\partial\Omega)}
  & \le C\rho_1^{\lambda_0 -1} \|\nabla_{\tan} u\|_{L^2(\partial\Omega)}
+C\rho_1^{\lambda_0/2} \| f\|_{L^2(\partial\Omega)},\\
\| (\nabla u)_-\|_{L^2(\partial\Omega)}
&\le C \rho_2^{\lambda_0-1} \|\nabla_{\tan} u\|_{L^2(\partial\Omega)} +C \rho_2^{\lambda_0/2} \| f\|_{L^2(\partial\Omega)}
+ C r_0^{-1}\| u\|_{L^2(\partial\Omega)}
\endaligned
\end{equation}
for any $\rho_1, \rho_2 \in (0,1)$. This, together with the jump relation,
implies
\begin{equation}\label{I-5.14}
\aligned
\| f\|_{L^2(\partial\Omega)}
& \le C\, \|\nabla_{\tan} \mathcal{S}(f)\|_{L^2(\partial\Omega)} +C r_0^{-1} \| \mathcal{S}(f)\|_{L^2(\partial\Omega)}\\
&\le C \, \| \mathcal{S}(f)\|_{W^{1,2}(\partial\Omega)}.
\endaligned
\end{equation}
Thus $\mathcal{S}: L^2(\partial\Omega; \mathbb{R}^m)
\to W^{1,2}(\partial\Omega; \mathbb{R}^m)$ is
one-to-one.
A continuity argument similar to that in the proof of Lemma \ref{lemma-5.6.2}
shows that the operator is in fact invertible.
}
\end{remark}

\bigskip

\begin{quote}
{\bf Step 2.}
 Given any $A\in \Lambda(\mu, \lambda, \tau)$
and $\Omega$ such that $A^*=A$, $0\in \Omega$ and $r_0=$diam$(\Omega)\le (1/4)$,
construct $\widetilde{A}\in \Lambda(\mu, \lambda_0, \tau_0)$ 
with $\lambda_0$ and $\tau_0$ depending only
on $\mu$, $\lambda$, $\tau$, and the Lipschitz character of $\Omega$,
such that
\begin{equation}\label{part-2.1}
\widetilde{A}(x)= A(x) \qquad \text{ if } \text{dist}(x, \partial \Omega)
\le cr_0,
\end{equation}
and such that the operators
\begin{equation}\label{part-2.2}
\aligned
 (1/2)I +\mathcal{K}_{\widetilde{A}}: 
L_0^2(\partial\Omega; \mathbb{R}^m) & \to L^2_0(\partial\Omega; \mathbb{R}^m),\\
 -(1/2)I +\mathcal{K}_{\widetilde{A}}: 
L^2(\partial\Omega; \mathbb{R}^m) & \to L^2(\partial\Omega; \mathbb{R}^m),\\
 \mathcal{S}_{\widetilde{A}}: L^2(\partial\Omega; \mathbb{R}^m)
& \to W^{1,2}(\partial\Omega; \mathbb{R}^m)
\endaligned
\end{equation}
are isomorphisms and the operator norms of their inverses are
bounded by constants depending only on $\mu$, $\lambda$, $\tau$, and the Lipschitz
character of $\Omega$.
\end{quote}

\begin{lemma}\label{construction-A-bar-lemma}
Given $A\in \Lambda (\mu, \lambda, \tau)$ and a Lipschitz domain $\Omega$
such that {\rm diam}$(\Omega)\le (1/4)$ and $0\in \Omega$.
There exists $\bar{A}\in \Lambda(\mu, \lambda_0, \tau_0)$
such that $\bar{A}=A$ on $\partial\Omega$ and $\bar{A}$
satisfies the condition (\ref{additional-assumption-1}),
where $\lambda_0\in (0,\lambda]$, $\tau_0$, and $C_0$ in (\ref{additional-assumption-1})
depend only on $\mu$, $\lambda$, $\tau$ and the Lipschitz
character of $\Omega$.
In addition, $\big(\bar{A}\big)^* =\bar{A}$ if $A^*=A$.
\end{lemma}

\begin{proof}
By periodicity it suffices to define
$\bar{A}=(\bar{a}_{ij}^{\alpha\beta})$ 
on $[-1/2,1/2]^d$. This is done as follows.
On $\Omega$ we define $\bar{A}$ to be the Poisson extension of  $A$ on 
$\partial\Omega$; i.e., 
$\bar{a}_{ij}^{\alpha\beta}$
is harmonic in $\Omega$ and $\bar{a}_{ij}^{\alpha\beta}
=a_{ij}^{\alpha\beta}$ on $\partial\Omega$,
 for each $i,j,\alpha,\beta$.
On $[-1/2,1/2]^d\setminus \overline{\Omega}$, we define
$\bar{A}$ to be the harmonic function in 
$(-1/2,1/2)^d\setminus \overline{\Omega}$
with boundary data $\bar{A}=A$ on $\partial\Omega$ and
$\bar{A}=I$ on $\partial [-1/2,1/2]^d$.
Note that the latter boundary condition allows us to extend $\bar{A}$
to $\mathbb{R}^d$ by periodicity.

Since $\bar{a}_{ij}^{\alpha\beta}\xi_i^\alpha\xi_j^\beta$ is harmonic in $(-1/2,1/2)^d
\setminus \partial\Omega$, the ellipticity
condition (\ref{ellipticity})
for $\bar{A}$ follows readily from the maximum principle.
By the solvability of Laplace's equation in Lipschitz domains
with H\"older continuous data  
 there exists $\lambda_1\in (0,1)$, depending only on
the Lipschitz character of $\Omega$, such that
$
\bar{A}\in C^{\lambda_0}(\overline{\Omega})
$
and $\bar{A}\in C^{\lambda_0} ([-1/2,1/2]^d\setminus \Omega)$,
where $\lambda_0=\lambda$ if $\lambda<\lambda_1$, and $\lambda_0=\lambda_1$
if $\lambda\ge \lambda_1$.
It follows that $\bar{A}\in C^{\lambda_0}(\mathbb{R}^d)$.
Using the well known interior estimates for harmonic functions, one may also show that
$$
|\nabla \bar{A}(x)|\le C_0 \big\{ \text{dist} (x, \partial\Omega)\big\}^{\lambda_0-1} \quad \text{
for  } x\in [-3/4,3/4]^d\setminus \partial\Omega,
$$
where $C_0$ depends only on $\mu$, $\lambda$, $\tau$, and the Lipschitz
character of $\Omega$.
Thus we have proved that $\bar{A}\in \Lambda (\mu, \lambda_0, \tau)$
and satisfies the condition (\ref{additional-assumption-1}).
Clearly, $(\bar{A})^* =\bar{A}$ if $A^*=A$.
\end{proof}

Let $\eta\in C_0^\infty (-1/2,1/2)$ such that $0\le \eta\le 1$
and $\eta=1$ on $(-1/4,1/4)$.
Given $A\in \Lambda(\mu, \lambda, \tau)$ with $A^*=A$, 
define
\begin{equation}\label{approximation-3}
A^\rho (x) =\eta \left(\frac{\delta(x)}{\rho}\right) A (x)
+ \left[ 1- \eta \left(\frac{\delta(x)}{\rho}\right)\right] \bar{A}(x)
\end{equation}
for $x\in [-1/2,1/2]^d$, 
where $\rho\in (0,1/8)$, $\delta(x)=\text{dist}(x, \partial\Omega)$ and
$\bar{A} (x)$ is the matrix constructed in Lemma \ref{construction-A-bar-lemma}.
Extend $A^\rho$ to $\mathbb{R}^d$ by periodicity.
Clearly, $A^\rho$ satisfies the ellipticity condition (\ref{ellipticity})
and $(A^\rho)^* =A^\rho$.

\begin{lemma}\label{lemma-6.2}
Let $A^\rho$ be defined by (\ref{approximation-3}). Then 
\begin{equation}\label{I-6.2}
 \| A^\rho -\bar{A}\|_\infty \le C\rho^{\lambda_0} \quad \text{ and } \quad
 \| A^\rho-\bar{A}\|_{C^{0, \lambda_0}(\mathbb{R}^d)}
\le C,
\end{equation}
where $C$ depends only on $\mu$, $\lambda$, $\tau$, and the Lipschitz character of $\Omega$.
\end{lemma}

\begin{proof}
Let $H^\rho =A^\rho-\bar{A}$.
Given $x\in [-1/2,1/2]^d$, let $z\in \partial\Omega$ such that
$|x-z|=\delta(x)$. Since $A(z)=\bar{A}(z)$, we have
$$
\aligned
|A(x)-\bar{A}(x)|
& \le |A(x)-A(z)| +|\bar{A}(z)-\bar{A}(x)|\\
&\le C |x-z|^{\lambda_0}=C \big\{ \delta(x)\big\}^{\lambda_0}.
\endaligned
$$
It follows that
$$
\aligned
|H^\rho(x)|  &\le C \theta (\rho^{-1}\delta(x)) \big\{ \delta(x)\big\} ^{\lambda_0}\\
&= C\theta (\rho^{-1} \delta(x)) \big\{ \rho^{-1} \delta(x)\big\}^{\lambda_0} \rho^{\lambda_0}
\le C \rho^{\lambda_0}.
\endaligned
$$
This gives $\| A^\rho -A\|_\infty \le C\rho^{\lambda_0}$.

Next we show $|H^\rho (x)-H^\rho (y)|\le C |x-y|^{\lambda_0}$ for any $x,y\in \mathbb{R}^d$.
Since $\|H^\rho\|_\infty \le C\rho^{\lambda_0}$, we may assume that
$|x-y|\le \rho$.
Note that $H^\rho=0$ on $[-1/2, 1/2]^d\setminus [-3/8,3/8]^d$. 
Thus it is enough to consider the case where
$x, y\in [-1/2,1/2]^d$.
We may further assume that $\delta(x)\le \rho$ or $\delta (y)\le \rho$.
For otherwise, $H^\rho(x)=H^\rho (y)=0$ and there is nothing to show.
Finally, suppose that $\delta(y)\le \rho$. Then
$$
\aligned
|H^\rho(x)-H^\rho (y)|
 & \le 
\theta (\rho^{-1} \delta(x)) |\big(A(x)-\bar{A}(x)) -( A(y)-\bar{A}(y)\big)|\\
& \qquad\qquad +|A(y)-\bar{A}(y)| |\theta (\rho^{-1} \delta(x)) -\theta (\rho^{-1} \delta (y))|\\
& 
\le C|x-y|^{\lambda_0}
+ C \big\{ \delta (y)\big\}^{\lambda_0} |x-y|\cdot \rho^{-1}\\
& \le C|x-y|^{\lambda_0} +C\rho^{\lambda_0-1} |x-y|\\
&\le C |x-y|^{\lambda_0}.
\endaligned
$$
The proof for the case $\delta(x)\le \rho$ is the same.
\end{proof}

It follows from Lemma \ref{lemma-6.2} that for $\rho\in (0,1/4)$,
\begin{equation}\label{I-7.2}
\| A^\rho -\bar{A}\|_{C^{\lambda_0/2}(\mathbb{R}^d)}
\le C\rho^{\lambda_0/2}.
\end{equation}
Since $A^\rho=A=\bar{A}$ on $\partial\Omega$, we may deduce from 
Theorem \ref{theorem-5.2-2} that
\begin{equation}\label{I-6.3}
\| \mathcal{K}_{A^\rho} -\mathcal{K}_{\bar{A}}\|_{L^2\to L^2}
\le C \| A^\rho-\bar{A}\|_{C^{\lambda_0/2}(\mathbb{R}^d)}
\le C \rho^{\lambda_0/2}
\end{equation}
for any $\rho\in (0,1/4)$.
Note that by Lemma \ref{construction-A-bar-lemma} and Theorem \ref{step-one-theorem},
 the operator $(1/2)I+\mathcal{K}_{\bar{A}}$
is invertible on $L^2_0(\partial\Omega; \mathbb{R}^m)$ and $\|\big((1/2)I +\mathcal{K}_{\bar{A}}\big)^{-1}\|
_{L_0^2\to L_0^2}\le C$.
Write $$
(1/2)I +\mathcal{K}_{A^\rho} =(1/2)I +\mathcal{K}_{\bar{A}}
+(\mathcal{K}_{A^\rho}-\mathcal{K}_{\bar{A}}).
$$
In view of (\ref{I-6.3}), one may choose $\rho>0$ depending only on $\mu$, $\lambda$,
$\tau$, and the Lipschitz character of $\Omega$ so that
$$
\| \big((1/2)I +\mathcal{K}_{\bar{A}}\big)^{-1} (\mathcal{K}_{A^\rho}-\mathcal{K}_{\bar{A}})\|
_{L_0^2\to L_0^2} \le 1/2.
$$
It follows that
$(1/2)I+\mathcal{K}_{A^\rho}$
is an isomorphism on $L^2_0(\partial\Omega; \mathbb{R}^m)$ and 
$$
\|\big((1/2)I +\mathcal{K}_{A^\rho}\big)^{-1}\|
_{L^2_0\to L^2_0}\le 2
\|\big((1/2)I +\mathcal{K}_{\bar{A}}\big)^{-1}\|
_{L^2_0\to L_0^2}\le 2C.
$$
Similar arguments show that  it is possible to choose $\rho$ depending only on
$\mu$, $\lambda$, $\tau$, and the Lipschitz character of $\Omega$
such that $-(1/2)I +\mathcal{K}^*_{A^\rho}: L^2(\partial\Omega; \mathbb{R}^m)
\to
L^2(\partial\Omega; \mathbb{R}^m)
$
and $\mathcal{S}_{A^\rho}: L^2(\partial\Omega; \mathbb{R}^m) \to
W^{1,2} (\partial\Omega; \mathbb{R}^m)$
are isomorphisms and the operator norms of their inverses are bounded
by constants depending only on $\mu$, $\lambda$, $\tau$,
and the Lipschitz character of $\Omega$.
Let $\widetilde{A}=A^\rho$.
Note that if $\text{dist}(x, \partial\Omega)\le (1/4)\rho$,
$A^\rho(x)=A(x)$. 
This completes Step 2.

\bigskip

\begin{quote}
{\bf Step 3.} Use a perturbation argument to complete the proof.
\end{quote}

\begin{lemma}\label{perturbation-lemma}
Let $A=\big(a_{ij}^{\alpha\beta}\big)$ and $B
=\big(b_{ij}^{\alpha\beta}\big)\in \Lambda(\mu, \lambda, \tau)$.
Let $\Omega$ be a bounded Lipschitz domain with connected boundary.
Suppose that 
$$
A=B \quad \text{ in } \big\{ x\in \Omega: \text{\rm dist}(x,\partial\Omega)
\le c_0r_0\big\}
$$
 for some $c_0>0$, where $r_0=\text{\rm diam}(\Omega)$.
Assume that $\mathcal{L}^A =-\text{\rm div}(A\nabla )$ has the Rellich property
in $\Omega$ with constant $C_0$.
Then $\mathcal{L}^B=-\text{\rm div}(B\nabla)$
has the Rellich property in $\Omega$
with constant $C_1$, where $C_1$ depends only on 
$\mu$, $\lambda$, $\tau$, $c_0$, $C_0$, and the Lipschitz character of $\Omega$.
\end{lemma}

\begin{proof}
Suppose that $\mathcal{L}^B (u)=0$ in $\Omega$, $(\nabla u)^*\in L^2(\partial\Omega)$
and $\nabla u$ exists n.t.\,on $\partial\Omega$.
Let $\varphi\in C_0^\infty(\mathbb{R}^d)$ such that $|\nabla \varphi|\le Cr_0^{-1}$,
$\varphi=1$ on $\{ x\in \mathbb{R}^d: \text{dist}(x,\partial\Omega)\le (1/4)c_0r_0\}$,
and
$\varphi=0$ on $\{ x\in \mathbb{R}^d: \text{dist}(x,\partial\Omega)\ge (1/2)c_0r_0\}$.
Let $\bar{u}=\varphi(u-E)$, where $E$ is the mean value of $u$ over $\Omega$.
Note that 
$$
\big( \mathcal{L}^A (\bar{u})\big)^\alpha
=-\partial_i \left\{ a_{ij}^{\alpha\beta} (\partial_j \varphi) (u-E)^\beta\right\}
-a_{ij}^{\alpha\beta} (\partial_i\varphi)(\partial_j u^\beta),
$$
where we have used the fact that $\mathcal{L}^A(u)=\mathcal{L}^B (u)=0$ on 
$\{ x\in \Omega: \text{dist}(x,\partial\Omega)<c_0r_0\}$.
It follows from the proof of  (\ref{Green's-representation-formula}) that
\begin{equation}\label{I-6.7}
\aligned
\bar{u}(x) &=\mathcal{S}_{A}
\left(\frac{\partial \bar{u}}{\partial \nu_{A}}\right)
-\mathcal{D}_{A} (\bar{u}) +v(x)\\
&=w(x)+ v(x),
\endaligned
\end{equation}
where $v$ satisfies
\begin{equation}\label{I-6.8}
\aligned
|\nabla v(x)|
&\le C \int_\Omega
|\nabla_x\nabla_y \Gamma (x,y;A)|\, |\nabla \varphi|\, |u-E|\,dy\\
& \qquad\qquad 
+C \int_\Omega |\nabla_x \Gamma (x,y;A)\, ||\nabla\varphi| \, |\nabla u|\, dy.
\endaligned
\end{equation}
This, together with estimates $|\nabla_x\Gamma(x,y;A)|\le C \, |x-y|^{1-d}$ and
$|\nabla_x\nabla_y\Gamma(x,y;A)|\le C\,  |x-y|^{-d}$, implies that
if $x\in \Omega$ and $\text{dist}(x, \partial\Omega)
\le (1/5)c_0r_0$,
\begin{equation}\label{estimate-of-v}
\aligned
|\nabla v(x)|^2
  & \le  \frac{C}{r_0^d}
\int_\Omega |\nabla u|^2\, dy\\
& \le Cr_0^{1-d} \big\|\frac{\partial u}{\partial\nu_{B}}\big\|_{L^2(\partial\Omega)} \|\nabla_{\tan} u\|_{L^2(\partial\Omega)},
\endaligned
\end{equation}
where we have used (\ref{Green's-estimate}) for the last inequality.

Next, note that
$\mathcal{L}^A (w)=0$ in $\Omega$, where $w=\bar{u}-v$.
Using (\ref{estimate-of-v}) and the assumption $(\nabla u)^*\in L^2
(\partial\Omega)$, we may deduce that
$(\nabla w)^*\in L^2(\partial\Omega)$ and
$\nabla w$ exists n.t. on $\partial\Omega$.
Since $\mathcal{L}^A$ has the Rellich property, this implies that
\begin{equation}\label{I-6.9}
\aligned
\|\nabla w\|_{L^2(\partial\Omega)}
  & \le C_0 \big\|\frac{\partial w}{\partial \nu_{A}}\big\|_{L^2(\partial\Omega)}\\
& \le C\left\{ \big \| \frac{\partial u}{\partial \nu_{B}} \big\|_{L^2(\partial\Omega)} 
+\|\nabla v\|_{L^2(\partial\Omega)}\right\} \\
&\le C \left\{ 
\big\| \frac{\partial u}{\partial \nu_{B}} \big\|_{L^2(\partial\Omega)} 
+\big\|\frac{\partial u}{\partial\nu_{B}}\big\|_{L^2(\partial\Omega)}
^{1/2} \|\nabla_{\tan} u\|_{L^2(\partial\Omega)}^{1/2}\right\},
\endaligned
\end{equation}
where we used (\ref{estimate-of-v}) in the last inequality.
Using (\ref{estimate-of-v}) again, we obtain
\begin{equation}\label{I-10}
\aligned
\| \nabla u \|_{L^2(\partial\Omega)}
& \le C \Big\{
\|\nabla w \|_{L^2(\partial\Omega)}
+\| \nabla v \|_{L^2(\partial\Omega)}\Big\}\\
&\le C \left\{
\big\|\frac{\partial u}{\partial \nu_{B}}\big\|_{L^2(\partial\Omega)}
+
\big\|\frac{\partial u}{\partial\nu_{B}}\big\|_{L^2(\partial\Omega)}
^{1/2} \|\nabla_{\tan} u\|_{L^2(\partial\Omega)}^{1/2}
\right\}.
\endaligned
\end{equation}
The desired estimate $\|\nabla u\|_{L^2(\partial\Omega)}
 \le C\, \big\|\frac{\partial u}{\partial \nu_{B}}\big\|_{L^2(\partial\Omega)}$
follows readily from (\ref{I-10}) by the Cauchy inequality (\ref{Cauchy}).
The proof of $\|\nabla u\|_{L^2(\partial\Omega)}
 \le C\, \|\nabla_{\tan} u\|_{L^2(\partial\Omega)}$ is similar and left to the reader.
\end{proof}

Finally we give the proof of Theorem \ref{Holder-continuous-Rellich}.

\begin{proof}[\bf Proof of Theorem \ref{Holder-continuous-Rellich}]
By Step 2, there exists $\widetilde{A}\in \Lambda (\mu, \lambda_0, \tau_0)$
such that $\widetilde{A}=A$ in $\{ x\in \mathbb{R}^d: \text{dist}(x, \partial\Omega)\le cr_0\}$
and $(1/2)I +\mathcal{K}_{\widetilde{A}}: L_0^2(\partial\Omega, \mathbb{R}^m)
\to L_0^2(\partial\Omega, \mathbb{R}^m)$, 
$\mathcal{S}_{\widetilde{A}}: L^2(\partial\Omega, \mathbb{R}^m)\to
W^{1,2} (\partial\Omega, \mathbb{R}^m)$
are isomorphisms.
Moreover, the operator norms of their inverses are bounded by
constants depending only on $\mu$, $\lambda$, $\tau$, and the Lipschitz character of $\Omega$.
It follows that the $L^2$ Neumann and regularity problems for
$\widetilde{\mathcal{L}}(u)=
\text{div}( \widetilde{A}\nabla u)=0$ in $\Omega$
are well-posed
and the solutions satisfy
$\|(\nabla u)^*\|_{L^2(\partial\Omega)}
 \le C\,\|\frac{\partial u}{\partial \nu}\|_{L^2(\partial\Omega)}$
and $\|(\nabla u)^*\|_{L^2(\partial\Omega}
 \le C\, \|\nabla_{\tan} u\|_{L^2(\partial\Omega)}$ with 
constant $C$ depending only on $\mu$, $\lambda$, $\tau$, and
the Lipschitz character of $\Omega$.
In particular,
 the operator $\widetilde{\mathcal{L}}$
has the Rellich property in $\Omega$ with constant $C(\Omega)$
depending only on $\mu$, $\lambda$, $\tau$, and
the Lipschitz character of $\Omega$.
By Lemma \ref{perturbation-lemma} this implies that
$\mathcal{L}$ has the Rellich property in $\Omega$ with a constant $C$
depending only on $\mu$, $\lambda$, $\tau$, and
the Lipschitz character of $\Omega$.
The proof is complete.
\end{proof}

%
%
%
%
%
%
%
%
%

\section{Rellich estimates for large scales}\label{section-5.8}
 
Let $\psi:\br^{d-1}\to \br$ be a Lipschitz function such that
$\psi(0)=0$ and $\|\nabla\psi\|_\infty\le M$. Recall that
$$
\aligned
Z_r & = Z(r, \psi)=\Big\{ (x^\prime, x_d)\in \br^d: \ |x^\prime|<r \ \text{ and } \ \psi(x^\prime)<x_d<10\sqrt{d}\, (M+1)r \Big\},\\
\Delta_r  &  =\Delta (\psi, r)=\Big\{ (x^\prime, \psi(x^\prime))\in \br^d: \ |x^\prime|<r \Big\}.
\endaligned
$$
It follows from Theorem \ref{Holder-continuous-Rellich} that if $A$ satisfies the ellipticity condition (\ref{s-ellipticity}),
the smoothness condition (\ref{smoothness}), and the symmetry condition $A^*=A$, and if $0<r\le 10$,
then the operator $\mathcal{L}=-\text{div}(A\nabla)$ has the Rellich property
in $\Omega= Z_r$ with a constant $C(\Omega)$
depending only on $\mu$, $\lambda$, $\tau$, and $M$.
In this section
we show that the same conclusion in fact holds for any $0<r<\infty$,  with $C(\Omega)$ independent of $r$,
if $A$ also satisfies the periodicity condition.
 
For $x=(x^\prime, \psi(x^\prime))$, let
$$
(w)^*_\rho (x)=\sup \Big\{ | w(y)|:\
|y-x|<C|y_d-\psi(y^\prime)| \ \text{ and }\
\psi(y^\prime)<y_d <\psi(y^\prime) +\rho \Big\}.
$$

\begin{thm}\label{periodic-Rellich-1}
Let $\mathcal{L}=-\text{\rm div}(A\nabla)$ with $A\in \Lambda(\mu, \lambda, \tau)$ and $A^*=A$.
Suppose that $\mathcal{L}(u)=0$ in $Z_{8r}$ for some $r>0$,
where $u\in C^1({Z_{8r}})$, 
$(\nabla u)^*_{r} \in L^2 (\Delta_{6r})$ and $\nabla u$ exists n.t. on $\Delta_ {6r}$.
Then
\begin{equation}\label{periodic-Rellich-estimate}
\aligned
\int_{\Delta_r} |\nabla u|^2\, d\sigma
 & \le C\int_{\Delta_{4r}} \big|\frac{\partial u}{\partial \nu} \big|^2\, d\sigma
+\frac{C}{r} \int_{Z_{4r}}
|\nabla u|^2\, dx,\\
\int_{\Delta _r} |\nabla u|^2\, d\sigma
 & \le C\int_{\Delta_{4r}} |\nabla_{\tan} u |^2\, d\sigma
+\frac{C}{r} \int_{Z_{4r}}
|\nabla u|^2\, dx,
\endaligned
\end{equation}
where $C$ depends only on $\mu$, $\lambda$, $\tau$, and $M$.
\end{thm}

As we mentioned above, the estimates in (\ref{periodic-Rellich-estimate}) hold
for $0<r\le 10$ with $C=C(\mu,\lambda, \tau, M)>0$. 
To treat the case $r>10$, we shall apply the small-scale estimates for $r=1$ and
reduce the problem to the control of integral of $|\nabla u|^2$ on a boundary layer
$$
\Big\{ (x^\prime, x_d)\in \rd: \, |x^\prime|<r \text{ and } \psi(x^\prime)< x_d< \psi(x^\prime) +1 \Big\}.
$$
We use the error estimates in $H^1$ for a two-scale expansion obtained in Chapter \ref{chapter-C}
to handle $\nabla u$ over the boundary layer.

\begin{lemma}\label{Neumann-Rellich} 
Suppose that $A$ is 1-periodic and satisfies (\ref{s-ellipticity}).
Also assume that $A^*=A$.
Let $u\in H^1(Z_{3r})$ be a weak solution of $\text{\rm div}(A\nabla u)=0$ in $Z_{3r}$
with $\frac{\partial u}{\partial \nu}=g$ on $\Delta_{3r}$, where $r>3$. Then
\begin{equation}\label{Neumann-Rellich-estimate}
\int_{\substack{|x^\prime|< r\\
\psi(x^\prime)<x_d <\psi(x^\prime)+1
}}
|\nabla u|^2\, dx
\le C\int_{\Delta_{2r}}
\big|g|^2\, d\sigma
+\frac{C}{r}
\int_{Z_{2r}} |\nabla u|^2\, dx,
\end{equation}
where $C$ depends only on $\mu$ and $M$.
\end{lemma}

\begin{proof} 
The lemma follows from the error estimate (\ref{c-2-3-00s}) in Section \ref{section-c-3}.
To see this, we consider the function $u_\e (x) =\e u(x/\e)$ in $\Omega$, where $\e=r^{-1}<1$,
$$
\Omega=Z(t, \psi_\e),
$$
$t\in (1,2)$ and $\psi_\e (x^\prime)=\e  \psi (x^\prime/\e)$. Observe that $\psi_\e (0)=0$ and $\|\nabla \psi_\e\|_\infty
=\|\nabla \psi\|_\infty\le M$.
It follows that the Lipschitz character of $\Omega$ depends only on $M$.

Since $\mathcal{L}_\e (u_\e)=0$   in $ \Omega$,
in view of (\ref{c-2-3-00s}), we obtain 
$$
\aligned
\| \nabla u_\e -\nabla u_0 \|_{L^2(\Omega_{4\e})}
 & \le \| w_\e\|_{H^1(\Omega)}\\
&\le C \sqrt{\e} \| g_\e\|_{L^2(\partial\Omega)},
\endaligned
$$
where $w_\e$ is defined by (\ref{w-c-2}), $g_\e (x)=g ( x/\e)$, and
$\Omega_{4\e} =\{ x\in \Omega:  \text{dist}(x, \partial\Omega)< 4\e \}$.
Moreover, since $\mathcal{L}_0 (u_0)=0$ in $\Omega$, 
we may use the nontangential-maximal-function estimate for the operator $\mathcal{L}_0$ to show that
\begin{equation}\label{l-10}
\aligned
\| \nabla u_0 \|_{L^2(\Omega_{4\e})}
 &\le C \sqrt{\e} \| (\nabla u_0)^*\|_{L^2(\partial\Omega)}\\
&\le C \sqrt{\e} \| g_\e\|_{L^2(\partial\Omega)},
\endaligned
\end{equation}
where $C$ depends only on $\mu$ and $M$.
It follows that
$$
\| \nabla u_\e \|_{L^2(\Omega_{4\e})}
\le C \sqrt{\e} \| g_\e\|_{L^2(\partial \Omega)}.
$$
By a change of variables this yields 
$$
\aligned
\int_{\substack{|x^\prime|<r\\
\psi(x^\prime)<x_d <\psi(x^\prime)+1
}}
|\nabla u|^2\, dx& \le C \int_{\partial Z(tr, \psi)} \Big|\frac{\partial u}{\partial \nu}\Big|^2\,d\sigma \\
&\le C \int_{\Delta_{2r}} |g|^2 \, d\sigma
+ C \int_{\partial Z_{tr}\setminus \Delta_{tr}} |\nabla u|^2\,  d\sigma
\endaligned
$$
for any $t\in (1,2)$.
By integrating the inequality above in $t$ over the interval $(1, 2)$ we obtain the first inequality in (\ref{Neumann-Rellich-estimate}.
\end{proof}

\begin{lemma}\label{regularity-Rellich} Let $r>1$.
Suppose that $A$ satisfies the same conditions as in Lemma \ref{Neumann-Rellich}.
Let $u\in H^1(Z_{3r})$ be a weak solution of $\text{\rm div} (A\nabla u)=0$ in $Z_{3r}$
with $u=f$ on $\Delta_{3r}$, where $r>3$. Then 
\begin{equation}\label{regularity-Rellich-estimate}
\int_{\substack{
 |x^\prime|<r\\
\psi(x^\prime)<x_d <\psi(x^\prime)+1
}}
|\nabla u|^2\, dx
\le C\int_{\Delta(2r)}
\big|\nabla_{\tan} f|^2\, d\sigma
+\frac{C}{r}
\int_{Z(3r)} |\nabla u|^2\, dx,
\end{equation}
where $C$ depends only on $\mu$ and $M$.
\end{lemma}

\begin{proof}
As in the case of Lemma \ref{Neumann-Rellich},
the estimate follows from the error estimate (\ref{c-2-2-00})
by a rescaling argument.
Indeed, let $u_\e$ and $\Omega$ be defined as in the proof of Lemma \ref{Neumann-Rellich}.
The same argument shows that
$$
\| \nabla u_\e\|_{L^2(\Omega_{4\e})}
\le C\sqrt{\e} \| u_\e\|_{H^1(\partial\Omega)}.
$$
Since $\partial\Omega$ is connected, by subtracting a constant from $u_\e$ and using Poincar\'e inequality,
we obtain
$$
\|\nabla u_\e\|_{L^2(\Omega_{4\e})}
\le C \sqrt{\e} \| \nabla_{\tan} u_\e\|_{L^2(\partial\Omega)}.
$$
By a change of variables this gives
$$
\aligned
\int_{\substack{|x^\prime|<r\\
\psi(x^\prime)<x_d <\psi(x^\prime)+1
}}
|\nabla u|^2\, dx& \le C \int_{\partial Z(tr, \psi)} |\nabla_{\tan}  u|^2\,d\sigma \\
&\le C \int_{\Delta_{2r}} |\nabla_{\tan} f|^2 \, d\sigma
+ C \int_{\partial Z_{tr}\setminus \Delta_{tr}} |\nabla u|^2\,  d\sigma
\endaligned
$$
for any $t\in (1,2)$.
By integrating the inequality above in $t$ over the interval $(1, 2)$ we obtain the 
second inequality in (\ref{Neumann-Rellich-estimate}).
\end{proof}

\begin{proof}[\bf Proof of Theorem \ref{periodic-Rellich-1}]
We may assume $r>3$.
By covering $\Delta_r$ with surfaces balls of small radius $c(M)$ on 
$\{ (x^\prime, \psi(x^\prime)): x^\prime \in \br^{d-1} \}$
and
using the first inequality in (\ref{periodic-Rellich-estimate}) on each small surface balls, we obtain 
$$
\int_{\Delta_r} |\nabla u|^2\, d\sigma
\le C \int_{\Delta_{2r}} \Big|\frac{\partial u}{\partial \nu} \Big|^2\, d\sigma
+ C \int_{\substack{|x^\prime|<r\\
\psi(x^\prime)<x_d <\psi(x^\prime)+1
}}
|\nabla u|^2\, dx.
$$
This, together with Lemma \ref{Neumann-Rellich}, gives the first inequality in (\ref{periodic-Rellich-estimate}).
The second inequality in (\ref{periodic-Rellich-estimate}) follows from Lemma \ref{regularity-Rellich}
in a similar fashion.
\end{proof}

%
%
%
%
%
%
%
%
%

\section{$L^2$ boundary value problems}\label{section-5.9}

Let $\Omega$ be a bounded Lipschitz domain in $\br^d$, $d\ge 3$,
with connected boundary. In this section we establish the 
well-posedness of the $L^2$ Dirichlet, Neumann, and  regularity problems
with uniform nontangential-maximal-function
estimates for $\mathcal{L}_\varep (u_\varep)=0$ in $\Omega$.

\begin{thm}\label{periodic-Rellich-property-theorem}
Suppose that $A\in \Lambda(\mu, \lambda, \tau)$ and $A^*=A$.
Let $\mathcal{L}=-\text{\rm div}(A\nabla)$ and
$\Omega$ be a bounded Lipschitz domain in $\br^d$, $d\ge 3$ with connected boundary.
Then  the operators 
$$
\aligned
(1/2)I +\mathcal{K}_A: L^2_0(\partial\Omega; \br^m)
 & \to L^2_0(\partial\Omega; \br^m),\\
-(1/2)I +\mathcal{K}_A: L^2(\partial\Omega; \br^m)
& \to L^2(\partial\Omega, \br^m),\\
\mathcal{S}_A : L^2(\partial\Omega, \br^m) & \to H^1(\partial\Omega, \br^m)
\endaligned
$$
are isomorphisms and the operator norms of their inverses
are bounded by constants depending only on $\mu$, $\lambda$, $\tau$,
and the Lipschitz character of $\Omega$.
\end{thm}

\begin{proof}
Suppose that $\mathcal{L}(u)=0$ in $\Omega$,
$(\nabla u)^*\in L^2(\partial\Omega)$, and $\nabla u$ exists n.t.\,on $\partial\Omega$.
Let $z\in \partial\Omega$.
It  follows from Theorem \ref{periodic-Rellich-1} by a change of the coordinate system that
\begin{equation}\label{5.9.1-1}
\int_{B(z, cr_0)\cap\partial\Omega}
|\nabla u|^2\, d\sigma
\le C\, \int_{\partial\Omega} \Big|\frac{\partial u}{\partial\nu}\Big|^2\, d\sigma
+ \frac{C}{r_0} \int_\Omega |\nabla u|^2\, dx,
\end{equation}
where $r_0=\text{diam}(\Omega)$ and
 $C$ depends at most on $\mu$, $\lambda$, $\tau$ and the Lipschitz character of $\Omega$.
 By covering $\partial\Omega$ with a finite number of balls $\{ B(z_k, cr_0): k=1, \cdots, N \}$,
 where $z_k\in  \partial\Omega$,
we obtain
$$
\aligned
\int_{\partial\Omega}
|\nabla u|^2\, d\sigma
& \le C \int_{\partial\Omega} \Big|\frac{\partial u}{\partial \nu}\Big|^2\, d\sigma
+\frac{C}{r_0} \int_\Omega |\nabla u|^2\, dx\\
&\le C \int_{\partial\Omega} \Big|\frac{\partial u}{\partial \nu}\Big|^2\, d\sigma
+C \big\|\frac{\partial u}{\partial\nu}\big\|_{L^2(\partial\Omega)}
\|\nabla_{\tan} u\|_{L^2(\partial\Omega)},
\endaligned
$$
where we have used (\ref{Green's-estimate}) for the second inequality.
By the Cauchy inequality (\ref{Cauchy}), this gives 
$$
\|\nabla u\|_{L^2(\partial\Omega)}\le C\, \big\|\frac{\partial u}{\partial\nu}\big\|_{L^2(\partial\Omega)}.
$$
The estimate 
$$
\|\nabla u\|_{L^2(\partial\Omega)}\le C\, \|\nabla_{\tan} u\|_{L^2(\partial\Omega)}
$$
 may be proved in a similar manner.
Thus we have proved that $\mathcal{L}$ has the Rellich property  in any 
Lipschitz domain $\Omega$ with connected boundary, and the constant $C=C(\Omega)$
in (\ref{Rellich-property})
depends only on $\mu$, $\lambda$, $\tau$, and the 
Lipschitz character of $\Omega$.
Clearly, the same is true when $A$ is replaced by $sA+(1-s)I$ for $0\le s\le 1$.
In view of Remark \ref{remark-5.6-3} this implies that
the operator norms of
$\big((1/2)I +\mathcal{K}_A\big)^{-1}$ on $L^2_0(\partial\Omega;\br^m)$,
$\big(-(1/2)I +\mathcal{K}_A\big)^{-1}$ on $L^2(\partial\Omega; \br^m)$,
and $\mathcal{S}_A^{-1}: H^1(\partial\Omega; \br^m)
\to L^2(\partial\Omega; \br^m)$ are bounded by
constants depending only on $\mu$, $\lambda$, $\tau$, and the 
Lipschitz character of $\Omega$.
\end{proof}

\begin{thm}[$L^2$ Neumann problem]
\label{L-2-periodic-Neumann-theorem}
Let $\mathcal{L}_\varep=-\text{\rm div}(A(x/\varep)\nabla)$
with $A\in \Lambda(\mu, \lambda,\tau)$ and $A^*=A$.
Let $\Omega$ be a bounded Lipschitz domain in $\br^d$, $d\ge 3$ with connected boundary.
Then for any $g\in L^2_0(\partial\Omega; \br^m)$,
there exists $u_\varep\in C^1(\Omega;\br^m)$, unique up to constants, such that
$\mathcal{L}_\varep (u_\varep)=0$ in $\Omega$,
$(\nabla u_\varep)^*\in L^2(\partial\Omega)$ and
$
\frac{\partial u_\varep}{\partial \nu_\varep}=g$ n.t.\,on $\partial\Omega$.
Moreover, the solution $u_\varep$ satisfies the estimate
$$
\|(\nabla u_\varep)^*\|_{L^2(\partial\Omega)}
\le C\, \| g\|_{L^2(\partial\Omega)},
$$
 and may be represented by a single
layer potential $\mathcal{S}_\varep(h_\varep)$ with $h_\varep\in L^2_0(\partial\Omega; \br^m)$
and 
$$
\| h_\varep\|_{L^2(\partial\Omega)} \le C\, \| g\|_{L^2(\partial\Omega)}.
$$
The constant $C$ depends only on $\mu$, $\lambda$, $\tau$, and
the Lipschitz character of $\Omega$.
\end{thm}

\begin{thm}[$L^2$ Dirichlet problem]
\label{L-2-periodic-Dirichlet-theorem}
Assume $A$ and $\Omega$ satisfy the same conditions as in 
Theorem \ref{L-2-periodic-Neumann-theorem}.
Then for any $f\in L^2(\partial\Omega; \br^m)$,
there exists a unique $u_\varep \in C^1(\Omega;\br^m)$
such that $\mathcal{L}_\varep (u_\varep)=0$ in $\Omega$,
$(u_\varep)^*\in L^2(\partial\Omega)$ and
$u_\varep =f$ n.t.\,on $\partial\Omega$.
Moreover, the solution $u_\varep$ satisfies the estimate
$$
\|(u_\varep)^*\|_{L^2(\partial\Omega)}
\le C\, \|f\|_{L^2(\partial\Omega)},
$$
 and may be represented 
by a double layer potential $\mathcal{D}_\varep (h_\varep)$
with $\|h_\varep\|_{L^2(\partial\Omega)}
\le C\,\| f\|_{L^2(\partial\Omega)}$.
The constant $C$ depends only on $\mu$, $\lambda$, $\tau$, and
the Lipschitz character of $\Omega$.
\end{thm}

\begin{thm}[$L^2$ regularity problem]
\label{L-2-periodic-regularity-theorem}
Assume $A$ and $\Omega$ satisfy the same conditions as in 
Theorem \ref{L-2-periodic-Neumann-theorem}.
Then for any $f\in W^{1,2} (\partial\Omega; \br^m)$,
there exists a unique $u_\varep \in C^1(\Omega;\br^m)$
such that $\mathcal{L}_\varep (u_\varep)=0$ in $\Omega$,
$(\nabla u_\varep)^*\in L^2(\partial\Omega)$ and
$u_\varep =f$ n.t.\,on $\partial\Omega$.
Moreover, the solution $u_\varep$ satisfies the estimate
$$
\|(\nabla u_\varep)^*\|_{L^2(\partial\Omega)}
\le C\, \|\nabla_{\tan} f\|_{L^2(\partial\Omega)},
$$ and may be represented 
by a single layer potential $\mathcal{S}_\varep (h_\varep)$
with $\|h_\varep\|_{L^2(\partial\Omega)}
\le C\,\| f\|_{W^{1,2} (\partial\Omega)}$.
The constant $C$ depends only on $\mu$, $\lambda$, $\tau$, and
the Lipschitz character of $\Omega$.
\end{thm}

\begin{proof}[\bf Proof of Theorems \ref{L-2-periodic-Neumann-theorem}, \ref{L-2-periodic-Dirichlet-theorem},
and \ref{L-2-periodic-regularity-theorem}]
By a simple rescaling we may assume that $\varep=1$.
In this case the existence of solutions in Theorems \ref{L-2-periodic-Neumann-theorem},
 \ref{L-2-periodic-Dirichlet-theorem},
and \ref{L-2-periodic-regularity-theorem}
is a direct consequence of the invertibility of $(1/2)I +\mathcal{K}_A$,
$-(1/2)I +\mathcal{K}_A$, and $\mathcal{S}_A$ as well as estimates of the operator
norms of their inverses, established in Theorem \ref{periodic-Rellich-property-theorem}.
The uniqueness for the $L^2$ Neumann and regularity
problems follows readily from the Green's identity (\ref{Green's-identity-5.6}),
while the uniqueness for the $L^2$ Dirichlet problem
may be proved by constructing a matrix of Green functions,
as in the proof of Theorems \ref{theorem-5.6.1}
and \ref{theorem-5.6.2}.
\end{proof}

\begin{remark}\label{connect-remark}
{\rm
The method of layer potentials also applies to Lipschitz domains  $\Omega$
whose boundaries may 
not be connected.
In this case, under the assumption that $A\in \Lambda(\mu, \lambda, \tau)$ and
$A^*=A$, the operator $(1/2)I +\mathcal{K}_A $ continues to be an isomorphism
on $L^2_0(\partial\Omega, \br^m)$, while
$-(1/2)I+\mathcal{K}_A$ is an isomorphism on the subspace
$$
L^2_{0^\prime}(\partial\Omega; \br^m)
=\Big\{ f\in L^2_{0^\prime} (\partial\Omega; \br^m):\
\int_{\partial\Omega_j^\prime} f\, d\sigma =0 \text{ for } j=1, \dots, \ell \Big\},
$$
where $\Omega^\prime_j$, $j=1, \dots, \ell$ are bounded connected components of $\Omega_-
=\br^d\setminus \overline{\Omega}$.
Also, $\mathcal{S}_A: L^2(\partial\Omega; \br^m)\to H^1(\partial\Omega; \br^m)$ is an isomorphism.
}
\end{remark}

We end this section with some $L^2$ estimates for local solutions with Dirichlet or Neumann conditions.
Recall that if $\Omega$ is a bounded Lipschitz domain and $u\in C(\Omega)$,
the nontangential maximal function $(u)^*$ is defined by
$$
( u)^*(z)
=\sup \Big\{
| u(x)|:\
x\in\Omega \text{ and } |x-z|< C_0 \, \text{dist}(x, \partial\Omega)\Big\}
$$
for $z\in \partial\Omega$, where $C_0=C(\Omega)>1$ is sufficiently large.
Let
\begin{equation}\label{definition-M-1-2}
\aligned
& \mathcal{M}_r^1 (u) (z)
=\sup \Big\{
| u (x)|:\
x\in\Omega, \ |x-z|\le r \text{ and } |x-z|< C_0 \, \text{dist}(x, \partial\Omega)\Big\},\\
& \mathcal{M}^2_{r} (u) (z)
=\sup \Big\{
|u(x)|:\
x\in\Omega, \ |x-z|>r \text{ and } |x-z|< C_0 \, \text{dist}(x, \partial\Omega)\Big\},
\endaligned
\end{equation}
where $0<r<c_0\, \text{diam}(\Omega)$.
Clearly,
\begin{equation}\label{nontangential-by-M-1}
(u)^* (z)=\max \Big\{ \mathcal{M}_r^1 (u)(z), \mathcal{M}_r^2 (u) (z)\Big\}.
\end{equation}

\begin{thm}\label{localized-L-2-theorem}
Let $\mathcal{L}_\varep=-\text{\rm div}(A(x/\varep)\nabla)$
with $A\in \Lambda(\mu, \lambda, \tau)$ and $A^*=A$.
Let $\Omega$ be a bounded Lipschitz domain in $\br^d$, $d\ge 3$.
Assume that $\mathcal{L}_\varep (u_\varep)=0$ in $B(z,4r)\cap\Omega$
for some $z\in \partial\Omega$ and $0<r<c_0\, \text{\rm diam}(\Omega)$.

\begin{enumerate}

\item

Suppose that $\mathcal{M}_r^1(\nabla u_\varep)\in L^2(B(z, 3r)\cap\partial\Omega;\br^m) $
 and $\frac{\partial u_\varep}{\partial\nu_\varep}
=g\in L^2(B(z,4r)\cap\partial\Omega;\br^m)$ n.t. on $B(z,4r)\cap\partial\Omega$.
Then
\begin{equation}\label{local-L-2-Neumann-estimate}
\int_{B(z,r)\cap\partial\Omega}
|\mathcal{M}_r^1 (\nabla u_\varep)|^2\, d\sigma
\le C \, \int_{B(z,3r)\cap\partial\Omega}
|g|^2\, d\sigma
+ \frac{C}{r} \int_{B(z,4r)\cap \Omega} |\nabla u_\varep|^2\, dx.
\end{equation}

\item

Suppose that $\mathcal{M}_r^1 (\nabla u_\varep)\in L^2 (B(z, 3r)\cap\partial\Omega;\br^m) $
and $u_\varep=f\in W^{1,2}(B(z,4r)\cap\partial\Omega;\br^m)$
 n.t. on $B(z,4r)\cap\partial\Omega$.
Then
\begin{equation}\label{local-L-2-regularity-estimate}
\int_{B(z,r)\cap\partial\Omega}
|\mathcal{M}_r^1 (\nabla u_\varep)|^2\, d\sigma
\le C \, \int_{B(z,3r)\cap\partial\Omega}
|\nabla_{\tan} f |^2\, d\sigma
+ \frac{C}{r} \int_{B(z,4r)\cap \Omega} |\nabla u_\varep|^2\, dx.
\end{equation}

\item

Suppose that $\mathcal{M}_r^1 (u_\varep)\in L^2(B(z, 3r)\cap\partial\Omega;\br^m) $ and 
$u_\varep=f\in L^2 (B(z,4r)\cap\partial\Omega;\br^m)$
 n.t. on $B(z,4r)\cap\partial\Omega$.
Then
\begin{equation}\label{local-L-2-Dirichlet-estimate}
\int_{B(z,r)\cap\partial\Omega}
|\mathcal{M}_r^1 (u_\varep)|^2\, d\sigma
\le C \, \int_{B(z,3r)\cap\partial\Omega}
| f |^2\, d\sigma
+ \frac{C}{r} \int_{B(z,4r)\cap \Omega} | u_\varep|^2\, dx.
\end{equation}

\end{enumerate}
The constants $C$ in (\ref{local-L-2-Neumann-estimate}), (\ref{local-L-2-regularity-estimate})
and (\ref{local-L-2-Dirichlet-estimate}) depend only on $\mu$, $\lambda$, $\tau$, and the Lipschitz
character of $\Omega$.
\end{thm}

\begin{proof}
We give the proof of (\ref{local-L-2-Neumann-estimate}).
The estimates (\ref{local-L-2-regularity-estimate}) and (\ref{local-L-2-Dirichlet-estimate})
may be proved in a similar manner.

Let 
$$
\aligned
Z_r &=\big\{ (x^\prime, x_d)\in \br^d:\ |x^\prime |<r \text{ and } \psi(x^\prime)<x_d< 10\sqrt{d} (M+1) r\big\},\\
\Delta_r &= \big\{ (x^\prime, \psi(x^\prime))\in \br^d:\
|x^\prime|<r\big\},
\endaligned
$$
where $\psi:\br^{d-1}\to \br$ is a Lipschitz function with
$\psi(0)=0$ and $\|\nabla \psi\|_\infty\le M$.
Suppose that $\mathcal{L}_\varep (u_\varep)=0$ in $Z_{5r}$,
$u_\varep\in H^1(Z_{5r};\br^m)$, and $\frac{\partial u_\varep}{\partial\nu_\varep}
=g$ on $\Delta _{3r}$.
We will show that
\begin{equation}\label{local-L-2-Neumann-1}
\int_{\Delta _r }
|\mathcal{M}_r^1 (\nabla u_\varep)|^2\, d\sigma
\le C \, \int_{\Delta_{3r}}
|g|^2\, d\sigma
+ \frac{C}{r} \int_{Z_{3r}} |\nabla u_\varep|^2\, dx.
\end{equation}
By translation and rotation as well as a simple covering argument,
this implies the estimate (\ref{local-L-2-Neumann-estimate}).
To see (\ref{local-L-2-Neumann-1}), we first assume $u_\varep\in C^1(\overline{Z_{3r}}; \br^m)$.
Let
 $(w)_{Z}^*$ denote the nontangential maximal function of $w$ with respect
to the Lipschitz domain $Z_{sr}$, where $s\in (2,3)$.
By applying the $L^2$ estimate in Theorem \ref{L-2-periodic-Neumann-theorem} in 
$Z_{sr}$, we obtain
$$
\aligned
\int_{\partial Z_{sr}}
|(\nabla u_\varep)^*_{Z}|^2\, d\sigma
 & \le C \int_{\partial Z_{sr}}
\Big|\frac{\partial u_\varep }{\partial\nu_\varep}\Big|^2\, d\sigma\\
& \le C\, \int_{\Delta_{3r}} |g|^2\, d\sigma
+\le C\, 
\int_{ \partial Z_{sr}\setminus \Delta_{3r}} 
|\nabla u_\varep|^2\, d\sigma,
\endaligned
$$
where $C$ depends only on $\mu$, $\lambda$, $\tau$, and $M$.
This leads to
\begin{equation}\label{5.9.5-1}
\int_{\Delta _r} | \mathcal{M}^1_{r} (\nabla u_\varep) |^2\, d\sigma
\le C\int_{\Delta_{3r}} |g|^2\, d\sigma
+ C\, \int_{\partial Z_{sr}\setminus \Delta_{3r}} 
|\nabla u_\varep|^2\, d\sigma.
\end{equation}
The estimate (\ref{local-L-2-Neumann-1}) now follows by
integrating both sides of (\ref{5.9.5-1}) in $s$ over the interval
$(2,3)$. Finally, to get rid of the assumption $u_\varep \in C^1(\overline{Z(3r)}$, we apply the estimate (\ref{local-L-2-Neumann-1}) to
 the function $v_\varep (x^\prime,x_d)=u_\varep (x^\prime, x_d +\delta)$ and then let $\delta\to 0^+$.
The proof is finished by a simple limiting argument.
\end{proof}

%
%
%
%
%
%

\section{$L^2$ estimates in arbitrary Lipschitz domains}\label{section-5.11}

In Theorems \ref{L-2-periodic-Neumann-theorem}, \ref{L-2-periodic-Dirichlet-theorem}, and \ref{L-2-periodic-regularity-theorem}, we 
solve the $L^2$ boundary value problems for $\mathcal{L}_\varep (u_\varep)=0$ in $\Omega$, assuming that
$d\ge 3$ and $\partial\Omega$ is a bounded Lipschitz domain with {\it connected} boundary. 
Although the method of layer potentials may be applied to a general Lipschitz domain in $\br^d$, $d\ge 2$, 
we will show in this section that  the general case follows from the localized $L^2$ estimates
in Theorem \ref{localized-L-2-theorem} by some approximation argument.
The method of descending is used to handle the case $d=2$.

\begin{thm}[$L^2$ Neumann problem]\label{L-2-Neumann-theorem-5.11}
Suppose that $A\in \Lambda(\mu, \lambda, \tau)$ and $A^*=A$.
Let $\Omega$ be a bounded Lipschitz domain in $\br^d$, $d\ge 2$.
Given $g\in L_0^2(\partial\Omega;\br^m)$, let $u_\varep\in H^1(\Omega;\br^m)$ 
be the weak solution to the Neumann problem:
$\mathcal{L}_\varep (u_\varep)=0$ in $\Omega$ and $\frac{\partial u_\varep}{\partial\nu_\varep}=g$ 
on $\partial\Omega$, given by Theorem \ref{theorem-1.1-3}.
Then the nontangential limits of $\nabla u_\varep$ exist a.e.\,on $\partial\Omega$, and
$\|(\nabla u_\varep)^*\|_{L^2(\partial\Omega)} \le C\, \| g\|_{L^2(\partial\Omega)}$, where $C$ depends only on
$\mu$, $\lambda$, $\tau$, and the Lipschitz character of $\Omega$. 
\end{thm}

\begin{proof}
By dilation we may assume that diam$(\Omega)=1$. We may also assume that $\int_\Omega u_\varep=0$.
Suppose that $d\ge 3$. To establish the estimate $\|(\nabla u_\varep)^*\|_{L^2(\partial\Omega)}
\le C\, \| g\|_{L^2(\partial\Omega)}$, we first consider the case $u_\varep \in C^1(\overline{\Omega})$. 
By covering $\partial\Omega$ with coordinate cylinders we may deduce from the estimate (\ref{local-L-2-Neumann-estimate}) and interior  Lipschitz estimates  that
\begin{equation}\label{5.9-10-1}
\aligned
\int_{\partial\Omega} |(\nabla u_\varep)^*|^2\, d\sigma
& \le C \int_{\partial\Omega} |g|^2\, d\sigma
+C\int_\Omega |\nabla u_\varep|^2\, dx\\
& \le C \, \int_{\partial\Omega} |g|^2\, d\sigma,
\endaligned
\end{equation}
where we have used the energy estimate $\|\nabla u_\varep\|_{L^2(\Omega)} \le C\,  \| g\|_{L^2(\partial\Omega)}$.

Next consider the case $g\in C_0^\infty (\br^d;\br^m)$.
We choose a sequence of smooth domains $\{ \Omega_\ell\}$ such that
$\Omega_\ell \downarrow \Omega$.  Let  $u_{\varep, \ell} \in H^1(\Omega_\ell;\br^m)$ be the weak solution
 to the Neumann problem for $\mathcal{L}_\varep (u_{\varep, \ell})=0$
  in $\Omega_\ell$ with boundary data $g-\average_{\partial\Omega_\ell} g$.
 Since $\Omega_\ell$ and $g$ are smooth, we have $u_{\varep, \ell}\in C^1(\overline{\Omega_\ell})$.
 It follows from (\ref{5.9-10-1}) that
 \begin{equation}\label{5.9-10-2}
 \|(\nabla u_{\varep, \ell} )^*\|_{L^2(\partial\Omega)}  \le C\, \| (\nabla u_{\varep, \ell})^*\|_{L^2(\partial\Omega_\ell)}
 \le C\, \| g\|_{L^2(\partial\Omega_\ell)}.
 \end{equation}
 Assume that $\int_{\Omega_\ell} u_{\varep, \ell}=0$.
 Since $\| u_{\varep, \ell}\|_{H^1(\Omega_\ell)} \le C\, \| g\|_{L^2(\partial\Omega_\ell)}$, 
 the sequence $\{ u_{\varep, \ell}\}$ is bounded in $H^1(\Omega;\br^m)$.
 As a result, by passing to a subsequence, we may assume that $u_{\varep, \ell}\to w_\varep$ weakly in $H^1(\Omega;\br^m)$,
 as $\ell \to \infty$.
 Note that for any $\varphi\in C^\infty_0(\br^d;\br^m)$,
 $$
 \aligned
& \int_{\partial\Omega_\ell} \left(g-\average_{\partial\Omega_\ell} g\right) \cdot \varphi\, d\sigma  \to \int_{\partial\Omega} g\cdot \varphi\, d\sigma,\\
 &\int_{\Omega_\ell\setminus \Omega} A(x/\varep)\nabla u_{\varep, \ell}\cdot \nabla \varphi\, dx
  \to 0,\\
 & \int_{\Omega} A(x/\varep)\nabla u_{\varep, \ell} \cdot \nabla \varphi\, dx \to
  \int_{\Omega} A(x/\varep)\nabla w_\varep\cdot \nabla \varphi\, dx,
  \endaligned
  $$
as $\ell\to \infty$. This implies that $w_\varep\in H^1(\Omega;\br^m)$ is a weak solution of the Neumann problem
for $\mathcal{L}_\varep (w_\varep)=0$ in $\Omega$ with boundary data $g$.
Since $\int_\Omega w_\varep=0$, we obtain $u_\varep =w_\varep$ in $\Omega$.
Using interior Lipschitz estimates and the observation  that $u_{\varep, \ell}\to u_\varep$ strongly in $L^2(\Omega;\br^m)$, we may deduce that
$\nabla u_{\varep, \ell} \to \nabla u_\varep$ uniformly on any compact subset of $\Omega$.
In view of (\ref{5.9-10-2}) this yields
\begin{equation}\label{5.9-10-2-1}
\|\mathcal{M}_\delta^2(\nabla u_\varep)\|_{L^2(\partial\Omega)} \le C \, \| g\|_{L^2(\partial\Omega)}.
\end{equation}
Letting $\delta\to 0$ in (\ref{5.9-10-2-1}) and using Fatou's Lemma, we obtain 
$$
\| (\nabla u_\varep)^*\|_{L^2(\partial\Omega)}
\le C\, \| g\|_{L^2(\partial\Omega)}.
$$

Suppose now that $g\in L^2_0(\partial\Omega;\br^m)$. 
We choose a sequence of functions $\{ g_\ell\}$ in $C_0^\infty(\br^d;\br^m)$
such that $g_\ell \to g$ in $L^2(\partial\Omega; \br^m)$ and $\int_{\partial\Omega} g_\ell =0$.
Let $v_{\varep, \ell}$ be the weak solution of the Neumann problem for $\mathcal{L}_\varep (v_{\varep, \ell})
=0$ in $\Omega$
with boundary data $g_\ell$ and $\int_\Omega v_{\varep, \ell} =0$. 
Since 
$$
\| v_{\varep, \ell} -u_\varep\|_{H^1(\Omega)} \le C\, \| g_\ell -g\|_{L^2(\partial\Omega)} \to 0, \quad \text{ as } \ell \to \infty,
$$
by interior Lipschitz estimates, we see that $v_{\varep, \ell}\to u_\varep$ uniformly on any compact subset of $\Omega$.
Note that 
\begin{equation}\label{5.9-10-2-2}
\| (\nabla v_{\varep, \ell})^*\|_{L^2(\partial\Omega)} \le C\, \| g_\ell \|_{L^2(\partial\Omega)}.
\end{equation}
 As before, by a simple limiting argument,
 this leads to the estimate (\ref{5.9-10-2-1}) and hence to 
 $$
 \| (\nabla u_\varep)^*\|_{L^2(\partial\Omega)}
\le C\, \| g\|_{L^2(\partial\Omega)}.
$$

To show that the nontangential limits of $\nabla u_\varep$ exist a.e. on $\partial\Omega$,
consider the set 
$$
E(T)=\Big\{ z\in \partial\Omega: (\nabla u_\varep)^*(z)\le T\Big\}.
$$
Since $(\nabla u_\varep)^*(z)<\infty$ for a.e. $z\in \partial\Omega$, 
it suffices to show that $\nabla u_\varep$ exist a.e.\,on $E(T)$ for
each fixed $T>1$.
Fix $z_0\in E(T)$ and $\rho>0$ small. We construct a bounded Lipschitz domain $\widetilde{\Omega}$
with connected boundary, such that
$$
\widetilde{\Omega}\subset \Big\{ x\in \Omega: x\in  \gamma (z) \text{ for some } z\in E(T)\Big\}
$$
and
$$
\partial\widetilde{\Omega}\cap \partial\Omega\cap B(z_0, \rho)
=E(T)\cap B(z_0,\rho),
$$
where $\gamma (z)=\big\{ x\in \Omega:\, |x-z|<C_0 \, \text{dist}(x, \partial\Omega) \big\}$.
Note that $|\nabla u_\varep|\le T$ on $\widetilde{\Omega}$.
Let $v_\varep$ be the solution of the $L^2$ Neumann problem for $\mathcal{L}_\varep (v_\varep)=0
$ in $\widetilde{\Omega}$
with boundary data $h=\frac{\partial u_\varep}{\partial\nu_\varep}$, given by Theorem \ref{L-2-periodic-Neumann-theorem}.
By the uniqueness of weak solutions of the Neumann problem in $H^1(\Omega;\br^m)$,
 $u_\varep-v_\varep$ is constant in $\widetilde{\Omega}$.
Since the nontangential limits of $\nabla v_\varep$ exist a.e.\,on $\partial\widetilde{\Omega}$,
we may conclude that $\nabla u_\varep$ exist a.e. on $E(T)\cap B(Q_0,  \rho)$ and hence a.e.\,on $\partial\Omega$.

Finally, we use the method of descending to treat the case $d=2$.
Consider the function $v_\varep (x, t)=u_\varep (x)$ in $\Omega_1=\Omega\times (0,1)$, a bounded Lipschitz domain in $\br^{3}$.
Observe that $v_\varep\in H^1(\Omega_1; \br^m)$ is a weak solution of the Neumann problem for
$$
\left( \mathcal{L}_\varep -\frac{\partial^2}{\partial t^2}\right) v_\varep =0\quad \text{ in } \Omega_1,
$$
with boundary data $\frac{\partial v_\varep}{\partial\nu_\varep}=g$ on $\partial\Omega\times (0,1)$, and 
$\frac{\partial v_\varep}{\partial \nu_\varep}=0$ on $\Omega \times \{t=0\}$ and on $\Omega\times \{t=1\}$.
It is not hard to see that the desired results for $u_\varep$ in $\Omega$ follow from the corresponding results for $v_\varep$
in $\Omega_1$. This completes the proof.
\end{proof}
 
\begin{thm}[$L^2$ regularity problem]\label{L-2-regularity-theorem-5.11}
Suppose that $A\in  \Lambda(\mu, \lambda, \tau)$ and $A^*=A$.
Let $\Omega$ be a bounded Lipschitz domain in $\br^d$, $d\ge 2$.
Given $f\in H^1(\partial\Omega;\br^m)$, let $u_\varep\in H^1(\Omega;\br^m)$ 
be the weak solution to the Dirichlet problem:
$\mathcal{L}_\varep (u_\varep)=0$ in $\Omega$ and $u_\varep=f$ on $\partial\Omega$, given by Theorem \ref{theorem-1.1-2}.
Then  the nontangential limits of $u_\varep$ and $\nabla u_\varep$ exist a.e.\,on $\partial\Omega$, $u_\varep=f$ n.t. on $\partial\Omega$, and
\begin{equation}\label{5.11-2-0}
\|(\nabla u_\varep)^*\|_{L^2(\partial\Omega)}
+r_0^{-1} \|(u_\varep)^*\|_{L^2(\partial\Omega)} \le C\, \Big\{ \| \nabla_{\tan} f \|_{L^2(\partial\Omega)}
+r_0^{-1} \| f\|_{L^2(\partial\Omega)} \Big\},
\end{equation}
where $r_0=$\text{\rm diam}$(\Omega)$ and $C$ depends only on $\mu$, $\lambda$, $\tau$, and the Lipschitz character of $\Omega$..
\end{thm}

\begin{proof}
By dilation we may assume that $r_0=1$. 
The case $d=2$ can be handled by the method of descending, as in the proof of Theorem \ref{L-2-Neumann-theorem-5.11}.
 We will assume $d\ge 3$.
Furthermore, by the approximation argument as well as the argument
 for the existence of nontangential limits in the proof of Theorem 
\ref{L-2-Neumann-theorem-5.11}, it suffices to
prove the estimate (\ref{5.11-2-0}) for $f\in C_0^\infty(\br^d;\br^m)$. 

Let $\{\Omega_\ell\}$ be a sequence of smooth domains such that $\Omega_\ell \downarrow \Omega$.
Let $u_{\varep, \ell}\in H^1(\Omega_\ell;\br^m)$ be the weak solution of 
$\mathcal{L}_\varep (u_{\varep, \ell})=0$ in $\Omega$ with boundary data $u_{\varep, \ell}=f$ on $\partial\Omega$.
Since $\Omega_\ell$ and $f$ are smooth, $u_{\varep, \ell}\in C^1(\overline{\Omega_\ell})$.
It follows from Theorem \ref{localized-L-2-theorem} that 
\begin{equation}\label{5.11-2-1}
\aligned
& \|(\nabla u_{\varep, \ell})^*  \|_{L^2(\partial\Omega_\ell )}   +\|(u_{\varep, \ell})^*\|_{L^2(\partial\Omega_\ell )}\\
& \qquad\le C \left\{  \| f\|_{H^1(\partial\Omega_\ell)} + \| \nabla u_{\varep, \ell}\|_{L^2(\Omega_\ell)}
+\| u_{\varep, \ell}\|_{L^2(\Omega_\ell)}\right\} \\
&\qquad \le C\, \| f\|_{H^1(\partial\Omega_\ell)},
\endaligned
\end{equation}
where we have used $\| u_{\varep, \ell}\|_{H^1(\Omega_\ell)} \le C\, \| f\|_{H^1(\partial\Omega_\ell)}$.
Hence,
\begin{equation}\label{5.11-2-2}
 \|(\nabla u_{\varep, \ell})^*  \|_{L^2(\partial\Omega)}   +\|(u_{\varep, \ell})^*\|_{L^2(\partial\Omega )}
 \le C \, \| f\|_{H^1(\partial\Omega_\ell)}.
 \end{equation}
 Note that 
 $$
 \| u_{\varep, \ell}\|_{H^1(\Omega)}\le
  \| u_{\varep, \ell}\|_{H^1(\Omega_\ell)}\le  C\, \| f\|_{H^1(\partial\Omega_\ell)}\le C_f<\infty.
 $$
 By passing to a subsequence we may assume that $u_{\varep, \ell}\to w_\varep$ weakly in $H^1(\Omega)$, as $\ell \to \infty$.
 Clearly, $\mathcal{L}_\varep (w_\varep) =0$ in $\Omega$.
 Since $u_{\varep, \ell} \to w_\varep$ strongly in $L^2(\Omega)$, by interior Lipschitz estimates, $u_{\varep, \ell}$ converges to $w_\varep$
 uniformly on any compact subset of $\Omega$. This, together with (\ref{5.11-2-2}), implies that
 \begin{equation}\label{5.11-2-3}
 \| (\nabla w_\varep)^*\|_{L^2(\partial\Omega)} +\| (w_\varep)^*\|_{L^2(\partial\Omega)}
 \le C\, \| f\|_{H^1(\partial\Omega)},
 \end{equation}
 by a limiting argument.
 
 Finally, note that the trace operator from $H^1(\Omega;\br^m)$ to $L^2(\partial\Omega;\br^m)$ is compact.
 It follows that $u_{\varep, \ell}|_{\partial\Omega} \to w_\varep|_{\partial\Omega}$ strongly in $L^2(\partial\Omega;\br^m)$, as
 $\ell \to \infty$.
 However, using (\ref{5.11-2-1}), it is not hard to verify that $u_{\varep, \ell}|_{\partial\Omega}\to f|_{\partial\Omega}$
 a.e.\,on $\partial \Omega$. Hence, $w_\varep=f$ on $\partial\Omega$. As a result, by the uniqueness of the Dirichlet problem,
  we obtain $w_\varep =u_\varep$ in $\Omega$. In view of (\ref{5.11-2-3}) this completes the proof.
 \end{proof}

To handled the $L^2$ Dirichlet problem we first establish two lemmas.

\begin{lemma}\label{lemma-5.11-1}
Let $F\in C_0^\infty(\br^d)$ and
$$
u(x)=\int_{\br^d}  \frac{F(y)}{|x-y|^{d-1}}\, dy.
$$
Let $\Omega$ be a bounded Lipschitz domain in $\br^d$, $d\ge 2$.
Then $\| u\|_{L^2(\partial\Omega)} \le C \, \| F\|_{L^p(\br^d)}$, where $p=\frac{2d}{d+1}$ and $C$ depends only on the Lipschitz character of $\Omega$.
\end{lemma}

\begin{proof}
By dilation we may assume that diam$(\Omega)=1$.
Choose a vector field $h=(h_1, \dots, h_d)\in C_0^1(\br^d;\br^d)$ such that $\langle h, n\rangle \ge c_0>0$ on $\partial\Omega$.
It follows by the divergence theorem that
\begin{equation}\label{5.11-3-1}
\aligned
c_0 \int_{\partial\Omega} |u|^2\, d\sigma  &\le \int_{\partial\Omega} \langle h, n\rangle  |u|^2\, d\sigma\\
&=\int_\Omega |u|^2\, \text{div} (h)\, dx
+2\int_\Omega u\, \frac{\partial u}{\partial x_i} \, h_i \, dx.
\endaligned
\end{equation}
By H\"older's inequality this gives
\begin{equation}\label{5.11-3-3}
\| u\|_{L^2(\partial\Omega)}
\le C \left\{ \| u\|_{L^{p^\prime}(\Omega)} +\| u\|^{1/2}_{L^{p^\prime}(\Omega)} \|\nabla u\|^{1/2}_{L^p(\Omega)} \right\}.
\end{equation}
Note that $1<p<d$ and $\frac{1}{p^\prime} =\frac{1}{p}-\frac{1}{d}$.
Thus, by the fractional integral estimates and singular integral estimates,
$$
\|u\|_{L^{p^\prime}(\Omega)} + \|\nabla u\|_{L^p(\Omega)} \le C\, \| F\|_{L^p(\br^d)}.
$$
which, together with (\ref{5.11-3-3}), gives $\|u\|_{L^2(\partial\Omega)} \le C\, \|F\|_{L^p(\br^d)}$.
\end{proof}

\begin{lemma}\label{lemma-5.11-2}
Suppose that $A\in \Lambda(\mu, \lambda, \tau)$ and $A^*=A$.
Let $\Omega$ be a bounded Lipschitz domain in $\br^d$, $d\ge 2$.
Let $u_\varep\in H^1(\Omega;\br^m)$ be a weak solution of the Dirichlet problem: $\mathcal{L}_\varep (u_\varep)=0$
in $\Omega$  and $u_\varep =f$ on $\partial\Omega$, where $f\in H^1(\partial\Omega; \br^m)$.
Then 
\begin{equation}\label{5.11-4-0}
\int_\Omega |u_\varep|^2\, dx \le C\, r_0 \int_{\partial\Omega} |f|^2\, d\sigma,
\end{equation}
where $r_0=\text{\rm diam}(\Omega)$ and $C$ depends only on $\mu$, $\lambda$, $\tau$, and the Lipschitz character of $\Omega$.
\end{lemma}

\begin{proof}
By dilation we may assume that $r_0=1$. Choose a ball $B$ of radius $2$ such that $\Omega\subset B$.
Let $G_\varep (x,y)$ denote the matrix of Green functions for $\mathcal{L}_\varep$ in $2B$.
Let $F\in C_0^\infty(\Omega; \br^m)$ and
$$
v_\varep (x)=\int_{\Omega} G_\varep (x, y) F(y)\, dy.
$$
Then $v_\varep \in C^1(2B;\br^m)$ and $\mathcal{L}_\varep (v_\varep)=F$ in $\Omega$. 
Recall  that $|G_\varep (x,y)|\le C |x-y|^{2-d}$
for $d\ge 3$, and $|\nabla_x G_\varep(x,y)|\le C |x-y|^{1-d}$ for $d\ge 2$.
Also, $|G_\varep(x,y)|\le C \left( 1+\big| \ln |x-y|\big|\right)$ if $d=2$.
Hence, by Lemma \ref{lemma-5.11-1},
$\| v_\varep\|_{L^2(\partial\Omega)} \le C\, \| F\|_{L^2(\Omega)}$ and 
 $\|\nabla v_\varep\|_{L^2(\partial\Omega)} \le C\, \| F\|_{L^2(\Omega)}$.
 In particular,  we see that $\| v_\varep\|_{H^1(\partial\Omega)}
 \le C\, \| F\|_{L^2(\Omega)}$.
 
 Let $w_\varep\in H^1(\Omega;\br^m)$ be the weak solution of
 $\mathcal{L}_\varep (w_\varep)=0$ in $\Omega$ with $w_\varep =v_\varep$ on $\partial\Omega$.
 By Theorem \ref{L-2-regularity-theorem-5.11}, $\nabla w_\varep$ has nontangential limit a.e.\,on $\partial\Omega$, and
 \begin{equation}\label{5.11-4-2}
 \|\nabla w_\varep\|_{L^2(\partial\Omega)}
 \le C\, \| v_\varep\|_{H^1(\partial\Omega)} \le C\, \| F\|_{L^2(\Omega)}.
 \end{equation}
 Let $z_\varep =v_\varep -w_\varep$ in $\Omega$. Then $z_\varep \in H^1_0(\Omega;\br^m)$ and $\mathcal{L}_\varep (z_\varep)=F$ in $\Omega$.
 It follows by the divergence theorem that
 \begin{equation}\label{5.11-4-3}
 \aligned
 \left|\int_\Omega u_\varep \cdot F \, dx \right|
 &=\left|\int_{\partial\Omega} \frac{\partial z_\varep}{\partial\nu_\varep} \cdot u_\varep \, d\sigma \right|\\
 &\le C\, \big\| \nabla z_\varep \big\|_{L^2(\partial\Omega)} \| f\|_{L^2(\partial\Omega)}\\
 & \le C\, \| F\|_{L^2(\Omega)} \| f\|_{L^2(\partial\Omega)}.
 \endaligned
 \end{equation}
 By duality this gives $\| u_\varep\|_{L^2(\Omega)}\le C\, \| f\|_{L^2(\partial\Omega)}$.
 \end{proof}

\begin{thm}[$L^2$ Dirichlet problem]\label{L-2-Dirichlet-theorem-5.11}
Let $\Omega$ be a bounded Lipschitz domain in $\br^d$, $d\ge 2$.
Given $f\in L^2(\partial\Omega;\br^m)$, there exists a unique $u_\varep\in C^1(\Omega;\br^m)$ such that
$\mathcal{L}_\varep (u_\varep)=0$ in $\Omega$, $u_\varep =f$ n.t.\,on $\partial\Omega$, and $(u_\varep)^*\in L^2(\partial\Omega)$.
Moreover, the solution satisfies $\|(u_\varep)^*\|_{L^2(\partial\Omega)} \le C\, \| f\|_{L^2(\partial\Omega)}$,
where $C$ depends only on $\mu$, $\lambda$, $\tau$, and the Lipschitz character of $\Omega$.
\end{thm}

\begin{proof}
By dilation we may assume that diam$(\Omega)=1$.
The uniqueness follows readily from Lemma \ref{lemma-5.11-2} by choosing a sequence of smooth domains $\{ \Omega_\ell\}$
such that $\Omega_\ell \uparrow \Omega$.
To establish the existence as well as the nontangential-maximal-function estimate, we first consider the case $d\ge 3$.
Choose a sequence of functions $\{ f_\ell\}$ in $C_0^\infty(\br^d;\br^m)$
such that $f_\ell \to f$ in $L^2(\partial\Omega; \br^m)$. Let $u_{\varep, \ell}$ be the solution to the
$L^2$ regularity problem for $\mathcal{L}_\varep$ in $\Omega$ with Dirichlet data $f_\ell$, given by Theorem 
\ref{L-2-regularity-theorem-5.11}.
By the localized $L^2$ estimate (\ref{local-L-2-Dirichlet-estimate}) we obtain
\begin{equation}\label{5.11-5-1}
\aligned
\int_{\partial\Omega} |(u_{\varep,\ell} )^*|^2\, d\sigma
& \le C \int_{\partial\Omega}| f_\ell|^2\, d\sigma + C \int_\Omega |u_{\varep, \ell}|^2\, dx\\
&\le C \int_{\partial\Omega} |f_\ell|^2\, d\sigma,
\endaligned
\end{equation}
where we have used Lemma \ref{lemma-5.11-2}.
Similarly, we have 
\begin{equation}\label{5.11-5-3}
\| (u_{\varep, \ell} -u_{\varep, k})^*\|_{L^2(\partial\Omega)}
\le C\, \| f_\ell -f_k\|_{L^2(\partial\Omega)}
\end{equation}
for any $k, \ell\ge 1$. By interior Lipschitz estimates this implies that $u_{\varep, \ell}\to u_\varep$ and
$\nabla u_{\varep, \ell} \to \nabla u_\varep$ uniformly on any compact subset of $\Omega$.
Clearly, $u_\varep\in H^1_{\loc} (\Omega;\br^m)$ and $\mathcal{L}_\varep (u_\varep) =0$ in $\Omega$.
By a limiting argument we may also deduce from (\ref{5.11-5-1}) and (\ref{5.11-5-3})  that
$\|(u_\varep)^*\|_{L^2(\partial\Omega)} \le C\, \| f\|_{L^2(\partial\Omega)}$ and
\begin{equation}\label{5.11-5-5}
\| (u_{\varep, \ell} -u_\varep)^*\|_{L^2(\partial\Omega)}
\le C\, \| f_\ell -f\|_{L^2(\partial\Omega)}.
\end{equation}

Next, to show $u_\varep=f$ n.t.\,on $\partial\Omega$, we observe that
$$
\aligned
& \limsup_{x\to Q} |u_\varep (x)-f(Q)|\\
&\le \limsup_{x\to Q} |u_\varep (x)-u_{\varep, \ell} (x)|
+\limsup_{x\to Q} |u_{\varep, \ell} (x)- f_\ell (Q)|
+| f_\ell (Q) -f(Q)|\\
&\le (u_\varep -u_{\varep, \ell})^* (Q) +|f_\ell (Q)- f(Q)|,
\endaligned
$$
where the limits are taken nontangentially from $\Omega$.
This, together with (\ref{5.11-5-5}), implies that $u_\varep=f$ n.t.\,on $\partial\Omega$.

Finally, we consider the case $d=2$. By the approximation argument above, it suffices to show that if
$u_\varep$ is the solution of the $L^2$ regularity problem for $\mathcal{L}_\varep$ in $\Omega$ with
$f\in H^1(\partial\Omega;\br^m)$, given by Theorem \ref{L-2-regularity-theorem-5.11}, then
$\| (u_\varep)^*\|_{L^2(\partial\Omega)} \le C\, \| f\|_{L^2(\partial\Omega)}$.
This may be done by the method of descending.
Indeed, consider $v_\varep (x_1, x_2, x_3)=u_\varep (x_1, x_2)$ in $\Omega_1 =\Omega \times (0,1)$.
Since $\| (v_\varep)^*\|_{L^2(\partial\Omega_1)} \le C\, \| v_\varep\|_{L^2(\partial\Omega_1)}$,
it follows that
$$
\aligned
\| (u_\varep)^*\|_{L^2(\partial\Omega)}
& \le C \left\{ \| f\|_{L^2(\partial\Omega)} +\| u_\varep\|_{L^2(\Omega)} \right\}\\
&\le C\, \| f\|_{L^2(\partial\Omega)},
\endaligned
$$
where we have used Lemma \ref{lemma-5.11-2} for the last inequality.
\end{proof}

\begin{remark}\label{Green-function-remark}
{\rm 
Let $A\in \Lambda(\mu, \lambda, \tau)$, $A^*=A$, and $\Omega$
be a bounded Lipschitz domain.
As in the proof of Theorem \ref{theorem-5.6.2}, one may construct a matrix of Green functions $G_\varep(x,y)
=\big(G_\varep^{\alpha\beta}(x,y)\big)_{m\times m}$
for $\mathcal{L}_\varep$ in $\Omega$,
where for $d\ge 3$,
\begin{equation}\label{5.11-6-1}
G_\varep^{\alpha\beta}(x,y) =\Gamma_\varep^{\alpha\beta}(x,y)
-W_\varep^{\alpha\beta} (x,y)
\end{equation}
and for each $\beta\in \{1, \dots, m\}$ and $y\in \Omega$, $W^\beta(\cdot, y)
=\big(W^{1\beta}(\cdot, y), \dots, W^{m\beta}(\cdot, y)\big)$
is the unique solution of the $L^2$ regularity problem
for $\mathcal{L}_\varep (u_\varep)=0$ in $\Omega$ with 
Dirichlet data $\big(\Gamma_\varep^{1\beta}(\cdot, y), \dots, \Gamma^{m\beta}_\varep(\cdot, y)\big)$
on $\partial\Omega$.
If $d=2$, we replace the matrix of fundamental solutions in (\ref{5.11-6-1}) by
the matrix of Green functions in a ball $B$ such that $\Omega\subset (1/2)B$.
Note that $G_\varep (\cdot, y)=0$ n.t.\,on $\partial\Omega$
and if $r<\text{dist}(y, \partial\Omega)/4$,
\begin{equation}\label{5.9.10-0}
\mathcal{M}_r^1(\nabla_x G_\varep(\cdot, y))\in L^2(\partial\Omega),
\end{equation}
where the operator $\mathcal{M}_r^1$ is defined in (\ref{definition-M-1-2}).
We claim that if $u_\varep$ is a solution of the $L^2$ Dirichlet problem for
$\mathcal{L}_\varep (u_\varep)=0$ in $\Omega$, it may be represented by the Poisson 
integral,
\begin{equation}\label{Poisson-representation-5}
u_\varep (x)
=-\int_{\partial\Omega} \frac{\partial}{\partial \nu_\varep  (y)}
\big\{ G_\varep(y,x)\big\} u_\varep (y)\, d\sigma (y)
\end{equation}
for any $x\in \Omega$.

To see (\ref{Poisson-representation-5}), we first assume that $u_\varep$ is a solution
of the $L^2$ regularity problem in $\Omega$.
Let $\Omega_\ell$ be a sequence of smooth domains such that $\Omega_\ell\uparrow\Omega$.
By the Green representation formula,
\begin{equation}\label{5.9.10-1}
u_\varep (x) =\int_{\partial\Omega_\ell }
G_\varep (y,x) \frac{\partial u_\varep }{\partial \nu_\varep}\, d\sigma (y) -\int_{\partial\Omega_\ell}
\frac{\partial}{\partial\nu_\varep (y)}
\big\{ G_\varep (y, x) \big\} u_\varep (y)\, d\sigma (y),
\end{equation}
where we have used the symmetry condition $A^*=A$. Letting $\ell \to \infty$ in (\ref{5.9.10-1}) and
using $(\nabla u_\varep)^*\in L^2(\partial\Omega)$ and the dominated convergence theorem,
we see that the first integral in (\ref{5.9.10-1})
converges to zero and the second converges to the right hand side of (\ref{Poisson-representation-5}).
In general, if $u_\varep$ is the solution of the $L^2$ Dirichlet problem with data $f\in L^2(\partial\Omega;\br^m)$,
we choose $f_\ell \in C_0^\infty(\br^d;\br^m)$ such that $f_\ell \to f$ in $L^2(\partial\Omega;\br^m)$.
Let $u_{\varep, \ell}$ be the solution of the $L^2$ Dirichlet problem in $\Omega$ with data $f_\ell$.
Then $(\nabla u_{\varep, \ell})^*\in L^2(\partial\Omega)$ and thus (\ref{Poisson-representation-5})
holds for $u_{\varep, \ell}$.
Since $\|(u_{\varep, \ell} -u_\varep)^*\|_{L^2(\partial\Omega)}\le C\, \| f_\ell -f\|_{L^2(\partial\Omega)}$,
it follows that $u_{\varep, \ell}$ converges to $u_\varep $ uniformly on any compact subset of $\Omega$.
By letting $\ell\to \infty$,
 we may conclude that (\ref{Poisson-representation-5}) 
continues to hold for $u_\varep$.
}
\end{remark}


\section{Square function and $H^{1/2}$ estimates}\label{section-5.10}

Let $\Omega$ be a bounded Lipschitz domain in $\br^d$.
For a function $u\in H^1_{\loc} (\Omega)$, the square function $S(u)$ on $\partial\Omega$ is defined 
by
$$
S(u)(x) =\left(\int_{\gamma (x)} |\nabla u(y)|^2 |y-x|^{2-d}\, dy \right)^{1/2}
$$
for $x\in \partial\Omega$, where $\gamma (x)=\{ y\in \Omega: |y-x|< (1+\alpha) \text{\rm dist}(x, \partial\Omega) \}$
and $\alpha=\alpha (\Omega)>1$ is sufficiently large.
It is not hard to see that
$$
\int_{\partial\Omega} |S(u)|^2\, d\sigma 
\approx \int_\Omega |\nabla u(x)|^2 \text{\rm dist} (x, \partial\Omega)\, dx.
$$
In this section we show that solutions to the $L^2$ Dirichlet problem for $\mathcal{L}_\e$ satisfy 
uniform square function estimates and $H^{1/2}$ estimates.

\begin{thm}\label{square-function-theorem}
Suppose that $A\in \Lambda(\mu, \lambda, \tau)$ and $A^*=A$.
Let $\Omega$ be a bounded Lipschitz domain and $f\in L^2(\partial\Omega;\br^m)$.
Let $u_\varep$ be the solution of the Dirichlet problem: $\mathcal{L}_\varep (u_\varep)=0$ in $\Omega$
and $ u_\varep =f$ n.t.\,on $\partial\Omega$ with $(u_\varep)^*\in L^2(\partial\Omega)$.
Then
\begin{equation}\label{square-function-estimate}
\aligned
\left(\int_\Omega |\nabla u_\varep (x)|^2\, \text{\rm dist}(x, \partial\Omega)\, dx \right)^{1/2}
 &\le C \| f\|_{L^2(\partial\Omega)},\\
 \left(\int_\Omega \int_\Omega \frac{|u_\e (x)-u_\e (y)|^2}{|x-y|^{d+1}}\, dxdy \right)^{1/2}
 &\le C \| f\|_{L^2(\partial\Omega)},
 \endaligned
\end{equation}
where $C$ depends only on $\mu$, $\lambda$, $\tau$, and the Lipschitz character of $\Omega$.
\end{thm}

Let $D$ be the region above a Lipschitz graph; i.e.,
\begin{equation}\label{definition-of-D-5.10}
D =\big\{ (x^\prime, x_d)\in \br^d:\  x^\prime \in \br^{d-1} \text{ and } x_d>\psi(x^\prime)\big\},
\end{equation}
where $\psi$ is a Lipschitz function on $\br^{d-1}$ such that $\psi(0)=0$ and $\|\nabla\psi\|_\infty \le M_0$.

\begin{lemma}\label{lemma-5.10-1}
Let $g\in L^2(\partial D;\br^m)$ and
$$
u_\varep (x) =\int_{\partial D} \frac{\partial}{\partial y_k} \big\{ \Gamma_\varep (x,y)\big\}
g(y)\, d\sigma (y),
$$
where $D$ is given by (\ref{definition-of-D-5.10}). Then
\begin{equation}\label{estimate-5.10-1}
\left(\int_D |\nabla u_\varep (x)|^2\, \text{\rm dist}(x, \partial D)\, dx \right)^{1/2}
\le C\, \| g\|_{L^2(\partial D)},
\end{equation}
where $C$ depends only on $\mu$, $\lambda$, $\tau$, and $M_0$.
\end{lemma}

\begin{proof}
By rescaling we may assume that $\varep=1$.
We first estimate the integral of \newline $|\nabla u_1 (x)|^2 \, \text{dist}(x, \partial D)$ on the set
$$
D_1 =D +(0, \dots, 1) =\big\{ (x^\prime, x_d): \ x_d>\psi (x^\prime) +1 \big\}.
$$
This is done by using the asymptotic estimates of $\nabla_x\nabla_y \Gamma_1(x,y)$ for 
$|x-y|\ge 1$. Indeed, by Theorem \ref{fundamental-solution-theorem-4}, 
$$
|\nabla_x\nabla_y \Gamma_1 (x,y) -\big(I +\nabla \chi(x)\big) \nabla_x\nabla_y \Gamma_0 (x,y)
\big(I+\nabla \chi^* (y)\big )^{T} |\le \frac{C\, \varep \ln [|x-y| +2]}{|x-y|^{d+1}}
$$
for any $x,y\in\br^d$ and $x\neq y$,
where $\Gamma_0(x,y)$ is the matrix of fundamental solutions for $\mathcal{L}_0$.
It follows that 
\begin{equation}\label{5.10-1-1}
|\nabla u_1 (x) -W(x)|\le C \int_{\partial D} \frac{ \ln [|x-y|+2]}{|x-y|^{d+1}}\, |g(y)|\, d\sigma (y),
\end{equation}
where $W(x)$ is a finite sum of functions of form
$$
e_{ij} (x) \int_{\partial D} \frac{\partial^2}{\partial x_i\partial x_j}
\big\{ \Gamma_0 (x,y)\big\} h(y)\, d\sigma (y),
$$
 $|e_{ij} (x)|\le C$, and $|h(y)|\le C |g (y)|$.
Since $\mathcal{L}_0$ is a second-order elliptic operator with constant coefficients, it is known that
\begin{equation}\label{5.10-1-3}
\int_{D} |W(x)|^2 \, \text{dist}(x, \partial D)\, dx \le C \int_{\partial D} |g|^2\, d\sigma
\end{equation}
(see e.g. \cite{DKPV-1997}).
Let $I(x)$ denote the integral in the RHS of (\ref{5.10-1-1}).
By the Cauchy inequality,
$$
|I(x)|^2 \le \frac{C \ln [\text{dist}(x, \partial D) +2]}{[\text{dist}(x, \partial D)]^2} 
\int_{\partial D} \frac{ \ln [|x-y|+2]}{|x-y|^{d+1}}\, |g(y)|^2\, d\sigma (y).
$$
This gives
\begin{equation}\label{5.10-1-5}
\int_{D_1} |I(x)|^2\, \text{dist}(x, \partial D) \, dx \le C\, \int_{\partial D} |g|^2\, d\sigma.
\end{equation}
In view of (\ref{5.10-1-3}) and (\ref{5.10-1-5})
 we have proved that 
 $$
  \int_{D_1} |\nabla u_1 (x)|^2\, \text{dist}(x, \partial D)\, dx \le \| g\|^2_{L^2(\partial D)}.
  $$

To handle $|\nabla u_1|$ in $D\setminus D_1$, we let 
$$
\aligned
\Delta_r & =\big\{ \big(x^\prime, \psi(x^\prime)\big)\in \br^d:\ |x^\prime|<r\big\},\\
T_r & =\big\{ (x^\prime, x_d)\in \br^d: \ |x^\prime|<r \text{ and } \psi(x^\prime)<x_d < \psi(x^\prime) + C_0 r\big\}.
\endaligned
$$
We will show that if $\mathcal{L}_1 (u)=0$ in the Lipschitz domain $T_2 $ and $(u)^*\in L^2(\partial T_2)$, then
\begin{equation}\label{claim-5.10-1}
\int_{T_1} |\nabla u(x)|^2 \, |x_d -\psi(x^\prime)|\, dx
\le C \int_{\Delta_2} |u|^2\, d\sigma 
+C \int_{T_2 } |u|^2\, dx,
\end{equation}
which is bounded by $C \int_{ \Delta_2 } |(u)^*|^2\, d\sigma$.
By a simple covering argument one may deduce from (\ref{claim-5.10-1}) that
\begin{equation}\label{5.10-1-7}
\aligned
\int_{D\setminus D_1}
|\nabla u_1 (x)|^2\, \text{dist}(x, \partial D)\, dx & \le C \int_{\partial D} |(u_1)^*|^2\, d\sigma\\
&\le C\int_{\partial D} |g|^2\, d\sigma,
\endaligned
\end{equation}
where we have used Theorem \ref{theorem-5.2-3} for the last inequality.

Finally, to see (\ref{claim-5.10-1}), we use the square function estimate for solutions of
$\mathcal{L}_1 (u)=0$ in the domain
$T_r$ for $(3/2)<r<2$,
\begin{equation}\label{5.10-1-9}
\int_{T_r} |\nabla u(x)|^2 \, \text{dist}(x, \partial T_r)\, dx
\le C \int_{\partial T_r} |u|^2\, d\sigma
\end{equation}
to obtain
\begin{equation}\label{5.10-1-11}
\int_{T_1} |\nabla u(x)|^2\, |x_d -\psi(x^\prime)|\, dx
\le C \int_{\Delta_2} |u|^2\, d\sigma
+C \int_{\partial T_r\setminus \Delta_2} |u|^2\, d\sigma.
\end{equation}
The estimate (\ref{claim-5.10-1}) follows by integrating both sides of (\ref{5.10-1-11}) in $r$ over the interval $(3/2,2)$.
We remark that  (\ref{5.10-1-9}) is a special case of (\ref{square-function-estimate})
with $\varep=1$ and diam$(\Omega)\le C$. Under the conditions that $A\in \Lambda(\mu, \lambda, \tau)$ and $A^*=A$,
the square function estimate (\ref{5.10-1-9}) follows from the double-layer-potential representation obtained in
Theorem \ref{L-2-periodic-Dirichlet-theorem} for solutions of the $L^2$ Dirichlet problem for
$\mathcal{L}_1$, by a $T(b)$-theorem argument (see e.g. \cite[pp.9-11]{Mitrea-D-2001}).
This completes the proof.
\end{proof}

The next lemma shows that for solutions of $\mathcal{L}_\varep (u_\varep)=0$ in a Lipschitz domain $\Omega$, 
$$
\| u_\varep\|^2_{H^{1/2}(\Omega)}
\approx
\left\{ \int_\Omega |\nabla u_\e (x)|^2\, \text{dist}(x, \partial\Omega)\, dx 
+\int_\Omega |u_\e (x)|^2\, dx \right\}.
$$

\begin{lemma}\label{lemma-5.10-2}
Let $\Omega$ be a bounded Lipschitz domain with $\text{\rm diam}(\Omega)=1$.
Then
\begin{equation}\label{estimate-5.10-2-1}
\| u\|_{H^{1/2}(\Omega)}^2 \le C \left\{ \int_\Omega |\nabla u(x)|^2\, \text{\rm dist}(x, \partial\Omega)\, dx 
+\int_\Omega |u(x)|^2\, dx \right\},
\end{equation}
where $C$ depends only on the Lipschitz character of $\Omega$.
Moreover, if  $A\in \Lambda(\mu, \lambda, \tau)$ and $\mathcal{L}_\varep (u_\varep) =0$ in $\Omega$, then
\begin{equation}\label{estimate-5.10-2-2}
\int_\Omega |\nabla u_\varep (x)|^2\, \text{\rm dist}(x, \partial\Omega)\, dx
\le C\, \| u_\varep\|^2_{H^{1/2}(\Omega)},
\end{equation}
where $C$ depends only on $\mu$, $\lambda$, $\tau$, and the Lipschitz character of $\Omega$.
\end{lemma}

\begin{proof}
We first prove (\ref{estimate-5.10-2-1}). 
By a partition of unity we may assume that supp$(u)\subset B(z, c)\cap \Omega$ for some $z\in \partial\Omega$
(it is easy to see that $\| u\|_{H^1(\Omega_1)}$ is bounded by the RHS of (\ref{estimate-5.10-2-1}), if
$\Omega_1\subset \subset \Omega$).
By a change of the coordinate system we may further assume that $z=0$ and
$$
\aligned
& \Omega\cap \big\{ (x^\prime, x_d):\ |x^\prime|<4c \text{ and } -4c<x_d<4 c \big\}\\
&\qquad
=\big\{ (x^\prime, x_d):\ |x^\prime|<4c \text{ and } \psi (x^\prime)<x_d<4 c \big\},
\endaligned
$$
where $\psi$ is a Lipschitz function in $\br^{d-1}$ such that $\psi(0)=0$ and $\|\nabla \psi\|_\infty\le M_0$.
We will use the fact that 
$$
H^{1/2}(\Omega)= \big[L^2(\Omega), H^1(\Omega)\big]_{1/2, 2}
$$
from real interpolation. Thus, $\| u\|_{H^{1/2}(\Omega)}$ is comparable to the infimum over all functions
$f: \, [0, \infty)\to L^2(\Omega) + H^1 (\Omega)$ with $f(0)=u$ of
\begin{equation}\label{5.10-2-1}
\left(\int_0^\infty \| t^{1/2} f(t)\|_{H^1(\Omega)}^2 \frac{dt}{t}\right)^{1/2}
+
\left(\int_0^\infty \| t^{1/2} f^\prime (t)\|_{L^2(\Omega)}^2 \frac{dt}{t}\right)^{1/2}.
\end{equation}
Let $f(t)= u(x^\prime, x_d +t) \theta (t)$, where $\theta\in C_0^\infty(\br)$, $\theta (t)=1$ for $|t|\le c$,
and $\theta (t)=0$ for $|t|\ge 2c$.
Clearly, $f(0)= u$. Also, it is not hard to see that the expression (\ref{5.10-2-1}) is bounded by
$$
C \left(\int_0^{2c } \int_\Omega
\Big\{ |\nabla u(x^\prime, x_d +t)|^2 +|u(x^\prime, x_d+t)|^2\Big\}
\Big\{ |\theta (t)|^2 +|\theta^\prime (t)|^2\Big\}\, dx dt\right)^{1/2}.
$$
This implies that
$$
\aligned
\| u\|_{H^{1/2}(\Omega)}^2
 &\le C \int_0^{2c} \int_{|x^\prime|<c} \int_{\psi(x^\prime)}^{2c}
\Big\{ |\nabla u(x^\prime, x_d +t)|^2 +|u(x^\prime, x_d +t)|^2\Big\} dx_d dx^\prime dt\\
& \le C \int_{|x^\prime|<c}
\int_{\psi(x^\prime)}^{4c}\left\{ |\nabla u (x^\prime, x_d)|^2
+|u(x^\prime, x_d)|^2\right\}\, |x_d -\psi(x^\prime)|\, dx_d dx^\prime\\
&\le C
\left\{ \int_\Omega |\nabla u(x)|^2\, \text{dist}(x, \partial\Omega)\, dx 
+\int_\Omega |u(x)|^2\, dx \right\}.
\endaligned
$$

Next, let $u_\varep\in H^1_{\loc}(\Omega)$ be a solution to $\mathcal{L}_\varep (u_\varep)=0$ in $\Omega$.
To prove (\ref{estimate-5.10-2-2}), we use the interior estimate (\ref{estimate-2.2}) to obtain
$$
\aligned
|\nabla u_\varep (x)|  & \le \frac{C}{r} \left(\average_{B(x,r/2)} |u_\varep (y)-u_\varep (x)|^2\, dy\right)^{1/2}\\
& \le C r^{\frac{d-1}{2}} \left(\average_{B(x, r/2)}
\frac{|u_\varep (y)-u_\varep (x)|^2}{|y-x|^{d+1}}\, dy\right)^{1/2},
\endaligned
$$
where $r=\text{dist}(x, \partial\Omega)$.
It follows that
$$
\aligned
|\nabla u_\varep (x)|^2\, \text{dist}(x, \partial\Omega)
&\le C \int_{B(x, r/2)} \frac{|u_\varep (y)-u_\varep (x)|^2}{|y-x|^{d+1}}\, dy\\
& \le C \int_\Omega  \frac{|u_\varep (y)-u_\varep (x)|^2}{|y-x|^{d+1}}\, dy.
\endaligned
$$
Hence,
$$
\aligned
\int_\Omega |\nabla u_\varep (x)|^2\, \text{dist}(x, \partial\Omega)\, dx
& \le  C \int_\Omega
 \int_\Omega  \frac{|u_\varep (y)-u_\varep (x)|^2}{|y-x|^{d+1}}\, dy dx\\
& \le C \| u_\varep \|_{H^{1/2}(\Omega)}^2.
\endaligned
 $$
 \end{proof}

\begin{proof}[\bf Proof of Theorem \ref{square-function-theorem}]
It suffices to show that
\begin{equation}\label{5.10-3-1}
\left(\int_\Omega |\nabla u_\varep (x)|^2\, \text{dist}(x, \partial\Omega)\, dx \right)^{1/2}
\le C \, \| f\|_{L^2(\partial\Omega)}.
\end{equation}
By Lemma \ref{lemma-5.10-2}
the estimate $\|u\|_{H^{1/2}(\Omega)} \le C\, \| f\|_{L^2(\partial\Omega)}$ follows from (\ref{5.10-3-1})
and the observation  that $\|u\|_{L^2(\Omega)}\le C\, \| (u)^*\|_{L^2(\partial\Omega)} \le C\, \| f\|_{L^2(\partial\Omega)}$.

By Theorem \ref{L-2-periodic-Dirichlet-theorem} the solution $u_\varep$ is given 
by a double layer potential with a density function $g_\varep$ and
$\| g_\varep\|_{L^2(\partial\Omega)} \le C \, \| f\|_{L^2(\partial\Omega)}$.
As a result, it suffices to show that the LHS of (\ref{5.10-3-1})
is bounded by $C \| g\|_{L^2(\partial\Omega)}$, if $u_\varep$ is given by
\begin{equation}\label{5.10-3-3}
u_\varep (x)=\int_{\partial\Omega} \frac{\partial}{\partial y_k} \big\{ \Gamma_\varep (x,y) \big\}
g(y)\, d\sigma (y)
\end{equation}
for some $g\in L^2(\partial\Omega)$. Using the estimate $|\nabla_x\nabla_y \Gamma_\varep (x,y)|\le 
C|x-y|^{-d}$ and a partition of unity, we may further reduce the problem to the estimate 
\begin{equation}\label{5.10-3-5}
\int_{\Omega \cap B(z, cr_0)} |\nabla u_\varep (x)|^2\, \text{dist}(x, \partial\Omega)\, dx
\le C \int_{\partial\Omega} |g|^2\, d\sigma,
\end{equation}
 where $u_\varep$ is given by (\ref{5.10-3-3}) and
 supp$(g)\subset B(z, cr_0)\cap \partial\Omega$
for $r_0=\text{diam}(\Omega)$ and some $z\in \partial\Omega$.
However, the estimate (\ref{5.10-3-5}) follows readily from Lemma \ref{lemma-5.10-1} by a change of the coordinate
system. This completes the proof.
\end{proof}

We end this section with  an error estimate in $H^{1/2}$.

\begin{thm}\label{last-thm}
Suppose that $A\in \Lambda (\mu, \lambda, \tau)$ and $A^*=A$.
Let $\Omega$ $(\e\ge 0)$ be a bounded Lipschitz domain in $\br^d$, $d\ge 2$.
Let $u_\e$ be the solution to the Dirichlet problem:
$\mathcal{L}_\e (u_\e)=F$ in $\Omega$ and
$u_\e=f$ on $\partial\Omega$,
where $F\in L^2(\Omega; \br^d)$ and $f\in H^1(\partial\Omega; \br^m)$.
Then
\begin{equation}\label{last-1}
\|  u_ \e -u_0 -\e \chi(x/\e) \nabla u_0\|_{H^{1/2}(\Omega)}
\le C\e  \| u_0\|_{H^2(\Omega)},
\end{equation}
where $C$ depends only on $\mu$, $\lambda$, $\tau$ and $\Omega$.
\end{thm}

\begin{proof}
Let 
$$
\aligned
w_\e   & =u_\e -u_0 -\e \chi(x/\e)\nabla u_0\\
 & = w_\e^{(1)}  + w_\e^{(2)},
 \endaligned
$$
where
$$
\mathcal{L}_\e (w_\e^{(1)}) =\mathcal{L}_\e (w_\e) \quad \text{ in } \Omega
\quad \text{ and } \quad w_\e^{(1)} =0 \quad \text{ on } \partial\Omega,
$$
and 
$$
\mathcal{L}_\e (w_\e^{(2)}) =0\quad \text{ in } \Omega
\quad \text{ and } \quad w_\e^{(2)} =w_\e \quad \text{ on } \partial\Omega.
$$
By (\ref{right-hand-side}) and the energy estimate,
\begin{equation}\label{last-3}
\| w_\e^{(1)} \|_{H^1(\Omega)} \le C \e \|\nabla ^2 u_0\|_{L^2(\Omega)}.
\end{equation}
Note that $w_\e = -\chi(x/\e)\nabla u_0$ on $\partial\Omega$.
Hence, by Theorem \ref{square-function-theorem}, 
$$
\| w_\e^{(2)}\|_{H^{1/2}(\Omega)}
\le C \e \|\nabla u_0\|_{L^2(\partial\Omega)},
$$
which, together with (\ref{last-3}), gives
$$
\aligned
\| w_\e \|_{H^{1/2}(\Omega)}
 &\le C \e \big\{ \|\nabla^2 u_0\|_{L^2(\Omega)} +\| \nabla u_0\|_{L^2(\partial\Omega)} \big\}\\
&\le C \e \| u_0\|_{H^2(\Omega)}.
\endaligned
$$
\end{proof}


\section{Notes}\label{section-5.12}

The main results in this chapter were proved by C. Kenig and Z. Shen in \cite{KS-2011-L}.
In particular, material in Sections \ref{section-5.2}-\ref{section-5.4}, 
\ref{section-5.6} - \ref{section-5.7}, and
\ref{section-5.9} are taken from \cite{KS-2011-L}.
The treatment for the large-scale Rellich estimates in Section \ref{section-5.8},
which is different from that in \cite{KS-2011-L}, follows an approach in \cite{Shen-2017-APDE}.
The square function estimate in Section \ref{section-5.10} was proved in \cite{KLS-2012}.
Earlier work on square function estimates  in Lipschitz domains may be found in 
\cite{D-1980, DKPV-1997}.

There exists an extensive literature on $L^p$ boundary value problems for elliptic  equations and systems
in nonsmooth domains. Classical results for Laplace's equation in Lipschitz domains, which is presented
in Section \ref{section-5.5} in the case $p=2$, may be found in \cite{D-1977, D-1979-P, Kenig-1981,Verchota-1984, DK-1987}. 
 We refer the reader to  \cite{Kenig-book, Mitrea-D-2001, Shen-2006, Shen-2007-L} for further references.
For elliptic operators with periodic coefficients, the $L^2$ Dirichlet problem  in Lipschitz domains for the 
scale case $(m=1)$ was solved by B. Dahlberg, using the method of harmonic measure (unpublished; 
a proof may be found in the Appendix of \cite{KS-2011-H}).
This extends an earlier work in \cite{AL-1987-L-P, AL-1987} for $C^{1, \alpha}$ domains by M. Avellaneda and F. Lin.
The well-posedness of the $L^2$ Neumann and regularity problems in the periodic setting was first obtained
by C. Kenig and Z. Shen in \cite{KS-2011-H} in the case $m=1$.
In \cite{KS-2011-H, KS-2011-L} the large-scale Rellich estimates were established by using integration by parts.

The estimates in (\ref{L-2-estimate}) continue to hold for elliptic systems of elasticity.
For the $L^2$ Dirichlet problem (\ref{L-2-DP}) and regularity problem (\ref{L-2-RP}),
this follows readily from the observation that the system of elasticity may be rewritten in such a way that
the new coefficient matrix $\widetilde{A}\in \Lambda (\mu, \lambda, \tau)$ and
is symmetric  (see the end of Section \ref{section-1.4}).
For the $L^2$ Neumann problem, the results were proved in \cite{GSS-2017} by J. Geng, L. Song, and Z. Shen.

For the $L^p$ estimates in the periodic setting,
the results are known for $1<p<\infty$, if $\Omega$ is $C^{1, \alpha}$ \cite{AL-1987, KLS-2013-N}.
For  a Lipschitz domain $\Omega$,  if $m=1$ or $d=2, 3$, one may  establish 
the $L^p$ estimate for $2-\delta <p<\infty$ in the case of Dirichlet problem, and
for $1<p<2+\delta$ in the case of Neumann and regularity problems.
With minor modifications, the method used for the case of constant coefficients  in \cite{DK-1987, Kenig-D-1990} 
works equally well in the periodic setting.
If $m\ge 2$ and $d\ge 4$, partial results may be obtained using the approaches in \cite{Shen-2006, Shen-2007-L}.


\backmatter

\bibliography{Homogenization-1.bbl}

\printindex

\end{document}